\documentclass[oneside,a4paper,english, 11pt]{amsart}

\usepackage{adjustbox}
\usepackage[latin9]{inputenc}
\usepackage{hyperref}
\usepackage{lscape}
\usepackage{slashbox}
\usepackage{mathrsfs}
\usepackage[title]{appendix}
\usepackage{enumerate}
\usepackage{enumitem}
\usepackage{xcolor}
\usepackage{bm}
\usepackage[sort=def]{glossaries}

\makeatletter

\numberwithin{equation}{subsection}

 \theoremstyle{plain}
\newtheorem{thm}[equation]{Theorem}
 \theoremstyle{plain}
\newtheorem{cond}[equation]{Condition}
\theoremstyle{plain}

\theoremstyle{plain}
\newtheorem{state}[equation]{Statement}
\theoremstyle{plain}
\newtheorem{claim}[equation]{Claim}
\theoremstyle{plain}

\theoremstyle{plain}
  \newtheorem{prop}[equation]{Proposition}
\theoremstyle{plain}
 \newtheorem{lemma}[equation]{Lemma}
\theoremstyle{plain}

\theoremstyle{plain}
\newtheorem{cor}[equation]{Corollary}
\theoremstyle{plain}

\theoremstyle{definition}
  \newtheorem{defn}[equation]{Definition}
\theoremstyle{definition}
  
 \theoremstyle{definition}
  
  \newtheorem{exam}[equation]{Example}
\theoremstyle{remark}
\newtheorem{rmk}[equation]{Remark}
\usepackage{amsmath}
\usepackage{amssymb}
\usepackage{amscd}
\usepackage{amsthm}
\usepackage{array,multirow}
\usepackage[arrow,matrix,tips,all,cmtip,poly,color]{xy}

\usepackage{stmaryrd}
\usepackage{url}
\usepackage{tikz-cd}
\usepackage{color}
\usepackage{caption}
\usepackage{float}

\pdfpageheight\paperheight
\pdfpagewidth\paperwidth
\newcommand{\Z}{\mathbb{Z}}
\newcommand{\ZZ}{\mathbb{Z}}
\newcommand{\Q}{\mathbb{Q}}

\newcommand{\Qp}{\mathbb{Q}_p}

\newcommand{\R}{\mathbb{R}}
\newcommand{\C}{\mathbb{C}}
\newcommand{\bG}{\mathbb{G}}

\newcommand{\F}{\mathbb{F}}

\newcommand{\N}{\mathbb{N}}

\newcommand{\A}{\mathbb{A}}

\newcommand{\fM}{\mathfrak{M}}

\newcommand{\fp}{\mathfrak{p}}

\newcommand{\fS}{\mathfrak{S}}

\newcommand{\fm}{\mathfrak{m}}

\newcommand{\bA}{\mathbb{A}}

\newcommand{\bP}{\mathbb{P}}

\newcommand{\bT}{\mathbb{T}}

\newcommand{\bV}{\mathbb{V}}

\newcommand{\cB}{\mathcal{B}}
\newcommand{\cC}{\mathcal{C}}

\newcommand{\cE}{\mathcal{E}}
\newcommand{\cF}{\mathcal{F}}
\newcommand{\cG}{\mathcal{G}}
\newcommand{\cH}{\mathcal{H}}
\newcommand{\cI}{\mathcal{I}}

\newcommand{\cJ}{\mathcal{J}}

\newcommand{\cL}{\mathcal{L}}
\newcommand{\cM}{\mathcal{M}}
\newcommand{\cN}{\mathcal{N}}
\newcommand{\cO}{\mathcal{O}}
\newcommand{\cP}{\mathcal{P}}

\newcommand{\cS}{\mathcal{S}}

\newcommand{\cX}{\mathcal{X}}
\newcommand{\cY}{\mathcal{Y}}

\newcommand{\al}{\alpha}

\newcommand{\eps}{\varepsilon}

\newcommand{\phz}{\varphi}

\newcommand{\La}{\Lambda}

\newcommand{\Zp}{\mathbb{Z}_p}

\newcommand{\val}{\mathrm{val}}
\newcommand{\Id}{\mathrm{id}}
\newcommand{\Gal}{\mathrm{Gal}}
\newcommand{\Hom}{\mathrm{Hom}}

\newcommand{\Res}{\mathrm{Res}}
\newcommand{\Ind}{\mathrm{Ind}}
\newcommand{\ind}{\mathrm{c}\text{-}\mathrm{Ind}}
\newcommand{\End}{\mathrm{End}}

\newcommand{\GL}{\mathrm{GL}}

\newcommand{\U}{\mathrm{U}}

\newcommand{\Spec}{\mathrm{Spec}\ }

\newcommand{\Fp}{\F_p}

\newcommand{\un}[1]{\underline{#1}}
\renewcommand{\bf}[1]{\mathbf{#1}}
\newcommand{\Rep}{\mathrm{Rep}}
\newcommand{\Diag}{\mathrm{Diag}}

\newcommand{\tld}[1]{\widetilde{#1}}

\newcommand{\JH}{\mathrm{JH}}

\newcommand{\BC}{\mathrm{BC}}

\newcommand{\Inv}{\mathrm{Inv}}

\newcommand{\WD}{\mathrm{WD}}

\newcommand{\Tcris}{\mathrm{T}^*_{\mathrm{cris}}}

\newcommand{\rbar}{\overline{r}}
\newcommand{\rhobar}{\overline{\rho}}
\newcommand{\taubar}{\overline{\tau}}
\newcommand{\FL}{\mathrm{FL}}
\newcommand{\Spf}{\mathrm{Spf}}
\newcommand{\semis}{\mathrm{ss}}
\newcommand{\speci}{\mathrm{sp}}

\newcommand{\obv}{\mathrm{extr}}
\newcommand{\out}{\mathrm{out}}

\newcommand{\Adm}{\mathrm{Adm}}

\newcommand{\soc}{\mathrm{soc}}

\newcommand{\Proj}{\mathrm{Proj}}

\newcommand{\defeq}{\stackrel{\textrm{\tiny{def}}}{=}}
\newcommand{\s}{^\times}

\newcommand{\ovl}[1]{\overline{#1}}

\DeclareMathOperator{\Dim}{dim}

\DeclareMathOperator{\Mod}{Mod}

\DeclareMathOperator{\Lie}{Lie}

\DeclareMathOperator{\Fil}{Fil}
\DeclareMathOperator{\Mat}{Mat}

\DeclareMathOperator{\gr}{gr}
\DeclareMathOperator{\coker}{coker}
\DeclareMathOperator{\Gr}{Gr}
\DeclareMathOperator{\Fl}{Fl}

\DeclareMathOperator{\Iw}{\cI}

\newcommand{\ra}{\rightarrow}

\newcommand{\iarrow}{\hookrightarrow}
\newcommand{\into}{\hookrightarrow}

\newcommand{\onto}{\twoheadrightarrow}

\title[Moduli of F.L. representations and a mod-\MakeLowercase{$p$} local-global compatibility result]{Moduli of Fontaine--Laffaille representations and a mod-\MakeLowercase{$p$} local-global compatibility result}

\author{Daniel Le}
\address{
Department of Mathematics,
Purdue University,
150 N. University Street,
West Lafayette, Indiana 47907, USA
}
\email{ledt@purdue.edu}

\author{Bao V.~Le Hung}
\address{Department of Mathematics,
Northwestern University,
2033 Sheridan Road,
Evanston, Illinois 60208, USA}
\email{lhvietbao@googlemail.com}

\author{Stefano Morra}
\address{Universit\'e Paris 8, Laboratoire d'Analyse, G\'eom\'etrie et Applications,  LAGA, Universit\'e Sorbonne Paris Nord, CNRS, UMR 7539,  F-93430, Villetaneuse, France}
\email{morra@math.univ-paris13.fr}

\author{Chol Park}
\address{Department of Mathematical Sciences, Ulsan National Institute of Science and Technology,
UNIST-gil 50, Ulsan 44919, South Korea}
\email{cholpark@unist.ac.kr}

\author{Zicheng Qian}
\address{Morningside Center of Mathematics, No.55, Zhongguancun East Road, Beijing, 100190, China }
\email{qianzicheng@amss.ac.cn}

\begin{document}

\begin{abstract}
Let $F/F^+$ be a CM field and let $\tld{v}$ be a finite unramified place of $F$ above the prime $p$.
Let $\rbar: \Gal(\ovl{\Q}/F)\rightarrow \GL_n(\ovl{\F}_p)$ be a continuous representation which we assume to be modular for a unitary group over $F^+$ which is compact at all real places.
We prove, under Taylor--Wiles hypotheses, that the smooth $\GL_n(F_{\tld{v}})$-action on the corresponding Hecke isotypical part of the mod-$p$ cohomology with infinite level above $\tld{v}|_{F^+}$ determines $\rbar|_{\Gal(\ovl{\Q}_p/F_{\tld{v}})}$, when this latter restriction is Fontaine--Laffaille and has a suitably generic semisimplification.
\end{abstract}
\maketitle
\tableofcontents

\clearpage{}%
\section{Introduction}

\subsection{Motivation and the main result}
Let $p$ be a prime number and let $K$ denote a finite extension of $\Qp$.
One formulation of a hypothetical mod-$p$ local Langlands correspondence is the existence of an injection from the set of $n$-dimensional representations of $\Gal(\ovl{\Q}_p/K)$ over $\ovl{\F}_p$ up to conjugation to the set of isomorphism classes of admissible representations of $\GL_n(K)$ over $\ovl{\F}_p$ which is compatible with global correspondences occurring in the mod-$p$ cohomology of locally symmetric spaces.
At present, the only known case of such a correspondence is when $K=\Qp$ and $n=2$ (see \cite{breuilI, breuilII} for the semisimple case).
In this case, a $p$-adic version which is compatible with deformations (see~\cite{colmez, paskunas-IHES, kisin-asterisque}) led to a proof of many cases of the Fontaine--Mazur conjecture \cite{kisin-fontaine-mazur,emerton-local-global}.
The current literature suggests that the situation is enormously more involved beyond the case of $\GL_2(\Qp)$.

We now introduce a global context for our discussions on mod-$p$ local-global compatibility alluded to earlier.
Let $F/F^+$ be a CM extension in which all $p$-adic places of $F^+$ split.
Let $G_{/F^+}$ be an outer form of $\GL_n$ which splits over $F$ and is definite at all infinite places.

We fix a place $v\mid p$ of $F^+$ with lift $\tld{v}$ in $F$ and let $K \defeq F_{\tld{v}}\cong F^{+}_{v}$. We denote by $\cO_K$ the ring of integers of $K$ and by $k$ its residue field. We fix an isomorphism $G_{/F} \cong \GL_{n/F}$ which identifies $G_{/F^{+}_{v}}$ with $\GL_{n/F_{\tld{v}}}$.
Let $U^v \leq G(\A_{F^+}^{\infty,v})$ be a compact open subgroup. We define the space of $\ovl{\F}_p$-valued algebraic automorphic forms on $G(F^+) \backslash G(\A_{F^+}^\infty)$ of level $U^v$ to be
$$S(U^v,\ovl{\F}_p)\defeq \left\{\textrm{continuous }f:\,G(F^{+})\backslash G(\bA^{\infty}_{F^{+}})/U^v \rightarrow \ovl{\F}_p\right\}.$$
This space carries commuting actions of a Hecke algebra $\bT$ and $G(F^+_v) \cong \GL_n(K)$.
To a continuous and essentially conjugate self-dual representation $\rbar: \Gal(\ovl{\Q}/F) \ra \GL_n(\ovl{\F}_p)$, one can attach a maximal ideal $\fm_{\rbar}=\fm_{\rbar,U^v} \subseteq \bT$ and set $\pi(\rbar)\defeq S(U^v,\ovl{\F}_p)[\fm_{\rbar}]$.
Mod-$p$ local-global compatibility asserts that 
\[
\underset{\substack{\ra\\U^v}}{\lim}\ S(U^v,\ovl{\F}_p)[\fm_{\rbar,U^v}]\cong \otimes'_{w\nmid \infty} \pi_w
\]
where $\pi_w$ is the hypothetical $G(F^+_w)$-representation corresponding to $\rbar|_{\Gal(\ovl{F}_{\tld{w}}/F_{\tld{w}})}$.
In particular, the $\GL_n(K)$-representation $\pi(\rbar)$ is expected to be isomorphic to $\pi_v\otimes(\pi^v)^{U^v}\cong \pi_v^{\oplus d_{U^v}}$ where $d_{U^v}$ is some positive integer.
Two natural questions arise:
\begin{enumerate}[label=(\Alph*)]
\item
\label{it:intro:LLC1}
Does $\rhobar$ determine the isomorphism class of the $\GL_n(K)$-representation $\pi(\rbar)$ (up to multiplicity)?
\item
\label{it:intro:LLC2}
Does the smooth $\GL_n(K)$-representation $\pi(\rbar)$ determine the conjugacy class of $\rhobar$?
\end{enumerate}
The limited evidence towards Question~\ref{it:intro:LLC1} at present comes mainly in the form of results towards the weight part of Serre's conjecture, i.e.~results on the $\GL_n(\cO_K)$-socle of $\pi(\rbar)$.
The main result of this paper affirmatively answers Question~\ref{it:intro:LLC2} in the generic Fontaine--Laffaille case under mild Taylor--Wiles hypotheses.

\begin{thm}[Corollary~\ref{thm:main:LGC:prime}]
\label{thm1:intro}
With the above setup, assume moreover that
\begin{enumerate}[label=(\roman*)]
\item
$K/\Q_p$ is unramified;
\item
\label{it1:thm1:intro}
$U^v$ is sufficiently small;
\item
\label{it1:thm1:intro}
$\rbar(\Gal(\ovl{\Q}/F(\zeta_p)))$ is adequate;
\item
\label{it2:thm1:intro}
$\Hom_{\ovl{\F}_p[\GL_n(\cO_{K})]}\big(V,\pi(\rbar)|_{\GL_n(\cO_{K})}\big)\neq 0$ for a Serre weight
$V$ which is $5n$-generic Fontaine--Laffaille (and in particular $\rbar$ is automorphic).
\end{enumerate}
Then the isomorphism class of the smooth $\GL_n(K)$-representation $\pi(\rbar)$ determines the isomorphism class of $\rhobar$.
\end{thm}

\begin{rmk}
\emph{A priori}, Question \ref{it:intro:LLC2} depends on the tame level $U^v$.
By decreasing $U^v$, one can make $d_{U^v}$ arbitrarily large.
On the other hand, one can sometimes introduce coefficient systems (minimally ramified types) to make the multiplicity $1$; see for example \cite[\S 7.3]{LLLM}.
\end{rmk}

We elaborate on the genericity condition appearing in item~\ref{it2:thm1:intro} above. By a \emph{Serre weight} we mean an irreducible smooth $\ovl{\F}_p$-representation of $\GL_n(\cO_{K})$.
The isomorphism class of $V$ is determined by
(an equivalence class of) a tuple $$\lambda=(\lambda_{j})_{1\leq j\leq [k:\Fp]}\in (\Z^n)^{[k:\Fp]}$$ (see (\ref{ex:bij:SW}) for example).
Then $V$ is \emph{$5n$-generic Fontaine--Laffaille} if for each $1\leq j\leq [k:\Fp]$ the tuple $\lambda_{j}=(\lambda_{j,i})_{1\leq i\leq n}\in \Z^n$ satisfies $\lambda_{j,1}-\lambda_{j,n}\leq p-6n$ and
\begin{equation*}
\lambda_{j,i}-\lambda_{j,i+1}\geq 5n
\end{equation*}
for all $1\leq i\leq n-1$.
This explicit and purely combinatorial genericity condition is of the same nature as the genericity conditions %
appearing in \cite{BD,HLM,LMP}, in contrast to the more elaborate conditions appearing in  \cite{MLM} (depending on an implicit polynomial $P\in\Z[X_1,\dots,X_n]$) or in \cite{PQ} (related to the geometry of the moduli spaces of Galois representations).

\subsection{Previous results}
\label{subsub:previous:results}

\cite{BD} considered the case of $\GL_2(\Q_{p^f})$, the first case of Theorem~\ref{thm1:intro} beyond $\GL_2(\Q_p)$. %
Let $K$ be $\Q_{p^f}$.
The subspace $\soc_{\GL_2(\cO_K)} \pi(\rbar)$ determines the restriction of $\rhobar\defeq \rbar|_{\Gal(\ovl{\Q}_p/K)}$ to the inertial subgroup if $\rhobar$ is irreducible, and the central character of $\pi(\rbar)$ determines the unramified twist.
The reducible case is more interesting.
Fixing the semisimplification of the restriction to inertia of a reducible $\rhobar$ i.e.~fixing the inertial weights, one can take unramified twists of the characters and vary the extension class.
When the ratio of the inertial weights is \emph{generic}, this gives a universal\footnote{we are ignoring a trivial $\bG_m$-action} algebraic family of $\rhobar$'s\footnote{really a family of $(\varphi,\Gamma)$-modules} over an $[\A^f/\bG_m]$-bundle over $\bG_m^2$.
The image of $\rhobar$ in $\bG_m^2(\ovl{\F}_p)$ is determined by Hecke eigenvalues, and so it remains to identify the corresponding point in $[\A^f/\bG_m]$.
The possible sets of Jordan--H\"older factors of $\soc_{\GL_2(\cO_K)} \pi(\rbar)$, i.e.~modular Serre weights, stratify $[\A^f/\bG_m]$ into $2^f$ strata defined by membership in the $f$ coordinate hyperplanes in $\A^f$ through the origin. (\cite{CEGM} gives a Galois-theoretic description of the stratification in terms of class field theory.)
Each stratum is isomorphic to $[\bG_m^r/\bG_m]$ for some $0\leq r\leq f$, and the $\ovl{\F}_p$-point corresponding to $\rhobar$ is determined by the action of ``normalized'' or ``hidden'' Hecke operators.
We delay the discussion of these normalized Hecke operators, the main conceptual breakthrough of \cite{BD}, until \S\,\ref{subsub:intro:lift}.

Next is the case of $\GL_3(\Q_p)$.
One has a stratification in terms of the socle and cosocle filtrations of $\rhobar$.
These filtrations are determined by the set of ``ordinary Serre weights" by \cite{gee-geraghty,MP}, and the semisimplification of $\rhobar$ is again determined by Hecke eigenvalues of these Serre weights.
Fixing the semisimplification, the only cases where a stratum has more than one point are in the maximally nonsplit niveau one case and two nonsplit niveau two cases considered in \cite{HLM} and \cite{LMP}, respectively.
There is an isomorphism\footnote{again ignoring a trivial $\bG_m$-action} of each of these strata to an open subscheme of $\bP^1$ given by the Fontaine--Laffaille invariant defined using a universal family of Fontaine--Laffaille modules.
In each of these $\bP^1$'s, there are two special points distinguished by the appearance of additional modular Serre weights, later termed \emph{extremal} Serre weights in \cite{OBW}.
Finally, the remaining points in $\bG_m(\ovl{\F}_p)$ are distinguished by a ``normalized Hecke operator", adapting the method of \cite{BD}.

In both the work on $\GL_3(K)$ with $K/\Q_p$ unramified \cite{EnnsGL3} and $\GL_n(\Q_p)$ \cite{PQ}, Theorem~\ref{thm1:intro} is established in a single open stratum containing only maximally nonsplit niveau one representations.
Already in these two contexts, there are many other interesting strata which Theorem~\ref{thm1:intro} includes.
Our proof unifies and generalizes the above approaches.

\begin{rmk}
In \cite{Scholze}, Scholze introduced a completely different approach to Question~\ref{it:intro:LLC2} using the $p$-adic cohomology of Lubin--Tate spaces and affirmatively answers this question for $n=2$. Recently, Liu has generalized some of the results of \cite{Scholze}, however there are essential difficulties to resolving Question~\ref{it:intro:LLC2} for $n>2$ using Scholze's approach, cf.~\cite{Kegang}.
\end{rmk}

\subsection{Main ingredients in the proof}\label{sub:ingredient}
We first introduce a moduli stack $\mathrm{FL}$ of Fontaine--Laffaille modules of a fixed generic weight which is a smooth modification of the flag variety for $(\Res_{\F_q/\F_p} \GL_n)_{\ovl{\F}_p}$.
If $P$ is a parabolic subgroup of $\GL_n$ with Levi quotient $M$ and $\rhobar: \Gal(\ovl{\Q}_p/K) \ra P(\ovl{\F}_p)$ is a Galois representation, we let $\rhobar^{M\textrm{-}\semis}$ be the composition of $\rhobar$ and the projection to $M(\ovl{\F}_p)$.
Fixing a parabolic subgroup $P$ of $\GL_n$ with Levi quotient $M$ and a tame inertial $\F$-type $\ovl{\tau}: I_K \ra M(\ovl{\F}_p)$ which extends to an irreducible representation $\Gal(\ovl{\Q}_p/K) \ra M(\ovl{\F}_p)$ lying in $\mathrm{FL}$, there is a locally closed substack $\mathrm{FL}_{\taubar,P}$ of $\mathrm{FL}$ whose $\ovl{\F}_p$-points correspond to $\rhobar$'s (up to conjugacy) which factor through $P$ and satisfy $\rhobar^{M\textrm{-}\semis}|_{I_K} \cong \taubar$.
Then $\mathrm{FL}$ can be written as a union of substacks of the form $\mathrm{FL}_{\taubar,P}$.
Now $\mathrm{FL}_{\taubar,P}$ is an $[N/Z(M)]$-bundle over $Z(M)$, where $N$ is a closed subgroup scheme of the unipotent radical of $(\Res_{\F_q/\F_p} P)_{/\ovl{\F}_p}$ and $Z(M)$ is the center of $M$.
The point in $Z(M)(\ovl{\F}_p)$ corresponding to $\rhobar$ is again given by Hecke eigenvalues.
The space $N$ can be partitioned based on vanishing of certain root entries.
A crucial fact is that the partition of $\mathrm{FL}$ coming from intersections of translated Bruhat cells in the flag variety for $(\Res_{\F_q/\F_p} \GL_n)_{\ovl{\F}_p}$ is a refinement of the partition described above.
Moreover, this Bruhat partition can be understood in terms of \emph{extremal weights} introduced in \cite{OBW}, and in particular can be understood in terms of $\soc_{\GL_n(\cO_K)} \pi(\rbar)$. Finally, we show that the remaining parts are distinguished by ``normalized Hecke operators"---this is the main challenge in our work (see \S\,\ref{subsub:intro:lift} and \S\,\ref{subsub:intro:separate pt}).

\vspace{2mm}

We now elaborate on this description. %
For simplicity, we assume that $K=\Qp$ for the rest of \S\,\ref{sub:ingredient}.
We also fix as coefficients a sufficiently large finite extension $E$ of $\Qp$, with ring of integers $\cO$ and residual field $\F$.

\subsubsection{Moduli of Fontaine--Laffaille modules and the niveau partition}\label{subsub:intro:FL moduli}
Fontaine--Laffaille theory \cite{Fontaine-Laffaille} gives a linear algebraic description of a large class of mod-$p$ representations of $\Gal(\ovl{\Q}_p/\Qp)$. This allows for a group theoretic description of the moduli of these Galois representations.

We fix a tuple $\lambda=(\lambda_1,\dots,\lambda_n)\in\Z^n$ satisfying $\lambda_1\geq \cdots\geq\lambda_n$, $\lambda_1-\lambda_n<p-n$ and let $\eta\defeq (n-1,n-2,\dots,1,0)\in \Z^n$.
Let $\mathrm{FL}_n^{\lambda+\eta}$ be the moduli stack  of rank $n$ Fontaine--Laffaille modules of weight  $\lambda+\eta$ (denoted $\mathrm{FL}$ above).
This is represented by the quotient of the basic (quasi-)affine $U\backslash\GL_n$ by the $T$-conjugation action (where $U,\ T$ denote respectively the subgroup of upper-triangular unipotent matrices and of diagonal matrices in $\GL_n$, %
see Proposition~\ref{thm:repr:FLgp}).
A Galois representation $\rhobar:\Gal(\ovl{\Q}_p/\Qp)\ra\GL_n(\F)$ is \emph{Fontaine--Laffaille} if it arises from a Fontaine--Laffaille module over $\F$ via the Fontaine--Laffaille functor constructed in \cite{Fontaine-Laffaille} (more precisely, after suitable twist).
Hence, Fontaine--Laffaille mod-$p$ local Galois representations can be studied through the geometry of $U\backslash\GL_n$ together with the $T$-conjugation action on it.
For each $x\in U\backslash\GL_n(\F)$, we write $\rhobar_{x,\lambda+\eta}$ for the (isomorphism class of the) mod-$p$ local Galois representation attached to $x$ via the Fontaine--Laffaille functor.

Let $W$ be the Weyl group of $\GL_n$ which we identify with the group of permutations of the set $\{1,\cdots,n\}$.
Using base change of Fontaine--Laffaille modules, we show that for each point $x\in U\backslash\GL_n(\F)$, $\rhobar_{x,\lambda+\eta}$ is semisimple if and only if $x$ lies in the schematic image of $Tw$ under $\GL_n\twoheadrightarrow U\backslash\GL_n$ for some permutation $w\in W$ uniquely determined by $x$ (see Lemma~\ref{lem: class of irr}). Moreover, there exists a semisimple representation $\taubar(w^{-1},\lambda+\eta):I_{\Qp}\rightarrow\GL_n(\F)$ such that $\rhobar_{x,\lambda+\eta}|_{I_{\Q_p}}\cong \taubar(w^{-1},\lambda+\eta)$ for all $\F$-points $x$ lying in the schematic image of $Tw$. In other words, the \emph{semisimple locus} in $U\backslash\GL_n$ is exactly given by the disjoint union of schematic images of $Tw$ in $U\backslash\GL_n$, for all permutations $w\in W$. Motivated by the classification of mod-$p$ local Galois representations by their semisimplifications, we introduce a partition $\{\cN_w\}_{w\in W}$ on $U\backslash\GL_n$ (the \emph{niveau partition}) characterized by the fact that a point $x\in U\backslash\GL_n(\F)$ lies in $\cN_w(\F)$ if and only if
$\rhobar_{x,\lambda+\eta}^{\semis}|_{I_{\Qp}}\cong \taubar(w^{-1},\lambda+\eta)$.
The points $x\in\cN_w$ admit the following characterization: the closure of the orbit of $x$ under $T$-conjugation intersects the schematic image of $Tw$ in $U\backslash\GL_n$.
This characterization comes from a geometric interpretation of taking semisimplification of a mod-$p$ local Galois representation.

The part $\cN_w$ need not be irreducible, a reflection of the fact that Galois representations with fixed semisimplification can have different submodule structures.
Indeed, $\cN_w$ can be written as a union of irreducible locally closed subschemes $\cN_{w,P}$ (whose quotient by $T$-conjugation is the stack $\mathrm{FL}_{\taubar(w^{-1},\lambda+\eta),P}$ mentioned above) where $P$ is a minimal possible parabolic containing $Tw$.

The partition $\{\cN_w\}_{w\in W}$ is still too coarse for our purpose, and we will introduce two other partitions of $U\backslash\GL_n$ in \S\,\ref{subsub:intro:SW} and \S\,\ref{subsub:intro:separate pt} which refine it.

\subsubsection{Serre weights and the partition $\cP$}\label{subsub:intro:SW}

Let $\cX_{n}$ denote the Emerton--Gee stack of rank $n$ projective \'etale $(\varphi,\Gamma)$ modules (\cite[Theorem 1.2.1]{EGstack}) and $\cX_{n,\mathrm{red}}$ its reduced substack which is an algebraic stack over $\F_p$.
We recall some properties of the Emerton--Gee stack:
\begin{enumerate}
\item there is a natural bijection between $|\cX_{n}(\F)|=|\cX_{n,\mathrm{red}}(\F)|$ and the set of isomorphism classes of $\rhobar:\Gal(\ovl{\Q}_p/\Qp)\ra\GL_n(\F)$;
\item
\label{it:properties:2}
there is a natural bijection between the set of irreducible components of $\cX_{n,\mathrm{red}}$ and the set of Serre weights (namely, absolutely irreducible $\F$-representations of $\GL_n(\F_p))$;
\item $\mathrm{FL}_n^{\lambda+\eta}$ is identified with the irreducible component of $\cX_{n,\mathrm{red}}$ corresponding, via the bijection in item~(\ref{it:properties:2}) just above (normalized as in \cite[\S\,7.4]{MLM}), to the Serre weight $F(\lambda)$.
\end{enumerate}
Here $F(\lambda)$ is the absolutely irreducible $\F$-representation of $\GL_n(\F_p)$ of highest weight $\lambda$. Let $\omega:\Q_p^\times\rightarrow\F^\times$ be the character corresponding to the mod-$p$ cyclotomic character.
Under mild technical assumptions, the fact that $F(\lambda)\otimes_{\F}(\omega^{n-1}\circ\det)$ embeds into $\pi(\rbar)|_{\GL_n(\Zp)}$ (cf.~item~\ref{it2:thm1:intro} in Theorem~\ref{thm1:intro})
ensures that $\rhobar=\rbar|_{\Gal(\ovl{\Q}_p/F_{\widetilde{v}})}\in |\mathrm{FL}_n^{\lambda+\eta}(\F)|$.
Given a semisimple, suitably generic, Galois representation $\rhobar':\Gal(\ovl{\Q}_p/\Q_p)\ra\GL_n(\F)$ we have a set of Serre weights $W^?(\rhobar')$, containing a subset $W_{\mathrm{extr}}(\rhobar')$ consisting of \emph{extremal Serre weights} (see \cite[\S\,9.2 and \S\,9.3]{GHS}).
We consider the following set of Serre weights
\begin{equation}\label{equ:intro:relevant SW}
\underset{\rhobar\in |\mathrm{FL}_n^{\lambda+\eta}(\F)|}{\bigcup}W_{\mathrm{extr}}(\rhobar^{\rm{ss}})
\end{equation}
as well as the set of irreducible components of $\cX_{n,\mathrm{red}}$ corresponding to it.
Such irreducible components have the property that their intersections with $\mathrm{FL}_n^{\lambda+\eta}$ correspond to right translates of Schubert varieties (in $B\backslash \GL_n$ with $B$ the upper-triangular Borel subgroup) under the local model diagram
\begin{equation}
\label{eq:LMD}
\xymatrix{
&U\backslash \GL_n\ar[dr]\ar[dl]&
\\
\mathrm{FL}_n^{\lambda+\eta}&&B\backslash \GL_n
}
\end{equation}
where the left arrow is the quotient by $T$-conjugation and the right arrow is the quotient by left $T$-multiplication.

Motivated by this, we consider the partition on $\mathrm{FL}_n^{\lambda+\eta}$ induced, by intersection, from the irreducible components of $\cX_{n,\mathrm{red}}$ indexed by~(\ref{equ:intro:relevant SW}).
This partition lifts to the partition $\cP$ on $U\backslash\GL_n$ defined as the coarsest common refinement of the stratifications $\{U\backslash UwBu\}_{w\in W}$ for each $u \in W$.
In particular, each $\cC\in\cP$ is stable under both the left $T$-multiplication action and the $T$-conjugation action.
For each $\rhobar\in|\mathrm{FL}_n^{\lambda+\eta}(\F)|$, we can associate a unique $\cC\in\cP$ such that $\rhobar_{x,\lambda+\eta}\cong\rhobar$ for some closed point $x\in\cC(\F)$. %
For $n=2$ and $3$, this partition recovers the strata described in \S\,\ref{subsub:previous:results}.

Given $\cC\in\cP$ and $x\in\cC(\F)$, we say that $\taubar(w^{-1},\lambda+\eta)$ is a \emph{specialization} of $\rhobar_{x,\lambda+\eta}$ if the schematic image of $Tw$ in $U\backslash\GL_n$ lies in the Zariski closure of $\cC$ (see the paragraph before Lemma~\ref{lem: specialization and open cell} for a different definition). This provides us with an important characterization of $\cP$: given two points $x,x'\in U\backslash\GL_n(\F)$, $\rhobar_{x,\lambda+\eta}$ and $\rhobar_{x',\lambda+\eta}$ share the same set of specializations if and only if $x,x'\in\cC(\F)$ for some $\cC\in\cP$ (see Theorem~\ref{thm: obv wt and partition}).
The results of \cite{OBW} prove that set of specializations can be detected from the set of (generalized) \emph{extremal Serre weights} and that these weights are modular.
We obtain the following.

\begin{thm}[cf.~Lemma~\ref{lem:glob:algorithm}]\label{thm3:SW}
Let $\soc_{\GL_n(\Zp)}\left(\pi(\rbar)\right)$ be the maximal semisimple subrepresentation of $\pi(\rbar)|_{\GL_n(\Zp)}$. Then the isomorphism class of the $\GL_n(\Zp)$-representation $\soc_{\GL_n(\Zp)}\left(\pi(\rbar)\right)$ determines the element $\cC\in\cP$ associated with $\rhobar$.
\end{thm}
\noindent When $n=2$ or $3$, this result directly generalizes the analysis in \cite{BD,HLM,LMP} (see \S\,\ref{subsub:previous:results}).
Note that every modular Serre weight is extremal when $n=2$.

\subsubsection{Relevant types and invariant functions}\label{subsub:intro:lift}
In this and the next subsection we address the main technical difficulty in our work, i.e.~how the geometric points of $\cC\in\cP$ are distinguished by a set of $T$-invariant functions, which can be interpreted as ``normalized Hecke eigenvalues".
Local-global compatibility dictates that these are equal to normalized Frobenius eigenvalues of Weil--Deligne representations.

Let $\tau:I_{\Qp}\rightarrow\GL_n(E)$ be a tame inertial type, i.e.~a smooth semisimple representation of~$I_{\Qp}$ which extends to $W_{\Qp}$.
We assume throughout that $\tau$ is multiplicity free as $I_{\Qp}$-representation.
Let $\tau_1\subseteq\tau$ a sub inertial type.
For each Weil--Deligne representation $\varsigma: W_{\Qp}\rightarrow\GL_n(E)$ satisfying $\varsigma|_{I_{\Qp}}\cong \tau$, there exists a unique sub $W_{\Qp}$-representation $\varsigma_1\subseteq \varsigma$ such that $\varsigma_1|_{I_{\Qp}}\cong \tau_1$.
Then $\wedge^{\Dim \varsigma_1}\varsigma_1$ is a character of $W_{\Qp}$ on which the geometric Frobenius element corresponding to $p$ (via local class field theory) acts by a scalar, which we denote by $\phi_{\tau,\tau_1}(\varsigma)$.
We thus obtain a function $\phi_{\tau,\tau_1}$ on the moduli stack $\mathrm{WD}_\tau$ of Weil--Deligne representations with inertial type $\tau$.
We remark that there exists an isomorphism between $\mathrm{WD}_\tau$ and a split torus $\bG_m^r$ such that $\phi_{\tau,\tau_1}$ is the product of the first $r_1$ projections to $\bG_m$, where $r_1\leq r$ are the numbers of irreducible sub inertial types of $\tau_1$ and $\tau$ respectively.

Our goal is to compute the mod-$p$ reduction of a normalization of $\phi_{\tau,\tau_1}$ as a rational function on $\mathrm{FL}_n^{\lambda+\eta}$ (for various tame inertial types $\tau$).
To do this, we compare $\mathrm{FL}_n^{\lambda+\eta}$ and $\mathrm{WD}_\tau$ inside the $p$-adic formal stack $\cX^\tau$ of potentially crystalline representations with inertial type $\tau$ and parallel Hodge--Tate weights $n-1,\dots,1,0$.
The reduced special fiber $\cX^\tau_{\mathrm{red}}$ of $\cX^\tau$ is a topological union of irreducible components of $\cX_{n,\mathrm{red}}$ (see \cite[\S 8.1]{EGstack}) which we choose to include $\mathrm{FL}_n^{\lambda+\eta}$.
There is a natural morphism from the rigid generic fiber $\cX^{\tau,\rm{rig}}$ of $\cX^\tau$ towards the rigid-analytification $\mathrm{WD}_\tau^{\rm{rig}}$ of $\mathrm{WD}_\tau$. %
We may pull back the function $\phi_{\tau,\tau_1}$ from $\mathrm{WD}_\tau^{\rm{rig}}$ to $\cX^{\tau,\rm{rig}}$. In general, one can use this to produce functions on parts of the special fiber $\overline{\cX}^\tau$ of $\cX^\tau$ as follows: for example, if on the tube of a locally closed substack $\cY$ of $\overline{\cX}^\tau$, the $p$-adic valuation of $\phi_{\tau,\tau_1}$ is bounded below by $d$, then one gets a function $p^{-d}\phi_{\tau,\tau_1}$ on $\cY$.
In general it is hard to analyze the kind of functions one gets this way.
However it turns out we can solve this problem when
$\tau$ is \emph{$F(\lambda)$-relevant}, namely chosen from a set of special tame inertial types $\{\tau_w\}_{w\in W}$ for which the Serre weight $F(\lambda)$ is an ``outer'' Jordan--H\"older factor in the mod-$p$ reduction of the $\GL_n(\Fp)$-representation attached to $\tau$ via the inertial local Langlands correspondence.
(The set of outer Jordan--H\"older factor for such representations is introduced in \cite[\S 2.3.1]{MLM}.)
For each $w\in W$, we have the following diagram
\begin{equation}\label{equ:intro:lift:inv}
\xymatrix{
U\backslash UTw_0Uw_0w\ar[d] \ar@{^{(}->}[rr] &&\tld{\cX}^{\tau_w,\circ}\ar[d] &&\tld{\cX}^{\tau_w,\rm{rig},\circ}\ar[d]\ar@{_{(}->}[ll]&&\\
\mathrm{FL}_n^{\lambda+\eta} \ar@{^{(}->}[rr]&&\cX^{\tau_w}&&\cX^{\tau_w,\rm{rig}}\ar@{_{(}->}[ll]\ar@{->>}[rr] && \mathrm{WD}_{\tau_w}^{\rm{rig}}
}
\end{equation}
where $w_0\in W$ is the longest element and $\tld{\cX}^{\tau_w,\circ}$ is a $p$-adic formal scheme whose special fiber is identified with $U\backslash UTw_0Uw_0w$. The first vertical map is given by the composition $U\backslash UTw_0Uw_0w\hookrightarrow U\backslash\GL_n\twoheadrightarrow \mathrm{FL}_n^{\lambda+\eta}$, and thus $U\backslash UTw_0Uw_0w$ is a $T$-torsor over some open substack of $\mathrm{FL}_n^{\lambda+\eta}$.
The second vertical map is a $T_\cO^{\wedge_p}$-torsor followed by an open immersion.
The third vertical map is induced from the second by taking rigid analytic fiber.

Given a sub inertial type $\tau_{w,1}\subseteq\tau_w$, there exists $d\in\Z$ such that $p^{-d}\phi_{\tau_w,\tau_{w,1}}$, after pulling back to $\tld{\cX}^{\tau_w,\rm{rig},\circ}$ via~(\ref{equ:intro:lift:inv}), extends to an invertible function on $\tld{\cX}^{\tau_w,\circ}$, which specializes to an invertible function on $U\backslash UTw_0Uw_0w$.
Explicitly, we can attach to $\tau_{w,1}$ a subset $I \subset \{1,\ldots,n\}$ such that $w(I) = I$ and the above invertible function is given by
$$f_{w,I}: U\backslash UTw_0Uw_0w\cong Tw_0Uw_0w\twoheadrightarrow T\xrightarrow{\prod_{k\in I}\epsilon_k}\bG_m$$
where $\epsilon_k:T\twoheadrightarrow \bG_m$ is the projection to $k$-th diagonal entry. The condition $w(I)=I$ ensures that the rational function $f_{w,I}$ on $U\backslash\GL_n$ is invariant under $T$-conjugation, and thus descends to $\mathrm{FL}_n^{\lambda+\eta}$.

Let $w\in W$, $x\in U\backslash UTw_0Uw_0w(\F)$ and $\rho:\Gal(\ovl{\Q}_p/\Qp)\ra\GL_n(\cO)$ be a potentially crystalline lift of $\rhobar_{x,\lambda+\eta}$ with type $\tau_w$ and Hodge--Tate weights $n-1,\dots,1,0$. Let $\mathrm{WD}^\ast(\rho)$ be the dual of Weil--Deligne representation associated with $\rho$, where $\rho\mapsto \mathrm{WD}(\rho)$ is the covariant functor to Weil--Deligne representations of \cite[Appendix B]{CDT}.
We have the following
\begin{thm}[Theorem~\ref{thm:Feval-2}]\label{thm4:inv lift}
There exists an integer $d\in\Z$ depending only on $\lambda$, $\tau_w$ and $\tau_{w,1}$, such that for any potentially crystalline lift  $\rho:\Gal(\ovl{\Q}_p/\Qp)\ra\GL_n(\cO)$ of $\rhobar_{x,\lambda+\eta}$, with type $\tau_w$ and Hodge--Tate weights $n-1,\dots,1,0$, we have
\begin{itemize}
\item $\phi_{\tau_w,\tau_{w,1}}(\mathrm{WD}^\ast(\rho))\in p^d\cO^\times$; and
\item the image of $p^{-d}\phi_{\tau_w,\tau_{w,1}}(\mathrm{WD}^\ast(\rho))$ under $\cO^\times\twoheadrightarrow \F^\times$ is given by $f_{w,I}(x)$.
\end{itemize}
\end{thm}
Motivated by this result, we define the set of \emph{invariant functions} as
$$\Inv\defeq \{f_{w,I}\mid w\in W,\,I\subseteq\{1,\dots,n\},\,w(I)=I\}.$$
These are rational functions on $U\backslash\GL_n$ which will be studied in the following paragraph.

\subsubsection{Invariant functions distinguish $T$-conjugacy classes}\label{subsub:intro:separate pt}
Let $\cC\in\cP$ and write $\Inv(\cC)\subseteq\Inv$ for the subset of invariant functions whose zero and pole divisors are away from $\cC$.
The central technical result we prove is the following.
\begin{thm}[Corollary~\ref{cor: separate points}]\label{thm5:inv fun prime}
The set $\Inv(\cC)$ distinguishes $T$-conjugacy classes in~$\cC$. In other words, for any Noetherian $\F$-algebra $R$ the elements $x,x'\in\cC(R)$ lie in the same $T$-conjugacy class if and only if $g(x)=g(x')\in R^\times$ for all $g\in\Inv(\cC)$.
\end{thm}
\noindent In \S\,\ref{subsub:intro:Hecke}, we describe how to determine $g(x)$ for all $g \in \Inv(\cC)$ from $\pi(\rbar)$ if $\rhobar_{x,\lambda+\eta} \cong \rbar|_{G_K}$.

The first difficulty in proving Theorem~\ref{thm5:inv fun prime} is that the affine schemes $\cC\in\cP$ are genuinely complicated.
They are in general far from being affine spaces or split tori and \emph{a priori} may not be irreducible.
For example, the unique open $\cC\in\cP$ when $n=3$ is isomorphic to a $\bG_m^2\times (\bP^1 \setminus \{0,1,\infty\})$-bundle over $\bG_m^3$ (see \cite{HLM}).
Instead, we would like to find a partition coarser than $\cP$ which is well-suited for explicit computations of the invariant functions.

Recall the partition $\{\cN_w\}_{w\in W}$ from \S\,\ref{subsub:intro:FL moduli}. For simplicity, we restrict ourselves here to those $\cN_w$ which are irreducible.
We fix a tuple of positive integers $(n_1,\dots,n_r)$ satisfying $\sum_{m=1}^r n_m=n$, and consider the standard Levi subgroup $M\subseteq\GL_n$ given by the diagonal blocks $\GL_{n_1}\times\cdots\times\GL_{n_r}$. We consider the parabolic subgroup $w_0Bw_0M\subseteq\GL_n$ and write $N^-$ for its unipotent radical. The unipotent group $N^-$ corresponds to a set of negative roots $\Phi_{N^-}$ (with respect to the Borel $B$).
The Weyl group of $M$ is isomorphic to a product (indexed by the Levi blocks of $M$) of permutation groups and we choose an element $w$ of it such that each of its factors is transitive.
Then our construction of $\cN_w$ in \S\,\ref{sub:niveau:elements} implies that the quotient map $\GL_n\twoheadrightarrow U\backslash\GL_n$ induces an isomorphism
\begin{equation}\label{equ:intro:matrix lift}
TN^-w\xrightarrow{\sim}\cN_w.
\end{equation}
In particular, we see that $\cN_w$ is the product of a torus $T$ with the affine space $N^-$.
The isomorphism~(\ref{equ:intro:matrix lift}) gives us standard coordinates to do computation on $\cN_w$, by writing $D_w^k:\cN_w\rightarrow\bG_m$ (resp.~$u_w^\al:\cN_w\rightarrow\bG_a$) for the morphism induced from extracting the $k$-th diagonal entry of $T$ (resp.~extracting the $\al$-th entry of $N^-$) for each $1\leq k\leq n$ (resp.~for each $\al\in\Phi_{N^-}$).
(Our notation here is slightly different from that of \S\,\ref{sub:std:note} for convenience.)

Note that the quotient of $\cN_w$ by $T$-conjugation is still a genuine stack in general.
To prove Theorem~\ref{thm5:inv fun prime}, we would like to have a $T$-conjugation stable partition of $\cN_w$ (coarser than $\cP$), such that the quotient by $T$-conjugation of each element (of the partition) is an affine scheme with simple coordinates.
A natural candidate is given by $\{\cN_{w,\Lambda}\}_{\Lambda\subseteq\Phi_{N^-}}$ where
$$\cN_{w,\Lambda}(R)\defeq \{A\in \cN_w(R)\mid \hbox{$u_w^\al(A)\neq0$ for each $\al\in\Lambda$ and $u_w^\al(A)=0$ for each $\al\in\Phi_{N^-}\setminus\Lambda$}\}.$$
For each $\Lambda\subseteq\Phi_{N^-}$, $\cN_{w,\Lambda}$ is a split torus of rank $n+\#\Lambda$ with ring of global sections given by the following ring of Laurent polynomials
$$\cO(\cN_{w,\Lambda})=\F[(D_w^k)^{\pm1},\,(u_w^\al)^{\pm1}\mid 1\leq k\leq n,\,\al\in\Lambda]$$
The quotient of $\cN_{w,\Lambda}$ by $T$-conjugation, written $\cN_{w,\Lambda}\slash{{\sim}_{T\text{-\textnormal{cnj}}}}$, does exist as an affine scheme whose ring of global sections is given by the following invariant subring
\begin{equation*}
\cO(\cN_{w,\Lambda}\slash{{\sim}_{T\text{-\textnormal{cnj}}}})=\F[(D_w^k)^{\pm1},\,(u_w^\al)^{\pm1}\mid 1\leq k\leq n,\,\al\in\Lambda]^{T\text{-\textnormal{cnj}}}.
\end{equation*}
Note that $\cN_{w,\Lambda}$ is a topological union of elements in the partition $\cP$ (see Lemma~\ref{lem:union:partition}).

Now we consider an element $\cC$ in $\cP$ with $\cC\subseteq\cN_{w,\Lambda}$. In order to prove Theorem~\ref{thm5:inv fun prime}, we study how to generate the ring $\cO(\cN_{w,\Lambda}\slash{{\sim}_{T\text{-\textnormal{cnj}}}})|_\cC$ from the set
\begin{equation*}
\{(g|_\cC)^{\pm 1}\mid g\in\Inv(\cC)\}.
\end{equation*}
A key observation is that we can find a set of \emph{special units} (cf.~Definition~\ref{def: constructible lifts}) in the ring $\cO(\cN_{w,\Lambda}\slash{{\sim}_{T\text{-\textnormal{cnj}}}})$ that satisfy the following:
\begin{itemize}
\item the set of special units generates the ring $\cO(\cN_{w,\Lambda}\slash{{\sim}_{T\text{-\textnormal{cnj}}}})$;
\item the restriction to $\cC$ of each special unit can be generated from the set $\{(g|_\cC)^{\pm 1}\mid g\in\Inv(\cC)\}$.
\end{itemize}
The construction of these special units is combinatorial and the technical heart of this paper (see \S\,\ref{sub: examples} and the examples below).
These conditions together suffice to prove Theorem~\ref{thm5:inv fun prime} in these cases.
When $\cN_w$ is not irreducible, essentially the same idea works, except that we need to treat different irreducible components of $\cN_w$ and define an analogue of $\cN_{w,\Lambda}$.

We give some examples below to illustrate some of the difficulties.
\begin{exam}\label{exam:intro:open}
Let $w = 1$, $\Lambda=\Phi_{N^-}$ be the set of all negative roots and $\cC$ be an element of $\cP$ satisfying $\cC\subseteq\cN_{1,\Phi_{N^-}}$. If we let
\[
F_1^{(i,j)} \defeq \frac{u_1^{(i,j)}}{u_1^{(i,i-1)}u_1^{(i-1,j)}}
\]
for each $1\leq j<j+1<i\leq n$,
then we have
\[
\cN_{1,\Phi_{N^-}} \cong \Spec \F[(D_1^k)^{\pm1}\mid 1\leq k\leq n][(u_1^\al)^{\pm1}\mid \al\in\Phi_{N^-}]
\]
and
\[
\cN_{1,\Phi_{N^-}}\slash{{\sim}_{T\text{-\textnormal{cnj}}}} \cong \Spec \F[(D_1^k)^{\pm1}\mid 1\leq k\leq n][(F_1^{(i,j)})^{\pm1}\mid 1\leq j<j+1<i\leq n].
\]
Since $f_{1,\{k\}} = D_1^k$ for each $1\leq k\leq n$, it suffices to recover $F_1^{(i,j)}$ from $f_{s,I}$ for various choices of $(s,I)$. For each $1\leq j<j+1<i\leq n$, we have the following two possibilities.
\begin{itemize}
\item $u_1^{(i,j)}-u_1^{(i,i-1)}u_1^{(i-1,j)}=0$ on $\cC$, and thus $F_1^{(i,j)}|_{\cC}=1$;
\item $u_1^{(i,j)}-u_1^{(i,i-1)}u_1^{(i-1,j)}\neq 0$ on $\cC$, in which case we set $s^{(i,j)}\defeq (i\,i-2\cdots j+1\,j)$ and observe that $f_{s^{(i,j)},\{i-1\}}\in\Inv(\cC)$ with $f_{s^{(i,j)},\{i-1\}}|_\cC=D_1^{i-1}|_\cC(1-F_1^{(i,j)}|_\cC^{-1})$.
\end{itemize}
This is clearly sufficient to prove Theorem~\ref{thm5:inv fun prime} for each $\cC\subseteq \cN_{1,\Phi_{N^-}}$. Note that we actually have $f_{s^{(i,j)},\{i-1\}}=f_{(i\,j),\{i-1\}}$, but the permutation $s^{(i,j)}$ is better than $(i\,j)$ because we have $\cC\subseteq U\backslash UTw_0Uw_0s^{(i,j)}$ (but not necessarily $\cC\subseteq U\backslash UTw_0Uw_0(i\,j)$), as long as $u_1^{(i,j)}-u_1^{(i,i-1)}u_1^{(i-1,j)}\neq 0$ on $\cC$. In \cite{PQ}, similar results were obtained under some unnatural open conditions on $\mathrm{FL}_n^{\lambda+\eta}$ which cannot be detected from $\pi(\rbar)$ in terms of Serre weights, whereas here $u_1^{(i,j)}-u_1^{(i,i-1)}u_1^{(i-1,j)}$ is either zero or invertible on $\cC$ for each $\cC\subseteq \cN_{1,\Phi_{N^-}}$.
\end{exam}

\begin{exam}
Let $n=4$, $w = 1$ (this implies $M=T$, $r=n=4$, $N^-=w_0Uw_0$ and $\Phi_{N^-}$ is the set of all negative roots), and $\Lambda = \{(i,j)\mid i\geq 3, j \leq 2\}$.
If we let
\[
u \defeq \frac{u_1^{(3,1)}u_1^{(4,2)}}{u_1^{(3,2)}u_1^{(4,1)}},
\]
then we have
\[
\cN_{1,\Lambda} \cong \Spec \F[(D_1^k)^{\pm1}\mid 1\leq k\leq 4][(u_1^\al)^{\pm1}\mid \al\in\Lambda]
\]
and
\[
\cN_{1,\Lambda}\slash{{\sim}_{T\text{-\textnormal{cnj}}}} \cong \Spec \F[(D_1^k)^{\pm1}\mid 1\leq k\leq 4][u^{\pm1}].
\]
Depending on the vanishing of the minor $\delta \defeq u_1^{(3,1)}u_1^{(4,2)}-u_1^{(3,2)}u_1^{(4,1)}$,
we can decompose $\cN_{1,\Lambda}$ into two $T$-stable pieces each of which is an element in $\cP$.
Since $f_{1,\{k\}} = D_1^k$ for each $1\leq k\leq 4$, we see that Theorem~\ref{thm5:inv fun prime} holds for $\cC$ on which $\delta = 0$ (so $u=1$ on $\cC$) by considering the subset $\{ f_{1,\{k\}} \mid 1\leq k\leq 4 \} \subseteq \Inv(\cC)$.
Otherwise if $u\neq 1$ on $\cC$, we also consider $f_{(13)(24),\{1,3\}}|_\cC = \left(D_1^1 D_1^3 \frac{1-u}{u}\right)|_\cC$ to see that Theorem~\ref{thm5:inv fun prime} holds.

We now let $n$ be arbitrary, but still take $w$ to be $1$.
When $\Lambda$ is rather smaller than $\Phi_{N^-}$, it seems to be challenging to find elements of $\Inv(\cC)$ not in the subalgebra generated by $\{(D_1^k)^{\pm1}\mid 1\leq k\leq n\}$ (if they exist).
In the above example, for any $f_{s,I} \in \Inv(\cC)$ with $s\neq (13)(24)$, $f_{s,I}|_\cC \in \F[(D_1^k|_\cC)^{\pm1}\mid 1\leq k\leq 4]$.
Furthermore, data we have accumulated suggests that this behavior is typical.
So it appears that the combinatorics involved in Theorem~\ref{thm5:inv fun prime} are rather delicate.
Moreover, since it is not sufficient to consider only invariant functions of the form $f_{1,\{k\}}$, it is essential that we use tame types that are not principal series in contrast to \cite{BD,HLM,LMP,EnnsGL3,PQ} (see \S\,\ref{subsub:intro:Hecke}).
\end{exam}

\subsubsection{Mod-$p$ reduction of normalized Hecke eigenvalues}\label{subsub:intro:Hecke}

We studied in \S\,\ref{subsub:intro:lift} how to lift invariant functions (which live in characteristic $p$) to Frobenius eigenvalues of Weil--Deligne representations (which live in characteristic $0$).
Using classical local-global compatibility, we interpret invariant functions as the mod-$p$ reduction of normalized eigenvalues of Hecke operators acting on a space of automorphic forms in characteristic $0$.
We now explain how to extract the mod-$p$ reduction of normalized eigenvalues from the $\GL_n(\Qp)$-representation $\pi(\rbar)$ adapting the method of \cite{BD}.

Let $\mathbf{K}\defeq \GL_n(\Zp)$ be the standard compact open subgroup of $\GL_n(\Qp)$.
Let $\theta$ be an irreducible $E$-representation of $\mathbf{K}$ equipped with a $\mathbf{K}$-stable $\cO$-lattice $\theta^\circ\subseteq\theta$. We consider the compact induction $\ind_{\mathbf{K}}^{\GL_n(\Qp)}\theta^\circ$ and write $\cH_{\mathbf{K}}^{\GL_n(\Qp)}(\theta^\circ)\defeq \End_{\GL_n(\Qp)}(\ind_{\mathbf{K}}^{\GL_n(\Qp)}\theta^\circ)$ for the Hecke algebra associated to it.
The space
\begin{equation}\label{equ:std Hom space}
\Hom_{\GL_n(\Qp)}(\ind_{\mathbf{K}}^{\GL_n(\Qp)}\theta^\circ, \pi(\rbar))
\end{equation}
which admits a natural right action of $\cH_{\mathbf{K}}^{\GL_n(\Qp)}(\theta^\circ)$.

We fix our choice of $\theta$ based on the setup in \S\,\ref{subsub:intro:lift}, as we now explain. We fix an integer $1\leq i\leq n$ and let $\omega^{(i)}(p)$ be the diagonal matrix $(p \mathrm{Id}_i, \mathrm{Id}_{n-i})$. Let $L \subseteq\GL_n$ be the standard Levi subgroup with diagonal blocks $\GL_i\times\GL_{n-i}$ and consider the pair of parahoric subgroups $\mathbf{P}^+,\mathbf{P}^- \subseteq \mathbf{K}$ whose image under $\mathbf{K}\twoheadrightarrow\GL_n(\Fp)$ is given by the $\F_p$-points of the standard (resp.~of the opposite of the standard) maximal parabolic subgroup of $\GL_n$ with Levi $L$.
Hence, each representation of $L(\F_p)$ is a smooth representation of $\mathbf{P}^+$ (resp.~$\mathbf{P}^-$) by inflation.
The parahoric subgroups $\mathbf{P}^+,\mathbf{P}^-$ are related by $\omega^{(i)}(p)\mathbf{P}^+\omega^{(i)}(p)^{-1}=\mathbf{P}^{-}$.
Let $\tau_1$ and $\tau_2$ be tame inertial types over $\Qp$ of dimension $i$ and $n-i$, respectively, such that $\tau\defeq \tau_1 \oplus \tau_2$ is multiplicity free as an $E$-representation of $I_{\Qp}$. Using the inertial local Langlands correspondence, we attach to $\tau=\tau_1\oplus\tau_2$ an irreducible $E$-representation $\sigma(\tau)$ of $\GL_n(\F_p)$. We similarly attach representations $\sigma(\tau_1)$ and $\sigma(\tau_2)$ and define the irreducible $E$-representation $\sigma\defeq \sigma(\tau_1)\otimes_E\sigma(\tau_2)$ of $L(\F_p)$.  Then $\sigma(\tau)$ and $\ind_{\mathbf{P}^+}^{\mathbf{K}}\sigma$ are isomorphic.
We also fix an arbitrary $L(\F_p)$-stable $\cO$-lattice $\sigma^\circ\subseteq\sigma$ and consider the Hecke algebra $\cH_{\mathbf{P}^+}^{\GL_n(\Qp)}(\sigma^\circ)\defeq \End_{\GL_n(\Qp)}(\ind_{\mathbf{P}^+}^{\GL_n(\Qp)}\sigma^\circ)$.

The Hecke algebra $\cH_{\mathbf{P}^+}^{\GL_n(\Qp)}(\sigma^\circ)$ contains an element $\mathbf{U}^{\tau_1}_\tau$, sometimes known as a \emph{$U_p$-operator} (see \S\,\ref{subsub:Up:act}).
Under a local Langlands correspondence in families \cite[\S 4]{CEGGPS}, the operator $\mathbf{U}^{\tau_1}_\tau$ recovers the function $\phi_{\tau,\tau_1}^{-1}$ on the moduli of Weil--Deligne representations, up to a specific scalar in $p^\Z$ depending on $i$ and $n$.
Hence, according to Theorem~\ref{thm4:inv lift}, our next goal is to capture the $p$-adic leading term of the eigenvalues of the action of $\mathbf{U}^{\tau_1}_\tau$.

We first observe that the equality $\mathbf{P}^+\omega^{(i)}(p)^{-1}\mathbf{P}^+=\omega^{(i)}(p)^{-1}\mathbf{P}^-\mathbf{P}^+$ induces the following decomposition of~$\mathbf{U}^{\tau_1}_\tau$
\begin{equation*}\label{equ:intro:factor}
\xymatrix{
\ind_{\mathbf{P}^+}^{\GL_n(\Qp)}\sigma^\circ\ar^{S_\sigma}[rrr]\ar@/^2pc/^{\mathbf{U}^{\tau_1}_\tau}[rrrrrr] &&&\ind_{\mathbf{P}^-}^{\GL_n(\Qp)}\sigma^\circ\ar^{t_i}_{\sim}[rrr] &&&\ind_{\mathbf{P}^{+}}^{\GL_n(\Qp)}\sigma^\circ
}
\end{equation*}
where $S_\sigma$ is the embedding induced (by applying $\ind_{\mathbf{K}}^{\GL_n(\Qp)}$) from an embedding of $\cO$-modules
\begin{equation}\label{equ:intro:embedding of lattice}
\sigma(\tau)\supseteq\ind_{\mathbf{P}^+}^{\mathbf{K}}\sigma^\circ\hookrightarrow \ind_{\mathbf{P}^-}^{\mathbf{K}}\sigma^\circ,
\end{equation}
and $t_i$ is the intertwining isomorphism induced from $\omega^{(i)}(p)\mathbf{P}^+\omega^{(i)}(p)^{-1}=\mathbf{P}^{-}$.
Note that~(\ref{equ:intro:embedding of lattice}) is an isomorphism after inverting $p$ and so identifies $\ind_{\mathbf{P}^-}^{\mathbf{K}} \sigma^\circ$ with a $\mathbf{K}$-stable $\cO$-lattice in~$\sigma(\tau)$.

Let $\sigma(\tau)^\circ$ be another $\mathbf{K}$-stable $\cO$-lattice in $\sigma(\tau)$ with $\sigma(\tau)^\circ \subseteq \ind_{\mathbf{P}^+}^{\mathbf{K}} \sigma^\circ$, and denote this inclusion by $S_{\sigma(\tau)}^+$.
Let $\kappa \in \Z$ be the maximal integer such that $p^{-\kappa} \sigma(\tau)^\circ \subseteq \ind_{\mathbf{P}^-}^{\mathbf{K}} \sigma^\circ$ via~\eqref{equ:intro:embedding of lattice}, and we define $S_{\sigma(\tau)}^-$ as the composite $\sigma(\tau)^\circ \overset{p^{-\kappa}}{\ra} p^{-\kappa} \sigma(\tau)^\circ \subseteq \ind_{\mathbf{P}^-}^{\mathbf{K}} \sigma^\circ$.
The maps $S_{\sigma(\tau)}^+$ and $S_{\sigma(\tau)}^-$ fit into the following commutative diagram involving $\mathbf{U}^{\tau_1}_\tau$ (see~(\ref{eqn:Up}))
\begin{equation}\label{equ:intro:integral}
\xymatrix{
\ind_{\mathbf{P}^+}^{\GL_n(\Q_p)}\sigma^\circ\ar^{S_\sigma}[rr] \ar@/^2pc/^{\mathbf{U}^{\tau_1}_\tau}[rrrr] &&\ind_{\mathbf{P}^-}^{\GL_n(\Q_p)}\sigma^\circ\ar^{t_i}_{\sim}[rr] &&\ind_{\mathbf{P}^{+}}^{\GL_n(\Q_p)}\sigma^\circ\\
\ind_{\mathbf{K}}^{\GL_n(\Q_p)} \sigma(\tau)^\circ \ar^{S_{\sigma(\tau)}^+}[u] \ar^{p^\kappa}[rr] &&\ind_{\mathbf{K}}^{\GL_n(\Q_p)}\sigma(\tau)^\circ\ar^{S_{\sigma(\tau)}^-}[u] && \ind_{\mathbf{K}}^{\GL_n(\Q_p)} \sigma(\tau)^\circ. \ar^{S_{\sigma(\tau)}^+}[u]
}
\end{equation}
Applying $\Hom_{\cO[\GL_n(\Qp)]}(-,\pi)$ to~(\ref{equ:intro:integral}) for a $\cO[\GL_n(\Qp)]$-module $\pi$ we obtain the following diagram (by abusing the notation of maps in~(\ref{equ:intro:integral}) for the induced maps between $\Hom$-spaces)
\begin{equation}\label{equ:intro:integral:Hom}
\xymatrix{
\Hom(\ind_{\mathbf{P}^+}^{\GL_n(\Qp)}\sigma^\circ,\pi)\ar_{S_{\sigma(\tau)}^+}[d] &\Hom(\ind_{\mathbf{P}^-}^{\GL_n(\Qp)}\sigma^\circ,\pi)\ar_{S_\sigma}[l]\ar_{S_{\sigma(\tau)}^-}[d] &\Hom(\ind_{\mathbf{P}^{+}}^{\GL_n(\Qp)}\sigma^\circ,\pi) \ar_{t_i}^{\sim}[l] \ar_{S_{\sigma(\tau)}^+}[d] \ar@/_2pc/_{\mathbf{U}^{\tau_1}_\tau}[ll]\\
\Hom(\ind_{\mathbf{K}}^{\GL_n(\Qp)} \sigma(\tau)^\circ,\pi) &\Hom(\ind_{\mathbf{K}}^{\GL_n(\Qp)}\sigma(\tau)^\circ,\pi) \ar_{p^\kappa}[l]
& \Hom(\ind_{\mathbf{K}}^{\GL_n(\Qp)} \sigma(\tau)^\circ,\pi) \ar@{-->}_{\tld{\mathbf{U}}^{\tau_1}_\tau}[l]
}
\end{equation}
where $\Hom$ denotes $\Hom_{\cO[\GL_n(\Qp)]}$. If $S_{\sigma(\tau)}^+$
is an isomorphism in~(\ref{equ:intro:integral:Hom}), then we can consider $(S_{\sigma(\tau)}^+)^{-1}$ and define a new map
$$\tld{\mathbf{U}}^{\tau_1}_\tau\defeq S_{\sigma(\tau)}^-\circ t_i\circ (S_{\sigma(\tau)}^+)^{-1}:~\Hom(\ind_{\mathbf{K}}^{\GL_n(\Qp)} \sigma(\tau)^\circ,\pi)\rightarrow\Hom(\ind_{\mathbf{K}}^{\GL_n(\Qp)} \sigma(\tau)^\circ,\pi).$$
We call the map $\tld{\mathbf{U}}^{\tau_1}_\tau$ a \emph{normalized $U_p$-operator}. Note that $\tld{\mathbf{U}}^{\tau_1}_\tau$ (if defined) is an isomorphism if and only if $S_{\sigma(\tau)}^-$ is an isomorphism in~(\ref{equ:intro:integral:Hom}). We caution the reader that $S_{\sigma(\tau)}^+$ and $S_{\sigma(\tau)}^-$ are almost never isomorphisms in~(\ref{equ:intro:integral}), but the induced maps in~(\ref{equ:intro:integral:Hom}) can nevertheless be isomorphisms for suitably chosen $\sigma^\circ$, $\sigma(\tau)^\circ$ and $\pi$. We also note that the maps $S_\sigma$, $p^{\kappa}$ and $\mathbf{U}^{\tau_1}_\tau$ in~(\ref{equ:intro:integral:Hom}) are usually zero when $\pi=\pi(\rbar)$.

Given a permutation $w\in W$ and a subset $I\subseteq\{1,\dots,n\}$ satisfying $w(I)=I$, we have an invariant function $f_{w,I}\in\Inv$, a $F(\lambda)$-relevant type $\tau_w$ as well as a sub inertial type $\tau_{w,1}\subseteq\tau_w$ from \S\,\ref{subsub:intro:lift}. We choose $\tau\defeq \tau_w\otimes_{\cO}\tld{\omega}^{n-1}$, $\tau_1\defeq \tau_{w,1}\otimes_{\cO}\tld{\omega}^{n-1}$ where $\tld{\omega}$ is the Teichm\"uller lift of the mod-$p$ cyclotomic character $\omega$. Let $x\in U\backslash\GL_n(\F)$ be a point satisfying $\rhobar_{x,\lambda+\eta}\cong\rhobar$, and $\cC\in\cP$ be the unique element satisfying $x\in\cC(\F)$.
The following is our main result on the action of $\tld{\mathbf{U}}^{\tau_1}_\tau$.
\begin{thm}[see the proof of Theorem~\ref{thm:main:LGC}]\label{thm2:Hecke}
Assume that $f_{w,I}\in\Inv(\cC)$ with $I\neq \emptyset,\{1,\dots,n\}$. Then there exist $\sigma^\circ$ and $\sigma(\tau)^\circ$ depending only on $\lambda, w, I$ such that
$$
 S_{\sigma(\tau)}^-\circ t_i =f_{w,I}(x)\cdot S_{\sigma(\tau)}^+ :~\Hom(\ind_{\mathbf{P}^{+}}^{\GL_n(\Qp)}\sigma^\circ,\pi(\rbar))\xrightarrow{\sim}\Hom(\ind_{\mathbf{K}}^{\GL_n(\Qp)} \sigma(\tau)^\circ,\pi(\rbar)).
$$
\end{thm}
The existence of $\sigma^\circ$ and $\sigma(\tau)^\circ$ such that both $S_{\sigma(\tau)}^+$ and $S_{\sigma(\tau)}^-$ are isomorphisms follows from the fact that the mod-$p$ reduction of $\sigma(\tau)$ contains a unique (counting multiplicity) \emph{modular Serre weight} of $\rbar$, which is $F(\lambda)\otimes_{\F}\omega^{n-1}\circ\det$. %

\subsubsection{Conclusion}\label{subsub:intro:conclusion}
Now we deduce Theorem~\ref{thm1:intro}, when $K=\Qp$, from the results in \S\,\ref{sub:ingredient}. (Recall that in Theorem~\ref{thm1:intro} $K/\Qp$ is a finite unramified extension, but in the introduction we only treat the case $K=\Q_p$.)

Let $\cC\in\cP$ and $x\in\cC(\F)$ be a point such that $\rhobar\cong \rhobar_{x,\lambda+\eta}$.
First of all, Theorem~\ref{thm3:SW} implies that the isomorphism class of $\pi(\rbar)$ determines the $\cC\in\cP$. Then it follows from Theorem~\ref{thm2:Hecke} and Theorem~\ref{thm4:inv lift} that the isomorphism class of $\pi(\rbar)$ determines the set $\{g(x)\mid g\in\Inv(\cC)\}$.
Finally, we deduce from Theorem~\ref{thm5:inv fun prime} that the set $\{g(x)\mid g\in\Inv(\cC)\}$ uniquely determines the $T$-conjugacy class of $x$, or equivalently the isomorphism class of $\rhobar$. Hence, we conclude that $\pi(\rbar)$ determines the isomorphism class of $\rhobar$, which finishes the proof of Theorem~\ref{thm1:intro} when $K=\Qp$.

\subsection{Overview of the paper}
We give a short overview of the various sections of the paper.

\emph{The reader who is primarily interested in the mod-$p$ Langlands program and local-global compatibility can skip \S\,\ref{sec:comb:lifts}, \S\,\ref{sec:const:inv}, \S\,\ref{sec:inv:cons}, \S\,\ref{sub:stack:scheme} and \S\,\ref{sub:main:state}: these concern the (shifted-)conjugation invariant functions on the basic (quasi-)affine and can be read independently from the rest of the paper.}

In \S\,\ref{sec:prel} we recall some categories of semi-linear algebra objects from $p$-adic Hodge theory: Fontaine--Laffaille theory in characteristic $p$ (\S\,\ref{subsub:FLT}), Breuil--Kisin modules with tame descent (\S\,\ref{subsubBKD}), \'etale $\varphi$-modules and their relation to Galois representations (\S\,\ref{subsub:Phi-mod} and \S\,\ref{sub:MFSandGR}).

In \S\,\ref{sec:geo:FL} we study the geometry of moduli spaces of Fontaine--Laffaille modules in characteristic $p$ and introduce the partition given by  translated Schubert cells (\S\,\ref{sub:finest:elements}).
The coarser niveau stratification, and its relation with Galois representations, is studied in \S\,\ref{sub:niveau:elements}.
Finally, \S\,\ref{sub:std:note} describes the standard coordinates which will be used to analyze the behavior of invariant functions on moduli of Fontaine--Laffaille modules.

In \S\,\ref{sub:def:inv} we define the functions on the basic (quasi-)affine which will be relevant to us (the ``invariant functions'').
The reader who is only interested in their arithmetic applications can skip \S\,\ref{sub:stack:scheme}, \S\,\ref{sub:main:state} and \S\,\ref{sec:comb:lifts} until \S\,\ref{sec:inv:cons}. %
These sections concern the construction of sufficiently many good invariant functions for Statement~\ref{state: separate points prime}, and form the technical heart of this work.

In \S\,\ref{sec:FL:SW} we study the embedding of the moduli space of Fontaine--Laffaille modules into the Emerton--Gee stack, using the local model for Galois representations of \cite{MLM} and recalled in \S\,\ref{subsec:LMforEG}.
We introduce the notion of \emph{extremal weights} from \cite{OBW} (\ref{sec:SWandGR}) and analyze in \S\,\ref{subsec:RTandSW} and \S\,\ref{subsec:PandOW} how it is related to the partition on the moduli of Fontaine--Laffaille modules from \S\,\ref{sub:finest:elements}.

In \S\,\ref{sec:CLIF} we study crystalline Frobenii on certain tubes of tamely potentially crystalline Emerton--Gee stacks.
After combinatorial preliminaries on tame inertial types (\S\,\ref{sub:dd} and \S\,\ref{sub:ext:shape}) we explain in \S\,\ref{sub:inv WD} how to recover the invariant functions of \S\,\ref{sub:def:inv} as renormalized crystalline Frobenii on such tubes.

Finally in \S\,\ref{sec:LGC} we analyze the automorphic side.
We first introduce the ``normalized Hecke operators'' for smooth representations of $\GL_n(K)$ in characteristic $p$ (\S\,\ref{subsub:ind:Hecke}) and then describe how they capture the invariant functions and establish local-global compatibility (\S\,\ref{subsub:AxSetup}) under an axiomatic setup for patching functors. \S\,\ref{sub:GSetup} finally produces patching functors from spaces of automorphic forms on compact unitary groups giving the main result on local-global compatibility.

\subsection{Notation}
\label{subsec:notation}

If $F$ is any field, we write $G_F\defeq \Gal(\overline{F}/F)$ for the absolute Galois group, where $\overline{F}$ is a separable closure of $F$. If $F$ is a number field and $v$ is a place of $F$ then we write $F_v$ for the completion of $F$ at the place $v$, and if we further assume $v$ is a finite place of $F$ then we write $\cO_{F_v}$ for the ring of integers of $F_v$ and $k_{v}$ for the residue field. If $F$ is a local field, we write $I_F$ to denote the inertia subgroup of $G_F$.
Moreover, if $W_{F}\leq G_{F}$ denotes the Weil group, we normalize Artin's reciprocity map $\mathrm{Art}_{F}: F\s\stackrel{\sim}{\ra} W_{F}^{\mathrm{ab}}$ in such a way that uniformizers are sent to geometric Frobenius elements.

We fix once and for all an algebraic closure $\overline{\Q}$ of $\Q$. All number fields are considered as subfields of our fixed $\overline{\Q}$. Similarly, if $\ell\in \Q$ is a prime, we fix algebraic closures $\overline{\Q}_\ell$ as well as embeddings $\overline{\Q}\iarrow\overline{\Q}_{\ell}$. All finite extensions of $\Q_{\ell}$ will thus be considered as subfields in $\overline{\Q}_{\ell}$. Moreover, the residue field of $\overline{\Q}_{\ell}$ is denoted by $\overline{\F}_\ell$.

Let $p>2$ be a prime. We write $\varepsilon:G_{\Q_p}\rightarrow\Z_p^\times$ for the cyclotomic character with mod-$p$ reduction~$\omega$, and write $\tld{\omega}:G_{\Q_p}\rightarrow\Z_p^\times$ for the Teichm\"uller lift of $\omega$.

For $f>0$ we let $K$ be the unramified extension of $\Qp$ of degree $f$.
We write $k$ for its residue field (of cardinality $q=p^f$) and  $W(k)$ for its ring of integers.
We write $\val_p: K^\times\rightarrow\Z$ for the $p$-adic valuation normalized by $\val_p(p)=1$, and then write $|\cdot|\defeq q^{-\val_p(\cdot)}$ for the $p$-adic norm.

\subsubsection{Galois theory}
\label{sec:notation:GT}

We write $e=p^f-1$ and fix a primitive $e$-th root $\pi\in \overline{K}$ of $-p$.
Define the extension $L = K(\pi)$. %
The choice of the root $\pi$ let us define a character
$$
\tld{\omega}_K:\,\Gal(L/K)\rightarrow W(k)\s,\quad g\mapsto \frac{g(\pi)}{\pi}.
$$

Let $E\subset \overline{\Q}_p$ be a finite extension of $\Qp$, which will be our coefficient field.
We write $\cO$ for its ring of integers, fix an uniformizer $\varpi\in \cO$ and let $\mathfrak{m}_E=(\varpi)$.
We write $\F\defeq \cO/\mathfrak{m}_E$ for its residue field.
We will always assume that $E$ is sufficiently large.
In particular, we will assume that any embedding $\sigma: K\iarrow \overline{\Q}_p$ factors through $E\subset \overline{\Q}_p$. We abuse the notation $\val_p$ and $|\cdot|\defeq q^{-\val_p(\cdot)}$ for their extension to $E^\times$.

We fix an embedding $\sigma_0: K\into E$. The embedding $\sigma_0$ induces maps $\cO_K \iarrow \cO$ and $k \iarrow \F$; we will abuse the notation and denote these all by $\sigma_0$. We let $\phz$ denote the $p$-th power Frobenius on $k$ and set $\sigma_j \defeq \sigma_0 \circ \phz^{-j}$. The choice of $\sigma_0$ gives $\omega_f\defeq \sigma_0 \circ \tld{\omega}_K:I_K \ra \cO^{\times}$, a fundamental character of niveau $f$, and an identification between the set $\cJ\defeq \{\sigma:K\into E\}$ and $\Z/f$. It is clear that $\tld{\omega}=\prod_{j\in\cJ}\sigma_j\circ\tld{\omega}_K=\omega_f^{\frac{p^f-1}{p-1}}$.
We fix once and for all a sequence $\un{p}\defeq (p_n)_{n\in \N}$ where $p_n\in\ovl{K}$ satisfies $p_{n+1}^p=p_n$, $p_0=-p$.
We let $K_\infty\defeq \underset{n\in\N}{\bigcup}K(p_n)$ and $G_{K_\infty}\defeq \Gal(\ovl{K}/K_\infty)$.

We now fix $n>1$, the dimension of the Galois representation we deal with in this paper.
We set $r\defeq n!$ and fix an unramified extension $K'\supseteq K$ of relative degree $r$ and residue field $k'$.
We assume that $E$ is sufficiently large so that it contains any embedding $\sigma'_0:K'\into \ovl{\Q}_p$ and fix $\sigma'_0:K'\into E$ which extends $\sigma_0$.
If $f'\defeq rf$ and $j'\in\{0,\dots,f'-1\}$ we set $\sigma_{j'}\defeq \sigma'_0\circ \phz^{-j'}$, hence an identification of the set of embeddings $K'\into E$ with $\Z/f'$ so that restriction to $K$ induces the natural projection $\Z/f'\onto \Z/f$.

Let $\rho: G_K\rightarrow \GL_n(E)$ be a $p$-adic, de Rham Galois representation.
For $\sigma: K\iarrow E\subset\overline{\Q}_p$, we define $\mathrm{HT}_\sigma(\rho)$ to be the multiset of $\sigma$-labeled Hodge-Tate weights of $\rho$, i.e. the set of integers $i$ such that $\dim_E\big(\rho\otimes_{\sigma,\Q_p}\C_p(-i)\big)^{G_K}\neq 0$ (with the usual notation for Tate twists).
In particular, the cyclotomic character $\varepsilon$ has Hodge--Tate weights 1 for all embedding $\sigma\into E$.

An inertial type for $K$ is a conjugacy class of a morphism $\tau: I_K\ra \GL_n(E)$ with open kernel and which extends to the Weil group $W_K$ of $G_K$.
The \emph{inertial type
 of $\rho$} is the isomorphism class of $\mathrm{WD}(\rho)|_{I_K}$, where $\mathrm{WD}(\rho)$ is the Weil--Deligne representation attached to $\rho$ as in \cite[Appendix B.1]{CDT} (in particular, $\rho\mapsto\mathrm{WD}(\rho)$ is \emph{covariant}).

\subsubsection{Linear algebraic groups}
\label{intro:LAG}
We consider the linear algebraic group ${\GL_n}_{/\Z}$ defined over $\Z$.
We omit the subscript $\Z$ where there is no risk of confusion.
Let $\Phi^+\subseteq \Phi$ (resp. $\Phi^{\vee,+}\subseteq \Phi^\vee$) be the subset of positive roots (resp. coroots) of $\GL_n$ with respect to the Borel $B\subseteq \GL_n$ of upper triangular matrices.
We further write $B=T\rtimes U$ where $U\subseteq B$ is the subgroup of upper triangular unipotent matrices of $\GL_n$ and $T\subseteq \GL_n$ is the torus of diagonal matrices. We use the notation $u_\al:~U\twoheadrightarrow \bG_a$ for the projection to the $\al$-entry, for each $\al\in\Phi^+$.
Write $W$ (resp. $W_a$, resp. $\tld{W}$) for the Weyl group (resp. the affine Weyl group, resp. the extended affine Weyl group) of $\GL_n$.

We have an injective group homomorphism $W\into \GL_n(\Z)$, which identifies an element $w\in W$ with the matrix whose $i$-th column is given by the $w(i)$-th vector in the standard basis of $\Z^n$. %
(We use the same symbol $w$ to denote the image of $w\in W$ via $W\into \GL_n(\Z)$; this will not cause confusion.)
We write $w_0\in W$ for the longest element in the Weyl group of $\GL_n$.
We let $X^*(T)$ denote the group of characters of $T$, $X_*(T)$ for the group of cocharacters of $T$, which are both  identified with $\Z^n$ in the usual way.
For instance, the $i$-th element of the standard basis $\eps_i\defeq (0,\dots,0,1,0\dots,0)$ (with the $1$ in the $i$-th position) corresponds to character extracting the $i$-diagonal entry of a diagonal matrix.
We write $\langle\cdot\,,\cdot\rangle: X^*(T)\times X_*(T)\rightarrow \Z$ for the standard pairing.
Let $\Phi$ (resp. $\Phi^\vee$) denote the set of roots (resp. coroots) of $\GL_n$ and $\La_R\subseteq X^*(T)$ the root lattice.
We then have
\begin{equation}
\label{eqn:iso:weyl}
W_a=\La_R\rtimes W\qquad\mbox{ and }\qquad \tld{W}=X^*(T)\rtimes W.
\end{equation}
Let $\un{G}$ be the group $\big(\Res_{\cO_K/\ZZ_p} \GL_n\big) \times_{\Zp} \cO$, and similarly define $\un{T}$, $\un{Z}$, $\un{B}$, $\un{U}$.
There is a natural isomorphism $\un{G} \cong \prod_{i\in \cJ} {\GL_n}_{/\cO}$.
One has similar isomorphisms for $\un{T}$, $\un{Z}$, $X^*(\un{T})$, $\un{\Phi}$, $\un{\Phi}^\vee$ where $\un{\Phi}$ (resp. $\un{\Phi}^\vee$) denotes the set of roots (resp. coroots) of $\un{G}$.
If $\mu \in X^*(\un{T})$, then we correspondingly write $\mu = \sum_{j\in \cJ} \mu_j$.
We use similar notation for similar decompositions.
Again we identify $X^*(\un{T})$ with $(\Z^n)^\cJ$ in the usual way and let $\eps_{j,i}\in (\Z^n)^\cJ$ be $(0,\dots,0,1,0,\dots,0)$ in the $j$-th coordinate, where $1$ appears in position $i$, and $n$-tuple $\un{0}$ otherwise.
In particular, we sometimes abuse notation and identify $\mu_j$ with an element of $\Z^n$, and write $0\in X^*(\un{T})$ for the element corresponding to zero element in $(\Z^n)^\cJ$.
The arithmetic Frobenius induces an automorphism $\pi$ on $X^*(\un{T})$.
It is characterized by $\pi(\lambda)_j=\lambda_{j+1}$.
Again, we write $X_*(\un{T})$ for the group of cocharacters of $\un{T}$, and write $\langle\cdot\,,\cdot\rangle$ for the standard pairing $\langle\cdot\,,\cdot\rangle: X^*(\un{T})\times X_*(\un{T})\rightarrow \Z$.

Let $\un{\Phi}^+\subseteq \un{\Phi}$ (resp. $\un{\Phi}^{\vee,+}\subseteq \un{\Phi}^\vee$) be the subset of positive roots (resp. coroots) of $\un{G}$ with respect to the upper triangular Borel in each embedding.
Let $\un{\Delta}\subseteq\un{\Phi}^+$ be the set of simple roots, and $\un{\Delta}^\vee\subseteq\un{\Phi}^{\vee,+}$ be the set of simple coroots.
We define dominant (co)characters with respect to these choices.
Let $X^*_+(\un{T})$ be the set of dominant weights.
We denote by $X_1(\un{T}) \subset X^*_+(\un{T})$ be the subset of weights $\lambda\in X^*_+(\un{T})$  satisfying $0\leq \langle \lambda,\alpha^\vee\rangle\leq p-1$ for all simple roots $\alpha\in \un{\Delta}$.
We call $X_1(\un{T})$ the set of $p$-restricted weights.
We write $X^0(\un{T})$ for the set consisting of elements $\lambda\in X^*(\un{T})$ such that $\langle \lambda,\alpha^\vee\rangle=0$ for all roots $\alpha\in \un{\Phi}$.
Let $\eta_j\in X^*(\un{T})$ be $(n-1,n-2\dots,1,0)$ in the $j$-th coordinate and $0$ otherwise, and let $\eta$ be $\sum_{j\in\cJ} \eta_j \in X^*(\un{T})$.
We sometimes abuse notation and consider $\eta_j$ as the element $(n-1,n-2\dots,1,0)\in \Z^n\cong X^*(T)$: this should cause no confusion.
Then $\eta$ is a lift of the half sum of the positive roots of $\un{G}$.

Let $\un{W}$ be the Weyl group of $\un{G}$.
We abuse notation and write $w_0$ for its longest element.
Let $\un{W}_a$ and $\widetilde{\un{W}}$ be the affine Weyl group and extended affine Weyl group, respectively, of $\un{G}$.
Let $\un{\La}_R \subset X^*(\un{T})$ denote the root lattice of $\un{G}$.
As above we have identifications $\un{W}\cong W^\cJ$, $\un{W}_a\cong W_a^\cJ$, $\tld{\un{W}}\cong \tld{W}^\cJ$ and isomorphisms analogous to~(\ref{eqn:iso:weyl}).

The Weyl group $\tld{\un{W}}$ acts naturally on $X^*(\un{T})$.
If $\lambda\in X^*(\un{T})$ and $\tld{w}_\cJ\in\tld{\un{W}}$ we write $\tld{w}_\cJ(\lambda)$ to denote the image of $\lambda$ by this action.
The image of $\lambda\in X^*(\un{T})$ via the standard injection $X^*(\un{T})\into \tld{\un{W}}$ is denoted by $t_\lambda$.
We have similar actions of
$\un{W}$ and $\un{W}_a$ on $X^*(\un{T})$.
These actions of $\tld{\un{W}}$, $\un{W}$ and $\un{W}_a$ on $X^*(\un{T})$ are compatible with one another when considering the natural inclusions $\un{W}\subseteq \un{W}_a\subseteq \tld{\un{W}}$.
Moreover, the Weyl groups $\un{W}$, $\tld{\un{W}}$, $\un{W}_a$ act on $X^*(\un{T})$ via the \emph{$p$-dot action}, given by $t_\lambda w \cdot \mu = t_{p\lambda} w (\mu+\eta) - \eta$.

\subsubsection{Miscellaneaous}
\label{sub:sub:misc}

For any ring $S$ we write $\mathrm{M}_n(S)$ to denote the set of $n$ by $n$ matrix with entries in $S$.
If $\alpha=\eps_i-\eps_j$ %
is a root of $\GL_n$, we also call the $(i,j)$-th entry of a matrix $A\in \mathrm{M}_n(S)$ the $\alpha$-entry.

If $M$ is an $R$ module and $h:R\ra R'$ is an homomophism of rings we write $h^*(M)$ to denote the pullbck of $M$ along $h$, i.e.~the $R'$-module $M\otimes_{R,h}R'$.
We have an obvious map $h^*: M\ra h^*(M)$ sending $m$ to $m\otimes 1$, which is an $h$-semilinear map.

If $X$ is a scheme over $\Z$ and $R$ is any ring, we write $X_R$ for the fibered product $X\times_{\Spec \Z}\Spec R$.

Let $V$ be a representation of a finite group $\Gamma$ over an $E$-vector space.
We write $\JH(\ovl{V})$ to denote the set of Jordan--H\"older factors of the mod $\varpi$-reduction of an $\cO$-lattice in $V$.
This set is independent of the choice of the lattice.

\subsection{Acknowledgements}
Part of the work was carried out during visits at the Ecole Normale Sup\'erieure de Lyon (2018), Ulsan National Institute of Science and Technology (2019), Laboratoire Analyse G\'eom\'etrie Applications (2020).
We would like to heartily thank these institutions for their support.

We sincerely thank Christophe Breuil and Florian Herzig for their constant support and interest in this work.
We thank Toby Gee for comments on an earlier draft.

D.L. was supported by the National Science Foundation under agreements Nos.~DMS-1128155 and DMS-1703182 and an AMS-Simons travel grant. B.LH. acknowledges support from the National Science Foundation under grant Nos.~DMS-1128155, DMS-1802037 and the Alfred P. Sloan Foundation.
S.M. was supported by the ANR-18-CE40-0026 (CLap CLap) and the Institut Universitaire de France. C.P. was supported by Samsung Science and Technology Foundation under Project Number~SSTF-BA2001-02.
\clearpage{}%
\clearpage{}%
\section{Preliminaries}
\label{sec:prel}
Throughout this section we let $K/\Q_p$ be a finite unramified extension of degree $f$.
Recall that the choice of $\sigma_0:k\into \F$ identifies $\{0,\dots,f-1\}$ with $\cJ$ via $j\mapsto \sigma_j\defeq\sigma_0\circ\phz^{-j}$.

\subsection{Inertial types}
\label{subsec:IT}

We record here some notation and facts pertaining to tame inertial types for~$K$. We start with defining the genericity conditions that will be used throughout the paper.
\begin{defn}
\label{defn:mGenFL}
Let $\mu\in X^*(\un{T})$ and $m\in\N$. We say that \emph{$\mu$ is $m$-generic Fontaine--Laffaille} if $m<\langle\mu,\alpha^\vee\rangle<p-m$ for all positive roots $\alpha\in \un{\Phi}^+$.
We say that $\mu$ is \emph{Fontaine--Laffaille} if $0\leq \langle\mu,\alpha^\vee\rangle\leq p-2$ for all positive roots $\alpha\in \un{\Phi}^+$.
\end{defn}
Note that if $\mu$ is $1$-generic Fontaine--Laffaille then it is, in particular, Fontaine--Laffaille. We also note that if $\mu\in X^*_+(\un{T})$ is dominant and $\mu+\eta$ is Fontaine-Laffaille, then $\mu+\eta$ is $0$-generic Fontaine--Laffaille.

Recall from \S\,\ref{sec:notation:GT} that an inertial type  for $K$ is a conjugacy class of representations $I_K \ra \GL_n(E)$ which have an open kernel and extend to the Weil group of $K$.
Similarly, we define an inertial $\F$-type for $K$ as a conjugacy class of representations $I_K \ra \GL_n(\F)$ which have open kernel and extend to the Weil group of $K$.
An inertial ($\F$-)type is \emph{tame} if it factors through the tame inertial quotient.
Given an inertial type $\tau$, one obtains an inertial $\F$-type $\ovl{\tau}$ by taking the semisimplification of the reduction of any $I_K$-stable $\cO$-lattice in $\tau$.

We have a combinatorial description of tame inertial types from \cite[Example 2.4.1]{MLM}:

\begin{defn}
 \label{defn:tau}
For $(s_\cJ,\mu)\in \un{W}\times X^*(\un{T})$ define the inertial type $\tau(s_\cJ, \mu+\eta):I_K \ra \GL_n(\cO)$ as follows:  If $s_{\cJ} = (s_0, \ldots, s_{f-1})$,  set $s_{\tau} \defeq s_0 s_{1} s_{2} \cdots s_{f-1} \in W$  and  $\bm{\alpha}_{(s_\cJ,\mu)} \in X^*(\un{T})$ such that $\bm{\alpha}_{(s_\cJ,\mu),0} = \mu_0+\eta_0$ and $\bm{\alpha}_{(s_\cJ,\mu),j} = s_{f-1}^{-1} s_{f-2}^{-1} \ldots s_{f-j}^{-1}(\mu_{f-j}+\eta_{f-j})$ for $1 \leq j \leq f-1$.
Recall from \S\,\ref{sec:notation:GT} that $r=n!$ so that $(s_{\tau})^r=1$. %
Then by letting $\chi_i\defeq \omega_{f'}^{\sum_{0 \leq k \leq r-1} \bf{a}^{(0)}_{s_{\tau}^{k}(i)} p^{fk}} $ with $\bf{a}^{(0)} \defeq \sum_{j =0}^{f-1}  \bm{\alpha}_{(s_\cJ,\mu),j} p^j\in \Z^n$, we define:
\begin{equation} \label{eq:def:type}
\tau(s_\cJ,\mu+\eta) \defeq \bigoplus_{1 \leq i \leq n}\chi_i.
\end{equation}
We set $\taubar(s_\cJ,\mu+\eta)$ to be the reduction of $\tau(s_\cJ,\mu+\eta)$ to the residue field of $\cO$.

\end{defn}
\begin{defn}
\label{defi:gen}
Let $\tau$ be a tame inertial type.
\begin{enumerate}
\item
\label{def:LApres}
A \emph{lowest alcove presentation} of $\tau$ is a pair $(s_\cJ, \mu) \in \un{W} \times X^*(\un{T})$ where $\mu+\eta$ is $0$-generic Fontaine--Laffaille and such that $\tau\cong \tau(s_\cJ, \mu + \eta)$.
Given a lowest alcove presentation $(s_\cJ, \mu)$ for $\tau$ we associate to it the element $\tld{w}(\tau)\defeq t_{\mu+\eta}s_{\cJ}\in \tld{\un{W}}$.
\item
\label{def:LApres:cmpt}
We say that the lowest alcove presentation of $\tau$ is \emph{compatible with} $\zeta\in X^*(\un{Z})$ if $\tld{w}(\tau) \un{W}_a$ corresponds to $\zeta$ via the isomorphism $\tld{\un{W}}/\un{W}_a\stackrel{\sim}{\ra} X^*(\un{Z})$.
Lowest alcove presentations of tame inertial types are said to be \emph{compatible} if they are compatible with the same element of $X^*(\un{Z})$.
\end{enumerate}
\end{defn}
(Note that lowest alcove presentations for a given tame inertial type are not unique, but in generic cases one can pass from one to another by \cite[Proposition 2.2.15]{LLL})

For a local Galois representation $\rhobar: G_{K}\ra\GL_n(\F)$, we consider the inertial representation $\rhobar^{\mathrm{ss}}|_{I_{K}}$ and let $[\rhobar^{\mathrm{ss}}|_{I_K}]$ be the Teichm\"uller lift  of $\rhobar^{\mathrm{ss}}|_{I_{K}}$.
Then (the conjugacy class of) $[\rhobar^{\mathrm{ss}}|_{I_{K}}]$ is a tame inertial type.

\begin{defn}
\label{def:gen:Gal}
We say that $\rhobar$ is \emph{$m$-generic} if $[\rhobar^{\mathrm{ss}}|_{I_K}]$ has a lowest alcove presentation $(s_\cJ, \mu)$ where $\mu+\eta$ is $m$-generic Fontaine--Laffaille.
\end{defn}

Let $\tau \defeq \tau(s_\cJ, \mu+\eta)$ be a tame inertial type for $K$ with $\mu+\eta$ being $1$-generic Fontaine--Laffaille.
Recall from \S\,\ref{sec:notation:GT} the fixed unramified extension $K'$ of $K$ of degree $r=n!$, with the embedding $\sigma_0':K'\into E$ extending $\sigma_0: K\into E$.
Let $\tau'$ denote the type $\tau$ viewed as a tame inertial type for $K'$.
(We call $\tau'$ the \emph{base change} of $\tau$.)
Define $\bm{\alpha}'_{(s_\cJ,\mu)} \in  X^*(T)^{\Hom(k', \F)}$ by
\[
\bm{\alpha}'_{(s_\cJ,\mu), j + kf} \defeq  s_{\tau}^{-k} (\bm{\alpha}_{(s_\cJ, \mu), j}) \text{ for } 0 \leq j \leq f-1, 0 \leq k \leq r-1.
\]
(The embedding $\sigma_0'$ induces an isomorphism $X^*(T)^{\Hom(k', \F)} \cong X^*(\un{T})^{r}$.)
If $s'_{\cJ'}=(s'_{j'})_{j'\in\cJ'}\in\un{W}^r$ is the element characterized by $s'_{j'}=s_j$ for $j\equiv j'\pmod{f}$ and similarly for $\mu'\in  X^*(\un{T})^{r}$, then $\tau' \cong \tau'(s'_{\cJ'}, \mu')\cong \tau'(1, \bm{\alpha}'_{(s_\cJ,\mu)})$ by \eqref{eq:def:type}.
The \emph{orientation} $s'_{\mathrm{or}} \in \un{W}^{r}$ of $\bm{\alpha}'_{(s_\cJ,\mu)}$ is defined by
\begin{equation} \label{primeorient}
s'_{\mathrm{or}, j + kf} \defeq s_{\tau}^{k+1} (s_{f-1}^{-1}s_{f-2}^{-1}\cdots s_{j+1}^{-1})\text{ for } 0 \leq j \leq f-1, 0 \leq k \leq r-1
\end{equation}
(see \cite[equation (5.4)]{MLM}).

\subsection{Fontaine--Laffaille theory}
\label{subsub:FLT}
The goal of this section is to define the stack of Fontaine--Laffaille modules. In all what follows $R$ is a Noetherian $\F$-algebra. The $R$-algebra $k\otimes_{\Fp}R$ is endowed with a canonical $R$-algebra endomorphism $\phz$ that acts as the arithmetic Frobenius on $k$, and as the identity on $R$.

Since all schemes and stacks are defined over $\Spec \F$ we omit the subscript $\bullet_{\F}$ from the notation when considering the base change to $\F$ of an object $\bullet$ defined over $\cO$ (e.g.~$\GL_{n,\F}$ will be denoted by $\GL_n$ and so on).

\begin{defn}
A \emph{pseudo Fontaine--Laffaille module with $R$-coefficients} is a finite projective $k\otimes_{\Fp}R$-module $M$ together with:
\begin{enumerate}
\item an exhaustive and separated decreasing filtration $\{\Fil^iM\}_{i\in \Z}$ by $k\otimes_{\Fp}R$-modules (the \emph{Hodge filtration of $M$}), whose associated graded pieces $\gr^i(M)\defeq \Fil^iM/\Fil^{i+1}M$ are projective $k\otimes_{\Fp}R$-modules;  %
\item a $\phz$-semilinear bijection $\phi_M: \gr^{\bullet}(M) \cong M$.
(We will often omit the $M$ in the subscript of $\phi_{M}$ and $\phi_{i,M}\defeq \phi_M|_{\gr^{i}(M)}$ when the module $M$ is clear from the context.)
\end{enumerate}
\end{defn}

Via the decomposition $k\otimes_{\Fp}R\cong\prod_{j\in\cJ}R$ induced by $x\otimes 1\mapsto (\sigma_j(x))_{j\in\cJ}$, we write $\epsilon_j\in k\otimes_{\Fp}R$ for the idempotent element corresponding to the component $j$.
A pseudo Fontaine--Laffaille module $M$ admits a canonical decomposition
$M\stackrel{\sim}{\ra}\prod_{j\in\cJ}M^{(j)}$ with $M^{(j)}\defeq \epsilon_jM$, a projective $R$-module. The action of $x\otimes 1 \in k\otimes_{\Fp} R$ on $M^{(j)}$ is given by $\sigma_j(x)\in R$.
Each $M^{(j)}$ inherits a decreasing, exhaustive and separated filtration $\Fil^i M^{(j)}$ by $R$-modules, and a collection of $R$-linear morphisms $$\phi_i^{(j)}:\gr^i(M^{(j)})\ra M^{(j+1)}$$ where $\gr^i(M^{(j)})\defeq \Fil^iM^{(j)}/\Fil^{i+1}M^{(j)}$ (which is projective).
Note that for each $j\in\cJ$ this family of morphisms $\phi_i^{(j)}$ induces a morphism $\phi^{(j)}:\gr^\bullet(M^{(j)})\ra M^{(j+1)}$.

\begin{defn}\label{def: Fontaine--Laffaille module}
A \emph{Fontaine--Laffaille module with $R$-coefficients} is a pseudo Fontaine--Laffaille module $M$ with $R$-coefficients such that for each $j\in\cJ$ $$\min\{i\in\Z\mid\Fil^iM^{(j)}=0\}-\max\{i\in\Z\mid\Fil^iM^{(j)}=M\}\leq p-1.$$
Fontaine--Laffaille modules with $R$-coefficients form a category, with morphisms being $k\otimes_{\Fp} R$-linear homomorphisms which respect the filtration and the maps $\phi$. There is an evident notion of base change along an $\F$-algebra homomorphism $R\to S$.
\end{defn}

For $\lambda\in X^*(\un{T})$, we write $\lambda=(\lambda_{j})_{j\in\cJ}$ with $\lambda_{j}=(\lambda_{j,1},\dots,\lambda_{j,n})\in\Z^{n}$.
\begin{defn}
A Fontaine--Laffaille module \emph{of weight $\lambda\in X_+^*(\un{T})$} (with $R$-coefficients) is a Fontaine--Laffaille module $(M,\{\Fil^i M\}_i,\{\phi_i\}_i)$ with $R$-coefficients such that $\gr^{i}(M^{(j)})\neq 0$ if and only if $i\in\{\lambda_{j,1},\lambda_{j,2},\dots,\lambda_{j,n}\}$, for each $j\in\cJ$.
\end{defn}

Note that if a Fontaine--Laffaille module is of weight $\lambda\in X_+^*(\un{T})$ then $\lambda$ is Fontaine--Laffaille.

We now fix a dominant weight $\lambda\in X_+^*(\un{T})$ such that $\lambda+\eta$ is Fontaine--Laffaille. Note that such a weight $\lambda+\eta$ is, in particular, $0$-generic Fontaine--Laffaille. Let $\mathrm{FL}_n^{\lambda+\eta}$ be the sheafification of the functor that sends a Noetherian $\F$-algebra $R$ to the groupoid of Fontaine--Laffaille modules of weight $\lambda+\eta$ with $R$-coefficients.
\begin{defn}
\label{def: compatible basis}
Let $(M, \{\Fil^iM\}_{i\in \Z}, \{\phi_i\}_{i\in\Z})\in \mathrm{FL}_n^{\lambda+\eta}(R)$, where $R$ is a Noetherian $\F$-algebra with residue field $\F$.
A \emph{basis} $\beta=(\beta^{(j)})_{j\in\cJ}$ for $M$ is a $\cJ$-tuple where for all $j\in\cJ$ the (ordered) $n$-tuple $\beta^{(j)}=(\beta^{(j)}_1,\dots, \beta^{(j)}_n)$ is a basis for $M^{(j)}$.

We say that a basis $\beta$ for $M$ is \emph{compatible} (\emph{with the Hodge filtration}) if for each $j\in\cJ$
$$\Fil^{\lambda_{j,i}+(n-i)}M^{(j)}=R\cdot\beta^{(j)}_1+\cdots+R\cdot\beta^{(j)}_{i}$$
for all $i\in\{1,2,\cdots,n\}$.
\end{defn}
Note that bases for $M$ do not necessarily exist, but they always do Zariski locally on $R$. Each compatible basis $\beta$ for $M$ induces a basis $\gr^{\bullet}(\beta)=(\gr^{\bullet}(\beta^{(j)}))_{j\in\cJ}$ for $\gr^{\bullet}(M)$, which together determine a matrix for $\phi_M$ called \emph{the matrix of $\phi_M$ attached to $\beta$}.

We now show that $ \mathrm{FL}_n^{\lambda+\eta}$ is representable
by an algebraic stack.
We let $\tld{\cF\cL}_{\cJ}= \un{U}\backslash\un{G}$ be the basic (quasi-)affine for $\un{G}$ (which is a quasi-affine variety, cf.~\cite[Theorem 2.1 and Corollary 2.7]{Grosshans}).
We define the shifted
conjugation action of $\un{T}$ on $\tld{\cF\cL}_{\cJ}$ by the formula (noting that $T$ normalizes $U$)
\begin{equation}
\label{it:T-act:1}
(A\cdot t)^{(j)}\defeq (t^{(j+1)}\big)^{-1} A^{(j)}t^{(j)}
\end{equation}
for all $j\in\cJ$, where $t= (t^{(j)})_{j\in\cJ}\in \un{T}(R)$ and $A=(A^{(j)})_{j\in\cJ}\in \tld{\cF\cL}_{\cJ}(R)$.

\begin{prop}
\label{thm:repr:FLgp}
$\mathrm{FL}_n^{\lambda+\eta}$ is representable by $\big[\tld{\cF\cL}_{\cJ}\slash{{\sim}_{\un{T}\text{-\textnormal{sh.cnj}}}}\big]$
where $\un{T}$ acts via the shifted conjugation action.
\end{prop}
\begin{proof} Let $\mathrm{FL}_n^{\lambda+\eta,\Box}$ be the functor which classifies objects $(M, \{\Fil^iM\}_{i\in \Z}, \{\phi_i\}_{i\in\Z})$ of $\mathrm{FL}_n^{\lambda+\eta}$ together with a choice of compatible basis $\beta$. It is represented by $\un{G}$, the isomorphism given by extracting the matrix $\mathrm{Mat}_{\gr^{\bullet}\beta,\beta}(\phi_M) $ of $\phi_M$ with respect to the bases $\gr^{\bullet}\beta$ and $\beta$. Since compatible bases exists Zariski locally, the forgetful map $\mathrm{FL}_n^{\lambda+\eta,\Box}\to \mathrm{FL}_n^{\lambda+\eta}$ is an $\un{B}$-torsor (recall from \S\,\ref{intro:LAG} that $\un{B}$ is the Borel subgroup of $\un{G}$ corresponding to matrices which are upper triangular in each embedding).

We conclude by computing the resulting $\un{B}$-action on $\un{G}$: the effect of changing $\beta$ on $\mathrm{Mat}_{\gr^{\bullet}\beta,\beta}(\phi_M)$ is given by the action of $\un{B}$  on $\un{G}$ given by the formula
\begin{equation}
(A\cdot b)^{(j)}\defeq (b^{(j+1)}\big)^{-1} A^{(j)}\ovl{b}^{(j)}
\end{equation}
for all $j\in\cJ$, where $b= (b^{(j)})_{j\in\cJ}\in \un{B}$, $\ovl{b}= (\ovl{b}^{(j)})_{j\in\cJ}$ is the image of $b$ in $\un{B}/\un{U}=\un{T}$, and $A=(A^{(j)})_{j\in\cJ}\in \un{G}$.
\end{proof}

We now discuss the effect of changing the field $K$ by an unramified extension.
Recall that we have fixed an unramified extension $K'/K$, with residue field $k'$ of degree $r=n!$ over $k$.
Tensoring $k'$ over $k$ gives a natural transformation of functors%
\begin{align*}
\BC:{\mathrm{FL}}_n^{\lambda+\eta}&\longrightarrow{\mathrm{FL}}_n^{\lambda'+\eta'}\\
\big(M,\{\Fil^i M\}_i,\{\phi_i\}_i\big)&\longmapsto
\big(M\otimes_k k' ,\{\Fil^i M\otimes_k k'\}_i,\{\phi_i\otimes_k\mathrm{id}\}_i\big)
\end{align*}
where $\lambda'=(\lambda'_{j'})_{j'\in\cJ'}\in X^*(\un{T})^r$ is characterized by $\lambda'_{j'}=\lambda_{j}$ when $j\equiv j'\pmod{f}$ and similarly for $\eta'$.
We have a similar result as Proposition~\ref{thm:repr:FLgp} for ${\mathrm{FL}}_n^{\lambda'+\eta'}$.
Passing to the quotient by the $\un{T}$-shifted conjugation and using our identifications of $\cJ'$ and $\cJ$ with $\Z/f'$ and $\Z/f$ respectively, we deduce a commutative diagram of stacks over $\Spec \F$:
\begin{equation*}
\xymatrix{\tld{\cF\cL}_{\cJ}\ar[r]\ar_{\tld{\BC}}[d]&
\big[\tld{\cF\cL}_{\cJ}\slash{{\sim}_{\un{T}\text{-\textnormal{sh.cnj}}}}\big]\ar^-{\sim}[r]\ar[d]&
\mathrm{FL}_n^{\lambda+\eta}\ar^{\BC}[d]\\
\tld{\cF\cL}_{\cJ'}\ar[r]&
\big[\tld{\cF\cL}_{\cJ'}\slash{{\sim}_{\un{T}^r\text{-\textnormal{sh.cnj}}}}\big]\ar^-{\sim}[r]&
\mathrm{FL}_n^{\lambda'+\eta'}
}
\end{equation*}
where $\tld{\BC}$ is the diagonal embedding compatible with the identification of $(M\otimes_k k')^{(j')}$ and $M^{(j)}$ when $j\equiv j'\pmod{f}$.

\subsection{Breuil--Kisin modules with descent}
\label{subsubBKD}
\label{par:GPD:KM}

In this section, we review Breuil--Kisin modules with descent data, and their necessary properties. %
We follow closely \cite[\S\,3.1]{LLLM2} and \cite[\S\,3.2]{LLL}.
Throughout \S\,\ref{subsubBKD}, $\tau=\tau(s_\cJ,\mu+\eta)$ is a tame inertial type with $\mu+\eta$ being $1$-generic Fontaine--Laffaille.

\vspace{2mm}

Write $e'\defeq p^{f'}-1$ and fix $\pi'\defeq(-p)^{\frac{1}{e'}}\in \ovl{\Q}_p$ such that $(\pi')^{\frac{e'}{e}}=\pi$.  Write $L'\defeq K'(\pi')$, $\Delta'\defeq \Gal(L'/K')\subseteq \Delta\defeq \Gal(L'/K)$.
As in \S\,\ref{sec:notation:GT}, the character $\tld{\omega}_{K'}:\Delta'\ra W(k')^\times$, $g\mapsto \frac{g(\pi')}{\pi'}$ is independent of the choice of $\pi'$ and $(\tld{\omega}_{K'})^{\frac{p^{f'}-1}{p^f-1}}=\tld{\omega}_{K}$. Let $\tau'$ be the inertial type for $K'$ induced from $\tau$. Then we can view $\tau$ as a $\Delta$-representation whose restriction to $\Delta'$ is given by $\tau'$.

For a $p$-adically complete Noetherian $\cO$-algebra $R$, let $\fS_{L', R} \defeq (W(k') \otimes_{\Zp} R)[\![u']\!]$.
The ring $\fS_{L', R}$ is endowed with an action of $\Delta=\Gal(L'/K)$: for any $g$ in $\Delta'$, we have $g(u') \defeq (\tld{\omega}_{K'}(g)\otimes_{\Zp}1_R) u'$ and $g$ acts trivially on the coefficients.
Let $\sigma \in\Gal(L'/\Qp)$ be the lift of the arithmetic Frobenius on $W(k')$ which fixes $\pi'$.
Then $\sigma^f$ acts on $\fS_{L', R}$, by letting $\sigma^f$ act trivially on both $u'$ and $R$, and through the usual action on $W(k')$.
(One checks that the above rule defines a group action of $\Delta$ on $\fS_{L', R}$.)
If we let $v \defeq (u')^{p^{f'}-1}$ then
\[
(\fS_{L', R})^{\Delta = 1} = (W(k) \otimes_{\Zp} R)[\![v]\!].
\]
As usual, we have the endomorphism $\varphi:\fS_{L', R} \ra \fS_{L', R}$ which acts as $\phz$ on $W(k')$, acts trivially on $R$, and sends $u'$ to $(u')^{p}$.

\begin{defn}
\label{defn:FCris:dd}
A \emph{Breuil--Kisin module} over $R$ with height in $[0, h]$ and descent datum of type~$\tau$ is a triple $(\fM, \phi_{\fM},\{\hat{g}\}_{g\in\Delta})$ where:
\begin{enumerate}
\item $\fM$ is a projective $\fS_{L', R}$-module;
\item $\phi_{\fM}:\phz^*(\fM)\ra\fM$ is an injective $\fS_{L', R}$-linear map  whose cokernel is $E(u')^h$-torsion;
\item $\{\hat{g}:\fM\ra\fM\}_{g\in\Delta}$ is the datum of a semilinear $\Delta$-action on $\fM$ compatible with $\phi_\fM$ (in particular this induces an isomorphism $\iota_\fM: (\sigma^f)^*(\fM)\cong\fM$ (cf. \cite[Remark~5.1.4~(1)]{MLM}));
\item for each $0 \leq j' \leq f' - 1$:
\[
\fM^{(j')}/ u' \fM^{(j')}\cong (\tau')^{\vee} \otimes_{\cO} R
\]
as $\Delta'$-representations.
(As for Fontaine--Laffaille modules, we have a decomposition $\fM\cong \oplus_{j'=0}^{f'-1}\fM^{(j')}$ induced by $W(k')\otimes_{\Zp}R\cong \prod_{j'=0}^{f'-1}R$.)
\end{enumerate}

We let $Y^{[0, h], \tau}$ denote the functor on $p$-adically complete Noetherian $\cO$-algebras taking $R$ to the groupoid of Breuil--Kisin modules over $R$ with height in $[0,h]$ and descent data of type $\tau$. (Recall that $\tau'$ denotes the type $\tau$ viewed as a tame inertial type for $K'$.) We also define $Y^{[0, h],\tau'}$ in a similar fashion.
\end{defn}

If $\chi:\Delta'\ra\cO^\times$ we write $\fM^{(j')}_\chi$ to denote the $\chi$-isotypical component of $\fM^{(j)}$.
Since $\Delta/\Delta'$ is cyclic of order $r$, generated by $\sigma^f$, whenever we have $(\fM,\phi_\fM)\in Y^{[0,h],\tau}(R)$ we have an isomorphism
\begin{align}
\label{eq:pull:back}
\fM^{(j')}_{\chi}\stackrel{\sim}{\rightarrow}\big((\sigma^f)^*(\fM)\big)^{(j'+f)}_{\chi}
\stackrel{\sim}{\rightarrow}\fM^{(j'+f)}_{\chi^{p^{f}}}
\end{align}
for each $j'\in\cJ'$ and $\chi:\Delta'\rightarrow \cO^\times$, where the first isomorphism is induced by the obvious one (i.e.~$(\sigma^f)^*$) and the second is induced by $\iota_\fM$.

\begin{defn}
\label{defn:eigenbasis}
An \emph{eigenbasis} (cf.~\cite[Definition 5.1.6]{MLM}) for $(\fM,\phi_\fM)\in Y^{[0,h],\tau}(R)$ is a collection $\beta=(\beta^{(j')})_{j'\in\cJ}$ where each $\beta^{(j')}=(\beta^{(j')}_i)_{1\leq i\leq n}$ is a basis for $\fM^{(j')}$ over $R$ such that $\Delta'$ acts on $\beta^{(j')}_i$ by the character $\chi_i^{-1}$ (defined in equation~(\ref{eq:def:type})) and satisfying
\begin{equation}
\label{eq:eigenbasis}
\iota_{\fM}((\sigma^f)^*(\beta^{(j'-f)}))=\beta^{(j')}
\end{equation}
for all $j'\in \cJ'$, $1\leq i\leq n$.
If $\beta$ is an eigenbasis for $(\fM,\phi_{\fM})\in Y^{[0, h], \tau}(R)$ we write $C^{(j')}_{\fM,\beta}$ to denote the matrix of $\phi_\fM^{(j')}$ with respect to $\beta$, i.e.~the element of $\Mat_n(\fS_{L', R})$ such that
\[
\phi_\fM^{(j')}\Big(\phz^*\big(\beta^{(j'-1)}\big)\Big)=\beta^{(j')}C^{(j')}_{\fM,\beta}.
\]
The notion of eigenbasis, hence the sequence $(C^{(j')}_{\fM,\beta})_{0\leq j\leq f-1}$, depends on the chosen lowest alcove presentation of $\tau$, since the sequence of character $(\chi_i)_{i}$ does, cf.~\cite[Remark 3.2.12]{LLL}.
\end{defn}

As $\cO$ is chosen to be sufficiently large and the order of $\Delta'$ is coprime to $p$, the objects in $Y^{[0, h], \tau}(R)$ have an eigenbasis Zariski locally.

Let $A_{\fM,\beta}=(A_{\fM,\beta}^{(j')})_{j'\in\cJ'}$ be the tuple of matrices $A_{\fM,\beta}^{(j')}\in \Mat_n(R[\![v]\!])$ defined via
\begin{equation}\label{equ: explicit formula for the morphism}
C^{(j')}_{\fM,\beta}=(s'_{\mathrm{or},j'}) (u')^{\bf{a}_{(s_\cJ,\mu)}^{\prime\, (j')}}\,
 A_{\fM,\beta}^{(j')}\, (u')^{-\bf{a}_{(s_\cJ,\mu)}^{\prime\, (j')}}(s'_{\mathrm{or},j'})^{-1}
\end{equation}
where $\bf{a}_{(s_\cJ,\mu)}^{\prime\, (j')}\defeq \sum_{i=0}^{f'-1}\bm{\alpha}'_{(s_\cJ,\mu),-j'+i}p^i$ and $-j'+i$ is taken modulo $f'$.
By \cite[\S\,5.1]{MLM}, the matrices $C^{(j')}_{\fM,\beta}$, $A_{\fM,\beta}^{(j')}$ only depend on $j'$ modulo $f$ (see the paragraphs after Definition 5.1.6 and Remark 5.1.7 in \emph{loc.~cit.}; note that ($C^{(j')}_{\fM,\beta})_{j'\in\cJ'}$ and $(A_{\fM,\beta}^{(j')})_{j'\in\cJ'}$ do depend on the choice of the lowest alcove presentation $(s_\cJ,\mu)$
of the tame inertial type $\tau$, see~\cite[Remark 5.1.5]{MLM}, \cite[Remark 3.1.12]{LLL}).
\subsection{\'Etale $\varphi$-modules}
\label{subsub:Phi-mod}

This section follows \cite[\S\,5.4]{MLM}. %
Recall that we have fixed a tame inertial type $\tau=\tau(s_\cJ,\mu+\eta)$ with $\mu+\eta$ being $1$-generic Fontaine--Laffaille.

\vspace{2mm}

Let $\cO_{\cE}$ denote the $p$-adic completion of $(W(k)[\![v]\!])[1/v]$, endowed with a Frobenius endomorphisms $\phz$ which extends the Frobenius $\phz$ on $W(k')$ and satisfies $\phz(v) = v^p$.
Let $R$ be a $p$-adically complete Noetherian $\cO$-algebra.
The ring $\cO_{\cE}\widehat{\otimes}_{\Zp}R$ is naturally endowed with a Frobenius endomorphism $\phz$ and we write $\Phi\text{-}\Mod^{\text{\'et},n}(R)$ for the groupoid of \'etale $\phz$-modules over $\cO_{\cE}\widehat{\otimes}_{\Zp}R$.
Its objects are projective modules $\cM$ of rank $n$ over $\cO_{\cE}\widehat{\otimes}_{\Zp}R$, endowed with a
$\cO_{\cE}\widehat{\otimes}_{\Zp}R$-linear isomorphism $\phi_{\cM}:\phz^*(\cM)\stackrel{\sim}{\longrightarrow}\cM$. %
As usual, we obtain a category fibered in groupoids $\Phi\text{-}\Mod^{\text{\'et},n}$ over $p$-adically complete Noetherian $\cO$-algebras.
Given an object $(\cM,\phi_{\cM})\in \Phi\text{-}\Mod^{\text{\'et},n}(R)$ we have an $R$-linear decomposition $\cM\cong \oplus_{j\in\cJ}\cM^{(j)}$ together with $R$-linear and $(v\mapsto v^p)$-semilinear isomorphisms
\[
\phi^{(j)}_\cM: \cM^{(j-1)}\ra \cM^{(j)}.
\]

\vspace{2mm}

Now let $\cO_{\cE,L'}$ denote the $p$-adic completion of $(W(k')[\![u']\!])[1/u']$, endowed with a Frobenius endomorphisms $\phz$ extending the Frobenius $\phz$ on $W(k')$ and such that $\phz(u') = (u')^p$.
We have an analogous definition for the groupoid $\Phi\text{-}\Mod^{\text{\'et},n}_{dd,L'}(R)$ of \'etale $\phz$-modules over $\cO_{\cE,L'}\widehat{\otimes}_{\Zp}R$ of rank $n$ with descent data. %
(An object of $\Phi\text{-}\Mod^{\text{\'et},n}_{dd,L'}(R)$ is the datum of an \'etale $\phz$-module over $\cO_{\cE,L'}\widehat{\otimes}_{\Zp}R$ of rank $n$, together with a collection of $R$-linear isomorphisms $\hat{g}:\cM\ra\cM$ satisfying the properties of Definition~\ref{defn:FCris:dd}, replacing $\fM$ and $\phi_\fM$ by $\cM$ and $\phi_{\cM}$, respectively. Note that the $\Delta$-action on $W(k')[\![u']\!]$ extends, by continuity, to a continuous action on $\cO_{\cE,L'}$.).
We write $\Phi\text{-}\Mod^{\text{\'et},n}_{dd,L'}$ for the corresponding groupoid-valued functor over $p$-adically complete Noetherian $\cO$-algebra.

As before, given an object $(\cM,\phi_{\cM},\{\hat{g}\}_{g\in \Delta})\in \Phi\text{-}\Mod^{\text{\'et},n}_{dd,L'}(R)$ we have an $R$-linear decomposition $\cM\cong \oplus_{j'\in\cJ'}\cM^{(j')}$ together with $R$-linear and $(u'\mapsto (u')^p)$-semilinear isomorphisms
\[
\phi^{(j')}_\cM: \cM^{(j'-1)}\ra \cM^{(j')}.
\]
Moreover, for all $g\in\Delta'$ we have an $R$-linear, $\phi_\cM^{(j')}$-compatible automorphism $\hat{g}^{(j')}: \cM^{(j')}\ra \cM^{(j')}$ giving an $R$-linear action of $\Delta'$ on each factor $\cM^{(j')}$.

\vspace{2mm}

If $(\fM,\phi_{\fM})\in {Y}^{[0,n-1], \tau}(R)$, then $\fM \otimes_{W(k')[\![u']\!]} \cO_{\cE, L'}$, endowed with a Frobenius and descent data induced from those on $\fM$, is an object of $\Phi\text{-}\Mod^{\text{\'et},n}_{dd,L'}(R)$. This produces a natural transformation of functors ${Y}^{[0,n-1], \tau}\ra \Phi\text{-}\Mod^{\text{\'et},n}_{dd,L'}$.
Moreover, since $v = (u')^{p^{f'} - 1}$, taking $\Delta$-fixed elements produces a natural transformation $\Phi\text{-}\Mod^{\text{\'et},n}_{dd,L'}\ra \Phi\text{-}\Mod^{\text{\'et},n}$ between functors over $p$-adically complete Noetherian $\cO$-algebras.
Composition of the two functors above produces a morphism of groupoid-valued functors over $p$-adically complete  Noetherian $\cO$-algebras:
\begin{align}
\label{eq:KMtoPHI}
\eps_\tau: Y^{[0,n-1], \tau}&\longrightarrow \Phi\text{-}\Mod^{\text{\'et},n}\\
\nonumber
(\fM,\phi_{\fM})&\mapsto (\fM \otimes_{W(k')[\![u']\!]}\cO_{\cE, L'},\phi_{\fM}\otimes_{W(k')[\![u']\!]}1_{\cO_{\cE, L'}} )^{\Delta=1}.
\end{align}
Recall that we have fixed a lowest alcove presentation $(s_\cJ,\mu)$ of $\tau$. By \cite[Proposition 5.4.1 and 5.4.3]{MLM}, if $\mu+\eta$ is $n$-generic Fontaine--Laffaille, then the morphism $\eps_\tau$ is a closed immersion of stacks over $\Spf\, \cO$.

\subsection{Galois representations and $\phz$-modules}
\label{sub:MFSandGR}
In this section we study the relations between the groupoids introduced above and $p$-adic Galois representations.
We keep the setting of the previous sections; in particular $\tau=\tau(s_\cJ,\mu+\eta)$ is a tame inertial type with $\mu+\eta$ being $1$-generic Fontaine--Laffaille.

If $R$ is a complete local Noetherian $\cO$-algebra with residue field $\F$, we write $\Rep^n_{R}(G_K)$ for the groupoid of $p$-adic  representations of $G_K$ on free $R$-modules of rank $n$, and %
we have the anti-equivalence of groupoids of J--M. Fontaine:
\[
\bV_K^*:\Phi\text{-}\Mod^{\text{\'et},n}(R) \ra \Rep^n_R(G_{K_{\infty}}),
\]
which induces $$T_{dd}^*:Y^{[0,n-1], \tau}(R)\rightarrow \Rep^n_R(G_{K_{\infty}})$$ as the composite of the functor~(\ref{eq:KMtoPHI}) followed by $\bV_K^*$.
By~\cite[Proposition 5.4.3]{MLM} we see that $T_{dd}^*$ is a fully faithful functor if $\mu+\eta$ is $n$-generic Fontaine--Laffaille.

If $R$ is an $\F$-algebra, by the main result of \cite[Th\'eor\`eme 6.1]{Fontaine-Laffaille} we have a fully faithful contravariant functor
\[
\Tcris: \mathrm{FL}_n^{\lambda+\eta}(R)\ra \Rep^n_{R}(G_K)
\]
(see also \cite[Theorem 2.1.3]{HLM}).

\begin{rmk}
Note that the definition of Fontaine--Laffaille modules, Definition~\ref{def: Fontaine--Laffaille module}, is a bit more general than the one in \cite{Fontaine-Laffaille}, since the original definition of a Fontaine--Laffaille module $M$ in \cite{Fontaine-Laffaille} further requires $\min\{i\in\Z\mid\Fil^iM=0\}-\max\{i\in\Z\mid\Fil^iM=M\}\leq p-1$. However, we still have a fully faithful functor $\Tcris: \mathrm{FL}_n^{\lambda+\eta}(R)\ra \Rep^n_{R}(G_K)$ with our definition, Definition~\ref{def: Fontaine--Laffaille module}, which is a minor variation of the results in \cite[Th\'eor\`eme 6.1]{Fontaine-Laffaille} by twisting an appropriate Lubin--Tate character.
\end{rmk}

We further define the map
\begin{equation}
\label{eq:to:Gal}
\rhobar_{\bullet,\lambda+\eta}:%
\xymatrix{
\tld{\cF\cL}_\cJ(R)\ar[r]&
\big[\tld{\cF\cL}_\cJ\slash{{\sim}_{\un{T}\text{-\textnormal{sh.cnj}}}}\big](R)\ar^-{\sim}[r]
&\FL_n^{\lambda+\eta}(R)\ar^-{\Tcris}[r]
& \Rep^n_{R}(G_K)
}
\end{equation}
(where the first arrow is the natural quotient map, and the second is described in Proposition~\ref{thm:repr:FLgp} above). We write $\rhobar_{x,\lambda+\eta}$ for the image of $x\in\tld{\cF\cL}_\cJ(R)$ under the map above.

For convenience, we record the effect of the functor $\Tcris$ on Fontaine--Laffaille modules of rank one, which will be used later.
\begin{lemma}
\label{lem:FL:rk1}
Let $M$ be a Fontaine--Laffaille module with $\F$-coefficient.
Assume that $M$ has rank one; for each $j\in \cJ$, let $\lambda_{j,1}\in[0,p-2]$ be the unique integer such that $\gr^{\lambda_{j,1}}(M^{(j)})\neq 0$.
Then $\Tcris(M)|_{I_K}\cong \omega_f^{\sum_{j\in\cJ}\lambda_{f-j,1}p^j}$.
\end{lemma}
\begin{proof}
This is a direct consequence of \cite[Th\'eor\`eme 5.3(iii)]{Fontaine-Laffaille}.
\end{proof}
\clearpage{}%
\clearpage{}%
\section{The geometry of $\tld{\cF\cL}_\cJ$}
\label{sec:geo:FL}
In this section, we construct and study an explicit partition $\cP_{\cJ}$ on $\tld{\cF\cL}_{\cJ}$ by locally closed subschemes of $\tld{\cF\cL}_{\cJ}$.

Throughout this section, $R$ denotes a Noetherian $\F$-algebra.
Since all schemes are defined over $\Spec \F$ we omit the subscript $\bullet_{\F}$ from the notation when considering the base change to $\F$ of an object $\bullet$ defined over $\cO$ (e.g.~$\GL_{n,\F}$ will be denoted by $\GL_n$ and so on).
This shall cause no confusion.

\subsection{A partition on $\tld{\cF\cL}_\cJ$}
\label{sub:finest:elements}
Recall that $\tld{\cF\cL}$ denotes the representative for the sheafification of the functor $R\mapsto U(R)\backslash\GL_n(R)$ on the $fpqc$ site of Noetherian $\F$-algebras.

In this section, we introduce an explicit partition $\cP$ on $\tld{\cF\cL}$ which admits a natural interpretation related to the usual Bruhat decomposition on the flag variety (see Proposition~\ref{prop: refinement of Bruhat partition}). At the end of this section, we use the partition $\cP$ on $\tld{\cF\cL}$ to define a partition $\cP_\cJ$ on $\tld{\cF\cL}_\cJ$.

We write $\mathbf{n}\defeq \{1,\dots,n\}$ and denote the power set of $\mathbf{n}$ by $\wp(\mathbf{n})$.
Let $S\subseteq \mathbf{n}$ be a subset.
For each $A\in \GL_n(R)$ we write $f_S(A)$ for the minor of $A$ with rows in $\{n-\#S+1,\dots,n\}$  and with columns in $S$. For each $w\in W$ and each $A\in\GL_n(R)$, our convention says that the $i$-th row (resp.~the $i$-th column) of $A$ is the same as the $w(i)$-th row of $wA$ (resp.~the $w(i)$-th column of $Aw^{-1}$).
Note that for each subset $S\subseteq\mathbf{n}$ the function $f_{w(S)}$ can be identified with the composition
$$\GL_{n}\xrightarrow{\cdot w}\GL_{n}\xrightarrow{f_S}\bA^1$$
up to $\pm$ signs. For each $S\subseteq \mathbf{n}$, it is clear that the map $f_S: \GL_{n}\rightarrow \bA^1$ descends to a map of schemes
\[
f_S:\tld{\cF\cL}\rightarrow\bA^1.
\]

For each $S\subseteq \mathbf{n}$, we let $\cH_S\subsetneq \tld{\cF\cL}$ be the vanishing locus of $f_S$.
From now on, we will consider intersection, union and complement of constructible subset(s) of $\tld{\cF\cL}$. We use the notation $\cdot^{\rm{c}}$ for the complement of a subset of $\mathbf{n}$, or the complement of a constructible subset of $\tld{\cF\cL}$.
For each $K\subseteq \wp(\mathbf{n})$, we define the locally closed subscheme
\begin{equation*}
\cC_K\defeq\underset{S\in K}{\bigcap}\cH_S\cap \underset{S\notin K}{\bigcap}\cH_S^{\rm{c}}.
\end{equation*}
Note that $\cC_K$ can be empty for certain choices of $K\subseteq \wp(\mathbf{n})$.

We define $\cP$ to be the set of non-empty locally closed subschemes of $\tld{\cF\cL}$ of the form $\cC_K$ for some choice of $K\subseteq \wp(\mathbf{n})$.

\begin{lemma}\label{lem: elementary, elements}
The subschemes $\cH_S$ and $\cC_K$ satisfy the following elementary properties.
\begin{enumerate}
\item If $K,K'\subseteq\wp(\mathbf{n})$ with $K\neq K'$, then $\cC_K\cap \cC_{K'}=\emptyset$;
\item $\tld{\cF\cL}=\bigcup_{K\subseteq \wp(\mathbf{n})}\cC_K$;
\item For each $S\subseteq\mathbf{n}$, $\cH_S=\bigcup_{S\in K}\cC_K$;
\item For each $K\subseteq\wp(\mathbf{n})$, $\bigcap_{S\in K}\cH_S=\bigcup_{K\subseteq K'}\cC_{K'}$;
\item $\cC_{\emptyset}$ is the unique element in $\cP$ which is an open subscheme of $\tld{\cF\cL}$.
\end{enumerate}
\end{lemma}

\begin{proof}
These are immediate consequences from the definitions.
\end{proof}

By Lemma~\ref{lem: elementary, elements} the set $\cP$ forms a topological partition of $\tld{\cF\cL}$ by reduced locally closed subschemes.

Let $S_\bullet$ denote a sequence $\mathbf{n}=S_1\supset S_2\supset \cdots \supset S_n$ satisfying $\# S_i=n-i+1$ for all $1\leq i\leq n$.
For convenience, we call a sequence $S_\bullet$ as above a \emph{strictly decreasing sequence}. For each $w\in W$, we associate a strictly decreasing sequence $S_{\bullet,w}$ by
$$S_{i,w}=w^{-1}(\{i,i+1,\cdots,n-1,n\})$$
for each $1\leq i\leq n$. By abuse of the notation, we also write $S_{\bullet, w}$ for the subset of $\wp(\mathbf{n})$ consisting of $S_{i,w}$ for all $1\leq i\leq n$.
Then it is easy to see that there is a bijection
\begin{equation}\label{eq:bjc:sis}
W\stackrel{\sim}{\longrightarrow}\left\{ \text{strictly decreasing sequences}\right\},\,\,
w\longmapsto S_{\bullet,w}.
\end{equation}

For $\al\in\Phi^+$ we write $u_\al:~U\twoheadrightarrow \bG_a$ for the projection to the $\al$-entry.
\begin{lemma}\label{lem: explicit projection}
For each $w\in W$, the natural projection
$$ w_0 B w_0 w=T w_0 U w_0 w\cong T\times w_0 U w_0 w\twoheadrightarrow T$$
is given by the restriction of
\begin{equation}\label{equ: explicit formula for torus}
\Diag\left(\pm f_{S_{1,w}}f_{S_{2,w}}^{-1},\cdots,\pm f_{S_{n-1,w}}f_{S_{n,w}}^{-1},\pm f_{S_{n,w}}\right),
\end{equation}
and the composition
$$ w_0 B w_0 w=T w_0 U w_0 w\cong T\times w_0 U w_0 w\twoheadrightarrow w_0 U w_0 w \cong U\xrightarrow{u_\al} \bG_{a}$$
is given by the restriction of
\begin{equation}\label{equ: explicit formula for root}
\pm f_{S_{w_0(i)+1,w}\sqcup\{w^{-1}w_0(i^\prime)\}}f_{S_{w_0(i),w}}^{-1}
\end{equation}
for each $\al=(i,i^\prime)\in\Phi^+$ with $1\leq i<i^\prime\leq n$. Here $\pm$ means up to a sign.
\end{lemma}
\begin{proof}
This is a simple computation of minors of matrices in $w_0Bw_0w(R)$.
\end{proof}

For a strictly decreasing sequence $S_\bullet$ we define the following open subschemes of $\tld{\cF\cL}$:
\begin{equation}
\cM_{S_\bullet}^\circ\defeq \underset{1\leq i\leq n}{\bigcap}\cH_{S_i}^{\rm{c}}.
\end{equation}
If the strictly decreasing sequence is given by $S_{\bullet,w}$ for some $w\in W$ (which is always possible from the bijection in (\ref{eq:bjc:sis})) we write $\cM^\circ_w$ for $\cM^\circ_{S_{\bullet,w}}$. Note that $\cM_{w}^\circ$ is a topological union of elements of $\cP$. More precisely,
\begin{equation}\label{equ: cM_w is union of elements}
\cM^\circ_w=\bigcup_{K\subseteq S_{\bullet,w}^{\rm{c}}}\cC_K.
\end{equation}
Moreover, we observe that $\underset{\substack{w\in W}}{\bigcap}\cM_{w}^\circ=\cC_{\emptyset}$ is the unique element in $\cP$ which is an open subscheme of $\tld{\cF\cL}$. We also consider
\begin{equation}
\overline{\cM}_{w}\defeq \cC_{S_{\bullet,w}^{\rm{c}}}
\end{equation}
which is clearly a closed subscheme of $\cM^\circ_w$. By computing different minors of matrices in $T w(R)$, one easily check that $\overline{\cM}_{w}$ is actually the schematic image of $T w$ in $\cM^\circ_w$, both characterized by the vanishing of $f_S$ for all $S\notin S_{\bullet,w}$.
We now see that the open subschemes $\cM^\circ_w$ have a more familiar description in terms of Schubert cells.

\begin{lemma}\label{lem: hypersurfaces are Schubert}
Let $w\in W$. Then we have
\begin{equation}\label{equ: open cell}
\cM_w^\circ=U\backslash U w_0 B w_0w,
\end{equation}
and
\begin{equation}\label{equ: hypersurface}
\cH_{S_{i+1,w}}=\overline{ U\backslash U s_iw_0 B w_0w}
\end{equation}
for each $1\leq i\leq n-1$, where $s_i=(i,i+1)\in W$.
\end{lemma}

\begin{proof}
The RHS of (\ref{equ: open cell}) is clearly inside the LHS as for all $1\leq i\leq n$ we have
\begin{equation}\label{equ: simple nonvanishing}
f_{S_{i,w}}\neq 0
\end{equation}
on $w_0Bw_0w\subsetneq\GL_n$, and hence on $U\backslash U w_0 B w_0w\subsetneq\tld{\cF\cL}$. Conversely, any matrix $A\in \GL_n(R)$ satisfying (\ref{equ: simple nonvanishing}) for each $1\leq i\leq n$ can be written as $uw_0bw_0w$ for some $u\in U(R)$ and $b\in B(R)$, and thus the LHS of (\ref{equ: open cell}) is also in the RHS. Hence the equality (\ref{equ: open cell}) holds. It follows from the definition of $\cM_w^\circ$, (\ref{equ: open cell}) and the property of Bruhat stratification that
$$
\bigcup_{i=1}^{n-1}\cH_{S_{i+1,w}}=(\cM_w^\circ)^{\rm{c}}=\left(U\backslash U w_0 B w_0w\right)^{\rm{c}}=\underset{w'<w_0}{\bigsqcup}U\backslash U w' B w_0w=\bigcup_{i=1}^{n-1}\overline{ U\backslash U s_iw_0 B w_0w}.
$$
Hence we observe that both sides of (\ref{equ: hypersurface}) are irreducible components of $(\cM_w^\circ)^{\rm{c}}$, and it suffices to notice that $f_{S_{k,w}}\neq0$ on $U\backslash U s_iw_0 B w_0w$ for all $k\in\mathbf{n}\setminus\{i+1\}$ by using Lemma~\ref{lem: explicit projection} and the fact that
$$U\backslash U s_iw_0 B w_0w\hookrightarrow U\backslash U w_0 B w_0s_iw\xleftarrow{\sim} w_0 B w_0s_iw,$$ which completes the proof.
\end{proof}

\begin{lemma}\label{lem: insert}
Let $A \in \GL_n(R)$. Suppose that we have a sequence $\emptyset \subsetneq S_n \subsetneq S_{n-1} \subsetneq \cdots \subsetneq S_k \subsetneq S^\prime\subseteq \mathbf{n}$ such that $\#S_\ell=n+1-\ell$ for each $k\leq \ell\leq n$. Assume further that $f_{S^\prime}(A) \neq 0$ and $f_{S_\ell}(A) \neq 0$ for each $k\leq \ell\leq n$. Then there exists $i \in S^\prime\setminus S_k$ such that $f_{S_k \sqcup \{i\}}(A) \neq 0$.
\end{lemma}
\begin{proof}
Upon replacing $A$ with $Aw$ for a certain $w\in W$, we may assume without loss of generality that
$$S_\ell=\{\ell,\dots,n-1,n\}$$
for each $k\leq \ell\leq n$. We may assume further that $S^\prime=\mathbf{n}$ by replacing $A$ with its submatrix given by $\{n-\#S^\prime+1,\dots,n\}$-th rows and $S^\prime$-th columns. Then there exists $u\in U(R)$ such that the submatrix of $uA$ given by $\{1,\dots,k-1\}$-th rows and $\{k,\dots,n\}$-th columns, is zero. We consider the submatrix $A^\prime$ of $uA$ given by $\{1,\dots,k-1\}$-th rows and columns. Using that $f_{S^\prime}(A)\neq 0$ and $f_{S_k}(uA)=f_{S_k}(A)\neq 0$, we deduce that $\mathrm{det}(A^\prime)=f_{S^\prime}(A)f_{S_k}(A)^{-1}\neq 0$, and that $f_{S_k\sqcup\{i\}}(A)f_{S_k}(A)^{-1}$ equals the $(k-1,i)$-entry of $A^\prime$ for each $1\leq i\leq k-1$.
As $\mathrm{det}(A^\prime)\neq 0$, there exists $1\leq i\leq k-1$ such that $(k-1,i)$-entry of $A^\prime$ is non-zero, and hence $f_{S_k\sqcup\{i\}}(A)\neq 0$.
\end{proof}

\begin{prop}\label{prop: union of open elements}
Let $\Sigma\subseteq\wp(\mathbf{n})$ be a subset contained in some strictly decreasing sequence.
Then we have
\begin{equation}\label{equ: affine cover by Borel}
\bigcap_{S\in\Sigma}\cH_S^{\rm{c}} = \underset{S_\bullet \supseteq\Sigma}{\bigcup}\cM_{S_\bullet}^\circ
\end{equation}
where $S_\bullet$ runs through all strictly decreasing sequences that contain $\Sigma$. In particular, the set $\left\{\cM_{w}^\circ\mid w\in W\right\}$ forms an affine open cover of $\tld{\cF\cL}$.
\end{prop}

\begin{proof}
The inclusion $\supseteq$ follows from the definition of $\cM_{S_\bullet}$.
We now prove the inclusion $\subseteq$ by induction on $\#\Sigma$. It suffices to show that, for each $A \in \bigcap_{S\in\Sigma}\cH_S^{\rm{c}}(R)$, there exists a strictly decreasing sequence $S_\bullet$ containing $\Sigma$ such that $A \in \cM_{S_\bullet}^\circ(R)$. We pick an arbitrary $A \in \bigcap_{S\in\Sigma}\cH_S^{\rm{c}}(R)$. If $\#\Sigma=n$, we can simply take $S_\bullet\defeq \Sigma$. If $\#\Sigma < n$, then by Lemma~\ref{lem: insert} there exists $\Sigma'\subseteq\wp(\mathbf{n})$ satisfying the following conditions:
\begin{enumerate}
\item $\Sigma'$ is contained in a certain strictly decreasing sequence;
\item $\Sigma'\supsetneq \Sigma$ and $\#\Sigma'=\#\Sigma+1$;
\item $A \in \bigcap_{S\in\Sigma'}\cH_S^{\rm{c}}(R)$.
\end{enumerate}
We apply our inductive assumption to $\Sigma'$ and obtain a strictly decreasing sequence $S_\bullet$ containing $\Sigma'$ (hence $\Sigma$ as well) such that $A \in \cM_{S_\bullet}^\circ(R)$. This finishes the proof of the inclusion $\subseteq$.

It follows from (\ref{equ: open cell}) that $\cM_{w}^\circ=U\backslash U w_0 B w_0w\cong B$ as a scheme, hence it is affine. The fact that
$$
\tld{\cF\cL}=\underset{w\in W}{\bigcup}\cM_{w}^\circ
$$
is the special case of (\ref{equ: affine cover by Borel}) when $\Sigma=\emptyset$.
\end{proof}

\begin{lemma}\label{lem: open cover}
Let $K\subseteq \wp(\mathbf{n})$ with $\cC_K\neq \emptyset$. Then
\begin{enumerate}[label=(\roman*)]
\item
\label{it: open cover:i}
for each strictly decreasing sequence $S_\bullet$, we have $\cC_K\subseteq \cM^\circ_{S_\bullet}$ if and only if $S_\bullet\cap K=\emptyset$;
\item
\label{it: open cover:ii}
there exists a strictly decreasing sequence $S_\bullet$ such that $\cC_K\subseteq \cM^\circ_{S_\bullet}$;
\item
\label{it: open cover:iii}
$\wp(\mathbf{n})\setminus K=\bigcup_{S_\bullet\cap K=\emptyset}S_\bullet$ and the following equality holds
$$\cC_K=\underset{S\in K}{\bigcap}\cH_S\cap\underset{S_\bullet\cap K=\emptyset}{\bigcap}\cM^\circ_{S_\bullet};$$
\item
\label{it: open cover:iv}
$\cC_K$ is affine.
\end{enumerate}
\end{lemma}
\begin{proof}
Note that \ref{it: open cover:i} follows directly from the definition of $\cC_K$ and $\cM^\circ_{S_\bullet}$. In order to prove \ref{it: open cover:ii} and \ref{it: open cover:iii}, we pick an arbitrary $A\in\cC_K(R)$ for some Noetherian $\F$-algebra $R$. It follows from Lemma~\ref{lem: insert} (by taking $k=0$ and $\Sigma_0=\mathbf{n}$) that there exists a strictly decreasing sequence $S_\bullet$ such that $f_{S_i}(A)\neq 0$ for all $1\leq i\leq n$. This means that $S_\bullet\cap K=\emptyset$ and hence $\cC_K\subseteq \cM^\circ_{S_\bullet}$ by \ref{it: open cover:i}, which implies \ref{it: open cover:ii}. It follows from Proposition~\ref{prop: union of open elements} (for $\Sigma=\{S\}$) that, for each $S\notin K$, there exists strictly decreasing sequence $S_\bullet$ containing $S$ such that $f_{S_i}(A)\neq 0$ for all $1\leq i\leq n$.
Hence $\cC_K\subseteq \cM^\circ_{S_\bullet}\subseteq \cH_S^{\rm{c}}$ or equivalently $S\in S_\bullet\subseteq \wp(\mathbf{n})\setminus K$. We deduce that $\wp(\mathbf{n})\setminus K=\bigcup_{S_\bullet\cap K=\emptyset}S_\bullet$, which implies (using the definition of $\cM^\circ_{S_\bullet}$)
$$
\underset{S_\bullet\cap K=\emptyset}{\bigcap}\cM^\circ_{S_\bullet}=\underset{S\in S_\bullet\subseteq \wp(\mathbf{n})\setminus K}{\bigcap}\cH_S^{\rm{c}}=\underset{S\notin K}{\bigcap}\cH_S^{\rm{c}}.
$$
Hence we finish the proof of \ref{it: open cover:iii} using the definition of $\cC_K$. Finally, note that $\cM^\circ_{S_\bullet}$ is affine for each strictly decreasing sequence $S_\bullet$ and it is easy to see that $\bigcap_{S_\bullet\cap K=\emptyset}\cM^\circ_{S_\bullet}$ appeared in \ref{it: open cover:iii} is still affine, so that \ref{it: open cover:iv} follows from \ref{it: open cover:iii}.
\end{proof}

Let $\Omega\subseteq\Phi^+$ be an arbitrary subset of the set of positive roots. We write $U_\Omega\subseteq U$ for the closed subscheme of $U$ defined by the condition that the $\al$-entry is zero for each $\al\in\Phi^+\setminus\Omega$. It is clear that the composition $w T U_\Omega w^\prime\hookrightarrow \GL_{n}\twoheadrightarrow\tld{\cF\cL}$ factors through $w T U_\Omega w^\prime\hookrightarrow\cM_{ww^\prime}^\circ$, as for each $S\in S_{\bullet,ww^\prime}$ the function $f_S$ is invertible on $w T U_\Omega w^\prime$ .

\begin{lemma}\label{lem: is union of elements}
The schematic image of $w T U_\Omega w^\prime$ in $\cM_{ww^\prime}^\circ$ is integral of the form
$$\cM_{ww^\prime}^\circ\cap\bigcap_{S\in K}\cH_S$$
for some $K\subseteq \wp(\mathbf{n})$ with $K\cap S_{\bullet,ww'}=\emptyset$, and is a topological union of elements of $\cP$. If moreover $\cC_K\neq\emptyset$, then $\cC_K$ is the unique element of $\cP$ that is an open subscheme of the schematic image.
\end{lemma}

\begin{proof}
Upon shrinking $\Omega$ to a smaller subset, we may assume without loss of generality that
\begin{equation}\label{equ: nondegenerate subset}
U w T U_\Omega w^\prime\cong U\times (w T U_\Omega w^\prime)
\end{equation}
or equivalently, $w(\al)<0$ for each $\al\in\Omega$, and in particular $w T U_\Omega w^\prime$ is naturally isomorphic to its schematic image in $\cM_{ww^\prime}^\circ$. Hence the schematic image of $w T U_\Omega w^\prime$ in $\cM_{ww^\prime}^\circ$ is integral as $w T U_\Omega w^\prime$ is. We may rewrite
$$w T U_\Omega w^\prime=w_0 T (w_0w U_\Omega w^{-1}w_0)w_0ww^\prime\hookrightarrow w_0 T U w_0ww^\prime.$$
Upon modifying notation, it suffices to show that the schematic image of $w_0 T U_\Omega w_0w$ in $\cM_w^\circ$ (written $X$ for short) is a topological union of elements in $\cP$, for each $\Omega\subseteq\Phi^+$ and each $w\in W$. It follows from Lemma~\ref{lem: explicit projection} that the projection
$$w_0 T U w_0 w\twoheadrightarrow U\xrightarrow{u_\al} \bG_{a}$$
is given by
$$f_{S_{w_0(i)+1,w}\sqcup\{w^{-1}w_0(i^\prime)\}}f_{S_{w_0(i),w}}^{-1}|_{w_0 T U w_0 w}$$
for each $\al=(i,i^\prime)\in\Phi^+$ with $1\leq i<i^\prime\leq n$. Therefore we conclude that $w_0 T U_\Omega w_0w$ is the closed subscheme of $w_0 T U w_0 w$ characterized by the condition
$$f_{S_{w_0(i)+1,w}\sqcup\{w^{-1}w_0(i^\prime)\}}=0$$
for each $\al\in\Phi^+\setminus\Omega$, which implies that the $X$ is characterized in $\cM_w^\circ$ by the same condition. In other words, we have
$$X=\cM_w^\circ\cap\bigcap_{S\in K_1}\cH_S$$
where $K_1\defeq \{S_{w_0(i)+1,w}\sqcup\{w^{-1}w_0(i^\prime)\}\mid \al\in\Phi^+\setminus\Omega\}$. In particular, $X$ is a topological union of elements in $\cP$. Now we define $K\defeq \{S\subseteq\mathbf{n}\mid f_S|_X=0\}$ and note that $K_1\subseteq K\subseteq\wp(\mathbf{n})$. As $X$ is an integral scheme, we may write $L_X$ for its function field and notice that $f_S|_X$ is a non-zero element in $L_X$ for each $S\notin K$. Consequently, the open subscheme $\cC_K\subseteq X$ equals the non-vanishing locus of $\underset{S\notin K}{\prod}f_S$ inside the integral scheme $X$.
In particular, $\cC_K$ is the unique element of $\cP$ which is an open dense subscheme of $X$.
The proof is thus finished.
\end{proof}

\begin{defn}\label{def: Sch var}
For $w,u\in W$ we define
\[
\tld{\cS}^\circ(w,u)\defeq U\backslash U w B u=U\backslash B w U u\subseteq\tld{\cF\cL}
\] and write $\tld{\cS}(w,u)$ for the closure of $\tld{\cS}^\circ(w,u)$ inside $\tld{\cF\cL}$.
We call  $\tld{\cS}^\circ(w,u)$ (resp.~$\tld{\cS}(w,u)$) the \emph{Schubert cell} (resp.~the \emph{Schubert variety}) associated with the pair $(w,u)\in W\times W$.
\end{defn}

\begin{prop}\label{prop: refinement of Bruhat partition}
The partition $\cP$ of $\tld{\cF\cL}$ is the coarsest common refinement of the partitions $\{\tld{\cS}^\circ(w,u)\mid w\in W\}$ for all $u\in W$. In particular, each $\cC_K$ with $\cC_K\neq\emptyset$ uniquely determines a map $\delta: W\rightarrow W$ such that
$$\cC_K=\underset{u\in W}{\bigcap}\tld{\cS}^\circ(\delta(u),u).$$
\end{prop}
\begin{proof}
We use the notation $\cP^\prime$ for the coarsest common refinement of the partitions $\{\tld{\cS}^\circ(w,u)\mid w\in W\}$ for all $u\in W$, and we will show that $\cP^\prime$ coincides with $\cP$.

We recall from Lemma~\ref{lem: hypersurfaces are Schubert} that both $\cH_S$ and $\cH_S^{\rm{c}}$ are union of elements in $\cP^\prime$ for each $\emptyset\neq S\subsetneq\mathbf{n}$. This implies that each element of $\cP$ (which is defined by intersection of locally closed subschemes of the form $\cH_S$ or $\cH_S^{\rm{c}}$) is a topological union of elements in $\cP^\prime$. Hence $\cP^\prime$ is finer than $\cP$.

As a special case of Lemma~\ref{lem: is union of elements}, we deduce that each Schubert cell $\tld{\cS}^\circ(w,u)=U\backslash U w T U u$ is a topological union of elements in $\cP$. This implies that $\cP$ is finer than $\cP^\prime$.

We have already shown that $\cP=\cP^\prime$. Now we fix a $\cC_K\neq\emptyset$. Then for each $u\in W$, as $\{ \tld{\cS}^\circ(w,u)\mid w\in W\}$ is a partition of $\tld{\cF\cL}$, there exists a unique $\delta(u)\in W$ such that $\cC_K\subseteq \tld{\cS}^\circ(\delta(u),u)$. Therefore we have
\begin{equation}\label{equ: inclusion of subscheme}
\cC_K\subseteq\underset{u\in W}{\bigcap}\tld{\cS}^\circ(\delta(u),u).
\end{equation}
But the locally closed subschemes of $\tld{\cF\cL}$ of the form $\bigcap_{u\in W}\tld{\cS}^\circ(\delta(u),u)$ for some $\delta: W\rightarrow W$ clearly form a partition of $\tld{\cF\cL}$ (note that $\bigcap_{u\in W}\tld{\cS}^\circ(\delta(u),u)$ could be empty for some choice of $\delta$). Indeed, they exhaust all possible elements of $\cP^\prime$. As we know that $\cP=\cP^\prime$, the inclusion (\ref{equ: inclusion of subscheme}) is necessarily an equality. Hence we finish the proof.
\end{proof}

\subsubsection{Product structures}
\label{subsubsec:PS}

We recall the set $\cJ$ from \S\,\ref{subsec:notation}. For each $K_\cJ=(K_j)_{j\in\cJ}\subseteq \wp(\mathbf{n})^f$, we define the following (possibly empty) locally closed subscheme of $\tld{\cF\cL}_\cJ$
$$\cC_{K_\cJ}\defeq \left(\cC_{K_j}\right)_{j\in\cJ}.%
$$
Hence we obtain a partition $\cP_\cJ$ on $\tld{\cF\cL}_\cJ$ by locally closed subschemes of the form $\cC_{K_\cJ}$.
Note that $\cC_{K_\cJ}$ is stable under the shifted $\un{T}$-conjugation action (defined in equation (\ref{it:T-act:1})).
We would frequently use the notation $\cC\subseteq\tld{\cF\cL}_\cJ$ for an arbitrary (non-empty) element of the partition $\cP_\cJ$.
For each $w_\cJ\in \un{W}$ we also define $$\cM_{w_\cJ}^\circ\defeq \prod_{j\in\cJ}\cM_{w_j}^\circ\quad\mbox{and}\quad\overline{\cM}_{w_\cJ}\defeq \prod_{j\in\cJ}\overline{\cM}_{w_j}.$$
It is clear that $\cM_{w_\cJ}^\circ$ is open in $\tld{\cF\cL}_\cJ$, and $\overline{\cM}_{w_\cJ}$ is the unique element in $\cP_\cJ$ which is a closed subscheme of $\cM_{w_\cJ}^\circ$.
Again, by letting
\[
\tld{\cS}^\circ(w_\cJ,u_\cJ)\defeq\prod_{j\in\cJ}\tld{\cS}^\circ(w_j,u_j) \quad\mbox{ and } \quad\tld{\cS}(w_\cJ,u_\cJ)\defeq\prod_{j\in\cJ}\tld{\cS}(w_j,u_j)\subseteq\tld{\cF\cL}_\cJ
\]
for elements $w_\cJ=(w_j)_{j\in\cJ},u_\cJ=(u_j)_{j\in\cJ}\in\un{W}$, Lemma~\ref{lem: hypersurfaces are Schubert} generalizes and we see for instance that
\[
\cM_{w_\cJ}^\circ=\tld{\cS}^\circ(w_0,w_0w_\cJ).
\]

As each element of $\cP_\cJ$ is contained in one of $\cM_{w_\cJ}^\circ$ (see Lemma~\ref{lem: open cover}), we deduce that $\overline{\cM}_{w_\cJ}$ exhausts all elements of the partition $\cP_\cJ$ which are closed subschemes of $\tld{\cF\cL}_\cJ$, when $w_\cJ$ runs through the elements of $\un{W}$. We use the notation $f_{S,j}$ for the composition
\begin{equation}\label{equ: j projection}
\tld{\cF\cL}_\cJ\xrightarrow{\Proj_j}\tld{\cF\cL}\xrightarrow{f_S}\bP^1_{/\F}
\end{equation}
where $\Proj_j$ is the projection to the $j$-th factor.

We end this section by studying the relation between $\cP_\cJ$, $\cP_{\cJ'}$ and the base change map $\tld{\BC}:\tld{\cF\cL}_\cJ\ra \tld{\cF\cL}_{\cJ'}$ introduced in \S\,\ref{subsub:FLT}.
\begin{lemma}\label{lem: unramified base change of elements}
Let $\cC\in \cP_\cJ$.
There exists a unique element $\cC'\in \cP_{\cJ'}$ such that $\tld{\BC}^{-1}(\cC')=\cC$.
In particular, $\cP_\cJ$ is the partition  on $\tld{\cF\cL}_{\cJ}$ induced from the partition $\cP_{\cJ'}$ on $\tld{\cF\cL}_{\cJ'}$ by pulling back along the embedding $\tld{\mathrm{BC}}$.
\end{lemma}
\begin{proof}
If $\cC=\cC_{K_{\cJ}}$ for $K_\cJ\in \wp(\mathbf{n})^{f}$ then we pick $\cC'\defeq \cC_{K'_{\cJ'}}$ with $K'_{\cJ'}\in \wp(\mathbf{n})^{f'}$ characterized by $K'_{j'}\defeq K_j$ for each $j'\equiv j$ modulo $f$.
All the other claims are clear.
\end{proof}

\subsection{Niveau partition}
\label{sub:niveau:elements}
In this section, we introduce a new partition $\{\cN_{w_\cJ}\mid w_\cJ\in \un{W}\}$ on $\tld{\cF\cL}_\cJ$ motivated by the notion of niveau for mod-$p$ Galois representations. Roughly speaking, two mod-$p$ Galois representations arise from the same $\cN_{w_\cJ}$ if and only if they have the same semisimplification. The main results of this section are Propositions~\ref{prop: niveau rep vers pt},~\ref{prop: niveau is a union of elements} and~\ref{prop: niveau partition}. We fix throughout this section a $\lambda\in X^*_+(\un{T})$ with $\lambda+\eta$ being Fontaine--Laffaille (see Definition~\ref{defn:mGenFL}).

The following description of semisimple Galois representations arising from $\tld{\cF\cL}_\cJ$, with weight $\lambda+\eta=(\lambda_j+\eta_j)_{j\in\cJ}$, is well-known.
\begin{lemma}\label{lem: class of irr}
Let $x\in \tld{\cF\cL}_\cJ(\F)$ be a closed point such that $\rhobar_{x,\lambda+\eta}$ is semisimple. Then there exists $w_\cJ\in \un{W}$ uniquely determined by $x$ such that
$$x\in \overline{\cM}_{w_\cJ}(\F)\textnormal{ and }\rhobar_{x,\lambda+\eta}|_{I_K}\cong \taubar(w^{-1}_\cJ,\lambda+\eta).$$
\end{lemma}

\begin{proof}
Let
$$(M_x,\{\Fil^h M_x\}_h,\{\phi_{x,h}\}_h)\in\mathrm{FL}_n^{\lambda+\eta}(\F)$$
be a Fontaine--Laffaille module attached to $x$ and $\cC$ be the (unique) element of $\cP_\cJ$ such that $x\in\cC(\F)$. Recall that $K'$ is an unramified extension of $K$ with degree $r=n!$, and thus we have
$$\rhobar_{x,\lambda+\eta}|_{G_{K'}}\cong \oplus_{i=1}^n\rhobar_i$$
for certain characters $\rhobar_i:~G_{K'}\rightarrow\F^\times$. We write $\bG_{m,\cJ'}$ for the $\F$-scheme given by the product of $f'$-copies of $\bG_{m}$, and for an element $\mu\in X^\ast(\bG_m)^{f^\prime}$ we have the map
$$\rhobar_{x,\mu}:~\bG_{m,\cJ'}(\F)\rightarrow \Rep^1_{\F}(G_{K'})$$
(obtained from the map (\ref{eq:to:Gal}) in the case when $n=1$ and $K$ taken to be $K'$). We set $y\defeq \tld{\BC}(x)\in\tld{\cF\cL}_{\cJ'}(\F)$ and write $\lambda'=(\lambda'_{j'})_{j'\in\cJ'}\in X^\ast(\un{T}^r)$ for the image of $\lambda$ under the diagonal embedding (see the very end of \S\,\ref{subsub:FLT} for the map $\tld{\BC}$). Therefore we can choose, for each $1\leq i\leq n$, a point $x_i\in \bG_{m,\cJ'}(\F)$ such that $\rhobar_{x_i,\mu_i}\cong \rhobar_i$. Here $\mu_i\in X^\ast(\bG_m)^{f^\prime}$ is a weight uniquely determined by $\rhobar_i$ and $\lambda+\eta$. We write
$$(M_{x_i},\{\Fil^h M_{x_i}\}_h,\{\phi_{x_i,h}\}_h)\in\mathrm{FL}_1^{\mu_i}(\F)$$
for the rank one Fontaine--Laffaille module of weight $\mu_i$ attached to $x_i$. Hence there exists an isomorphism
\begin{equation}\label{equ: isom of FL mod}
(M_y,\{\Fil^h M_y\}_h,\{\phi_{y,h}\}_h)\cong\bigoplus_{i=1}^n(M_{x_i},\{\Fil^h M_{x_i}\}_h,\{\phi_{x_i,h}\}_h)
\end{equation}
inside $\mathrm{FL}_n^{\lambda'+\eta'}(\F)$ where $(M_y,\{\Fil^h M_y\}_h,\{\phi_{y,h}\}_h)$ is the Fontaine--Laffaille module of weight $\lambda'+\eta'$ attached to $y$.
We write $\nu=(\nu_{j'})_{j'\in\cJ'}\defeq (\mu_1,\dots,\mu_n)\in X^\ast(\un{T}^r)$ and assume without loss of generality that $\nu_0=\lambda'_0+\eta'_0$. Since $M_{x_i}$ has rank one, choosing a basis $\beta_{x_i}=(\beta_{x_i}^{(j')})_{j'\in\cJ'}$ for $M_{x_i}=\prod_{j'\in\cJ'}M_{x_i}^{(j')}$ is the same as choosing a non-zero vector $\beta_{x_i}^{(j')}\in M_{x_i}^{(j')}$ for each $j'\in\cJ'$. There exists $s_{\cJ'}=(s_{j'})_{j'\in\cJ'}\in\un{W}^r$ such that
$$\nu_{j'}=s_{f'-1}^{-1}\cdots s_{j'+1}^{-1} \cdot s_{j'}^{-1}(\lambda'_{j'}+\eta'_{j'})$$
for each $1\leq i\leq n$ and $j'\in\cJ'$, and hence
$$\rhobar_{x,\lambda+\eta}|_{I_K}\cong \rhobar_{y,\lambda'+\eta'}|_{I_{K'}}\cong \oplus_{i=1}^n\rhobar_i|_{I_{K'}}\cong \taubar(s_{\cJ'},\lambda'+\eta')$$
according to Lemma~\ref{lem:FL:rk1} (describing $\Tcris$ in the rank one case) as well as Definition~\ref{defn:tau}. For each $j'\in\cJ'$, we claim that
$$\beta_y^{(j')}\defeq (\beta_{x_1}^{(j')},\dots,\beta_{x_n}^{(j')})\cdot s_{f'-1}^{-1}\cdots s_{j'+1}^{-1}\cdot s_{j'}^{-1}$$
is the unique permutation of the basis $\{\beta_{x_1}^{(j')},\dots,\beta_{x_n}^{(j')}\}$ of $M_y^{(j')}$, which is compatible with the Hodge filtration of  $M_y$ (see also Definition~\ref{def: compatible basis}). As $\phi_y$ sends $\gr^\bullet M_{x_i}\subseteq \gr^\bullet M_y$ to $M_{x_i}\subseteq M_y$ for each $1\leq i\leq n$, the matrix of $\phi_y=\{\phi_{y,h}\}_h: \gr^\bullet(M_y)\rightarrow M_y$ attached to the basis $\beta_y\defeq (\beta_y^{(j')})_{j'\in\cJ'}$ lies in
$$(\un{T}^r\cdot w_{\cJ'})(\F)\subseteq\un{G}^r(\F)$$
with $w_{\cJ'}=(w_{j'})_{j'\in\cJ'}$ defined by
$$w_{j'}\defeq \left(s_{f'-1}^{-1}\cdots s_{j'+1}^{-1} \right)^{-1}\cdot \left(s_{f'-1}^{-1}\cdots s_{j'+1}^{-1}\cdot s_{j'}^{-1} \right)=s_{j'}^{-1}.$$
In particular, we have $y\in\overline{\cM}_{w_{\cJ'}}(\F)$. Finally, the fact
$$y\in\tld{\BC}(\cC)(\F)\cap\overline{\cM}_{w_{\cJ'}}(\F)\neq \emptyset$$
together with Lemma~\ref{lem: unramified base change of elements} implies that there exists $w_\cJ\in \un{W}$ such that $\cC=\overline{\cM}_{w_\cJ}$ and $w_{\cJ'}$ is the image of $w_\cJ$ under the diagonal embedding $\un{W}\hookrightarrow\un{W}^r$. Hence we finish the proof by the identification of representations of $I_{K'}=I_K$
$$\taubar(s_{\cJ'},\lambda'+\eta')=
\taubar(w^{-1}_{\cJ'},\lambda'+\eta')=\taubar(w^{-1}_\cJ,\lambda+\eta)$$
which follows directly from the Definition~\ref{defn:tau}.
\end{proof}

Let $P\supseteq B$ be a standard parabolic subgroup of $\GL_{n}$ with standard Levi subgroup $M$ and unipotent radical $N$. Let $P^-\subseteq\GL_{n}$ be the opposite parabolic subgroup satisfying $P\cap P^-=M$.
(Be careful to distinguish our notation for the Levi subgroup and a Fontaine--Laffaille module.) We write $W_M\subseteq W$ for the Weyl group of $M$, embedded inside the Weyl group $W$ of $\GL_{n}$. Then there exists a positive integer $r_0$ with $r_0\leq n$ and a tuple of integers $(n_m)_{1\leq m\leq r_0}$ partitioning $n$ such that $M$ is the image of the standard embedding
$$\GL_{n_1}\times\cdots\times\GL_{n_{r_0}}\hookrightarrow \GL_{n}.$$
Given a point $x\in\tld{\cF\cL}_\cJ(\F)$, we say that $x$ is \emph{$P$-ordinary} if there exists $\rhobar:~G_K\rightarrow P^-(\F)$ such that $\rhobar_{x,\lambda+\eta}\cong \rhobar$. We write $\gr_{P-}(\rhobar_{x,\lambda+\eta})$ for the isomorphism class of the composition of $\rhobar$ with $P^-(\F)\twoheadrightarrow M(\F)$. (Note that the appearance of $P^-$ is due to the fact that $\Tcris$ is contravariant.) Thanks to the full faithfulness of $\Tcris$, $x$ is $P$-ordinary if and only if
the following holds: there exists a filtration by Fontaine--Laffaille submodules
\begin{equation}
\label{eq:FL:sumod}
0\subsetneq M_{x,1}\subsetneq\cdots\subsetneq M_{x,r_0-1}\subsetneq M_{x,r_0}=M_x
\end{equation}
such that the $k\otimes_{\Fp}\F$-module $M_{x,m}\subseteq M_x$ has rank $n_m^+\defeq \sum_{d=1}^mn_d$ for all $1\leq m\leq r_0$.
Note that the above implies that $\rhobar_{x,\lambda+\eta}$ has image contained in $P^-(\F)$ and furthermore the $n_m^+$-dimensional $G_K$-representation $\Tcris(M_{x,m})$ is isomorphic to the quotient of $\rhobar_{x,\lambda+\eta}$ induced from the standard surjective morphism $P^-\onto \GL_{n_m^+}$ for all $1\leq m\leq r_0$.

\begin{lemma}\label{lem: filtration vers parabolic}
Let $x\in\tld{\cF\cL}_\cJ(\F)$ be a point. Then $x$ is $P$-ordinary if and only if there exist $A=(A^{(j)})_{j\in\cJ}\in\un{G}(\F)$ and $u_\cJ=(u_j)_{j\in\cJ}\in \un{W}$ such that $x$ is the image of $A$ in $\tld{\cF\cL}_\cJ(\F)$ and
\begin{equation}\label{equ: P form}
A^{(j)}\in u_j P u_{j-1}^{-1}
\end{equation}
for all $j\in\cJ$.
\end{lemma}
\begin{proof}
We only prove the $\Rightarrow$ direction as the opposite direction can be proved by reversing the argument. Consider the increasing sequence of Fontaine--Laffaille submodules attached to $x$, written as in (\ref{eq:FL:sumod}). We write as usual $M_x=\prod_{j\in\cJ}M_x^{(j)}$ and $M_{x,m}=\prod_{j\in\cJ}M_{x,m}^{(j)}$ for each $1\leq m\leq r_0$.

We choose a basis $\beta=(\beta^{(j)})_{j\in\cJ}$ of $M_x$, compatible with the Hodge filtration of $M_x$, in the following way.
Write $\beta^{(j)}=\{\beta^{(j)}_1,\dots,\beta^{(j)}_n\}$ for all $j\in\cJ$. For each $j\in\cJ$, we choose $\beta^{(j)}_k$ by an increasing induction on $k$. Assume for the moment that we have chosen $\beta^{(j)}_1,\dots,\beta^{(j)}_{k-1}$ for some $1\leq k\leq n$ such that $\beta^{(j)}_1,\dots,\beta^{(j)}_{k-1}$ forms a basis of $\mathrm{Fil}^{\lambda_{j,k}+(n-k)+1}M_x^{(j)}$, then we want to choose the next vector $\beta^{(j)}_k$. We define $m_k$ as the smallest integer satisfying $1\leq m_k\leq r_0$ and
$$\mathrm{Fil}^{\lambda_{j,k}+(n-k)}M_x^{(j)}\cap M_{x,m_k}^{(j)}\neq \mathrm{Fil}^{\lambda_{j,k}+(n-k)+1}M_x^{(j)}\cap M_{x,m_k}^{(j)}.$$
As $(\mathrm{Fil}^{\lambda_{j,k}+(n-k)}M_x^{(j)}\cap M_{x,m_k}^{(j)})/(\mathrm{Fil}^{\lambda_{j,k}+(n-k+1)}M_x^{(j)}\cap M_{x,m_k}^{(j)})$ is one dimensional and is a subspace of $\mathrm{Fil}^{\lambda_{j,k}+(n-k)}M_x^{(j)}/\mathrm{Fil}^{\lambda_{j,k}+(n-k+1)}M_x^{(j)}$, which is also one dimensional, we have
\begin{equation}\label{equ: sum of subspace}
(\mathrm{Fil}^{\lambda_{j,k}+(n-k)}M_x^{(j)}\cap M_{x,m_k}^{(j)})+\mathrm{Fil}^{\lambda_{j,k}+(n-k)+1}M_x^{(j)}=\mathrm{Fil}^{\lambda_{j,k}+(n-k)}M_x^{(j)}.
\end{equation}
Hence, we choose an arbitrary non-zero vector
$$\beta^{(j)}_k\in \mathrm{Fil}^{\lambda_{j,k}+(n-k)}M_x^{(j)}\cap M_{x,m_k}^{(j)}\setminus \mathrm{Fil}^{\lambda_{j,k}+(n-k)+1}M_x^{(j)}\cap M_{x,m_k}^{(j)}$$
and note that $\beta^{(j)}_1,\dots,\beta^{(j)}_k$ necessarily forms a basis of $\mathrm{Fil}^{\lambda_{j,k}+(n-k)}M_x^{(j)}$ thanks to (\ref{equ: sum of subspace}). According to the choice above, it is clear that $\beta$ is compatible with the Hodge filtration of $M_x$.

We consider an element $u_\cJ=(u_j)_{j\in\cJ}\in\un{W}$ and the following reordering $\beta\cdot \pi^{-1}(u_\cJ)=(\beta^{(j)}\cdot u_{j-1})_{j\in\cJ}$ of the basis $\beta$, where $\beta^{(j)}\cdot u_{j-1}$ is the basis of $M_x^{(j)}$ given by $\beta_{u_{j-1}(1)}^{(j)},\cdots,\beta_{u_{j-1}(n)}^{(j)}$. Now we observe that, for each $j\in\cJ$, there exists $u_{j-1}$ such that for each $1\leq m\leq r_0$ the list of vectors
\begin{equation}\label{equ: an induced basis}
\beta_{u_{j-1}(1)}^{(j)},\beta_{u_{j-1}(2)}^{(j)},\cdots,\beta_{u_{j-1}(n_m^+)}^{(j)}
\end{equation}
forms a basis of $M_{x,m}^{(j)}$ inducing a basis compatible with the Hodge filtration on the quotient $M_{x,m}^{(j)}/M_{x,m-1}^{(j)}$.
We write $A=(A^{(j)})_{j\in\cJ}$ (resp.~$A_1=(A_1^{(j)})_{j\in\cJ}$) for the matrix of $\phi_{x}$ (induced from $\{\phi_{x,r_0,h}\}_h$) attached to the basis $\beta$ (resp.~$\beta\cdot \pi^{-1}(u_\cJ)$). It follows from (\ref{eq:FL:sumod}) that $A_1^{(j)}\in P(\F)$ for each $j\in\cJ$, which implies (\ref{equ: P form}) as we have $A^{(j)}=u_jA_1^{(j)}u_{j-1}^{-1}$ for each $j\in\cJ$.
\end{proof}

\begin{lemma}\label{lem: filtration vers graded pieces}
Let $x\in\tld{\cF\cL}_\cJ(\F)$ be a $P$-ordinary point, and let $u_\cJ=(u_j)_{j\in\cJ}\in \un{W}$ be as in Lemma~\ref{lem: filtration vers parabolic}.
If there exists
$$A_0=(A_0^{(j)})_{j\in\cJ}\in\prod_{j\in\cJ}u_j M u_{j-1}^{-1}(\F)$$
whose image $y$ in $\tld{\cF\cL}_{\cJ}(\F)$ satisfies $\rhobar_{y,\lambda+\eta}\cong\gr_{P^-}(\rhobar_{x,\lambda+\eta})$, then there exists $t=(t^{(j)})_{j\in\cJ}\in\un{T}(\F)$ and
$$A=(A^{(j)})_{j\in\cJ}\in\prod_{j\in\cJ}u_j P u_{j-1}^{-1}(\F)$$
such that the image of $A$ in $\tld{\cF\cL}_{\cJ}(\F)$ is $x$ and the image of $A$ under
\begin{equation}\label{equ: pass to Levi}
\prod_{j\in\cJ}u_j P u_{j-1}^{-1}(\F)\twoheadrightarrow\prod_{j\in\cJ}u_j M u_{j-1}^{-1}(\F)
\end{equation}
is $A_0\cdot t$ (cf.~(\ref{it:T-act:1})).
\end{lemma}
\begin{proof}
This is a simple refinement of the proof of Lemma~\ref{lem: filtration vers parabolic} in the sense that we can choose the basis $\beta$ more carefully. We write $M_{x,0}\defeq 0\subseteq M_x$ for convenience. According to Proposition~\ref{thm:repr:FLgp}, our assumption on $A_0\in \prod_{j\in\cJ}u_j M u_{j-1}^{-1}(\F)$ simply means that there exists a basis of the Fontaine--Laffaille module $\bigoplus_{m=1}^{r_0}M_{x,m}/M_{x,m-1}$, written $\gr_P(\beta)$, such that it is compatible with the Hodge filtration and the matrix of Frobenius attached to the basis $\gr_P(\beta)$ is given by $A_0$. We recall from the proof of Lemma~\ref{lem: filtration vers parabolic} that the basis $\beta\cdot \pi^{-1}(u_\cJ)$ satisfies the condition that (\ref{equ: an induced basis}) forms a basis of $M_{x,m}$ for each $1\leq m\leq r_0$. Therefore the basis $\beta\cdot \pi^{-1}(u_\cJ)$ induces a basis of $\bigoplus_{m=1}^{r_0}M_{x,m}/M_{x,m-1}$ which is compatible with the Hodge filtration by our minimality assumption on the length of $u_{j-1}$ (see Lemma~\ref{lem: filtration vers parabolic}). It is clear that, given the basis $\gr_P(\beta)$ of $\bigoplus_{m=1}^{r_0}M_{x,m}/M_{x,m-1}$, we can always choose the basis $\beta$ as in Lemma~\ref{lem: filtration vers parabolic} with the extra requirement that the basis of $\bigoplus_{m=1}^{r_0}M_{x,m}/M_{x,m-1}$ induced from $\beta\cdot \pi^{-1}(u_\cJ)$ is exactly $\gr_P(\beta)$. We write $A_1=(A_1^{(j)})_{j\in\cJ}$ for the matrix of Frobenius $\{\phi_{x,0,h}\}_h$ under $\beta$ and write $x_1$ for the image of $A_1$ in $\tld{\cF\cL}_\cJ(\F)$. Hence the image of $A_1$ under (\ref{equ: pass to Levi}) is $A_0$. As $A_1$ is constructed from $M_x$ by a choice of basis (compatible with Hodge filtration), it is clear that $\rhobar_{x_1,\lambda+\eta}\cong \rhobar_{x,\lambda+\eta}$ and there exists $t=(t^{(j)})_{j\in\cJ}\in \un{T}(\F)$ such that $x=x_1\cdot t$. Hence we set $A\defeq A_1\cdot t$ and finish the proof.
\end{proof}

We use the notation $w^\flat_\cJ=(w^\flat_j)_{j\in\cJ}$
with
\begin{equation}\label{equ: full niveau}
w^\flat_j\defeq w_j\cdot w_{j-1}\cdots w_{j-f+1}
\end{equation}
for each $w_\cJ=(w_j)_{j\in\cJ}\in\un{W}$. Note that $w^\flat_{j-1}=w_j^{-1}w^\flat_jw_j$ for each $j\in\cJ$.

\begin{defn}\label{def: associate Levi}
We say that a Levi subgroup $M'\subseteq\GL_{n}$ is \emph{$W$-standard} if its conjugation by an element of $W$ is standard. We define $M_w$ as the minimal $W$-standard Levi subgroup that contains $w$ and call it \emph{the Levi subgroup associated with $w$}. Note by definition that we have $M_{uwu^{-1}}=u M_w u^{-1}$ for any choice of $w,u\in W$.
\end{defn}
Each element $w\in W$ induces a partition of $\mathbf{n}=\{1,\dots,n\}$ into orbits of $w$. For each Noetherian $\F$-algebra $R$, $M_w(R)\subseteq\GL_n(R)$ consists of those matrices whose $(i,i^\prime)$-entry is zero if $i$ and $i^\prime$ lies in different orbits of $w$. We have the following useful observation from Definition~\ref{def: associate Levi}: given an element $w\in W$ and a $W$-standard Levi subgroup $M'\subseteq\GL_{n}$ such that $M_w\subseteq M'$, then $M'=M_w$ if and only if the number of Levi blocks inside $M'$ equals the number of orbits of $w$.

Given two elements $w_\cJ=(w_j)_{j\in\cJ},~u_\cJ=(u_j)_{j\in\cJ}\in \un{W}$ and a standard parabolic subgroup of $ P\subseteq\GL_{n}$ with standard Levi subgroup $ M$ and unipotent radical $N$, we observe that the composition
$$\prod_{j\in\cJ}T u_jN  u_j^{-1} w_j\rightarrow \un{G}\rightarrow\tld{\cF\cL}_\cJ$$
factors through %
$$\prod_{j\in\cJ}T u_j N u_j^{-1} w_j\rightarrow\cM_{w_\cJ}^\circ.$$
For each $j\in\cJ$, if we assume that $u_j^{-1}w_ju_{j-1}\in M$, (in which case we have $u_j M u_{j-1}^{-1}=u_j M u_j^{-1} w_j$ and $u_j P u_{j-1}^{-1}=u_j P u_j^{-1} w_j$), then $T u_jN  u_j^{-1} w_j$ is the fiber of
$$u_j P u_{j-1}^{-1}\twoheadrightarrow u_j M u_{j-1}^{-1}$$
over $T w_j$.

Now we return to the set up of Lemma~\ref{lem: filtration vers graded pieces}. We assume further that there is no strictly smaller standard parabolic subgroup $P'\subsetneq P$ such that $x$ is $P'$-ordinary, which implies that $\gr_{P^-}(\rhobar_{x,\lambda+\eta})$ is semisimple. It follows from Lemma~\ref{lem: class of irr} that there exists $y\in\overline{\cM}_{w_\cJ}(\F)$ for some $w_\cJ\in\un{W}$ such that $\gr_{P^-}(\rhobar_{x,\lambda+\eta})\cong\rhobar_{y,\lambda+\eta}$. As $\overline{\cM}_{w_\cJ}$ is the schematic image of $\un{T}w_\cJ$ in $\cM_{w_\cJ}^\circ$, we may choose $A_0\in \un{T}w_\cJ(\F)$ whose image in $\overline{\cM}_{w_\cJ}(\F)$ is $y$. Then it follows from Lemma~\ref{lem: filtration vers graded pieces} that there exists $A\in u_j P u_{j-1}^{-1}(\F)$ and $t=(t^{(j)})_{j\in\cJ}\in\un{T}(\F)$ such that the image of $A$ under (\ref{equ: pass to Levi}) is $A_0\cdot t\in \un{T}w_\cJ(\F)$, and the image of $A$ in $\tld{\cF\cL}_\cJ$ is $x$. It is clear that $A_0\in\un{T}w_\cJ(\F)\cap \prod_{j\in\cJ}u_j M u_{j-1}^{-1}(\F)$ which implies that $$u_j^{-1}w_ju_{j-1}\in M$$ for each $j\in\cJ$. In particular, we deduce that
$$u_j^{-1}w^\flat_ju_j=(u_j^{-1}w_ju_{j-1})\cdot (u_{j-1}^{-1}w_{j-1}u_{j-2})\cdots (u_{j-f+1}^{-1}w_{j-f+1}u_j)\in M$$
for each $j\in\cJ$. Then we observe that the number of orbits of $u_j^{-1}w^\flat_ju_j$, which equals the number of orbits of $w^\flat_j$, which (by Lemma~\ref{lem: class of irr}) equals the number of irreducible direct summands of $\gr_{P^-}(\rhobar_{x,\lambda+\eta})$, which finally equals the number of Levi blocks of $M$.
It follows from the paragraph after Definition~\ref{def: associate Levi} that we must have
$$M_{u_j^{-1}w^\flat_ju_j}=M$$
for each $j\in\cJ$. Consequently, we arrive at the following definition.
\begin{defn}\label{def: niveau partition}
For each $w_\cJ\in\un{W}$, we define $\Xi_{w_\cJ}\subseteq\{w_\cJ\}\times\un{W}$ as the subset consisting of pairs $\xi=(w_\cJ,u_\cJ)$ such that $M_{u_j^{-1}w^\flat_ju_j}$ is a standard Levi subgroup of $\GL_{n}$ independent of $j\in\cJ$, written $M_\xi$, and such that $u_j^{-1}w_ju_{j-1}\in M_\xi $ for each $j\in\cJ$. Note that there exists a unique standard parabolic subgroup $P_\xi\subseteq \GL_{n}$ containing $M_\xi$ and we denote the unipotent radical of $P_\xi$ by $N_\xi$. For each element $\xi=(w_\cJ,u_\cJ)\in\Xi_{w_\cJ}$, we define $\cN_\xi$ as the schematic image of
$$\prod_{j\in\cJ}T u_j N_\xi u_j^{-1} w_j$$
in $\cM_{w_\cJ}^\circ$.
\end{defn}

\begin{rmk}
For a fixed $w_\cJ\in\un{W}$, the following closed subscheme of $\un{G}$
$$
\prod_{j\in\cJ}u_j M_\xi u_{j-1}^{-1}
$$
does not depend on the choice of $\xi=(w_\cJ,u_\cJ)\in \Xi_{w_\cJ}$. In fact, this directly follows from the observation that (see Definition~\ref{def: niveau partition} for the properties of $\xi$)
$$u_j M_\xi u_{j-1}^{-1}=u_j M_\xi u_j^{-1}w_j=u_j M_{u_j^{-1}w^\flat_ju_j}u_j^{-1}w_j=M_{w^\flat_j}w_j$$
for each $j\in\cJ$. Given two pairs $\xi=(w_\cJ,u_\cJ),~\xi^\prime=(w_\cJ,u_\cJ')\in\Xi_{w_\cJ}$, one can show that $\cN_\xi=\cN_{\xi^\prime}$ if and only if
$$\prod_{j\in\cJ}u_j P_\xi u_{j-1}^{-1}=\prod_{j\in\cJ}u_j' P_{\xi^\prime} (u_{j-1}')^{-1}$$
if and only if $M_\xi=M_{\xi^\prime}$ and $u_j^{-1}u_j'\in W_{M_\xi}$ for each $j\in\cJ$.
\end{rmk}

\begin{prop}\label{prop: niveau rep vers pt}
A point $x\in \tld{\cF\cL}_\cJ(\F)$ belongs to $\cN_\xi(\F)$ for some $\xi\in\Xi_{w_\cJ}$ if and only if
$$\rhobar_{x,\lambda+\eta}^{\semis}|_{I_K}\cong \overline{\tau}(w^{-1}_\cJ,\lambda+\eta).$$
\end{prop}
\begin{proof}
This follows directly from Lemma~\ref{lem: class of irr} and the discussion right before Definition~\ref{def: niveau partition}.
\end{proof}

\begin{prop}\label{prop: niveau is a union of elements}
Let $\xi\in\Xi_{w_{\cJ}}$. Then $\cN_\xi$ is integral, and
\begin{equation}\label{equ: intersection expression}
\cN_\xi=\prod_{j\in\cJ}\left(\cM_{w_j}^{\circ}\cap\bigcap_{S\in K_j}\cH_S\right)
\end{equation}
for a uniquely determined $K_\cJ=(K_j)_{j\in\cJ}\subseteq \wp(\mathbf{n}_\cJ)$ with $\cC_{K_{\cJ}}\neq\emptyset$ and $K_j\cap S_{\bullet,w_j}=\emptyset$ for all $j\in\cJ$.
In particular, $\cN_\xi$ is a topological union of elements in $\cP_\cJ$.
\end{prop}

\begin{proof}
The fact that $\cN_\xi$ is integral, the equality (\ref{equ: intersection expression}), and the unique existence of $K_{\cJ}$ follow immediately from Lemma~\ref{lem: is union of elements}, which together with Lemma~\ref{lem: elementary, elements} and (\ref{equ: cM_w is union of elements}) implies that $\cN_\xi$ is a topological union of elements in $\cP_\cJ$.
\end{proof}

\begin{defn}
\label{defn:Nxi:irr:cmpt}
For each $w_\cJ \in \un{W}$, we define $\cN_{w_\cJ}$ as the topological union of $\cN_\xi$ for all $\xi\in\Xi_{w_\cJ}$. As $\cN_\xi$ is closed in $\cM_{w_\cJ}^\circ$ for each $\xi\in\Xi_{w_\cJ}$, $\cN_{w_\cJ}$ is naturally a reduced closed subscheme of $\cM_{w_\cJ}^\circ$.
\end{defn}
\begin{prop}\label{prop: niveau partition}
The set of locally closed subschemes $\{\cN_{w_\cJ}\mid w_\cJ\in\un{W}\}$ forms a topological partition of $\tld{\cF\cL}_\cJ$.
Moreover, a point $x\in\tld{\cF\cL}_\cJ(\F)$ belongs to $\cN_{w_\cJ}(\F)$ if and only if
$$\rhobar_{x,\lambda+\eta}^{\semis}|_{I_K}\cong \taubar(w^{-1}_\cJ,\lambda+\eta).$$
\end{prop}
\begin{proof}
It follows from Lemma~\ref{lem: class of irr} that, for each $x\in\tld{\cF\cL}_\cJ(\F)$, the semisimplification of $\rhobar_{x,\lambda+\eta}$ has the form $\taubar(w^{-1}_\cJ,\lambda+\eta)$ for some $w_\cJ\in\un{W}$. Hence the desired result follows directly from Proposition~\ref{prop: niveau rep vers pt} and the definition of $\cN_{w_\cJ}$.
\end{proof}

\begin{rmk}\label{rmk: irr comp of niveau}
Given $w_\cJ\in\un{W}$, the scheme $\cN_{w_\cJ}$ is not irreducible in general. As $\cN_{w_\cJ}$ is topological union of the integral schemes $\cN_\xi$ for all $\xi\in\Xi_{w_\cJ}$, each irreducible component of $\cN_{w_\cJ}$ must have the form $\cN_\xi$ for some $\xi\in\Xi_{w_\cJ}$. The converse is not true, namely there exist $w_\cJ\in\un{W}$ and $\xi\in\Xi_{w_\cJ}$ such that $\cN_\xi$ is strictly contained in some irreducible component of $\cN_{w_\cJ}$ (see for example the case $w_\cJ=1$). One can prove that $\cN_\xi$ is an irreducible component if and only if there exists $x\in\cN_\xi(\F)$ such that $\rhobar_{x,\lambda+\eta}$ is \emph{maximally non-split}, namely each non-zero semisimple subquotient of $\rhobar_{x,\lambda+\eta}$ is irreducible.
\end{rmk}

\subsection{Standard coordinates}\label{sub:std:note}
In this section, we further fix some notation that will be frequently used in later sections. In particular, we introduce a standard coordinate on $\cN_\xi$ (see (\ref{equ: std section prime})) for each $w_\cJ\in\un{W}$ and each $\xi\in\Xi_{w_{\cJ}}$.

We fix a choice of $w_\cJ\in\un{W}$ and $\xi\in\Xi_{w_{\cJ}}$ (as in Definition~\ref{def: niveau partition}) and use the usual notation $ M_\xi, N_\xi\subseteq  P_\xi$ for the subgroups of $\GL_{n}$ associated with $\xi$. We write $\Phi^+_\xi\subseteq\Phi^+$ for the subset such that $N_\xi\subseteq U$ is the closed subscheme characterized by the vanishing of the $\al$-entry for each $\al\in\Phi^+\setminus \Phi^+_\xi$.

We associate a tuple of integers $\un{n}^\xi=(n_k^\xi)_{1\leq k\leq r_\xi}$ with $ M_\xi$ such that $n=\sum_{k=1}^{r_\xi}n_k^\xi$ and
$$ M_\xi\cong \GL_{n_1^\xi}\times\cdots\GL_{n_{r_\xi}^\xi}$$
where $r_\xi$ is the number of Levi blocks of $ M_\xi$. We set
\begin{equation}\label{equ: single block of integers}
[m]_\xi\defeq\left\{k\Bigm\vert 1+\sum_{d=1}^{m-1}n_d^\xi\leq k\leq \sum_{d=1}^mn_d^\xi\right\}
\end{equation}
for each $1\leq m\leq r_\xi$. For each $\al=(i_\al,i_\al^\prime)\in \Phi^+_\xi$, there exists a unique pair of integers $(h_\al, \ell_\al)$ such that $i_\al\in[h_\al]_\xi$, $i_\al^\prime\in[\ell_\al]_\xi$ and $1\leq h_\al<\ell_\al\leq r_\xi$. We consider the set $\Phi^+_{\GL_{r_\xi}}$ of positive roots of $\GL_{r_\xi}$ and there exists a natural map $\Phi^+_\xi\rightarrow\Phi^+_{\GL_{r_\xi}}$ given by
$$\al\mapsto (h_\al, \ell_\al).$$
We often call a root $\gamma\in\Phi^+_{\GL_{r_\xi}}$ a \emph{block} as it corresponds to a block (subgroup) of $N_\xi$, and $\gamma$ can be written as $\gamma=(h,\ell)$ for a pair of integers satisfying $1\leq h<\ell\leq r_\xi$.

We set
$$ N_{\xi,j}^+\defeq u_j N_\xi u_j^{-1}\cap U\textnormal{ and } N_{\xi,j}^-\defeq u_j N_\xi u_j^{-1}\cap w_0 U w_0.$$
Note that multiplication inside $u_j N_\xi u_j^{-1}$ induces an isomorphism of schemes
$$u_j N_\xi u_j^{-1}=N_{\xi,j}^+  N_{\xi,j}^-\cong  N_{\xi,j}^+\times  N_{\xi,j}^-.$$
We deduce from Definition~\ref{def: niveau partition} and the definition of $\tld{\cF\cL}_\cJ$ that the composition
$$\prod_{j\in\cJ}  T N_{\xi,j}^- w_j\hookrightarrow\prod_{j\in\cJ}T u_j N_\xi u_j^{-1} w_j\rightarrow \cN_\xi$$
induces an isomorphism
\begin{equation}\label{equ: std section}
\prod_{j\in\cJ} T N_{\xi,j}^- w_j\xrightarrow{\sim} \cN_\xi.
\end{equation}
Note that the LHS of (\ref{equ: std section}) is a closed subscheme of $\un{G}$, and thus (\ref{equ: std section}) is a standard way to lift the subscheme $\cN_\xi\subseteq \tld{\cF\cL}_\cJ$ into $\un{G}$.

We now define
$$\mathrm{Supp}_{\xi,\cJ}\defeq\{(\al,j)\in\Phi^+_\xi\times\cJ\mid u_j(\al)<0\}$$
and
$$\mathrm{Supp}_{\xi,j}\defeq \mathrm{Supp}_{\xi,\cJ}\cap\left(\Phi^+_\xi\times\{j\}\right)$$
for each $j\in\cJ$. We note that $\mathrm{Supp}_{\xi,j}$ is closed under the natural addition induced from $\Phi^+_\xi$, for each fixed $j\in\cJ$. We would abuse the notation $\mathrm{Supp}_{\xi,j}$ for the corresponding subset of $\Phi^+_\xi$ (by omitting $j$) whenever necessary. For each $1\leq \ell\leq n$ and $j\in\cJ$, we write $D_{\xi,\ell}^{(j)}$ for the composition of the following morphisms
\begin{equation}
\label{eq:Dxij}
D_{\xi,\ell}^{(j)}:~\cN_\xi\cong \prod_{j\in\cJ}T N_{\xi,j}^- w_j\twoheadrightarrow T N_{\xi,j}^- w_j\twoheadrightarrow T\twoheadrightarrow \bG_{m}
\end{equation}
where the last morphism is extracting the $\ell$-th diagonal entry. Similarly, for each $(\al,j)\in \mathrm{Supp}_{\xi,\cJ}$, we also consider the composition
$$
u_\xi^{(\al,j)}:~\cN_\xi\cong \prod_{j\in\cJ}T N_{\xi,j}^-  w_j\twoheadrightarrow N_{\xi,j}^-\twoheadrightarrow\bG_{a}
$$
where the last morphism is extracting the $u_j(\al)$-entry. Given a rational function $g$ on $\tld{\cF\cL}_\cJ$, if the regular locus of $g$ is an open subscheme of $\tld{\cF\cL}_\cJ$ that contains $\cN_\xi$, we write $g|_{\cN_\xi}$ for the restriction of $g$ from its regular locus to $\cN_\xi$. It is not difficult to see that (see (\ref{equ: j projection}) for notation)
\begin{equation}\label{equ: comparison of sections}
D_{\xi,\ell}^{(j)}=\pm\left.\frac{f_{S_{\ell,w_j},j}}{f_{S_{\ell+1,w_j},j}}\right|_{\cN_\xi}
\quad\mbox{ and }\quad
u_\xi^{(\al,j)}=\pm\left. \frac{f_{S_{u_j(i_\al)+1,w_j}\sqcup\{w_j^{-1}u_j(i_\al^\prime)\},j}}{ f_{S_{u_j(i_\al),w_j},j}}\right|_{\cN_\xi}
\end{equation}
Here $\pm$ means up to sign, depending only on $w_j$. Note that (\ref{equ: std section}) together with various $D_{\xi,\ell}^{(j)}$ and $u_\xi^{(\al,j)}$ induces an isomorphism of schemes
\begin{equation}\label{equ: std section prime}
\cN_\xi \cong \un{T}\times (\bG_{a})^{\#\mathrm{Supp}_{\xi,\cJ}}.
\end{equation}

We also denote by $\mathrm{Supp}^\square_\xi$ the image of $\mathrm{Supp}_{\xi,\cJ}$ under the composition
\begin{equation}
\label{eq:map:block}
\Phi^+_\xi\times\cJ\twoheadrightarrow \Phi^+_\xi\rightarrow \Phi^+_{\GL_{r_\xi}}.
\end{equation}
For each $\gamma=(h,\ell)\in\Phi^+_{\GL_{r_\xi}}$ and each $j\in\cJ$, we set
$$\mathrm{Supp}_{\xi,\cJ}^\gamma\defeq\{(\al,j)\in\mathrm{Supp}_{\xi,\cJ}\mid h_\al=h,~\ell_\al=\ell\}$$ and
$$\mathrm{Supp}_{\xi,j}^\gamma\defeq \mathrm{Supp}_{\xi,\cJ}^\gamma\cap \mathrm{Supp}_{\xi,j}.$$
Note that we have
\begin{equation}\label{equ: equivalent notion}
\mathrm{Supp}_{\xi,\cJ}^\gamma\neq \emptyset\Leftrightarrow \gamma\in \mathrm{Supp}^\square_\xi.
\end{equation}

Fix $\xi\in\Xi_{w_{\cJ}}$, and let $\Lambda$ be a subset of $\mathrm{Supp}_{\xi,\cJ}$ with $\Lambda^\square$ its image in $\mathrm{Supp}^\square_\xi$. For each $j\in\cJ$, we write $N_{\xi,\Lambda,j}^-(R)$ for the subset of $N_{\xi,j}^-(R)$ consisting of the matrices whose $u_j(\al)$-entry is non-zero if and only if $(\al,j)\in\Lambda\cap\mathrm{Supp}_{\xi,j}$. This defines a locally closed subscheme $N_{\xi,\Lambda,j}^-\subseteq N_{\xi,j}^-$. Similarly, we write $\cN_{\xi,\Lambda}\subseteq\cN_\xi$ for the fiber of (\ref{equ: std section prime}) over
$$\prod_{(\al,j)\in\Lambda}\bG_m\times \prod_{(\al,j)\in\mathrm{Supp}_{\xi,\cJ}\setminus\Lambda}\mathbf{0}$$
where $\mathbf{0}\subseteq\bG_a$ is the closed subscheme given by the zero point. In other words, the morphism $u_\xi^{(\al,j)}:~\cN_\xi\rightarrow \bG_a$ restricts to $u_\xi^{(\al,j)}|_{\cN_{\xi,\Lambda}}:~\cN_{\xi,\Lambda}\rightarrow \bG_m$ if $(\al,j)\in\Lambda$ and to $u_\xi^{(\al,j)}|_{\cN_{\xi,\Lambda}}:~\cN_{\xi,\Lambda}\rightarrow \mathbf{0}$ otherwise. We notice that the isomorphism (\ref{equ: std section}) induces an isomorphism
$$
\prod_{j\in\cJ}T N_{\xi,\Lambda,j}^-w_j\xrightarrow{\sim} \cN_{\xi,\Lambda}.
$$
Note that the isomorphism (\ref{equ: std section prime}) restricts to an isomorphism
\begin{equation}\label{equ: std section Lambda}
\cN_{\xi,\Lambda} \cong \un{T}\times (\bG_{m})^{\#\Lambda}.
\end{equation}
It is also easy to see that $\{\cN_{\xi,\Lambda}\mid \Lambda\subseteq \mathrm{Supp}_{\xi,\cJ}\}$ forms a partition of $\cN_\xi$ by integral locally closed subschemes.
\begin{lemma}\label{lem:union:partition}
The scheme $\cN_{\xi,\Lambda}$ is a topological union of elements in $\cP_\cJ$.
\end{lemma}
\begin{proof}
This follows immediately from Proposition~\ref{prop: niveau is a union of elements} and (\ref{equ: comparison of sections}).
\end{proof}

\begin{lemma}\label{lem: new partition}
Let $\xi,\xi'\in\Xi_{w_{\cJ}}$ be two elements, and let $\Lambda\subseteq\mathrm{Supp}_{\xi,\cJ}$ and $\Lambda'\subseteq\mathrm{Supp}_{\xi',\cJ}$ be two subsets. Then $\cN_{\xi,\Lambda}\cap\cN_{\xi',\Lambda'}\neq\emptyset$ if and only if $\cN_{\xi,\Lambda}=\cN_{\xi',\Lambda'}$.
\end{lemma}

\begin{proof}
Using the identification $\cM_{w_{\cJ}}^{\circ}\cong \prod_{j\in\cJ}w_0Bw_0w_j$ where $w_{\cJ}=(w_j)_{j\in\cJ}$, we may naturally embed both $\cN_{\xi,\Lambda}$ and $\cN_{\xi',\Lambda'}$ into $\prod_{j\in\cJ}w_0Bw_0w_j$. Then the locally closed subscheme $\cN_{\xi,\Lambda}$ (resp. $\cN_{\xi',\Lambda'}$) of $\prod_{j\in\cJ}w_0Bw_0w_j$ is characterized by the vanishing or non-vanishing of each single entry, which can be read off from an arbitrary element of the intersection $\cN_{\xi,\Lambda}\cap\cN_{\xi',\Lambda'}(\F)\neq\emptyset$. The proof is thus finished.
\end{proof}

\begin{rmk}\label{rmk: new partition}
It follows from Lemma~\ref{lem: new partition} that
$$\bigcup_{ w_\cJ\in\un{W}}\{\cN_{\xi,\Lambda}\mid \xi\in\Xi_{w_\cJ}\mbox{ and } \Lambda\subseteq\mathrm{Supp}_{\xi,\cJ}\}$$
forms a partition of $\tld{\cF\cL}_\cJ$.
\end{rmk}

\newpage
\section{The invariant functions on $\tld{\cF\cL}_{\cJ}$}\label{sec:invariant functions}
Recall that the quotient of $\tld{\cF\cL}_\cJ$ by shifted $\un{T}$-conjugation parameterizes isomorphism classes of mod-$p$ Galois representations which are Fontaine--Laffaille of weight $\lambda+\eta$ (Proposition~\ref{thm:repr:FLgp}).
In \S\,\ref{sub:def:inv} below, we introduce a set of rational functions on $\tld{\cF\cL}_\cJ$ that descend to $[\tld{\cF\cL}_{\cJ}\slash{{\sim}_{\un{T}\text{-\textnormal{sh.cnj}}}}]$ and call them the \emph{invariant functions}. Our main goal is to show that invariant functions \emph{separate} $\F$-points of the stack $[\tld{\cF\cL}_{\cJ}\slash{{\sim}_{\un{T}\text{-\textnormal{sh.cnj}}}}]$, namely each $x\in |[\tld{\cF\cL}_{\cJ}\slash{{\sim}_{\un{T}\text{-\textnormal{sh.cnj}}}}]|(\F)$ is uniquely determined by the set
$$\{g(x)\mid \hbox{$g$ is an invariant function which is regular over $x$}\}\subseteq \F.$$
To achieve this, we first cut $\tld{\cF\cL}_{\cJ}$ along the partition $\{\cN_{\xi,\Lambda}\}$ (see Remark~\ref{rmk: new partition}) and then give an explicit construction of the geometric quotient $\cN_{\xi,\Lambda}\slash{{\sim}_{\un{T}\text{-\textnormal{sh.cnj}}}}$ (\cite[\href{https://stacks.math.columbia.edu/tag/04AD}{\S\,04AD}]{stacks-project})
in \S\,\ref{sub:stack:scheme} (Proposition~\ref{prop: pass to scheme}), which guarantee the existence of the geometric quotient $\cC\slash{{\sim}_{\un{T}\text{-\textnormal{sh.cnj}}}}$ for each $\cC\in\cP_\cJ$ satisfying $\cC\subseteq\cN_{\xi,\Lambda}$. Then we introduce Statement~\ref{state: goal} in \S\,\ref{sub:main:state} as a convenient sufficient condition for invariant functions to distinguish $\F$-points of $[\cC\slash{{\sim}_{\un{T}\text{-\textnormal{sh.cnj}}}}]$ (see Statement~\ref{state: separate points prime}). The proof of Statement~\ref{state: goal} (and thus of Statement~\ref{state: separate points prime}) will occupy the entire \S\,\ref{sec:comb:lifts}, \S\,\ref{sec:const:inv} and \S\,\ref{sec:inv:cons}, and will be finally completed in \S\,\ref{sub:main:criterions} (Theorem~\ref{thm: constructible and inv fun} and Corollary~\ref{cor: separate points}).

Throughout this section, by $R$ we mean a Noetherian $\F$-algebra.

\subsection{Definition of invariant functions}
\label{sub:def:inv}
In this section, we introduce the set of the invariant functions as rational functions on $\tld{\cF\cL}_\cJ$ and then give the first precise statement on how they distinguish points in the stack $[\tld{\cF\cL}_{\cJ}\slash{{\sim}_{\un{T}\text{-\textnormal{sh.cnj}}}}]$ (see Statement~\ref{state: separate points prime}).

Consider the set
\begin{equation}\label{equ: general index}
\mathbf{n}_\cJ\defeq \mathbf{n}\times\cJ.
\end{equation}
There is an action of $\Z/f$ on $\cJ$ (with $a\in\Z/f$ acting by $j\mapsto j-a$ on $\cJ$), which induces an action of $\Z/f$ on $\un{W}$ with $a\in\Z/f$ acting by
$$(w_j)_{j\in\cJ}\mapsto (w_{j-a})_{j\in\cJ}$$
for each $w_\cJ=(w_j)_{j\in\cJ}\in\un{W}$. Consequently, we can form the semidirect product
$$\un{W}\rtimes \Z/f$$
with the multiplication given by
$$(w_\cJ,a)\cdot (w_\cJ',a^\prime)=((w_jw_{j-a}^\prime)_{j\in\cJ},a+a^\prime)$$
for $(w_\cJ,a), (w_\cJ',a^\prime)\in \un{W}\rtimes \Z/f$.
Hence, the group $\un{W}\rtimes \Z/f$ has a right action on $\mathbf{n}_\cJ$ given by
\begin{equation}\label{equ: std action}
(k,j)\cdot (w_\cJ,a)\defeq (w_j^{-1}(k),j-a)
\end{equation}
for each $(k,j)\in\mathbf{n}_\cJ$ and $(w_\cJ,a)\in\un{W}\rtimes\Z/f$.

Let $I_\cJ\subseteq\mathbf{n}_\cJ$ be a subset. By abuse of notation, we often write $I_\cJ=(I_j)_{j\in\cJ }$ where
$$I_j\defeq \{1\leq k\leq n\mid (k,j)\in I_\cJ\}\subseteq\mathbf{n}.$$
We consider an element $w_\cJ=(w_j)_{j\in\cJ }\in \un{W}$ and a subset $I_\cJ=(I_j)_{j\in\cJ }\subseteq\mathbf{n}_\cJ$ satisfying
\begin{equation}\label{equ: inv cond}
I_\cJ=I_\cJ\cdot (w_\cJ,1).
\end{equation}
(Equivalently, $w_j^{-1}(I_j)=I_{j-1}$ for each $j\in\cJ$.)
We recall from (\ref{equ: j projection}) the notation $f_{S,j}$ for each $S\subseteq\mathbf{n}$ and $j\in\cJ$. We let $S_{\bullet,w_j}$ be the strictly decreasing sequence corresponding to $w_j$ via (\ref{eq:bjc:sis}), and define the following morphism (with $\cM_{w_\cJ}^\circ=\tld{\cS}^\circ_{\cJ}(w_0,w_0w_\cJ)$)
\begin{equation}\label{equ: def of inv fun}
f_{w_\cJ,I_\cJ}\defeq \prod_{(k,j)\in I_\cJ}\frac{f_{S_{k,w_j},j}}{f_{S_{k+1,w_j},j}}:~\cM_{w_\cJ}^\circ\rightarrow\bG_{m},
\end{equation}
which can be viewed as a rational function on $\tld{\cF\cL}_\cJ$. Here, we understand that $f_{S_{n+1,w_j}}= 1$ for all $j\in\cJ$.
The rational function $f_{w_\cJ,I_\cJ}$ is called an \emph{invariant function} on $\tld{\cF\cL}_\cJ$. If $I_\cJ=\emptyset$, we understand $f_{w_\cJ,I_\cJ}$ to be the constant function $1$ on $\tld{\cF\cL}_\cJ$. Note that $f_{w_\cJ,I_\cJ}$ always determines $I_\cJ$ but not $w_\cJ$ in general. For example, if $I_\cJ=\mathbf{n}_\cJ$, then we always have $f_{w_\cJ,\mathbf{n}_\cJ}=\mathrm{det}_\cJ$ for each $w_\cJ\in\un{W}$, with $\mathrm{det}_\cJ$ defined by
\begin{equation}\label{equ: total det}
\mathrm{det}_\cJ:~\tld{\cF\cL}_\cJ\rightarrow \bG_{m},~(A^{(j)})_{j\in\cJ}\rightarrow \prod_{j\in\cJ}\mathrm{det}(A^{(j)}).
\end{equation}

For each choice of $w_\cJ\in\un{W}$ and $I_\cJ\subseteq\mathbf{n}_\cJ$ satisfying (\ref{equ: inv cond}), we write $I_\cJ^{\rm{c}}\defeq \mathbf{n}_\cJ\setminus I_\cJ$ and $\cM_{w_\cJ,I_\cJ}^\circ$ for the intersection of the regular loci of $f_{w_\cJ,I_\cJ}$ and $f_{w_\cJ,I_\cJ^{\rm{c}}}$ as rational functions on $\tld{\cF\cL}_\cJ$. Hence, the morphism $f_{w_\cJ,I_\cJ}:~\cM_{w_\cJ}^\circ\rightarrow\bG_{m}$ extends to a morphism $\cM_{w_\cJ,I_\cJ}^\circ\rightarrow\bG_{m}$. Note that $\cM_{w_\cJ,I_\cJ}^\circ=\cM_{w_\cJ,I_\cJ^{\rm{c}}}^\circ$ and that the set $\{f_{w_\cJ,I_\cJ},f_{w_\cJ,I_\cJ^{\rm{c}}}\}$ can be recovered from the open subscheme $\cM_{w_\cJ,I_\cJ}^\circ$. It is always true that $\cM_{w_\cJ}^\circ\subseteq\cM_{w_\cJ,I_\cJ}^\circ$ but the inclusion could be strict in general (for example, $\cM_{w_\cJ,\mathbf{n}_\cJ}^{\circ}=\tld{\cF\cL}_\cJ$ for each $w_\cJ\in \un{W}$).

For each $I\subseteq\mathbf{n}$, we can decompose $\mathbf{n}$ into a disjoint union of a minimal number of sets of consecutive integers each of which sits either in $I$ or in $I^{\rm{c}}\defeq\mathbf{n}\setminus I$. We can associate a standard Levi subgroup $M_{I}\subseteq \GL_n$, with each set of consecutive integers corresponding to a Levi block. Note that we have $M_{I}=M_{I^{\rm{c}}}$.
Hence, applying the construction to $I_\cJ=(I_j)_{j\in\cJ}$ we obtain a standard Levi subgroup $M_{I_\cJ}=(M_{I_j})_{j\in\cJ}\subseteq\un{G}$ whose associated Weyl group is denoted by $W_{I_\cJ}=(W_{I_j})_{j\in\cJ}\subseteq\un{W}$.

\begin{lemma}\label{lem: covering regular locus}
Let $w_\cJ\in\un{W}$ and let $I_\cJ\subseteq\mathbf{n}_\cJ$ satisfy (\ref{equ: inv cond}).
We have the following:
\begin{enumerate}[label=(\roman*)]
\item
\label{it: covering regular locus:i}
If $w'_\cJ\in\un{W}$ and $I'_\cJ\subseteq\mathbf{n}_\cJ$ satisfies (\ref{equ: inv cond}) then
$f_{w_\cJ,I_\cJ}=f_{w_\cJ',I_\cJ'}$ if and only if $I_\cJ=I_\cJ'$ and $w_\cJ^{-1}w_\cJ'\in W_{I_\cJ}$;
\item
\label{it: covering regular locus:ii}
$\cM_{w_\cJ,I_\cJ}^{\circ}=\bigcup_{w_\cJ'\in w_\cJ\cdot W_{I_\cJ}}\cM_{w_\cJ'}^\circ$.
\end{enumerate}
In particular, if $\cC\in\cP_\cJ$  is contained in $\cM_{w_\cJ,I_\cJ}^{\circ}$,
then there exists $w_\cJ'\in w_\cJ\cdot W_{I_\cJ}$ such that $\cC\subseteq\cM_{w_\cJ'}^{\circ}$.
\end{lemma}

\begin{proof}
For convenience, in this proof we write
$$I_j^+\defeq \{k\in I_j\mid k-1\in I_j^{\rm{c}}\}\textnormal{ and }I_j^-\defeq \{k\in I_j^{\rm{c}}\mid k-1\in I_j\}$$
for each $j\in\cJ$. There is a unique way to write $f_{w_\cJ,I_\cJ}$ (resp. $f_{w_\cJ',I_\cJ'}$) as a rational function with coprime numerator and denominator, each of them a product of $f_{S_{k,w_j},j}$ for certain choices of $k\in\mathbf{n}$ and $j\in\cJ$. More precisely, $f_{S_{k,w_j},j}$ appears in the numerator (resp. in the denominator) if and only if $k\in I_j^+$ (resp. if and only if $k\in I_j^-$). We observe that $f_{w_\cJ,I_\cJ}=f_{w_\cJ',I_\cJ'}$ is equivalent to the condition that $I_\cJ=I_\cJ'$ and that $S_{k,w_j}=S_{k,w_j'}$ for each $k\in I_j^+\sqcup I_j^-$ and each $j\in\cJ$.
Hence \ref{it: covering regular locus:i} follows from the observation that $w_\cJ^{-1}w_\cJ'\in W_{I_\cJ}$ if and only if $S_{k,w_j}=S_{k,w_j'}$ for each $k\in I_j^+\sqcup I_j^-$ and each $j\in\cJ$.  Concerning \ref{it: covering regular locus:ii}, there exists $\cM_{w_j,I_j}^{\circ}\subseteq\tld{\cF\cL}$ for each $j\in\cJ$ such that
\begin{equation}\label{equ: decompose into embedding}
\cM_{w_\cJ,I_\cJ}^{\circ}=\prod_{j\in\cJ}\cM_{w_j,I_j}^{\circ}.
\end{equation}
Writing $f_{w_\cJ,I_\cJ}$ as a rational function with coprime numerator and denominator, we see from the definition of $\cM_{w_j,I_j}^{\circ}$ and Proposition~\ref{prop: union of open elements} that
\begin{equation}\label{equ: open cover one embedding}
\cM_{w_j,I_j}^{\circ}=\bigcap_{S\in\Sigma_j}\cH_S^{\rm{c}} = \underset{S_\bullet \supseteq\Sigma_j}{\bigcup}\cM_{S_\bullet}^\circ
\end{equation}
(taking $\Sigma$ to be $\Sigma_j\defeq \{S_{k,w_j}\mid k\in I_j^+\sqcup I_j^-\}$ in \emph{loc.~cit.}).
Then a crucial observation is that $S_\bullet \supseteq\Sigma_j$ if and only if $S_\bullet=S_{\bullet,w_j'}$ for some $w_j'\in W$ satisfying $w_j^{-1}w_j'\in W_{I_j}$, which together with (\ref{equ: decompose into embedding}) and (\ref{equ: open cover one embedding}) finish the proof of \ref{it: covering regular locus:ii}. The last part is obvious from \ref{it: covering regular locus:ii}.
\end{proof}

\begin{lemma}\label{lem: descend to quotient}
The rational function $f_{w_\cJ,I_\cJ}$ descends to $\tld{\cF\cL}_{\cJ}\slash{{\sim}_{\un{T}\text{-\textnormal{sh.cnj}}}}$ for each $w_\cJ\in\un{W}$ and $I_\cJ\subseteq\mathbf{n}_\cJ$ satisfying (\ref{equ: inv cond}).
\end{lemma}

\begin{proof}
For each Noetherian $\F$-algebra $R$, we use the notation
$$t=(t^{(j)})_{j\in\cJ}\in \un{T}(R),~A=(A^{(j)})_{j\in\cJ}\in \un{G}(R)$$
and recall the right action of $\un{T}$:
$$\un{G}(R)\times \un{T}(R)\rightarrow \un{G}(R),~(A,t)\mapsto A\cdot t=((t^{(j+1)})^{-1}A^{(j)}t^{(j)})_{j\in\cJ}.$$
We define $X$ as the fiber of $\un{G}\twoheadrightarrow\tld{\cF\cL}_\cJ$ over $\cM_{w_\cJ,I_\cJ}^{\circ}$ and abuse the notation $f_{w_\cJ,I_\cJ}$ for the composition
$$f_{w_\cJ,I_\cJ}:~X\rightarrow\cM_{w_\cJ,I_\cJ}^{\circ}\rightarrow\bG_{m}.$$
It suffices to show that
\begin{equation}\label{equ: T inv}
f_{w_\cJ,I_\cJ}(A\cdot t)=f_{w_\cJ,I_\cJ}(A)
\end{equation}
for each $A=(A^{(j)})_{j\in\cJ}\in X(R)$ and $t=(t^{(j)})_{j\in\cJ}\in\un{T}(R)$. For each $(k,j)\in\mathbf{n}_\cJ$, we write $t^{(j)}_k$ for the $k$-th diagonal entry of $t^{(j)}$. We observe that
$$\frac{f_{S_{k,w_j},j}}{f_{S_{k+1,w_j},j}}((t^{(j+1)})^{-1}A^{(j)}t^{(j)})=(t^{(j+1)}_k)^{-1}t^{(j)}_{w_j^{-1}(k)} \frac{f_{S_{k,w_j},j}}{f_{S_{k+1,w_j},j}}(A^{(j)}).$$
This together with (\ref{equ: inv cond}) implies (\ref{equ: T inv}) by taking product over all $(k,j)\in I_\cJ$.
\end{proof}

We set
$$\Inv\defeq \{f_{w_\cJ,I_\cJ}\mid w_\cJ\in\un{W},\,I_\cJ\subseteq\mathbf{n}_\cJ,\,
I_\cJ\cdot (w_\cJ,1)=I_\cJ\}.$$
For each $\cC\in\cP_\cJ$, we write $\Inv(\cC)\subseteq\Inv$ for the subset consisting of those $f_{w_\cJ,I_\cJ}$ which are invertible over $\cC$ (namely $\cC\subseteq\cM_{w_\cJ,I_\cJ}^\circ$). The set $\Inv(\cC)$ induces a morphism of stacks
$$
\iota_\cC:~[\cC\slash{{\sim}_{\un{T}\text{-\textnormal{sh.cnj}}}}]\rightarrow
(\bG_{m})^{\# \Inv(\cC)}.
$$
The following is the main property satisfied by the set $\Inv(\cC)$.
\begin{state}\label{state: separate points prime}
For each $\cC\in\cP_\cJ$ and Noetherian $\F$-algebra $R$, the following map induced from~$\iota_\cC$
$$|[\cC\slash{{\sim}_{\un{T}\text{-\textnormal{sh.cnj}}}}]|(R)\rightarrow
(R^\times)^{\# \Inv(\cC)}.$$
is injective.
\end{state}
We will deduce Statement~\ref{state: separate points prime} from Statement~\ref{state: goal}) (to be introduced in \S\,\ref{sub:main:state}) whose proof is very involved and will occupy \S\,\ref{sec:comb:lifts}, \S\,\ref{sec:const:inv} and \S\,\ref{sec:inv:cons}.

\subsection{From stacks to schemes}\label{sub:stack:scheme}
Recall that we expect the set of invariant functions $\Inv(\cC)$ to satisfy Statement~\ref{state: separate points prime}, which \emph{a priori} involves the algebraic stack $[\cC\slash{{\sim}_{\un{T}\text{-\textnormal{sh.cnj}}}}]$. In this section, we give an explicit construction of the geometric quotient $\cN_{\xi,\Lambda}\slash{{\sim}_{\un{T}\text{-\textnormal{sh.cnj}}}}$ in Proposition~\ref{prop: pass to scheme}, which implies the existence of the geometric quotient $\cC_{\xi,\Lambda}\slash{{\sim}_{\un{T}\text{-\textnormal{sh.cnj}}}}$ (see Proposition~\ref{prop: quotient of strata}). This allows us to introduce a convenient sufficient condition for Statement~\ref{state: separate points prime} in \S\,\ref{sub:main:state}. Note that $\cN_{\xi,\Lambda}$ is a topological union of elements in $\cP_\cJ$ (see Lemma~\ref{lem:union:partition}).

We fix an element $\xi=(w_\cJ,u_\cJ)\in\Xi_{w_\cJ}$ for some $w_\cJ\in \un{W}$ throughout this section.

\subsubsection{$\mathrm{Supp}^\square_\xi$ as a graph}\label{subsub:graph}
We recall the set $\mathrm{Supp}^\square_\xi\subseteq\Phi^+_{\GL_{r_\xi}}$ from the end of \S\,\ref{sub:std:note}. We can associate an undirected graph $\mathfrak{G}_\xi$ with $\mathrm{Supp}^\square_\xi$ in the following way:
\begin{itemize}
\item the set of vertices of $\mathfrak{G}_\xi$, written $V(\mathfrak{G}_\xi)$, is in bijection with $\{1,\dots,r_\xi\}$;
\item the set of edges of $\mathfrak{G}_\xi$, written $E(\mathfrak{G}_\xi)$, is in bijection with $\mathrm{Supp}^\square_\xi$, so that there exists an edge connecting two vertex $h<\ell$ if and only if $(h,\ell)\in \mathrm{Supp}^\square_\xi$.
\end{itemize}
Similarly, we write $E(\cdot)$ (resp. $V(\cdot)$) for the set of edges (resp. the set of vertices) for an arbitrary graph.

\begin{defn}
\label{def:loop}
Let $\mathfrak{G}\subseteq\mathfrak{G}_\xi$ be an arbitrary subgraph. A \emph{directed loop} inside $\mathfrak{G}$, written $\Gamma$, is defined to be an ordered pair of non-empty subsets $E(\Gamma)^+, E(\Gamma)^-\subseteq E(\mathfrak{G})$ satisfying the following:
\begin{itemize}
\item we have $\sum_{\gamma\in E(\Gamma)^+}\gamma=\sum_{\gamma\in E(\Gamma)^-}\gamma$;
\item we have either $E(\Gamma)^+\cap E(\Gamma)^-=\emptyset$ or $E(\Gamma)^+=E(\Gamma)^-=\{\gamma\}$ for some $\gamma\in E(\mathfrak{G})$; and
\item for any proper non-empty subset $E^{+}\subsetneq E(\Gamma)^+$ (resp.~$E^{-}\subsetneq E(\Gamma)^-$) we have $\sum_{\gamma\in E^{+}}\gamma\neq \sum_{\gamma\in E^{-}}\gamma$.
\end{itemize}
\end{defn}
If $\Gamma$ is a directed loop we also define $V(\Gamma)\subseteq V(\mathfrak{G})$ as the subset consisting of all the elements $m\in V(\mathfrak{G})$ such that at least one of $(m,m')$ and $(m',m)$ belongs to $E(\Gamma)^+\cup E(\Gamma)^-$ for some choice of $m'\in V(\mathfrak{G})$. Note that if $E(\Gamma)^+\cap E(\Gamma)^-=\emptyset$, then this notion of directed loop in Definition~\ref{def:loop} coincides with the usual one, namely picking up a connected subgraph of $\mathfrak{G}$ which is homeomorphic to a circle and then equipping this subgraph with a choice of direction such that the in-degree and out-degree of each vertex are one. In other words, we extend the usual notion of directed loop by allowing some degenerate cases when $E(\Gamma)^+=E(\Gamma)^-=\{\gamma\}$ for some $\gamma\in E(\mathfrak{G})$

\subsubsection{Functions invariant under shifted $\un{T}$-conjugation}\label{subsub:conj inv}
We recall the set $\mathbf{n}_\cJ$ from (\ref{equ: general index}) and, for $1\leq m\leq r_\xi$, the set $[m]_\xi$ from (\ref{equ: single block of integers}). We also recall from (\ref{equ: std action}) that there is a right action of $\un{W}\rtimes\Z/f$ on $\mathbf{n}_\cJ$. We write $\langle (w_\cJ,1)\rangle$ for the cyclic subgroup of $\un{W}\rtimes\Z/f$ generated by $(w_\cJ,1)$. We define $$I_\cJ^m\defeq \{(k,j)\mid u_j^{-1}(k)\in [m]_\xi\}\subseteq\mathbf{n}_\cJ$$ for each $1\leq m\leq r_\xi$. We also define
\begin{equation}\label{equ: ss function}
F_\xi^m\defeq \prod_{(k,j)\in I_\cJ^m}D_{\xi,k}^{(j)}:~\cN_\xi\ra\bG_{m}
\end{equation}
for each $1\leq m\leq r_\xi$, where $D_{\xi,k}^{(j)}$ was defined by equation (\ref{eq:Dxij}).
\begin{lemma}\label{lem: bijection orbit}
The map $m\mapsto I_\cJ^m$ gives a bijection between $\{1,\dots,r_\xi\}$ and the set of $\langle(w_\cJ,1)\rangle$-orbits inside $\mathbf{n}_\cJ$. In particular, $I_\cJ^m$ depends only on $w_\cJ$ and not on $\xi$.
\end{lemma}
\begin{proof}
Let $(k,j)\in I_\cJ^m$ be an arbitrary element, and thus $u_j^{-1}(k)\in[m]_\xi$. It follows from Definition~\ref{def: niveau partition} that $u_{j-1}^{-1}w_j^{-1}u_j=(u_j^{-1}w_ju_{j-1})^{-1}\in M_\xi$. Hence $u_{j-1}^{-1}w_j^{-1}u_j$ stablizes $[m]_\xi$, from which we deduce that
$$u_{j-1}^{-1}w_j^{-1}(k)=(u_{j-1}^{-1}w_j^{-1}u_j)(u_j^{-1}(k))\in[m]_\xi.$$
Thus we have $(k,j)\cdot (w_\cJ,1)=(w_j^{-1}(k),j-1)\in I_\cJ^m$. Consequently, $I_\cJ^m$ is a disjoint union of $\langle(w_\cJ,1)\rangle$-orbits.

Now we fix an element $(k,j)\in I_\cJ^m$ and count the cardinality of the $\langle(w_\cJ,1)\rangle$-orbit containing $(k,j)$. Let $c$ be the minimal positive integer such that $(k,j)\cdot (w_\cJ,1)^c=(k,j)$. According to definition of the action of $(w_\cJ,1)$, it is clear that there exists $b\geq 1$ such that $c=bf$ and that $(k,j)\cdot (w_\cJ,1)^c=((w_j^\flat)^{-b}(k),j)$. In other words, we have
\begin{equation}\label{equ: one orbit}
u_j^{-1}(k)=(u_jw_j^\flat u_j^{-1})^{-b}u_j^{-1}(k)
\end{equation}
and $b$ is the minimal positive integer satisfying (\ref{equ: one orbit}). Then it follows from Definition~\ref{def: niveau partition} that $M_{u_jw_j^\flat u_j^{-1}}=M_\xi$, which together with $u_j^{-1}(k)\in [m]_\xi$ imply that $b=\#[m]_\xi=n_m^\xi$. Hence we deduce that the cardinality of the $\langle(w_\cJ,1)\rangle$-orbit containing $(k,j)$ equals $\#I_\cJ^m=fn_m^\xi$, which implies that $I_\cJ^m$ forms a single $\langle(w_\cJ,1)\rangle$-orbit. Hence we finish the proof.
\end{proof}

If $(k_1,j_1),\,(k_2,j_2)\in I_\cJ^m$ for some $1\leq m\leq r_\xi$, we define
\begin{equation}\label{equ: general interval}
](k_1,j_1),(k_2,j_2)]_{w_\cJ}
\defeq \{(k_1,j_1)\cdot (w_\cJ,1)^x\mid 1\leq x\leq b\}\subseteq I_\cJ^m
\end{equation}
where $1\leq b\leq f'$ is the minimal possible integer that satisfies
$$(k_2,j_2)=(k_1,j_1)\cdot (w_\cJ,1)^{b}.$$
It is easy to see that the definition of $](k_1,j_1),(k_2,j_2)]_{w_\cJ}$ depends only on $w_\cJ$ and not on $\xi$.

Now we recall the graph $\mathfrak{G}_\xi$ from \S\,\ref{subsub:graph} and pick a directed loop $\Gamma$ inside $\mathfrak{G}_\xi$ (see \S\,\ref{subsub:graph} for the definition of a directed loop).

\begin{defn}\label{def: lift of loop}
A pair of disjoint subsets $\Omega^+,~\Omega^-\subseteq\mathrm{Supp}_{\xi,\cJ}$ is called a \emph{lift} of $\Gamma$ if $\Omega^+$ (resp.~$\Omega^-$) maps bijectively to $E(\Gamma)^+$ (resp.~$E(\Gamma)^-$) under the surjection $\mathrm{Supp}_{\xi,\cJ}\twoheadrightarrow\mathrm{Supp}^\square_\xi$. Given a subset $\Lambda\subseteq \mathrm{Supp}_{\xi,\cJ}$, we say that a pair $\Omega^+,\Omega^-$ is a \emph{lift of $\Gamma$ supported in $\Lambda$}, if it is a lift of $\Gamma$ and $\Omega^+,~\Omega^-\subseteq\Lambda$.
We say that a pair $\Omega^+,\Omega^-$ is a \emph{$\Lambda$-lift} if it is a lift supported in $\Lambda$ of a directed loop inside $\mathfrak{G}_\xi$.
\end{defn}
We use the shortened notation $\Omega^\pm$ for the pair of sets $\Omega^+$ and $\Omega^-$. Note that if $E(\Gamma)^+=E(\Gamma)^-=\{\gamma\}$ for some $\gamma\in\mathrm{Supp}^\square_\xi$, then to choose a lift $\Omega^\pm$ of $\Gamma$ is equivalent to choose two distinct elements in $\mathrm{Supp}_{\xi,\cJ}^\gamma$.

We use the notation $\al=(i_\al,i_\al^\prime)$ for each $\al\in\Phi^+$. We consider a directed loop $\Gamma$ inside $\mathfrak{G}_\xi$ as well as a lift $\Omega^\pm$ of it.
Let $m$ be an element in $V(\Gamma)$. If we write $m\rightarrow m'$ (resp. $m'\rightarrow m$) for the edge of a directed loop $\Gamma$ indicating the direction by $\rightarrow$, we write $(\al_m^+,j_m^+)\in\Omega^+\sqcup \Omega^-$ (resp. $(\al_m^-,j_m^-)\in\Omega^+\sqcup \Omega^-$) for the element corresponding to the edge $m\rightarrow m'$ (resp. $m'\rightarrow m$) under the surjection $\mathrm{Supp}_{\xi,\cJ}\twoheadrightarrow\mathrm{Supp}^\square_\xi$. Namely, there exists an element $(\al_m^+,j_m^+)$ (resp.~$(\al_m^-,j_m^-)$) in $\Omega^+\sqcup \Omega^-$ such that the following holds:
\begin{itemize}
\item  $(\al_m^+,j_m^+)\in\mathrm{Supp}_{\xi,\cJ}^{(m,m')}\cap\Omega^+$ if $m'>m$, and $(\al_m^+,j_m^+)\in \mathrm{Supp}_{\xi,\cJ}^{(m',m)}\cap \Omega^-$ if $m'<m$;
\item $(\al_m^-,j_m^-)\in\mathrm{Supp}_{\xi,\cJ}^{(m',m)}\cap \Omega^+$ if $m'<m$, and $(\al_m^-,j_m^-)\in \mathrm{Supp}_{\xi,\cJ}^{(m,m')}\cap \Omega^-$ if $m'>m$.
\end{itemize}
Then we set
\begin{equation}\label{equ: two different cases}
k_{\Omega^\pm,m}^\bullet=\left\{\begin{array}{ll}
u_{j_m^\bullet}(i_{\al_m^\bullet})&\hbox{if $m'>m$;}\\
u_{j_m^\bullet}(i_{\al_m^\bullet}^\prime)&\hbox{if $m'<m$}
\end{array}\right.
\end{equation}
for each $\bullet\in\{+,-\}$. Note that we have $u_{j_m^\bullet}^{-1}(k_{\Omega^\pm,m}^\bullet)\in[m]_\xi$ for each $\bullet\in\{+,-\}$. We define
$$F_\xi^{\Omega^\pm,\flat}\defeq \prod_{m\in V(\Gamma)}\prod_{(k,j)\in I_\cJ^{\Omega^\pm,m}} D_{\xi,k}^{(j)}:~\cN_\xi\rightarrow \bG_m$$
where
\begin{equation}\label{equ: m interval}
I_\cJ^{\Omega^\pm,m}\defeq \left\{\begin{array}{cl}
](k_{\Omega^\pm,m}^-,j_m^-),(k_{\Omega^\pm,m}^+,j_m^+)]_{w_\cJ}&\hbox{if $(k_{\Omega^\pm,m}^-,j_m^-)\neq (k_{\Omega^\pm,m}^+,j_m^+)$};\\
\emptyset&\hbox{if $(k_{\Omega^\pm,m}^-,j_m^-)=(k_{\Omega^\pm,m}^+,j_m^+)$},
\end{array}\right.
\end{equation}
for each $m\in V(\Gamma)$. We also define
\begin{equation}\label{equ: sharp factor}
F_\xi^{\Omega^\pm,\sharp}\defeq
\frac{\prod_{(\al,j)\in \Omega^+}u_\xi^{(\al,j)}}{\prod_{(\al,j)\in\Omega^-}u_\xi^{(\al,j)}}.
\end{equation}
Then we set
\begin{equation}\label{equ: loop function}
F_\xi^{\Omega^\pm}\defeq F_\xi^{\Omega^\pm,\flat}\cdot F_\xi^{\Omega^\pm,\sharp}.
\end{equation}
Hence $F_\xi^{\Omega^\pm}$ is a rational function on $\cN_\xi$.

\begin{lemma}\label{lem: loop function}
The rational function $F_\xi^{\Omega^\pm}$ descends to $[\cN_\xi\slash{{\sim}_{\un{T}\text{-\textnormal{sh.cnj}}}}]$ for each choice of lift $\Omega^\pm$ of some directed loop $\Gamma$ as above. Similarly, the function $F_\xi^m$ descends to $[\cN_\xi\slash{{\sim}_{\un{T}\text{-\textnormal{sh.cnj}}}}]$ for each $1\leq m\leq r_\xi$.
\end{lemma}
\begin{proof}
We only prove the case of $F_\xi^{\Omega^\pm}$, as the proof for $F_\xi^m$ is simpler. We write $X\subseteq\cN_\xi$ for the open subscheme defined by the condition that $u_\xi^{(\al,j)}\neq 0$ for each $(\al,j)\in\Omega^+\sqcup\Omega^-$. In particular, $X$ is inside the regular locus of $F_\xi^{\Omega^\pm}$ and we only need to prove that
$$F_\xi^{\Omega^\pm}(A\cdot t)=F_\xi^{\Omega^\pm}(A)$$
for each $A=(A^{(j)})_{j\in\cJ}\in X(R)$ and $t=(t^{(j)})_{j\in\cJ}\in \un{T}(R)$. We write $t^{(j)}=\Diag(t^{(j)}_1,\dots,t^{(j)}_n)$ for convenience. Then we observe that
\begin{equation}\label{equ: T conj formula ss}
D_{\xi,k}^{(j)}(A\cdot t)=\big(t^{(j+1)}_k\big)^{-1}t^{(j)}_{w_j^{-1}(k)}D_{\xi,k}^{(j)}(A)
\end{equation}
for each $(k,j)\in\mathbf{n}_\cJ$, and
\begin{equation}\label{equ: T conj formula u}
u_\xi^{(\al,j)}(A\cdot t)=\big(t^{(j)}_{w_j^{-1}u_j(i_\al)}\big)^{-1}t^{(j)}_{w_j^{-1}u_j(i_\al^\prime)}u_\xi^{(\al,j)}(A)
\end{equation}
for each $(\al,j)\in\mathrm{Supp}_{\xi,\cJ}$. It follows from (\ref{equ: T conj formula ss}) and the definition of $I_\cJ^{\Omega^\pm,m}$ (see (\ref{equ: general interval}) and (\ref{equ: m interval})) that
$$
\prod_{(k,j)\in I_\cJ^{\Omega^\pm,m}} D_{\xi,k}^{(j)}(A\cdot t)=\big(t^{(j_m^-)}_{w_{j_m^-}^{-1}(k_{\Omega^\pm,m}^-)}\big)^{-1}t^{(j_m^+)}_{w_{j_m^+}^{-1}(k_{\Omega^\pm,m}^+)}\prod_{(k,j)\in I_\cJ^{\Omega^\pm,m}} D_{\xi,k}^{(j)}(A)
$$
for each $m\in V(\Gamma)$. It follows from (\ref{equ: T conj formula u}) and (\ref{equ: sharp factor}) that
$$F_\xi^{\Omega^\pm,\sharp}(A\cdot t)=\frac{\prod_{(\al,j)\in \Omega^+}\big(t^{(j)}_{w_j^{-1}u_j(i_\al)}\big)^{-1}t^{(j)}_{w_j^{-1}u_j(i_\al^\prime)}}{\prod_{(\al,j)\in\Omega^-} \big(t^{(j)}_{w_j^{-1}u_j(i_\al)}\big)^{-1}t^{(j)}_{w_j^{-1}u_j(i_\al^\prime)}}F_\xi^{\Omega^\pm,\sharp}(A).$$
Hence, it remains to prove that
$$
\left(\prod_{m\in V(\Gamma)}\big(t^{(j_m^-)}_{w_{j_m^-}^{-1}(k_{\Omega^\pm,m}^-)}\big)^{-1}t^{(j_m^+)}_{w_{j_m^+}^{-1}(k_{\Omega^\pm,m}^+)}\right)\frac{\prod_{(\al,j)\in \Omega^+}\big(t^{(j)}_{w_j^{-1}u_j(i_\al)}\big)^{-1}t^{(j)}_{w_j^{-1}u_j(i_\al^\prime)}}{\prod_{(\al,j)\in\Omega^-} \big(t^{(j)}_{w_j^{-1}u_j(i_\al)}\big)^{-1}t^{(j)}_{w_j^{-1}u_j(i_\al^\prime)}}=1,
$$
which is a consequence of (\ref{equ: two different cases}). Hence we finish the proof.
\end{proof}

\subsubsection{Explicit geometric quotient}\label{subsub: remove}
Let $\Lambda$ be a subset of $\mathrm{Supp}_{\xi,\cJ}$ with $\Lambda^\square$ its image in $\mathrm{Supp}^\square_\xi$, and recall the definitions of $\cN_{\xi,\Lambda}\subseteq\cN_\xi$ from \S\,\ref{sub:std:note}. We recall from (\ref{equ: loop function}) the rational function $F_\xi^{\Omega^\pm}$ on $\cN_\xi$. If $\Omega^\pm$ is a $\Lambda$-lift (see Definition~\ref{def: lift of loop}), then $F_\xi^{\Omega^\pm}$ clearly restricts to an invertible function on $\cN_{\xi,\Lambda}$. We abuse the same notation $F_\xi^{\Omega^\pm}$ for this restriction. Similarly, we also abuse the notation $F_\xi^m$ (see (\ref{equ: ss function})) for its restriction to $\cN_{\xi,\Lambda}$. In the following, we will use functions of the form $F_\xi^{\Omega^\pm}$ and $F_\xi^m$ to explicitly construct the geometric quotient $\cN_{\xi,\Lambda}\slash{{\sim}_{\un{T}\text{-\textnormal{sh.cnj}}}}$ in Proposition~\ref{prop: pass to scheme}.

We can naturally associate a subgraph $\mathfrak{G}_{\xi,\Lambda}\subseteq\mathfrak{G}_\xi$ with the subset $\Lambda^\square\subseteq \mathrm{Supp}^\square_\xi$. We fix a choice of a subset $\cB\subseteq\Lambda$ that maps bijectively to a subset of $\Lambda^{\square}$, denoted by $\cB^\square$, under $\Lambda\twoheadrightarrow\Lambda^\square$ such that the subgraph of $\mathfrak{G}_{\xi,\Lambda}$ corresponding to $\cB^\square$ is a maximal tree (a not necessarily connected maximal possible subgraph such that the underlying topological space of each connected component is simply connected).
As a result, for each $\gamma\in \Lambda^\square$ there exists a unique directed loop $\Gamma_{\gamma,\cB}$ inside $\mathfrak{G}_{\xi,\Lambda}$ (see Definition~\ref{def:loop}) such that $E(\Gamma_{\gamma,\cB})^+\cup E(\Gamma_{\gamma,\cB})^-\subseteq\cB^\square\cup \{\gamma\}$ and $\gamma\in E(\Gamma_{\gamma,\cB})^+$. For each element $(\al,j)\in\Lambda\setminus\cB$ with $\gamma$ its image in $\Lambda^\square$, there exists a unique $\Lambda$-lift $\Omega_{(\al,j),\cB}^\pm$ of $\Gamma_{\gamma,\cB}$ such that $\Omega_{(\al,j),\cB}^+\sqcup\Omega_{(\al,j),\cB}^-\subseteq \cB\sqcup\{(\al,j)\}$ and $(\al,j)\in\Omega_{(\al,j),\cB}^+$. Then we set
\begin{equation}
\label{eq:omega:b}
F_{\xi,\Lambda}^{(\al,j),\cB}\defeq F_\xi^{\Omega_{(\al,j),\cB}^\pm}.
\end{equation}

Now we consider the following morphism
$$p_{\xi,\Lambda}:~\cN_{\xi,\Lambda}\rightarrow \bG_m^{r_\xi}\times\bG_m^{\#\Lambda-\#\cB}$$
given by $(F_\xi^1,\dots, F_\xi^{r_\xi})$ on the first $r_\xi$ coordinates and $(F_{\xi,\Lambda}^{(\al,j),\cB})_{(\al,j)\in\Lambda\setminus\cB}$ for the rest.
We write $\cO(\cN_{\xi,\Lambda})$ (resp.~$\cO(\cN_{\xi,\Lambda})^\times$) for the ring of global sections (resp.~for the group of invertible global sections) on $\cN_{\xi,\Lambda}$. We also write $\cO(\cN_{\xi,\Lambda})^{\un{T}\text{-\textnormal{sh.cnj}}}\subseteq\cO(\cN_{\xi,\Lambda})$ for the subring consisting of global sections invariant under shifted-$\un{T}$-conjugation. We understand \emph{monomials} to have degrees in $\Z$.

\begin{lemma}\label{lem: quotient of torus by torus}
Let $T_0$ be a split torus, $r\geq 1$ be an integer and $(\chi_i)_{1\leq i\leq r}\in X(T_0)^r$ be $r$-tuple of characters of $T_0$. Let $r_1$ be the rank of the span of $\chi_1,\cdots,\chi_r$ in $X(T_0)$. We consider the $T_0$-action on $\bG_m^r$ given by
$$T_0\times \bG_m^r\rightarrow\bG_m^r:~(t,x_1,\cdots,x_r)\mapsto (\chi_1(t)x_1,\cdots \chi_r(t)x_r).$$
Then the geometric quotient $\bG_m^r\slash_{{\sim}_{T_0}}$ exists and is a split torus of rank $r-r_1$. In particular, $T_0$ acts transitively on $\bG_m^r$ if and only if $r_1=r$.
\end{lemma}
\begin{proof}
This is clear.
\end{proof}

\begin{prop}\label{prop: pass to scheme}
The geometric quotient $\cN_{\xi,\Lambda}\slash{{\sim}_{\un{T}\text{-\textnormal{sh.cnj}}}}$ exists and $p_{\xi,\Lambda}$ induces an isomorphism
$$\cN_{\xi,\Lambda}\slash{{\sim}_{\un{T}\text{-\textnormal{sh.cnj}}}}\xrightarrow{\sim}\bG_m^{r_\xi}\times\bG_m^{\#\Lambda-\#\cB}.$$
In particular, we have the following natural isomorphism
\begin{equation}\label{equ: inv subring}
\cO(\cN_{\xi,\Lambda})^{\un{T}\text{-\textnormal{sh.cnj}}}\cong \F[(F_\xi^m)^{\pm1}\mid 1\leq m\leq r_\xi][(F_\xi^{(\al,j),\cB})^{\pm1}\mid (\al,j)\in\Lambda\setminus\cB].
\end{equation}
\end{prop}
\begin{proof}
The existence of the geometric quotient $\cN_{\xi,\Lambda}\slash{{\sim}_{\un{T}\text{-\textnormal{sh.cnj}}}}$ (which is a split torus) follows directly from Lemma~\ref{lem: quotient of torus by torus}, and it suffices to prove (\ref{equ: inv subring}). More precisely, we prove that any monomial in $\cO(\cN_{\xi,\Lambda})$ (with variables $\{D_{\xi,k}^{(j)}\mid (k,j)\in\mathbf{n}_\cJ\}$ and $\{u_\xi^{(\al,j)}\mid (\al,j)\in\Lambda\}$) invariant under the $\un{T}$-action must have the form $$c\underset{1\leq m\leq r_\xi}{\prod}(F_\xi^m)^{d_m}\underset{(\al,j)\in\Lambda\setminus\cB}{\prod}(F_\xi^{(\al,j),\cB})^{n_{(\al,j)}}$$
for some choice of $c\in\F^\times$, $(d_1,\dots,d_{r_\xi})\in\Z^{r_\xi}$ and $(n_{(\al,j)})_{(\al,j)\in\Lambda\setminus\cB}\in\Z^{\#\Lambda-\#\cB}$ (and the choice is clearly unique).

We first claim that any monomial $F\in\cO(\cN_{\xi,\Lambda})^{\un{T}\text{-\textnormal{sh.cnj}}}$ with variables $\{D_{\xi,k}^{(j)}\mid (k,j)\in\mathbf{n}_\cJ\}$ is a monomial with variables $\{F_\xi^m\mid 1\leq m\leq r_\xi\}$.
In fact, $F\in\cO(\cN_{\xi,\Lambda})^{\un{T}\text{-\textnormal{sh.cnj}}}$ together with the formula (\ref{equ: T conj formula ss}) forces the degree of $D_{\xi,k}^{(j)}$ in $F$ to be a constant function on each $(w_\cJ,1)$-orbit in $\mathbf{n}_\cJ$, and the claim clearly follows.

Now we fix an arbitrary monomial $F\in\cO(\cN_{\xi,\Lambda})^{\un{T}\text{-\textnormal{sh.cnj}}}$ and write $n_{(\al,j)}\in\Z$ for the degree of $u_\xi^{(\al,j)}$, for each $(\al,j)\in\Lambda$. Let $Z_\xi$ be the center of $M_\xi$ and consider the $Z_\xi$-action on $\cN_{\xi,\Lambda}$ induced from the embedding
$Z_\xi\hookrightarrow \un{T}$ given by $z\mapsto (u_j z u_j^{-1})_{j\in\cJ}$. The fact that $F$ is $Z_\xi$-invariant together with the formula (\ref{equ: T conj formula u}) implies that
\begin{equation}\label{equ: balanced cond on deg}
\underset{\gamma\in\Lambda^\square}{\sum}\left(\underset{(\al,j)\in\Lambda\cap\mathrm{Supp}_{\xi,\cJ}^\gamma}{\sum}n_{(\al,j)}\right)\gamma=0.
\end{equation}
We write $F_\xi^{(\al,j),\cB}\defeq 1$ for each $(\al,j)\in\cB$ for convenience. Then it follows from the choice of $\cB$ (with $\cB^\square$ being a basis for the $\Z$-span of $\Lambda^\square$) and (\ref{equ: balanced cond on deg}) that
$\underset{(\al,j)\in\Lambda}{\prod}(u_\xi^{(\al,j)}(F_\xi^{(\al,j),\cB})^{-1})^{n_{(\al,j)}}$
is a monomial with variables $\{D_{\xi,k}^{(j)}\mid (k,j)\in\mathbf{n}_\cJ\}$. Consequently, we obtain an element
$$\tld{F}\defeq F\underset{(\al,j)\in\Lambda}{\prod}(F_\xi^{(\al,j),\cB})^{-n_{(\al,j)}}\in \cO(\cN_{\xi,\Lambda})^{\un{T}\text{-\textnormal{sh.cnj}}}$$
which is a monomial with variables $\{D_{\xi,k}^{(j)}\mid (k,j)\in\mathbf{n}_\cJ\}$. Hence, our claim above forces $\tld{F}$ to be a monomial with variables $\{F_\xi^m\mid 1\leq m\leq r_\xi\}$.  The proof is thus finished.
\end{proof}

\begin{prop}\label{prop: quotient of strata}
There exists a partition on $\cN_{\xi,\Lambda}\slash{{\sim}_{\un{T}\text{-\textnormal{sh.cnj}}}}$ whose pull back to $\cN_{\xi,\Lambda}$ is the restriction of $\cP$ to $\cN_{\xi,\Lambda}$. In particular, the geometric quotient $\cC\slash{{\sim}_{\un{T}\text{-\textnormal{sh.cnj}}}}$ exists for each $\cC\in\cP$ satisfying $\cC\subseteq \cN_{\xi,\Lambda}$.
\end{prop}
\begin{proof}
It suffices to show that, for each $S\subseteq\{1,\cdots,n\}$ and $j_0\in\cJ$, there exists $F_0\in \cO(\cN_{\xi,\Lambda})^{\un{T}\text{-\textnormal{sh.cnj}}}$ and $F_1\in\cO(\cN_{\xi,\Lambda})^\times$ such that $f_{S,j_0}|_{\cN_{\xi,\Lambda}}=F_0F_1$. There clearly exists a monomial $F_3$ with variables $\{D_{\xi,k}^{(j_0)}\mid k\in\mathbf{n}\}$ and a polynomial $F_4$ with variables $\{u_\xi^{(\al,j_0)}\mid (\al,j_0)\in\mathrm{Supp}_{\xi,\cJ}\}$ such that $f_{S,j_0}|_{\cN_{\xi,\Lambda}}=F_3F_4$. We define $F_5$ by replacing $u_\xi^{(\al,j_0)}$ with $F_\xi^{(\al,j_0),\cB}$ for each $u_\xi^{(\al,j_0)}$ that appears inside $F_4$. A key observation (from the definition of $\cB$ and the fact that $f_{S,j_0}|_{\cN_{\xi,\Lambda}}$ is a $\un{T}$-eigenvector for both the left and right $\un{T}$-multiplication action on $\cN_{\xi,\Lambda}$) is the existence of a monomial $F_6$ with variables $\{u_\xi^{(\al,j)}\mid (\al,j)\in\cB\}$ and $\{D_{\xi,k}^{(j_0)}\mid k\in\mathbf{n}\}$ such that $F_5=F_4F_6$. We finish the proof by taking $F_0\defeq F_5$ and $F_1\defeq F_3F_6^{-1}$.
\end{proof}

\subsection{Main results on invariant functions: statement}\label{sub:main:state}
In this section, we introduce a convenient sufficient condition (see Statement~\ref{state: goal}) which implies Statement~\ref{state: separate points prime}.

We fix a choice of $w_\cJ\in\un{W}$ and $\xi\in\Xi_{w_\cJ}$ and let $\Lambda\subseteq\mathrm{Supp}_{\xi,\cJ}$ be a subset with $\Lambda^\square$ its image in $\mathrm{Supp}^\square_\xi$. We write $\cO(X)$ (resp.~$\cO(X)^\times$) for the ring of global sections (resp.~the group of invertible global sections) on a $\F$-scheme $X$. Recall the subset $\Inv(\cC)\subseteq\Inv$ from the paragraph before Statement~\ref{state: separate points prime}.

\begin{defn}\label{def: similar and below a block}
We write
$\cO_{\xi,\Lambda}^{\rm{ss}}$ for the multiplicative subgroup of $\cO(\cN_{\xi,\Lambda})^\times$ generated by $-1$ and $D_{\xi,\ell}^{(j)}|_{\cN_{\xi,\Lambda}}$ for all $(\ell,j)\in\mathbf{n}_\cJ$. We say that two elements $F,F'\in \cO(\cN_{\xi,\Lambda})$ are \emph{similar}, written as $F\sim F'$, if there exists $F''\in\cO_{\xi,\Lambda}^{\rm{ss}}$ such that $F=F'F''$. We write $\cO_\cC^{\rm{ss}}$ for the restriction of $\cO_{\xi,\Lambda}^{\rm{ss}}$ to $\cC$ and define $F\sim F'$ similarly for two elements $F,F'\in \cO(\cC)$. We define $\cO_\cC'$ as the subring of $\cO(\cC)$ generated by $\cO_\cC^{\rm{ss}}$ and $g^{\pm 1}|_\cC$ for all $g\in\Inv(\cC)$. Then we define $\cO_\cC$ as the localisation of $\cO_\cC'$ with respect to $\cO_\cC'\cap \cO(\cC)^\times$.
\end{defn}

Even though different pairs $(\xi,\Lambda)$ may give rise to the same $\cN_{\xi,\Lambda}$, $\cN_{\xi,\Lambda}$ is uniquely determined by $\cC$ with $\cC\subseteq \cN_{\xi,\Lambda}$. We also note that $\cO_{\xi,\Lambda}^{\rm{ss}}$ depends only on $\cN_{\xi,\Lambda}$ and that $\cO_\cC^{\rm{ss}}$, $\cO_\cC'$, and $\cO_\cC$ depend only on~$\cC$.

Now we introduce our main result on invariant functions whose proof will occupy \S\,\ref{sec:comb:lifts}, \S\,\ref{sec:const:inv} and~\S\,\ref{sec:inv:cons}.
\begin{state}\label{state: goal}
For all $\Lambda$-lifts $\Omega^\pm$ (cf.~Definition~\ref{def: lift of loop}), we have
\begin{equation*}
F_\xi^{\Omega^\pm}|_\cC\in \cO_\cC
\end{equation*}
\end{state}

\begin{lemma}\label{lem: reduce to sec state}
Statement~\ref{state: goal} implies Statement~\ref{state: separate points prime}.
\end{lemma}
\begin{proof}
It follows from the existence of geometric quotient $\cC\slash{{\sim}_{\un{T}\text{-\textnormal{sh.cnj}}}}$ (see Proposition~\ref{prop: quotient of strata}) that there exists a canonical bijection
$$|[\cC\slash{{\sim}_{\un{T}\text{-\textnormal{sh.cnj}}}}]|(R)\xrightarrow{\sim}\cC\slash{{\sim}_{\un{T}\text{-\textnormal{sh.cnj}}}}(R)$$
for each $R$. Hence Statement~\ref{state: separate points prime} holds if and only if $\iota_\cC$ induces a monomorphism
\begin{equation}\label{equ: monomorphism of quotient}
\cC\slash{{\sim}_{\un{T}\text{-\textnormal{sh.cnj}}}}\rightarrow \bG_m^{\#\Inv(\cC)}.
\end{equation}
Assume that Statement~\ref{state: goal} holds in the rest of the proof, and we want to show that (\ref{equ: monomorphism of quotient}) is a monomorphism. We fix a choice of $\cB\subseteq\Lambda$ as in \S\,\ref{subsub: remove} and a choice of $\mathbf{n}_\cJ^\flat\subseteq \mathbf{n}_\cJ$ satisfying $\#(\mathbf{n}_\cJ^\flat\cap I_\cJ^m)=\# I_\cJ^m-1$ for each $1\leq m\leq r_\xi$. The $\un{T}$-action on $\cN_{\xi,\Lambda}$ induces a $\un{T}$-action on $\bG_m^{fn-r_\xi}\times\bG_m^{\#\cB}$ by projection to the entries indexed by $\mathbf{n}_\cJ^\flat$ and $\cB$, and the key observation is that $\un{T}$ acts transitively on $\bG_m^{fn-r_\xi}\times\bG_m^{\#\cB}$ using (\ref{equ: T conj formula ss}), (\ref{equ: T conj formula u}) and Lemma~\ref{lem: quotient of torus by torus}.

Let $x_1,x_2\in\cC(R)\subseteq\cN_{\xi,\Lambda}(R)$ be two points such that $g(x_1)=g(x_2)$ for all $g\in\Inv(\cC)$. As the $\un{T}$-action on $\bG_m^{fn-r_\xi}\times\bG_m^{\#\cB}$ above is transitive, upon replacing $x_2$ with $x_2\cdot t$ for some $t\in\un{T}(R)$, we may assume further that
\begin{equation}\label{equ: two point with same image}
\left\{
  \begin{array}{cl}
    g(x_1)=g(x_2) & \hbox{for each $g\in\Inv(\cC)$;} \\
    D_{\xi,\ell}^{(j)}(x_1)=D_{\xi,\ell}^{(j)}(x_2) & \hbox{for each $(\ell,j)\in\mathbf{n}_\cJ^\flat$;} \\
    u_\xi^{(\al,j')}(x_1)=u_\xi^{(\al,j')}(x_2) & \hbox{for each $(\al,j')\in\cB$.}
  \end{array}
\right.
\end{equation}
For each $1\leq m\leq r_\xi$, the element $f_{w_\cJ,I_\cJ^m}\in\Inv(\cC)$ satisfies
$$f_{w_\cJ,I_\cJ^m}|_\cC=\prod_{(\ell,j)\in I_\cJ^m}D_{\xi,\ell}^{(j)}|_\cC,$$
which together with (\ref{equ: two point with same image}) and the definition of $\mathbf{n}_\cJ^\flat$ implies that $D_{\xi,\ell}^{(j)}(x_1)=D_{\xi,\ell}^{(j)}(x_2)$ for each $(\ell,j)\in\mathbf{n}_\cJ$, and thus $g(x_1)=g(x_2)$ for each $g\in\cO_{\xi,\Lambda}^{\rm{ss}}$. On the other hand, for each $(\al,j)\in\Lambda\setminus\cB$, it follows from (\ref{equ: two point with same image}), Statement~\ref{state: goal} and $g(x_1)=g(x_2)$ for each $g\in\cO_{\xi,\Lambda}^{\rm{ss}}$ that
$$F_\xi^{(\al,j),\cB}(x_1)=F_\xi^{(\al,j),\cB}(x_2)$$
and thus $u_\xi^{(\al,j)}(x_1)=u_\xi^{(\al,j)}(x_2)$ (using the definition of $F_\xi^{(\al,j),\cB}$). Hence we deduce that $x_1=x_2$ from (\ref{equ: std section Lambda}). The proof is thus finished.
\end{proof}

\newpage
\section{Combinatorics of $\Lambda$-lifts}\label{sec:comb:lifts}
In order to prove Statement~\ref{state: goal}, we need to systematically study the set of all $\Lambda$-lifts. A natural question arises: for which choice of $\Lambda$-lift $\Omega^\pm$ and $\cC\in\cP_\cJ$ satisfying $\cC\subseteq\cN_{\xi,\Lambda}$, there exists an invariant functions $g\in\Inv(\cC)$ such that $g|_\cC\sim F_\xi^{\Omega^\pm}|_\cC$? This is a very delicate question in general. To solve it, we restrict our attention to the set of \emph{constructible $\Lambda$-lifts} (see Definition~\ref{def: constructible lifts}), a special class of $\Lambda$-lifts which are closely related to invariant functions. The main result of this section (see Theorem~\ref{thm: reduce to constructible}) says that all $\Lambda$-lifts can be generated from constructible ones, and in particular it suffices to prove Statement~\ref{state: goal} for constructible $\Lambda$-lifts. The relation between constructible $\Lambda$-lifts and invariant functions will be further explored in \S\,\ref{sec:const:inv} and \S\,\ref{sec:inv:cons}.

Throughout this section, we fixed a choice of $\Lambda\subseteq\mathrm{Supp}_{\xi,\cJ}$ and write $\widehat{\Lambda}\subseteq\mathrm{Supp}_{\xi,\cJ}$ for the closure of $\Lambda$ in $\mathrm{Supp}_{\xi,\cJ}$, i.e.~the subset consisting of all elements $(\al,j)$ satisfying the condition that there exists a subset $\Omega\subseteq\Lambda\cap\mathrm{Supp}_{\xi,j}$ (depending on $(\al,j)$) such that $\sum_{(\beta,j)\in\Omega}\beta=\al$.

\subsection{Preliminary on $\Lambda$-lifts}\label{sub:prelim:lifts}
In this section, we introduce the notion \emph{a balanced pair} as a direct generalization of $\Lambda$-lifts, and then prove some elementary combinatorial results on it. Balanced pairs are technically more convenient to manipulate than $\Lambda$-lifts as standard set theoretical operations preserve balanced pairs but not $\Lambda$-lifts. In fact, a balanced pair naturally arises when we try to write down an element of $\cO(\cN_{\xi,\Lambda})^\times\cap \cO(\cN_{\xi,\Lambda})^{\un{T}\text{-\textnormal{sh.cnj}}}$ explicitly (see Remark~\ref{rmk: balanced pair and invariance}).

\begin{defn}\label{def: balanced cond}
We write $\N \Lambda^\square$ for the submonoid of the root lattice $\Z\Phi^+_{\GL_{r_\xi}}$ generated by the elements of $\Lambda^\square$, and write $\N^\Lambda$ for the free abelian monoid with basis $\Lambda$. We view an element $\Omega\in\N^\Lambda$ as a \emph{$\Lambda$-multi-set}, namely as a collection of elements $(\al,j)$ of $\Lambda$ each equipped with a multiplicity $n_{(\al,j)}\in \N $.
(Equivalently, $\Omega$ is seen as a subset of $\Lambda\times\N $ such that $\Omega$ maps injectively into $\Lambda$ under the projection $\Lambda\times\N\twoheadrightarrow\Lambda$.)
We write $\Omega^{\square}$ for the multi-set induced from $\Omega$ under the map $\Lambda\twoheadrightarrow \Lambda^\square$, with the multiplicity $n_\gamma$ of each element $\gamma\in\Omega^{\square}$ defined as the sum of all $n_{(\al,j)}$ over all $(\al,j)\in\Omega$ having image $\gamma$ under $\Lambda\twoheadrightarrow \Lambda^\square$.

We say that a pair of $\Lambda$-multi-sets $\Omega^\pm$ is \emph{balanced} if
$$\sum_{\gamma\in\Omega^{+,\square}}n_\gamma^+\gamma=\sum_{\gamma\in\Omega^{-,\square}}n_\gamma^-\gamma\in \N \Lambda^\square$$
where $n_\gamma^+$ (resp.~$n_\gamma^-$) is the multiplicity of each element $\gamma$ of $\Omega^{+,\square}$ (resp.~$\Omega^{-,\square}$). Note that a $\Lambda$-lift is, in particular, a balanced pair of $\Lambda$-multi-sets. We will frequently use the short term \emph{a balanced pair} for a balanced pair of $\Lambda$-multi-sets, whenever the choice of $\Lambda$ is clear. If $\Omega^\pm$ is a balanced pair, we define its \emph{norm} to be
$$|\Omega^\pm|\defeq \sum_{\gamma\in \Omega^{+,\square}}n_\gamma^+\gamma=\sum_{\gamma\in \Omega^{-,\square}}n_\gamma^-\gamma.$$

Let $\Omega$ and $\Omega'$ be two $\Lambda$-multi-sets which contains $(\al,j)$ with multiplicity $n_{(\al,j)}$ and $n_{(\al,j)}'$ respectively, for each $(\al,j)\in\Lambda$. We define their \emph{disjoint union} $\Omega\sqcup\Omega'$ (resp.~\emph{intersection} $\Omega\cap\Omega'$, resp.~\emph{difference} $\Omega\setminus\Omega'$) as the $\Lambda$-multi-set with the multiplicity of $(\al,j)$ given by $n_{(\al,j)}+n_{(\al,j)}'$ (resp.~by $\min\{n_{(\al,j)},\,n_{(\al,j)}'\}$, resp.~by $\max\{n_{(\al,j)}-n_{(\al,j)}',0\}$) for each $(\al,j)\in\Lambda$. Given a balanced pair $\Omega^\pm$, the balanced pair $\Omega_0^\pm$ satisfying $\Omega_0^+=\Omega^-$ and $\Omega_0^-=\Omega^+$ is called the \emph{inverse} of $\Omega^\pm$. For each $\delta\in\N \Lambda^\square$, we write $\cO_{\xi,\Lambda}^{<\delta}$ for the multiplicative subgroup of $\cO(\cN_{\xi,\Lambda})^\times$ generated by $\cO_{\xi,\Lambda}^{\rm{ss}}$ and $F_\xi^{\Omega^\pm}$ for all $\Lambda$-lifts $\Omega^\pm$ satisfying $|\Omega^\pm|<\delta$. Here we use the partial order on $\N \Lambda^\square$ inherited from $\mathrm{Supp}^\square_\xi\subseteq\Phi^+_{\GL_{r_\xi}}$.
\end{defn}

\begin{lemma}\label{lem: union of lifts}
For each balanced pair $\Omega^\pm$, there exists a sequence of $\Lambda$-lifts $\Omega_1^\pm,\dots,\Omega_s^\pm$ for some $s\geq 1$ such that we have the following disjoint unions of $\Lambda$-multi-sets
\begin{equation*}
    \Omega^+=(\Omega^+\cap\Omega^-)\sqcup\bigsqcup_{s'=1}^s\Omega_{s'}^+\qquad\mbox{ and }\qquad
\Omega^-=(\Omega^+\cap\Omega^-)\sqcup\bigsqcup_{s'=1}^s\Omega_{s'}^-.
\end{equation*}
Moreover, we have $|\Omega_{s'}^\pm|<|\Omega^\pm|$ for each $1\leq s'\leq s$ if either $s \geq 2$ or $\Omega^+\cap\Omega^-\neq\emptyset$.
\end{lemma}

\begin{proof}
We argue by induction on $|\Omega^\pm|$ with respect to the partial order on $\N \Lambda^\square$ inherited from $\mathrm{Supp}_\xi^\square\subseteq\Phi^+_{\GL_{r_\xi}}$. If $\Omega^+\cap\Omega^-\neq \emptyset$, then we can simply replace $\Omega^\pm$ with the balanced pair $\Omega^+\setminus\Omega^-,\Omega^-\setminus\Omega^+$ and finish the proof by our inductive assumption.

Therefore we may assume without loss of generality that $\Omega^+\cap\Omega^-=\emptyset$. We pick up a minimal (under inclusion of $\Lambda$-multi-sets) possible non-empty $\Lambda$-multi-set $\Omega_0^+\subseteq \Omega^+$ (resp.~$\Omega_0^-\subseteq\Omega^-$) such that the pair of sets $\Omega_0^\pm$ is balanced. We observe from Definition~\ref{def:loop} that the minimality condition on $\Omega_0^\pm$ exactly means that there exists a directed loop $\Gamma$ inside $\mathfrak{G}_{\xi,\Lambda}$ such that $\Omega_0^\pm$ is a $\Lambda$-lift of $\Gamma$. If $\Omega_0^+=\Omega^+$, then we must also have $\Omega_0^-=\Omega^-$ and in particular the balanced pair $\Omega^\pm$ is a $\Lambda$-lift; otherwise, we repeat the same argument for the balanced pair $\Omega^+\setminus\Omega_0^+,~\Omega^-\setminus\Omega_0^-$ and finish the proof by our inductive assumption as the norm of $\Omega^+\setminus\Omega_0^+,~\Omega^-\setminus\Omega_0^-$ is strictly smaller than $|\Omega^\pm|$.
\end{proof}

By Lemma~\ref{lem: union of lifts} a balanced pair $\Omega^\pm$ is a $\Lambda$-lift if and only if $\Omega^+\cap\Omega^-=\emptyset$ and the pair $\Omega^+,\Omega^-$ is minimal (among all balanced pairs) under inclusion of non-empty $\Lambda$-multi-sets. Moreover, for each balanced pair $\Omega^\pm$ which is not necessarily a $\Lambda$-lift, we can define
\begin{equation}
\label{def:gen:FOmega}
F_\xi^{\Omega^\pm}\defeq \prod_{s'=1}^s F_\xi^{\Omega_{s'}^\pm}\in \cO(\cN_{\xi,\Lambda})^\times.
\end{equation}
(Recall that $F_\xi^{\Omega_{s'}^\pm}$ for $\Lambda$-lifts $\Omega_{s'}^\pm$ are defined in (\ref{equ: loop function}).)
The function $F_\xi^{\Omega^\pm}$ clearly depends on the choice of $\Omega_1^\pm,\dots,\Omega_s^\pm$ in general, but Lemma~\ref{lem: partition implies similar} below shows that $F_\xi^{\Omega^\pm}$ is independent of the choice of $\Omega_1^\pm,\dots,\Omega_s^\pm$ up to the equivalence relation $\sim$ on $\cO(\cN_{\xi,\Lambda})$ (cf.~Definition~\ref{def: similar and below a block}).

\begin{lemma}\label{lem: partition implies similar}
Let $s_1,s_2\geq 1$ be two integers, and let $\Omega_{1,1}^\pm,\dots,\Omega_{1,s_1}^\pm$ and $\Omega_{2,1}^\pm,\dots,\Omega_{2,s_2}^\pm$ be two sequences of balanced pairs that satisfy
\begin{equation}
\label{eq:cond:disj}
\bigsqcup_{s'=1}^{s_1}\Omega_{1,s'}^+=\bigsqcup_{s'=1}^{s_2}\Omega_{2,s'}^+\,\,\mbox{and}\,\, \bigsqcup_{s'=1}^{s_1}\Omega_{1,s'}^-=\bigsqcup_{s'=1}^{s_2}\Omega_{2,s'}^-.
\end{equation}
Then we have
$$\prod_{s'=1}^{s_1} F_\xi^{\Omega_{1,s'}^\pm}\sim\prod_{s'=1}^{s_2} F_\xi^{\Omega_{2,s'}^\pm}.$$
\end{lemma}
\begin{proof}
For each $a=1,2$ and each $1\leq s'\leq s_a$, it follows from the definition of $F_\xi^{\Omega_{a,s'}^\pm}$ given in (\ref{def:gen:FOmega}) that
$$F_\xi^{\Omega_{a,s'}^\pm}\sim F_\xi^{\Omega_{a,s'}^\pm,\sharp}\defeq \frac{\prod_{(\al,j)\in \Omega_{a,s'}^+}u_\xi^{(\al,j)}}{\prod_{(\al,j)\in \Omega_{a,s'}^+}u_\xi^{(\al,j)}}.$$
Then condition (\ref{eq:cond:disj}) obviously implies that
$$\prod_{s'=1}^{s_1} F_\xi^{\Omega_{1,s'}^\pm,\sharp}=\prod_{s'=1}^{s_2} F_\xi^{\Omega_{2,s'}^\pm,\sharp}$$
(see (\ref{equ: sharp factor})) which finishes the proof.
\end{proof}

\begin{rmk}\label{rmk: balanced pair and invariance}
It is clear that we have $F_\xi^{\Omega^\pm}\in\cO(\cN_{\xi,\Lambda})^\times\cap \cO(\cN_{\xi,\Lambda})^{\un{T}\text{-\textnormal{sh.cnj}}}$ for each balanced pair $\Omega^\pm$. Conversely, it is easy to deduce from (\ref{equ: inv subring})) that each element of $\cO(\cN_{\xi,\Lambda})^\times\cap \cO(\cN_{\xi,\Lambda})^{\un{T}\text{-\textnormal{sh.cnj}}}$ has the form $F_\xi^{\Omega^\pm}$ for some balanced pair $\Omega^\pm$, upon multiplying a monomial with variables $\{F_\xi^m\mid 1\leq m\leq r_\xi\}$.
\end{rmk}

\begin{defn}\label{def: separated condition}
Let $\Omega\subseteq\Lambda$ be a subset. We define two subsets $\mathbf{I}_\Omega$, $\mathbf{I}_\Omega^\prime$ of $\mathbf{n}$ by $\mathbf{I}_\Omega\defeq \{(i_\beta,j)\mid (\beta,j)\in\Omega\}$ and $\mathbf{I}_\Omega^\prime\defeq \{(i_\beta^\prime,j)\mid (\beta,j)\in\Omega\}$ (where we write as usual $\beta=(i_\beta,i_\beta^\prime)$ for an element $\beta\in\Phi^+$). We define $\Delta_{\Omega}\defeq (\mathbf{I}_\Omega\setminus \mathbf{I}_\Omega^\prime)\cup (\mathbf{I}_\Omega^\prime\setminus \mathbf{I}_\Omega)\subseteq\mathbf{I}_\Omega\cup\mathbf{I}_\Omega^\prime$. We say that an element $(i,j)\in\mathbf{n}_\cJ$ is an \emph{interior point} of $\Omega$ if $(i,j)\in \mathbf{I}_\Omega\cap\mathbf{I}_\Omega^\prime$. We say that $\Omega$ is \emph{$\Lambda$-separated} if for each $(i,j),(i',j)\in\mathbf{I}_\Omega\cup\mathbf{I}_\Omega^\prime$ satisfying $((i,i'),j)\in\widehat{\Lambda}$, there exists $\Omega'\subseteq\Omega\cap\mathrm{Supp}_{\xi,j}$ such that $\sum_{(\beta^{\prime\prime},j)\in\Omega'}\beta^{\prime\prime}=(i,i')$.

Now we consider a $\Lambda$-lift $\Omega^\pm$. We say that a subset $\Omega\subseteq\Omega^+\sqcup\Omega^-$ is a \emph{$\Lambda^\square$-interval} of $\Omega^\pm$ if it is a maximal possible subset with image $\Omega^\square$ in $\Lambda^\square$ such that $\sum_{\gamma\in\Omega^\square}\gamma\in\N \Lambda^\square$ is actually in $\Phi^+_{\GL_{r_\xi}}$. Hence $\Omega^+\sqcup\Omega^-$ is clearly a disjoint union of all of its $\Lambda^\square$-intervals and each $\Lambda^\square$-interval is either inside $\Omega^+$ or inside $\Omega^-$. Given a $\Lambda^\square$-interval $\Omega$ of $\Omega^\pm$, we say that an element $(i,j)\in\mathbf{I}_{\Omega^+\sqcup\Omega^-}\cup \mathbf{I}_{\Omega^+\sqcup\Omega^-}^\prime$ \emph{lies in the $\Lambda^\square$-interval $\Omega$} if $(i,j)\in \mathbf{I}_\Omega\cup \mathbf{I}_\Omega^\prime$.

For each subset $\Omega\subsetneq\Omega^+\sqcup\Omega^-\subseteq\Lambda$ we define $\widehat{\Omega}\subseteq\widehat{\Lambda}$ as the unique subset which has no interior points and each of whose element is a sum of elements in $\Omega$.
More precisely, there exists a unique partition $\Omega=\bigsqcup_{(\al,j)\in\widehat{\Omega}}\Omega_{(\al,j)}$ such that $\sum_{(\beta,j)\in\Omega_{(\al,j)}}\beta=\al$ for each $(\al,j)\in\widehat{\Omega}$.
In particular, we can associate a subset $\widehat{\Omega}^+\subseteq\widehat{\Lambda}$ (resp. $\widehat{\Omega}^-\subseteq\widehat{\Lambda}$) with $\Omega^+$ (resp. $\Omega^-$), and observe that exactly one of the following holds:
\begin{itemize}
\item $\widehat{\Omega}^+=\widehat{\Omega}^-=\{(\al,j)\}$ for some $(\al,j)\in\widehat{\Lambda}$;
\item $\widehat{\Omega}^+\cap\widehat{\Omega}^-=\emptyset$ and $\widehat{\Omega}^\pm$ is a $\widehat{\Lambda}$-lift of some directed loop inside $\mathfrak{G}_{\xi,\widehat{\Lambda}}$ satisfying $|\widehat{\Omega}^\pm|=|\Omega^\pm|$.
\end{itemize}
\end{defn}

\begin{lemma}\label{lem: reduction to separated lift}
For each $\Lambda$-lift $\Omega^\pm$, there exists a sequence of $\Lambda$-lifts $\Omega_1^\pm,\dots,\Omega_s^\pm$ for some $s\geq 1$ such that
\begin{itemize}
\item $\Omega_{s'}^+\sqcup\Omega_{s'}^-$ is $\Lambda$-separated and $|\Omega_{s'}^\pm|\leq |\Omega^\pm|$ (cf.~Definition~\ref{def: balanced cond}) for each $1\leq s'\leq s$;
\item $F_\xi^{\Omega^\pm}\sim\prod_{s'=1}^s F_\xi^{\Omega_{s'}^\pm}$.
\end{itemize}
Assume moreover that there exist $(i,j),\,(i',j)\in\mathbf{I}_{\Omega^+\sqcup\Omega^-}\cup \mathbf{I}_{\Omega^+\sqcup\Omega^-}^\prime$ such that $((i,i'),j)\in\widehat{\Lambda}$ and $(i,j),\,(i',j)$ do not lie in the same $\Lambda^\square$-interval (hence $\Omega^+\sqcup\Omega^-$ is not $\Lambda$-separated). Then we have $F_\xi^{\Omega^\pm}\in\cO_{\xi,\Lambda}^{<|\Omega^\pm|}$.
\end{lemma}
\begin{proof}
In the following, we assume inductively that the result holds for any $\Lambda$-lift $\Omega_0^\pm$ satisfying either $|\Omega_0^\pm|<|\Omega^\pm|$ or $|\Omega_0^\pm|=|\Omega^\pm|$ and $\#\Delta_{\Omega_0^+\sqcup\Omega_0^-}<\#\Delta_{\Omega^+\sqcup\Omega^-}$. If $\Omega^\pm$ is a $\Lambda$-lift with $\Omega^+\sqcup\Omega^-$ being $\Lambda$-separated, then we simply set $s\defeq 1$ and $\Omega_1^\pm\defeq \Omega^\pm$. Hence we assume from now on that $\Omega^+\sqcup\Omega^-$ is not $\Lambda$-separated, and thus there exists a pair of elements $(i,j),~(i',j)\in\mathbf{I}_{\Omega^+\sqcup\Omega^-}\cup \mathbf{I}_{\Omega^+\sqcup\Omega^-}^\prime$ as well as a non-empty subset $\Omega'\subseteq\Lambda\cap\mathrm{Supp}_{\xi,j}$ such that:
\begin{enumerate}[label=(\roman*)]
\item
\label{it:non-La-sep:1}
$(i,i')=\sum_{(\beta^{\prime\prime},j)\in\Omega'}\beta^{\prime\prime}$ and in particular $((i,i'),j)\in\widehat{\Lambda}$; and
\item
\label{it:non-La-sep:2}
$((i,i'),j)$ is not a sum of some elements in $\Omega^+\sqcup\Omega^-$.
\end{enumerate}
We may assume without loss of generality that the non-empty set $\Omega'$ is minimal (under inclusion of subsets of $\Lambda$) among all possible choices of $(i,j),(i',j)$. If there exists $(\beta,j)\in\Omega'\cap(\Omega^+\sqcup\Omega^-)\neq \emptyset$, then at least one of the following holds
\begin{itemize}
\item $i_\beta\neq i$, $((i,i_\beta),j)\in\widehat{\Lambda}$ and $((i,i_\beta),j)$ is not a sum of some elements in $\Omega^+\sqcup\Omega^-$;
\item $i_\beta^\prime\neq i'$, $((i_\beta^\prime,i'),j)\in\widehat{\Lambda}$ and $((i_\beta^\prime,i'),j)$ is not a sum of some elements in $\Omega^+\sqcup\Omega^-$,
\end{itemize}
which clearly contradicts the minimality of $\Omega'$. Hence we deduce that $\Omega'\cap(\Omega^+\sqcup\Omega^-)= \emptyset$.

Then a key observation (based on the fact that $\Omega'\cap(\Omega^+\sqcup\Omega^-)=\emptyset$) is that there exist two balanced pairs $\Omega_\sharp^\pm$ and $\Omega_\flat^\pm$ such that the following holds:
\begin{itemize}
\item $\Omega_\sharp^+,\Omega_\sharp^-,\Omega_\flat^+,\Omega_\flat^-$ all have multiplicity one, and $\Omega_\sharp^+\cap \Omega_\sharp^-=\emptyset=\Omega_\flat^+\cap \Omega_\flat^-$;
\item $\Omega'\subseteq \Omega_\sharp^+, \Omega_\flat^-$;
\item $\Omega_\sharp^+\sqcup\Omega_\flat^+=\Omega^+\sqcup\Omega'$ and $\Omega_\sharp^-\sqcup\Omega_\flat^-=\Omega^-\sqcup\Omega'$.
\end{itemize}
Note that the three conditions above imply $(\Omega_\sharp^+\sqcup\Omega_\sharp^-)\cap(\Omega_\flat^+\sqcup\Omega_\flat^-)=\Omega'$.

We write $\Omega_\sharp^{+,\square}$ (resp.~$\Omega_\sharp^{-,\square}$, resp.~$\Omega_\flat^{+,\square}$, resp.~$\Omega_\flat^{-,\square}$) for the multi-set induced from $\Omega_\sharp^+$ (resp.~$\Omega_\sharp^-$, resp.~$\Omega_\flat^+$, resp.~$\Omega_\flat^-$) under $\Lambda\twoheadrightarrow\Lambda^\square$ (cf.~Definition~\ref{def: balanced cond}).
We also write $\Omega^{\prime,\square}$ for the multi-set induced from $\Omega'$ under $\Lambda\twoheadrightarrow\Lambda^\square$. Note that $\Omega^{\prime,\square}$, $\Omega_\sharp^{-,\square}$, $\Omega_\flat^{+,\square}$, $\Omega_\sharp^{+,\square}\setminus \Omega^{\prime,\square}$ and $\Omega_\flat^{-,\square}\setminus \Omega^{\prime,\square}$ have multiplicity one, but $\Omega_\sharp^{+,\square}$ and $\Omega_\flat^{-,\square}$ might have multiplicity greater than one.
Then we deduce from the corresponding results on $\Omega_\sharp^\pm,\Omega_\flat^\pm$ and $\Omega'$ that
\begin{itemize}
\item $\Omega^{\prime,\square}\subseteq \Omega_\sharp^{+,\square}, \Omega_\flat^{-,\square}$;
\item $(\Omega_\sharp^{+,\square}\setminus \Omega^{\prime,\square})\sqcup \Omega_\flat^{+,\square}=E(\Gamma)^+$ and $\Omega_\sharp^{-,\square}\sqcup (\Omega_\flat^{-,\square}\setminus \Omega^{\prime,\square})=E(\Gamma)^-$
\end{itemize}

If $\Omega'\subsetneq \Omega_\flat^-$, we have $\Omega^{\prime,\square}\subsetneq \Omega_\flat^{-,\square}$ (as $\Lambda$-multi-sets) and
$$\sum_{\gamma\in\Omega^{\prime,\square}}\gamma<\sum_{\gamma\in \Omega_\flat^{-,\square}}\gamma=\sum_{\gamma\in \Omega_\flat^{+,\square}}\gamma<\sum_{\gamma\in \Omega_\flat^{+,\square}}\gamma+ \sum_{\gamma\in \Omega_\flat^{-,\square}\setminus \Omega^{\prime,\square}}\gamma,$$
and so
\begin{align*}
2|\Omega_\sharp^\pm|&=\sum_{\gamma\in\Omega^{\prime,\square}}\gamma+
\sum_{\gamma\in \Omega_\sharp^{-,\square}}\gamma+
\sum_{\gamma\in \Omega_\sharp^{+,\square}\setminus \Omega^{\prime,\square}}\gamma\\
&<\sum_{\gamma\in \Omega_\flat^{+,\square}}\gamma+ \sum_{\gamma\in \Omega_\flat^{-,\square}\setminus \Omega^{\prime,\square}}\gamma
+\sum_{\gamma\in \Omega_\sharp^{-,\square}}\gamma+ \sum_{\gamma\in \Omega_\sharp^{+,\square}\setminus \Omega^{\prime,\square}}\gamma\\
&=\sum_{\gamma\in E(\Gamma)^+}\gamma+ \sum_{\gamma\in E(\Gamma)^-}\gamma=2|\Omega^\pm|.
\end{align*}

If $\Omega'=\Omega_\flat^-$, then:
\begin{itemize}
\item we have $|\Omega_\sharp^\pm|=|\Omega^\pm|$ as $(\Omega_{\flat}^-\setminus\Omega')\sqcup \Omega_{\sharp}^-=\Omega^-$ by definition of $\Omega_{\sharp}^{\pm}$, $\Omega_{\flat}^{\pm}$; and
\item
we have $\#\Delta_{\Omega_\sharp^+\sqcup\Omega_\sharp^-}<\#\Delta_{\Omega^+\sqcup\Omega^-}$ as:
\begin{itemize}
\item
we have $\Delta_{\Omega'}=\{(i,j),\,(i',j)\}$ and a natural inclusion $\Delta_{\Omega_\sharp^+\sqcup\Omega_\sharp^-}\subseteq \Delta_{\Omega^+\sqcup\Omega^-}$;
\item
the latter inclusion must be strict, as the equality $\Delta_{\Omega_\sharp^+\sqcup\Omega_\sharp^-}= \Delta_{\Omega^+\sqcup\Omega^-}$ would imply $\Delta_{\Omega^+_\flat}=\{(i,j),\ (i',j)\}$ (namely, $\Omega^+_\flat\subseteq\mathrm{Supp}_{\xi,j}$ and $(i,i')=\sum_{(\beta,j)\in \Omega^+_\flat}\beta$), which contradicts the choice of $((i,i'),j)$ as $\Omega^+_\flat\subseteq\Omega^+$.
\end{itemize}
\end{itemize}
Similarly, if $\Omega'\subsetneq \Omega_\sharp^+$, we have $|\Omega_\flat^\pm|<|\Omega^\pm|$; if $\Omega'=\Omega_\sharp^+$, we have $|\Omega_\flat^\pm|=|\Omega^\pm|$ and $\#\Delta_{\Omega_\flat^+\sqcup\Omega_\flat^-}<\#\Delta_{\Omega^+\sqcup\Omega^-}$.

Given these inequalities, we can apply our induction hypothesis on each $\Lambda$-lift in the decomposition of $\Omega_\sharp^\pm$ and $\Omega_\flat^\pm$ (obtained by  Lemma~\ref{lem: union of lifts}): we hence get
two integers $1\leq s_\sharp<s$ and a sequence of $\Lambda$-lifts $\Omega_1^\pm,\dots,\Omega_s^\pm$ such that
\begin{itemize}
\item $F_\xi^{\Omega_\sharp^\pm}\sim \prod_{s'=1}^{s_\sharp} F_\xi^{\Omega_{s'}^\pm}$ and $F_\xi^{\Omega_\flat^\pm}\sim \prod_{s'=s_{\sharp}+1}^s F_\xi^{\Omega_{s'}^\pm}$;
\item $|\Omega_{s'}^\pm|\leq\max\{|\Omega_\sharp^\pm|,|\Omega_\flat^\pm|\}$ and $\Omega_{s'}^+\sqcup\Omega_{s'}^-$ is $\Lambda$-separated for each $1\leq s'\leq s$.
\end{itemize}
This together with Lemma~\ref{lem: partition implies similar} clearly implies that
$$F_\xi^{\Omega^\pm}\sim F_\xi^{\Omega_\sharp^\pm}F_\xi^{\Omega_\flat^\pm}\sim \prod_{s'=1}^s F_\xi^{\Omega_{s'}^\pm}.$$
So the proof of the first statement of the lemma is finished by an induction on $|\Omega^\pm|$ and $\#\Delta_{\Omega^+\sqcup\Omega^-}$ as above.

As for the second statement of the lemma, we now observe that if either $\Omega'=\Omega_\flat^-$ or $\Omega'=\Omega_\sharp^+$, then $(i,j),~(i',j)$ necessarily lie in the same $\Lambda^\square$-interval of $\Omega^\pm$ (cf.~Definition~\ref{def: separated condition}).
Hence if there exists a choice of $(i,j),~(i',j)\in\mathbf{I}_{\Omega^+\sqcup\Omega^-}\cup \mathbf{I}_{\Omega^+\sqcup\Omega^-}^\prime$ and of $\emptyset\neq\Omega'\subseteq\Lambda$ satisfying items \ref{it:non-La-sep:1}--\ref{it:non-La-sep:2} above, and with moreover $(i,j),~(i',j)$ not lying in the same $\Lambda^\square$-interval of $\Omega^\pm$, then we can always assume further that $\Omega'$ is minimal without losing the condition that $(i,j),~(i',j)$ do not lie in the same $\Lambda^\square$-interval.
Consequently, we have $\Omega'\subsetneq \Omega_\sharp^+,\Omega_\flat^-$ which implies that $|\Omega_\sharp^\pm|,|\Omega_\flat^\pm|<|\Omega^\pm|$ and
$$F_\xi^{\Omega^\pm}\sim F_\xi^{\Omega_\sharp^\pm}F_\xi^{\Omega_\flat^\pm}\in\cO_{\xi,\Lambda}^{<|\Omega^\pm|}.$$
The proof is thus finished.
\end{proof}

\begin{defn}\label{def: modification of subsets}
Let $\Omega_1,\Omega_2$ be two $\Lambda$-multi-sets and $\gamma\in\widehat{\Lambda}^\square$ be a block. We say that $\Omega_1$ is a \emph{$\Lambda$-modification with level $\gamma$} of $\Omega_2$ if there exists an embedding $j\in\cJ$ together with subsets with multiplicity one $\Omega_a'\subseteq \Omega_a\cap\mathrm{Supp}_{\xi,j}$ for all $a=1,2$ such that the following holds:
\begin{itemize}
\item for each $a=1,2$, $\sum_{(\beta,j)\in\Omega_a'}\beta=\al_a$  for an element $(\al_a,j)\in\widehat{\Lambda}\cap\mathrm{Supp}_{\xi,\cJ}^\gamma$;
\item $\Omega_1\setminus\Omega_1'=\Omega_2\setminus\Omega_2'$.
\end{itemize}
For each $\delta\in\N \Lambda^\square$, we say that $\Omega_1$ and $\Omega_2$ are \emph{$\Lambda$-equivalent with level $<\delta$} if there exists a finite sequence of $\Lambda$-multi-sets $\Omega_1=\Omega_{1,0},\Omega_{1,1},\dots,\Omega_{1,s}=\Omega_2$ such that $\Omega_{1,s'}$ is a $\Lambda$-modification of $\Omega_{1,s'-1}$ with level $\gamma_{s'}$ for some $\gamma_{s'}\in\widehat{\Lambda}^\square$ satisfying $\gamma_{s'}<\delta$, for each $1\leq s'\leq s$. Here we use the following convention, for each $\Omega\subseteq\Lambda$ and each $\delta\in\N \Lambda^\square$, $\Omega$ is $\Lambda$-equivalent to itself with level $<\delta$.
\end{defn}

\begin{lemma}\label{lem: two equivalent pairs}
Let $\Omega_1^\pm,\Omega_2^\pm$ be two balanced pairs of $\Lambda$-multi-sets, and assume that $\Omega_1^+$ (resp. $\Omega_1^-$) is $\Lambda$-equivalent to $\Omega_2^+$ (resp. $\Omega_2^-$) with level $<\delta$ for some $\delta\in\N \Lambda^\square$. Then we have $F_\xi^{\Omega_1^\pm}(F_\xi^{\Omega_2^\pm})^{-1}\in \cO_{\xi,\Lambda}^{<\delta}$.
\end{lemma}
\begin{proof}
Without loss of generality, it is enough to consider the case when $\Omega_1^+$ (resp. $\Omega_1^-$) is a $\Lambda$-modification with level $\gamma'$ of $\Omega_2^+$ (resp. $\Omega_2^-$) for some $\gamma'<\delta$. Following Definition~\ref{def: modification of subsets}, we replace the set $\Omega_a$ there with $\Omega_a^+$ (resp. $\Omega_a^-$), and obtain an element $(\al_a^+,j_+)$ (resp. $(\al_a^-,j_-)$) and a multi-subset $\Omega_a^{+,\prime}\subseteq \Omega_a^+$ (resp. $\Omega_a^{-,\prime}\subseteq \Omega_a^-$) for each $a=1,2$. We let $\Omega_3^+\defeq \Omega_1^{+,\prime}$, $\Omega_3^-\defeq \Omega_2^{+,\prime}$ (resp. $\Omega_4^+\defeq \Omega_1^{-,\prime}$, $\Omega_4^-\defeq \Omega_2^{-,\prime}$), and note that $\Omega_3^\pm$ (resp. $\Omega_4^\pm$) is clearly a balanced pair of $\Lambda$-multi-sets satisfying $|\Omega_3^\pm|<\delta$ (resp. $|\Omega_4^\pm|<\delta$). Hence, Lemma~\ref{lem: union of lifts} implies that $F_\xi^{\Omega_3^\pm}\in \cO_{\xi,\Lambda}^{<\delta}$ and $F_\xi^{\Omega_4^\pm}\in \cO_{\xi,\Lambda}^{<\delta}$. Then the other conditions $\Omega_1^+\setminus\Omega_3^+=\Omega_2^+\setminus\Omega_3^-$ and $\Omega_1^-\setminus\Omega_4^+=\Omega_2^-\setminus\Omega_4^-$ clearly imply that $F_\xi^{\Omega_1^\pm}(F_\xi^{\Omega_2^\pm})^{-1}\sim F_\xi^{\Omega_3^\pm}(F_\xi^{\Omega_4^\pm})^{-1}$. Hence, we finish the proof.
\end{proof}

\subsection{Combinatorics of $\Lambda$-decompositions} \label{sub:comb:gen}
Before we define constructible $\Lambda$-lifts, we first need to better understand decompositions of elements of $\widehat{\Lambda}$ into that of $\Lambda$. In this section, we start with introducing \emph{$\Lambda$-decompositions} and more generally \emph{pseudo $\Lambda$-decompositions} of some $(\al,j)\in\widehat{\Lambda}$. We attach some combinatorial data to each $\Lambda$-decomposition, and then use these data to study the internal structure of the set of all $\Lambda$-decompositions of some fixed $(\al,j)\in\widehat{\Lambda}$. We show that the study of a general $\Lambda$-decomposition can be reduced to that of either \emph{$\Lambda$-exceptional} or \emph{$\Lambda$-extremal} ones (see Lemma~\ref{lem: reduce to extremal}). Last but not least, we introduce the notion of \emph{$\Lambda$-ordinary $\Lambda$-decompositions} and explain how to reduce the study of $\Lambda$-exceptional or $\Lambda$-extremal $\Lambda$-decompositions to the ones that are furthermore $\Lambda$-ordinary (see Lemma~\ref{lem: reduce to ordinary}). All the combinatorial constructions in this section will be crucially used in the definition of constructible $\Lambda$-lifts and the proof of Theorem~\ref{thm: reduce to constructible} in \S\,\ref{sub: cons lifts}. These combinatorial constructions or conditions are mainly motivated by later applications in \S\,\ref{sec:inv:cons}.

\begin{defn}\label{def: decomposition of element}
Let $\Lambda^\square$ (resp.~$\widehat{\Lambda}^\square$) be the image of $\Lambda$ (resp.~$\widehat{\Lambda}$) in $\mathrm{Supp}^\square_\xi$ and $\gamma\in\Phi^+_{\GL_{r_\xi}}$ be a block. For $(\al,j)\in\widehat{\Lambda}\cap\mathrm{Supp}_{\xi,\cJ}^\gamma$, a subset $\Omega\subseteq\Lambda\cap \mathrm{Supp}_{\xi,j}$ is called a \emph{pseudo $\Lambda$-decomposition of $(\al,j)$} if the following conditions hold:
\begin{itemize}
\item $\Omega$ maps bijectively to a subset $\Omega^\square\subseteq \Lambda^\square$ under $\Lambda\twoheadrightarrow \Lambda^\square$ and $\gamma=\sum_{\gamma'\in\Omega^\square}\gamma'$;
\item there exist $(\beta,j),(\beta',j)\in\Omega$ such that $i_\beta=i_\al$ and $i_{\beta'}^\prime=i_\al^\prime$;
\item $((i_\al,i),j),((i,i_\al^\prime),j)\in\widehat{\Lambda}$ for each $(i,j)\in \mathbf{I}_\Omega\cup\mathbf{I}_{\Omega}'\setminus\{(i_\al,j),(i_\al^\prime,j)\}$.
\end{itemize}
For each pseudo $\Lambda$-decomposition $\Omega$ of $(\al,j)$, we write $(i_{\Omega,1},j)\in\mathbf{I}_\Omega$ for the unique element such that there exists $(\beta,j)\in\Omega$ that satisfies $i_\beta^\prime=i_\al^\prime$ and $i_\beta=i_{\Omega,1}$.

For each pseudo $\Lambda$-decomposition $\Omega$ of $(\al,j)$, there exists a unique pseudo $\widehat{\Lambda}$-decomposition $\widehat{\Omega}$ of $(\al,j)$ such that $\widehat{\Omega}$ has no interior points and each element of $\widehat{\Omega}$ is a sum of some elements in $\Omega$. More precisely, there exists a partition $$\Omega=\underset{(\al',j)\in\widehat{\Omega}}{\bigsqcup}\Omega_{(\al',j)}$$ such that $\Omega_{(\al',j)}\subseteq \mathrm{Supp}_{\xi,j}$ and $\sum_{(\beta,j)\in\Omega_{(\al',j)}}\beta=\al'$ for each $(\al',j)\in\widehat{\Omega}$.

A pseudo $\Lambda$-decomposition of $(\al,j)$ is called a \emph{$\Lambda$-decomposition of $(\al,j)$} if $\sum_{(\beta,j)\in\Omega}\beta=\al$. We write $\mathbf{D}_{(\al,j),\Lambda}$ for the set of $\Lambda$-decompositions of $(\al,j)$ (cf.~Definition~\ref{def: decomposition of element}). For each $\Omega\in \mathbf{D}_{(\al,j),\Lambda}$, we write $\Omega=\{((i_{\Omega,c},i_{\Omega,c-1}),j)\mid 1\leq c\leq \#\Omega\}$ satisfying $$i_\al=i_{\Omega, \#\Omega}< i_{\Omega, \#\Omega-1}<\cdots<i_{\Omega,1}<i_{\Omega,0}=i_\al^\prime.$$ We set $i_{\Omega,c}\defeq i_\al$ for each $c\geq \#\Omega$ for convenience.
\end{defn}

\begin{rmk}\label{rmk: two case for hat set}
Recall that we have defined $\widehat{\Omega}$ twice, once in Definition~\ref{def: separated condition} and once in Definition~\ref{def: decomposition of element}. These two definitions of $\widehat{\Omega}$ are identical whenever both definitions apply. However, a pseudo $\Lambda$-decomposition is \emph{a priori} not necessarily a subset of $\Omega^+\sqcup\Omega^-$ for some $\Lambda$-lift $\Omega^\pm$, so the definition of $\widehat{\Omega}$ in Definition~\ref{def: decomposition of element} is not covered by that of Definition~\ref{def: separated condition}.
\end{rmk}

\begin{defn}\label{def: partial order}
Let $\Omega,\Omega'\in \mathbf{D}_{(\al,j),\Lambda}$ be two $\Lambda$-decompositions of $(\al,j)$.
We say that \emph{$\Omega$ is smaller than $\Omega'$}, written $\Omega<\Omega'$, if there exists $c\geq 1$ such that $u_j(i_{\Omega,c})<u_j(i_{\Omega',c})$ and $i_{\Omega,c'}=i_{\Omega',c'}$ for each $0\leq c'\leq c-1$. It is easy to check that this defines a partial order on the set $\mathbf{D}_{(\al,j),\Lambda}$ and there exists a unique maximal element in $\mathbf{D}_{(\al,j),\Lambda}$ under this partial order. We denote this maximal element by $\Omega_{(\al,j),\Lambda}^{\rm{max}}$.
Note that $\Omega_{(\al,j),\Lambda}^{\rm{max}}=\{(\al,j)\}$ if and only if $(\al,j)\in\Lambda$.
\end{defn}

We fix a subset $\Lambda\subseteq\mathrm{Supp}_{\xi,\cJ}$, an element $(\al,j)\in\widehat{\Lambda}$ and a $\Lambda$-decomposition $\Omega$ of $(\al,j)$ in the following. We write $\psi\defeq (\Omega,\Lambda)$ to lighten the notation. Now we inductively define
\begin{itemize}
\item a finite sequence of integers $\#\Omega=c_\psi^0>c_\psi^1>\cdots>c_\psi^{d_\psi}\geq 0$;
\item for each $0\leq s\leq d_\psi$, an integer $e_{\psi,s}$ and a finite set of integers $\{i_\psi^{s,1},\dots,i_\psi^{s,e_{\psi,s}}\}$ satisfying the following
      \begin{itemize}
      \item $e_{\psi,s}\geq 1$ for each $1\leq s\leq d_\psi-1$;
      \item $((i_\psi^{s-1,e_{\psi,s-1}},i_\psi^{s,1}),j)\in\widehat{\Lambda}$ for each $1\leq s\leq d_\psi$ satisfying $e_{\psi,s}\geq 1$;
      \item $((i_\psi^{s,e},i_{\Omega,c_\psi^s}),j)\in\Lambda$ for each $1\leq s\leq d_\psi$ and $1\leq e\leq e_{\psi,s}$;
      \item $((i_\psi^{s,e-1},i_\psi^{s,e}),j)\in\widehat{\Lambda}$ for each $1\leq s\leq d_\psi$ and each $2\leq e\leq e_{\psi,s}$;
      \item $u_j(i_{\Omega,c_\psi^s+1})>u_j(i_\psi^{s,1})>\cdots>u_j(i_\psi^{s,e_{\psi,s}})>u_j(i_{\Omega,c_\psi^s})$ for each $1\leq s\leq d_\psi$.
      \end{itemize}
\end{itemize}
If $s=0$, we set $$c_\psi^0\defeq \#\Omega,\,\,\,e_{\psi,0}\defeq 1,\,\,\,\mbox{and}\,\,\,i_\psi^{0,1}\defeq i_{\Omega,\#\Omega}=i_\al.$$
Assume that $c_\psi^{s-1}$, $e_{\psi,s-1}$, and the set $\{i_{\psi}^{s-1,1},\cdots,i_\psi^{s-1,e_{\psi,s-1}}\}$ (with the listed properties) have been defined for $s\geq 1$. Then we define $$c_\psi^s\defeq\max\left\{c\,\bigg{|}c_\psi^{s-1}>c\,\mbox{ and }\,\#\mathbf{D}_{((i_\psi^{s-1,e_{\psi,s-1}},i_{\Omega,c}),j),\Lambda}\geq 2\right\}.$$
If such an integer $c_\psi^s$ does not exist, we stop the process and set $d_{\psi}\defeq s-1$. If $c_\psi^s$ exists, we consider the set
\begin{equation}\label{equ: next step}
\left\{i_{\Omega',1}\,\bigg{|}\, \Omega'\in\mathbf{D}_{((i_\psi^{s-1,e_{\psi,s-1}},i_{\Omega,c_\psi^s}),j),\Lambda} \right\}\setminus\{i_{\Omega,c_\psi^s+1}\}.
\end{equation}
If the set (\ref{equ: next step}) is empty, we stop the process and set $d_{\psi}\defeq s-1$.
If the set (\ref{equ: next step}) is non-empty, but the set
\begin{equation}\label{equ: next step prime}
\left\{i_{\Omega',1}\,\bigg{|}\, u_j(i_{\Omega',1})<u_j(i_{\Omega,c_\psi^s+1}),\,\,\Omega'\in\mathbf{D}_{((i_\psi^{s-1,e_{\psi,s-1}},i_{\Omega,c_\psi^s}),j),\Lambda} \right\}
\end{equation}
is empty, then we stop the process and set $d_\psi\defeq s$ and $e_{\psi,s}\defeq 0$. If the set (\ref{equ: next step prime}) is non-empty, then
we define $i_\psi^{s,1}$ by the equality
$$u_j(i_\psi^{s,1})=\max\left\{u_j(i_{\Omega',1})\,\bigg{|}\,u_j(i_{\Omega',1})<u_j(i_{\Omega,c_\psi^s+1}),\,\,\Omega'\in\mathbf{D}_{((i_\psi^{s-1,e_{\psi,s-1}},i_{\Omega,c_\psi^s}),j),\Lambda} \right\}.$$
For each fixed $s\geq 1$ with (\ref{equ: next step prime}) being non-empty, we define the integer $i_\psi^{s,e}$ by an increasing induction on $e$.
Assume that $i_\psi^{s,e}$ has been defined for some $e\geq 1$. If $\mathbf{D}_{((i_\psi^{s,e},i_{\Omega,c_\psi^s}),j),\Lambda}=\{((i_\psi^{s,e},i_{\Omega,c_\psi^s}),j)\}$, we stop the process and set $e_{\psi,s}\defeq e$. If $\mathbf{D}_{((i_\psi^{s,e},i_{\Omega,c_\psi^s}),j),\Lambda}\neq\{((i_\psi^{s,e},i_{\Omega,c_\psi^s}),j)\}$, we define $i_\psi^{s,e+1}$ by the equality
$$u_j(i_\psi^{s,e+1})=\max\{u_j(i_{\Omega',1})\mid \Omega'\in\mathbf{D}_{((i_\psi^{s,e},i_{\Omega,c_\psi^s}),j),\Lambda}\setminus\{((i_\psi^{s,e},i_{\Omega,c_\psi^s}),j)\}\}.$$
The desired properties for the sequence $i_\psi^{s,1},\dots,i_\psi^{s,e_{\psi,s}}$ clearly follows from the inductive definition above. We observe that $\#\mathbf{D}_{((i_\psi^{s-1,e_{\psi,s-1}},i_{\Omega,c}),j),\Lambda}=1$ for each $1\leq s\leq d_\psi$ and $c_\psi^s+1\leq c\leq \min\{c_\psi^{s-1},\#\Omega-1\}$, and $\#\mathbf{D}_{((i_\psi^{d_\psi,e_{\psi,d_\psi}},i_{\Omega,c}),j),\Lambda}=1$ for each $0\leq c\leq \min\{c_\psi^{d_\psi},\#\Omega-1\}$ (if $e_{\psi,d_\psi}\geq 1$).

We investigate the case $d_\psi=0$. According to our definition, $d_\psi=0$ if and only if either $c_\psi^1$ is not defined (namely $\#\mathbf{D}_{(\al,j),\Lambda}=1$) or $c_\psi^1$ is defined and the set (\ref{equ: next step}) is empty for $s=1$. However, if $c_\psi^1$ is defined and the set (\ref{equ: next step}) is empty for $s=1$, we must have $c_\psi^1+1<\#\Omega$ and $\#\mathbf{D}_{((i_\al,i_{\Omega,c_\psi^1+1}),j),\Lambda}\geq 2$ which contradicts the maximality condition in the definition of $c_\psi^1$. Consequently, $d_\psi=0$ if and only if $\#\mathbf{D}_{(\al,j),\Lambda}=1$.

\begin{defn}\label{def: extremal decomposition}
We say that $\Omega\in\mathbf{D}_{(\al,j),\Lambda}$ is \emph{$\Lambda$-exceptional} if either $i_{\Omega,1}=i_\al$ (namely $\Omega=\{(\al,j)\}$) or $i_{\Omega,1}>i_\al$ and $\#\mathbf{D}_{((i_\al,i_{\Omega,1}),j),\Lambda}=1$. We say that $\Omega\in\mathbf{D}_{(\al,j),\Lambda}$ is \emph{$\Lambda$-extremal} if it is not $\Lambda$-exceptional and satisfies
\begin{equation}\label{equ: extremal cond}
u_j(i_{\Omega,c_\psi^s+1})=\max\left\{u_j(i_{\Omega',1})\,\bigg{|}\,\Omega'\in\mathbf{D}_{((i_\psi^{s-1,e_{\psi,s-1}},i_{\Omega,c_\psi^s}),j),\Lambda} \right\}
\end{equation}
for each $1\leq s\leq d_\psi$.
\end{defn}
Note that $\Omega\in\mathbf{D}_{(\al,j),\Lambda}$ is not $\Lambda$-exceptional if and only if $d_{\psi}\geq 1$ and $c_\psi^1\geq 1$.

Let $\Omega\in\mathbf{D}_{(\al,j),\Lambda}$ be either $\Lambda$-exceptional or $\Lambda$-extremal, and let $\psi=(\Omega,\Lambda)$. Then exactly one of the following holds:
\begin{itemize}
\item $d_\psi=0$;
\item $d_\psi=1$, $c_\psi^1=0$ and $e_{\psi,1}=0$;
\item $d_\psi\geq 1$ and $e_{\psi,s}\geq 1$ for each $1\leq s\leq d_\psi$.
\end{itemize}
For each $k\in\mathbf{n}$, we attach a subset $\Omega_{\psi,k}\subseteq\Lambda$. We first define the following: for each $u_j(i_\al)\geq k>u_j(i_\al^\prime)$
$$\lceil k\rceil\defeq\min\left\{k'\in\{u_j(i_{\Omega,c})\mid \#\Omega\geq c\geq 0\} \cup\{u_j(i_\psi^{s,e})\mid 1\leq s\leq d_\psi\mbox{ and }1\leq e\leq e_{\psi,s}\}\mid k'\geq k\right\}$$
and
$$\lceil k\rceil'\defeq\min\left\{k'\in\{u_j(i_{\Omega,c_\psi^s})\mid 0\leq s\leq d_\psi\} \mid k'\geq k\right\}.$$
We are now ready to define $\Omega_{\psi,k}$ for each $k\in\mathbf{n}$ (cf.~Figure~\ref{fig:attach:set:to:seq}).
\begin{itemize}
\item If $\lceil k\rceil=u_j(i_{\Omega,c})$ for some $\#\Omega\geq c\geq 1$ and $\lceil k\rceil'=u_j(i_{\Omega,c_\psi^0})$ then $$\Omega_{\psi,k}\defeq \{((i_{\Omega,c'},i_{\Omega,c'-1}),j)\mid \#\Omega\geq c'\geq c\};$$
\item If $\lceil k\rceil=u_j(i_{\Omega,c})$ for some $\#\Omega\geq c\geq 1$ and $\lceil k\rceil'=u_j(i_{\Omega,c_\psi^s})$ for some $1\leq s\leq d_\psi$ then $$\Omega_{\psi,k}\defeq \{((i_\psi^{s,e_{\psi,s}},i_{\Omega,c_\psi^s}),j)\}\sqcup\{((i_{\Omega,c'},i_{\Omega,c'-1}),j)\mid c_\psi^s\geq c'\geq c\};$$
\item If $\lceil k\rceil=u_j(i_\psi^{s,e})$ for some $1\leq s\leq d_\psi$ and $1\leq e\leq e_{\psi,s}$ then $$\Omega_{\psi,k}\defeq \{((i_\psi^{s,e},i_{\Omega,c_\psi^s}),j)\};$$
\item If $k>u_j(i_\al)$ or $u_j(i_\al^\prime)\geq k$, we set $$\Omega_{\psi,k}\defeq \emptyset.$$
\end{itemize}
It is not difficult to observe that $\Omega_{\psi,k}\neq \emptyset$ if and only if $u_j(i_\al)\geq k>u_j(i_\al^\prime)$. Moreover, if $\Omega_{\psi,k}\neq \emptyset$, there exists $\al_{\psi,k}=(i_{\psi,k},i_{\psi,k}^\prime)$ such that $(\al_{\psi,k},j)\in\widehat{\Lambda}$ and $\Omega_{\psi,k}\in\mathbf{D}_{(\al_{\psi,k},j),\Lambda}$.

The following lemma is the main reason for us to introduce the combinatorial data above, and will be extensively used in \S\,\ref{sub: exp type I}, \S\,\ref{sub: exp type II} and \S\,\ref{sub: exp type III}.
\begin{lemma}\label{lem: unique k decomposition}
Let $\Omega\in\mathbf{D}_{(\al,j),\Lambda}$ be either $\Lambda$-exceptional or $\Lambda$-extremal, and let $\psi=(\Omega,\Lambda)$. Then $\Omega_{\psi,k}$ is $\Lambda$-exceptional and we have
\begin{equation}\label{equ: unique element bounded by k}
\{\Omega'\mid \Omega'\in\mathbf{D}_{(\al_{\psi,k},j),\Lambda},~u_j(i_{\Omega',1})\geq k\}=\{\Omega_{\psi,k}\}
\end{equation}
for each $k\in\mathbf{n}$ with $\al_{\psi,k}\notin\{\al,0\}$. %
Moreover, if $\Omega=\Omega_{(\al,j),\Lambda}^{\rm{max}}$ then $\Omega_{\psi,k}=\Omega_{(\al_{\psi,k},j),\Lambda}^{\rm{max}}$, and the equality (\ref{equ: unique element bounded by k}) still holds for each $k\in\mathbf{n}$ with $\al_{\psi,k}=\al$.
\end{lemma}
\begin{proof}
If $u_j(i_{\Omega,c_\psi^{s-1}})\geq k>u_j(i_{\Omega,c_\psi^s+1})$ for some $1\leq s\leq d_\psi$, then we have $\mathbf{D}_{(\al_{\psi,k},j),\Lambda}=\{\Omega_{\psi,k}\}$ by the definition of $c_\psi^s$ and the claims are clear. The case when $u_j(i_{\Omega,c_\psi^{d_\psi}})\geq k>u_j(i_\al^\prime)$ is similar. If $u_j(i_\psi^{s,e})\geq k>u_j(i_\psi^{s,e+1})$ for some $1\leq s\leq d_\psi$ and $1\leq e\leq e_{\psi,s}$, then we have $\Omega_{\psi,k}=\{(\al_{\psi,k},j)\}$ which is clearly $\Lambda$-exceptional, and we deduce (\ref{equ: unique element bounded by k}) from the definition of $i_\psi^{s,e+1}$. The case when $u_j(i_\psi^{s,e_{\psi,s}})\geq k>u_j(i_{\Omega,c_\psi^s})$ for some $1\leq s\leq d_\psi$ is similar. If $u_j(i_{\Omega,c_\psi^s+1})\geq k>u_j(i_\psi^{s,1})$ for some $1\leq s\leq d_\psi$ with $e_{\psi,s}\geq 1$, then $\Omega_{\psi,k}$ is $\Lambda$-exceptional by the definition of $c_\psi^s$. If moreover $\al_{\psi,k}\neq \al$, then $\Omega$ is $\Lambda$-extremal and we deduce (\ref{equ: unique element bounded by k}) from the definition of $\Omega$ being $\Lambda$-extremal. If $\Omega=\Omega_{(\al,j),\Lambda}^{\rm{max}}$, the claims are immediate. The proof is thus finished.
\end{proof}

For a given $\Lambda$-decomposition $\Omega$ of $(\al,j)\in\widehat{\Lambda}\cap\mathrm{Supp}_{\xi,\cJ}^\gamma$, we construct $\Omega_{s,e}$ for each $1\leq s\leq d_\psi$ and $1\leq e\leq e_{\psi,s}$, where $\psi\defeq (\Omega,\Lambda)$, as follows. We first construct an element $\Omega_{s,e}^\sharp\in\mathbf{D}_{((i_\al,i_{\Omega,c_\psi^s}),j),\Lambda}$ for each $1\leq s\leq d_\psi$ and each $1\leq e\leq e_{\psi,s}$ by an increasing induction on $s$.
We set $\Omega_{0,1}^\sharp\defeq \emptyset$ for convenience. Let $1\leq s\leq d_\psi$ be an integer and we assume inductively that for each $1\leq s'\leq s-1$ and each $1\leq e\leq e_{\psi,s'}$, there exists a $\Omega_{s',e}^\sharp\in\mathbf{D}_{((i_\al,i_{\Omega,c_\psi^{s'}}),j),\Lambda}$ which contains $((i_\psi^{s',e},i_{\Omega,c_\psi^{s'}}),j)$ (and thus $i_{\Omega_{s',e}^\sharp,1}=i_\psi^{s',e}$). If $e_{\psi,s}\geq 1$, it follows from the definition of $i_\psi^{s,1}$ that there exists $\Omega_{s,1}^\flat\in \mathbf{D}_{((i_\psi^{s-1,e_{\psi,s-1}},i_{\Omega,c_\psi^s}),j),\Lambda}$ which contains $((i_\psi^{s,1},i_{\Omega,c_\psi^s}),j)$ (and thus $i_{\Omega_{s,1}^\flat,1}=i_\psi^{s,1}$). Hence we set
$$\Omega_{s,1}^\sharp\defeq (\Omega_{s-1,e_{\psi,s-1}}^\sharp\setminus\{((i_\psi^{s-1,e_{\psi,s-1}},i_{\Omega,c_\psi^{s-1}}),j)\})\sqcup\Omega_{s,1}^\flat.$$
Here we understand $\{((i_\psi^{s-1,e_{\psi,s-1}},i_{\Omega,c_\psi^{s-1}}),j)\}$ to be $\emptyset$ if $s=1$. For each $1\leq e\leq e_{\psi,s}-1$, there exists $\Omega_{s,e}^\natural\in\mathbf{D}_{((i_\psi^{s,e},i_{\Omega,c_\psi^s}),j),\Lambda}$ such that $i_{\Omega_{s,e}^\natural,1}=i_\psi^{s,e+1}$, and thus we set
$$\Omega_{s,e+1}^\sharp\defeq (\Omega_{s,e}^\sharp\setminus\{((i_\psi^{s,e},i_{\Omega,c_\psi^s}),j)\})\sqcup\Omega_{s,e}^\natural$$
for each $1\leq e\leq e_{\psi,s}-1$. For each $0\leq s\leq d_\psi$ and each $1\leq e\leq e_{\psi,s}$, we set
\begin{equation}\label{eq: Omega_s,e}
\Omega_{s,e}\defeq \Omega_{s,e}^\sharp\sqcup\{((i_{\Omega,c},i_{\Omega,c-1}),j)\mid 1\leq c\leq c_\psi^s\}.
\end{equation}
It is clear that $\Omega_{s,e}\in \mathbf{D}_{(\al,j),\Lambda}$.

\begin{lemma}\label{lem: special equivalent lift}
Let $\gamma\in\widehat{\Lambda}^\square$ be a block, $(\al,j)\in\widehat{\Lambda}\cap\mathrm{Supp}_{\xi,\cJ}^\gamma$ be an element, and $\Omega\in\mathbf{D}_{(\al,j),\Lambda}$ be a $\Lambda$-decomposition of $(\al,j)$. Then for each $1\leq s\leq d_\psi$ and each $1\leq e\leq e_{\psi,s}$, $\Omega_{s,e}\in\mathbf{D}_{(\al,j),\Lambda}$ satisfies the following properties.
\begin{itemize}
\item $(i_\psi^{s,e},j)$ is an interior point of $\Omega_{s,e}$;
\item $\Omega_{s,e}$ is $\Lambda$-equivalent to $\Omega_{s,e'}$ with level $<\gamma$ for each $1\leq e,e'\leq e_{\psi,s}$;
\item $\Omega_{s,e}\supseteq\{((i_{\Omega,c},i_{\Omega,c-1}),j)\mid 1\leq c\leq c_\psi^s\}$.
\end{itemize}
Moreover, if $\Omega$ is not $\Lambda$-exceptional, then $\Omega_{s,e}$ is $\Lambda$-equivalent to $\Omega$ with level $<\gamma$ for each $1\leq s\leq d_\psi$ and each $1\leq e\leq e_{\psi,s}$.
\end{lemma}

\begin{proof}
By construction of $\Omega_{s,e}$, it is clear that it satisfies the three properties. For the last part,
we write $\gamma_s$ for the image of $((i_\psi^{s-1,e_{\psi,s-1}},i_{\Omega,c_\psi^s}),j)$ under $\widehat{\Lambda}\twoheadrightarrow\widehat{\Lambda}^\square$. If $\Omega$ is not $\Lambda$-exceptional, we clearly have $d_\psi\geq 1$ and $c_\psi^1\geq 1$ and thus $\gamma_s<\gamma$ for each $1\leq s\leq d_\psi$. Then we observe that $\Omega_{s,e}\in\mathbf{D}_{(\al,j),\Lambda}$ is a $\Lambda$-modification of $\Omega_{s-1,e_{\psi,s-1}}$ with level $\gamma_s<\gamma$ (see Definition~\ref{def: modification of subsets}) for each $1\leq s\leq d_\psi$ and each $1\leq e\leq e_{\psi,s}$. Hence, we finish the proof by Definition~\ref{def: modification of subsets} and the fact that $\Omega_{0,1}=\Omega$.
\end{proof}

\begin{lemma}\label{lem: reduce to extremal}
Let $\gamma\in\widehat{\Lambda}^\square$ be a block, $(\al,j)\in\widehat{\Lambda}\cap\mathrm{Supp}_{\xi,\cJ}^\gamma$ be an element, and $\Omega\in\mathbf{D}_{(\al,j),\Lambda}$ be a $\Lambda$-decomposition of $(\al,j)$. Then there exists a $\Omega'\in\mathbf{D}_{(\al,j),\Lambda}$ such that
\begin{itemize}
\item $\Omega'$ is $\Lambda$-equivalent to $\Omega$ with level $<\gamma$;
\item $\Omega'$ is either $\Lambda$-exceptional or $\Lambda$-extremal;
\item either $\Omega'=\Omega$ or $\Omega<\Omega'$.
\end{itemize}
In particular, $\Omega_{(\al,j),\Lambda}^{\rm{max}}$ is either $\Lambda$-exceptional or $\Lambda$-extremal.
\end{lemma}
\begin{proof}
We argue by induction on the partial order on $\mathbf{D}_{(\al,j),\Lambda}$ introduced in Definition~\ref{def: partial order}. Now we assume inductively that for each $\Omega_0\in\mathbf{D}_{(\al,j),\Lambda}$ with $\Omega<\Omega_0$, there exists a $\Omega'\in\mathbf{D}_{(\al,j),\Lambda}$, which is $\Lambda$-equivalent to $\Omega_0$ with level $<\gamma$, such that $\Omega'$ is either $\Lambda$-exceptional or $\Lambda$-extremal, and either $\Omega'=\Omega_0$ or $\Omega_0<\Omega'$.
If $\Omega$ is $\Lambda$-exceptional or $\Lambda$-extremal, then we have nothing to prove. Otherwise, $\Omega$ is neither $\Lambda$-exceptional nor $\Lambda$-extremal, and thus there exists $1\leq s\leq d_\psi$ and $\Omega^\flat\in\mathbf{D}_{((i_\psi^{s-1,e_{\psi,s-1}},i_{\Omega,c_\psi^s}),j),\Lambda}$ such that $u_j(i_{\Omega^\flat,1})>u_j(i_{\Omega,c_\psi^s+1})$. We recall $\Omega_{s-1,e_{\psi,s-1}}^\sharp$ from the paragraph right before Lemma~\ref{lem: special equivalent lift} and set
$$\Omega_0\defeq (\Omega_{s-1,e_{\psi,s-1}}^\sharp\setminus\{((i_\psi^{s-1,e_{\psi,s-1}},i_{\Omega,c_\psi^{s-1}}),j)\})\sqcup\Omega^\flat\sqcup\{((i_{\Omega,c},i_{\Omega,c-1}),j)\mid 1\leq c\leq c_\psi^s\}.$$
Then it is clear that $\Omega<\Omega_0$ and thus there exists $\Omega'\in\mathbf{D}_{(\al,j),\Lambda}$ which is either $\Lambda$-exceptional or $\Lambda$-extremal, such that $\Omega_0$ is $\Lambda$-equivalent to $\Omega'$ with level $<\gamma$ and satisfies either $\Omega_0=\Omega'$ or $\Omega_0<\Omega'$. It is clear that $\Omega'$ satisfies all the desired properties and the proof is finished.
\end{proof}

We observe that $((i_\psi^{s-1,e_{\psi,s-1}},i_\psi^{s,1}))\in\widehat{\Lambda}$ for each $2\leq s\leq d_\psi$ with $e_{\psi,s}\geq 1$ and $((i_\psi^{s,e},i_\psi^{s,e+1}),j)\in\widehat{\Lambda}$ for each $1\leq s\leq d_\psi$ and $1\leq e\leq e_{\psi,s}-1$, which implies that
\begin{equation}\label{equ: elements in hat}
((i_\psi^{s,e},i_\psi^{s',e'}),j)\in\widehat{\Lambda}
\end{equation}
and thus
\begin{equation}\label{equ: pair of disjoint orbits}
](u_j(i_\psi^{s,e}),j),(u_j(i_\psi^{s,e}),j)]_{w_\cJ}\cap ](u_j(i_\psi^{s',e'}),j),(u_j(i_\psi^{s',e'}),j)]_{w_\cJ}=\emptyset
\end{equation}
whenever either $s<s'$ or $s=s'$ and $e<e'$ holds. For each $\Omega\in\mathbf{D}_{(\al,j),\Lambda}$, we set
\begin{equation}\label{equ: interior interval}
I_\cJ^{\psi,+}\defeq~ \bigsqcup_{c=1}^{\#\Omega-1}](u_j(i_{\Omega,c}),j),(u_j(i_{\Omega,c}),j)]_{w_\cJ}\subseteq\mathbf{n}_{\cJ}
\end{equation}
and
\begin{equation}\label{equ: interior interval prime}
I_\cJ^{\psi,-}\defeq~ \bigsqcup_{s=1}^{d_\psi}\bigsqcup_{e=1}^{e_{\psi,s}}](u_j(i_\psi^{s,e}),j),(u_j(i_\psi^{s,e}),j)]_{w_\cJ}\subseteq\mathbf{n}_{\cJ}.
\end{equation}
It is clear that $I_\cJ^{\psi,-}\neq \emptyset$ if and only if $d_\psi\geq 1$ and $e_{\psi,1}\geq 1$.

\begin{defn}\label{def: ordinary decomposition}
Let $(\al,j)$ be an element of $\widehat{\Lambda}$ and $\Omega$ be a $\Lambda$-decomposition of $(\al,j)$.
We say that $\Omega$ is \emph{$\Lambda$-ordinary} if $I_\cJ^{\psi,+}\cap I_\cJ^{\psi,-}=\emptyset$.
\end{defn}

\begin{rmk}\label{rmk: ord is Lambda ord}
If the niveau $w_\cJ$ is \emph{ordinary}, namely $r_\xi=n$ (and thus $[m]_\xi=\{m\}$ for each $1\leq m\leq n$), then $\Omega$ is $\Lambda$-ordinary for each $\Omega\in\mathbf{D}_{(\al,j),\Lambda}$ and each $(\al,j)\in\widehat{\Lambda}$.
\end{rmk}

For each $\gamma\in\widehat{\Lambda}^\square$, $(\al,j)\in\widehat{\Lambda}\cap\mathrm{Supp}_{\xi,\cJ}^\gamma$ and each $\Omega\in\mathbf{D}_{(\al,j),\Lambda}$, we define a pseudo $\Lambda$-decomposition $\Omega_\dagger$ of $(\al,j)$ in the following. Intuitively, $\Omega_\dagger$ is a kind of ``ordinarization'' of $\Omega$. We assume inductively that $\Omega'_\dagger$ has been defined for each $\gamma'\in\widehat{\Lambda}$, $(\al',j)\in\widehat{\Lambda}\cap\mathrm{Supp}_{\xi,\cJ}^{\gamma'}$ and $\Omega'\in\mathbf{D}_{(\al',j),\Lambda}$ satisfying $\gamma'<\gamma$. If $\Omega$ is $\Lambda$-ordinary, we set $\Omega_\dagger\defeq \Omega$. If $\Omega$ is not $\Lambda$-ordinary, then there exists $1\leq c_\dagger\leq \#\Omega-1$, $1\leq s_\dagger\leq d_\psi$, $1\leq e_\dagger\leq e_{\psi,s_\dagger}$
and $1\leq m_\dagger\leq r_\xi$ such that $i_{\Omega,c_\dagger},i_\psi^{s_\dagger,e_\dagger}\in[m_\dagger]_\xi$. It follows from $((i_\psi^{s_\dagger,e_\dagger},i_{\Omega,c_\psi^{s_\dagger}}),j)\in\Lambda$ that we have $c_\dagger\geq c_\psi^{s_\dagger}+1$.
We choose $s_\dagger$ and $e_\dagger$ such that $s_\dagger$ is maximal possible and $e_\dagger$ is minimal possible for the fixed $s_\dagger$, and then set
$$\Omega_\dagger^\flat\defeq \{((i_\psi^{s_\dagger,e_\dagger},i_{\Omega,c_\psi^{s_\dagger}}))\}\sqcup\{((i_{\Omega,c},i_{\Omega,c-1}),j)\mid 1\leq c\leq c_\psi^{s_\dagger}\}.$$
We claim that $\Omega_\dagger^\flat$ is $\Lambda$-ordinary (otherwise there exists $1\leq c'\leq c_\psi^{s_\dagger}-1$, $s_\dagger+1\leq s'\leq d_\psi$, $1\leq e'\leq e_{\psi,s'}$ and $1\leq m'\leq r_\xi$ such that $i_\psi^{s',e'},i_{\Omega,c'}\in[m']_\xi$, contradicting the maximality of $s_\dagger$).
Then we set $\Omega'\defeq \{((i_{\Omega,c},i_{\Omega,c-1}),j)\mid 1+c_{\dagger}\leq c\leq \#\Omega\}$ and note that $\Omega'_\dagger$ is defined by our inductive assumption. Then we set
\begin{equation}\label{eq: dagger}
\Omega_\dagger\defeq \Omega'_\dagger\sqcup\Omega_\dagger^\flat
\end{equation}
and note that $\Omega_\dagger$ is a pseudo $\Lambda$-decomposition of $(\al,j)$. We write $\widehat{\Omega}_\dagger$ for the pseudo $\widehat{\Lambda}$-decomposition of $(\al,j)$ associated with $\Omega_\dagger$ via Definition~\ref{def: decomposition of element}.
\begin{lemma}\label{lem: reduce to ordinary}(Properties of ordinarization)
Let $(\al,j)\in\widehat{\Lambda}\cap\mathrm{Supp}_{\xi,\cJ}^\gamma$ for a block $\gamma\in\widehat{\Lambda}^\square$ and $\psi=(\Omega,\Lambda)$ for $\Omega\in\mathbf{D}_{(\al,j),\Lambda}$. Then the pseudo $\Lambda$-decomposition $\Omega_\dagger$ of $(\al,j)$ satisfies the following conditions:
\begin{itemize}
\item $\Omega_{\dagger}$ is $\Lambda$-equivalent to $\Omega$ with level $<\gamma$;
\item $\Omega_{\dagger,(\al',j)}$ is $\Lambda$-ordinary for each $(\al',j)\in\widehat{\Omega}_\dagger$;
\item if $\Omega$ is $\Lambda$-exceptional and not $\Lambda$-ordinary, then there exists $1\leq c_\dagger\leq \#\Omega-1$ and $1\leq e_\dagger\leq e_{\psi,1}$ such that $\Omega_\dagger=\Omega_{((i_\al,i_{\Omega,c_\dagger}),j),\Lambda}^{\rm{max}} \sqcup\{((i_\psi^{1,e_\dagger},i'_{\al'}),j)\}$ and $\mathbf{D}_{((i_\al,i_{\Omega,c_\dagger}),j),\Lambda} =\big\{\Omega_{((i_\al,i_{\Omega,c_\dagger}),j),\Lambda}^{\rm{max}}\big\}$;
\item if $\Omega=\Omega_{(\al,j),\Lambda}^{\rm{max}}$, then $\Omega_{\dagger,(\al',j)}=\Omega_{(\al',j),\Lambda}^{\rm{max}}$ for each $(\al',j)\in\widehat{\Omega}_\dagger$;
\item if $\Omega$ is $\Lambda$-extremal, then $\Omega_{\dagger,(\al',j)}$ either equals $\Omega_{(\al',j),\Lambda}^{\rm{max}}$ or is $\Lambda$-extremal for each $(\al',j)\in\widehat{\Omega}_\dagger$;
\item for each $(\beta,j),\,(\beta',j)\in\widehat{\Omega}_\dagger$ satisfying $((i_{\beta'},i_\beta^\prime),j)\in\widehat{\Lambda}$, there exists a unique subset of $\Omega_\dagger$ which is a pseudo $\Lambda$-decomposition of $((i_{\beta'},i_\beta^\prime),j)$.
\end{itemize}
\end{lemma}
\begin{proof}
This follows from an immediate induction on $\gamma$ as in the construction of $\Omega_\dagger=\Omega'_\dagger\sqcup\Omega_\dagger^\flat$. The key observation is that $\Omega_\dagger^\flat$ is $\Lambda$-ordinary, and equals $\Omega_{((i_\psi^{s_\dagger,e_\dagger},i_\al^\prime),j),\Lambda}^{\rm{max}}$ (resp.~is $\Lambda$-exceptional, resp.~either equals $\Omega_{(\al,j),\Lambda}^{\rm{max}}$ or is $\Lambda$-extremal) if $\Omega$ equals $\Omega_{(\al,j),\Lambda}^{\rm{max}}$ (resp.~is $\Lambda$-exceptional, resp.~is $\Lambda$-extremal). Note that each element of $\widehat{\Omega}_\dagger$ is of the form $((i_\psi^{s,e},i_{\Omega,c}),j)$ for some $1\leq s\leq d_\psi$, $1\leq e\leq e_{\psi,s}$ and $0\leq c\leq \#\Omega-1$. The last claim follows from (\ref{equ: elements in hat}) and the fact that $((i_{\Omega,c},i_{\Omega,c'}),j)\in\widehat{\Lambda}$ for each $0\leq c'<c\leq \#\Omega$.
\end{proof}

\begin{rmk}\label{rmk: ord and pseudo decomp}
Let $\Omega$ be a $\Lambda$-decomposition of some $(\al,j)\in\widehat{\Lambda}$. It is clear that $\Omega_\dagger\in\mathbf{D}_{(\al,j),\Lambda}$ if and only if $\Omega_\dagger=\Omega$ if and only if $\Omega$ is $\Lambda$-ordinary. In other words, if $\Omega$ is not $\Lambda$-ordinary, then $\Omega_\dagger$ is a pseudo $\Lambda$-decomposition which is not a $\Lambda$-decomposition. This is actually the main reason for us to introduce the notion of pseudo $\Lambda$-decompositions (see Definition~\ref{def: decomposition of element}), which is a convenient generalization of $\Lambda$-decompositions that covers objects of the form $\Omega_\dagger$ for arbitrary $\Omega\in\mathbf{D}_{(\al,j),\Lambda}$.
\end{rmk}

\subsection{Constructible $\Lambda$-lifts}\label{sub: cons lifts}
In this section, we introduce a key notion of this paper, namely \emph{constructible $\Lambda$-lifts}. The main result of this section (see Theorem~\ref{thm: reduce to constructible}) says that all $\Lambda$-lifts can be generated from constructible ones. The heart of the proof of Theorem~\ref{thm: reduce to constructible} is to understand precisely which constructible $\Lambda$-lifts are sufficient to build up all $\Lambda$-lifts. Note that Definition~\ref{def: constructible lifts} is directly motivated by \S\,\ref{sub: exp type I}, \S\,\ref{sub: exp type II} and \S\,\ref{sub: exp type III}, and the conditions in Definition~\ref{def: constructible lifts} precisely ensure that there exists an invariant function (to be constructed in \S\,\ref{sec:const:inv}) whose restriction to $\cN_{\xi,\Lambda}$ (if defined) is closely related to the given constructible $\Lambda$-lift.

\begin{defn}\label{def: constructible lifts}
Let $\Omega^\pm$ be a $\Lambda$-lift. As in Definition~\ref{def: separated condition}, we can associate a subset $\widehat{\Omega}^+$ (resp.~$\widehat{\Omega}^-$) of $\widehat{\Lambda}$ which does not have any interior points, and we have partitions
\begin{equation*}
\Omega^+=\underset{(\al,j)\in \widehat{\Omega}^+}{\bigsqcup}\Omega_{(\al,j)}^+\,\,\,\mbox{and}\,\,\,\Omega^-= \underset{(\al,j)\in \widehat{\Omega}^-}{\bigsqcup}\Omega_{(\al,j)}^-.
\end{equation*}
We write $\psi$ for an arbitrary pair in
\begin{equation}\label{equ: set of pairs}
\{(\Omega_{(\al,j)}^+,\Lambda)\mid (\al,j)\in\widehat{\Omega}^+\}\sqcup\{(\Omega_{(\al,j)}^-,\Lambda)\mid (\al,j)\in\widehat{\Omega}^-\}.
\end{equation}
For each pair $\psi$ in (\ref{equ: set of pairs}), we use the notation $\Omega_\psi$ for the first factor of the pair $\psi$, $(\al_\psi,j_\psi)$ for the element $\widehat{\Lambda}$ such that $\Omega_\psi$ is a $\Lambda$-decomposition of $(\al_\psi,j_\psi)$ and $\gamma_\psi$ for the block that satisfies $(\al_\psi,j_\psi)\in\mathrm{Supp}_{\xi,\cJ}^{\gamma_\psi}$.

We say that $\Omega^\pm$ is a \emph{constructible $\Lambda$-lift of type $\rm{I}$} if it satisfies
\begin{enumerate}[label=(\roman*)]
\item \label{it: I 1} $\widehat{\Omega}^+=\widehat{\Omega}^-=\{(\al,j)\}$ for some $(\al,j)\in\widehat{\Lambda}$: we write $\psi_1\defeq (\Omega^+,\Lambda)$ and $\psi_2\defeq (\Omega^-,\Lambda)$;
\item \label{it: I 2} $\Omega^-=\Omega_{(\al,j),\Lambda}^{\rm{max}}$ and $\Omega^+$ is either $\Lambda$-exceptional or $\Lambda$-extremal;
\item \label{it: I 3} both $\Omega^+$ and $\Omega^-$ are $\Lambda$-ordinary (see Definition~\ref{def: ordinary decomposition});
\item \label{it: I 4} $(u_j(i_{\psi_2}^{s,e}),j)\notin I_\cJ^{\psi_1,-}$ for each $1\leq s\leq d_{\psi_2}$ and each $1\leq e\leq e_{\psi_2,s}$ satisfying $u_j(i_{\psi_2}^{s,e})>u_j(i_{\Omega^+,1})$;
\item \label{it: I 5} $((i,i'),j),\,((i',i),j)\notin\widehat{\Lambda}$ for each interior point $(i,j)$ of $\Omega^+$ and each $(i',j)$ of $\Omega^-$;
\item \label{it: I 6} $i_{\psi_1}^{s,e}\neq i_{\Omega^-,c}$ and $((i_{\psi_1}^{s,e},i_{\Omega^-,c}),j)\notin\widehat{\Lambda}$ for each $1\leq c\leq \#\Omega^--1$, $1\leq s\leq d_{\psi_1}$ and $1\leq e\leq e_{\psi_1,s}$;
\item \label{it: I 7} $i_{\psi_2}^{s,e}\neq i_{\Omega^+,c}$ and $((i_{\psi_2}^{s,e},i_{\Omega^+,c}),j)\notin\widehat{\Lambda}$ for each $1\leq c\leq \#\Omega^+-1$, $1\leq s\leq d_{\psi_2}$ and $1\leq e\leq e_{\psi_2,s}$ satisfying $u_j(i_{\psi_2}^{s,e})>u_j(i_{\Omega^+,1})$.
\end{enumerate}

We say that $\Omega^\pm$ is a \emph{constructible $\Lambda$-lift of type $\rm{II}$} if it satisfies
\begin{enumerate}[label=(\roman*)]
\item \label{it: II 1} $\widehat{\Omega}^+=\{(\al,j)\}$ for some $(\al,j)\in\widehat{\Lambda}$ and $\Omega^-$ is a pseudo $\Lambda$-decomposition of $(\al,j)$ (see Definition~\ref{def: decomposition of element}) with $\widehat{\Omega}^+\cap\widehat{\Omega}^-=\emptyset$: we write $\psi_1\defeq (\Omega^+,\Lambda)$;
\item \label{it: II 2} $\Omega^+$ is either $\Lambda$-exceptional or $\Lambda$-extremal, and $\Omega_{(\al',j)}^-=\Omega_{(\al',j),\Lambda}^{\rm{max}}$ for each $(\al',j)\in \widehat{\Omega}^-$;
\item \label{it: II 3} $\Omega^+$ is $\Lambda$-ordinary, and $\Omega_{(\al',j)}^-$ is $\Lambda$-ordinary for each $(\al',j)\in \widehat{\Omega}^-$;
\item \label{it: II 4} if there exist $\psi=(\Omega_{(\al',j)}^-,\Lambda)$ (for some $(\al',j)\in\widehat{\Omega}^-\}$) and $1\leq m\leq r_\xi$ such that $i_\psi^{s,e},i_{\psi_1}^{s',e'}\in[m]_\xi$ for some $1\leq s\leq d_\psi$, $1\leq e\leq e_{\psi,s}$, $1\leq s'\leq d_{\psi_1}$, and $1\leq e'\leq e_{\psi_1,s'}$, then we have
    $i_{\al'}^\prime=i_\al^\prime$ and either $u_j(i_\psi^{s,e})\leq u_j(i_{\Omega^+,1})<u_j(i_{\Omega^-,1})$ or $u_j(i_{\psi_1}^{s',e'})\leq u_j(i_{\Omega^-,1})<u_j(i_{\Omega^+,1})$;
\item \label{it: II 5} for each $(\al',j)\in\widehat{\Omega}^-$ and for each $1\leq s\leq d_{\psi_1}$ and $1\leq e\leq e_{\psi_1,s}$ such that $i_{\al'}^\prime,i_{\psi_1}^{s,e}\in[m]_\xi$ for some $1\leq m\leq r_\xi$, we have $u_j(i_{\psi_1}^{s,e})\leq u_j(i_{\Omega^-,1})<u_j(i_{\Omega^+,1})$;
\item \label{it: II 6} $((i,i'),j)\notin\widehat{\Lambda}$ for each pair of elements $(i,j),(i',j)\in\mathbf{I}_{\Omega^+\sqcup\Omega^-}\cup\mathbf{I}_{\Omega^+\sqcup\Omega^-}^\prime$ that do not lie in the same $\Lambda^\square$-interval of $\Omega^\pm$ (cf.~Definition~\ref{def: separated condition}); %
\item \label{it: II 7} if there exist $(\al',j)\in\widehat{\Omega}^-$ and $0\leq c\leq \#\Omega^-_{(\al',j)}$ such that $i_{\Omega^-_{(\al',j)},c}\neq i_\al^\prime$ and either $i_{\psi_1}^{s,e}=i_{\Omega^-_{(\al',j)},c}$ or $((i_{\psi_1}^{s,e},i_{\Omega^-_{(\al',j)},c}),j)\in\widehat{\Lambda}$ for some $1\leq s\leq d_{\psi_1}$ and $1\leq e\leq e_{\psi_1,s}$, then we have either $u_j(i_{\psi_1}^{s,e})\leq u_j(i_{\Omega^-,1})<u_j(i_{\Omega^+,1})$ or $u_j(i_{\al'})\leq u_j(i_{\Omega^-,1})<u_j(i_{\Omega^+,1})$;
\item \label{it: II 8} $i_\psi^{s,e}\neq i_{\Omega^+,c}$ and $((i_\psi^{s,e},i_{\Omega^+,c}),j)\notin\widehat{\Lambda}$ for each $\psi\in\{(\Omega_{(\al',j)}^-,\Lambda)\mid (\al',j)\in\widehat{\Omega}^-\}$ and for each $1\leq c\leq \#\Omega^+-1$, $1\leq s\leq d_\psi$ and $1\leq e\leq e_{\psi,s}$;
\item \label{it: II 9} if $\Omega^+=\Omega_{(\al,j),\Lambda}^{\rm{max}}$ and $u_j(i_{\al'}^\prime)<u_j(i_{\Omega^-,1})$ for the unique $(\al',j)\in\widehat{\Omega}^-$ satisfying $i_{\al'}=i_\al$, then $\Omega^+$ is $\Lambda$-exceptional and $i_{\Omega^-,1}=i_{\psi_1}^{1,1}$;
\item \label{it: II 10} if $\Omega^+\neq\Omega_{(\al,j),\Lambda}^{\rm{max}}$ and $\Omega^+$ is $\Lambda$-exceptional, then we have $u_j(i_{\al'}^\prime)>\max\{u_j(i_{\Omega^-,1}),u_j(i_{\Omega^+,1})\}$ for the unique $(\al',j)\in\widehat{\Omega}^-$ satisfying $i_{\al'}=i_\al$;
\item \label{it: II 11} if $\Omega^+\neq\Omega_{(\al,j),\Lambda}^{\rm{max}}$ and $\Omega^+$ is $\Lambda$-exceptional, then for each $(\al',j)\in\widehat{\Omega}^-$, exactly one of the following holds:
    \begin{itemize}
    \item $u_j(i_{\al'}^\prime)>u_j(i_{\Omega^+,1})$;
    \item $i_{\al'}^\prime=i_\al^\prime$ and $u_j(i_{\al'})\geq u_j(i_{\Omega^-,1})>u_j(i_{\Omega^+,1})$;
    \item $u_j(i_{\al'})\leq u_j(i_{\Omega^-,1})<u_j(i_{\Omega^+,1})$.
    \end{itemize}
\end{enumerate}

We say that $\Omega^\pm$ is a \emph{constructible $\Lambda$-lift of type $\rm{III}$} if it satisfies
\begin{enumerate}[label=(\roman*)]
\item \label{it: III 1} if both $\Omega^+$ and $\Omega^-$ are pseudo $\Lambda$-decomposition of some $(\al,j)\in\widehat{\Lambda}$, then we have $\widehat{\Omega}^+\neq \{(\al,j)\}\neq \widehat{\Omega}^-$;
\item \label{it: III 2} $\Omega_{(\al,j)}^+=\Omega_{(\al,j),\Lambda}^{\rm{max}}$ (resp.~$\Omega_{(\al,j)}^-=\Omega_{(\al,j),\Lambda}^{\rm{max}}$) for each $(\al,j)\in \widehat{\Omega}^+$ (resp.~for each $(\al,j)\in \widehat{\Omega}^-$);
\item \label{it: III 3} $\Omega_{(\al,j)}^+$ (resp.~$\Omega_{(\al,j)}^-$) is $\Lambda$-ordinary for each $(\al,j)\in \widehat{\Omega}^+$ (resp.~for each $(\al,j)\in \widehat{\Omega}^-$);
\item \label{it: III 4} the subsets $I_\cJ^{\psi,+}\cup I_\cJ^{\psi,-}\subseteq\mathbf{n}_\cJ$ are pairwise disjoint for $\psi$ running among all the pairs in (\ref{equ: set of pairs});
\item \label{it: III 5} $\{(u_j(i_\al),j),\,(u_j(i_\al^\prime),j)\}\cap I_\cJ^{\psi,-}=\emptyset$ for each $(\al,j)\in\widehat{\Omega}^+\sqcup\widehat{\Omega}^-$ and each pair $\psi$ in (\ref{equ: set of pairs});
\item \label{it: III 6} if $\Omega^+$ and $\Omega^-$ are not pseudo $\Lambda$-decompositions of the same element in $\widehat{\Lambda}$, then for each pair of elements $(\beta,j),\,(\beta',j)\in\widehat{\Omega}^+\sqcup\widehat{\Omega}^-$ satisfying $((i_{\beta'},i_\beta^\prime),j)\in\widehat{\Lambda}$, there exists a pseudo $\Lambda$-decomposition $\Omega\subseteq\Omega^+\sqcup\Omega^-$ of some $(\al,j)\in\widehat{\Lambda}$ such that $(i_{\beta'},j)\in\mathbf{I}_{\widehat{\Omega}}$ and $(i_\beta^\prime,j)\in\mathbf{I}_{\widehat{\Omega}}^\prime$;
\item \label{it: III 7} $((i,i'),j)\notin\widehat{\Lambda}$ for each pair of elements $(i,j),(i',j)\in\mathbf{I}_{\Omega^+\sqcup\Omega^-}\cup\mathbf{I}_{\Omega^+\sqcup\Omega^-}^\prime$ that do not lie in the same $\Lambda^\square$-interval of $\Omega^\pm$;
\item \label{it: III 8} for each pair $\psi$ in (\ref{equ: set of pairs}) and each element $(i,j_\psi)\in \mathbf{I}_{\Omega^+\sqcup\Omega^-}\cup \mathbf{I}_{\Omega^+\sqcup\Omega^-}^\prime$ which does not lie in a $\Lambda^\square$-interval containing $\Omega_{\psi}$, there does not exist $1\leq s\leq d_\psi$ and $1\leq e\leq e_{\psi,s}$ such that $((i_\psi^{s,e},i),j_\psi)\in\widehat{\Lambda}$;
\item \label{it: III 9} if $\Omega^+$ and $\Omega^-$ are not pseudo $\Lambda$-decompositions of the same element in $\widehat{\Lambda}$, then for each pair of distinct $\Lambda^\square$-intervals $\Omega,\Omega'$ which are pseudo $\Lambda$-decompositions of some $(\al,j),(\al',j')\in\widehat{\Lambda}$ respectively, there do not exist $(i,j)$, $(i',j')$ that satisfy the following:
    \begin{itemize}
    \item $((i_\al,i),j), ((i,i_\al^\prime),j)\in\widehat{\Lambda}$;
    \item either $i'\in\{i_{\al'},i_{\al'}^\prime\}$ or $((i_{\al'},i'),j'), ((i',i_{\al'}^\prime),j')\in\widehat{\Lambda}$
    \item $i,i'\in[m]_\xi$ for some $1\leq m\leq r_\xi$.
    \end{itemize}
\end{enumerate}
We say that $\Omega^{\pm}$ is a \emph{constructible $\Lambda$-lift} if it is a constructible $\Lambda$-lift of either type $\rm{I}$, type $\rm{II}$, or type $\rm{III}$. We write $\cO_{\xi,\Lambda}^{\rm{con}}$ for the subgroup of $\cO(\cN_{\xi,\Lambda})^\times$ generated by $\cO_{\xi,\Lambda}^{\rm{ss}}$ and $F_\xi^{\Omega^\pm}$ for all constructible $\Lambda$-lifts $\Omega^\pm$.

In the following, for example, we will write Condition~\rm{I}-\ref{it: I 1} for the condition~\ref{it: I 1} in the definition of constructible $\Lambda$-lifts of type $\rm{I}$.
\end{defn}

\begin{rmk}\label{rmk: motivation of def}
The definition of constructible $\Lambda$-lifts above is directly motivated by constructions of invariant functions in \S\,\ref{sec:const:inv} and \S\,\ref{sec:inv:cons}. In other words, given a constructible $\Lambda$-lift $\Omega^\pm$, we will construct in \S\,\ref{sec:const:inv} an invariant function $f_\xi^{\Omega^\pm}$ whose restriction to $\cN_{\xi,\Lambda}$ (if defined) is closely related to $F_\xi^{\Omega^\pm}$ (see \S\,\ref{sub:exp:for} for precise statements). The set of constructible $\Lambda$-lifts of one type is clearly disjoint for the set of constructible $\Lambda$-lift of another type, by Condition~\rm{I}-\ref{it: I 1}, \rm{II}-\ref{it: II 1} and \rm{III}-\ref{it: III 1}. Among the list of conditions in Definition~\ref{def: constructible lifts}, there are three families of conditions that stand out. The first family of conditions, notably \rm{I}-\ref{it: I 2}, \rm{II}-\ref{it: II 2} and \rm{III}-\ref{it: III 2}, all require certain $\Lambda$-decompositions to be $\Lambda$-exceptional or $\Lambda$-extremal (or even maximal). The second family of conditions, notably \rm{I}-\ref{it: I 3}, \rm{I}-\ref{it: I 4}, \rm{II}-\ref{it: II 3} to \rm{II}-\ref{it: II 5}, \rm{III}-\ref{it: III 3} to \rm{III}-\ref{it: III 5} and \rm{III}-\ref{it: III 9}, all require that certain $(w_\cJ,1)$-orbits inside $\mathbf{n}_\cJ$ are disjoint. The third family of conditions, notably \rm{I}-\ref{it: I 5} to \rm{I}-\ref{it: I 7}, \rm{II}-\ref{it: II 6} to \rm{II}-\ref{it: II 8}, \rm{III}-\ref{it: III 6}, and \rm{III}-\ref{it: III 8}, all require that certain elements $((i,i'),j)$ do not lie in $\widehat{\Lambda}$. The first and third families of conditions are related to controlling the relative position of the zero and pole divisor of $f_\xi^{\Omega^\pm}$ (as a rational function on $\tld{\cF\cL}_\cJ$) with respect to $\cN_{\xi,\Lambda}$, while the second family ensures that the restriction $f_\xi^{\Omega^\pm}|_{\cN_{\xi,\Lambda}}$ (if defined) is closely related to $F_\xi^{\Omega^\pm}$. The rest of conditions, namely \rm{II}-\ref{it: II 9} to \rm{II}-\ref{it: II 11}, will be used to reduce the number of necessary cases to be discussed in \S\,\ref{sub: exp type II}.
\end{rmk}

The rest of this section is devoted to proving Theorem~\ref{thm: reduce to constructible}, which says that the set of constructible $\Lambda$-lifts is sufficient to generate all $\Lambda$-lifts. We start with three simple lemmas which will be frequently used in the rest of the section.

Recall that for a given $\Lambda$-decomposition $\Omega$ of $(\al,j)\in\widehat{\Lambda}\cap\mathrm{Supp}_{\xi,\cJ}^\gamma$, we construct $\Omega_{s,e}$ in (\ref{eq: Omega_s,e}) for each $1\leq s\leq d_\psi$ and $1\leq e\leq e_{\psi,s}$, where $\psi\defeq (\Omega,\Lambda)$. (See also Lemma~\ref{lem: special equivalent lift} for its properties.) Hence, if $\psi$ is a pair in (\ref{equ: set of pairs}), then we write $\Omega_{\psi,s,e}$ for the corresponding $\Lambda$-decomposition of $(\al_\psi,j_\psi)$, and we further define
$$\Omega_{\psi,s,e}^+\defeq \left\{
  \begin{array}{ll}
     \Omega_{\psi,s,e} & \hbox{if $\Omega_\psi\subseteq\Omega^+$;} \\
     \Omega_\psi & \hbox{if $\Omega_\psi\subseteq\Omega^-$}
  \end{array}
\right.
\,\mbox{ and }\,
\Omega_{\psi,s,e}^-\defeq \left\{
  \begin{array}{ll}
     \Omega_{\psi} & \hbox{if $\Omega_\psi\subseteq\Omega^+$;} \\
     \Omega_{\psi,s,e} & \hbox{if $\Omega_\psi\subseteq\Omega^-$.}
  \end{array}
\right.$$
It is clear that $\Omega_{\psi,s,e}^{+},\Omega_{\psi,s,e}^-\in \mathbf{D}_{(\al_\psi,j_\psi),\Lambda}$ and so $\Omega_{\psi,s,e}^{\pm}$ forms a balanced pair.

For each $\delta\in\N \Lambda^\square$, we recall the group $\cO_{\xi,\Lambda}^{<\delta}$ from Definition~\ref{def: balanced cond}.

\begin{lemma}\label{lem: reduce non ord to smaller norm}
Let $\Omega^\pm$ be a $\Lambda$-lift, and let $\psi,\psi'$ be two distinct pairs in (\ref{equ: set of pairs}) such that there exists $1\leq m\leq r_\xi$ satisfying
\begin{equation*}
i_\psi^{s,e},i_{\psi'}^{s',e'}\in[m]_{\xi}
\end{equation*}
for some $1\leq s\leq d_\psi$ and $1\leq e\leq e_{\psi,s}$ and for some $1\leq s'\leq d_{\psi'}$ and $1\leq e'\leq e_{\psi',s'}$. Then we have
$$F_\xi^{\Omega^\pm}\cdot F_\xi^{\Omega_{\psi,s,e}^\pm}\cdot F_\xi^{\Omega_{\psi',s',e'}^\pm}\in \cO_{\xi,\Lambda}^{<|\Omega^\pm|}.$$
Moreover,
$$
\left\{
\begin{array}{ll}
    F_\xi^{\Omega_{\psi,s,e}^\pm}\in \cO_{\xi,\Lambda}^{<|\Omega^\pm|} & \hbox{if $\Omega_\psi$ is not $\Lambda$-exceptional or $\gamma_{\psi}<|\Omega^\pm|$;} \\
    F_\xi^{\Omega_{\psi',s',e'}^\pm}\in \cO_{\xi,\Lambda}^{<|\Omega^\pm|}  & \hbox{if $\Omega_{\psi'}$ is not $\Lambda$-exceptional or $\gamma_{\psi'}<|\Omega^\pm|$.}
  \end{array}
\right.
$$
\end{lemma}

\begin{proof}
We set
\begin{equation}\label{eq: def of Omega1}
\left\{
\begin{array}{ll}
     \Omega_1^+\defeq(\Omega^+\sqcup\Omega_{\psi,s,e}^+)\setminus\Omega_\psi; & \hbox{} \\
    \Omega_1^-\defeq(\Omega^-\sqcup\Omega_{\psi,s,e}^-)\setminus\Omega_\psi.  & \hbox{}
  \end{array}
\right.
\end{equation}
Then $\Omega_1^\pm$ is clearly a balanced pair of sets satisfying $|\Omega_1^\pm|=|\Omega^\pm|$ and $F_\xi^{\Omega^\pm}F_\xi^{\Omega_{\psi,s,e}^\pm}\sim F_\xi^{\Omega_1^\pm}$. Note that we have either $\Omega_{\psi'}\subseteq\Omega_1^{+}$ or $\Omega_{\psi'}\subseteq\Omega_1^{-}$. We also set
$$
\left\{
\begin{array}{ll}
     \Omega_2^+\defeq(\Omega_1^+\sqcup\Omega_{\psi',s',e'}^+)\setminus\Omega_{\psi'}; & \hbox{} \\
    \Omega_2^-\defeq(\Omega_1^-\sqcup\Omega_{\psi',s',e'}^-)\setminus\Omega_{\psi'}.  & \hbox{}
  \end{array}
\right.
$$
Then $\Omega_2^\pm$ is also a balanced pair such that $|\Omega_2^\pm|=|\Omega_1^\pm|=|\Omega^\pm|$ and
\begin{equation}\label{equ: similar product}
F_\xi^{\Omega^\pm}F_\xi^{\Omega_{\psi,s,e}^\pm}F_\xi^{\Omega_{\psi',s',e'}^\pm}\sim F_\xi^{\Omega_1^\pm}F_\xi^{\Omega_{\psi',s',e'}^\pm}\sim F_\xi^{\Omega_2^\pm}.
\end{equation}

Now we observe that both $(i_\psi^{s,e},j_\psi)$ and $(i_{\psi'}^{s',e'},j_{\psi'})$ are interior points of $\Omega_2^+\sqcup\Omega_2^-$ by Lemma~\ref{lem: special equivalent lift}, which implies that $\Omega_2^\pm$ is not a $\Lambda$-lift, as a $\Lambda$-lift can not have two distinct interior points in the same $[m]_{\xi}$ for some $1\leq m\leq r_{\xi}$. Hence, we deduce from Lemma~\ref{lem: union of lifts} that $F_\xi^{\Omega_2^\pm}\in \cO_{\xi,\Lambda}^{<|\Omega^\pm|}$, which together with (\ref{equ: similar product}) implies $F_\xi^{\Omega^\pm}F_\xi^{\Omega_{\psi,s,e}^\pm}F_\xi^{\Omega_{\psi',s',e'}^\pm}\in \cO_{\xi,\Lambda}^{<|\Omega^\pm|}$. Finally, the last part is a direct consequence of Lemma~\ref{lem: special equivalent lift} together with Lemma~\ref{lem: two equivalent pairs}. The proof is thus finished.
\end{proof}

\begin{lemma}\label{lem: reduce non ord to smaller norm II}
Let $\Omega^\pm$ be a $\Lambda$-lift, and let $\psi,\psi'$ be two distinct pairs in (\ref{equ: set of pairs}) such that
\begin{equation}\label{equ: intersection of index II}
\big\{(u_{j_{\psi'}}(i_{\al_{\psi'}}),j_{\psi'}),\,(u_{j_{\psi'}}(i_{\al_{\psi'}}'),j_{\psi'})\big\}\,\cap\, ](u_{j_\psi}(i_\psi^{s,e}),j_\psi),(u_{j_\psi}(i_\psi^{s,e}),j_\psi)]_{w_\cJ}\neq \emptyset
\end{equation}
for some $1\leq s\leq d_\psi$ and $1\leq e\leq e_{\psi,s}$. Then we have
$$F_\xi^{\Omega^\pm}\cdot F_\xi^{\Omega_{\psi,s,e}^\pm}\in \cO_{\xi,\Lambda}^{<|\Omega^\pm|}.$$
Moreover, $F_\xi^{\Omega_{\psi,s,e}^\pm}\in \cO_{\xi,\Lambda}^{<|\Omega^\pm|}$ if $\Omega_{\psi}$ is not $\Lambda$-exceptional or $\gamma_{\psi}<|\Omega^\pm|$.
\end{lemma}

\begin{proof}
The proof is very similar to that of Lemma~\ref{lem: reduce non ord to smaller norm}. We set $\Omega_1^{\pm}$ as in (\ref{eq: def of Omega1}). Then we have
$|\Omega_1^\pm|=|\Omega^\pm|$ and $F_\xi^{\Omega^\pm}F_\xi^{\Omega_{\psi,s,e}^\pm}\sim F_\xi^{\Omega_1^\pm}$.
Since $\Omega_{\psi'}\subseteq\Omega_1^+\sqcup\Omega_1^-$ and $(i_{\psi}^{s,e},j_{\psi})$ is an interior point of $\Omega_1^+\sqcup\Omega_1^-$, $\Omega_1^{\pm}$ is not a $\Lambda$-lift due to (\ref{equ: intersection of index II}), so that we conclude by Lemma~\ref{lem: union of lifts}. The last part is a direct consequence of Lemma~\ref{lem: special equivalent lift} together with Lemma~\ref{lem: two equivalent pairs}.
\end{proof}

\begin{lemma}\label{lem: reduce non sep to smaller norm}
Let $\Omega^\pm$ be a $\Lambda$-lift, and let $\psi$ be a pair in (\ref{equ: set of pairs}). Assume that there exists an element $(i^\prime,j_{\psi})\in \mathbf{I}_{\Omega^+\sqcup\Omega^-}\cup \mathbf{I}_{\Omega^+\sqcup\Omega^-}^\prime$ such that
\begin{itemize}
\item $(i^\prime,j_\psi)$ does not lie in the $\Lambda^\square$-interval of $\Omega^\pm$ containing $\Omega_\psi$;
\item there exist $1\leq s\leq d_\psi$ and $1\leq e\leq e_{\psi,s}$ satisfying either $i^\prime,i_\psi^{s,e}\in[m]_{\xi}$ for some $1\leq m\leq r_\xi$ or $((i_\psi^{s,e},i^\prime),j_\psi)\in\widehat{\Lambda}$.
\end{itemize}
Then we have
$$F_\xi^{\Omega^\pm}\cdot F_\xi^{\Omega_{\psi,s,e}^\pm}\in \cO_{\xi,\Lambda}^{<|\Omega^\pm|}.$$
Moreover, $F_\xi^{\Omega_{\psi,s,e}^\pm}\in \cO_{\xi,\Lambda}^{<|\Omega^\pm|}$ if $\Omega_\psi$ is not $\Lambda$-exceptional or $\gamma_{\psi}<|\Omega^\pm|$.
\end{lemma}

\begin{proof}
The proof is also very similar to that of Lemma~\ref{lem: reduce non ord to smaller norm}. We set $\Omega_1^{\pm}$ as in (\ref{eq: def of Omega1}). Then we have
$|\Omega_1^\pm|=|\Omega^\pm|$ and $F_\xi^{\Omega^\pm}F_\xi^{\Omega_{\psi,s,e}^\pm}\sim F_\xi^{\Omega_1^\pm}$. If $\Omega_1^\pm$ is not a $\Lambda$-lift, then we deduce from Lemma~\ref{lem: union of lifts} that $F_\xi^{\Omega_1^\pm}\in \cO_{\xi,\Lambda}^{<|\Omega^\pm|}$. Assume now that $\Omega_1^\pm$ is a $\Lambda$-lift. As $(i^\prime,j_\psi)$ does not lie in the $\Lambda^\square$-interval of $\Omega^\pm$ containing $\Omega_\psi$, we deduce that $(i^\prime,j_{\psi})\in \mathbf{I}_{\Omega_1^+\sqcup\Omega_1^-}\cup \mathbf{I}_{\Omega_1^+\sqcup\Omega_1^-}^\prime$, and $(i^\prime,j_\psi),\,(i_\psi^{s,e},j_\psi)$ do not lie in the same $\Lambda^\square$-interval of $\Omega_1^\pm$.
If $i^\prime,i_\psi^{s,e}\in[m]_{\xi}$ for some $1\leq m\leq r_\xi$, then $\Omega_1^\pm$ is not a $\Lambda$-lift any more. If $((i_\psi^{s,e},i^\prime),j_\psi)\in\widehat{\Lambda}$, then $F_\xi^{\Omega_1^\pm}\in \cO_{\xi,\Lambda}^{<|\Omega^\pm|}$ by Lemma~\ref{lem: reduction to separated lift}, using the fact that $(i^\prime,j_\psi),\,(i_\psi^{s,e},j_\psi)$ do not lie in the same $\Lambda^\square$-interval of $\Omega_1^\pm$. Finally, by Lemma~\ref{lem: special equivalent lift} together with Lemma~\ref{lem: two equivalent pairs} it is clear that $F_\xi^{\Omega_{\psi,s,e}^\pm}\in \cO_{\xi,\Lambda}^{<|\Omega^\pm|}$ if $\Omega_\psi$ is not $\Lambda$-exceptional. The proof is thus finished.
\end{proof}

The following is a road map which summarizes the logic of the proof of Theorem~\ref{thm: reduce to constructible}. The source of each red arrow is used as an ingredient in the proof of the target. Taking Proposition~\ref{prop: maximal vers exceptional: ord} for example, the terms \emph{type \rm{I}} and \emph{$<|\cdot|$} in blue mean that all the balanced pairs $\Omega^\pm$ treated in Proposition~\ref{prop: maximal vers exceptional: ord} can be generated from balanced pairs $\Omega_0^\pm$ satisfying one of the following
\begin{itemize}
\item $|\Omega_0^\pm|<|\Omega^\pm|$;
\item $\Omega_0^\pm$ is a constructible $\Lambda$-lift of type \rm{I};
\item $\Omega_0^\pm$ can be generated from the balanced pairs treated in the lemmas or propositions that have red arrows towards Proposition~\ref{prop: maximal vers exceptional: ord}.
\end{itemize}

\begin{figure}[ht]
\label{fig:red:constructible}
\includegraphics[scale=0.30]{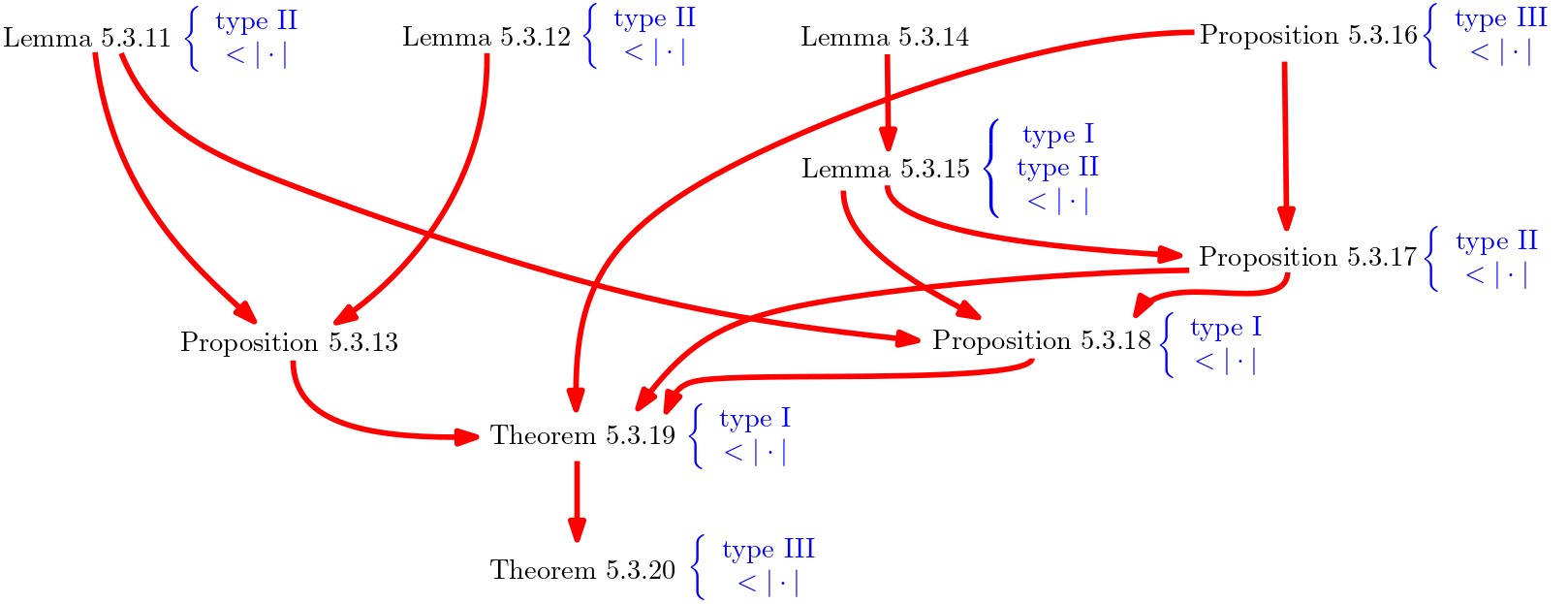}
\end{figure}%

\begin{lemma}\label{lem: special case 1}
Let $(\al,j)$ be an element of $\widehat{\Lambda}\cap\mathrm{Supp}_{\xi,\cJ}^\gamma$, and let $\Omega^\pm$ be a balanced pair such that
\begin{itemize}
\item $\Omega^+\in\mathbf{D}_{(\al,j),\Lambda}$ is $\Lambda$-exceptional and $\Lambda$-ordinary with $\Omega^+\neq \Omega_{(\al,j),\Lambda}^{\rm{max}}$;
\item $\Omega^-$ is a pseudo $\Lambda$-decomposition of $(\al,j)$;
\item $((i_{\Omega^-,1},i_\al^\prime),j)\in\Omega^-\cap\widehat{\Omega}^-$ and $u_j(i_{\Omega^-,1})<u_j(i_{\Omega^+,1})$;
\item $\Omega^-\setminus\{((i_{\Omega^-,1},i_\al^\prime),j)\}$ is a pseudo $\Lambda$-decomposition of some $((i_\al,i_\sharp),j)\in\widehat{\Lambda}$ with $u_j(i_\sharp)>u_j(i_{\Omega^+,1})$.
\end{itemize}
Then one of the following holds:
\begin{itemize}
\item $F_\xi^{\Omega^\pm}\in\cO_{\xi,\Lambda}^{<\gamma}$;
\item there exists a pseudo $\Lambda$-decomposition $\Omega'$ of $(\al,j)$
such that
      \begin{itemize}
      \item the balanced pair $\Omega^+,\Omega'$ is a constructible $\Lambda$-lift of type $\rm{II}$;
      \item $F_\xi^{\Omega_0^\pm}\in\cO_{\xi,\Lambda}^{<\gamma}$ for the balanced pair $\Omega_0^\pm$ defined by $\Omega_0^+\defeq \Omega'$ and $\Omega_0^-\defeq \Omega^-$.
      \end{itemize}
\end{itemize}
In particular, we have $F_\xi^{\Omega^\pm}\in\cO_{\xi,\Lambda}^{\rm{con}}\cdot \cO_{\xi,\Lambda}^{<\gamma}$.
\end{lemma}
\begin{proof}
We write $\psi_1=(\Omega^+,\Lambda)$ for short and note that $e_{\psi_1,1}\geq 1$ (as the pseudo $\Lambda$-decomposition $\Omega^-$ satisfying $u_j(i_{\Omega^-,1})<u_j(i_{\Omega^+,1})$ guarantees the existence of some $\Omega_0\in\mathbf{D}_{(\al,j),\Lambda}$ satisfying $u_j(i_{\Omega_0,1})<u_j(i_{\Omega^+,1})$).
Upon replacing $\Omega^-_{(\al',j)}$ with $(\Omega_{(\al',j),\Lambda}^{\rm{max}})_\dagger$ (cf.~(\ref{eq: dagger})), we may assume that $\Omega^-_{(\al',j)}=\Omega_{(\al',j),\Lambda}^{\rm{max}}$ is $\Lambda$-ordinary for each $(\al',j)\in\widehat{\Omega}^-\setminus \{((i_{\Omega^-,1},i_\al^\prime),j)\}$. We will also use the following facts without further comments
\begin{itemize}
\item for each $1\leq e<e'\leq e_{\psi_1,1}$, we have $((i_{\psi_1}^{1,e},i_{\psi_1}^{1,e'}),j)\in\widehat{\Lambda}$ and in particular there does not exist $1\leq m\leq r_\xi$ such that $i_{\psi_1}^{1,e},i_{\psi_1}^{1,e'}\in[m]_\xi$;
\item for each $1\leq e\leq e_{\psi_1,1}$, if we write $\psi$ for the pair $(\{((i_{\psi_1}^{1,e},i_\al^\prime),j)\},\Lambda)$, then exactly one of the following holds:
    \begin{itemize}
    \item $e=e_{\psi_1,1}\geq 1$ and $d_\psi=0$;
    \item $d_\psi=1$, $e_{\psi,1}=e_{\psi_1,1}-e\geq 1$ and $i_{\psi_1}^{1,e'}=i_{\psi}^{1,e'-e}$ for each $e+1\leq e'\leq e_{\psi_1,1}$.
    \end{itemize}
\end{itemize}

\textbf{Case~A:} We first consider the case when there exist $1\leq e_1\leq e_{\psi_1,1}$ and $i_\natural\in\mathbf{n}$ such that
\begin{itemize}
\item $i_\natural,i_{\psi_1}^{1,e_1}\in[m_\natural]_\xi$ for some $1\leq m_\natural\leq r_\xi$;
\item either $i_\natural=i_\sharp$ or $((i_\al,i_\natural),j),~((i_\natural,i_\sharp),j)\in\widehat{\Lambda}$.
\end{itemize}
In this case, we choose $e_1$ to be minimal possible, and then set $\Omega_\natural^+\defeq \Omega^+$ and $\Omega_\natural^-\defeq \Omega_{\sharp,\natural}\sqcup \{((i_{\psi_1}^{1,e_1},i_\al^\prime),j)\}$ where $\Omega_{\sharp,\natural}\defeq (\Omega_{((i_\al,i_\natural),j),\Lambda}^{\rm{max}})_\dagger$.
Note that $\Omega_{\sharp,\natural}$ satisfies the condition that $\Omega_{\sharp,\natural,(\al',j)}=\Omega_{(\al',j),\Lambda}^{\rm{max}}$ is $\Lambda$-ordinary for each $(\al',j)\in\widehat{\Omega}_{\sharp,\natural}$. We consider the balanced pair $\Omega_1^\pm$ defined by $\Omega_1^+\defeq \Omega_\natural^-$ and $\Omega_1^-\defeq \Omega^-$. If $i_\natural=i_\sharp$, then $\Omega_1^\pm$ is not a $\Lambda$-lift and thus $F_\xi^{\Omega_1^\pm}\in\cO_{\xi,\Lambda}^{<\gamma}$ by Lemma~\ref{lem: union of lifts}. If $((i_\natural,i_\sharp),j)\in\widehat{\Lambda}$, then we deduce $F_\xi^{\Omega_1^\pm}\in\cO_{\xi,\Lambda}^{<\gamma}$ from Lemma~\ref{lem: reduction to separated lift}. Hence we always have $F_\xi^{\Omega_1^\pm}\in\cO_{\xi,\Lambda}^{<\gamma}$ and $F_\xi^{\Omega^\pm}=F_\xi^{\Omega_\natural^\pm}F_\xi^{\Omega_1^\pm}$. Consequently, by taking $\Omega'\defeq \Omega_\natural^-$, it suffices to check the conditions in the definition of constructible $\Lambda$-lifts of type $\rm{II}$ for the balanced pair $\Omega_\natural^\pm$. If $\Omega_\natural^\pm$ is not a $\Lambda$-lift, then $F_\xi^{\Omega_\natural^\pm}\in \cO_{\xi,\Lambda}^{<\gamma}$ thanks to Lemma~\ref{lem: union of lifts}. If $\Omega_\natural^\pm$ is a $\Lambda$-lift, then Conditions~\rm{II}-\ref{it: II 1}, \rm{II}-\ref{it: II 2}, \rm{II}-\ref{it: II 3}, and \rm{II}-\ref{it: II 9} are true by our assumption on $\Omega_\natural^\pm$. Conditions~\rm{II}-\ref{it: II 10} and \rm{II}-\ref{it: II 11} hold for $\Omega_\natural^\pm$ as $u_j(i_\natural)\geq u_j(i_\sharp)>u_j(i_{\Omega^+,1})$ and $u_j(i_{\psi_1}^{1,e_1})<u_j(i_{\Omega^+,1})$. If $\Omega_\natural^\pm$ fails Condition~\rm{II}-\ref{it: II 6}, then we deduce $F_\xi^{\Omega_\natural^\pm}\in \cO_{\xi,\Lambda}^{<\gamma}$ from Lemma~\ref{lem: reduction to separated lift}. If $\Omega_\natural^\pm$ fails either Condition~\rm{II}-\ref{it: II 4} or Condition~\rm{II}-\ref{it: II 5}, there exist $1\leq e_1'\leq e_1-1$ and $i_\natural'\in\mathbf{n}$ such that
\begin{itemize}
\item $i_\natural',i_{\psi_1}^{1,e_1'}\in[m_\natural']_\xi$ for some $1\leq m_\natural'\leq r_\xi$;
\item $((i_\al,i_\natural'),j),~((i_\natural',i_\sharp),j)\in\widehat{\Lambda}$,
\end{itemize}
which clearly contradicts the minimality of the choice of $e_1$. Condition~\rm{II}-\ref{it: II 7} holds as we have $$u_j(i_{\psi_1}^{1,e})\leq u_j(i_{\psi_1}^{1,1})<u_j(i_{\Omega^+,1})<u_j(i_\sharp)\leq u_j(i_\natural)\leq u_j(i_{\Omega_{\natural,(\al',j)},c})$$ for each $1\leq e\leq e_{\psi_1,1}$, $(\al',j)\in\widehat{\Omega}_\natural\setminus\{((i_{\psi_1}^{1,e_1},i_\al^\prime),j)\}$ and $0\leq c\leq \#\Omega_{\natural,(\al',j)}$. If $\Omega_\natural^\pm$ fails Condition~\rm{II}-\ref{it: II 8}, then we deduce $F_\xi^{\Omega_\natural^\pm}\in \cO_{\xi,\Lambda}^{<\gamma}$ from Lemma~\ref{lem: reduce non sep to smaller norm} (using the fact $\gamma_\psi<\gamma$ for each $\psi\in \{(\Omega_{(\al',j)}^-,\Lambda)\mid (\al',j)\in\widehat{\Omega}^-\}$). If $\Omega_\natural^\pm$ satisfies all the conditions from Condition~\rm{II}-\ref{it: II 1} to Condition~\rm{II}-\ref{it: II 8}, then it is clearly a constructible $\Lambda$-lift of type $\rm{II}$.

\textbf{Case~B:} If $e_1$ does not exist but there exist $1\leq e_1^\flat\leq e_{\psi_1,1}$ and $i_\flat\in\mathbf{n}$ such that
\begin{itemize}
\item $i_\flat,i_{\psi_1}^{1,e_1^\flat}\in[m_\flat]_\xi$ for some $1\leq m_\flat\leq r_\xi$;
\item $((i_{\Omega^-,1},i_\flat),j),~((i_\flat,i_\al^\prime),j)\in\widehat{\Lambda}$;
\item $u_j(i_{\psi_1}^{1,e_1^\flat})>u_j(i_{\Omega^-,1})$,
\end{itemize}
we choose $e_1^\flat$ to be minimal possible, and then set $\Omega_\flat^+\defeq \Omega^+$ and
$$\Omega_\flat^-\defeq \Omega_\sharp\sqcup\Omega_{\flat,\natural}\sqcup \{((i_{\psi_1}^{1,e_1^\flat},i_\al^\prime),j)\}$$ where $\Omega_\sharp\defeq \Omega^-\setminus\{((i_{\Omega^-,1},i_\al^\prime),j)\}$ and $\Omega_{\flat,\natural}\defeq (\Omega_{((i_{\Omega^-,1},i_\flat),j),\Lambda}^{\rm{max}})_\dagger$.
Note that $\Omega_{\flat,\natural}$ satisfies the condition that $\Omega_{\flat,\natural,(\al',j)}=\Omega_{(\al',j),\Lambda}^{\rm{max}}$ is $\Lambda$-ordinary for each $(\al',j)\in\widehat{\Omega}_{\flat,\natural}$. We consider the balanced pair $\Omega_2^\pm$ defined by $\Omega_2^+\defeq \Omega_\flat^-$ and $\Omega_2^-\defeq \Omega^-$. Then we clearly have $F_\xi^{\Omega^\pm}=F_\xi^{\Omega_\flat^\pm}F_\xi^{\Omega_2^\pm}$ and $F_\xi^{\Omega_2^\pm}\in\cO_{\xi,\Lambda}^{<\gamma}$ (as $\Omega_\sharp\subseteq\Omega^-\cap\Omega_\flat^-$).  Consequently, by taking $\Omega'\defeq \Omega_\flat^-$, it suffices to check the conditions in the definition of constructible $\Lambda$-lifts of type $\rm{II}$ for the balanced pair $\Omega_\flat^\pm$. If $\Omega_\flat^\pm$ is not a $\Lambda$-lift, then $F_\xi^{\Omega_\flat^\pm}\in \cO_{\xi,\Lambda}^{<\gamma}$ thanks to Lemma~\ref{lem: union of lifts}. If $\Omega_\flat^\pm$ is a $\Lambda$-lift, then Conditions~\rm{II}-\ref{it: II 1}, \rm{II}-\ref{it: II 2}, \rm{II}-\ref{it: II 3}, and \rm{II}-\ref{it: II 9} are true by our assumption on $\Omega_\flat^\pm$. Conditions~\rm{II}-\ref{it: II 10} and \rm{II}-\ref{it: II 11} hold for $\Omega_\flat^\pm$ as $u_j(i_\sharp)>u_j(i_{\Omega^+,1})$ and $u_j(i_{\Omega^-,1})<u_j(i_{\psi_1}^{1,e_1^\flat})<u_j(i_{\Omega^+,1})$. If $\Omega_\flat^\pm$ fails Condition~\rm{II}-\ref{it: II 6}, then we deduce $F_\xi^{\Omega_\flat^\pm}\in \cO_{\xi,\Lambda}^{<\gamma}$ from Lemma~\ref{lem: reduction to separated lift}. If $\Omega_\flat^\pm$ fails either Condition~\rm{II}-\ref{it: II 4} or Condition~\rm{II}-\ref{it: II 5}, there exist $1\leq e_1^{\flat,\prime}\leq e_1^\flat-1$ and $i_\flat'\in\mathbf{n}$ such that
\begin{itemize}
\item $i_\flat',i_{\psi_1}^{1,e_1^{\flat,\prime}}\in[m_\flat']_\xi$ for some $1\leq m_\flat'\leq r_\xi$;
\item $((i_{\Omega^-,1},i_\flat'),j),~((i_\flat',i_\flat),j)\in\widehat{\Lambda}$;
\item $u_j(i_{\psi_1}^{1,e_1^{\flat,\prime}})>u_j(i_{\psi_1}^{1,e_1^\flat})$ (as $((i_{\psi_1}^{1,e_1^{\flat,\prime}},i_{\psi_1}^{1,e_1^\flat}),j)\in\widehat{\Lambda}$),
\end{itemize}
which clearly contradicts the minimality of the choice of $e_1^\flat$. Condition~\rm{II}-\ref{it: II 7} holds for $\Omega_\flat^\pm$ as we have
$$u_j(i_{\Omega_{\flat,(\al^{\prime\prime},j)}^-,c^\flat})\leq u_j(i_{\Omega^-,1})<u_j(i_{\psi_1}^{1,e_1^\flat})<u_j(i_{\psi_1}^{1,e})\leq u_j(i_{\psi_1}^{1,1})<u_j(i_{\Omega^+,1})<u_j(i_\sharp)\leq u_j(i_{\Omega_{\flat,(\al',j)}^-,c})$$
for each $1\leq e\leq e_1^\flat-1$, $(\al',j)\in\widehat{\Omega}_\sharp$, $0\leq c\leq \#\Omega_{\sharp,(\al',j)}$, $(\al^{\prime\prime},j)\in\widehat{\Omega}_{\flat,\natural}$ and $0\leq c^\flat\leq \#\Omega_{\flat,\natural,(\al^{\prime\prime},j)}$. If $\Omega_\flat^\pm$ fails Condition~\rm{II}-\ref{it: II 8}, then we deduce $F_\xi^{\Omega_\flat^\pm}\in \cO_{\xi,\Lambda}^{<\gamma}$ from Lemma~\ref{lem: reduce non sep to smaller norm} (using the fact $\gamma_\psi<\gamma$ for each $\psi\in \{(\Omega_{(\al',j)}^-,\Lambda)\mid (\al',j)\in\widehat{\Omega}_\flat^-\}$). If $\Omega_\flat^\pm$ satisfies all the conditions from Condition~\rm{II}-\ref{it: II 1} to Condition~\rm{II}-\ref{it: II 8}, then it is clearly a constructible $\Lambda$-lift of type $\rm{II}$.

\textbf{Case~C:} If neither $e_1$ nor $e_1^\flat$ exists, we check the conditions in the definition of constructible $\Lambda$-lifts of type $\rm{II}$ for the balanced pair $\Omega^\pm$. If $\Omega^\pm$ is not a $\Lambda$-lift, then $F_\xi^{\Omega^\pm}\in \cO_{\xi,\Lambda}^{<\gamma}$ thanks to Lemma~\ref{lem: union of lifts}. If $\Omega^\pm$ is a $\Lambda$-lift, then Conditions~\rm{II}-\ref{it: II 1}, \rm{II}-\ref{it: II 2}, \rm{II}-\ref{it: II 3}, and \rm{II}-\ref{it: II 9} are true by our assumption on $\Omega^\pm$. Conditions~\rm{II}-\ref{it: II 10} and \rm{II}-\ref{it: II 11} hold for $\Omega^\pm$ as $u_j(i_\sharp)>u_j(i_{\Omega^+,1})$ and $u_j(i_{\Omega^-,1})<u_j(i_{\Omega^+,1})$. If $\Omega^\pm$ fails Condition~\rm{II}-\ref{it: II 6}, then we deduce $F_\xi^{\Omega^\pm}\in \cO_{\xi,\Lambda}^{<\gamma}$ from Lemma~\ref{lem: reduction to separated lift}. Conditions~\rm{II}-\ref{it: II 4} and \rm{II}-\ref{it: II 5} hold for $\Omega^\pm$ thanks to the non-existence of $e_1$ and $e_1^\flat$. Condition~\rm{II}-\ref{it: II 7} holds for $\Omega^\pm$ as we have
$$u_j(i_{\psi_1}^{1,e})\leq u_j(i_{\psi_1}^{1,1})<u_j(i_{\Omega^+,1})<u_j(i_\sharp)\leq u_j(i_{\Omega^-_{(\al',j)},c})$$
for each $1\leq e\leq e_{\psi_1,1}$, $(\al',j)\in\widehat{\Omega}_\sharp$ and $0\leq c\leq \#\Omega^-_{(\al',j)}$. If $\Omega^\pm$ fails Condition~\rm{II}-\ref{it: II 8}, then we deduce $F_\xi^{\Omega^\pm}\in \cO_{\xi,\Lambda}^{<\gamma}$ from Lemma~\ref{lem: reduce non sep to smaller norm} (using the fact $\gamma_\psi<\gamma$ for each $\psi\in \{(\Omega_{(\al',j)}^-,\Lambda)\mid (\al',j)\in\widehat{\Omega}^-\}$). Finally, if $\Omega^\pm$ satisfies all the conditions from Condition~\rm{II}-\ref{it: II 1} to Condition~\rm{II}-\ref{it: II 8}, then it is clearly a constructible $\Lambda$-lift of type $\rm{II}$. The proof is thus finished.
\end{proof}

\begin{lemma}\label{lem: special case 2}
Let $(\al,j)$ be an element of $\widehat{\Lambda}\cap\mathrm{Supp}_{\xi,\cJ}^\gamma$, and let $\Omega^\pm$ be a balanced pair such that
\begin{itemize}
\item $\Omega^+\in\mathbf{D}_{(\al,j),\Lambda}$ is $\Lambda$-exceptional and $\Lambda$-ordinary with $\Omega^+\neq \Omega_{(\al,j),\Lambda}^{\rm{max}}$;
\item $\Omega_{(\al,j),\Lambda}^{\rm{max}}$ is not $\Lambda$-ordinary and $\Omega^-=(\Omega_{(\al,j),\Lambda}^{\rm{max}})_\dagger$;
\item $u_j(i_{\Omega^-,1})>u_j(i_{\Omega^+,1})$.
\end{itemize}
Then one of the following holds:
\begin{itemize}
\item $F_\xi^{\Omega^\pm}\in\cO_{\xi,\Lambda}^{<\gamma}$;
\item there exists a pseudo $\Lambda$-decomposition $\Omega'$ of $(\al,j)$ such that
      \begin{itemize}
      \item the balanced pair $\Omega^+,\Omega'$ is a constructible $\Lambda$-lift of type $\rm{II}$;
      \item $F_\xi^{\Omega_0^\pm}\in\cO_{\xi,\Lambda}^{<\gamma}$ for the balanced pair $\Omega_0^\pm$ defined by $\Omega_0^+\defeq \Omega'$ and $\Omega_0^-\defeq \Omega^-$.
      \end{itemize}
\end{itemize}
In particular, we have $F_\xi^{\Omega^\pm}\in\cO_{\xi,\Lambda}^{\rm{con}}\cdot \cO_{\xi,\Lambda}^{<\gamma}$.
\end{lemma}
\begin{proof}
We write $\psi_1=(\Omega^+,\Lambda)$ for short. As $\Omega_{(\al,j),\Lambda}^{\rm{max}}$ is not $\Lambda$-ordinary and $\Omega^-=(\Omega_{(\al,j),\Lambda}^{\rm{max}})_\dagger$, there exist a pseudo $\Lambda$-decomposition $\Omega_\sharp$ of some $((i_\al,i_\sharp),j)\in\widehat{\Lambda}$ and $\Lambda$-decomposition $\Omega_\flat$ of some $((i_\flat,i_\al^\prime),j)\in\widehat{\Lambda}$ such that
\begin{itemize}
\item $\Omega^-=\Omega_\sharp\sqcup\Omega_\flat$;
\item $\Omega^-_{(\al',j)}=\Omega_{(\al',j),\Lambda}^{\rm{max}}$ is $\Lambda$-ordinary for each $(\al',j)\in\widehat{\Omega}^-$;
\item $u_j(i_\sharp)>u_j(i_\flat)\geq u_j(i_{\Omega_\flat,1})=u_j(i_{\Omega^-,1})>u_j(i_{\Omega^+,1})$.
\end{itemize}

\textbf{Case~A:} If there exist $1\leq e_1\leq e_{\psi_1,1}$ and $i_\natural\in\mathbf{n}$ such that
\begin{itemize}
\item $i_\natural,i_{\psi_1}^{1,e_1}\in[m_\natural]_\xi$ for some $1\leq m_\natural\leq r_\xi$;
\item exactly one of the following holds:
      \begin{itemize}
      \item $i_\natural=i_\sharp$;
      \item $((i_\al,i_\natural),j),\,((i_\natural,i_\sharp),j)\in\widehat{\Lambda}$;
      \item $((i_\flat,i_\natural),j),\,((i_\natural,i_\al'),j)\in\widehat{\Lambda}$;
      \end{itemize}
\item $u_j(i_\natural)>u_j(i_{\Omega^+,1})$,
\end{itemize}
then we choose $e_1$ to be minimal possible and set $\Omega_\natural^+\defeq \Omega^+$. If $i_\natural=i_\sharp$, we set $\Omega_\natural^-\defeq \Omega_\sharp\sqcup\{((i_{\psi_1}^{1,e_1},i_\al^\prime),j)\}$. If $((i_\al,i_\natural),j),\,((i_\natural,i_\sharp),j)\in\widehat{\Lambda}$, we set $\Omega_\natural^-\defeq \Omega_{\sharp,\natural}\sqcup \{((i_{\psi_1}^{1,e_1},i_\al^\prime),j)\}$ where $\Omega_{\sharp,\natural}\defeq (\Omega_{((i_\al,i_\natural),j),\Lambda}^{\rm{max}})_\dagger$. If $((i_\flat,i_\natural),j),\,((i_\natural,i_\al'),j)\in\widehat{\Lambda}$, we set $\Omega_\natural^-\defeq \Omega_\sharp\sqcup\Omega_{\flat,\natural}\sqcup \{((i_{\psi_1}^{1,e_1},i_\al^\prime),j)\}$ where $\Omega_{\flat,\natural}\defeq (\Omega_{((i_\flat,i_\natural),j),\Lambda}^{\rm{max}})_\dagger$. Note that $\Omega_{\sharp,\natural}$ satisfies the condition that $\Omega_{\sharp,\natural,(\al',j)}=\Omega_{(\al',j),\Lambda}^{\rm{max}}$ is $\Lambda$-ordinary for each $(\al',j)\in\widehat{\Omega}_{\sharp,\natural}$, and similarly for $\Omega_{\flat,\natural}$. It is not difficult to see that the balanced pair $\Omega_1^\pm$ defined by $\Omega_1^+\defeq \Omega_\natural^-$ and $\Omega_1^-\defeq \Omega^-$ satisfies $F_\xi^{\Omega_1^\pm}\in\cO_{\xi,\Lambda}^{<\gamma}$ and $F_\xi^{\Omega^\pm}=F_\xi^{\Omega_\natural^\pm}F_\xi^{\Omega_1^\pm}$ in all three cases above. Consequently, by taking $\Omega'\defeq \Omega_\natural^-$, it suffices to check the conditions in the definition of constructible $\Lambda$-lifts of type $\rm{II}$ for the balanced pair $\Omega_\natural^\pm$. If $\Omega_\natural^\pm$ is not a $\Lambda$-lift, then $F_\xi^{\Omega_\natural^\pm}\in \cO_{\xi,\Lambda}^{<\gamma}$ thanks to Lemma~\ref{lem: union of lifts}. If $\Omega_\natural^\pm$ is a $\Lambda$-lift, then Conditions~\rm{II}-\ref{it: II 1}, \rm{II}-\ref{it: II 2}, \rm{II}-\ref{it: II 3}, and \rm{II}-\ref{it: II 9} are true by our assumption on $\Omega_\natural^\pm$. Conditions~\rm{II}-\ref{it: II 10} and \rm{II}-\ref{it: II 11} hold for $\Omega_\natural^\pm$ as $u_j(i_\natural)>u_j(i_{\Omega^+,1})$ and $u_j(i_{\psi_1}^{1,e_1})<u_j(i_{\Omega^+,1})$. If $\Omega_\natural^\pm$ fails Condition~\rm{II}-\ref{it: II 6}, then we deduce $F_\xi^{\Omega_\natural^\pm}\in \cO_{\xi,\Lambda}^{<\gamma}$ from Lemma~\ref{lem: reduction to separated lift}. If $\Omega_\natural^\pm$ fails either Condition~\rm{II}-\ref{it: II 4} or Condition~\rm{II}-\ref{it: II 5}, there exist $1\leq e_1'\leq e_1-1$ and $i_\natural'\in\mathbf{n}$ such that
\begin{itemize}
\item $i_\natural',i_{\psi_1}^{1,e_1'}\in[m_\natural']_\xi$ for some $1\leq m_\natural'\leq r_\xi$;
\item one of the following holds:
      \begin{itemize}
      \item $i_\natural'=i_\sharp$;
      \item $((i_\al,i_\natural'),j),\,((i_\natural',i_\sharp),j)\in\widehat{\Lambda}$;
      \item $((i_\flat,i_\natural'),j),\,((i_\natural',i_\al'),j)\in\widehat{\Lambda}$;
      \end{itemize}
\item $u_j(i_\natural')>u_j(i_{\Omega^+,1})$,
\end{itemize}
which clearly contradicts the minimality of the choice of $e_1$. Condition~\rm{II}-\ref{it: II 7} holds as we have $$u_j(i_{\psi_1}^{1,e})\leq u_j(i_{\psi_1}^{1,1})<u_j(i_{\Omega^+,1})<u_j(i_\natural)\leq u_j(i_{\Omega_{\natural,(\al',j)},c})$$ for each $1\leq e\leq e_{\psi_1,1}$, $(\al',j)\in\widehat{\Omega}_\natural\setminus\{((i_{\psi_1}^{1,e_1},i_\al^\prime),j)\}$ and $0\leq c\leq \#\Omega_{\natural,(\al',j)}$. If $\Omega_\natural^\pm$ fails Condition~\rm{II}-\ref{it: II 8}, then we deduce $F_\xi^{\Omega_\natural^\pm}\in \cO_{\xi,\Lambda}^{<\gamma}$ from Lemma~\ref{lem: reduce non sep to smaller norm} (using the fact $\gamma_\psi<\gamma$ for each $\psi\in \{(\Omega_{(\al',j)}^-,\Lambda)\mid (\al',j)\in\widehat{\Omega}^-\}$). If $\Omega_\natural^\pm$ satisfies all the conditions from Condition~\rm{II}-\ref{it: II 1} to Condition~\rm{II}-\ref{it: II 8}, then it is clearly a constructible $\Lambda$-lift of type $\rm{II}$.

\textbf{Case~B:} If such $e_1$ does not exist (for example if $e_{\psi_1,1}=0$), we check the conditions in the definition of constructible $\Lambda$-lifts of type $\rm{II}$ for the balanced pair $\Omega^\pm$. If $\Omega^\pm$ is not a $\Lambda$-lift, then $F_\xi^{\Omega^\pm}\in \cO_{\xi,\Lambda}^{<\gamma}$ thanks to Lemma~\ref{lem: union of lifts}. If $\Omega^\pm$ is a $\Lambda$-lift, then Conditions~\rm{II}-\ref{it: II 1}, \rm{II}-\ref{it: II 2}, \rm{II}-\ref{it: II 3}, and \rm{II}-\ref{it: II 9} are true by our assumption on $\Omega^\pm$. Conditions~\rm{II}-\ref{it: II 10} and \rm{II}-\ref{it: II 11} hold for $\Omega^\pm$ as $u_j(i_\sharp)>u_j(i_\flat)>u_j(i_{\Omega^+,1})$, $((i_\flat,i_\al^\prime),j)\in \widehat{\Omega}^-$ and $u_j(i_{\Omega^-,1})=u_j(i_{\Omega_\flat,1})>u_j(i_{\Omega^+,1})$. If $\Omega^\pm$ fails Condition~\rm{II}-\ref{it: II 6}, then we deduce $F_\xi^{\Omega^\pm}\in \cO_{\xi,\Lambda}^{<\gamma}$ from Lemma~\ref{lem: reduction to separated lift}. Conditions~\rm{II}-\ref{it: II 4} and \rm{II}-\ref{it: II 5} hold for $\Omega^\pm$ due to the non-existence of $e_1$. Condition~\rm{II}-\ref{it: II 7} holds as we have $$u_j(i_{\psi_1}^{1,e})\leq u_j(i_{\psi_1}^{1,1})<u_j(i_{\Omega^+,1})<u_j(i_\sharp)\leq u_j(i_{\Omega^-_{(\al',j)},c})$$ for each $1\leq e\leq e_{\psi_1,1}$, $(\al',j)\in\widehat{\Omega}_\sharp$ and $0\leq c\leq \#\Omega^-_{(\al',j)}$. If $\Omega^\pm$ fails Condition~\rm{II}-\ref{it: II 8}, then we deduce $F_\xi^{\Omega^\pm}\in \cO_{\xi,\Lambda}^{<\gamma}$ from Lemma~\ref{lem: reduce non sep to smaller norm} (using the fact $\gamma_\psi<\gamma$ for each $\psi\in \{(\Omega_{(\al',j)}^-,\Lambda)\mid (\al',j)\in\widehat{\Omega}^-\}$). Finally, if $\Omega^\pm$ satisfies all the conditions from Condition~\rm{II}-\ref{it: II 1} to Condition~\rm{II}-\ref{it: II 8}, then it is clearly a constructible $\Lambda$-lift of type $\rm{II}$. The proof is thus finished.
\end{proof}

\begin{prop}\label{prop: maximal vers exceptional: non ord}
Let $(\al,j)$ be an element of $\widehat{\Lambda}\cap\mathrm{Supp}_{\xi,\cJ}^\gamma$, and $\Omega^\pm$ be a balanced pair such that
\begin{itemize}
\item $\Omega^+\in\mathbf{D}_{(\al,j),\Lambda}$ is $\Lambda$-exceptional and $\Lambda$-ordinary;
\item $\Omega_{(\al,j),\Lambda}^{\rm{max}}$ is not $\Lambda$-ordinary and $\Omega^-=(\Omega_{(\al,j),\Lambda}^{\rm{max}})_\dagger$.
\end{itemize}
Then one of the following holds:
\begin{itemize}
\item $F_\xi^{\Omega^\pm}\in\cO_{\xi,\Lambda}^{<\gamma}$;
\item there exists a pseudo $\Lambda$-decomposition $\Omega'$ of $(\al,j)$ such that
      \begin{itemize}
      \item the balanced pair $\Omega^+,\Omega'$ is a constructible $\Lambda$-lift of type $\rm{II}$;
      \item $F_\xi^{\Omega_0^\pm}\in\cO_{\xi,\Lambda}^{<\gamma}$ for the balanced pair $\Omega_0^\pm$ defined by $\Omega_0^+\defeq \Omega'$ and $\Omega_0^-\defeq \Omega^-$.
      \end{itemize}
\end{itemize}
In particular, we have $F_\xi^{\Omega^\pm}\in\cO_{\xi,\Lambda}^{\rm{con}}\cdot\cO_{\xi,\Lambda}^{<\gamma}$.
\end{prop}
\begin{proof}
It is harmless to assume that $u_j(i_{\Omega^+,1})<u_j(i_{\Omega_{(\al,j),\Lambda}^{\rm{max}},1})$ (and thus $\Omega^+\neq \Omega_{(\al,j),\Lambda}^{\rm{max}}$), otherwise neither the balanced pair $\Omega^+,\Omega_{(\al,j),\Lambda}^{\rm{max}}$ nor the balanced pair $\Omega^-,\Omega_{(\al,j),\Lambda}^{\rm{max}}$ is a $\Lambda$-lift, which implies $F_\xi^{\Omega^\pm}\in\cO_{\xi,\Lambda}^{<\gamma}$ by Lemma~\ref{lem: union of lifts}.

It follows from the definition of $\Omega^-=(\Omega_{(\al,j),\Lambda}^{\rm{max}})_\dagger$ (cf.~(\ref{eq: dagger})) that exactly one of the following holds:
\begin{itemize}
\item $u_j(i_{\Omega^-,1})>u_j(i_{\Omega^+,1})$;
\item $i_{\Omega^-,1}=i_{\Omega^+,1}$;
\item $u_j(i_{\Omega^-,1})<u_j(i_{\Omega^+,1})$ and $((i_{\Omega^-,1},i_\al^\prime),j)\in\Omega^-\cap\widehat{\Omega}^-$.
\end{itemize}
If $u_j(i_{\Omega^-,1})>u_j(i_{\Omega^+,1})$, we conclude by applying Lemma~\ref{lem: special case 2} to the balanced pair $\Omega^\pm$. If $i_{\Omega^-,1}=i_{\Omega^+,1}$, then the balanced pair $\Omega^\pm$ is not a $\Lambda$-lift, which implies $F_\xi^{\Omega^\pm}\in\cO_{\xi,\Lambda}^{<\gamma}$ by Lemma~\ref{lem: union of lifts}. If $u_j(i_{\Omega^-,1})<u_j(i_{\Omega^+,1})$ and $((i_{\Omega^-,1},i_\al^\prime),j)\in\Omega^-\cap\widehat{\Omega}^-$, then we conclude by applying Lemma~\ref{lem: special case 1} to the balanced pair $\Omega^\pm$.
The proof is thus finished.
\end{proof}

\begin{lemma}\label{lem: special pseudo decomposition}
Let $(\al,j)$ be an element of $\widehat{\Lambda}\cap\mathrm{Supp}_{\xi,\cJ}^\gamma$ and $\Omega\in\mathbf{D}_{(\al,j),\Lambda}$. Assume that $\Omega$ is not $\Lambda$-ordinary.
Then there exists a pseudo $\Lambda$-decomposition $\Omega'$ of $(\al,j)$ such that
\begin{itemize}
\item $\Omega'$ is $\Lambda$-equivalent to $\Omega$ with level $<\gamma$;
\item $\widehat{\Omega}'\neq \{(\al,j)\}$ and $\Omega'_{(\al',j)}=\Omega_{(\al',j),\Lambda}^{\rm{max}}$ is $\Lambda$-ordinary for each $(\al',j)\in\widehat{\Omega}'$;
\item $u_j(i_{\Omega',1})\geq u_j(i_{\Omega,1})$.
\end{itemize}
\end{lemma}
\begin{proof}
We argue by induction on the block $\gamma$.
As $\Omega$ is not $\Lambda$-ordinary, we consider $1\leq c_\dagger\leq \#\Omega-1$, $1\leq s_\dagger\leq d_\psi$ (with $\psi=(\Omega,\Lambda)$) and $1\leq e_\dagger\leq e_{\psi,s_\dagger}$ as defined at (\ref{eq: dagger}). %
We set $\beta\defeq (i_{\Omega,c_\dagger},i_\al^\prime)$ and
$$ \Omega_1\defeq\Omega_{(\beta,j),\Lambda}^{\rm{max}}\in\mathbf{D}_{(\beta,j),\Lambda}.$$
Note that we must have $u_j(i_{\Omega_{(\al,j),\Lambda}^{\rm{max}},1})\geq u_j(i_{\Omega_1,1})\geq u_j(i_{\Omega,1})$. We write $\gamma_1$ for the image of $(\beta,j)$ under $\widehat{\Lambda}\twoheadrightarrow\widehat{\Lambda}^\square$. It is obvious that $\gamma_1<\gamma$.

Note that $\Omega_1$ is either $\Lambda$-exceptional or $\Lambda$-extremal. If $\Omega_1$ is $\Lambda$-ordinary, we set $\Omega_1'\defeq \Omega_1$. If $\Omega_1$ is not $\Lambda$-ordinary, we may apply our inductive assumption to $\Omega_1$ (as $\gamma_1<\gamma$) and obtain a pseudo $\Lambda$-decomposition $\Omega_1'$ of $(\beta,j)$ that satisfies
\begin{itemize}
\item $\Omega_1'$ is $\Lambda$-equivalent to $\Omega_1$ with level $<\gamma_1$;
\item $\widehat{\Omega}_1'\neq \{(\beta,j)\}$ and $\Omega'_{(\al',j)}=\Omega_{(\al',j),\Lambda}^{\rm{max}}$ is $\Lambda$-ordinary for each $(\al',j)\in\widehat{\Omega}_1'$;
\item $u_j(i_{\Omega_1',1})\geq u_j(i_{\Omega_1,1})$.
\end{itemize}
Then we define
$$\Omega'\defeq (\Omega_{((i_\al,i_{\psi_1}^{s_\dagger,e_\dagger}),j),\Lambda}^{\rm{max}})_\dagger\sqcup \Omega_1'$$
We can clearly deduce from $u_j(i_{\Omega_1,1})\geq u_j(i_{\Omega,1})$, $i_{\Omega',1}=i_{\Omega_1',1}$ and $u_j(i_{\Omega_1',1})\geq u_j(i_{\Omega_1,1})$ that $u_j(i_{\Omega',1})\geq u_j(i_{\Omega,1})$, which implies that $\Omega'$ satisfies the desired conditions. In all, the proof is finished by an induction on $\gamma$.
\end{proof}

\begin{lemma}\label{lem: special case 3}
Let $(\al,j)$ be an element of $\widehat{\Lambda}\cap\mathrm{Supp}_{\xi,\cJ}^\gamma$, and let $\Omega^\pm$ be a balanced pair such that
\begin{itemize}
\item $\Omega^-=\Omega_{(\al,j),\Lambda}^{\rm{max}}$ is $\Lambda$-exceptional and $\Lambda$-ordinary with $e_{\psi_2,1}\geq 1$, where $\psi_2=(\Omega^-,\Lambda)$;
\item $\Omega^+\in\mathbf{D}_{(\al,j),\Lambda}$ and $i_{\Omega^+,1}=i_{\psi_2}^{1,e^+}$ for some $1\leq e^+\leq e_{\psi_2,1}$.
\end{itemize}
Then one of the following holds:
\begin{itemize}
\item $F_\xi^{\Omega^\pm}\in\cO_{\xi,\Lambda}^{<\gamma}$;
\item there exists $\Omega'\in\mathbf{D}_{(\al,j),\Lambda}$ which is $\Lambda$-equivalent to $\Omega^+$ with level~$<\gamma$ such that $i_{\Omega',1}=i_{\psi_2}^{1,1}$ and the balanced pair $\Omega',\Omega^-$ is a constructible $\Lambda$-lift of type $\rm{I}$;
\item there exists a pseudo $\Lambda$-decomposition $\Omega'$ of $(\al,j)$ which is $\Lambda$-equivalent to $\Omega^+$ with level~$<\gamma$ such that $i_{\Omega',1}=i_{\psi_2}^{1,1}$ and the balanced pair $\Omega^-,\Omega'$ is a constructible $\Lambda$-lift of type $\rm{II}$.
\end{itemize}
In particular, we always have $F_\xi^{\Omega^\pm}\in\cO_{\xi,\Lambda}^{\rm{con}}\cdot \cO_{\xi,\Lambda}^{<\gamma}$.
\end{lemma}
\begin{proof}
Note that $e_{\psi_2,1}\geq 1$ is the same as saying $\#\mathbf{D}_{(\al,j),\Lambda}\geq 2$ in this case. Replacing $\Omega^+$ with
$$\Omega_{((i_\al,i_{\Omega^+,1}),j),\Lambda}^{\rm{max}}\sqcup\{((i_{\Omega^+,1},i_\al^\prime),j)\}$$
if necessary, it is harmless to assume from now on that $\Omega_{((i_\al,i_{\Omega^+,1}),j),\Lambda}^{\rm{max}}\subseteq\Omega^+$. If we consider a balanced pair $\Omega_0^\pm$ with both $\Omega_0^+$ and $\Omega_0^-$ being $\Lambda$-decomposition of $(\al,j)$ satisfying $i_{\Omega_0^+,1}=i_{\psi_2}^{1,1}$ and $i_{\Omega_0^-,1}=i_{\psi_2}^{1,e^+}$, then we deduce $F_\xi^{\Omega_0^\pm}\in \cO_{\xi,\Lambda}^{<\gamma}$ from Lemma~\ref{lem: reduction to separated lift} and the fact $((i_{\psi_2}^{1,1},i_{\psi_2}^{1,e^+}),j)\in\widehat{\Lambda}$ if $e^+>1$. Consequently, we may assume in the rest of the proof that $e^+=1$.

If $\Omega^+$ is neither $\Lambda$-exceptional nor $\Lambda$-extremal, then it follows from Lemma~\ref{lem: reduce to extremal} that there exists $\Omega_\star\in\mathbf{D}_{(\al,j),\Lambda}$ such that
\begin{itemize}
\item $\Omega_\star$ is $\Lambda$-equivalent to $\Omega^+$ with level $<\gamma$;
\item $\Omega^+<\Omega_\star$;
\item $\Omega_\star$ is either $\Lambda$-exceptional or $\Lambda$-extremal.
\end{itemize}
It follows from $i_{\Omega^+,1}=i_{\psi_2}^{1,1}$, $\Omega^-=\Omega_{(\al,j),\Lambda}^{\rm{max}}$, $\Omega_{((i_\al,i_{\Omega^+,1}),j),\Lambda}^{\rm{max}}\subseteq\Omega^+$ and $\Omega^+<\Omega_\star$ that we must have $i_{\Omega_\star,1}=i_{\Omega^-,1}=i_{\Omega_{(\al,j),\Lambda}^{\rm{max}},1}$. Hence $\Omega^-$ is $\Lambda$-equivalent to $\Omega_\star$ (and thus $\Omega^+$ as well) with level $<\gamma$, which implies $F_\xi^{\Omega^\pm}\in\cO_{\xi,\Lambda}^{<\gamma}$. Consequently, we may assume that $\Omega^+$ is either $\Lambda$-exceptional or $\Lambda$-extremal in the rest of the proof.

If $\Omega^+$ is $\Lambda$-ordinary, then we check the conditions in the definition of constructible $\Lambda$-lifts of type $\rm{I}$ for the balanced pair $\Omega^\pm$. We write $\psi_1=(\Omega^+,\Lambda)$ for short. If $\Omega^\pm$ is not a $\Lambda$-lift, then $F_\xi^{\Omega^\pm}\in \cO_{\xi,\Lambda}^{<\gamma}$ thanks to Lemma~\ref{lem: union of lifts}. If $\Omega^\pm$ is a $\Lambda$-lift, then Conditions~\rm{I}-\ref{it: I 1}, \rm{I}-\ref{it: I 2}, and \rm{I}-\ref{it: I 3} are true by our assumption on $\Omega^\pm$. We can also deduce Conditions~\rm{I}-\ref{it: I 4} and \rm{I}-\ref{it: I 7} from the fact $u_j(i_{\psi_2}^{1,e})\leq u_j(i_{\psi_2}^{1,1})=u_j(i_{\Omega^+,1})<u_j(i_{\Omega^-,1})$ for each $1\leq e\leq e_{\psi_2,1}$. If $\Omega^+$ is $\Lambda$-exceptional, Condition~\rm{I}-\ref{it: I 6} holds for $\Omega^\pm$ as we have $$u_j(i_{\psi_1}^{1,e})<u_j(i_{\Omega^+,1})<u_j(i_{\Omega^-,1})\leq u_j(i_{\Omega^-,c})$$ for each $1\leq e\leq e_{\psi_1,1}$ and $1\leq c\leq \#\Omega^--1$. If $\Omega^+$ is not $\Lambda$-exceptional (and thus $\Lambda$-extremal by previous assumption) and fails Condition~\rm{I}-(\ref{it: I 6}, then we deduce $F_\xi^{\Omega^\pm}\in \cO_{\xi,\Lambda}^{<\gamma}$ from Lemma~\ref{lem: reduce non sep to smaller norm}. If $\Omega^\pm$ fails Condition~\rm{I}-\ref{it: I 5}, then $F_\xi^{\Omega^\pm}\in \cO_{\xi,\Lambda}^{<\gamma}$ thanks to Lemma~\ref{lem: reduction to separated lift}. If $\Omega^\pm$ satisfies all the conditions from \rm{I}-\ref{it: I 1}to \rm{I}-\ref{it: I 7}, then it is a constructible $\Lambda$-lift of type $\rm{I}$.

If $\Omega^+$ is not $\Lambda$-ordinary, we apply Lemma~\ref{lem: special pseudo decomposition} to $\Omega^+$ and define $\Omega_1^-$ as the pseudo $\Lambda$-decomposition of $(\al,j)$ associated with $\Omega^+$ as in Lemma~\ref{lem: special pseudo decomposition}.
In particular, $\Omega_1^-$ satisfies the following conditions
\begin{itemize}
\item $\Omega_1^-$ is $\Lambda$-equivalent to $\Omega^+$ with level $<\gamma$;
\item $\widehat{\Omega}_1^-\neq \{(\al,j)\}$ and $\Omega_{1,(\al',j)}^-=\Omega_{(\al',j),\Lambda}^{\rm{max}}$ is $\Lambda$-ordinary for each $(\al',j)\in\widehat{\Omega}_1^-$;
\item $u_j(i_{\Omega_1^-,1})\geq u_j(i_{\Omega^+,1})$.
\end{itemize}
If $u_j(i_{\Omega_1^-,1})>u_j(i_{\Omega^+,1})$, then we must have $i_{\Omega_1^-,1}=i_{\Omega^-,1}$, and thus the balanced pair $\Omega_1^-,\Omega^-$ is not a $\Lambda$-lift, which implies $F_\xi^{\Omega^\pm}\in\cO_{\xi,\Lambda}^{<\gamma}$. Hence, we may assume from now on that $i_{\Omega_1^-,1}=i_{\Omega^+,1}=i_{\psi_2}^{1,1}$ and check the conditions in the definition of constructible $\Lambda$-lifts of type $\rm{II}$ for the balanced pair $\Omega_1^\pm$ defined by $\Omega_1^+\defeq \Omega^-$ and $\Omega_1^-$. If $\Omega_1^\pm$ is not a $\Lambda$-lift, then $F_\xi^{\Omega_1^\pm}\in \cO_{\xi,\Lambda}^{<\gamma}$ thanks to Lemma~\ref{lem: union of lifts}. If $\Omega_1^\pm$ is a $\Lambda$-lift, then Conditions~\rm{II}-\ref{it: II 1}, \rm{II}-\ref{it: II 2}, \rm{II}-\ref{it: II 3}, \rm{II}-\ref{it: II 9}, \rm{II}-\ref{it: II 10}, and \rm{II}-\ref{it: II 11} are true by our assumption on $\Omega_1^\pm$.
If $\Omega_1^\pm$ fails Condition~\rm{II}-\ref{it: II 6}, then we deduce $F_\xi^{\Omega_1^\pm}\in \cO_{\xi,\Lambda}^{<\gamma}$ from Lemma~\ref{lem: reduction to separated lift}. Conditions~\rm{II}-\ref{it: II 4}, \rm{II}-\ref{it: II 5} and \rm{II}-\ref{it: II 7} hold for $\Omega_1^\pm$ as we have $$u_j(i_{\psi_2}^{1,e})\leq u_j(i_{\psi_2}^{1,1})=u_j(i_{\Omega^+,1})=u_j(i_{\Omega_1^-,1})<u_j(i_{\Omega^-,1})=u_j(i_{\Omega_1^+,1})$$ for each $1\leq e\leq e_{\psi_2,1}$ (modulo difference on notation between Definition~\ref{def: constructible lifts} and this proof).
If $\Omega_1^\pm$ fails Condition~\rm{II}-\ref{it: II 8}, then we deduce $F_\xi^{\Omega_1^\pm}\in \cO_{\xi,\Lambda}^{<\gamma}$ from Lemma~\ref{lem: reduce non sep to smaller norm} (using the fact $\gamma_\psi<\gamma$ for each $\psi\in \{(\Omega_{1,(\al',j)}^-,\Lambda)\mid (\al',j)\in\widehat{\Omega}_1^-\}$). Finally, if $\Omega_1^\pm$ satisfies all the conditions from Condition~\rm{II}-\ref{it: II 1} to Condition~\rm{II}-\ref{it: II 8}, then it is clearly a constructible $\Lambda$-lift of type $\rm{II}$. The proof is thus finished.
\end{proof}

\begin{prop}\label{prop: special type III}
Let $(\al,j)$ be an element of $\widehat{\Lambda}\cap\mathrm{Supp}_{\xi,\cJ}^\gamma$ and $\Omega^\pm$ be a balanced pair such that both $\Omega^+$ and $\Omega^-$ are pseudo $\Lambda$-decompositions of $(\al,j)$ satisfying $\widehat{\Omega}^+\neq \{(\al,j)\}\neq \widehat{\Omega}^-$.
Then one of the following holds:
\begin{itemize}
\item $F_\xi^{\Omega^\pm}\in\cO_{\xi,\Lambda}^{<\gamma}$;
\item there exists a constructible $\Lambda$-lift $\Omega_0^\pm$ of type $\rm{III}$ such that
      \begin{itemize}
      \item both $\Omega_0^+$ and $\Omega_0^-$ are pseudo $\Lambda$-decompositions of $(\al,j)$;
      \item $F_\xi^{\Omega_0^\pm}(F_\xi^{\Omega^\pm})^{-1}\in\cO_{\xi,\Lambda}^{<\gamma}$.
      \end{itemize}
\end{itemize}
In particular, we have $F_\xi^{\Omega^\pm}\in\cO_{\xi,\Lambda}^{\rm{con}}\cdot \cO_{\xi,\Lambda}^{<\gamma}$.
\end{prop}
\begin{proof}
First of all, as $\widehat{\Omega}^+\neq \{(\al,j)\}\neq \widehat{\Omega}^-$, we observe that $\Omega^+$ (resp.~$\Omega^-$) is clearly $\Lambda$-equivalent to $\Omega_0^+\defeq\bigsqcup_{(\al',j)\in\widehat{\Omega}^+}(\Omega_{(\al',j),\Lambda}^{\rm{max}})_\dagger$ (resp.~$\Omega_0^-\defeq \bigsqcup_{(\al',j)\in\widehat{\Omega}^+}(\Omega_{(\al',j),\Lambda}^{\rm{max}})_\dagger$) with level $<\gamma$ (using Lemma~\ref{lem: reduce to ordinary}). Hence $\Omega^+_{0,(\al',j)}=\Omega_{(\al',j),\Lambda}^{\rm{max}}$ (resp.~$\Omega^-_{0,(\al',j)}=\Omega_{(\al',j),\Lambda}^{\rm{max}}$) is $\Lambda$-ordinary for each $(\al',j)\in\widehat{\Omega}_0^+$ (resp.~for each $(\al',j)\in\widehat{\Omega}_0^-$).

We check the definition of constructible $\Lambda$-lifts of type $\rm{III}$ for the balanced pair $\Omega_0^\pm$. If $\Omega_0^\pm$ is not a $\Lambda$-lift, then $F_\xi^{\Omega_0^\pm}\in \cO_{\xi,\Lambda}^{<\gamma}$ thanks to Lemma~\ref{lem: union of lifts}. If $\Omega_0^\pm$ is a $\Lambda$-lift, then Conditions~\rm{III}-\ref{it: III 1},~\rm{III}-\ref{it: III 2},~\rm{III}-\ref{it: III 3} \rm{III}-\ref{it: III 6} and \rm{III}-\ref{it: III 9} clearly hold. If $\Omega_0^\pm$ fails Condition~\rm{III}-\ref{it: III 4}, then we deduce $F_\xi^{\Omega_0^\pm}\in \cO_{\xi,\Lambda}^{<\gamma}$ by Lemma~\ref{lem: reduce non ord to smaller norm}, as $\gamma_\psi<\gamma$ and $\gamma_{\psi'}<\gamma$ in this case.
If $\Omega_0^\pm$ fails Condition~\rm{III}-\ref{it: III 5} for some choice of $\psi,\psi'$ in (\ref{equ: set of pairs}), then
we deduce $F_\xi^{\Omega_0^\pm}\in \cO_{\xi,\Lambda}^{<\gamma}$ by Lemma~\ref{lem: reduce non ord to smaller norm II}.
Similarly, if $\Omega_0^\pm$ fails Condition~\rm{III}-\ref{it: III 8}, we deduce $F_\xi^{\Omega_0^\pm}\in \cO_{\xi,\Lambda}^{<\gamma}$ from Lemma~\ref{lem: reduce non sep to smaller norm}.
If $\Omega_0^\pm$ fails Condition~\rm{III}-\ref{it: III 7}, we deduce $F_\xi^{\Omega_0^\pm}\in \cO_{\xi,\Lambda}^{<\gamma}$ from Lemma~\ref{lem: reduction to separated lift}.

Finally, if $\Omega_0^\pm$ satisfies all the conditions from Condition~\rm{III}-\ref{it: III 1} to Condition~\rm{III}-\ref{it: III 9}, then it is clearly a constructible $\Lambda$-lift of type $\rm{III}$. The proof is thus finished.
\end{proof}

\begin{prop}\label{prop: special type II}
Let $(\al,j)$ be an element of $\widehat{\Lambda}\cap\mathrm{Supp}_{\xi,\cJ}^\gamma$ and $\Omega^\pm$ be a balanced pair such that
\begin{itemize}
\item $\Omega^-$ is a pseudo $\Lambda$-decomposition of $(\al,j)$ satisfying $\widehat{\Omega}^-\neq \{(\al,j)\}$;
\item $\Omega^+\in\mathbf{D}_{(\al,j),\Lambda}$ is $\Lambda$-ordinary
\item either $\Omega^+=\Omega_{(\al,j),\Lambda}^{\rm{max}}$ or $\Omega^+$ is $\Lambda$-extremal.
\end{itemize}
Then there exists a pseudo $\Lambda$-decomposition $\Omega'$ of $(\al,j)$ with $\widehat{\Omega}'\neq \{(\al,j)\}$ such that the balanced pair $\Omega_0^\pm$ defined by $\Omega_0^+\defeq \Omega^+$ and $\Omega_0^-\defeq\Omega'$ satisfies one of the following
\begin{itemize}
\item $F_\xi^{\Omega_0^\pm}\in\cO_{\xi,\Lambda}^{<\gamma}$;
\item $\Omega_0^\pm$ is a constructible $\Lambda$-lift of type $\rm{II}$;
\item there exists $\Omega''\in\mathbf{D}_{(\al,j),\Lambda}$ such that
      \begin{itemize}
      \item the balanced pair $\Omega'',\Omega^+$ satisfies the conditions in Lemma~\ref{lem: special case 3};
      \item $F_\xi^{\Omega_1^\pm}\in\cO_{\xi,\Lambda}^{<\gamma}$ for the balanced pair defined by $\Omega_1^+\defeq \Omega''$ and $\Omega_1^-\defeq \Omega_0^-$.
      \end{itemize}
\end{itemize}
In particular, we have $F_\xi^{\Omega^\pm}\in\cO_{\xi,\Lambda}^{\rm{con}}\cdot \cO_{\xi,\Lambda}^{<\gamma}$.
\end{prop}
\begin{proof}
Let $(\beta,j),(\beta',j)\in\widehat{\Omega}^-$ be the elements satisfying $i_\beta=i_\al$ and $i_\beta^\prime,i_{\beta'}\in[m]_\xi$ for some $1\leq m\leq r_\xi$. We define two integers $i_\sharp,\,i_\flat$ by $u_j(i_\sharp)=\max\{u_j(i_\beta^\prime),\,u_j(i_{\beta'})\}$, $u_j(i_\flat)=\min\{u_j(i_\beta^\prime),\,u_j(i_{\beta'})\}$, and then set
$$\Omega'\defeq (\Omega_{((i_\al,i_\sharp),j),\Lambda}^{\rm{max}})_\dagger\sqcup(\Omega_{((i_\flat,i_\al^\prime),j),\Lambda}^{\rm{max}})_\dagger.$$ Note that $u_j(i_\beta^\prime)>u_j(i_{\beta''})\geq u_j(i_{\Omega^-,1})$ for the elements $(\beta,j),\,(\beta'',j)\in\widehat{\Omega}'$ characterized by $i_\beta=i_\al$ and $i_{\beta''}^\prime=i_\al^\prime$.
We also note that $\Omega'_{(\al',j)}=\Omega_{(\al',j),\Lambda}^{\rm{max}}$ is $\Lambda$-ordinary for each $(\al',j)\in\widehat{\Omega}'$ (cf. Lemma~\ref{lem: reduce to ordinary}).

Now we check the conditions in the definition of constructible $\Lambda$-lifts of type $\rm{II}$ for the balanced pair $\Omega_0^\pm$ defined by $\Omega_0^+\defeq \Omega^+$ and $\Omega_0^-\defeq \Omega'$. If $\Omega_0^\pm$ is not a $\Lambda$-lift, then $F_\xi^{\Omega_0^\pm}\in \cO_{\xi,\Lambda}^{<\gamma}$ thanks to Lemma~\ref{lem: union of lifts}. If $\Omega_0^\pm$ is a $\Lambda$-lift, then Conditions~\rm{II}-\ref{it: II 1}, \rm{II}-\ref{it: II 2}, \rm{II}-\ref{it: II 3}, \rm{II}-\ref{it: II 9}, \rm{II}-\ref{it: II 10}, and \rm{II}-\ref{it: II 11} are true by our assumption on $\Omega_0^\pm$. If $\Omega_0^\pm$ fails Condition~\rm{II}-\ref{it: II 6}, then we deduce $F_\xi^{\Omega_0^\pm}\in \cO_{\xi,\Lambda}^{<\gamma}$ from Lemma~\ref{lem: reduction to separated lift}. If $\Omega^\pm$ fails Condition~\rm{II}-\ref{it: II 8}, then we deduce $F_\xi^{\Omega_0^\pm}\in \cO_{\xi,\Lambda}^{<\gamma}$ from Lemma~\ref{lem: reduce non sep to smaller norm} (using the fact $\gamma_\psi<\gamma$ for each $\psi\in \{(\Omega_{(\al',j)}^-,\Lambda)\mid (\al',j)\in\widehat{\Omega}^-\}$). If $\Omega_0^+$ is not $\Lambda$-exceptional and $\Omega_0^\pm$ fails Condition~\rm{II}-\ref{it: II 4} (resp. Condition~\rm{II}-\ref{it: II 5}, resp. Condition~\rm{II}-\ref{it: II 7}), then we deduce $F_\xi^{\Omega_0^\pm}\in \cO_{\xi,\Lambda}^{<\gamma}$ by Lemma~\ref{lem: reduce non ord to smaller norm} (resp. by Lemma~\ref{lem: reduce non ord to smaller norm II}, resp. by Lemma~\ref{lem: reduce non sep to smaller norm}).

We now treat Condition~\rm{II}-\ref{it: II 4}, Condition~\rm{II}-\ref{it: II 5}, and Condition~\rm{II}-\ref{it: II 7}, when $\Omega_0^+$ is $\Lambda$-exceptional and in particular $\Omega_0^+=\Omega_{(\al,j),\Lambda}^{\rm{max}}$ by our assumption.

If $\Omega_0^\pm$ fails Condition~\rm{II}-\ref{it: II 5} and $\Omega_0^+=\Omega_{(\al,j),\Lambda}^{\rm{max}}$ is $\Lambda$-exceptional, then there exists $1\leq e\leq e_{\psi_1,1}$, $(\al',j)\in\widehat{\Omega}^-$ and $1\leq m\leq r_{\xi}$ such that $i'_{\al'},i_{\psi_1}^{1,e}\in[m]_{\xi}$. We set $\Omega''\defeq \Omega_{\psi_1,1,e}\in\mathbf{D}_{(\al,j),\Lambda}$ with $i_{\Omega'',1}=i_{\psi_1}^{1,e}$ (cf. the paragraph before Lemma~\ref{lem: reduce non ord to smaller norm}) and note that the balanced pair $\Omega'',\Omega_0^+$ satisfies the conditions of Lemma~\ref{lem: special case 3}. We set $\Omega_1^+\defeq \Omega''$ and $\Omega_1^-\defeq \Omega_0^-$, and then observe that the balanced pair $\Omega_1^\pm$ is not a $\Lambda$-lift, so that we deduce $F_\xi^{\Omega_1^{\pm}}\in \cO_{\xi,\Lambda}^{<\gamma}$ by Lemma~\ref{lem: union of lifts}.

If $\Omega_0^\pm$ fails Condition~\rm{II}-\ref{it: II 7} and $\Omega_0^+=\Omega_{(\al,j),\Lambda}^{\rm{max}}$ is $\Lambda$-exceptional, then there exists $1\leq e\leq e_{\psi_1,1}$, $(\al',j)\in\widehat{\Omega}_0^-$ and $1\leq c\leq \#\Omega^-_{0,(\al',j)}-1$ such that $i_{\Omega^-_{0,(\al',j)},c}\neq i'_{\al}$ and $i_{\psi_1}^{1,e}$ satisfies either $i_{\psi_1}^{1,e}=i_{\Omega^-_{0,(\al',j)},c}$ or $((i_{\psi_1}^{1,e},i_{\Omega^-_{0,(\al',j)},c}),j)\in\widehat{\Lambda}$. We define $\Omega''$ and $\Omega_1^\pm$ the same way as in the last paragraph and note that the balanced pair $\Omega'',\Omega_0^+$ satisfies the conditions of Lemma~\ref{lem: special case 3}. If $i_{\psi_1}^{1,e}=i_{\Omega^-_{0,(\al',j)},c}$ then $\Omega_1^{\pm}$ is not a $\Lambda$-lift so that we have $F_\xi^{\Omega_1^{\pm}}\in \cO_{\xi,\Lambda}^{<\gamma}$ by Lemma~\ref{lem: union of lifts}, and if $((i_{\psi_1}^{1,e},i_{\Omega^-_{0,(\al',j)},c}),j)\in\widehat{\Lambda}$ then we have $F_\xi^{\Omega_1^{\pm}}\in\cO_{\xi,\Lambda}^{<\gamma}$ by Lemma~\ref{lem: reduction to separated lift} as $i_{\Omega^-_{0,(\al',j)},c}\neq i'_{\al}$.

If $\Omega_0^\pm$ fails Condition~\rm{II}-\ref{it: II 4} and $\Omega_0^+=\Omega_{(\al,j),\Lambda}^{\rm{max}}$ is $\Lambda$-exceptional, then there exists $(\al',j)\in\widehat{\Omega}_0^-$ (with $\psi=(\Omega^-_{0,(\al',j)},\Lambda)$), $1\leq s\leq d_\psi$, $1\leq e\leq e_{\psi_1,1}$, $1\leq e'\leq e_{\psi,s}$ and $1\leq m\leq r_\xi$ such that $i_{\psi}^{s,e'},\,i_{\psi_1}^{1,e}\in[m]_\xi$. We define $\Omega''$ and $\Omega_1^\pm$ the same way as before and note that the balanced pair $\Omega'',\Omega_0^+$ satisfies the conditions of Lemma~\ref{lem: special case 3}. Then we deduce $F_\xi^{\Omega_1^{\pm}}\in\cO_{\xi,\Lambda}^{<\gamma}$ by Lemma~\ref{lem: reduce non sep to smaller norm}

Finally, if $\Omega^\pm$ satisfies all the conditions from Condition~\rm{II}-\ref{it: II 1} to Condition~\rm{II}-\ref{it: II 8}, then it is clearly a constructible $\Lambda$-lift of type $\rm{II}$. The proof is thus finished.
\end{proof}

\begin{prop}\label{prop: maximal vers exceptional: ord}
Let $(\al,j)$ be an element of $\widehat{\Lambda}\cap\mathrm{Supp}_{\xi,\cJ}^\gamma$ and $\Omega^\pm$ be a balanced pair such that
\begin{itemize}
\item $\Omega^+\in\mathbf{D}_{(\al,j),\Lambda}$ is $\Lambda$-exceptional and $\Lambda$-ordinary;
\item $\Omega^-=\Omega_{(\al,j),\Lambda}^{\rm{max}}$ is $\Lambda$-ordinary.
\end{itemize}
Then one of the following holds:
\begin{itemize}
\item $F_\xi^{\Omega^\pm}\in\cO_{\xi,\Lambda}^{<\gamma}$;
\item $\Omega^\pm$ is a constructible $\Lambda$-lift of type $\rm{I}$;
\item there exists a pseudo $\Lambda$-decomposition $\Omega'$ of $(\al,j)$ such that
      \begin{itemize}
      \item the balanced pair $\Omega^+,\Omega'$ satisfies the conditions of Lemma~\ref{lem: special case 1};
      \item the balanced pair $\Omega^-,\Omega'$ satisfies the conditions of Proposition~\ref{prop: special type II};
      \end{itemize}
\item there exists $\Omega'\in\mathbf{D}_{(\al,j),\Lambda}$ such that
      \begin{itemize}
      \item the balanced pair $\Omega',\Omega^-$ satisfies the conditions of Lemma~\ref{lem: special case 3};%
      \item $F_\xi^{\Omega_0^\pm}\in\cO_{\xi,\Lambda}^{<\gamma}$ for the balanced pair $\Omega_0^\pm$ defined by $\Omega_0^+\defeq \Omega'$ and $\Omega_0^-\defeq \Omega^+$.
      \end{itemize}
\end{itemize}
In particular, we have $F_\xi^{\Omega^\pm}\in\cO_{\xi,\Lambda}^{\rm{con}}\cdot \cO_{\xi,\Lambda}^{<\gamma}$.
\end{prop}
\begin{proof}
We write $\psi_1\defeq (\Omega^+,\Lambda)$ and $\psi_2\defeq (\Omega^-,\Lambda)$ for short. It is harmless to assume that $u_j(i_{\Omega^+,1})<u_j(i_{\Omega^-,1})=u_j(i_{\Omega_{(\al,j),\Lambda}^{\rm{max}},1})$, since the result is clear otherwise.

We check the definition of constructible $\Lambda$-lifts of type $\rm{I}$ for the balanced pair $\Omega^\pm$. If $\Omega^\pm$ is not a $\Lambda$-lift, then $F_\xi^{\Omega^\pm}\in \cO_{\xi,\Lambda}^{<\gamma}$ thanks to Lemma~\ref{lem: union of lifts}. If $\Omega^\pm$ is a $\Lambda$-lift, then Conditions~\rm{I}-\ref{it: I 1}, \rm{I}-\ref{it: I 2}, and \rm{I}-\ref{it: I 3} are true by our assumption on $\Omega^\pm$.

If $\Omega^\pm$ fails Condition~\rm{I}-\ref{it: I 4}, there exist $1\leq s\leq d_{\psi_2}$, $1\leq e_2\leq e_{\psi_2,s}$, and $1\leq e_1\leq e_{\psi_1,1}$ such that
\begin{itemize}
\item $i_{\psi_1}^{1,e_1},\,i_{\psi_2}^{s,e_2}\in[m]_\xi$ for some $1\leq m\leq r_\xi$;
\item $u_j(i_{\psi_2}^{s,e_2})>u_j(i_{\Omega^+,1})>u_j(i_{\psi_1}^{1,e_1})$.
\end{itemize}
We set $\Omega'\defeq(\Omega_{((i_\al,i_{\psi_2}^{s,e_2}),j),\Lambda}^{\rm{max}})_\dagger\sqcup\{((i_{\psi_1}^{1,e_1},i_\al^\prime),j)\}$. Then we note that the balanced pair $\Omega^+,\Omega'$ satisfies the conditions in Lemma~\ref{lem: special case 1}, and the balanced pair $\Omega^-,\Omega'$ satisfies the conditions in Lemma~\ref{prop: special type II}.

Condition~\rm{I}-\ref{it: I 6} holds for $\Omega^\pm$ as we have $$u_j(i_{\psi_1}^{1,e})<u_j(i_{\Omega^+,1})<u_j(i_{\Omega^-,1})\leq u_j(i_{\Omega^-,c})$$ for each $1\leq e\leq e_{\psi_1,1}$ and $1\leq c\leq \#\Omega^--1$. If $\Omega^\pm$ fails Condition~\rm{I}-\ref{it: I 5}, then $F_\xi^{\Omega^\pm}\in \cO_{\xi,\Lambda}^{<\gamma}$ thanks to Lemma~\ref{lem: reduction to separated lift}. If $\Omega^\pm$ fails Condition~\rm{I}-\ref{it: I 7} and if $\Omega^-$ is not $\Lambda$-exceptional, then we deduce $F_\xi^{\Omega^\pm}\in \cO_{\xi,\Lambda}^{<\gamma}$ by Lemma~\ref{lem: reduce non sep to smaller norm}. If $\Omega^\pm$ fails Condition~\rm{I}-\ref{it: I 7} and if $\Omega^-$ is $\Lambda$-exceptional, then there exists $1\leq e\leq e_{\psi_2,1}$ such that either $i_{\psi_2}^{1,e}=i_{\Omega^+,c}$ or $((i_{\psi_2}^{1,e},i_{\Omega^+,c}),j)\in\widehat{\Lambda}$ for some $1\leq c\leq \#\Omega^+-1$. We set $\Omega_1^+\defeq \Omega^+$ and $\Omega_1^-\defeq \Omega_{\psi_2,1,e}\in\mathbf{D}_{(\al,j),\Lambda}$ (see the paragraph before Lemma~\ref{lem: reduce non ord to smaller norm}) with $i_{\Omega_1^-,1}=i_{\psi_2}^{1,e}$, and note that the balanced pair $\Omega_1^-,\Omega^-$ satisfies the conditions in Lemma~\ref{lem: special case 3}.
It remains to check that $F_\xi^{\Omega_1^\pm}\in\cO_{\xi,\Lambda}^{<\gamma}$. If $i_{\psi_2}^{1,e}=i_{\Omega^+,c}$ then $\Omega_1^\pm$ is not a $\Lambda$-lift so that we have $F_\xi^{\Omega_1^\pm}\in \cO_{\xi,\Lambda}^{<\gamma}$ by Lemma~\ref{lem: union of lifts}, and if $((i_{\psi_2}^{1,e},i_{\Omega^+,c}),j)\in\widehat{\Lambda}$ then we have $F_\xi^{\Omega_1^\pm}\in\cO_{\xi,\Lambda}^{<\gamma}$ by Lemma~\ref{lem: reduction to separated lift}.

Finally, if $\Omega^\pm$ satisfies all the conditions from \rm{I}-\ref{it: I 1} to \rm{I}-\ref{it: I 7}, then it is clearly a constructible $\Lambda$-lift of type $\rm{I}$. The last claim follows from Lemma~\ref{lem: special case 1}, Lemma~\ref{lem: special case 3} and Proposition~\ref{prop: special type II}. The proof is thus finished.
\end{proof}

\begin{thm}\label{thm: pair of decompositions}
Let $(\al,j)$ be an element of $\widehat{\Lambda}\cap\mathrm{Supp}_{\xi,\cJ}^\gamma$, and $\Omega^\pm$ be a balanced pair. Assume that both $\Omega^+$ and $\Omega^-$ are pseudo $\Lambda$-decompositions of $(\al,j)$. Then we have $F_\xi^{\Omega^\pm}\in\cO_{\xi,\Lambda}^{\rm{con}}\cdot \cO_{\xi,\Lambda}^{<\gamma}$.
\end{thm}
\begin{proof}
We may start with defining two balanced pairs $\Omega_1^\pm$, $\Omega_2^\pm$ by $\Omega_1^+\defeq \Omega^+$, $\Omega_1^-\defeq \Omega_{(\al,j),\Lambda}^{\rm{max}}$, $\Omega_2^+\defeq \Omega_{(\al,j),\Lambda}^{\rm{max}}$, $\Omega_2^-\defeq \Omega^-$, and then observe that $F_\xi^{\Omega^\pm}\sim F_\xi^{\Omega_1^\pm}F_\xi^{\Omega_2^\pm}$. Hence it suffices to prove $F_\xi^{\Omega^\pm}\in\cO_{\xi,\Lambda}^{\rm{con}}\cdot \cO_{\xi,\Lambda}^{<\gamma}$ for all balanced pairs $\Omega^\pm$ with a pseudo $\Lambda$-decomposition $\Omega^+$ of $(\al,j)$ and $\Omega^-=\Omega_{(\al,j),\Lambda}^{\rm{max}}$.

If $\widehat{\Omega}^+\neq \{(\al,j)\}$ and $\Omega^-=\Omega_{(\al,j),\Lambda}^{\rm{max}}$ is $\Lambda$-ordinary, we deduce $F_\xi^{\Omega^\pm}\in\cO_{\xi,\Lambda}^{\rm{con}}\cdot \cO_{\xi,\Lambda}^{<\gamma}$ by applying Proposition~\ref{prop: special type II} to the balanced pair $\Omega^-,\Omega^+$ (inverse of $\Omega^\pm$). If $\widehat{\Omega}^+\neq \{(\al,j)\}$ and $\Omega^-=\Omega_{(\al,j),\Lambda}^{\rm{max}}$ is not $\Lambda$-ordinary, we deduce $F_\xi^{\Omega^\pm}\in\cO_{\xi,\Lambda}^{\rm{con}}\cdot \cO_{\xi,\Lambda}^{<\gamma}$ by applying Proposition~\ref{prop: special type III} to the balanced pair $\Omega^+,\Omega^-_\dagger$. Hence we assume in the rest of the proof that $\Omega^+$ is a $\Lambda$-decomposition of $(\al,j)$. According to Lemma~\ref{lem: reduce to extremal}, it is harmless to assume that $\Omega^+$ is either $\Lambda$-exceptional or $\Lambda$-extremal. If $\Omega^+$ is not $\Lambda$-ordinary, then we consider the balanced pair $\Omega^+_\dagger,\Omega^-$ and deduce $F_\xi^{\Omega^\pm}\in\cO_{\xi,\Lambda}^{\rm{con}}\cdot \cO_{\xi,\Lambda}^{<\gamma}$ from previous discussion. If $\Omega^+$ is $\Lambda$-exceptional and $\Lambda$-ordinary, then we deduce $F_\xi^{\Omega^\pm}\in\cO_{\xi,\Lambda}^{\rm{con}}\cdot \cO_{\xi,\Lambda}^{<\gamma}$ from Proposition~\ref{prop: maximal vers exceptional: non ord} and Proposition~\ref{prop: maximal vers exceptional: ord}.
Hence, it remains to treat the case when $\Omega^+$ is $\Lambda$-extremal and $\Lambda$-ordinary and $\Omega^-=\Omega_{(\al,j),\Lambda}^{\rm{max}}$. We may also assume that $\Omega^+\neq \Omega_{(\al,j),\Lambda}^{\rm{max}}$, since the statement is trivial otherwise. If $\Omega^-=\Omega_{(\al,j),\Lambda}^{\rm{max}}$ is not $\Lambda$-ordinary, then we may deduce $F_\xi^{\Omega^\pm}\in\cO_{\xi,\Lambda}^{\rm{con}}\cdot \cO_{\xi,\Lambda}^{<\gamma}$ by applying Proposition~\ref{prop: special type II} to the balanced pair $\Omega^+,\Omega^-_\dagger$.

Therefore we can assume from now on that $\Omega^+$ is $\Lambda$-extremal and $\Lambda$-ordinary and $\Omega^-=\Omega_{(\al,j),\Lambda}^{\rm{max}}$ is $\Lambda$-ordinary. In this case, we prove $F_\xi^{\Omega^\pm}\in\cO_{\xi,\Lambda}^{\rm{con}}\cdot \cO_{\xi,\Lambda}^{<\gamma}$  by checking the definition of constructible $\Lambda$-lifts of type $\rm{I}$. If $\Omega^\pm$ is not a $\Lambda$-lift, then $F_\xi^{\Omega^\pm}\in \cO_{\xi,\Lambda}^{<\gamma}$ thanks to Lemma~\ref{lem: union of lifts}. If $\Omega^\pm$ is a $\Lambda$-lift, then Conditions~\rm{I}-\ref{it: I 1}, \rm{I}-\ref{it: I 2}, and \rm{I}-\ref{it: I 3} are true by our assumption on $\Omega^\pm$. If $\Omega^\pm$ fails Condition~\rm{I}-\ref{it: I 5}, then $F_\xi^{\Omega^\pm}\in \cO_{\xi,\Lambda}^{<\gamma}$ thanks to Lemma~\ref{lem: reduction to separated lift}. If $\Omega^\pm$ fails Condition~\rm{I}-\ref{it: I 6}, then we deduce $F_\xi^{\Omega^\pm}\in \cO_{\xi,\Lambda}^{<\gamma}$ from Lemma~\ref{lem: reduce non sep to smaller norm}, as $\Omega^+$ is not $\Lambda$-exceptional. If $\Omega^\pm$ fails Condition~\rm{I}-\ref{it: I 7}, then we deduce $F_\xi^{\Omega^\pm}\in \cO_{\xi,\Lambda}^{\rm{con}}\cdot\cO_{\xi,\Lambda}^{<\gamma}$, by the same argument as at the end of the proof of Proposition~\ref{prop: maximal vers exceptional: ord}(see the construction of $\Omega_1^\pm$ there). If $\Omega^\pm$ fails Condition~\rm{I}-\ref{it: I 4} and $\Omega^-$ is not $\Lambda$-exceptional, then we deduce $F_\xi^{\Omega^\pm}\in \cO_{\xi,\Lambda}^{<\gamma}$ from Lemma~\ref{lem: reduce non ord to smaller norm}. If $\Omega^\pm$ fails Condition~\rm{I}-\ref{it: I 4} and $\Omega^-$ is $\Lambda$-exceptional, then we may choose $\Omega_0^-\in\mathbf{D}_{(\al,j),\Lambda}$ such that $i_{\Omega_0^-,1}=i_{\psi_2}^{1,e}$ for some $1\leq e\leq e_{\psi_2,1}$ where $u_j(i_{\psi_2}^{1,e})\in\,](u_j(i_{\psi_1}^{s,e_1}),j),(u_j(i_{\psi_1}^{s,e_1}),j)]_{w_\cJ}$ for some $1\leq s\leq d_{\psi_1}$ and $1\leq e_1\leq e_{\psi_1,s}$, so that letting $\Omega_0^+\defeq \Omega^+$ it is enough to check that $F_\xi^{\Omega_0^\pm}\in \cO_{\xi,\Lambda}^{\rm{con}}\cdot\cO_{\xi,\Lambda}^{<\gamma}$ by Lemma~\ref{lem: special case 3}. But this follows immediately from Lemma~\ref{lem: reduce non sep to smaller norm} as $\Omega^+$ is not $\Lambda$-exceptional.

Finally, if $\Omega^\pm$ satisfies all the conditions from \rm{I}-\ref{it: I 1} to \rm{I}-\ref{it: I 7}, then it is a constructible $\Lambda$-lift of type $\rm{I}$ and thus $F_\xi^{\Omega^\pm}\in\cO_{\xi,\Lambda}^{\rm{con}}$.
The proof is thus finished.
\end{proof}

\begin{thm}\label{thm: reduce to constructible}
For each $\Lambda$-lift $\Omega^\pm$, we have $F_\xi^{\Omega^\pm}\in\cO_{\xi,\Lambda}^{\rm{con}}$.
\end{thm}
\begin{proof}
As usual, we can associate with $\Omega^\pm$ the sets $\widehat{\Omega}^+$, $\widehat{\Omega}^-$ and then $\Omega_{(\al,j)}^+$ (resp.~$\Omega_{(\al,j)}^-$) for each $(\al,j)\in \widehat{\Omega}^+$ (resp.~for each $(\al,j)\in \widehat{\Omega}^-$). We argue by induction on the norm $|\Omega^\pm|$ (cf.~Definition~\ref{def: balanced cond}). In other words, we only need to prove that
\begin{equation}\label{equ: main statement in induction}
F_\xi^{\Omega^\pm}\in \cO_{\xi,\Lambda}^{\rm{con}}\cdot \cO_{\xi,\Lambda}^{<|\Omega^\pm|}\subseteq\cO(\cN_{\xi,\Lambda})^\times
\end{equation}
for each $\Lambda$-lift $\Omega^\pm$. It follows from the first half of Lemma~\ref{lem: reduction to separated lift} that it suffices to prove (\ref{equ: main statement in induction}) when $\Omega^+\sqcup\Omega^-$ is $\Lambda$-separated. If $\widehat{\Omega}^+=\widehat{\Omega}^-$, then the result is covered by Theorem~\ref{thm: pair of decompositions}. Hence we assume from now on that $\Omega^\pm$ is a $\Lambda$-lift such that $\widehat{\Omega}^+\cap\widehat{\Omega}^-=\emptyset$ and $\Omega^+\sqcup\Omega^-$ is $\Lambda$-separated.

For each $(\al,j)\in\widehat{\Omega}^+\sqcup\widehat{\Omega}^-$, we consider the following pseudo $\Lambda$-decomposition $\Omega_{(\al,j),\natural}\defeq (\Omega_{(\al,j),\Lambda}^{\rm{max}})_\dagger$ of $(\al,j)$. Then we define
$$
\Omega_{(\al,j),\natural}^+\defeq\left\{
  \begin{array}{ll}
    \Omega_{(\al,j),\natural} & \hbox{if $(\al,j)\in\widehat{\Omega}^+$;} \\
    \Omega^-_{(\al,j)} & \hbox{if $(\al,j)\in\widehat{\Omega}^-$}
  \end{array}
\right.
\,\mbox{ and }\,
\Omega_{(\al,j),\natural}^-\defeq\left\{
  \begin{array}{ll}
    \Omega^+_{(\al,j)} & \hbox{if $(\al,j)\in\widehat{\Omega}^+$;} \\
    \Omega_{(\al,j),\natural} & \hbox{if $(\al,j)\in\widehat{\Omega}^-$.}
  \end{array}
\right.
$$
We also define
$$\Omega_\natural^+\defeq \underset{(\al,j)\in\widehat{\Omega}^+}{\bigsqcup}\Omega_{(\al,j),\natural}\,\,\mbox{and}\,\,\Omega_\natural^-\defeq \underset{(\al,j)\in\widehat{\Omega}^-}{\bigsqcup}\Omega_{(\al,j),\natural}.$$
Then it follows from Theorem~\ref{thm: pair of decompositions} that $F_\xi^{\Omega_{(\al,j),\natural}^\pm}\in \cO_{\xi,\Lambda}^{\rm{con}}\cdot \cO_{\xi,\Lambda}^{<|\Omega^\pm|}$ for each $(\al,j)\in \widehat{\Omega}^+\sqcup\widehat{\Omega}^-$, and therefore
\begin{equation}\label{equ: switch to a new pair}
F_\xi^{\Omega_\natural^\pm}(F_\xi^{\Omega^\pm})^{-1}\sim \prod_{(\al,j)\in\widehat{\Omega}^+\sqcup\widehat{\Omega}^-}F_\xi^{\Omega_{(\al,j),\natural}^\pm}\in \cO_{\xi,\Lambda}^{\rm{con}}\cdot \cO_{\xi,\Lambda}^{<|\Omega^\pm|}.
\end{equation}
Hence it suffices to prove that $F_\xi^{\Omega_\natural^\pm}\in \cO_{\xi,\Lambda}^{\rm{con}}\cdot \cO_{\xi,\Lambda}^{<|\Omega^\pm|}$ by checking the definition of constructible $\Lambda$-lifts of type $\rm{III}$. If both $\Omega_\natural^+$ and $\Omega_\natural^-$ are pseudo $\Lambda$-decompositions of some $(\al,j)\in\widehat{\Lambda}$, then we clearly have $F_\xi^{\Omega_\natural^\pm}\in \cO_{\xi,\Lambda}^{\rm{con}}\cdot \cO_{\xi,\Lambda}^{<|\Omega^\pm|}$ by Theorem~\ref{thm: pair of decompositions}, and thus we may assume from now on that such $(\al,j)$ does not exist. If $\Omega_\natural^\pm$ is not a $\Lambda$-lift, then we clearly have $F_\xi^{\Omega_\natural^\pm}\in \cO_{\xi,\Lambda}^{<|\Omega^\pm|}$ by Lemma~\ref{lem: union of lifts}. If $\Omega_\natural^\pm$ is a $\Lambda$-lift, then Conditions~\rm{III}-\ref{it: III 1},~\rm{III}-\ref{it: III 2}, and \rm{III}-\ref{it: III 3} clearly hold. If $\Omega_\natural^\pm$ fails Condition~\rm{III}-\ref{it: III 4}, then we deduce from Lemma~\ref{lem: reduce non ord to smaller norm}, Lemma~\ref{lem: reduce non sep to smaller norm}, and Theorem~\ref{thm: pair of decompositions} that $F_\xi^{\Omega_\natural^\pm}\in \cO_{\xi,\Lambda}^{\rm{con}}\cdot \cO_{\xi,\Lambda}^{<|\Omega^\pm|}$. If $\Omega_\natural^\pm$ fails Condition~\rm{III}-\ref{it: III 5}, then we deduce from Lemma~\ref{lem: reduce non ord to smaller norm II} and Theorem~\ref{thm: pair of decompositions} that $F_\xi^{\Omega_\natural^\pm}\in \cO_{\xi,\Lambda}^{\rm{con}}\cdot \cO_{\xi,\Lambda}^{<|\Omega^\pm|}$. If $\Omega_\natural^\pm$ fails Condition~\rm{III}-\ref{it: III 7}, then we deduce $F_\xi^{\Omega_\natural^\pm}\in \cO_{\xi,\Lambda}^{<|\Omega^\pm|}$ from Lemma~\ref{lem: reduction to separated lift}. If $\Omega_\natural^\pm$ fails Condition~\rm{III}-\ref{it: III 8}, then we deduce from Lemma~\ref{lem: reduce non sep to smaller norm} and Theorem~\ref{thm: pair of decompositions} that $F_\xi^{\Omega_\natural^\pm}\in \cO_{\xi,\Lambda}^{\rm{con}}\cdot \cO_{\xi,\Lambda}^{<|\Omega^\pm|}$. If $\Omega_\natural^\pm$ fails Condition~\rm{III}-\ref{it: III 9} for some $\Omega$, $\Omega'$, $(\al,j)$, $(\al',j')$, $(i,j)$, $(i',j')$ and $m$ there, then we choose an arbitrary $\Omega_\star\in\mathbf{D}_{(\al,j),\Lambda}$ (resp.~$\Omega_\star'\in\mathbf{D}_{(\al',j'),\Lambda}$) which has $(i,j)$ as an interior point (resp.~which satisfies $(i',j')\in\mathbf{I}_{\Omega_\star'}\cup\mathbf{I}_{\Omega_\star'}^\prime$) and construct a new balanced pair $\Omega_\star^\pm$ by replacing the $\Lambda^\square$-intervals $\Omega,\Omega'$ with $\Omega_\star,\Omega_\star'$ respectively. On the one hand, it is clear that $F_\xi^{\Omega_\natural^\pm}(F_\xi^{\Omega_\star^\pm})^{-1}\in \cO_{\xi,\Lambda}^{\rm{con}}\cdot \cO_{\xi,\Lambda}^{<|\Omega^\pm|}$ by applying Theorem~\ref{thm: pair of decompositions} to the pair $\Omega,\Omega_\star$ and the pair $\Omega',\Omega_\star'$. On the other hand, the balanced pair $\Omega_\star^\pm$ is not a $\Lambda$-lift as $(i,j),(i',j')\in[m]_\xi$ (and $\Omega,\Omega'$ are distinct $\Lambda^\square$-intervals of $\Omega_\natural^\pm$), which implies that $F_\xi^{\Omega_\star^\pm}\in \cO_{\xi,\Lambda}^{<|\Omega^\pm|}$ by Lemma~\ref{lem: union of lifts}, and thus $F_\xi^{\Omega_\natural^\pm}\in \cO_{\xi,\Lambda}^{\rm{con}}\cdot \cO_{\xi,\Lambda}^{<|\Omega^\pm|}$.

Now we assume that $\Omega_\natural^\pm$ fails Condition~\rm{III}-\ref{it: III 6} for some $(\beta,j),\,(\beta',j)\in\widehat{\Omega}_\natural^+\sqcup\widehat{\Omega}_\natural^-$ satisfying $((i_{\beta'},i_\beta^\prime),j)\in\widehat{\Lambda}$.
If $(i_{\beta'},j)$ and $(i_\beta^\prime,j)$ do not lie in the same $\Lambda^\square$-interval of $\Omega_\natural^\pm$, then we deduce $F_\xi^{\Omega_\natural^\pm}\in \cO_{\xi,\Lambda}^{<|\Omega^\pm|}$ from Lemma~\ref{lem: reduction to separated lift}. Otherwise there exists a $\Lambda^\square$-interval $\Omega_\natural$ of $\Omega_\natural^\pm$ in which both $(i_{\beta'},j)$ and $(i_\beta^\prime,j)$ lie. According to our construction of $\Omega_\natural^\pm$, there is a natural bijection between $\Lambda^\square$-intervals of $\Omega^\pm$ and $\Lambda^\square$-intervals of $\Omega_\natural^\pm$, and therefore the $\Lambda^\square$-interval $\Omega_\natural$ of $\Omega_\natural^\pm$ uniquely determines a $\Lambda^\square$-interval $\Omega$ of $\Omega^\pm$. Let $(\al,j)\in\widehat{\Omega}^+\sqcup\widehat{\Omega}^-$ (resp. $(\al',j)\in\widehat{\Omega}^+\sqcup\widehat{\Omega}^-$) with $(\beta,j)\in\widehat{\Omega}_{(\al,j),\natural}$ (resp. $(\beta',j)\in\widehat{\Omega}_{(\al',j),\natural}$), and note that we have
$$((i_{\al'},i_{\beta'}),j),\,((i_{\beta'},i_{\al'}^\prime),j),\,((i_{\al},i_\beta^\prime),j),\, ((i_\beta^\prime,i_{\al}^\prime),j)\in\widehat{\Lambda}\sqcup(\{0\}\times\{j\}).$$
If $(\al,j)=(\al',j)$, then due to the construction of $\Omega_\natural^\pm$ there is necessarily a pseudo $\Lambda$-decomposition of $((i_{\beta'},i_\beta^\prime),j)$ which is a subset of $\Omega_{(\al,j),\natural}$ (using Lemma~\ref{lem: reduce to ordinary}), and thus contradicts our assumption. Hence, we may assume that $(\al,j)\neq (\al',j)$ and $((i_{\beta'},i_\beta^\prime),j)\in\widehat{\Lambda}$. As $(\beta,j)$ is an element of $\widehat{\Omega}_{(\al,j),\natural}$, we have either $i_\beta^\prime=i_\al^\prime$ or $((i_\beta^\prime,i_\al^\prime),j)\in\widehat{\Lambda}$. Similarly, we have either $i_{\al'}=i_{\beta'}$ or $((i_{\al'},i_{\beta'}),j)\in\widehat{\Lambda}$. Consequently, we deduce that $((i_{\al'},i_\al^\prime),j)\in\widehat{\Lambda}$, which necessarily implies that $((i_{\al'},i_\al^\prime),j)\in\widehat{\Omega}^+\sqcup\widehat{\Omega}^-$ as $\Omega^+\sqcup\Omega^-$ is $\Lambda$-separated. (Be careful that $\Omega_\natural^+\sqcup\Omega_\natural^-$ is not $\Lambda$-separated in general, as a pseudo $\Lambda$-decomposition is $\Lambda$-separated if and only if it is a $\Lambda$-decomposition.) We write $\Omega'$ for an arbitrary $\Lambda$-decomposition of $((i_{\beta'},i_\beta^\prime),j)$. We may assume without loss of generality that $(\al,j),\,(\al',j)\in\widehat{\Omega}^-$ and $((i_{\al'},i_\al^\prime),j)\in\widehat{\Omega}^+$.
If $i_{\beta'}=i_{\al'}$ and $i_\beta^\prime=i_\al^\prime$, then $\Omega_{((i_{\al'},i_\al^\prime),j),\natural}\subseteq\Omega_\natural^+\sqcup\Omega_\natural^-$ is clearly a pseudo $\Lambda$-decomposition of $((i_{\beta'},i_\beta^\prime),j)=((i_{\al'},i_\al^\prime),j)$ which again contradicts our assumption.
Hence, we have either $i_{\beta'}\neq i_{\al'}$ or $i_\beta^\prime\neq i_\al^\prime$. Then $\Omega_\natural^\pm$ has exactly two $\Lambda^\square$-intervals given by $\Omega_\natural^-$ and $\Omega_\natural^+$ (with $\Omega_\natural^+$ being a pseudo $\Lambda$-decomposition of $((i_{\al'},i_\al^\prime),j)$ and $\widehat{\Omega}^+=\{((i_{\al'},i_\al^\prime),j)\}$) and there exist two balanced pairs $\Omega_\sharp^\pm$ and $\Omega_\flat^\pm$ (cf. the proof of Lemma~\ref{lem: reduction to separated lift}) such that
\begin{itemize}
\item $\Omega_\sharp^+=\Omega_\natural^+$, $\Omega_\flat^+=\Omega'\subseteq\Omega_\sharp^-$ and $(\Omega_\sharp^-\setminus\Omega')\sqcup\Omega_\flat^-=\Omega_\natural^-$;
\item both $\Omega_\sharp^+$ and $\Omega_\sharp^-$ are pseudo $\Lambda$-decompositions of $((i_{\al'},i_\al^\prime),j)$;
\item $|\Omega_\flat^\pm|<|\Omega_\natural^\pm|=|\Omega_\sharp^\pm|=|\Omega^\pm|$.
\end{itemize}
As we clearly have $F_\xi^{\Omega_\flat^\pm}\in\cO_{\xi,\Lambda}^{<|\Omega^\pm|}$ and $F_\xi^{\Omega_\sharp^\pm}\in\cO_{\xi,\Lambda}^{\rm{con}}\cdot \cO_{\xi,\Lambda}^{<|\Omega^\pm|}$ by Theorem~\ref{thm: pair of decompositions}, we deduce that
$$F_\xi^{\Omega_\natural^\pm}=F_\xi^{\Omega_\flat^\pm}\cdot F_\xi^{\Omega_\sharp^\pm}\in\cO_{\xi,\Lambda}^{\rm{con}}\cdot \cO_{\xi,\Lambda}^{<|\Omega^\pm|}.$$

Finally, if $\Omega^\pm$ satisfies all the conditions from \rm{III}-\ref{it: III 1} to \rm{III}-\ref{it: III 9}, then it is a constructible $\Lambda$-lift of type $\rm{III}$ and thus $F_\xi^{\Omega^\pm}\in\cO_{\xi,\Lambda}^{\rm{con}}$. In all, we have shown that $F_\xi^{\Omega_\natural^\pm}\in \cO_{\xi,\Lambda}^{\rm{con}}\cdot \cO_{\xi,\Lambda}^{<|\Omega^\pm|}$ in all cases, which together with (\ref{equ: switch to a new pair}) implies (\ref{equ: main statement in induction}). The proof is thus finished by an induction on $|\Omega^\pm|$.
\end{proof}

\newpage

\section{Construction of invariant functions}
\label{sec:const:inv}
We fix a choice of $w_\cJ\in\un{W}$, $\xi=(w_\cJ,u_\cJ)\in\Xi_{w_\cJ}$ and a subset $\Lambda\subseteq\mathrm{Supp}_{\xi,\cJ}$ whose image in $\mathrm{Supp}^\square_\xi$ is $\Lambda^\square$. In this section, we construct an invariant function $f_\xi^{\Omega^\pm}\in\Inv$ for each constructible $\Lambda$-lift $\Omega^\pm$. The construction when $\Omega^\pm$ is of type \rm{I}, \rm{II} or \rm{III} is done in \S\,\ref{sub: type I}, \S\,\ref{sub: type II} and \S\,\ref{sub: type III} respectively. More precisely, for each constructible $\Lambda$-lift, we will construct an element $v_\cJ^{\Omega^\pm}=(v_j^{\Omega^\pm})_{j\in\cJ}\in\un{W}$ and a subset $I_\cJ^{\Omega^\pm}\subseteq\mathbf{n}_{\cJ}$ satisfying
$I_\cJ^{\Omega^\pm}\cdot(v_\cJ^{\Omega^\pm},1)=I_\cJ^{\Omega^\pm}$ (cf. Lemma~\ref{lem: invariance of function}), and then we define the invariant function $f_\xi^{\Omega^\pm}$ by (cf.~(\ref{equ: def of inv fun}))
$$
f_\xi^{\Omega^\pm}\defeq f_{v_\cJ^{\Omega^\pm}, I_\cJ^{\Omega^\pm}}.
$$
The relation between $f_\xi^{\Omega^\pm}$ and $F_\xi^{\Omega^\pm}$ will be further explored in \S\,\ref{sec:inv:cons}.

We recall $\widehat{\Lambda}$ from the beginning of \S\,\ref{sec:comb:lifts} and write $\widehat{\Lambda}^\square$ for its image in $\mathrm{Supp}^\square_\xi$. We recall the set $\mathbf{n}_\cJ$ from (\ref{equ: general index}) and the notation $[m]_\xi$ from (\ref{equ: single block of integers}). We also recall from (\ref{equ: std action}) the right action of $\un{W}\rtimes\Z/f$ on $\mathbf{n}_\cJ$. For each pair of elements $(k_1,j_1), (k_2,j_2)\in \mathbf{n}_\cJ$ lying in the same orbit of $\langle (w_\cJ,1) \rangle\subseteq\un{W}\rtimes\Z/f$, we recall (see \S\,\ref{subsub:conj inv}) the definition of $](k_1,j_1),(k_2,j_2)]_{w_\cJ}\subseteq\mathbf{n}_\cJ$. For a $\Lambda$-lift $\Omega^\pm$, we also recall the sets $\widehat{\Omega}^+$ and $\widehat{\Omega}^-$ from Definition~\ref{def: separated condition}. Note that we have the partitions
$$\Omega^+=\bigsqcup_{(\al,j)\in\widehat{\Omega}^+}\Omega_{(\al,j)}^+\quad\mbox{ and }\quad\Omega^-=\bigsqcup_{(\al,j)\in\widehat{\Omega}^-}\Omega_{(\al,j)}^-.$$

We fix some notation which will be frequently used throughout the rest of \S\,\ref{sec:const:inv} as well as \S\,\ref{sec:inv:cons}. We fix a $\Lambda$-lift $\Omega^\pm$, and give the sets $\widehat{\Omega}^+\sqcup\widehat{\Omega}^-$ a numbering. We write $\Z/t$ for the cyclic group of order $t$ for each $t\geq 2$. If $\#\widehat{\Omega}^+=\#\widehat{\Omega}^-=1$, then we write $\widehat{\Omega}^+=\{(\al_1,j_1)\}$, $\widehat{\Omega}^-=\{(\al_2,j_2)\}$ ($(\al_1,j_1)$ and $(\al_2,j_2)$ might be equal) and then set $t\defeq 2$, $m_2\defeq h_{\al_2}=h_{\al_1}$ and $m_1\defeq \ell_{\al_2}=\ell_{\al_1}$. Otherwise, $\widehat{\Omega}^\pm$ is a $\widehat{\Lambda}$-lift of some directed loop $\widehat{\Gamma}$ inside $\mathfrak{G}_{\xi,\widehat{\Lambda}}$ that satisfies $E(\widehat{\Gamma})^+\cap E(\widehat{\Gamma})^-=\emptyset$. We set $t\defeq \#E(\widehat{\Gamma})^++\#E(\widehat{\Gamma})^-\geq 3$ and there exists a set of integers $\{m_a\mid a\in \Z/t\}\subseteq\{1,\dots,r_\xi\}$ such that either $(m_{a-1},m_a)\in E(\widehat{\Gamma})^+$ or $(m_a,m_{a-1})\in E(\widehat{\Gamma})^-$ for each $a\in\Z/t$. It is clear that $\{m_a\mid a\in \Z/t\}$ is uniquely determined up to a cyclic permutation on the index set $\Z/t$. We fix a choice of $\{m_a\mid a\in \Z/t\}$ from now on. We write $(\Z/t)^+$ (resp. $(\Z/t)^-$) for the subset of $\Z/t$ characterized by $a\in (\Z/t)^+$ (resp.~$a\in (\Z/t)^-$) if and only if $(m_{a-1},m_a)\in E(\widehat{\Gamma})^+$ (resp.~ $(m_a,m_{a-1})\in E(\widehat{\Gamma})^-$). Hence we have a decomposition $\Z/t=(\Z/t)^+\sqcup(\Z/t)^-$. We write $(\al_a,j_a)$ for the unique element of $\widehat{\Omega}^+$ (resp.~$\widehat{\Omega}^-$) whose image in $\mathrm{Supp}_\xi^\square$ is $(m_{a-1},m_a)$ (resp. $(m_a,m_{a-1})$). Then for $\bullet\in\{+,-\}$ we set
$$\Omega_a\defeq \Omega_{(\al_a,j_a)}^\bullet\quad\mbox{ and }\quad \psi_a\defeq(\Omega_{(\al_a,j_a)}^\bullet,\Lambda)$$
for each $a\in(\Z/t)^\bullet$.
For each $a\in\Z/t$, we set $$c_a\defeq \#\Omega_a\,\, \mbox{and} \,\,d_a\defeq d_{\psi_a}.$$ For each $1\leq s\leq d_a$ and each $1\leq e\leq e_{a,s}\defeq e_{\psi_a,s}$, we set
$$c_a^s\defeq c_a-c_{\psi_a}^s,\,\,i_a^{s,e}\defeq i_{\psi_a}^{s,e},\,\,\mbox{ and }\,\,k_a^{s,e}\defeq u_{j_a}(i_a^{s,e}).$$
For each $0\leq c\leq c_a$, we set
$$i_{a,c}\defeq i_{\Omega_{\psi_a},c_a-c}\,\,\mbox{ and }\,\,k_{a,c}\defeq u_{j_a}(i_{a,c}),$$
so that we have
\begin{equation}\label{equ: inequality for part a}
k_{a,0}>k_{a,1}>\cdots>k_{a,c_a-1}>k_{a,c_a}
\end{equation}
and
\begin{equation}\label{equ: inequality for part a prime}
k_{a,c_a^s-1}>k_a^{s,1}>\cdots>k_a^{s,e_{a,s}}>k_{a,c_a^s}
\end{equation}
for each $1\leq s\leq d_a$ (satisfying $e_{a,s}\geq 1$).

\subsection{Construction of type $\rm{I}$}\label{sub: type I}
In this section, we fix a constructible $\Lambda$-lift $\Omega^\pm$ of type $\rm{I}$ as in Definition~\ref{def: constructible lifts} and construct an element $v_\cJ^{\Omega^\pm}=(v_j^{\Omega^\pm})_{j\in\cJ}\in\un{W}$ as well as a subset $I_\cJ^{\Omega^\pm}\subseteq\mathbf{n}_{\cJ}$.

Following the notation at the beginning of \S\,\ref{sec:const:inv}, we have $\widehat{\Omega}^+=\{(\al_1,j_1)\}=\{(\al_2,j_2)\}=\widehat{\Omega}^-$, $\Omega^-=\Omega_2=\Omega_{(\al_1,j_1),\Lambda}^{\rm{max}}$, and $\Omega^+=\Omega_1$ is a $\Lambda$-decomposition of $(\al_1,j_1)$ which is either $\Lambda$-exceptional or $\Lambda$-extremal. In particular, we have $t=2$, $k_{2,0}=k_{1,0}$ and $k_{2,c_2}=k_{1,c_1}$. As $\Omega^+\neq\Omega^-$ and we clearly have $\#\mathbf{D}_{(\al_1,j_1),\Lambda}\geq 2$, we deduce that $d_1,d_2\geq 1$. It follows from $\Omega^-=\Omega_{(\al_1,j_1),\Lambda}^{\rm{max}}$ that $k_{2,c_2-1}>k_{1,c_1-1}$, $e_{2,1}\geq 1$ (namely $k_2^{1,1}$ is defined) and $k_2^{1,1}\geq k_{1,c_1-1}$. We set
$$e_{\sharp,2}\defeq\max\{e\mid 1\leq e\leq e_{2,d_2}\mbox{ and } k_2^{d_2,e}>k_{1,c_1-1}\},$$
and if such a $e_{\sharp,2}$ does not exist (i.e. $k_2^{d_2,e}\leq k_{1,c_1-1}$ for all $1\leq e\leq e_{2,d_2}$) then we set $e_{\sharp,2}\defeq 0$.
Hence, the following set (which is empty if $d_2=1$ and $e_{\sharp,2}=0$)
$$\{k_2^{1,1},\dots,k_2^{1,e_{2,1}},\cdots,k_2^{d_2-1,1},\cdots,k_2^{d_2-1,e_{2,d_2-1}},k_2^{d_2,1},\dots,k_2^{d_2,e_{\sharp,2}}\}$$
exhausts all possible $k_2^{s,e}$ between $k_{1,0}$ and $k_{1,c_1-1}$. Thanks to (\ref{equ: pair of disjoint orbits}), we define
$$I_\cJ^{\Omega^\pm,\sharp}\defeq \bigsqcup_{e=1}^{e_{\sharp,2}}](k_2^{d_2,e},j_1),(k_2^{d_2,e},j_1)]_{w_\cJ}\sqcup\bigsqcup_{s=1}^{d_2-1}\bigsqcup_{e=1}^{e_{2,s}}](k_2^{s,e},j_1),(k_2^{s,e},j_1)]_{w_\cJ}.$$
Note that we understand $I_\cJ^{\Omega^\pm,\sharp}$ to be $\emptyset$ if $d_2=1$ and $e_{\sharp,2}=0$.

We are now ready to define $v_\cJ^{\Omega^\pm}$ and $I_\cJ^{\Omega^\pm}$. Our definition of $v_\cJ^{\Omega^\pm}=(v_j^{\Omega^\pm})_{j\in\cJ}$ is always of the form
$$
v_j^{\Omega^\pm}\defeq\left\{
  \begin{array}{ll}
    v_{j}^{\Omega^\pm,\sharp}v_{j}^{\Omega^\pm,\flat}w_{j} & \hbox{if $j=j_1$;} \\
    w_j & \hbox{otherwise}
  \end{array}
\right.
$$
with $v_{j_1}^{\Omega^\pm,\sharp}$ and $v_{j_1}^{\Omega^\pm,\flat}$ to be defined below. The construction of $v_{j_1}^{\Omega^\pm,\sharp}$ and $v_{j_1}^{\Omega^\pm,\flat}$ is visualized in Figure~\ref{fig:ex-TypeI}.

If $\Omega^+$ is $\Lambda$-exceptional (and thus $d_1=1$ and $c_1^1=c_1$, as $d_1\geq 1$), then we have either $e_{1,1}=0$ (namely $k_1^{1,1}$ is not defined) or $e_{1,1}\geq 1$ and $k_2^{1,1}\geq k_{1,c_1-1}>k_1^{1,1}$. If $\Omega^+$ is $\Lambda$-extremal (and thus $d_1\geq 1$ and $c_1^1<c_1$), then we have $e_{1,s}\geq 1$ for each $1\leq s\leq d_1$ and $k_1^{1,1}>k_{1,c_1^1}\geq k_{1,c_1-1}$, and moreover $k_2^{1,1}\neq k_1^{1,1}$ thanks to Condition~$\rm{I}$-\ref{it: I 4}. Consequently, if $k_1^{1,1}$ is defined ($e_{1,1}\geq 1$), we always have $k_1^{1,1}\neq k_2^{1,1}$. If $e_{1,1}\geq 1$ and $k_2^{1,1}<k_1^{1,1}$, we define
$$v_{j_1}^{\Omega^\pm,\sharp}\defeq (k_{2,c_2},k_{2,c_2-1},\dots,k_{2,1},k_1^{1,1},\dots, k_1^{1,e_{1,1}},\dots,k_1^{d_1,1},\dots, k_1^{d_1,e_{1,d_1}}),$$
$$v_{j_1}^{\Omega^\pm,\flat}\defeq (k_{1,c_1-1},\dots,k_{1,1},k_{1,0},k_2^{1,1},\dots, k_2^{1,e_{2,1}},\dots, k_2^{d_2-1,e_{2,d_2-1}},k_2^{d_2,1},\dots, k_2^{d_2,e_{\sharp,2}})$$
and
$$I_\cJ^{\Omega^\pm}\defeq ](k_{1,0},j_1),(k_{1,0},j_1)]_{w_\cJ}\cup I_\cJ^{\Omega^\pm,\sharp}\cup I_\cJ^{\psi_1,+}.$$
(Recall that $I_\cJ^{\psi_1,+}$ is defined in \eqref{equ: interior interval}.)
If either $e_{1,1}=0$ or $k_2^{1,1}>k_1^{1,1}$, we define
$$v_{j_1}^{\Omega^\pm,\sharp}\defeq (k_{2,c_2},k_{2,c_2-1},\dots,k_{2,1},k_{1,0},k_1^{1,1},\dots, k_1^{1,e_{1,1}},\dots,k_1^{d_1,1},\dots, k_1^{d_1,e_{1,d_1}}),$$
$$v_{j_1}^{\Omega^\pm,\flat}\defeq(k_{1,c_1-1},\dots,k_{1,1},k_2^{1,1},\dots, k_2^{1,e_{2,1}},\dots,k_2^{d_2-1,e_{2,d_2-1}},k_2^{d_2,1},\dots, k_2^{d_2,e_{\sharp,2}})$$
and
$$I_\cJ^{\Omega^\pm}\defeq I_\cJ^{\Omega^\pm,\sharp}\cup I_\cJ^{\psi_1,+}.$$
It is easy to see that $v_{j_1}^{\Omega^\pm,\sharp}$ (resp. $v_{j_1}^{\Omega^\pm,\flat}$) is well-defined in $W$ due to the Condition~$\rm{I}$-\ref{it: I 6} (resp. Condition~$\rm{I}$-\ref{it: I 7} and the definition of $e_{\sharp,2}$). In particular, we have
$v_{j_1}^{\Omega^\pm,\sharp}=(k_{2,c_2},\dots,k_{2,1},k_{1,0})$
if $d_1=1$ and $e_{1,1}=0$.

\subsection{Construction of type $\rm{II}$}\label{sub: type II}
In this section, we fix a constructible $\Lambda$-lift $\Omega^\pm$ of type $\rm{II}$ as in Definition~\ref{def: constructible lifts} and construct an element $v_\cJ^{\Omega^\pm}=(v_j^{\Omega^\pm})_{j\in\cJ}\in\un{W}$ as well as a subset $I_\cJ^{\Omega^\pm}\subseteq\mathbf{n}_{\cJ}$.

Following the notation at the beginning of \S\,\ref{sec:const:inv}, we have $\widehat{\Omega}^+=\{(\al_1,j_1)\}$, $\widehat{\Omega}^-=\{(\al_a,j_a)\mid 2\leq a\leq t\}$ with $j_a=j_1$ for each $2\leq a\leq t$. Moreover, we have $\Omega_a=\Omega_{(\al_a,j_a),\Lambda}^{\rm{max}}$ for each $2\leq a\leq t$, and that $\Omega^+$ is a $\Lambda$-decomposition of $(\al_1,j_1)$ which is either $\Lambda$-exceptional or $\Lambda$-extremal. As $\widehat{\Omega}^+\cap\widehat{\Omega}^-=\emptyset$ and $\Omega^-$ is a pseudo $\Lambda$-decomposition of $(\al_1,j_1)$, we have $t\geq 3$, $k_{1,0}=k_{t,0}$, $k_{1,c_1}=k_{2,c_2}$. As we clearly have $\#\mathbf{D}_{(\al_1,j_1),\Lambda}\geq 2$ (namely $d_1\geq 1$), we deduce that $\Omega^+$ is $\Lambda$-exceptional if and only if $d_1=1$ and $c_1^1=c_1$.

If $k_{2,c_2-1}<k_{1,c_1-1}$ (which implies $e_{1,1}\geq 1$), we set
$$e_{\sharp,1}\defeq\max\{e\mid 1\leq e\leq e_{1,d_1}\mbox{ and } k_1^{d_1,e}>k_{2,c_2-1}\},$$
and if such a $e_{\sharp,1}$ does not exist (i.e. $k_1^{d_1,e}\leq k_{2,c_2-1}$ for all $1\leq e\leq e_{1,d_1}$) then we set $e_{\sharp,1}\defeq 0$.
Hence the following set (which is empty if $d_1=1$ and $e_{\sharp,1}=0$)
$$\{k_1^{1,1},\dots,k_1^{1,e_{1,1}},\cdots,k_1^{d_1-1,1},\cdots,k_1^{d_1-1,e_{1,d_1-1}},k_1^{d_1,1},\dots,k_1^{d_1,e_{\sharp,1}}\},$$
exhausts all possible $k_1^{s,e}$ between $k_{1,0}$ and $k_{2,c_2-1}$.

If $k_{2,c_2-1}>k_{1,c_1-1}$, we set
$$e_{\sharp,2}\defeq\max\{e\mid 1\leq e\leq e_{2,d_2}\mbox{ and } k_2^{d_2,e}>k_{1,c_1-1}\},$$
and if such a $e_{\sharp,2}$ does not exist (i.e. $k_2^{d_2,e}\leq k_{1,c_1-1}$ for all $1\leq e\leq e_{2,d_2}$) then we set $e_{\sharp,2}\defeq 0$.
Hence the following set (which is empty if $d_2=1$ and $e_{\sharp,2}=0$)
$$\{k_2^{1,1},\dots,k_2^{1,e_{2,1}},\cdots,k_2^{d_2-1,1},\cdots,k_2^{d_2-1,e_{2,d_2-1}},k_2^{d_2,1},\dots,k_2^{d_2,e_{\sharp,2}}\},$$
exhausts all possible $k_2^{s,e}$ between $k_{2,0}$ and $k_{1,c_1-1}$. Thanks to (\ref{equ: pair of disjoint orbits}), we define
$$I_\cJ^{\Omega^\pm,\sharp,2}\defeq \bigsqcup_{e=1}^{e_{\sharp,2}}](k_2^{d_2,e},j_1),(k_2^{d_2,e},j_1)]_{w_\cJ}\sqcup\bigsqcup_{s=1}^{d_2-1}\bigsqcup_{e=1}^{e_{2,s}}](k_2^{s,e},j_1),(k_2^{s,e},j_1)]_{w_\cJ}.$$

We are now ready to define $v_\cJ^{\Omega^\pm}$ and $I_\cJ^{\Omega^\pm}$. Our definition of $v_\cJ^{\Omega^\pm}=(v_j^{\Omega^\pm})_{j\in\cJ}$ is always of the form
$$
v_j^{\Omega^\pm}\defeq\left\{
  \begin{array}{ll}
    v_{j}^{\Omega^\pm,\sharp}v_{j}^{\Omega^\pm,\flat}w_{j} & \hbox{if $j=j_1$;} \\
    w_j & \hbox{otherwise}
  \end{array}
\right.
$$
with $v_{j_1}^{\Omega^\pm,\sharp}$ and $v_{j_1}^{\Omega^\pm,\flat}$ to be defined below. The construction of $v_{j_1}^{\Omega^\pm,\sharp}$ is visualized in Figure~\ref{fig:ex-TypeII}.

For each $3\leq a\leq t-1$, we set
$$
v_{j_1}^{\Omega^\pm,a}\defeq
\left\{
  \begin{array}{ll}
    (k_{a,c_a},k_{a,c_a-1},\dots,k_{a,1},k_{a,0},k_a^{1,1},\cdots,k_a^{d_a,e_{a,d_a}}) & \hbox{if $d_a\geq 1$;} \\
    (k_{a,c_a},k_{a,c_a-1},\dots,k_{a,1},k_{a,0}) & \hbox{if $d_a=0$.}
  \end{array}
\right.
$$
(Note that this is well-defined as $\Omega_a=\Omega_{(\al_a,j_a),\Lambda}^{\rm{max}}$ is $\Lambda$-ordinary.)
Then we observe that, since $\Omega^-$ is a pseudo $\Lambda$-decomposition of $(\al,j)$, $v_{j_1}^{\Omega^\pm,a}$ clearly commutes with each other for different $3\leq a\leq t-1$, and thus we can define
$$v_{j_1}^{\Omega^\pm,\flat}\defeq
\prod_{a=3}^{t-1}v_{j_1}^{\Omega^\pm,a}.$$
We also define
$$I_\cJ^{\Omega^\pm,\flat}\defeq \bigsqcup_{a=2}^{t-1}](k_{a,0},j_1),(k_{a+1,c_{a+1}},j_1)]_{w_\cJ}$$
and note that the sets in the union are disjoint as $\Omega^-$ is a pseudo $\Lambda$-decomposition of $(\al,j)$.

As $\Omega^\pm$ is a $\Lambda$-lift, we always have $k_{2,c_2-1}\neq k_{1,c_1-1}$. We define
$$k_t^\prime\defeq
\left\{
  \begin{array}{ll}
    k_{t,c_t} & \hbox{if $d_t=0$;} \\
    k_t^{1,1} & \hbox{if $d_t\geq 1$.}
  \end{array}
\right.
$$
Note that if $e_{1,1}\geq 1$ and $k_{2,c_2-1}<k_{1,c_1-1}$,
then we have $k_{2,c_2-1}\leq k_1^{1,1}$. Now we claim that if $e_{1,1}\geq 1$, then $k_t^\prime\neq k_1^{1,1}$.
Indeed, if $k_{2,c_2-1}>k_{1,c_1-1}$, then we deduce $k_t^\prime\neq k_1^{1,1}$ from Condition~$\rm{II}$-\ref{it: II 4} and~$\rm{II}$-\ref{it: II 7}. If $k_{2,c_2-1}<k_{1,c_1-1}$ and $k_{2,c_2-1}<k_1^{1,1}$, then we deduce $k_t^\prime\neq k_1^{1,1}$ from Condition~$\rm{II}$-\ref{it: II 4} and~$\rm{II}$-\ref{it: II 7}. %
If $k_{2,c_2-1}<k_{1,c_1-1}$ and $k_{2,c_2-1}=k_1^{1,1}$, then we deduce $k_t^\prime\neq k_1^{1,1}=k_{2,c_2-1}$ from the fact that $\Omega^-$ is a pseudo $\Lambda$-decomposition of $(\al_1,j_1)$ satisfying $\#\widehat{\Omega}^-\geq 2$.

Now we are ready to define $v_{j_1}^{\Omega^\pm,\sharp}$ and $I_\cJ^{\Omega^\pm}$.
If $k_{2,c_2-1}<k_{1,c_1-1}$ and $k_t^\prime>k_1^{1,1}$, we define
\begin{multline*}
v_{j_1}^{\Omega^\pm,\sharp}\defeq
(k_{2,c_2-1},\dots,k_{2,1},k_{2,0},k_2^{1,1},\dots,k_2^{d_2,e_{2,d_2}},k_{1,c_1},k_{1,c_1-1},\dots,k_{1,1},\\
k_t^{1,1},\dots,k_t^{d_t,e_{t,d_t}},k_{t,c_t},k_{t,c_t-1},\dots,k_{t,1},k_{1,0},k_1^{1,1},\dots, k_1^{d_1,e_{\sharp,1}})
\end{multline*}
and
$$I_\cJ^{\Omega^\pm}\defeq I_\cJ^{\psi_1,+}\cup \bigcup_{a=2}^tI_\cJ^{\psi_a,-}\cup I_\cJ^{\Omega^\pm,\flat}\cup ](k_{1,c_1},j_1),(k_{1,c_1},j_1)]_{w_\cJ}.$$
If $k_{2,c_2-1}<k_{1,c_1-1}$ and $k_t^\prime<k_1^{1,1}$, we define
\begin{multline*}
v_{j_1}^{\Omega^\pm,\sharp}\defeq
(k_{2,c_2-1},\dots,k_{2,1},k_{2,0},k_2^{1,1},\dots,k_2^{d_2,e_{2,d_2}},k_{1,c_1},k_{1,c_1-1},\dots,k_{1,1},k_{1,0},\\
k_t^{1,1},\dots,k_t^{d_t,e_{t,d_t}},k_{t,c_t},k_{t,c_t-1},\dots,k_{t,1},k_1^{1,1},\dots, k_1^{d_1,e_{\sharp,1}})
\end{multline*}
and
$$I_\cJ^{\Omega^\pm}\defeq I_\cJ^{\psi_1,+}\cup \bigcup_{a=2}^tI_\cJ^{\psi_a,-}\cup I_\cJ^{\Omega^\pm,\flat}\cup ](k_{1,c_1},j_1),(k_{1,c_1},j_1)]_{w_\cJ}\cup ](k_{1,0},j_1),(k_{1,0},j_1)]_{w_\cJ}.$$
If $k_{2,c_2-1}>k_{1,c_1-1}$ and either $e_{1,1}=0$ or $k_t^\prime>k_1^{1,1}$, we define
\begin{multline*}
v_{j_1}^{\Omega^\pm,\sharp}\defeq
(k_{1,c_1},k_{2,c_2-1},\dots,k_{2,1},k_{2,0},k_2^{1,1},\dots,k_2^{d_2,e_{\sharp,2}},k_{1,c_1-1},\dots,k_{1,1},\\
k_t^{1,1},\dots,k_t^{d_t,e_{t,d_t}},k_{t,c_t},k_{t,c_t-1},\dots,k_{t,1},k_{1,0},k_1^{1,1},\dots, k_1^{d_1,e_{1,d_1}})
\end{multline*}
and
$$I_\cJ^{\Omega^\pm}\defeq I_\cJ^{\psi_1,+}\cup I_\cJ^{\Omega^\pm,\sharp,2}\cup \bigcup_{a=3}^tI_\cJ^{\psi_a,-}\cup I_\cJ^{\Omega^\pm,\flat}.$$
If $k_{2,c_2-1}>k_{1,c_1-1}$, $e_{1,1}\geq 1$ and $k_t^\prime<k_1^{1,1}$, we define
\begin{multline*}
v_{j_1}^{\Omega^\pm,\sharp}\defeq
(k_{1,c_1},k_{2,c_2-1},\dots,k_{2,1},k_{2,0},k_2^{1,1},\dots,k_2^{d_2,e_{\sharp,2}},k_{1,c_1-1},\dots,k_{1,1},k_{1,0},\\
k_t^{1,1},\dots,k_t^{d_t,e_{t,d_t}},k_{t,c_t},k_{t,c_t-1},\dots,k_{t,1},k_1^{1,1},\dots, k_1^{d_1,e_{1,d_1}}).
\end{multline*}
and
$$I_\cJ^{\Omega^\pm}\defeq I_\cJ^{\psi_1,+}\cup I_\cJ^{\Omega^\pm,\sharp,2}\cup \bigcup_{a=3}^tI_\cJ^{\psi_a,-}\cup I_\cJ^{\Omega^\pm,\flat}\cup ](k_{1,0},j_1),(k_{1,0},j_1)]_{w_\cJ}.$$
Note that in each case above, the permutation $v_{j_1}^{\Omega^\pm,\sharp}$ is well-defined as the integers appearing in $v_{j_1}^{\Omega^\pm,\sharp}$ are all distinct thanks to Condition~\rm{II}-\ref{it: II 4}, \rm{II}-\ref{it: II 5}, \rm{II}-\ref{it: II 6}, \rm{II}-\ref{it: II 7} and \rm{II}-\ref{it: II 8} in Definition~\ref{def: constructible lifts}.

\subsection{Construction of type $\rm{III}$}\label{sub: type III}
In this section, we fix a constructible $\Lambda$-lift $\Omega^\pm$ of type $\rm{III}$ as in Definition~\ref{def: constructible lifts} and construct an element $v_\cJ^{\Omega^\pm}=(v_j^{\Omega^\pm})_{j\in\cJ}\in\un{W}$ as well as a subset $I_\cJ^{\Omega^\pm}\subseteq\mathbf{n}_{\cJ}$.

Let $a,a'\in\Z/t$ be two distinct elements and $\varepsilon\in\{1,-1\}$.
We say that \emph{$a'$ is $\varepsilon$-adjacent to $a$} if $a'=a+\varepsilon$ and either $(k_{a,0},j_a)=(k_{a',0},j_{a'})$ or $(k_{a,c_a},j_a)=(k_{a',c_{a'}},j_{a'})$. We say that \emph{$a'$ is $\varepsilon$-connected to $a$} if there exist an integer $t'\geq 1$ and a sequence of elements $a=a_0,\dots,a_{t'}=a'$ in $\Z/t$ such that $a_{t''}$ is $\varepsilon$-adjacent to $a_{t''-1}$ for each $1\leq t''\leq t'$. It is obvious by definition that $a'$ is $\varepsilon$-adjacent (resp.~$\varepsilon$-connected) to $a$ if and only if $a$ is $-\varepsilon$-adjacent (resp.~$-\varepsilon$-connected) to $a'$.
We say that a subset $\Sigma\subseteq\Omega^+\sqcup\Omega^-$ is a \emph{connected component} of $\Omega^+\sqcup\Omega^-$ if it is a maximal subset (under inclusion) satisfying the condition that, for each pair of distinct elements $a,a'$ inside, $a'$ is $\varepsilon$-connected to $a$ for some $\varepsilon\in\{1,-1\}$.
In other words, if we consider the graph whose vertices are indexed by $\mathbf{n}_\cJ$ and whose edges are indexed by $\Omega^+\sqcup\Omega^-$, then $\Sigma$ is a connected component of $\Omega^+\sqcup\Omega^-$ if and only if $\Sigma$ corresponds to the set of edges of a connected component of this graph.
We write $\pi_0(\Omega^\pm)$ for the set of connected components of $\Omega^+\sqcup\Omega^-$ and it is clear that we have
$$\Omega^+\sqcup\Omega^-=\underset{\Sigma\in\pi_0(\Omega^\pm)}{\bigsqcup}\Sigma.$$
As $\Omega_a$ is clearly a subset of one connected component for each $a\in\Z/t$, we have a natural decomposition
$$\Z/t=\underset{\Sigma\in\pi_0(\Omega^\pm)}{\bigsqcup}(\Z/t)_{\Sigma}$$
where $a\in(\Z/t)_{\Sigma}$ if and only if $\Omega_a\subseteq\Sigma$, for each $\Sigma\in\pi_0(\Omega^\pm)$. For each $\Sigma\in\pi_0(\Omega^\pm)$, we define $(\Z/t)_\Sigma^+\defeq (\Z/t)_\Sigma\cap(\Z/t)^+$, $(\Z/t)_\Sigma^-\defeq (\Z/t)_\Sigma\cap(\Z/t)^-$, and $b_\Sigma\defeq \#(\Z/t)_\Sigma$, and we write $j_\Sigma\in\cJ$ for the embedding determined by $\Sigma\in\pi_0(\Omega^\pm)$. As $\Omega^\pm$ is a constructible $\Lambda$-lift of type $\rm{III}$ and so $\Omega_a=\Omega_{(\al_a,j_a),\Lambda}^{\rm{max}}$ for each $a\in\Z/t$, we have $\Omega_a$ is $\Lambda$-extremal (resp.~$\Omega_a$ is $\Lambda$-exceptional) if and only if $d_a\geq 1$ and $k_a^{1,1}>k_{a,c_a-1}$ (resp.~if and only if either $d_a=0$ or $d_a= 1$ and $k_a^{1,1}<k_{a,c_a-1}$).

We fix a connected component $\Sigma\in\pi_0(\Omega^\pm)$ for the moment. For each $a\in(\Z/t)_\Sigma$, we set
$$\mathbf{n}^{a,+}\defeq \{k_{a,c}\mid 1\leq c\leq c_a\}\quad\mbox{ and }\quad\mathbf{n}^{a,-}\defeq \{k_{a,0}\}\sqcup\{k_a^{s,e}\mid 1\leq s\leq d_a,~1\leq e\leq e_{a,s}\}.$$
We define
$$k_a^\prime\defeq
\left\{
  \begin{array}{ll}
    k_a^{1,1} & \hbox{if $d_a\geq 1$;} \\
    k_{a,c_a} & \hbox{if $d_a=0$.}
  \end{array}
\right.
$$
By conditions III-\ref{it: III 3}, III-\ref{it: III 4} and III-\ref{it: III 5} and the definition of $\Sigma$ we observe that $k_{a',1}\notin \mathbf{n}^{a,-}$ and $k_a^\prime\neq k_{a',c_{a'}-1}$ for each pair of (possibly equal) elements $a,a'\in(\Z/t)_\Sigma$.
We also define
$$\mathbf{n}_\Sigma\defeq \underset{a\in(\Z/t)_\Sigma}{\bigcup}(\mathbf{n}^{a,+}\sqcup\mathbf{n}^{a,-}).$$

We start with defining $v_\cJ^{\Omega^\pm}$ and $I_\cJ^{\Omega^\pm}$ for a constructible $\Lambda$-lift $\Omega^\pm$ of type $\rm{III}$ satisfying $t=2$.

We first consider the case $t=\#\pi_0(\Omega^\pm)=2$. For each $a=1,2$, we define $v_j^{\Omega^\pm,a}\defeq 1$ for each $j\neq j_a$ and
$$v_{j_a}^{\Omega^\pm,a}\defeq (k_{a,c_a},\dots,k_{a,1},k_{a,0},k_a^{1,1},\dots,k_a^{d_a,e_{a,d_a}}).$$
Then we define $v_j^{\Omega^\pm}\defeq v_j^{\Omega^\pm,1}v_j^{\Omega^\pm,2}w_j$ for each $j\in\cJ$
and
$$
I_\cJ^{\Omega^\pm}\defeq I_\cJ^{\psi_1,+}\cup I_\cJ^{\psi_2,-}\cup ](k_{1,c_1},j_1),(k_{2,c_2},j_2)]_{w_\cJ}\cup ](k_{2,0},j_2),(k_{1,0},j_1)]_{w_\cJ}.
$$

Now we consider the case $t=2$ and $\#\pi_0(\Omega^\pm)=1$, and in particular $\Omega^+\sqcup\Omega^-$ is a connected component which is not circular (due to the condition \rm{III}-\ref{it: III 1}). We have either $k_{1,0}=k_{2,0}$ or $k_{1,c_1}=k_{2,c_2}$ and exactly one of them holds. If $k_{1,0}=k_{2,0}$ and $k_1^\prime>k_2^\prime$, we define
$$
v_{j_1}^{\Omega^\pm,\sharp}\defeq
(k_2^{1,1},\dots,k_2^{d_2,e_{2,d_2}},k_{2,c_2},k_{2,c_2-1},\dots,k_{2,1},k_1^{1,1},\dots,k_1^{d_1,e_{1,d_1}},k_{1,c_1},k_{1,c_1-1},\dots,k_{1,1},k_{1,0}).
$$
and
$$I_\cJ^{\Omega^\pm}\defeq I_\cJ^{\psi_1,+}\cup I_\cJ^{\psi_2,-}\cup ](k_{1,c_1},j_1),(k_{2,c_2},j_1)]_{w_\cJ}\cup ](k_{1,0},j_1),(k_{1,0},j_1)]_{w_\cJ}.$$
If $k_{1,0}=k_{2,0}$ and $k_1^\prime<k_2^\prime$, we define
$$
v_{j_1}^{\Omega^\pm,\sharp}\defeq
(k_2^{1,1},\dots,k_2^{d_2,e_{2,d_2}},k_{2,c_2},k_{2,c_2-1},\dots,k_{2,1},k_{1,0},k_1^{1,1},\dots,k_1^{d_1,e_{1,d_1}},k_{1,c_1},k_{1,c_1-1},\dots,k_{1,1}).
$$
and
$$I_\cJ^{\Omega^\pm}\defeq I_\cJ^{\psi_1,+}\cup I_\cJ^{\psi_2,-}\cup ](k_{1,c_1},j_1),(k_{2,c_2},j_1)]_{w_\cJ}.$$
If $k_{1,c_1}=k_{2,c_2}$ and $k_{1,c_1-1}>k_{2,c_2-1}$, we define
$$
v_{j_1}^{\Omega^\pm,\sharp}\defeq
(k_{2,c_2-1},\dots,k_{2,1},k_{2,0},k_2^{1,1},\dots,k_2^{d_2,e_{2,d_2}},k_{1,c_1},k_{1,c_1-1},\dots,k_{1,1},k_{1,0},k_1^{1,1},\dots,k_1^{d_1,e_{\sharp,1}}).
$$
and
$$I_\cJ^{\Omega^\pm}\defeq I_\cJ^{\psi_1,+}\cup I_\cJ^{\psi_2,-}\cup ](k_{1,c_1},j_1),(k_{1,c_1},j_1)]_{w_\cJ}\cup ](k_{2,0},j_1),(k_{1,0},j_1)]_{w_\cJ}.$$
If $k_{1,c_1}=k_{2,c_2}$ and $k_{1,c_1-1}<k_{2,c_2-1}$, we define
$$
v_{j_1}^{\Omega^\pm,\sharp}\defeq
(k_{1,c_1},k_{2,c_2-1},\dots,k_{2,1},k_{2,0},k_2^{1,1},\dots,k_2^{d_2,e_{\sharp,2}},k_{1,c_1-1},\dots,k_{1,1},k_{1,0},k_1^{1,1},\dots,k_1^{d_1,e_{1,d_1}}).
$$
and
$$I_\cJ^{\Omega^\pm}\defeq I_\cJ^{\psi_1,+}\cup I_\cJ^{\Omega^\pm,\sharp,2}\cup ](k_{2,0},j_1),(k_{1,0},j_1)]_{w_\cJ}.$$
Here the definitions of $e_{\sharp,1}$, $e_{\sharp,2}$ and $I_\cJ^{\Omega^\pm,\sharp,2}$ are parallel to the ones that have already appeared in~\S\,\ref{sub: type II}.
Finally we define $v_j^{\Omega^\pm}\defeq w_j$ for each $j\neq j_1$ and $v_{j_1}^{\Omega^\pm}\defeq v_{j_1}^{\Omega^\pm,\sharp}w_{j_1}$ for all four cases above.

We devote the rest of this section to the cases when $\Omega^\pm$ is a constructible $\Lambda$-lift of type $\rm{III}$ with $t\geq 3$. We say that $k\in\mathbf{n}_\Sigma$ is a \emph{$1$-end} (resp.~\emph{$-1$-end}) of $\Sigma$ if there exist a unique $a\in(\Z/t)_\Sigma$ and $k'\in\{k_{a+1,0},k_{a+1,c_{a+1}}\}$ (resp.~$k'\in\{k_{a-1,0},k_{a-1,c_{a-1}}\}$) such that $k\in\{k_{a,0},k_{a,c_a}\}$ and the elements $(k,j_a), (k',j_{a+1})$ (resp.~$(k,j_a),(k',j_{a-1})$) are different elements in the same $(w_\cJ,1)$-orbit.
We say that $\Sigma$ is \emph{circular} if it has neither $1$-end nor $-1$-end.
It is clear that exactly one of the following holds:
\begin{itemize}
\item each $\Sigma\in\pi_0(\Omega^\pm)$ has exactly one $1$-end and exactly one $-1$-end;
\item $\pi_0(\Omega^\pm)=\{\Omega^+\sqcup\Omega^-\}$ and $\Omega^+\sqcup\Omega^-$ is circular.
\end{itemize}
We will use the term \emph{direction} for an element $\varepsilon\in\{1,-1\}$. The $(\Z/t)_\Sigma^+$ and $(\Z/t)_\Sigma^-$ are visualized in Figure~\ref{fig:two:dir}.

\begin{defn}\label{def: next element}
Let $\Sigma\in\pi_0(\Omega^\pm)$ be a connected component and $k,k'\in\mathbf{n}_\Sigma$ be two elements. If $k$ is not a $1$-end of $\Sigma$, then we say that \emph{$k'$ is the $1$-successor of $k$} if exactly one of the following holds:
\begin{itemize}
\item $k\in\mathbf{n}^{a,+}\setminus\{k_{a,c_a}\}$ for some $a\in(\Z/t)_\Sigma^+$ and $k'=\max\{k''\in\mathbf{n}^{a,+}\mid k''<k\}$;
\item $k\in\mathbf{n}^{a,-}\setminus\{k_{a,0}\}$ for some $a\in(\Z/t)_\Sigma^-$ and $k'=\min\{k''\in\mathbf{n}^{a,-}\mid k''>k\}$;
\item $k=k_{a,c_a}$ for some $a\in(\Z/t)_\Sigma^-$ and $k'=\min\mathbf{n}^{a,-}$;
\item $k=k_{a,0}$ for some $a\in(\Z/t)_\Sigma^+$ and $k'=\max\mathbf{n}^{a,+}$.
\end{itemize}
If $k$ is not a $-1$-end of $\Sigma$, then we say that \emph{$k'$ is the $-1$-successor of $k$} if exactly one of the following holds:
\begin{itemize}
\item $k\in\mathbf{n}^{a,+}\setminus\{k_{a,c_a}\}$ for some $a\in(\Z/t)_\Sigma^-$ and $k'=\max\{k''\in\mathbf{n}^{a,+}\mid k''<k\}$;
\item $k\in\mathbf{n}^{a,-}\setminus\{k_{a,0}\}$ for some $a\in(\Z/t)_\Sigma^+$ and $k'=\min\{k''\in\mathbf{n}^{a,-}\mid k''>k\}$;
\item $k=k_{a,c_a}$ for some $a\in(\Z/t)_\Sigma^+$ and $k'=\min\mathbf{n}^{a,-}$;
\item $k=k_{a,0}$ for some $a\in(\Z/t)_\Sigma^-$ and $k'=\max\mathbf{n}^{a,+}$.
\end{itemize}
Let $\varphi:~\mathbf{n}_\varphi\rightarrow\mathbf{n}_\Sigma$ be an injective map with $\mathbf{n}_\varphi\subseteq\mathbf{n}_\Sigma$ a non-empty subset. For each $\varepsilon\in\{1,-1\}$, we say that $\varphi$ \emph{has a $\varepsilon$-crawl from $k$ to $k'$} if there exist an integer $s\geq 1$ and a sequence of elements $k=k_0,\dots,k_s=k'$ in $\mathbf{n}_\Sigma$ such that $\varphi(k_{s'-1})=k_{s'}$ is the $\varepsilon$-successor of $k_{s'-1}$ for each $1\leq s'\leq s$. The set $\{k_{s'}\mid 0\leq s'\leq s-1\}$ is called the \emph{orbit} of the $\varepsilon$-crawl above. See Figure~\ref{fig:crawl} for an example of $\varepsilon$-crawl.
\end{defn}

For each $a\in(\Z/t)_\Sigma$ and $\varepsilon\in\{1,-1\}$, we write
$$k_a^{[\varepsilon]}\defeq
\left\{
  \begin{array}{ll}
     k_{a,c_a-1} & \hbox{if $a\in(\Z/t)_\Sigma^+$ and $\varepsilon=1$;} \\
     k_{a,c_a-1} & \hbox{if $a\in(\Z/t)_\Sigma^-$ and $\varepsilon=-1$;} \\
     k_a^\prime & \hbox{if $a\in(\Z/t)_\Sigma^-$ and $\varepsilon=1$;} \\
     k_a^\prime & \hbox{if $a\in(\Z/t)_\Sigma^+$ and $\varepsilon=-1$.}
  \end{array}
\right.
$$
It is clear that $k_a^{[\varepsilon]}$ is the unique element in $\mathbf{n}^{a,+}\sqcup\mathbf{n}^{a,-}$ with a unique $\varepsilon$-successor of the form $k_{a,c_a}$ or $k_{a,0}$.

\begin{defn}\label{def: jump of map}
Let $\Sigma\in\pi_0(\Omega^\pm)$ be a connected component and $\varphi:~\mathbf{n}_\varphi\rightarrow\mathbf{n}_\Sigma$ be an injective map for some non-empty subset $\mathbf{n}_\varphi\subseteq\mathbf{n}_\Sigma$. For $k\in\mathbf{n}_\varphi$, we say that $\varphi$ \emph{has a $\varepsilon$-jump at $k$} for some $\varepsilon\in\{1,-1\}$ if there exist an element $a\in(\Z/t)_\Sigma$ and an integer $1\leq b\leq b_\Sigma$ such that
$$k_{a+b'\varepsilon,c_{a+b'\varepsilon}-1}>k>k_{a+b'\varepsilon}^\prime$$ (and thus $\Omega_{a+b'\varepsilon}$ is $\Lambda$-exceptional) for each $1\leq b'\leq b-1$ and exactly one of the following holds:
\begin{itemize}
\item $k=k_a^{[\varepsilon]}>k_{a+b\varepsilon}^\prime$, $\varphi(k)=k_{a+b\varepsilon,1}$ and $\varphi$ has a $-\varepsilon$-crawl from $k_{a+b\varepsilon,0}$ to the $\varepsilon$-successor of~$k_a^{[\varepsilon]}$;
\item $k=k_a^{[\varepsilon]}<k_{a+b\varepsilon,c_{a+b\varepsilon}-1}$, $\varphi(k)=\min\{k'\in\mathbf{n}^{a+b\varepsilon,-}\mid k'>k\}$ and $\varphi$ has a $-\varepsilon$-crawl from $k_{a+b\varepsilon,c_{a+b\varepsilon}}$ to the $\varepsilon$-successor of $k_a^{[\varepsilon]}$.
\end{itemize}
We note from the injectivity of $\varphi$ that $k_{a+b\varepsilon,1}$ is the $\varepsilon$-successor of $k_{a+b\varepsilon,0}$ in the first case, and $\min\mathbf{n}^{a+b\varepsilon,-}$ is the $\varepsilon$-successor of $k_{a+b\varepsilon,c_{a+b\varepsilon}}$ in the second case.
We say that \emph{the $\varepsilon$-jump at $k$ covers $k_{a+b\varepsilon,0}$} in the first case, and \emph{the $\varepsilon$-jump at $k$ covers $k_{a+b\varepsilon,c_{a+b\varepsilon}}$} in the second case.
We also say that \emph{the $\varepsilon$-jump at $k$ covers $k'$} for each $k'\in\bigcup_{1\leq b'\leq b-1}\{k_{a+b'\varepsilon,0},k_{a+b'\varepsilon,c_{a+b'\varepsilon}}\}$. We understand $\{k\}$ to be the orbit of a $\varepsilon$-jump at $k$. See Figure~\ref{fig:jumps} for typical examples of $\varepsilon$-jumps.
\end{defn}

\begin{defn}\label{def: oriented permutation}
Let $\Sigma\in\pi_0(\Omega^\pm)$ be a connected component and $\varphi:~\mathbf{n}_\varphi\rightarrow\mathbf{n}_\Sigma$ be an injective map for some non-empty subset $\mathbf{n}_\varphi\subseteq\mathbf{n}_\Sigma$. For each $\varepsilon\in\{1,-1\}$ and each pair of (possibly equal) elements $k,k'\in\mathbf{n}_\Sigma$, we say that $\varphi$ \emph{has a $\varepsilon$-tour from $k$ to $k'$} if there exists $s\geq 1$ and a sequence of elements $k=k_0,\dots,k_s=k'$ such that, for each $1\leq s'\leq s$, we have $\varphi(k_{s'-1})=k_{s'}$ and exactly one of the following holds:
\begin{itemize}
\item $k_{s'}$ is the $\varepsilon$-successor of $k_{s'-1}$;
\item $\varphi$ has a $\varepsilon$-jump at $k_{s'-1}$.
\end{itemize}
We call the set $\{k_{s'}\mid 0\leq s'\leq s-1\}$ the \emph{orbit} of the $\varepsilon$-tour. We can say that a $\varepsilon$-tour contains a $\varepsilon$-crawl, a $\varepsilon$-jump or another $\varepsilon$-tour by checking their orbits.

A permutation $\varphi:~\mathbf{n}_\Sigma\rightarrow\mathbf{n}_\Sigma$ is called \emph{oriented} if $\varphi$ has a $1$-tour and a $-1$-tour satisfying the following
\begin{enumerate}[label=(\roman*)]
\item \label{it: disjoint orbits} the orbit of $1$-tour is disjoint from that of $-1$-tour, and $\varphi$ fixes each element of $\mathbf{n}_\Sigma$ that appears in neither orbit;
\item \label{it: end of tour}
      \begin{itemize}
      \item if $\Sigma$ is not circular, then the fixed $\varepsilon$-tour goes from the $-\varepsilon$-end to $\varepsilon$-end for each $\varepsilon\in\{1,-1\}$;
      \item if $\Sigma$ is circular, then the orbit of the fixed $\varepsilon$-tour is a single orbit of the permutation $\varphi$, for each $\varepsilon\in\{1,-1\}$;
      \end{itemize}
\item \label{it: unique cover} for each $k\in\bigcup_{a\in(\Z/t)_\Sigma}\{k_{a,0},k_{a,c_a}\}$ which is neither the $1$-end or $-1$-end of $\Sigma$, there exists a unique $\varepsilon\in\{1,-1\}$ such that
    \begin{itemize}
    \item $k$ lies in the orbit of the fixed $\varepsilon$-tour of $\varphi$;
    \item the fixed $-\varepsilon$-tour of $\varphi$ contains a unique $-\varepsilon$-jump that covers $k$.
    \end{itemize}
\item \label{it: interaction of jumps} if there exist $a\in(\Z/t)_\Sigma$, $k,k'\in\mathbf{n}_\Sigma$ and $\varepsilon\in\{1,-1\}$ such that
      \begin{itemize}
      \item $\varphi(k)=k_{a,0}$ and the fixed $\varepsilon$-tour of $\varphi$ contains a $\varepsilon$-jump at $k$ that covers $k_{a,c_a}$;
      \item $\varphi(k')=k_{a,1}$ and the fixed $-\varepsilon$-tour of $\varphi$ contains a $-\varepsilon$-jump at $k'$ that covers $k_{a,0}$,
      \end{itemize}
      then we have $k'>k$;
\item \label{it: degeneration of tour} if there exists $\varepsilon\in\{1,-1\}$ and $a\in(\Z/t)_\Sigma$ such that the fixed $\varepsilon$-tour contains a $\varepsilon$-jump at $k_a^{[\varepsilon]}$ which satisfies either $\varphi(k_a^{[\varepsilon]})=k_a^{[\varepsilon]}$ or $\varphi(k_a^{[\varepsilon]})\in(\mathbf{n}^{a,+}\sqcup\mathbf{n}^{a,-})\setminus (\mathbf{n}^{a-\varepsilon,+}\sqcup\mathbf{n}^{a-\varepsilon,-})$, then $\Sigma$ is circular, $c_a\geq 2$, $k_a^{[\varepsilon]}=k_{a,c_a-1}=\min\{k_{a',c_{a'}-1}\mid a'\in\Z/t\}$ and $\varphi(k_{a,c_a-1})=k_{a,1}$;
\end{enumerate}
For each oriented permutation $\varphi$ and each $\varepsilon\in\{1,-1\}$, we always fix a choice of $\varepsilon$-tour as above, and say that $\varphi$ has \emph{direction $\varepsilon$} at some $k\in\mathbf{n}_\Sigma$ if $k$ belongs to the orbit of the fixed $\varepsilon$-tour. Two examples of oriented permutation (when $\mathbf{n}^{a,+}=\{k_{a,c_a}\}$ and $\mathbf{n}^{a,-}=\{k_{a,0}\}$ for each $a\in(\Z/t)_\Sigma$) are visualized in Figure~\ref{fig:OrPer}. Item~\ref{it: interaction of jumps} is also visualized in Figure~\ref{fig:item4Def}.
\end{defn}

Assuming that there exists an oriented permutation of $\mathbf{n}_\Sigma$ for each $\Sigma\in\pi_0(\Omega^\pm)$, we define $v_\cJ^{\Omega^\pm}$ and $I_\cJ^{\Omega^\pm}$ for a constructible $\Lambda$-lift $\Omega^\pm$ of type $\rm{III}$ with $t\geq 3$.
For each $\Sigma\in\pi_0(\Omega^\pm)$, we define $(v_\Sigma^{\Omega^\pm})^{-1}\in W$ to be an arbitrary element of $W$ which fixes $\mathbf{n}\setminus\mathbf{n}_\Sigma$ and restricts to an oriented permutation of $\mathbf{n}_\Sigma$. It follows from Condition~\rm{III}-\ref{it: III 4} of Definition~\ref{def: constructible lifts} that, for each $a,a'\in\Z/t$ lying in different connected components with $j_a=j_{a'}$, we have $(\mathbf{n}^{a,+}\sqcup\mathbf{n}^{a,-})\cap(\mathbf{n}^{a',+}\sqcup\mathbf{n}^{a',-})=\emptyset$, which implies that $v_\Sigma^{\Omega^\pm}$ commutes with $v_{\Sigma'}^{\Omega^\pm}$ for each distinct pair $\Sigma,\Sigma'\in\pi_0(\Omega^\pm)$ with $j_\Sigma=j_{\Sigma'}$. Hence we can define $v_\cJ^{\Omega^\pm}=(v_j^{\Omega^\pm})_{j\in\cJ}$ by letting
$$
v_j^{\Omega^\pm}\defeq \left(\underset{\substack{\Sigma\in\pi_0(\Omega^\pm),\,j_\Sigma=j}}{\prod}v_\Sigma^{\Omega^\pm}\right)w_j
$$
for each $j\in\cJ$. If $\pi_0(\Omega^\pm)=\{\Omega^+\sqcup\Omega^-\}$ and $\Omega^+\sqcup\Omega^-$ is circular, then we write $\mathbf{n}_{\Omega^+\sqcup\Omega^-,1}$ for the orbit of the fixed $1$-tour of the oriented permutation $(v_{\Omega^+\sqcup\Omega^-}^{\Omega^\pm})^{-1}|_{\mathbf{n}_{\Omega^+\sqcup\Omega^-}}$, and set
$$I_\cJ^{\Omega^\pm}\defeq \underset{k\in\mathbf{n}_{\Omega^+\sqcup\Omega^-,1}}{\bigcup}](k,j_{\Omega^+\sqcup\Omega^-}),(k,j_{\Omega^+\sqcup\Omega^-})]_{w_\cJ}.$$
If $\Omega^\pm$ does not have a circular connected component, then we write $k_\Sigma$ (resp.~$k_\Sigma^\prime$) for the $-1$-end (resp.~$1$-end) of $\Sigma$ and write $\mathbf{n}_{\Sigma,1}$ for the orbit of the fixed $1$-tour for the oriented permutation $(v_\Sigma^{\Omega^\pm})^{-1}|_{\mathbf{n}_\Sigma}$, for each $\Sigma\in\pi_0(\Omega^\pm)$. We write $h\defeq\#\pi_0(\Omega^\pm)$ and order $\pi_0(\Omega^\pm)$ as $\{\Sigma_{h'}\mid h'\in\Z/h\}$ in a way that $(k_{\Sigma_{h'}}^\prime,j_{\Sigma_{h'}})$ and $(k_{\Sigma_{h'+1}},j_{\Sigma_{h'+1}})$ lie in the same $(w_\cJ,1)$-orbit, for each $h'\in\Z/h$. Then we define
$$I_\cJ^{\Omega^\pm}\defeq \underset{h'\in\Z/h}{\bigcup}\left(](k_{\Sigma_{h'}}^\prime,j_{\Sigma_{h'}}),(k_{\Sigma_{h'+1}},j_{\Sigma_{h'+1}})]_{w_\cJ}\cup\underset{k\in\mathbf{n}_{\Sigma_{h'},1}\setminus\{k_{\Sigma_{h'}}\}}{\bigcup}](k,j_{\Sigma_{h'}}),(k,j_{\Sigma_{h'}})]_{w_\cJ}\right).$$

The rest of this section is devoted to the construction of an oriented permutation of $\mathbf{n}_\Sigma$ for each $\Sigma\in\pi_0(\Omega^\pm)$ when $t\geq 3$.
\begin{lemma}\label{lem: backwards jump}
Let $\Sigma\in\pi_0(\Omega^\pm)$ be a connected component. Fix $a\in(\Z/t)_\Sigma$ and $1\leq b\leq b_\Sigma$ such that $a+b$ is $1$-connected to $a$. Assume that
\begin{equation}\label{eq: assumption on backward jump}
\left\{
  \begin{array}{cl}
    k_{a+b',c_{a+b'}-1}<k_{a+b'-1,c_{a+b'-1}-1} & \hbox{for each $1\leq b'\leq b$ with $a+b'\in(\Z/t)_\Sigma^-$;} \\
    k_{a+b'}^\prime>k_{a+b'-1}^\prime & \hbox{for each $1\leq b'\leq b$ with $a+b'\in(\Z/t)_\Sigma^+$.}
  \end{array}
\right.
\end{equation}
Then there exists a sequence $0\leq b_0<b_1<\cdots<b_s=b$ for some $s\geq 0$ such that
$$k_{a+b',c_{a+b'}-1}>k_{a+b_0}^{[-1]}>k_{a+b'}^\prime$$ for each $0\leq b'<b_0$ and the following hold: for each $1\leq s'\leq s$
\begin{itemize}
\item $k_{a+b',c_{a+b'}-1}>k_{a+b_{s'}}^{[-1]}>k_{a+b'}^\prime$ for each $b_{s'-1}<b'<b_{s'}$;
\item if $a+b_{s'-1}\in (\Z/t)_\Sigma^-$, then $k_{a+b_{s'}}^{[-1]}>\max\{k_{a+b_{s'-1},c_{a+b_{s'-1}}-1},k_{a+b_{s'-1}}^\prime\}$;
\item if $a+b_{s'-1}\in (\Z/t)_\Sigma^+$, then $k_{a+b_{s'}}^{[-1]}<\min\{k_{a+b_{s'-1},c_{a+b_{s'-1}}-1},k_{a+b_{s'-1}}^\prime\}$.
\end{itemize}
\end{lemma}
\begin{proof}
We argue by increasing induction on $b\geq 1$. First of all, it is clear that either $k_{a+b',c_{a+b'}-1}>k_{a+b}^{[-1]}>k_{a+b'}^\prime$ for each $0\leq b'<b$, or there exists an integer $0\leq b_\flat<b$ such that
\begin{itemize}
\item $k_{a+b',c_{a+b'}-1}>k_{a+b}^{[-1]}>k_{a+b'}^\prime$ for each $b_\flat<b'<b$;
\item either $k_{a+b}^{[-1]}>k_{a+b_\flat,c_{a+b_\flat}-1}$ or $k_{a+b}^{[-1]}<k_{a+b_\flat}^\prime$.
\end{itemize}
We may assume without loss of generality that $b_\flat$ exists, otherwise we simply take $s\defeq 0$ and $b_0\defeq b$. If $b_\flat=b-1$ and $k_{a+b}^{[-1]}>k_{a+(b-1),c_{a+(b-1)}-1}$, then we must have $a+b\in(\Z/t)_\Sigma^+$ and $k_{a+b}^{[-1]}>k_{a+(b-1)}^\prime$. If $b_\flat=b-1$ and $k_{a+b}^{[-1]}<k_{a+(b-1)}^\prime$, then we must have $a+b\in(\Z/t)_\Sigma^-$ and $k_{a+b}^{[-1]}<k_{a+(b-1),c_{a+(b-1)}-1}$. If $b_\flat\leq b-2$ and $k_{a+b}^{[-1]}>k_{a+b_\flat,c_{a+b_\flat}-1}$, then we have $k_{a+(b_\flat+1),c_{a+(b_\flat+1)}-1}>k_{a+b}^{[-1]}>k_{a+b_\flat,c_{a+b_\flat}-1}$
which forces $a+b_\flat\in(\Z/t)_\Sigma^-$, and so
$k_{a+b}^{[-1]}>k_{a+(b_\flat+1)}^\prime>k_{a+b_\flat}^\prime$.
If $b_\flat\leq b-2$ and $k_{a+b}^{[-1]}<k_{a+b_\flat}^\prime$, then we have
$k_{a+b_\flat}^\prime>k_{a+b}^{[-1]}>k_{a+(b_\flat+1)}^\prime$
which forces $a+b_\flat\in(\Z/t)_\Sigma^+$, and so
$k_{a+b}^{[-1]}<k_{a+(b_\flat+1),c_{a+(b_\flat+1)}-1}<k_{a+b_\flat,c_{a+b_\flat}-1}$.
Up to this stage, we have just shown that
\begin{itemize}
\item $k_{a+b',c_{a+b'}-1}>k_{a+b}^{[-1]}>k_{a+b'}^\prime$ for each $b_\flat<b'<b$;
\item if $a+b_\flat\in (\Z/t)_\Sigma^-$, then $k_{a+b}^{[-1]}>\max\{k_{a+b_\flat,c_{a+b_\flat}-1},k_{a+b_\flat}^\prime\}$;
\item if $a+b_\flat\in (\Z/t)_\Sigma^+$, then $k_{a+b}^{[-1]}<\min\{k_{a+b_\flat,c_{a+b_\flat}-1},k_{a+b_\flat}^\prime\}$.
\end{itemize}
Using our inductive assumption, we obtain a sequence $0\leq b_0<b_1<\cdots<b_{s_\flat}=b_\flat$. We set $s\defeq s_\flat+1$ and $b_s\defeq b$ and it is not difficult to check that the so obtained sequence $0\leq b_0<b_1<\cdots<b_s=b$ satisfies all the desired properties. The proof is thus finished.
\end{proof}

\begin{prop}\label{prop:oriented:non circular}
Let $\Sigma\in\pi_0(\Omega^\pm)$ be a connected component which is not circular. Then there exists an oriented permutation of $\mathbf{n}_\Sigma$.
\end{prop}

\begin{proof}
It is clear that $\Sigma$ has exactly one $-1$-end and one $1$-end. The proof is divided into two steps.

We first construct a certain injective map $\varphi_0:~\mathbf{n}_{\varphi_0}\rightarrow\mathbf{n}_\Sigma$ which has a $1$-tour from the $-1$-end of $\Sigma$ to the $1$-end of $\Sigma$ and with $\mathbf{n}_{\varphi_0}$ minimal possible. We construct $\varphi_0$ by the following inductive procedure. Let $\varphi:~\mathbf{n}_\varphi\rightarrow\mathbf{n}_\Sigma$ an injective map constructed from the previous step, we want to construct another $\varphi':~\mathbf{n}_{\varphi'}\rightarrow\mathbf{n}_\Sigma$ that satisfies $\mathbf{n}_\varphi\subsetneq\mathbf{n}_{\varphi'}$ and $\varphi'|_{\mathbf{n}_\varphi}=\varphi$. If $\varphi$ already has a $1$-tour from the $-1$-end of $\Sigma$ to the $1$-end of $\Sigma$, then we set $\varphi_0\defeq \varphi$. Otherwise (by inductive construction) $\varphi$ has a $1$-tour from the $-1$-end of $\Sigma$ to some $k\in\mathbf{n}_\Sigma\setminus\mathbf{n}_\varphi$. We write $a_0$ for the unique element of $(\Z/t)_\Sigma$ such that the $1$-end of $\Sigma$ has the form $k_{a_0,0}$ or $k_{a_0,c_{a_0}}$.

If either $k\neq k_a^{[1]}$ for any $a\in(\Z/t)_\Sigma$ or the $1$-successor of $k$ is the $1$-end, we define $\mathbf{n}_{\varphi'}\defeq \mathbf{n}_\varphi\sqcup\{k\}$ and $\varphi'(k)$ to be the $1$-successor of $k$.

If $k= k_a^{[1]}$ for some $a\in(\Z/t)_\Sigma$, then we define $b\geq 1$ to be the unique integer such that $k_{a+b',c_{a+b'}-1}>k>k_{a+b'}^\prime$ for each $1\leq b'\leq b-1$, and either $k>k_{a+b,c_{a+b}-1}$ or $k<k_{a+b}^\prime$.

If $b$ does not exist, then we define $\varphi'$ as the unique injective map such that $\varphi'$ has a $1$-jump at $k$ with $\varphi'(k)\in(\mathbf{n}^{a_0,+}\sqcup\mathbf{n}^{a_0,-})\setminus(\mathbf{n}^{a_0-1,+}\sqcup\mathbf{n}^{a_0-1,-})$, and $\mathbf{n}_{\varphi'}\supseteq \mathbf{n}_\varphi\sqcup\{k\}$ is minimal possible.

If $b=1$ and either $k=k_{a,c_a-1}>k_{a+1,c_{a+1}-1}$ or $k=k_a^\prime<k_{a+1}^\prime$, we define $\mathbf{n}_{\varphi'}\defeq \mathbf{n}_\varphi\sqcup\{k\}$ and $\varphi'(k)$ to be the $1$-successor of $k$. We assume in the rest of the construction of $\varphi'$ that $b$ exists and
\begin{itemize}
\item if $k=k_{a,c_a-1}$, then $k_{a,c_a-1}<k_{a+1,c_{a+1}-1}$;
\item if $k=k_a^\prime$, then $k_a^\prime>k_{a+1}^\prime$.
\end{itemize}

If $a+b\in(\Z/t)_\Sigma^+$ and $k<k_{a+b}^\prime$, then we define $\varphi'$ as the unique injective map such that $\varphi'$ has a $1$-jump at $k$ with $\varphi'(k)=k_{a+b,0}\in\mathbf{n}^{a+b-1,-}$, and $\mathbf{n}_{\varphi'}\supseteq \mathbf{n}_\varphi\sqcup\{k\}$ is minimal possible.

If $a+b\in(\Z/t)_\Sigma^+$ and $k>\max\{k_{a+b,c_{a+b}-1},k_{a+b}^\prime\}$, then we define $\varphi'$ as the unique injective map such that $\varphi'$ has a $1$-jump at $k$ with $\varphi'(k)=k_{a+b,1}\in\mathbf{n}^{a+b,+}$, and $\mathbf{n}_{\varphi'}\supseteq \mathbf{n}_\varphi\sqcup\{k\}$ is minimal possible.

If $a+b\in(\Z/t)_\Sigma^-$ and $k>k_{a+b,c_{a+b}-1}$, then we define $\varphi'$ as the unique injective map such that $\varphi'$ has a $1$-jump at $k$ with $\varphi'(k)=k_{a+b-1,1}\in\mathbf{n}^{a+b-1,+}$, and $\mathbf{n}_{\varphi'}\supseteq \mathbf{n}_\varphi\sqcup\{k\}$ is minimal possible.

If $a+b\in(\Z/t)_\Sigma^-$ and $k<\min\{k_{a+b,c_{a+b}-1},k_{a+b}^\prime\}$, then we define $\varphi'$ as the unique injective map such that $\varphi'$ has a $1$-jump at $k$ with $\varphi'(k)=\min\{k'\in\mathbf{n}^{a+b,-}\mid k'>k\}$, and $\mathbf{n}_{\varphi'}\supseteq \mathbf{n}_\varphi\sqcup\{k\}$ is minimal possible.

Our definition of $b$ and division of cases ensure that $\varphi'$ is always well-defined (mainly checking Definition~\ref{def: jump of map}), and in fact the cases above exhaust all possibilities. Up to this stage, we finish the construction of our desired $\varphi_0$. An example of $\varphi_0$ is visualized in Figure~\ref{fig:step1}.

Now we extend $\varphi_0$ to an oriented permutation of $\Sigma$. It suffices to extend $\varphi_0$ to another injective map which also has a $-1$-tour from the $1$-end of $\Sigma$ to $-1$-end of $\Sigma$ and such that $\mathbf{n}_{\varphi_1}$ is minimal possible. In fact, if $\varphi_1$ exists, then we can trivially extend $\varphi_1$ to an oriented permutation $\varphi_2$ of $\mathbf{n}_\Sigma$ by setting $\varphi_2(k)\defeq k$ for each $k\in\mathbf{n}_\Sigma\setminus\mathbf{n}_{\varphi_1}$. Roughly speaking, each $1$-jump in $\varphi_0$ already produces some $-1$-crawl in $\varphi_0$, and thus our construction of $\varphi_1$ reduces to construct the desired $-1$-tour from the $1$-end of $\Sigma$ to $-1$-end of $\Sigma$ by connecting the $-1$-crawls in $\varphi_0$ together.

We choose two elements $k_{[1]},k_{[1]}'\in\mathrm{E}_\Sigma\defeq \bigcup_{a\in(\Z/t)_\Sigma}\{k_{a,0},k_{a,c_a}\}$ such that exactly one of the following holds:
\begin{itemize}
\item $k_{[1]}=k_{[1]}'$ is in the orbit of the $1$-tour of $\varphi_0$ from the $-1$-end to the $1$-end, and $\varphi_0$ has no $1$-crawl either from an element in $\mathrm{E}_\Sigma$ to $k_{[1]}$ or from $k_{[1]}$ to an element in $\mathrm{E}_\Sigma$;
\item $k_{[1]}\neq k_{[1]}'$, $\varphi_0$ has a $1$-crawl from $k_{[1]}$ to $k_{[1]}'$, and the orbit of this $1$-crawl is maximal (under inclusion of subsets of $\mathbf{n}_\Sigma$) among all possible such choices.
\end{itemize}
We can uniquely determine $a_1\in(\Z/t)_\Sigma$ and $k_{[-1]}'\in\{k_{a_1,0},k_{a_1,c_{a_1}}\}$ such that exactly one of the two possibilities holds:
\begin{itemize}
\item $k_{[1]}=k_{[-1]}'$ is the $-1$-end of $\Sigma$;
\item \begin{itemize}
      \item $\{k_{a_1,0},k_{a_1,c_{a_1}}\}=\{k_{[1]},k_{[-1]}'\}$;
      \item $\varphi_0$ has a $1$-tour from the $-1$-end to $k_{[1]}$ which contains a $1$-jump that covers $k_{[-1]}'$.
      \end{itemize}
\end{itemize}
Similarly, we can uniquely determine $a_1'\in(\Z/t)_\Sigma$ and $k_{[-1]}\in\{k_{a_1',0},k_{a_1',c_{a_1'}}\}$ such that exactly one of the two possibilities holds:
\begin{itemize}
\item $k_{[1]}'=k_{[-1]}$ is the $1$-end of $\Sigma$;
\item \begin{itemize}
      \item $\{k_{a_1',0},k_{a_1',c_{a_1'}}\}=\{k_{[1]}',k_{[-1]}\}$;
      \item $\varphi_0$ has a $1$-tour from $k_{[1]}'$ to the $1$-end which contains a $1$-jump that covers $k_{[-1]}$.
      \end{itemize}
\end{itemize}
Then we apply Lemma~\ref{lem: backwards jump} by replacing $a$ and $a+b$ there with $a_1$ and $a_1'$ respectively, and obtain a sequence of integers $0\leq b_0<\cdots<b_s$ for some $s\geq0$ as stated there. Then we require that $\varphi_1$ has a $-1$-tour from $k_{[-1]}$ to $k_{[-1]}'$ which satisfies
\begin{itemize}
\item for each $1\leq s'\leq s$, $\varphi_1$ has a $-1$-jump at $k_{a_1+b_{s'}}^{[-1]}$ with $\varphi_1(k_{a_1+b_{s'}}^{[-1]})\in(\mathbf{n}^{a_1+b_{s'-1},+}\sqcup\mathbf{n}^{a_1+b_{s'-1},-})\setminus\{k_{a_1+b_{s'-1},0},k_{a_1+b_{s'-1},c_{a_1+b_{s'-1}}}\}$;
\item if $b_0>0$, then $\varphi_1$ has a $-1$-jump at $k_{a_1+b_0}^{[-1]}$ with $\varphi_1(k_{a_1+b_0}^{[-1]})\in(\mathbf{n}^{a_1,+}\sqcup\mathbf{n}^{a_1,-})\setminus (\mathbf{n}^{a_1+1,+}\sqcup\mathbf{n}^{a_1+1,-})$.
\end{itemize}
Note that this $-1$-tour from $k_{[-1]}$ to $k_{[-1]}'$ is uniquely determined by the conditions above. Once we run through all possible choices of the pair $k_{[1]},k_{[1]}'$ as above, we complete the construction of $\varphi_1$. An example of the $-1$-tour from $k_{[-1]}$ to $k_{[-1]}'$ is visualized in Figure~\ref{fig:BJump}. The construction of an orientation permutation $\varphi_2$ which extends $\varphi_1$ is immediate by letting $\varphi_2$ fix $\mathbf{n}_\Sigma\setminus\mathbf{n}_{\varphi_1}$.

It is easy to see that $\varphi_2$ satisfies item~\ref{it: disjoint orbits}, ~\ref{it: end of tour}, and ~\ref{it: unique cover} of Definition~\ref{def: oriented permutation}, from the construction above. Item~\ref{it: degeneration of tour} also trivially holds as such a $\varepsilon$-jump never exist if $\Sigma$ is not circular. It remains to check item~\ref{it: interaction of jumps} in Definition~\ref{def: oriented permutation}. Given $a,k,k',\varepsilon$ as in item~\ref{it: interaction of jumps}, we want to show that $k'>k$. If $\varepsilon=1$, then the construction of $\varphi_0$ (especially the $1$-jump at $k$) forces $k<k_{a+1}^\prime$, which together with $k'\geq k_{a+1}^\prime$ (using the fact that $\varphi_2$ has a $-1$-jump at $k'$ which covers $k_{a,0}$) implies $k'>k$. If $\varepsilon=-1$ and $c_a=1$, then the construction of $\varphi_0$ (especially the $1$-jump at $k'$) forces $k'>k_{a+1,c_{a+1}-1}$, which together with $k\leq k_{a+1,c_{a+1}-1}$ (using the fact that $\varphi_2$ has a $-1$-jump at $k$ which covers $k_{a,c_a}$) implies $k'>k$. If $\varepsilon=-1$ and $c_a\geq 2$, then the construction of $\varphi_0$ (especially the $1$-jump at $k'$) forces $k'>k_{a,c_a-1}$, which together with $k\leq k_{a,c_a-1}$ (using the fact that $\varphi_2$ has a $-1$-jump at $k$ which covers $k_{a,c_a}$) implies $k'>k$. The proof is thus finished.
\end{proof}

\begin{prop}\label{prop:oriented:circular}
Let $\pi_0(\Omega^\pm)=\{\Omega^+\sqcup\Omega^-\}$ and $\Omega^+\sqcup\Omega^-$ is circular. Then there exists an oriented permutation of $\mathbf{n}_{\Omega^+\sqcup\Omega^-}$.
\end{prop}
\begin{proof}
Note that $\Omega^\pm$ is a constructible $\Lambda$-lift of type \rm{III} if and only if so is its inverse (see Definition~\ref{def: separated condition} for inverse). By replacing $\Omega^\pm$ with its inverse, we are simply exchanging $\Omega^+$ and $\Omega^-$, and thus exchanging $(\Z/t)^+$ and $(\Z/t)^-$. Also, the fact that $\Omega^+\sqcup\Omega^-$ is circular clearly remains if we exchange $\Omega^+$ and $\Omega^-$. Upon replacing $\Omega^\pm$ with its inverse, there exists a unique $a_0\in(\Z/t)^-$ such that $k_{a_0,c_{a_0}-1}=\min\{k_{a,c_a-1}\mid a\in\Z/t\}$ and exactly one of the following holds:
\begin{itemize}
\item $k_{a_0,0}=k_{a_0+1,0}$, $k_{a_0}^\prime<k_{a_0+1}^\prime$ and $c_{a_0}\geq 2$;
\item $k_{a_0,0}=k_{a_0+1,0}$ and $k_{a_0}^\prime>k_{a_0+1}^\prime$.
\end{itemize}

The rest of the proof is similar to that of Proposition~\ref{prop:oriented:non circular} and is divided into two steps. We first construct a certain injective map $\varphi_0:~\mathbf{n}_{\varphi_0}\rightarrow\mathbf{n}_\Sigma$ which has a $1$-tour from $k_{a_0,c_{a_0}}$ to itself and with $\mathbf{n}_{\varphi_0}$ minimal possible. We construct $\varphi_0$ by the following inductive procedure. Let $\varphi:~\mathbf{n}_\varphi\rightarrow\mathbf{n}_\Sigma$ an injective map constructed from the previous step, we want to construct another $\varphi':~\mathbf{n}_{\varphi'}\rightarrow\mathbf{n}_\Sigma$ that satisfies $\mathbf{n}_\varphi\subsetneq\mathbf{n}_{\varphi'}$ and $\varphi'|_{\mathbf{n}_\varphi}=\varphi$. If $\varphi$ already has a $1$-tour from $k_{a_0,c_{a_0}}$ to itself, then we set $\varphi_0\defeq \varphi$. Otherwise (by inductive construction) $\varphi$ has a $1$-tour from $k_{a_0,c_{a_0}}$ to some $k\in\mathbf{n}_\Sigma\setminus\mathbf{n}_\varphi$. The construction of $\varphi'$ is parallel to the one in the proof of Proposition~\ref{prop:oriented:non circular} and we can define $b$ similarly. The construction for each case remain the same except the following two cases
\begin{itemize}
\item $k_{a_0,c_{a_0}}$ is the $1$-successor of $k$, and we define $\varphi'$ by $\mathbf{n}_{\varphi'}\sqcup\{k\}$ and $\varphi'(k)=k_{a_0,c_{a_0}}$;
\item $k=k_a^{[1]}$ for some $a\in(\Z/t)_\Sigma$, $k$ does not have $k_{a_0,c_{a_0}}$ as $1$-successor and $b$ does not exist, in which case we define $\varphi'$ as the unique injective map such that $\varphi'$ has a $1$-jump at $k$ with $\varphi'(k)=k_{a_0-1,1}\in\mathbf{n}^{a_0-1,+}$, and $\mathbf{n}_{\varphi'}\supseteq \mathbf{n}_\varphi\sqcup\{k\}$ is minimal possible.
\end{itemize}
We fix the $1$-tour of $\varphi_0$ from $k_{a_0,c_{a_0}}$ to itself in the rest of the proof.

Now that $\varphi_0$ has been defined, we extend $\varphi_0$ to an oriented permutation of $\Sigma$ again by extending $\varphi_0$ to some injective map $\varphi_1$ which also has a $-1$-tour from $k_{a_0,c_{a_0}-1}$ to itself, and with $\mathbf{n}_{\varphi_1}$ minimal possible among all such choices of $\varphi_1$. We can run the same argument as in Proposition~\ref{prop:oriented:non circular}, namely choosing a pair $k_{[1]},k_{[1]}'$, attach with it $a_1,a_1',k_{[-1]},k_{[-1]}'$ and then apply Lemma~\ref{lem: backwards jump} to construct a $-1$-tour of $\varphi_1$ from $k_{[-1]}$ to $k_{[-1]}'$. However, the definition of $a_1,a_1',k_{[-1]},k_{[-1]}'$ is slightly different as $\Sigma$ has neither $1$-end or $-1$-end. In fact, $k_{[1]}$, $a_1$ and $k_{[-1]}'$ satisfy exactly one of the two possibilities
\begin{itemize}
\item $a_1=a_0$, $k_{[1]}=k_{a_0,0}$ and $k_{[-1]}'=k_{a_0,c_{a_0}-1}\neq k_{[1]}$;
\item \begin{itemize}
      \item $\{k_{a_1,0},k_{a_1,c_{a_1}}\}=\{k_{[1]},k_{[-1]}'\}$ and $a_1\neq a_0$;
      \item the fixed $1$-tour of $\varphi_0$ contains a $1$-tour from the $k_{a_0,c_{a_0}}$ to $k_{[1]}$ which contains a $1$-jump that covers $k_{[-1]}'$.
      \end{itemize}
\end{itemize}
Similarly, $k_{[1]}'$, $a_1'$ and $k_{[-1]}$ satisfy exactly one of the two possibilities
\begin{itemize}
\item $a_1'=a_0$, $k_{[1]}'=k_{a_0,c_{a_0}}$ and $k_{[-1]}=k_{a_0,c_{a_0}-1}$;
\item \begin{itemize}
      \item $\{k_{a_1',0},k_{a_1',c_{a_1'}}\}=\{k_{[1]}',k_{[-1]}\}$ and $a_1'\neq a_0$;
      \item the fixed $1$-tour of $\varphi_0$ contains a $1$-tour from $k_{[1]}'$ to the $k_{a_0,c_{a_0}}$ which contains a $1$-jump that covers $k_{[-1]}$.
      \end{itemize}
\end{itemize}
Finally, $\varphi_1$ has a $-1$-tour from $k_{[-1]}$ to $k_{[-1]}'$ characterized by
\begin{itemize}
\item for each $1\leq s'\leq s$, $\varphi_1$ has a $-1$-jump at $k_{a_1+b_{s'}}^{[-1]}$ with $\varphi_1(k_{a_1+b_{s'}}^{[-1]})\in(\mathbf{n}^{a_1+b_{s'-1},+}\sqcup\mathbf{n}^{a_1+b_{s'-1},-})\setminus\{k_{a_1+b_{s'-1},0},k_{a_1+b_{s'-1},c_{a_1+b_{s'-1}}}\}$;
\item if $b_0>0$ and $a_1\neq a_0$, then $\varphi_1$ has a $-1$-jump at $k_{a_1+b_0}^{[-1]}$ with $\varphi_1(k_{a_1+b_0}^{[-1]})\in(\mathbf{n}^{a_1,+}\sqcup\mathbf{n}^{a_1,-})\setminus(\mathbf{n}^{a_1-1,+}\sqcup\mathbf{n}^{a_1-1,-})$;
\item if $b_0>0$ and $a_1=a_0$, then $\varphi_1$ has a $-1$-jump at $k_{a_1+b_0}^{[-1]}$ with $\varphi_1(k_{a_1+b_0}^{[-1]})=k_{a_0,1}\neq k_{a_0,c_{a_0}}$.
\end{itemize}
The construction of $\varphi_1$ is finished by running through all possible choices of the pair $k_{[1]},k_{[1]}'$. We extend $\varphi_1$ to an oriented permutation $\varphi_2$ of $\mathbf{n}_\Sigma$ by setting $\varphi_2(k)\defeq k$ for each $k\in\mathbf{n}_\Sigma\setminus\mathbf{n}_{\varphi_1}$. It is easy to see that $\varphi_2$ satisfies items~\ref{it: disjoint orbits}, ~\ref{it: end of tour}, and ~\ref{it: unique cover} of Definition~\ref{def: oriented permutation}, from the construction above. The same argument as in Proposition~\ref{prop:oriented:non circular} proves that $\varphi_2$ satisfies item~\ref{it: interaction of jumps} of Definition~\ref{def: oriented permutation}.

It remains to check item~\ref{it: degeneration of tour} of Definition~\ref{def: oriented permutation} and we borrow the notation $\varepsilon\in\{1,-1\}$ and $a\in(\Z/t)_\Sigma=\Z/t$ from there. If $\varepsilon=1$, then the construction of $\varphi_0$ forces $a=a_0$, $c_{a_0-1}=1$ and $\varphi_0(k_{a_0}^{[1]})=k_{a_0}^{[1]}=k_{a_0,c_{a_0}}=k_{a_0-1,1}$, which is impossible as $k_{a_0,c_{a_0}}=k_{a_0-1,c_{a_0-1}}\leq k_{a_0-1}^\prime$ contradicts the definition of a $1$-jump at $k_{a_0,c_{a_0}}$. If $\varepsilon=-1$ and $c_{a_0}=1$, then the construction of $\varphi_1$ forces $a=a_0$ (and thus $k_a^{[\varepsilon]}=k_{a_0,0}=\min\{k_{a,c_a-1}\mid a\in \Z/t\}$) and we must have $c_{a_0+1}=1$, $k_{a_0+1}^\prime<k_{a_0}^\prime$ and
$$k_{a_0+2,c_{a_0+2}-1}>k_{a_0,c_{a_0}-1}=k_{a_0+1,c_{a_0+1}-1}=k_{a_0,0}>k_{a_0}^\prime$$
by our choice of $a_0$. This implies that the fixed $1$-tour of $\varphi_2$ contains a $1$-jump at $k_{a_0}^\prime$ that covers $k_{a_0,0}$ and $k_{a_0+1,c_{a_0+1}}$, and thus $\varphi_2(k_{a_0,0})\notin\{k_{a_0,0},k_{a_0,1}\}$ by the construction of $\varphi_1$, which contradicts our assumption on $\varepsilon$ and $a$ in item~\ref{it: degeneration of tour} of Definition~\ref{def: oriented permutation}. Hence we deduce that $\varepsilon=-1$ and $c_{a_0}\geq 2$ which together with the construction of $\varphi_1$ force $a=a_0$, $k_a^{[\varepsilon]}=k_{a_0,c_{a_0}-1}$ and $\varphi(k_a^{[\varepsilon]})=k_{a_0,1}$. The proof is thus finished.
\end{proof}

\subsection{Invariance condition}\label{subsub: invariance}
In this section, we show that our construction of the permutations $v_\cJ^{\Omega^\pm}=(v_j^{\Omega^\pm})_{j\in\cJ}\in\un{W}$ and the subsets $I_\cJ^{\Omega^\pm}\subseteq\mathbf{n}_{\cJ}$ actually gives an invariant function in the sense of (\ref{equ: def of inv fun}), for each constructible $\Lambda$-lift $\Omega^\pm$ of either type $\rm{I}$, type $\rm{II}$, or type $\rm{III}$.

\begin{lemma}\label{lem: invariance of function}
Let $\Omega^\pm$ be a constructible $\Lambda$-lift of either type $\rm{I}$, type $\rm{II}$, or type $\rm{III}$. Then we have
\begin{equation*}
I_\cJ^{\Omega^\pm}\cdot(v_\cJ^{\Omega^\pm},1)=I_\cJ^{\Omega^\pm}.
\end{equation*}
\end{lemma}
\begin{proof}
We only prove it when $\Omega^\pm$ is a constructible $\Lambda$-lift of type $\rm{III}$ as the other cases are similar.

We first treat the case when $\pi_0(\Omega^\pm)=\{\Omega^+\sqcup\Omega^-\}$ and $\Omega^+\sqcup\Omega^-$ is circular. It follows from Condition~$\rm{III}$-\ref{it: III 3} and $\rm{III}$-\ref{it: III 4} of Definition~\ref{def: constructible lifts} that
$(k',j)\cdot (v_\cJ^{\Omega^\pm},1)=(k',j)\cdot (w_\cJ,1)\in I_\cJ^{\Omega^\pm}$ for each $(k',j)\in I_\cJ^{\Omega^\pm}\setminus\left(\mathbf{n}_{\Omega^+\sqcup\Omega^-,1}\times\{j_{\Omega^+\sqcup\Omega^-}\}\right)$. On the other hand, we have
\begin{multline*}
(k,j_{\Omega^+\sqcup\Omega^-})\cdot (v_\cJ^{\Omega^\pm},1)=((v_{\Omega^+\sqcup\Omega^-}^{\Omega^\pm})^{-1}(k),j_{\Omega^+\sqcup\Omega^-})\cdot (w_\cJ,1)\\
\in\, ]((v_{\Omega^+\sqcup\Omega^-}^{\Omega^\pm})^{-1}(k),j_{\Omega^+\sqcup\Omega^-}),((v_{\Omega^+\sqcup\Omega^-}^{\Omega^\pm})^{-1}(k),j_{\Omega^+\sqcup\Omega^-})]_{w_\cJ}\subseteq I_\cJ^{\Omega^\pm}
\end{multline*}
for each $k\in\mathbf{n}_{\Omega^+\sqcup\Omega^-,1}$. Hence we finish the proof in this case.

Now we treat the case when $\pi_0(\Omega^\pm)=\{\Sigma_{h'}\mid h'\in\Z/h\}$ and $\Omega^\pm$ do not have circular connected component. It still follows from Condition~$\rm{III}$-\ref{it: III 3} and $\rm{III}$-\ref{it: III 4} of Definition~\ref{def: constructible lifts} that $(k',j)\cdot (v_\cJ^{\Omega^\pm},1)=(k',j)\cdot (w_\cJ,1)\in I_\cJ^{\Omega^\pm}$ for each
$(k',j)\in  I_\cJ^{\Omega^\pm}\setminus \left(\bigcup_{\Sigma\in\pi_0(\Omega^\pm)}\mathbf{n}_{\Sigma,1}\times\{j_\Sigma\}\right)$. If $k\in\mathbf{n}_{\Sigma,1}\setminus\{v_\Sigma^{\Omega^\pm}(k_\Sigma^\prime)\}$ for some $\Sigma\in\pi_0(\Omega^\pm)$, then we have
$$(k,j_\Sigma)\cdot (v_\cJ^{\Omega^\pm},1)=((v_\Sigma^{\Omega^\pm})^{-1}(k),j_\Sigma)\cdot (w_\cJ,1)\in\,]((v_\Sigma^{\Omega^\pm})^{-1}(k),j_\Sigma),((v_\Sigma^{\Omega^\pm})^{-1}(k),j_\Sigma)]_{w_\cJ}\subseteq I_\cJ^{\Omega^\pm}.$$
If $\Sigma=\Sigma_{h'}$ for some $h'\in\Z/h$ and $k=v_{\Sigma_{h'}}^{\Omega^\pm}(k_{\Sigma_{h'}}^\prime)$, then we have
$$(v_{\Sigma_{h'}}^{\Omega^\pm}(k_{\Sigma_{h'}}^\prime),j_{\Sigma_{h'}})\cdot (v_\cJ^{\Omega^\pm},1)=(k_{\Sigma_{h'}}^\prime,j_{\Sigma_{h'}})\cdot (w_\cJ,1)\in\,](k_{\Sigma_{h'}}^\prime,j_{\Sigma_{h'}}),(k_{\Sigma_{h'+1}},j_{\Sigma_{h'+1}})]_{w_\cJ}\subseteq I_\cJ^{\Omega^\pm}.$$
Up to this stage, we have shown that $(k',j)\cdot (v_\cJ^{\Omega^\pm},1)\in I_\cJ^{\Omega^\pm}$ for each $(k',j)\in I_\cJ^{\Omega^\pm}$. The proof is thus finished.
\end{proof}

\newpage
\section{Invariant functions and constructible $\Lambda$-lifts}
\label{sec:inv:cons}
We fix $w_\cJ\in\un{W}$, $\xi\in\Xi_{w_\cJ}$ and a subset $\Lambda\subseteq\mathrm{Supp}_{\xi,\cJ}$ throughout this section. In this section, we use the invariant functions $f_\xi^{\Omega^\pm}\in\Inv$ constructed in \S\,\ref{sec:const:inv} to prove a list of results stated in~\S\,\ref{sub:exp:for}, whose proof will be given in \S\,\ref{sub: exp type I}, \S\,\ref{sub: exp type II}, and \S\,\ref{sub: exp type III}) when $\Omega^\pm$ is a constructible $\Lambda$-lift of type~\rm{I}, of type~\rm{II}, and of type~\rm{III}, respectively. Finally, we combine the results in \S\,\ref{sub:exp:for} with that of \S\,\ref{sub: cons lifts} to complete the proof of Statement~\ref{state: goal} in Theorem~\ref{thm: constructible and inv fun} and Corollary~\ref{cor: separate points}.

\subsection{Explicit invariant functions: statements}
\label{sub:exp:for}
We fix an element $\cC\in\cP_\cJ$ satisfying $\cC\subseteq\cN_{\xi,\Lambda}$ and recall the subring $\cO_\cC\subseteq\cO(\cC)$ from Definition~\ref{def: similar and below a block}. We state here a list of crucial ingredients for the proof of Theorem~\ref{thm: constructible and inv fun} and Corollary~\ref{cor: separate points}. Some rough idea behind these results is summarized in Remark~\ref{rmk: rough idea}.

The following two propositions are for constructible $\Lambda$-lifts of type $\rm{I}$. The proofs of these two propositions will occupy \S\,\ref{sub: exp type I}.
\begin{prop}\label{prop: type I exceptional}
Let $\Omega^\pm$ be a constructible $\Lambda$-lift of type $\rm{I}$, and assume that $\Omega^+$ is $\Lambda$-exceptional. Then we have
$$F_\xi^{\Omega^\pm}|_\cC+\sum_{\Omega_0^\pm}\varepsilon(\Omega_0^\pm)F_\xi^{\Omega_0^\pm}|_\cC\in\cO_\cC$$
where $\Omega_0^\pm$ runs through balanced pairs satisfying $\Omega^+<\Omega_0^+$ and $\Omega^-=\Omega_0^-$ with $\varepsilon(\Omega_0^\pm)\in\{-1,1\}$ a sign determined by $\Omega^\pm$ and $\Omega_0^\pm$.
\end{prop}

\begin{prop}\label{prop: type I extremal}
Let $\Omega^\pm$ be a constructible $\Lambda$-lift of type $\rm{I}$, and assume that $\Omega^+$ is $\Lambda$-extremal. Then we have
$$F_\xi^{\Omega^\pm}|_\cC\in\cO_\cC\,.$$
\end{prop}

We now state the results for constructible $\Lambda$-lifts of type $\rm{II}$, whose proof will occupy \S\,\ref{sub: exp type II}.
\begin{prop}\label{prop: type II exceptional}
Let $\Omega^\pm$ be a constructible $\Lambda$-lift of type $\rm{II}$, and assume that $\Omega^+$ is $\Lambda$-exceptional. Then we have
$$F_\xi^{\Omega^\pm}|_\cC+\sum_{\Omega_0^\pm}\varepsilon(\Omega_0^\pm)F_\xi^{\Omega_0^\pm}|_\cC\in\cO_\cC$$
where $\Omega_0^\pm$ runs through balanced pairs satisfying $\Omega^+<\Omega_0^+$ and $\Omega^-=\Omega_0^-$ with $\varepsilon(\Omega_0^\pm)\in\{-1,1\}$ a sign determined by $\Omega^\pm$ and $\Omega_0^\pm$.
\end{prop}

\begin{prop}\label{prop: type II extremal}
Let $\Omega^\pm$ be a constructible $\Lambda$-lift of type $\rm{II}$, and assume that $\Omega^+$ is $\Lambda$-extremal. Then we have
$$F_\xi^{\Omega^\pm}|_\cC\in\cO_\cC\,.$$
\end{prop}

Finally, we state the result for constructible $\Lambda$-lifts of type $\rm{III}$, whose proof will occupy \S\,\ref{sub: exp type III}, after fixing some notation. We define $\cO_{\xi,\Lambda}^{\rm{ps}}$ as the subgroup of $\cO(\cN_{\xi,\Lambda})^\times$ generated by $\cO_{\xi,\Lambda}^{\rm{ss}}$ and $F_\xi^{\Omega^\pm}$ for all balanced pairs with both $\Omega^+$ and $\Omega^-$ being pseudo $\Lambda$-decompositions of some $(\al,j)\in\widehat{\Lambda}$. We write $\cO_\cC^{\rm{ps}}$ and $\cO_\cC^{<\delta}$ for the restriction of $\cO_{\xi,\Lambda}^{\rm{ps}}$ and $\cO_{\xi,\Lambda}^{<\delta}$ to $\cC$ respectively (for each $\delta\in\N\Lambda^\square$). For each subset $Y\subseteq\cO(\cC)$, we write $\langle Y\rangle$ for the subring of $\cO(\cC)$ generated by $Y$, and write $\langle Y\rangle_+$ for the localization of $\langle Y\rangle$ with respect to $\langle Y\rangle\cap\cO(\cC)^\times$.

\begin{prop}\label{prop: type III}
Let $\Omega^\pm$ be a constructible $\Lambda$-lift of type $\rm{III}$. If both $\Omega^+$ and $\Omega^-$ are pseudo $\Lambda$-decomposition of some $(\al,j)\in\widehat{\Lambda}$, then we have
$$F_\xi^{\Omega^\pm}|_\cC\in\cO_\cC\,.$$
Otherwise, we have
$$F_\xi^{\Omega^\pm}|_\cC\in\langle\cO_\cC^{\rm{ps}}\cdot\cO_\cC^{<|\Omega^\pm|}\cdot\cO_\cC\rangle_+.$$
\end{prop}

\begin{rmk}\label{rmk: rough idea}
The main idea behind Proposition~\ref{prop: type I exceptional}, Proposition~\ref{prop: type I extremal}, Proposition~\ref{prop: type II exceptional}, Proposition~\ref{prop: type II extremal} and Proposition~\ref{prop: type III} is to compute the restriction $f_\xi^{\Omega^\pm}|_\cC$ explicitly for each constructible $\Lambda$-lift and each $\cC\in\cP_\cJ$ satisfying $\cC\subseteq\cN_{\xi,\Lambda}$. However, the subtlety is that we do not always have $f_\xi^{\Omega^\pm}\in\Inv(\cC)$ and the restriction $f_\xi^{\Omega^\pm}|_\cC$ might not make sense. In the proof of the results above, we actually know exactly when $f_\xi^{\Omega^\pm}\in\Inv(\cC)$ holds, and even if $f_\xi^{\Omega^\pm}\notin\Inv(\cC)$, we can still prove the same technical results stated as above, which is sufficient for our application in \S\,\ref{sub:main:criterions}.
\end{rmk}

\subsection{Explicit determinants}\label{sub:exp:det}
Before starting the proof of the propositions in \S\,\ref{sub:exp:for}, we need an elementary result (see Lemma~\ref{lem: general formula for det}) on explicit formula for determinants of various submatrices of an upper-triangular matrix.

Given a pair of subsets
\begin{equation}\label{equ: pair of subsets}
\mathbf{I}=\{i_1<\cdots<i_h\}\,\,\mbox{ and }\,\, \mathbf{I}^\prime=\{i^\prime_1<\cdots<i^\prime_h\}\subseteq\mathbf{n},
\end{equation}
we associate the element
$$\al_{\mathbf{I},\mathbf{I}^\prime}\defeq \sum_{s=1}^h(i_s,i_s^\prime)$$
in the root lattice, where $(i_s,i_s^\prime)$ is understood to be the zero element in the root lattice if $i_s=i_s^\prime$. Note that we have an identity $\al_{\mathbf{I},\mathbf{I}^\prime}=\sum_{s=1}^h(i_s,\sigma(i_s))$ for any bijection $\sigma:\mathbf{I}\rightarrow \mathbf{I}^{\prime}$. We write $\sigma_{\mathbf{I},\mathbf{I}^\prime}:~\mathbf{I}\rightarrow\mathbf{I}^\prime$ for the bijection that sends $i_s$ to $i_s^\prime$ for each $1\leq s\leq h$. We say that $\mathbf{I}$ is \emph{lower than} $\mathbf{I}^\prime$, written as $\mathbf{I}\leq \mathbf{I}^\prime$, if $i_s\leq i^\prime_s$ for all $1\leq s\leq h$.  We notice that $\al_{\mathbf{I},\mathbf{I}^\prime}$ lies in the submonoid of the root lattice generated by $\Phi^+$, if $\mathbf{I}\leq\mathbf{I}^\prime$. We also note that $\al_{\mathbf{I},\mathbf{I}^\prime}=0$ if and only if $\mathbf{I}=\mathbf{I}^\prime$.

\begin{defn}\label{def: decomposition into roots}
Let $\mathbf{I},\mathbf{I}^\prime\subseteq\mathbf{n}$ be a pair of subsets (\ref{equ: pair of subsets}) with associated element $\al_{\mathbf{I},\mathbf{I}^\prime}$ in the root lattice.
A subset $\Omega\subset\Phi^+$ is called an \emph{$(\mathbf{I},\mathbf{I}^\prime)$-indexed decomposition of $\alpha_{\bf{I},\bf{I}'}$}, or an \emph{$(\mathbf{I},\mathbf{I}^\prime)$-indexed decomposition} for short, if the following holds:
\begin{itemize}
\item $i_\al\in \mathbf{I}$ and $i_\al^\prime\in\mathbf{I}^\prime$ for each $\al=(i_\al,i_\al^\prime)\in\Omega$;
\item $\al_{\mathbf{I},\mathbf{I}^\prime}=\sum_{\al\in\Omega}\al$;
\item the map $\Omega\rightarrow \mathbf{I}$ sending $\al\mapsto i_\al$ and the map $\Omega\rightarrow \mathbf{I}'$ sending $\al\mapsto i_\al'$ are both injective.
\end{itemize}
For an arbitrary subset $\Theta\subseteq \Phi^+$, an $(\mathbf{I},\mathbf{I}^\prime)$-indexed decomposition $\Omega$ of $\alpha_{\bf{I},\bf{I}'}$ is said to be \emph{supported in $\Theta$} if $\Omega\subseteq\Theta$. Roughly speaking, an $(\mathbf{I},\mathbf{I}^\prime)$-indexed decomposition supported in $\Theta$ is simply one way to decompose $\al_{\mathbf{I},\mathbf{I}^\prime}$ into a sum of certain elements in $\Theta$. We use the convention that $\emptyset$ is an $(\mathbf{I},\mathbf{I})$-indexed decomposition, for each $\mathbf{I}\subseteq\mathbf{n}$.
\end{defn}

Note that an $(\mathbf{I},\mathbf{I}^\prime)$-indexed decomposition does not always exist (cf.~Lemma~\ref{lem: lower vers index}). For each $(\mathbf{I},\mathbf{I}^\prime)$-indexed decomposition $\Omega$, we consider the subset $\mathbf{J}_{\Omega}\subseteq\mathbf{n}$ uniquely determined by the property
$$\mathbf{I}=\mathbf{J}_{\Omega}\sqcup \{i_\al\mid \al\in\Omega\}\textnormal{ and }\mathbf{I}^\prime=\mathbf{J}_{\Omega}\sqcup\{i_\al^\prime\mid \al\in\Omega\}.$$
Note that $\mathbf{J}_{\Omega}$ exists since $$\al_{\mathbf{I}\setminus\{i_\al\mid \al\in\Omega\},\mathbf{I}'\setminus\{i_\al^\prime\mid \al\in\Omega\}}=\al_{\mathbf{I},\mathbf{I}^\prime}-\sum_{\al\in\Omega}\al=0.$$ There exists a bijection $\sigma_{\Omega}:~\mathbf{I}\rightarrow \mathbf{I}^\prime$ that sends $i_\al$ to $i_\al^\prime$ for each $\al\in\Omega$ and restricts to the identity on $\mathbf{J}_{\Omega}$. Hence, $\sigma_{\mathbf{I},\mathbf{I}^\prime}\sigma_{\Omega}^{-1}$ is a permutation of $\mathbf{I}^\prime$ and we write $\varepsilon_{\mathbf{I},\mathbf{I}^\prime}(\Omega)\in\{1,-1\}$ for its sign.

\begin{lemma}\label{lem: lower vers index}
We have $\mathbf{I}\leq \mathbf{I}^\prime$ if and only if there exists an $(\mathbf{I},\mathbf{I}^\prime)$-indexed decomposition.
\end{lemma}
\begin{proof}
If $\mathbf{I}\leq \mathbf{I}^\prime$, then we can choose an obvious $(\mathbf{I},\mathbf{I}^\prime)$-indexed decomposition to be
$$\{(i_s,i_s^\prime)\mid i_s<i_s^\prime\}.$$
Conversely, assume that there exists an $(\mathbf{I},\mathbf{I}^\prime)$-indexed decomposition, called $\Omega$, from which we obtain a map $\sigma_{\Omega}:~\mathbf{I}\rightarrow \mathbf{I}^\prime$ as above.
The choice of $\Omega$ is equivalent to the choice of
\begin{equation}\label{equ: pair vers index}
\{(i,\sigma_{\Omega}(i))\mid i\in\mathbf{I}\}\subseteq\mathbf{I}\times\mathbf{I}^\prime.
\end{equation}
If there exists $i<i^\prime\in \mathbf{I}$ such that $i<i^\prime\leq \sigma_{\Omega}(i^\prime)<\sigma_{\Omega}(i)$, then we replace the elements $(i,\sigma_{\Omega}(i))$, $(i^\prime,\sigma_\Sigma(i^\prime))$ in (\ref{equ: pair vers index}) with $(i,\sigma_{\Omega}(i^\prime))$, $(i^\prime,\sigma_\Sigma(i))$, and hence obtain another subset of $\mathbf{I}\times\mathbf{I}^\prime$ which corresponds to a new $(\mathbf{I},\mathbf{I}^\prime)$-indexed decomposition. We can repeat this procedure until we have $\sigma_{\Omega}(i)<\sigma_{\Omega}(i^\prime)$ for each $i<i^\prime$, which exactly means $\mathbf{I}\leq \mathbf{I}^\prime$.
\end{proof}

Let $R$ be a Noetherian $\F$-algebra. For each $A\in B(R)$, we have a unique decomposition $A=A^\prime A^{\prime\prime}$ with $A^\prime\in T(R)$ and $A^{\prime\prime}\in U(R)$. We write $D_i(A)$ for the $i$-th diagonal entry of $A^\prime$ (hence of $A$ as well) and $u_\al(A)$ for the $\al$-entry of $A^{\prime\prime}$ (hence the $\al$-entry of $A$ is $D_{i_\al}(A)u_\al(A)$). For each pair of subsets (\ref{equ: pair of subsets}), we write $\mathrm{det}_{\mathbf{I},\mathbf{I}^\prime}(A)\defeq \mathrm{det}(A_{\mathbf{I},\mathbf{I}^\prime})$ where $A_{\mathbf{I},\mathbf{I}^\prime}$ is the submatrix of $A$ given by $\mathbf{I}$-th rows and $\mathbf{I}^\prime$-th columns. Hence, we obtain the elements $D_i$, $u_\al$ and $\mathrm{det}_{\mathbf{I},\mathbf{I}^\prime}$ in the ring of global sections of $B$. %

We have the following formula of determinant.
\begin{lemma}\label{lem: general formula for det}
If $\mathbf{I}\leq \mathbf{I}^\prime$, then we have
\begin{equation}\label{equ: std formula}
\mathrm{det}_{\mathbf{I},\mathbf{I}^\prime}=\left(\prod_{i\in\mathbf{I}} D_i\right)\sum_{\Omega}\varepsilon_{\mathbf{I},\mathbf{I}^\prime}(\Omega)\left(\prod_{\al\in\Omega}u_\al\right)
\end{equation}
where $\Omega$ runs through all $(\mathbf{I},\mathbf{I}^\prime)$-indexed decompositions. %
\end{lemma}
\begin{proof}
For each $A\in B(R)$, we write $A=A^\prime A^{\prime\prime}$ with $A^\prime\in T(R)$ and $A^{\prime\prime}\in U(R)$. We first observe that
$$\mathrm{det}_{\mathbf{I},\mathbf{I}^\prime}(A)=\left(\prod_{i\in\mathbf{I}} D_i(A)\right)\mathrm{det}_{\mathbf{I},\mathbf{I}^\prime}(A^{\prime\prime}).$$
Then the formula (\ref{equ: std formula}) reduces to the formula of $\mathrm{det}_{\mathbf{I},\mathbf{I}^\prime}(A^{\prime\prime})$, which follows directly from definition of determinant and the fact that the only possibly non-zero entries of $A^{\prime\prime}$ are $1$ on the diagonal and $u_\al(A^{\prime\prime})=u_\al(A)$ for some $\al\in\Phi^+$. The proof is thus finished.
\end{proof}

\subsection{Data associated with a constructible $\Lambda$-lift}\label{sub: notation for each k j}
In this section, we apply Lemma~\ref{lem: general formula for det} to prove Lemma~\ref{lem: from sets to formula} which gives a criterion for $f_\xi^{\Omega^\pm}$ to be regular over $\cN_{\xi,\Lambda}$ as well as an explicit formula for $f_\xi^{\Omega^\pm}|_{\cN_{\xi,\Lambda}}$.
We recall the definitions of $D_{\xi,\ell}^{(j)}$ and $u_\xi^{(\al,j)}$ from \S\,\ref{sub:std:note}.

Let $\Omega^\pm$ be a constructible $\Lambda$-lift (cf.~Definition~\ref{def: constructible lifts}). We have associated with $\Omega^\pm$ an element $v_\cJ^{\Omega^\pm}=(v_j^{\Omega^\pm})_{j\in\cJ}\in\un{W}$ and a subset $I_\cJ^{\Omega^\pm}\subseteq\mathbf{n}_\cJ$ in \S\,\ref{sub: type I}, \S\,\ref{sub: type II}, and \S\,\ref{sub: type III}. We write $S^{j,\Omega^\pm}_\bullet$ be the sequence corresponding to $v_j^{\Omega^\pm}$ for each $j\in\cJ$ via (\ref{eq:bjc:sis}).
We recall from (\ref{equ: std section}) that $\prod_{j\in\cJ} T N_{\xi,\Lambda,j}^- w_j$ is a standard lift of $\cN_{\xi,\Lambda}$ into $\un{G}$. If $R$ is Noetherian $\F$-algebra, then we write $A=(A^{(j)})_{j\in\cJ}$ for an arbitrary matrix in $\prod_{j\in\cJ} T N_{\xi,j}^- w_j (R)$. We define
$$\mathbf{I}^{\Omega^\pm}_{k,j}\defeq u_j^{-1}\left(\{k,\cdots,n\}\right)\,\,\,\mbox{ and }\,\,\,\mathbf{I}^{\Omega^\pm,\prime}_{k,j}\defeq u_j^{-1}w_j(v_j^{\Omega^\pm})^{-1}\left(\{k,\cdots,n\}\right)$$
and write $\al_{k,j}^{\Omega^\pm}\defeq \al_{\mathbf{I}^{\Omega^\pm}_{k,j},\mathbf{I}^{\Omega^\pm,\prime}_{k,j}}\in\Z\Phi^+$ for the element associate with the pair of subsets $\mathbf{I}^{\Omega^\pm}_{k,j},~\mathbf{I}^{\Omega^\pm,\prime}_{k,j}\subseteq\mathbf{n}$. It is easy to see that
\begin{equation}\label{equ: difference at k}
\al_{k,j}^{\Omega^\pm}=\al_{k+1,j}^{\Omega^\pm}+(u_j^{-1}(k),u_j^{-1}w_j(v_j^{\Omega^\pm})^{-1}(k)),
\end{equation}
and, in particular, we have $\al_{k,j}^{\Omega^\pm}=\al_{k+1,j}^{\Omega^\pm}$ if $(v_j^{\Omega^\pm})^{-1}(k)=w_j^{-1}(k)$.
For each $(k,j)\in\mathbf{n}_\cJ$, we define
$$\mathbf{D}^{\Omega^\pm}_{k,j}\defeq\left\{(\mathbf{I}^{\Omega^\pm}_{k,j}, \mathbf{I}^{\Omega^\pm,\prime}_{k,j})\mbox{-indexed decompositions supported in }\{\beta\in\Phi^+\mid (\beta,j)\in\Lambda\}\right\}\times\{j\}.$$
We also define
$$I_\cJ^{\Omega^\pm,\star}\defeq \left\{(k,j)\in I_\cJ^{\Omega^\pm} \mid \mathbf{D}^{\Omega^\pm}_{k,j}\neq \mathbf{D}^{\Omega^\pm}_{k+1,j}\right\}\subseteq I_\cJ^{\Omega^\pm}.$$
For each subset $\Omega\subseteq\Lambda$, we use the shortened notation
$$F_\xi^\Omega\defeq \underset{(\beta,j)\in\Omega}{\prod}u_\xi^{(\beta,j)}.$$

\begin{lemma}\label{lem: from sets to formula}
Let $\Omega^\pm$ be a constructible $\Lambda$-lift, $\cC\in\cP_\cJ$ be an element satisfying $\cC\subseteq\cN_{\xi,\Lambda}$, and $(k_\star,j_\star)\in I_\cJ^{\Omega^\pm,\star}$. Assume moreover that
\begin{itemize}
\item for each $(k,j)\in \mathbf{n}_\cJ$, we have $f_{S^{j,\Omega^\pm}_{k},j}|_\cC\neq 0$;
\item for each $(k,j)\in I_\cJ^{\Omega^\pm,\star}$, we have $\mathbf{D}^{\Omega^\pm}_{k+1,j}=\{\Omega_{k+1,j}\}$ for some $\Omega_{k+1,j}\subseteq\Lambda\cap\mathrm{Supp}_{\xi,j}$;
\item for each $(k,j)\in I_\cJ^{\Omega^\pm,\star}\setminus\{(k_\star,j_\star)\}$, we have $\mathbf{D}^{\Omega^\pm}_{k,j}=\{\Omega_{k,j}\}$ for some $\Omega_{k,j}\subseteq\Lambda\cap\mathrm{Supp}_{\xi,j}$.
\end{itemize}
Then we have $f_\xi^{\Omega^\pm}\in \Inv(\cC)$ and
\begin{equation*}
f_\xi^{\Omega^\pm}|_\cC\sim \frac{\underset{(k,j)\in I_\cJ^{\Omega^\pm,\star}\setminus\{(k_\star,j_\star)\}}{\prod}F_\xi^{\Omega_{k,j}}|_\cC}{\underset{(k,j)\in I_\cJ^{\Omega^\pm,\star}}{\prod}F_\xi^{\Omega_{k+1,j}}|_\cC}\cdot\underset{\Omega\in \mathbf{D}^{\Omega^\pm}_{k_\star,j_\star}}{\sum}\varepsilon(\Omega)F_\xi^{\Omega}|_\cC
\end{equation*}
where $\varepsilon(\Omega)\in\{1,-1\}$ is a sign determined by $\Omega$ for each $\Omega\in \mathbf{D}^{\Omega^\pm}_{k_\star,j_\star}$.
\end{lemma}
\begin{proof}
As $f_{S^{j,\Omega^\pm}_{k},j}|_\cC\neq 0$ for each $(k,j)\in \mathbf{n}_\cJ$, it is clear that $f_\xi^{\Omega^\pm}\in \Inv(\cC)$. It follows directly from our assumption and Lemma~\ref{lem: general formula for det} that
\begin{itemize}
\item for each $(k,j)\in I_\cJ^{\Omega^\pm}\setminus I_\cJ^{\Omega^\pm,\star}$, $f_{S^{j,\Omega^\pm}_{k},j}|_{\cN_{\xi,\Lambda}}\sim f_{S^{j,\Omega^\pm}_{k+1},j}|_{\cN_{\xi,\Lambda}}$;
\item for each $(k,j)\in I_\cJ^{\Omega^\pm,\star}$, $f_{S^{j,\Omega^\pm}_{k+1},j}|_{\cN_{\xi,\Lambda}}\sim F_\xi^{\Omega_{k+1,j}}$;
\item for each $(k,j)\in I_\cJ^{\Omega^\pm,\star}\setminus\{(k_\star,j_\star)\}$, $f_{S^{j,\Omega^\pm}_{k},j}|_{\cN_{\xi,\Lambda}}\sim F_\xi^{\Omega_{k,j}}$;
\item the global section $f_{S^{j,\Omega^\pm}_{k_\star},j_\star}\in\cO(\tld{\cF\cL}_\cJ)$ satisfies $f_{S^{j,\Omega^\pm}_{k_\star},j_\star}|_{\cN_{\xi,\Lambda}}\sim \sum_{\Omega\in \mathbf{D}^{\Omega^\pm}_{k_\star,j_\star}}\varepsilon(\Omega)F_\xi^{\Omega}$
    where $\varepsilon(\Omega)\in\{1,-1\}$ is a sign determined by $\Omega$, for each $\Omega\in \mathbf{D}^{\Omega^\pm}_{k_\star,j_\star}$.
\end{itemize}
The lemma follows directly from the above formulas by further restriction to $\cC$.
\end{proof}

Thanks to Lemma~\ref{lem: from sets to formula}, the proof of the propositions in \S\,\ref{sub:exp:for} can be completely reduced to the study of the set $\mathbf{D}^{\Omega^\pm}_{k,j}$ for each $(k,j)\in\mathbf{n}_\cJ$, which will be done in \S\,\ref{sub: exp type I}, \S\,\ref{sub: exp type II} and \S\,\ref{sub: exp type III}.

\subsection{Explicit formula: type~\rm{I}}\label{sub: exp type I}
In this section, we explicitly write down the set $\mathbf{D}^{\Omega^\pm}_{k,j}$ for each $(k,j)\in\mathbf{n}_\cJ$ when the $\Omega^\pm$ is a constructible $\Lambda$-lift of type $\rm{I}$. Consequently, we apply Lemma~\ref{lem: from sets to formula} and finish the proofs of Proposition~\ref{prop: type I exceptional} and Proposition~\ref{prop: type I extremal}. We will frequently use all the notation from \S\,\ref{sub: type I}, \S\,\ref{sub: notation for each k j} and the beginning of \S\,\ref{sec:const:inv}.

We start this section with the following elementary lemma, which will be frequently used throughout this section.
\begin{lemma}\label{lem: existence of partition}
Let $\Omega^\pm$ be a constructible $\Lambda$-lift of type $\rm{I}$ with $\Omega^+\sqcup\Omega^-\subseteq\mathrm{Supp}_{\xi,j}$ for some $j\in\cJ$. Assume that there exist a pair of elements $(\beta_1,j),(\beta_2,j)\in\widehat{\Lambda}$ together with $(k,j)\in\mathbf{n}_\cJ$ such that
\begin{itemize}
\item $i_{\beta_1}\neq i_{\beta_2}$, $i_{\beta_1}^\prime\neq i_{\beta_2}^\prime$;%
\item $\al^{\Omega^\pm}_{k,j}=\beta_1+\beta_2$;
\item there does not exist $\Omega^{\prime\prime}\in\mathbf{D}_{((i_{\beta_1},i_{\beta_2}^\prime),j),\Lambda}$ such that $u_j(i_{\Omega^{\prime\prime},1})\geq k$.
\end{itemize}
Then for each $\Omega'\in\mathbf{D}^{\Omega^\pm}_{k,j}$, there exists a partition $\Omega'=\Omega'_1\sqcup\Omega'_2$ such that
$$\sum_{(\beta,j)\in\Omega'_1}\beta=\beta_1\,\,\mbox{and}\,\,\sum_{(\beta,j)\in\Omega'_2}\beta=\beta_2.$$
\end{lemma}

\begin{proof}
Let $\Omega'\in\mathbf{D}^{\Omega^\pm}_{k,j}$ be an arbitrary element. As we have $\al^{\Omega^\pm}_{k,j}=\sum_{(\beta,j)\in\Omega'}\beta$, we must have
$$\mathbf{I}_{\Omega'}\setminus\mathbf{I}_{\Omega'}^\prime=\{(i_{\beta_1},j),~(i_{\beta_2},j)\}\,\,\mbox{and}\,\,\mathbf{I}_{\Omega'}^\prime\setminus\mathbf{I}_{\Omega'}=\{(i_{\beta_1}^\prime,j),~(i_{\beta_2}^\prime,j)\}.$$
Hence if $\Omega'$ does not admit the desired partition, then there must exist a partition $\Omega'=\Omega^{\prime\prime}_1\sqcup\Omega^{\prime\prime}_2$ such that $\Omega^{\prime\prime}_1\in\mathbf{D}_{((i_{\beta_1},i_{\beta_2}^\prime),j),\Lambda}$ and $\Omega^{\prime\prime}_2\in\mathbf{D}_{((i_{\beta_2},i_{\beta_1}^\prime),j),\Lambda}$. Moreover, as $\Omega'\in\mathbf{D}^{\Omega^\pm}_{k,j}$, we necessarily have $u_j(i_{\Omega^{\prime\prime}_1,1}),~u_j(i_{\Omega^{\prime\prime}_2,1})\geq k$, and thus $\Omega^{\prime\prime}_1$ contradicts our assumption.
The proof is thus finished.
\end{proof}

\subsubsection{Proof of Proposition~\ref{prop: type I exceptional}}\label{sub: exp type I exceptional}

Given a constructible $\Lambda$-lift $\Omega^\pm$ of type $\rm{I}$ with $\Omega^+$ being $\Lambda$-exceptional, we set $\Omega_{2,k,j}\defeq \emptyset$ if $j\neq j_1$, and define $\Omega_{2,k,j_1}$ as
$$
\Omega_{2,k,j_1}\defeq
\left\{\begin{array}{cl}
\Omega_{\psi_2,k}&\hbox{if $k>k_{1,c_1-1}$};\\
\emptyset&\hbox{if $k\leq k_{1,c_1-1}$}.
\end{array}\right.
$$
Similarly, we set $\Omega_{1,k,j}\defeq \emptyset$ if $j\neq j_1$, and define $\Omega_{1,k,j_1}$ as
$$
\Omega_{1,k,j_1}\defeq
\left\{\begin{array}{cl}
\Omega_{\psi_1,k}\setminus\{((i_{1,0},i_{1,1}),j_1)\}&\hbox{if $k>k_2^{1,1}$};\\
\Omega_{\psi_1,k}&\hbox{if $k\leq k_2^{1,1}$}.
\end{array}\right.$$

For each $(k,j)\in\mathbf{n}_\cJ$, we write $\al^{\Omega^\pm}_{1,k,j}\defeq \sum_{(\beta,j_1)\in\Omega_{1,k,j}}\beta$ and $\al^{\Omega^\pm}_{2,k,j}\defeq \sum_{(\beta,j_1)\in\Omega_{2,k,j}}\beta$.
It follows from the definition of $\Omega_{1,k,j}$ and $\Omega_{2,k,j}$ above that $\al^{\Omega^\pm}_{1,k,j},~\al^{\Omega^\pm}_{2,k,j}\in\Phi^+\sqcup\{0\}$. Hence we can write $\al^{\Omega^\pm}_{a,k,j}=(i_{a,k,j},i_{a,k,j}^\prime)$ for each $a=1,2$ and $(k,j)\in\mathbf{n}_\cJ$ with $\al^{\Omega^\pm}_{a,k,j}\neq 0$.

\begin{lemma}\label{lem: std type I exceptional}
Let $\Omega^\pm$ be a constructible $\Lambda$-lift of type $\rm{I}$, and assume that $\Omega^+$ is $\Lambda$-exceptional. Then we have $\Omega_{1,k,j}\cap\Omega_{2,k,j}=\emptyset$ for each $(k,j)\in\mathbf{n}_\cJ$. Moreover, we have
$$\mathbf{D}^{\Omega^\pm}_{k,j_1}=\{\Omega^+\}\sqcup\{\Omega\in\mathbf{D}_{(\al_1,j_1),\Lambda}\mid \Omega^+<\Omega\}$$
for each $k\in\mathbf{n}$ satisfying $\al^{\Omega^\pm}_{k,j_1}=\al_1=(i_{1,0},i_{1,c_1})$, and $\mathbf{D}^{\Omega^\pm}_{k,j}=\{\Omega_{1,k,j}\sqcup\Omega_{2,k,j}\}$ for other choices of $(k,j)\in\mathbf{n}_\cJ$.
\end{lemma}
\begin{proof}
Assume that there exists an element $(\beta,j)\in \Omega_{1,k,j}\cap\Omega_{2,k,j}$ for some $(k,j)\in\mathbf{n}_\cJ$. According to the definition of $\Omega_{1,k,j}$ and $\Omega_{2,k,j}$ above, we necessarily have $j=j_1$ and there exist $1\leq c_1'\leq c_1$ and $1\leq c_2'\leq c_2$ such that $i_\beta^\prime=i_{1,c_1'}=i_{2,c_2'}$, which implies that $i_\beta^\prime=i_{1,c_1}=i_{2,c_2}$ as $\Omega^\pm$ is a $\Lambda$-lift. Then we observe that $(i_{1,c_1},j_1)\in\mathbf{I}_{\Omega_{1,k,j_1}}^\prime$ implies $k\leq k_{1,c_1-1}$. On the other hand, $(i_{1,c_1},j_1)\in\mathbf{I}_{\Omega_{2,k,j_1}}^\prime$ implies $k> k_{1,c_1-1}$. This contradicts the existence of $(\beta,j_1)$. Hence $\Omega_{1,k,j}\cap\Omega_{2,k,j}=\emptyset$ for each $(k,j)\in\mathbf{n}_\cJ$.

As we clearly have $\al^{\Omega^\pm}_{k,j}=0$ and thus $\mathbf{D}^{\Omega^\pm}_{k,j}=\{\emptyset\}$ if $j\neq j_1$, it suffices to study the root $\al^{\Omega^\pm}_{k,j_1}$ and the set $\mathbf{D}^{\Omega^\pm}_{k,j_1}$ for each $k\in\mathbf{n}$.
We claim that $\al^{\Omega^\pm}_{k,j_1}=\al^{\Omega^\pm}_{1,k,j_1}+\al^{\Omega^\pm}_{2,k,j_1}$, which immediately implies that $\Omega_{1,k,j_1}\sqcup\Omega_{2,k,j_1}\in \mathbf{D}^{\Omega^\pm}_{k,j_1}$, for each $k\in\mathbf{n}$. We set $\al^{\Omega^\pm}_{n+1,j}\defeq 0$, $\al^{\Omega^\pm}_{1,n+1,j}\defeq 0$ and $\al^{\Omega^\pm}_{2,n+1,j}\defeq 0$ for convenience and check by decreasing induction on $k$. The claim is clear by the following observations:
\begin{itemize}
\item if $\al^{\Omega^\pm}_{a,k,j_1}=\al^{\Omega^\pm}_{a,k+1,j_1}$ for each $a=1,2$, then we clearly have $(v_{j_1}^{\Omega^\pm})^{-1}(k)=w_{j_1}^{-1}(k)$ and $\al^{\Omega^\pm}_{k,j_1}=\al^{\Omega^\pm}_{k+1,j_1}$;
\item otherwise, there exists a unique $a\in\{1,2\}$ determined by $k$ such that $\al^{\Omega^\pm}_{a,k,j_1}\neq \al^{\Omega^\pm}_{a,k+1,j_1}$, and moreover $\al^{\Omega^\pm}_{k,j_1}-\al^{\Omega^\pm}_{k+1,j_1}=\al^{\Omega^\pm}_{a,k,j_1}-\al^{\Omega^\pm}_{a,k+1,j_1}$.
\end{itemize}

Let $(k,j_1)\in\mathbf{n}_\cJ$ be a pair and $\Omega_{k,j_1}^\natural$ be an arbitrary element of $\mathbf{D}^{\Omega^\pm}_{k,j_1}$, and we want to show that there exists a partition
\begin{equation}\label{equ: partition type I exceptional}
\Omega_{k,j_1}^\natural=\Omega_{1,k,j_1}^\natural\sqcup\Omega_{2,k,j_1}^\natural
\end{equation}
such that
$\sum_{\beta\in \Omega_{a,k,j_1}^\natural}\beta=\al^{\Omega^\pm}_{a,k,j_1}$
for each $a=1,2$. If $\al^{\Omega^\pm}_{1,k,j_1}=0$ (resp.~$\al^{\Omega^\pm}_{2,k,j_1}=0$), then we can clearly set $\Omega_{1,k,j_1}^\natural\defeq \emptyset$ and $\Omega_{2,k,j_1}^\natural\defeq \Omega_{k,j_1}^\natural$ (resp.~$\Omega_{2,k,j_1}^\natural\defeq \emptyset$ and $\Omega_{1,k,j_1}^\natural\defeq \Omega_{k,j_1}^\natural$). Hence it is harmless to assume that $\al^{\Omega^\pm}_{1,k,j_1}\neq 0\neq \al^{\Omega^\pm}_{2,k,j_1}$.
This condition implies that $k_{1,1}\geq k>k_{1,c_1-1}$.
We now produce the desired partition by checking the hypotheses of Lemma~\ref{lem: existence of partition} in each of the following cases (which exhausts all possible cases):
\begin{itemize}
\item if $k_{1,1}\geq k>\max\{k_2^{1,1},k_{2,c_2-1}\}$ then $i_{1,k,j_1}=i_{1,1}$ and $i_{2,k,j_1}^\prime=i_{2,c}$ for some $1\leq c\leq c_2-1$, which implies that $((i_{1,k,j_1},i_{2,k,j_1}^\prime),j_1)\notin \widehat{\Lambda}$, due to Condition~\rm{I}-\ref{it: I 5}.

\item if $k_2^{1,1}\geq k>k_{1,c_1-1}$ then we have $i_{1,k,j_1}^\prime=i_{1,c}$ for some $1\leq c\leq c_1-1$, and $i_{2,k,j_1}=i_2^{s,e}$ for some $1\leq s\leq d_2$ and $1\leq e\leq e_{2,s}$ satisfying $k_2^{s,e}>k_{1,c_1-1}$. Hence we deduce from Condition~\rm{I}-\ref{it: I 7} that $((i_{2,k,j_1},i_{1,k,j_1}^\prime),j_1)\notin\widehat{\Lambda}$.

\item if $\min\{k_{1,1},k_{2,c_2-1}\}\geq k>k_2^{1,1}$ (in particular $d_2=1$, $c_2^1=c_2$) then $\al^{\Omega^\pm}_{2,k,j_1}=(i_{2,0},i_{2,c_2})=(i_{1,0},i_{1,c_1})$ and $\al^{\Omega^\pm}_{1,k,j_1}=(i_{1,1},i_{1,c})$ for some $2\leq c\leq c_1-1$. Let $\Omega'\in\mathbf{D}_{((i_{1,1},i_{1,c_1}),j_1)}$ be an arbitrary element, and $\Omega'_0\defeq \Omega'\sqcup\{((i_{1,0},i_{1,1}),j_1)\}\in\mathbf{D}_{(\al_1,j_1)}$. According to the definition of $k_2^{1,1}$ (and the fact that $\Omega^-=\Omega_{(\al_1,j_1),\Lambda}^{\rm{max}}$), if there exists $\Omega'\in\mathbf{D}_{((i_{1,1},i_{1,c_1}),j_1)}$ such that $u_j(i_{\Omega',1})\geq k$, we must have $k_{1,0}>k_{1,1}>u_j(i_{\Omega',1})=u_j(i_{\Omega'_0,1})=k_{2,c_2-1}$ (and thus $c_2\geq 2$).
This implies $((i_{1,1},i_{2,c_{2}-1}),j_1)\in\widehat{\Lambda}$ which contradicts Condition~\rm{I}-\ref{it: I 5}.
\end{itemize}
Hence, by Lemma~\ref{lem: existence of partition} we have a partition as in (\ref{equ: partition type I exceptional}).

We consider a pair $(k,j_1)\in\mathbf{n}_\cJ$ satisfying $\al^{\Omega^\pm}_{k,j_1}=(i_{1,0},i_{1,c_1})$, which implies $\al^{\Omega^\pm}_{1,k,j_1}=(i_{1,0},i_{1,c_1})$ and $\al^{\Omega^\pm}_{2,k,j_1}=0$ by the constructions of $\Omega_{a,k,j_1}$ together with Condition $\rm{I}$-\ref{it: I 6}.
We observe that $\al^{\Omega^\pm}_{1,k,j_1}=\al_1=(i_{1,0},i_{1,c_1})$ if and only if exactly one of the following holds:
\begin{itemize}
\item $e_{1,1}=0$ and $k_{1,c_1-1}\geq k>k_{1,c_1}$;
\item $e_{1,1}\geq 1$ and $k_{1,c_1-1}\geq k>k_1^{1,1}$.
\end{itemize}
In particular, we have $\Omega_{2,k,j_1}=\emptyset$ and $\al^{\Omega^\pm}_{2,k,j_1}=0$ for such $k$.
For any  $\Omega\in\mathbf{D}^{\Omega^\pm}_{k,j_1}\subseteq\mathbf{D}_{(\al_1,j_1),\Lambda}$ we have $u_{j_1}(i_{\Omega,1})\geq k$.
By the definition of $e_{1,1}$ and $k_1^{1,1}$ (if it exists), any such $\Omega$ satisfies $ u_{j_1}(i_{\Omega,1})\geq k_{1,c_1-1}$.
Furthermore, if the equality holds, then $\Omega=\Omega^+$ since $\Omega^+$ is $\Lambda$-exceptional.
Thus
$$\mathbf{D}^{\Omega^\pm}_{k,j_1}=\{\Omega^+\}\sqcup\{\Omega\in\mathbf{D}_{(\al_1,j_1),\Lambda}\mid u_{j_1}(i_{\Omega',1})>k_{1,c_1-1}\}.$$

We now check the equality $\Omega_{1,k,j_1}^\natural=\Omega_{1,k,j_1}$ for each $(k,j_1)\in\mathbf{n}_\cJ$ satisfying $\al^{\Omega^\pm}_{1,k,j_1}\neq\al_1=(i_{1,0},i_{1,c_1})$.
Such a $(k,j_1)$ satisfies either $k>k_{1,c_1-1}$ or $e_{1,1}\geq 1$ and $k\leq k_1^{1,1}$.
In the first case, the equality follows from the fact that $\Omega^+$ is $\Lambda$-exceptional (which implies that $\#\mathbf{D}_{((i_{1,0},i_{1,c}),j_1)}=1$ for each $1\leq c\leq c_1-1$).
In the second case, the equality follows from Lemma~\ref{lem: unique k decomposition}.

Finally, the equality  $\Omega_{2,k,j_1}^\natural=\Omega_{2,k,j_1}$ follows from Lemma~\ref{lem: unique k decomposition}. The proof is thus finished.
\end{proof}

\begin{proof}[Proof of Proposition~\ref{prop: type I exceptional}]
Note that we fix a $\cC\in\cP_\cJ$ satisfying $\cC\subseteq\cN_{\xi,\Lambda}$. We recall that $I_\cJ^{\Omega^\pm,\star}\subseteq I_\cJ^{\Omega^\pm}$ is the subset consisting of those $(k,j)$ satisfying $\mathbf{D}^{\Omega^\pm}_{k,j}\neq \mathbf{D}^{\Omega^\pm}_{k+1,j}$, and it is clear that $I_\cJ^{\Omega^\pm,\star}\subseteq\mathbf{n}\times\{j_1\}\subseteq\mathbf{n}_\cJ$ in our case.
It follows from Condition~\rm{I}-\ref{it: I 4} that
$$](k_2^{s,e},j_1),(k_2^{s,e},j_1)]_{w_\cJ}\cap ](k_1^{1,e'},j_1),(k_1^{1,e'},j_1)]_{w_\cJ}=\emptyset$$
for each $1\leq s\leq d_2$, $1\leq e\leq e_{2,s}$ and each $1\leq e'\leq e_{1,1}$ satisfying $k_2^{s,e}>k_{1,c_1-1}>k_1^{1,e'}$. This together with Condition~\rm{I}-\ref{it: I 3} implies that
$$
I_\cJ^{\Omega^\pm,\star}=\{(k_{1,c},j_1)\mid 1\leq c\leq c_1-1\}\sqcup\{(k_2^{s,e},j_1)\mid 1\leq s\leq d_2,~1\leq e\leq e_{2,s},~k_2^{s,e}>k_{1,c_1-1}\}.
$$
It follows from Lemma~\ref{lem: std type I exceptional} that
$$f_{S^{j,\Omega^\pm}_{k},j}|_{\cN_{\xi,\Lambda}}\sim F_\xi^{\Omega^+,\star}\defeq F_\xi^{\Omega^+}+\sum_{\substack{\Omega'\in \mathbf{D}_{(\al_1,j_1)},\, \Omega^+<\Omega'}}\varepsilon(\Omega')F_\xi^{\Omega'}$$
for each $(k,j)$ satisfying $\al^{\Omega^\pm}_{k,j_1}=\al_1$, and
$$f_{S^{j,\Omega^\pm}_{k},j}|_{\cN_{\xi,\Lambda}}\sim F_\xi^{\Omega_{1,k,j_1}}F_\xi^{\Omega_{2,k,j_1}}$$
otherwise. Here $\varepsilon(\Omega')\in\{1,-1\}$ is a sign determined by $\Omega'$. If $F_\xi^{\Omega^+,\star}|_\cC=0$, then Proposition~\ref{prop: type I exceptional} clearly follows as $F_\xi^{\Omega^+,\star}|_\cC(F_\xi^{\Omega^-}|_\cC)^{-1}=0\in\cO_\cC$. If $F_\xi^{\Omega^+,\star}|_\cC\neq 0$, then we take $(k_\star,j_\star)\defeq (k_{1,c_1-1},j_1)$ and deduce from Lemma~\ref{lem: from sets to formula} and Lemma~\ref{lem: std type I exceptional} that $f_\xi^{\Omega^\pm}\in\Inv(\cC)$ and
\begin{equation}\label{equ: exp formula type I exceptional}
f_\xi^{\Omega^\pm}|_\cC\sim F_\xi^{\Omega^\pm,1}|_\cC\cdot F_\xi^{\Omega^\pm,2}|_\cC\cdot F_\xi^{\Omega^+,\star}|_\cC
\end{equation}
where
$$
F_\xi^{\Omega^\pm,a}\defeq (F_\xi^{\Omega_{a,k_\star+1,j_1}})^{-1}\underset{(k,j_1)\in I_\cJ^{\Omega^\pm,\star}\setminus\{(k_\star,j_1)\}}{\prod}F_\xi^{\Omega_{a,k,j_1}}(F_\xi^{\Omega_{a,k+1,j_1}})^{-1}
$$
for each $a=1,2$. We write $\mathbf{n}^a\subseteq\mathbf{n}\setminus\{k_\star\}$ for the subset consisting of those $k$ satisfying $(k,j_1)\in I_\cJ^{\Omega^\pm,\star}$ and $\Omega_{a,k,j_1}\neq \Omega_{a,k+1,j_1}$, for each $a=1,2$. Then it follows from our definition of $\Omega_{1,k,j_1}$ and $\Omega_{2,k,j_1}$ that
$$\mathbf{n}^1=\left\{\begin{array}{cl}
\{k_{1,c}\mid 1\leq c\leq c_1-2\}& \hbox{if $k_2^{1,1}=k_{1,c_1-1}$};\\
\{k_2^{1,1}\}\sqcup\{k_{1,c}\mid 1\leq c\leq c_1-2\}& \hbox{if $k_2^{1,1}>k_{1,c_1-1}$}
\end{array}\right.$$
and
$$\mathbf{n}^2=\{k_2^{s,e}\mid 1\leq s\leq d_2,~1\leq e\leq e_{2,s},~k_2^{s,e}>k_{1,c_1-1}\}.$$
Then we observe that
$$
F_\xi^{\Omega^\pm,a}=(F_\xi^{\Omega_{a,k_\star+1,j_1}})^{-1}\underset{k\in\mathbf{n}^a}{\prod}F_\xi^{\Omega_{a,k,j_1}}(F_\xi^{\Omega_{a,k+1,j_1}})^{-1}
$$
for each $a=1,2$. If $k_2^{1,1}=k_{1,c_1-1}$, then we have
$$F_\xi^{\Omega_{1,k_\star+1,j_1}}=\underset{1\leq c\leq c_1-2}{\prod}u_\xi^{((i_{1,c},i_{1,c+1}),j_1)}$$
and $F_\xi^{\Omega_{1,k_{1,c},j_1}}=u_\xi^{((i_{1,c},i_{1,c+1}),j_1)}F_\xi^{\Omega_{1,k_{1,c}+1,j_1}}$ for each $1\leq c\leq c_1-2$, which imply that
\begin{equation}\label{equ: I exceptional formula 1}
F_\xi^{\Omega^\pm,1}=1.
\end{equation}
If $k_2^{1,1}>k_{1,c_1-1}$, then we have
$$F_\xi^{\Omega_{1,k_\star+1,j_1}}=\underset{0\leq c\leq c_1-2}{\prod}u_\xi^{((i_{1,c},i_{1,c+1}),j_1)},\qquad F_\xi^{\Omega_{1,k_2^{1,1},j_1}} =u_\xi^{((i_{1,0},i_{1,1}),j_1)}F_\xi^{\Omega_{1,k_2^{1,1}+1,j_1}},$$ and $$F_\xi^{\Omega_{1,k_{1,c},j_1}}=u_\xi^{((i_{1,c},i_{1,c+1}),j_1)}F_\xi^{\Omega_{1,k_{1,c}+1,j_1}}$$ for each $1\leq c\leq c_1-2$, which again implies (\ref{equ: I exceptional formula 1}). Similarly, by checking the definition of $\Omega_{2,k_\star+1,j_1}$ as well as the definition of $\Omega_{2,k,j_1}$ and $\Omega_{2,k+1,j_1}$ for each $k\in\mathbf{n}^2$, we deduce that
\begin{equation}\label{equ: I exceptional formula 2}
F_\xi^{\Omega^\pm,2}=\left(\underset{0\leq c\leq c_2-1}{\prod}u_\xi^{((i_{2,c},i_{2,c+1}),j_1)}\right)^{-1}=(F_\xi^{\Omega^-})^{-1}.
\end{equation}
We can clearly combine (\ref{equ: I exceptional formula 1}) and (\ref{equ: I exceptional formula 2}) with (\ref{equ: exp formula type I exceptional}) and deduce that
$$ F_\xi^{\Omega^\pm}|_\cC+\sum_{\Omega_0^\pm}\varepsilon(\Omega_0^\pm)F_\xi^{\Omega_0^\pm}|_\cC=(F_\xi^{\Omega^-}|_\cC)^{-1}F_\xi^{\Omega^+,\star}|_\cC\sim f_\xi^{\Omega^\pm}|_\cC\in\cO_\cC
$$
where $\Omega_0^\pm$ runs through balanced pair satisfying $\Omega^+<\Omega_0^+$ and $\Omega^-=\Omega_0^-$ with $\varepsilon(\Omega_0^\pm)\defeq \varepsilon(\Omega_0^+)$. The proof is thus finished.
\end{proof}

\subsubsection{Proof of Proposition~\ref{prop: type I extremal}}\label{sub: exp type I extremal}
Given a constructible $\Lambda$-lift $\Omega^\pm$ of type $\rm{I}$ with $\Omega^+$ being $\Lambda$-extremal (and thus $e_{1,1}\geq 1$ and $k_1^{1,1}>k_{1,c_1-1}$), we define a subset $\Omega_{a,k,j}\subseteq \Lambda$ for each pair $(k,j)\in\mathbf{n}_\cJ$ and each $a=1,2$ as follows.

We set $\Omega_{2,k,j}\defeq \emptyset$ if $j\neq j_1$. If $k_2^{1,1}>k_1^{1,1}$, we define $\Omega_{2,k,j_1}$ as
$$
\Omega_{2,k,j_1}\defeq
\left\{\begin{array}{cl}
\Omega_{\psi_2,k}&\hbox{if $k>k_{1,c_1-1}$};\\
\emptyset&\hbox{if $k\leq k_{1,c_1-1}$}.
\end{array}\right.
$$
If $k_2^{1,1}<k_1^{1,1}$, we define $\Omega_{2,k,j_1}$ as
$$
\Omega_{2,k,j_1}\defeq
\left\{\begin{array}{cl}
\Omega_{\psi_2,k}\setminus\{((i_{1,0},i_{2,1}),j_1)\}&\hbox{if $k>k_1^{1,1}$};\\
\Omega_{\psi_2,k}&\hbox{if $k_1^{1,1}\geq k>k_{1,c_1-1}$};\\
\emptyset&\hbox{if $k\leq k_{1,c_1-1}$}.
\end{array}\right.
$$
Similarly, we set $\Omega_{1,k,j}\defeq \emptyset$ if $j\neq j_1$. If $k_2^{1,1}>k_1^{1,1}$, we define $\Omega_{1,k,j_1}$ as
$$
\Omega_{1,k,j_1}\defeq
\left\{\begin{array}{cl}
\Omega_{\psi_1,k}\setminus\{((i_{1,0},i_{1,1}),j_1)\}&\hbox{if $k>k_2^{1,1}$};\\
\Omega_{\psi_1,k}&\hbox{if $k\leq k_2^{1,1}$}.
\end{array}\right.$$
If $k_2^{1,1}<k_1^{1,1}$, we define $\Omega_{1,k,j_1}\defeq \Omega_{\psi_1,k}$ for each $k\in\mathbf{n}$.

For each $(k,j)\in\mathbf{n}_\cJ$, we write $\al^{\Omega^\pm}_{1,k,j}\defeq \sum_{(\beta,j_1)\in\Omega_{1,k,j}}\beta$ and $\al^{\Omega^\pm}_{2,k,j}\defeq \sum_{(\beta,j_1)\in\Omega_{2,k,j}}\beta$. It follows from the definition of $\Omega_{1,k,j}$ and $\Omega_{2,k,j}$ above that $\al^{\Omega^\pm}_{1,k,j},~\al^{\Omega^\pm}_{2,k,j}\in\Phi^+\sqcup\{0\}$. Hence we can write $\al^{\Omega^\pm}_{a,k,j}=(i_{a,k,j},i_{a,k,j}^\prime)$ for each $a=1,2$ and $(k,j)\in\mathbf{n}_\cJ$ with $\al^{\Omega^\pm}_{a,k,j}\neq 0$.

\begin{lemma}\label{lem: std type I extremal}
Let $\Omega^\pm$ be a constructible $\Lambda$-lift of type $\rm{I}$, and assume that $\Omega^+$ is $\Lambda$-extremal. Then we have $\Omega_{1,k,j}\cap\Omega_{2,k,j}=\emptyset$ and $\mathbf{D}^{\Omega^\pm}_{k,j}=\{\Omega_{1,k,j}\sqcup\Omega_{2,k,j}\}$ for each $(k,j)\in\mathbf{n}_\cJ$.
\end{lemma}
\begin{proof}
A similar argument as in the proof of Lemma~\ref{lem: std type I exceptional} and a case by case check shows that $\Omega_{1,k,j}\cap\Omega_{2,k,j}=\emptyset$, $\al^{\Omega^\pm}_{k,j}=\al^{\Omega^\pm}_{1,k,j}+\al^{\Omega^\pm}_{2,k,j}$ and $\Omega_{1,k,j}\sqcup\Omega_{2,k,j}\in\mathbf{D}^{\Omega^\pm}_{k,j}$ for each $(k,j)\in\mathbf{n}_\cJ$. As we clearly have $\al^{\Omega^\pm}_{k,j}=0$ and thus $\mathbf{D}^{\Omega^\pm}_{k,j}=\{\emptyset\}$ if $j\neq j_1$, it suffices to study the set $\mathbf{D}^{\Omega^\pm}_{k,j_1}$ for each $k\in\mathbf{n}$.

Now we consider a pair $(k,j_1)\in\mathbf{n}_\cJ$ and let $\Omega_{k,j_1}^\natural$ be an arbitrary element of $\mathbf{D}^{\Omega^\pm}_{k,j_1}$. We want to show that there exists a partition
\begin{equation}\label{equ: partition type I extremal}
\Omega_{k,j_1}^\natural=\Omega_{1,k,j_1}^\natural\sqcup\Omega_{2,k,j_1}^\natural
\end{equation}
such that $\sum_{\beta\in \Omega_{a,k,j_1}^\natural}\beta=\al^{\Omega^\pm}_{a,k,j_1}$
for each $a=1,2$. It is harmless to assume that $\al^{\Omega^\pm}_{1,k,j_1}\neq 0\neq \al^{\Omega^\pm}_{2,k,j_1}$ (in particular we have $k>k_{1,c_1-1}$).
If $i_{1,k,j_1}\neq i_{1,0}$ and $i_{2,k,j_1}^\prime\neq i_{1,c_1}$, then we have $((i_{1,k,j_1},i_{2,k,j_1}^\prime),j_1)\notin\widehat{\Lambda}$ (see Condition~\rm{I}-\ref{it: I 5} and \rm{I}-\ref{it: I 6}) and can deduce the partition (\ref{equ: partition type I extremal}) from Lemma~\ref{lem: existence of partition}. Similarly, if $i_{2,k,j_1}\neq i_{1,0}$ and $i_{1,k,j_1}^\prime\neq i_{1,c_1}$, then we have $((i_{2,k,j_1},i_{1,k,j_1}^\prime),j_1)\notin\widehat{\Lambda}$ (see Condition~\rm{I}-\ref{it: I 5} and \rm{I}-\ref{it: I 7}) and can deduce the partition (\ref{equ: partition type I extremal}) from Lemma~\ref{lem: existence of partition}.
Thus, we just need to consider the following two cases:
\begin{itemize}
\item $i_{1,k,j_1}=i_{1,0}$ and $i_{1,k,j_1}^\prime=i_{1,c_1}$: this can not happen as $\Omega^+$ is $\Lambda$-extremal.

\item $i_{2,k,j_1}=i_{1,0}$ and $i_{2,k,j_1}^\prime=i_{1,c_1}$: this forces $\min\{k_{2,c_2-1},k_{1,1}\}\geq k>k_2^{1,1}$, in particular $d_2=1$, $c_2^1=c_2$, $\al^{\Omega^\pm}_{2,k,j_1}=\al_1=(i_{1,0},i_{1,c_1})$ and $\al^{\Omega^\pm}_{1,k,j_1}=(i_{1,1},i_{1,c})$ for some $2\leq c\leq c_1-1$.
The same argument as in the proof of Lemma~\ref{lem: std type I exceptional} gives the desired partition.
\end{itemize}

Now it remains to show that $\Omega_{a,k,j_1}^\natural=\Omega_{a,k,j_1}$ for each $a=1,2$ and each $(k,j_1)\in\mathbf{n}_\cJ$. Note that $\alpha^{\Omega^\pm}_{1,k,j_1}\neq (i_{1,0},i_{1,c_1})=\alpha_1$ for each $(k,j_1)\in\mathbf{n}_\cJ$, as $\Omega^+$ is $\Lambda$-extremal.
After applying Lemma~\ref{lem: unique k decomposition} the only non-trivial cases are:
\begin{itemize}
\item $\Omega_{1,k,j_1}=\Omega_{\psi_1,k}\setminus\{((i_{1,0},i_{1,1}),j_1)\}$: this forces $k>k_{2}^{1,1}>k_{1}^{1,1}$.  Lemma~\ref{lem: unique k decomposition} implies $\Omega^\natural_{1,k,j_1}\sqcup \{((i_{1,0},i_{1,1}),j_1)\}= \Omega_{\psi_1,k}$ and we are done.

\item $\Omega_{2,k,j_1}=\Omega_{\psi_2,k}\setminus\{((i_{1,0},i_{2,1}),j_1)\}$: this forces $k>k_{1}^{1,1}>k_{2}^{1,1}$.
We conclude by the same argument as in the previous case.
\end{itemize}
The proof is thus finished.
\end{proof}

\begin{proof}[Proof of Proposition~\ref{prop: type I extremal}]
Note that we fix a $\cC\in\cP_\cJ$ satisfying $\cC\subseteq\cN_{\xi,\Lambda}$.
We recall that $I_\cJ^{\Omega^\pm,\star}\subseteq I_\cJ^{\Omega^\pm}$ is the subset consisting of those $(k,j)$ satisfying $\mathbf{D}^{\Omega^\pm}_{k,j}\neq \mathbf{D}^{\Omega^\pm}_{k+1,j}$, and it is clear that $I_\cJ^{\Omega^\pm,\star}\subseteq\mathbf{n}\times\{j_1\}\subseteq\mathbf{n}_\cJ$ in our case.
It follows from Conditions~\rm{I}-\ref{it: I 3} and \rm{I}-\ref{it: I 4} that
$$
I_\cJ^{\Omega^\pm,\star}
=\{(k_{1,c},j_1)\mid 0\leq c\leq c_1-1\}\sqcup\{(k_2^{s,e},j_1)\mid 1\leq s\leq d_2,~1\leq e\leq e_{2,s},~k_2^{s,e}>k_{1,c_1-1}\}$$
if $k_2^{1,1}<k_1^{1,1}$, and
$$I_\cJ^{\Omega^\pm,\star}=\{(k_{1,c},j_1)\mid 1\leq c\leq c_1-1\}\sqcup\{(k_2^{s,e},j_1)\mid 1\leq s\leq d_2,~1\leq e\leq e_{2,s},~k_2^{s,e}>k_{1,c_1-1}\}$$
if $k_2^{1,1}>k_1^{1,1}$.
Hence it follows from Lemma~\ref{lem: from sets to formula} and Lemma~\ref{lem: std type I extremal} that $f_\xi^{\Omega^\pm}\in\Inv(\cC)$ and
\begin{equation}\label{equ: exp formula type I extremal}
f_\xi^{\Omega^\pm}|_\cC\sim F_\xi^{\Omega^\pm,1}|_\cC\cdot F_\xi^{\Omega^\pm,2}|_\cC
\end{equation}
where
$$
F_\xi^{\Omega^\pm,a}\defeq \underset{(k,j_1)\in I_\cJ^{\Omega^\pm,\star} }{\prod}F_\xi^{\Omega_{a,k,j_1}}(F_\xi^{\Omega_{a,k+1,j_1}})^{-1} $$
for each $a=1,2$. We write $\mathbf{n}^a\subseteq\mathbf{n}$ for the subset consisting of those $k$ satisfying $(k,j_1)\in I_\cJ^{\Omega^\pm,\star}$ and $\Omega_{a,k,j_1}\neq \Omega_{a,k+1,j_1}$, for each $a=1,2$. Then it follows from our definition of $\Omega_{1,k,j_1}$ and $\Omega_{2,k,j_1}$ that
$$\mathbf{n}^1=\left\{\begin{array}{cl}
\{k_2^{1,1}\}\sqcup\{k_{1,c}\mid 1\leq c\leq c_1-2\}&\hbox{if $k_2^{1,1}>k_1^{1,1}$};\\
\{k_{1,c}\mid 0\leq c\leq c_1-2\}&\hbox{if $k_2^{1,1}<k_1^{1,1}$}
\end{array}\right.$$
and
$$\mathbf{n}^2=\{k_{1,c_1-1}\}\sqcup\{k_2^{s,e}\mid 1\leq s\leq d_2,~1\leq e\leq e_{2,s},~k_2^{s,e}>k_{1,c_1-1}\}.$$
Then we observe that
$$
F_\xi^{\Omega^\pm,a}=\underset{k\in\mathbf{n}^a}{\prod}F_\xi^{\Omega_{a,k,j_1}}(F_\xi^{\Omega_{a,k+1,j_1}})^{-1}
$$
for each $a=1,2$. By carefully checking our definition of $\Omega_{1,k,j}$ and $\Omega_{2,k,j}$ for various $(k,j)\in\mathbf{n}_\cJ$, we observe that
\begin{equation}\label{equ: I extremal formula}
F_\xi^{\Omega^\pm,1}=F_\xi^{\Omega^+}\,\mbox{and}\,\,F_\xi^{\Omega^\pm,2}=(F_\xi^{\Omega^-})^{-1}.
\end{equation}
We can clearly combine (\ref{equ: I extremal formula}) with (\ref{equ: exp formula type I extremal}) and deduce that
$$
F_\xi^{\Omega^\pm}|_\cC\sim F_\xi^{\Omega^+}|_\cC(F_\xi^{\Omega^-}|_\cC)^{-1}\sim f_\xi^{\Omega^\pm}|_\cC\in\cO_\cC.
$$
The proof is thus finished.
\end{proof}

\subsection{Explicit formula: type~\rm{II}}\label{sub: exp type II}
In this section, we explicitly write down the set $\mathbf{D}^{\Omega^\pm}_{k,j}$ for each $(k,j)\in\mathbf{n}_\cJ$ when $\Omega^\pm$ is a constructible $\Lambda$-lift of type $\rm{II}$. Consequently, we apply Lemma~\ref{lem: from sets to formula} and finish the proofs of Proposition~\ref{prop: type II exceptional} and Proposition~\ref{prop: type II extremal} at the end of this section. We will frequently use all the notation from \S\,\ref{sub: type II}, \S\,\ref{sub: notation for each k j} and the beginning of \S\,\ref{sec:const:inv}.

We want to define a set $\Omega_{a,k,j}$ for each $1\leq a\leq t$ and each $(k,j)\in\mathbf{n}_\cJ$. Recall that $k_t^\prime=k_{t,c_t}$ if $d_t=0$, and $k_t^\prime=k_t^{1,1}$ if $d_t\geq 1$.
For each $1\leq a\leq t$, we set $\Omega_{a,k,j}\defeq \emptyset$ if $j\neq j_1$. For each $3\leq a\leq t-1$, we define $\Omega_{a,k,j_1}\defeq \Omega_{\psi_a,k}$ for each $k\in\mathbf{n}$.
If $k_{2,c_2-1}<k_{1,c_1-1}$, we define $\Omega_{2,k,j_1}\defeq \Omega_{\psi_2,k}$ for each $k\in\mathbf{n}$.
If $k_{2,c_2-1}>k_{1,c_1-1}$, we define $\Omega_{2,k,j_1}$ as
$$
\Omega_{2,k,j_1}\defeq
\left\{\begin{array}{cl}
\Omega_{\psi_2,k}&\hbox{if $k>k_{1,c_1-1}$};\\
\emptyset&\hbox{if $k\leq k_{1,c_1-1}$}.
\end{array}\right.
$$
If either $e_{1,1}=0$ or $k_t^\prime>k_1^{1,1}$, we define $\Omega_{t,k,j_1}\defeq \Omega_{\psi_t,k}$ for each $k\in\mathbf{n}$.
If $e_{1,1}\geq 1$ and $k_t^\prime<k_1^{1,1}$, we define $\Omega_{t,k,j_1}$ as
$$
\Omega_{t,k,j_1}\defeq
\left\{\begin{array}{cl}
\Omega_{\psi_t,k}\setminus\{((i_{1,0},i_{t,1}),j_1)\}&\hbox{if $k>k_1^{1,1}$};\\
\Omega_{\psi_t,k}&\hbox{if $k\leq k_1^{1,1}$}.
\end{array}\right.
$$
Note that if $k_{2,c_2-1}<k_{1,c_1-1}$, then we automatically have $e_{1,1}\geq 1$ and $k_{2,c_2-1}\leq k_1^{1,1}$.
If $k_{2,c_2-1}<k_{1,c_1-1}$ and $k_t^\prime>k_1^{1,1}$, then we define $\Omega_{1,k,j_1}$ as
$$
\Omega_{1,k,j_1}\defeq
\left\{\begin{array}{cl}
\Omega_{\psi_1,k}\setminus\{((i_{1,0},i_{1,1}),j_1)\}&\hbox{if $k>k_t^\prime$};\\
\Omega_{\psi_1,k}&\hbox{if $k_t^\prime\geq k>k_{2,c_2-1}$};\\
\emptyset&\hbox{if $k\leq k_{2,c_2-1}$}.
\end{array}\right.
$$
If $k_{2,c_2-1}<k_{1,c_1-1}$ and $k_t^\prime<k_1^{1,1}$, then we define $\Omega_{1,k,j_1}$ as
$$
\Omega_{1,k,j_1}\defeq
\left\{\begin{array}{cl}
\Omega_{\psi_1,k}&\hbox{if $k>k_{2,c_2-1}$};\\
\emptyset&\hbox{if $k\leq k_{2,c_2-1}$}.
\end{array}\right.
$$
If $k_{2,c_2-1}>k_{1,c_1-1}$ and either $e_{1,1}=0$ or $k_t^\prime>k_1^{1,1}$, then we define $\Omega_{1,k,j_1}$ as
$$
\Omega_{1,k,j_1}\defeq
\left\{\begin{array}{cl}
\Omega_{\psi_1,k}\setminus\{((i_{1,0},i_{1,1}),j_1)\}&\hbox{if $k>k_t^\prime$};\\
\Omega_{\psi_1,k}&\hbox{if $k\leq k_t^\prime$}.
\end{array}\right.
$$
If $k_{2,c_2-1}>k_{1,c_1-1}$, $e_{1,1}\geq 1$ and $k_t^\prime<k_1^{1,1}$, then we define $\Omega_{1,k,j_1}\defeq \Omega_{\psi_1,k}$ for each $k\in\mathbf{n}$.

For each $(k,j)\in\mathbf{n}_\cJ$ and each $1\leq a\leq t$, we write $\al^{\Omega^\pm}_{a,k,j}\defeq \sum_{(\beta,j_1)\in\Omega_{a,k,j}}\beta$. It follows from the definition of $\Omega_{a,k,j}$ above that $\al^{\Omega^\pm}_{a,k,j}\in\Phi^+\sqcup\{0\}$, and thus we write $\al^{\Omega^\pm}_{a,k,j}=(i_{a,k,j},i_{a,k,j}^\prime)$ for each $(k,j)\in\mathbf{n}_\cJ$ and each $1\leq a\leq t$ with $\al^{\Omega^\pm}_{a,k,j}\neq 0$.

\begin{lemma}\label{lem: std type II exceptional}
Let $\Omega^\pm$ be a constructible $\Lambda$-lift of type $\rm{II}$, and assume that $\Omega^+$ is $\Lambda$-exceptional. Then we have $\Omega_{a,k,j}\cap\Omega_{a',k,j}=\emptyset$ for each $1\leq a<a'\leq t$ and each $(k,j)\in \mathbf{n}_\cJ$. Moreover, we have
$$\mathbf{D}^{\Omega^\pm}_{k,j_1}=\bigg\{\Omega^+\sqcup\bigsqcup_{a=2}^t\Omega_{a,k,j_1}\bigg\}\sqcup \bigg\{\Omega'\sqcup\bigsqcup_{a=2}^t\Omega_{a,k,j_1}\mid \Omega^+<\Omega'\in\mathbf{D}_{(\al_1,j_1),\Lambda}\bigg\}$$
for each $k\in\mathbf{n}$ satisfying $\al^{\Omega^\pm}_{1,k,j_1}=\al_1=(i_{1,0},i_{1,c_1})$, and
$\mathbf{D}^{\Omega^\pm}_{k,j}=\{\bigsqcup_{a\in\Z/t}\Omega_{a,k,j}\}$ for other choices of $(k,j)\in\mathbf{n}_\cJ$.
\end{lemma}
\begin{proof}
A similar argument as in the proof of Lemma~\ref{lem: std type I exceptional} and case-by-case checking show that $\Omega_{a,k,j}\cap\Omega_{a',k,j}=\emptyset$ for each $1\leq a<a'\leq t$ and $\al^{\Omega^\pm}_{k,j}=\sum_{a=1}^t\al^{\Omega^\pm}_{a,k,j}$ as well as $\bigsqcup_{a\in\Z/t}\Omega_{a,k,j}\in\mathbf{D}^{\Omega^\pm}_{k,j}$ for each $(k,j)\in\mathbf{n}_\cJ$. As we clearly have $\al^{\Omega^\pm}_{k,j}=0$ and thus $\mathbf{D}^{\Omega^\pm}_{k,j}=\{\emptyset\}$ if $j\neq j_1$, it suffices to study the set $\mathbf{D}^{\Omega^\pm}_{k,j_1}$ for each $k\in\mathbf{n}$.

Let $(k,j_1)\in\mathbf{n}_\cJ$ be a pair and $\Omega_{k,j_1}^\natural$ be an arbitrary element of $\mathbf{D}^{\Omega^\pm}_{k,j_1}$.
We first show the following
\begin{claim}
\label{equ: partition type II exceptional}
There exists a partition $\Omega_{k,j_1}^\natural=\bigsqcup_{a\in\Z/t}\Omega_{a,k,j_1}^\natural$
such that
$\sum_{(\beta,j_1)\in \Omega_{a,k,j_1}^\natural}\beta=\al^{\Omega^\pm}_{a,k,j_1}$
for each $a\in\Z/t$.
\end{claim}
\begin{proof}[Proof of Claim~\ref{equ: partition type II exceptional}]
We fix a choice of $(k,j_1)\in\mathbf{n}_\cJ$ such that $\Omega_{a,k,j_1}\neq \emptyset$ for at least two different choices of $a\in\Z/t$ (otherwise, the claim is trivial). Then we choose two non-empty subsets $$\Omega_\sharp\subseteq\underset{a\in\Z/t}{\bigsqcup}\Omega_{a,k,j_1}\,\,\mbox{and}\,\,\Omega_\flat\subseteq\Omega_{k,j_1}^\natural$$ such that
$\sum_{\beta\in\Omega_\sharp}\beta=\sum_{\beta\in\Omega_\flat}\beta$
and there do not exist proper non-empty subsets of $\Omega_\sharp$ and $\Omega_\flat$ satisfying the similar equality. According to our assumption on $\Omega_\sharp$ and $\Omega_\flat$, there exist $s\geq 1$ and an ordering $\widehat{\Omega}_\sharp=\{(\al_{\sharp,s'},j_1)\mid s'\in\Z/s\}$ and an ordering $\widehat{\Omega}_\flat=\{(\al_{\flat,s'},j_1)\mid s'\in\Z/s\}$ such that $i_{\al_{\sharp,s'}}^\prime=i_{\al_{\flat,s'}}^\prime$ and $i_{\al_{\flat,s'}}=i_{\al_{\sharp,s'+1}}$ for each $s'\in\Z/s$. In particular, we observe that
\begin{equation}\label{equ: connecting root: type II}
((i_{\al_{\sharp,s'+1}},i_{\al_{\sharp,s'}}^\prime),j_1)\in\widehat{\Lambda}
\end{equation}
for each $s'\in\Z/s$. Moreover, we observe that the sets $\{i_{\al_{\sharp,s'}},i_{\al_{\sharp,s'}}^\prime\}$ are disjoint for different choices of $s'\in\Z/s$ and we have (cf.~Definition~\ref{def: separated condition})
$$\Delta_{\Omega_\sharp}=\Delta_{\Omega_\flat}=\underset{s'\in\Z/s}{\bigsqcup}\{(i_{\al_{\sharp,s'}},j),(i_{\al_{\sharp,s'}}^\prime,j)\}.$$
It follows from $\Omega_\sharp\subseteq \bigsqcup_{a\in\Z/t}\Omega_{a,k,j_1}$ that, for each $s'\in\Z/s$, there exists a unique $a\in\Z/t$ such that $\Omega_{\sharp,s'}\subseteq\Omega_{a,k,j_1}\neq \emptyset$, and thus we have a well-defined map $\phi:~\Z/s\rightarrow\Z/t$. We prove that we always have $s=1$ by dividing into several cases.

We first treat the case when $\Omega_\sharp\cap\Omega_{1,k,j_1}=\emptyset$. We choose $s'\in\Z/s$ such that $2\leq \phi(s')\leq t$ is maximal possible and $u_{j_1}(i_{\al_{\sharp,s'}}^\prime)$ is maximal possible for the fixed choice of $\phi(s')$. Then we deduce from $((i_{\al_{\sharp,s'+1}},i_{\al_{\sharp,s'}}^\prime),j_1)\in\widehat{\Lambda}$ that $\phi(s'+1)=\phi(s')$ (as $\phi(s')$ is maximal) and thus $u_{j_1}(i_{\al_{\sharp,s'+1}})>u_{j_1}(i_{\al_{\sharp,s'+1}}^\prime)\geq u_{j_1}(i_{\al_{\sharp,s'}}^\prime)$, which together with the maximality of $u_{j_1}(i_{\al_{\sharp,s'}}^\prime)$ implies $s=1$.

Secondly we treat the case when $\Omega_\sharp\subseteq\Omega_{1,k,j_1}\neq \emptyset$. We choose $s'\in\Z/s$ such that $u_{j_1}(i_{\al_{\sharp,s'}}^\prime)$ is maximal possible and deduce from $((i_{\al_{\sharp,s'+1}},i_{\al_{\sharp,s'}}^\prime),j_1)\in\widehat{\Lambda}$ that  $u_{j_1}(i_{\al_{\sharp,s'+1}})>u_{j_1}(i_{\al_{\sharp,s'+1}}^\prime)\geq u_{j_1}(i_{\al_{\sharp,s'}}^\prime)$, which together with the maximality of $u_{j_1}(i_{\al_{\sharp,s'}}^\prime)$ implies $s=1$.

Now we treat the case when $\Omega_\sharp\cap\Omega_{1,k,j_1}\neq\emptyset$ and $\Omega_\sharp\cap\Omega_{a,k,j_1}\neq\emptyset$ for some $2\leq a\leq t$ (and thus $s\geq 2$). We choose an arbitrary $s_0'\in \Z/s$ satisfying $\phi(s_0')=1$ and note that $s_0'+1\neq s_0'\neq s_0'-1$ as $s\geq 2$. We divide into the following cases.

If there does not exist any choice of $s_0'$ such that $i_{\al_{\sharp,s_0'}}^\prime=i_{1,c_1}$, then we choose our $s_0'$ such that $\phi(s_0'+1)\neq 1$. Then the inclusion (\ref{equ: connecting root: type II}) (when $s=s_0'$)
together with either Condition~\rm{II}-\ref{it: II 6} or \rm{II}-\ref{it: II 8} implies that $i_{\al_{\sharp,s_0'+1}}=i_{1,0}$, which forces $\phi(s_0'+1)=t$ and $i_{\al_{\sharp,s'}}\neq i_{1,0}$ for each $s'\neq s_0'+1$. Note that $i_{\al_{\sharp,s_0'+1}}=i_{1,0}$ also forces $\phi(s_0'+2)=1$. If $i_{\al_{\sharp,s_0'+2}}=i_{1,c}$ for some $1\leq c\leq c_1-1$, then the inclusion (\ref{equ: connecting root: type II}) (when $s=s_0'+1$)
violates Condition~\rm{II}-\ref{it: II 6} as $i_{\al_{\sharp,s_0'+1}}^\prime\neq i_{1,c_1}$ by our assumption. Otherwise, we have $i_{\al_{\sharp,s_0'+1}}^\prime=i_{t,c'}$ for some $1\leq c'\leq c_t$ and $i_{\al_{\sharp,s_0'+2}}=i_1^{1,e}$ for some $1\leq e\leq e_{1,1}$, and so we deduce from Condition~\rm{II}-\ref{it: II 7} and the inclusion (\ref{equ: connecting root: type II}) (when $s=s_0'+1$) that $k_1^{1,e}\leq k_{2,c_2-1}<k_{1,c_1-1}$, which is a contradiction as there does not exist $k$ such that $(i_1^{1,e},j_1)\in\mathbf{I}_{\Omega_{1,k,j_1}}$ in this case.

If there exists $s_0'$ such that $i_{\al_{\sharp,s_0'}}^\prime=i_{1,c_1}$ and $i_{\al_{\sharp,s_0'}}\neq i_{1,0}$, then we must have $i_{\al_{\sharp,s'}}^\prime\neq i_{\al_{\sharp,s_0'}}^\prime=i_{1,c_1}$ for each $s'\neq s_0'$ and $\phi(s_0'-1)\neq 1$. If $i_{\al_{\sharp,s_0'}}=i_{1,c}$ for some $1\leq c\leq c_1-1$, then the inclusion (\ref{equ: connecting root: type II}) (when $s=s_0'-1$)
violates Condition~\rm{II}-\ref{it: II 6} as $i_{\al_{\sharp,s_0'-1}}^\prime\neq i_{1,c_1}$. Otherwise, $i_{\al_{\sharp,s_0'-1}}^\prime=i_{\phi(s_0'-1),c'}$ for some $1\leq c'\leq c_{\phi(s_0'-1)}$ and $i_{\al_{\sharp,s_0'}}=i_1^{1,e}$ for some $1\leq e\leq e_{1,1}$, then we deduce from Condition~\rm{II}-\ref{it: II 7} and the inclusion (\ref{equ: connecting root: type II}) (when $s=s_0'-1$) that exactly one of the following holds:
\begin{itemize}
\item $k_1^{1,e}\leq k_{2,c_2-1}<k_{1,c_1-1}$;
\item $k_{\phi(s_0'-1),0}\leq k_{2,c_2-1}<k_{1,c_1-1}$.
\end{itemize}
If $k_1^{1,e}\leq k_{2,c_2-1}<k_{1,c_1-1}$, we obtain a contradiction as there does not exist $k$ such that $(i_1^{1,e},j_1)\in\mathbf{I}_{\Omega_{1,k,j_1}}$ in this case. If $k_{\phi(s_0'-1),0}\leq k_{2,c_2-1}<k_{1,c_1-1}$, then $\Omega_{\sharp,s_0'-1}\subseteq\Omega_{\phi(s_0'-1),k,j_1}\neq \emptyset$ forces $k\leq k_{\phi(s_0'-1),0}\leq k_{2,c_2-1}<k_{1,c_1-1}$ which implies $\Omega_{1,k,j_1}=\emptyset$ and thus a contradiction.

Hence we may assume from now on that $i_{\al_{\sharp,s_0'}}^\prime=i_{1,c_1}$ and $i_{\al_{\sharp,s_0'}}=i_{1,0}$ (namely $\Omega_{\sharp,s_0'}=\Omega_{1,k,j_1}=\Omega^+$), which implies
\begin{itemize}
\item $k\leq k_{1,c_1-1}$ and either $e_{1,1}=0$ or $k>k_1^{1,1}$;
\item $i_{\al_{\sharp,s'}}^\prime\neq i_{1,c_1}$ and $i_{\al_{\sharp,s'}}\neq i_{1,0}$ for each $s'\neq s_0'$;
\item $\phi(s')\neq 1$ for each $s'\neq s_0'$.
\end{itemize}

If $\Omega^+=\Omega_{(\al_1,j_1),\Lambda}^{\rm{max}}$, then as $\Omega_{\flat,s_0'}$ is a $\Lambda$-decomposition of $((i_{\al_{\sharp,s_0'+1}},i_{1,c_1}),j_1)$ satisfying $u_{j_1}(i_{\Omega_{\flat,s_0'},1})\geq k$, we must have $i_{\Omega_{\flat,s_0'},1}=i_{1,c_1-1}=i_{\Omega_{(\al_1,j_1),\Lambda}^{\rm{max}},1}$ (using $k\leq k_{1,c_1-1}$ and either $e_{1,1}=0$ or $k>k_1^{1,1}$), which implies $$((i_{\al_{\sharp,s_0'+1}},i_{1,c_1-1}),j_1)\in\widehat{\Lambda}$$ and thus violates either Condition~\rm{II}-\ref{it: II 6} or \rm{II}-\ref{it: II 8} as $\phi(s_0'+1)\neq 1$ and $i_{\al_{\sharp,s_0'+1}}\neq i_{1,0}$.

If $\Omega^+\neq \Omega_{(\al_1,j_1),\Lambda}^{\rm{max}}$, then it follows from Condition~\rm{II}-\ref{it: II 11} (as well as $k\leq k_{1,c_1-1}$ and either $e_{1,1}=0$ or $k>k_1^{1,1}$) that
\begin{itemize}
\item $k_{t,c_t}>k_{1,c_1-1}\geq k$;
\item for each $3\leq a\leq t-1$, we have either $k_{a,c_a}>k_{1,c_1-1}$ or $k_{a,0}<k_{2,c_2-1}<k_{1,c_1-1}$,
\end{itemize}
which implies that $\Omega_{a,k,j_1}=\emptyset$ for each $3\leq a\leq t$ and $\phi(s_0'+1)=2$. If $k_{2,c_2-1}>k_{1,c_1-1}$, then we have $\Omega_{2,k,j_1}=\emptyset$ (using $k\leq k_{1,c_1-1}$) which contradicts $\phi(s_0'+1)=2$. Thus $k_{2,c_2-1}\leq k_{1,c_1-1}$, hence $e_{1,1}\geq 1$, $k_{2,c_2-1}\leq k_1^{1,1}$ and $\Omega_{\flat,s_0'}$ is a $\Lambda$-decomposition of $((i_{\al_{\sharp,s_0'+1}},i_{1,c_1}),j_1)$ satisfying $u_{j_1}(i_{\Omega_{\flat,s_0'},1})\geq k>k_1^{1,1}\geq k_{2,c_2-1}$, contradicting $\Omega_{\al_2,j_1}^-=\Omega_{(\al_2,j_1),\Lambda}^{\max}$.

Up to this stage, we have shown that $s=1$ for all possible choices of $\Omega_\sharp,\Omega_\flat$.
This finishes the proof of Claim~\ref{equ: partition type II exceptional}.
\end{proof}

Now we continue the proof of Lemma~\ref{lem: std type II exceptional}. It remains to analyze $\Omega_{a,k,j_1}^\natural$ for each $(k,j_1)\in\mathbf{n}_\cJ$ and $a\in\Z/t$ such that $\Omega_{a,k,j_1}\neq \emptyset$.

By Lemma~\ref{lem: unique k decomposition}, $\Omega_{a,k,j_1}\neq \emptyset$ (which equals either $\Omega_{\psi_a,k}$ or $\Omega_{\psi_a,k}\setminus\{(i_{a,0},i_{a,1})\}$) is the unique $\Lambda$-decomposition $\Omega'$ of $\alpha^{\Omega^\pm}_{a,k,j_1}$ such that $u_{j_1}(i_{\Omega',1})\geq k$, provided either $\Omega_{a}$ is maximal or $\alpha_{\psi_a,k}\neq (i_{a,0},i_{a,c_a})$.
Hence $\Omega^\natural_{a,k,j_1}=\Omega_{a,k,j_1}$ in such cases.

Thus it remains to study $\Omega_{1,k,j_1}$ when $\Omega^+\neq \Omega_{(\alpha_1,j_1),\Lambda}^{\max}$ and $\alpha_{\psi_1,k}=\alpha_1=(i_{1,0},i_{1,c_1})$.
Condition~\rm{II}-\ref{it: II 10} implies $k_t^\prime>k_{1,c_1-1}$.
Furthermore $k_t^\prime>k_{1,c_1-1}\geq k$ and either $e_{1,1}=0$ or $k>k_1^{1,1}$ (in particular $\alpha^{\Omega^\pm}_{1,k,j_1}=\alpha_1$).
But then since $\Omega^+$ is $\Lambda$-exceptional, any $\Lambda$-decomposition $\Omega'$ of $\alpha^{\Omega^\pm}_{1,k,j_1}$ with $u_{j_1}(i_{\Omega',1})\geq k$ must either equal $\Omega^+$ or satisfy $u_{j_1}(i_{\Omega',1})>k_{1,c_1-1}$.
Thus we have either $\Omega^+=\Omega_{a,k,j_1}^\natural$ or $\Omega^+<\Omega_{a,k,j_1}^\natural$. The proof is thus finished.
\end{proof}

\begin{lemma}\label{lem: std type II extremal}
Let $\Omega^\pm$ be a constructible $\Lambda$-lift of type $\rm{II}$, and assume that $\Omega^+$ is $\Lambda$-extremal. Then we have $\Omega_{a,k,j}\cap\Omega_{a',k,j}=\emptyset$ for each $1\leq a<a'\leq t$ and $\mathbf{D}^{\Omega^\pm}_{k,j}=\{\bigsqcup_{a\in\Z/t}\Omega_{a,k,j}\}$
for each $(k,j)\in\mathbf{n}_\cJ$.
\end{lemma}
\begin{proof}
The same proof as Lemma~\ref{lem: std type II exceptional} works with the following observation: since $\Omega^+$ is $\Lambda$-extremal, $\Omega_{\psi_1,k}\neq \Omega^+$ for any $k$.
This implies (in notation of \emph{loc.~cit.}):
\begin{itemize}

\item In the proof of Claim~\ref{equ: partition type II exceptional}, $i_{\alpha_{\sharp, s_0'}}^\prime=i_{1,c_1}$ and $i_{\alpha_{\sharp, s_0'}}=i_{1,0}$ can't simultaneously happen.

\item $\Omega_{1,k,j_1}^\natural=\Omega_{1,k,j_1}$ because Lemma~\ref{lem: unique k decomposition} automatically applies.
\end{itemize}
This completes the proof.
\end{proof}

\begin{proof}[Proof of Proposition~\ref{prop: type II exceptional}]
Note that we fix a $\cC\in\cP_\cJ$ satisfying $\cC\subseteq\cN_{\xi,\Lambda}$.
We recall that $I_\cJ^{\Omega^\pm,\star}\subseteq I_\cJ^{\Omega^\pm}$ is the subset consisting of those $(k,j)$ satisfying $\mathbf{D}^{\Omega^\pm}_{k,j}\neq \mathbf{D}^{\Omega^\pm}_{k+1,j}$, and it is clear that $I_\cJ^{\Omega^\pm,\star}\subseteq\mathbf{n}\times\{j_1\}\subseteq\mathbf{n}_\cJ$ in our case.
It follows from Condition~\rm{II}-\ref{it: II 3} and \rm{II}-\ref{it: II 4} that
\begin{multline*}
I_\cJ^{\Omega^\pm,\star}=\{(k_{1,c},j_1)\mid 1\leq c\leq c_1\}\\
\sqcup\{(k_a^{s,e},j_1)\mid 2\leq a\leq t,~1\leq s\leq d_a,~1\leq e\leq e_{a,s}\}\sqcup\{k_{a,c_a}\mid 3\leq a\leq t\}
\end{multline*}
if $k_{2,c_2-1}<k_{1,c_1-1}$ and $k_t^\prime>k_1^{1,1}$,
\begin{multline*}
I_\cJ^{\Omega^\pm,\star}=\{(k_{1,c},j_1)\mid 0\leq c\leq c_1\}\\
\sqcup\{(k_a^{s,e},j_1)\mid 2\leq a\leq t,~1\leq s\leq d_a,~1\leq e\leq e_{a,s}\}\sqcup\{k_{a,c_a}\mid 3\leq a\leq t\}
\end{multline*}
if $k_{2,c_2-1}<k_{1,c_1-1}$ and $k_t^\prime<k_1^{1,1}$,
\begin{multline*}
I_\cJ^{\Omega^\pm,\star}=\{(k_{1,c},j_1)\mid 1\leq c\leq c_1-1\}\\
\sqcup\{(k_a^{s,e},j_1)\mid 2\leq a\leq t,~1\leq s\leq d_a,~1\leq e\leq e_{a,s},~k_a^{s,e}>k_{1,c_1-1}\}\sqcup\{k_{a,c_a}\mid 3\leq a\leq t\}
\end{multline*}
if $k_{2,c_2-1}>k_{1,c_1-1}$ and either $e_{1,1}=0$ or $k_t^\prime>k_1^{1,1}$, and
\begin{multline*}
I_\cJ^{\Omega^\pm,\star}=\{(k_{1,c},j_1)\mid 0\leq c\leq c_1-1\}\\
\sqcup\{(k_a^{s,e},j_1)\mid 2\leq a\leq t,~1\leq s\leq d_a,~1\leq e\leq e_{a,s},~k_a^{s,e}>k_{1,c_1-1}\}\sqcup\{k_{a,c_a}\mid 3\leq a\leq t\}
\end{multline*}
if $k_{2,c_2-1}>k_{1,c_1-1}$, $e_{1,1}\geq 1$ and $k_t^\prime<k_1^{1,1}$.
It follows from Lemma~\ref{lem: general formula for det} and Lemma~\ref{lem: std type II exceptional} that
$$
f_{S^{j,\Omega^\pm}_{k},j}|_{\cN_{\xi,\Lambda}}\sim
\left\{
  \begin{array}{ll}
    F_\xi^{\Omega^+,\star}\prod_{2\leq a\leq t}F_\xi^{\Omega_{a,k,j}} & \hbox{for each $(k,j)$ satisfying $\al^{\Omega^\pm}_{1,k,j}=\al_1$;} \\
    \prod_{a\in\Z/t}F_\xi^{\Omega_{a,k,j}} & \hbox{otherwise}
  \end{array}
\right.
$$
where
\begin{equation}\label{equ: formula for special elt II}
F_\xi^{\Omega^+,\star}\defeq F_\xi^{\Omega^+}+\sum_{\substack{\Omega'\in \mathbf{D}_{(\al_1,j_1)},\, \Omega^+<\Omega'}}\varepsilon(\Omega')F_\xi^{\Omega'}.
\end{equation}
Here, $\varepsilon(\Omega')\in\{-1,1\}$ is a sign determined by $\Omega'$. If $F_\xi^{\Omega^+,\star}|_\cC=0$, then Proposition~\ref{prop: type II exceptional} clearly follows as $F_\xi^{\Omega^+,\star}(F_\xi^{\Omega^-})^{-1}|_\cC=0\in\cO_\cC$. If $F_\xi^{\Omega^+,\star}|_\cC\neq 0$, then we take $(k_\star,j_\star)\defeq (k_{1,c_1-1},j_1)$ and deduce from Lemma~\ref{lem: from sets to formula} and Lemma~\ref{lem: std type II exceptional} that $f_\xi^{\Omega^\pm}\in\Inv(\cC)$ and
\begin{equation}\label{equ: exp formula type II exceptional}
f_\xi^{\Omega^\pm}|_\cC\sim F_\xi^{\Omega^+,\star}|_\cC \underset{a\in\Z/t}{\prod}F_\xi^{\Omega^\pm,a}|_\cC
\end{equation}
where
$$
F_\xi^{\Omega^\pm,a}\defeq (F_\xi^{\Omega_{a,k_\star+1,j_1}})^{-1}\underset{(k,j_1)\in I_\cJ^{\Omega^\pm,\star}\setminus\{(k_\star,j_1)\}}{\prod}F_\xi^{\Omega_{a,k,j_1}}(F_\xi^{\Omega_{a,k+1,j_1}})^{-1}
$$
for each $a\in\Z/t$. For each $a\in\Z/t$, we write $\mathbf{n}^a\subseteq\mathbf{n}\setminus\{k_\star\}$ as the subset consisting of those $k$ satisfying $(k,j_1)\in I_\cJ^{\Omega^\pm,\star}$ and $\Omega_{a,k,j_1}\neq \Omega_{a,k+1,j_1}$. Then we observe from our definition of various $\Omega_{a,k,j_1}$ that
$$\mathbf{n}^a=\{k_{a,c_a}\}\sqcup\{k_a^{s,e}\mid 1\leq s\leq d_a,~1\leq e\leq e_{a,s}\}$$
for each $3\leq a\leq t$,
$$\mathbf{n}^2=\left\{\begin{array}{cl}
\{k_{1,c_1}\}\sqcup\{k_2^{s,e}\mid 1\leq s\leq d_2,~1\leq e\leq e_{2,s}\}&\hbox{if $k_{2,c_2-1}<k_{1,c_1-1}$};\\
\{k_2^{s,e}\mid 1\leq s\leq d_2,~1\leq e\leq e_{2,s},~k_2^{s,e}>k_{1,c_1-1}\}&\hbox{if $k_{2,c_2-1}>k_{1,c_1-1}$},
\end{array}\right.$$
and
$$\mathbf{n}^1=\left\{\begin{array}{cl}
\{k_t^\prime\}\sqcup\{k_{1,c}\mid 1\leq c\leq c_1-2\}&\hbox{if either $e_{1,1}=0$ or $k_t^\prime>k_1^{1,1}$};\\
\{k_{1,c}\mid 0\leq c\leq c_1-2\}&\hbox{if $e_{1,1}\geq 1$ and $k_t^\prime<k_1^{1,1}$}.
\end{array}\right.$$
Then we observe from the definition of $\Omega_{a,k,j_1}$ for each $1\leq a\leq t$ that
\begin{equation}\label{equ: II exceptional formula 0}
F_\xi^{\Omega^\pm,a}=\left(\underset{0\leq c\leq c_a-1}{\prod}u_\xi^{((i_{a,c},i_{a,c+1}),j_1)}\right)^{-1}=(F_\xi^{\Omega_a})^{-1}
\end{equation}
for each $2\leq a\leq t$ and
\begin{equation}\label{equ: II exceptional formula 1}
F_\xi^{\Omega^\pm,1}=1.
\end{equation}
We can clearly combine (\ref{equ: II exceptional formula 0}) and (\ref{equ: II exceptional formula 1}) with (\ref{equ: exp formula type II exceptional}) and deduce that
$$
F_\xi^{\Omega^\pm}|_\cC+\sum_{\Omega_0^\pm}\varepsilon(\Omega_0^\pm)F_\xi^{\Omega_0^\pm}|_\cC\sim F_\xi^{\Omega^+,\star}|_\cC(F_\xi^{\Omega^-}|_\cC)^{-1} =F_\xi^{\Omega^+,\star}|_\cC\underset{2\leq a\leq t}{\prod}(F_\xi^{\Omega_a}|_\cC)^{-1}\sim f_\xi^{\Omega^\pm}|_\cC\in\cO_\cC
$$
where $\Omega_0^\pm$ runs through balanced pair satisfying $\Omega^+<\Omega_0^+$ and $\Omega^-=\Omega_0^-$ with $\varepsilon(\Omega_0^\pm)\defeq \varepsilon(\Omega_0^+)$. The proof is thus finished.
\end{proof}

\begin{proof}[Proof of Proposition~\ref{prop: type II extremal}]
The proof above carries over \emph{verbatim} except that in (\ref{equ: formula for special elt II}) the sum over $\Omega'$ disappears (and we always have $f_\xi^{\Omega^\pm}\in\Inv(\cC)$).
\end{proof}

\subsection{Explicit formula: type~\rm{III}}\label{sub: exp type III}
In this section, we explicitly write down the set $\mathbf{D}^{\Omega^\pm}_{k,j}$ for each $(k,j)\in\mathbf{n}_\cJ$ when $\Omega^\pm$ is a constructible $\Lambda$-lift of type $\rm{III}$. Consequently, we apply Lemma~\ref{lem: from sets to formula} and finish the proof of Proposition~\ref{prop: type III}. We will use frequently all the notation from \S\,\ref{sub: type III}, \S\,\ref{sub: notation for each k j}, and the beginning of \S\,\ref{sec:const:inv}. For simplicity of presentation, we assume $t\geq 3$ throughout this section. The proof of Proposition~\ref{prop: type III} when $t=2$ is simpler and omitted.

We start with constructing a pair of integers $1\leq k_a^\flat<k_a^\sharp\leq n$ and a set $\Omega_{a,k,j}\subseteq\Lambda$ for each $(k,j)\in\mathbf{n}_\cJ$ and $a\in\Z/t$ (uniquely determined by the choice of $v_\cJ^{\Omega^\pm}$). We fix a connected component $\Sigma\in\pi_0(\Omega^\pm)$ for convenience and assume that $a\in(\Z/t)_\Sigma$ throughout the construction. Recall that $(v_\Sigma^{\Omega^\pm})^{-1}|_{\mathbf{n}_\Sigma}$ is an oriented permutation of $\mathbf{n}_\Sigma$ and we always fix a choice of $1$-tour and $-1$-tour as in Definition~\ref{def: oriented permutation}. All the constructions below depend on this choice. Let $\varepsilon\in\{1,-1\}$ be the direction such that $k_{a,1}$ is the $\varepsilon$-successor of $k_{a,0}$, and define $k_a^\sharp$ as follows:
\begin{itemize}
\item If $c_a\geq 2$ and $v_\Sigma^{\Omega^\pm}(k_{a,1})\neq k_{a,1}$, we set $k_a^\sharp\defeq \max\{v_\Sigma^{\Omega^\pm}(k_{a,1}),k_{a,1}\}$.
\item If the fixed $\varepsilon$-tour of $(v_\Sigma^{\Omega^\pm})^{-1}|_{\mathbf{n}_\Sigma}$ contains a $\varepsilon$-jump at $k_{a,0}$ or a $\varepsilon$-crawl at $k_{a,0}$, we set $k_a^\sharp\defeq k_{a,0}$.
\item Otherwise, we define $k_a^\sharp$ as the unique element such that the fixed $\varepsilon$-tour of $(v_\Sigma^{\Omega^\pm})^{-1}|_{\mathbf{n}_\Sigma}$ contains a $\varepsilon$-jump at $k_a^\sharp$ which covers $k_{a,0}$ and satisfies $(v_\Sigma^{\Omega^\pm})^{-1}(k_a^\sharp)\not\in\mathbf{n}^{a,+}\setminus\{k_{a,c_a}\}$.
\end{itemize}
Observe that in the first item either $v_\Sigma^{\Omega^\pm}(k_{a,1})=k_{a,0}$ or there is an $\varepsilon$-jump at $v_\Sigma^{\Omega^\pm}(k_{a,1})$ that covers~$k_{a,0}$.

Let $\varepsilon\in\{1,-1\}$ be the direction such that $\min\mathbf{n}^{a,-}$ is the $\varepsilon$-successor of $k_{a,c_a}$, and define $k_a^\flat$ as follows:
\begin{itemize}
\item If the fixed $\varepsilon$-tour of $(v_\Sigma^{\Omega^\pm})^{-1}|_{\mathbf{n}_\Sigma}$ contains a $\varepsilon$-jump at $k_{a,c_a}$ or a $\varepsilon$-crawl at $k_{a,c_a}$, we set $k_a^\flat\defeq k_{a,c_a}$.
\item Otherwise, we define $k_a^\flat$ as the unique element such that the fixed $\varepsilon$-tour of $(v_\Sigma^{\Omega^\pm})^{-1}|_{\mathbf{n}_\Sigma}$ contains a $\varepsilon$-jump at $k_a^\flat$ which covers $k_{a,c_a}$.
\end{itemize}

The definition of $k_a^\sharp$ and $k_a^\flat$ is visualized in Figure~\ref{fig:sharp:flat}.

\begin{lemma}\label{lem: non empty interval}
The following inequalities hold:
\begin{enumerate}[label=(\roman*)]
\item \label{it: bound sharp} $k_{a,0}\geq k_a^\sharp> k_a^\prime$;
\item \label{it: bound flat} $k_{a,c_a-1}> k_a^\flat\geq k_{a,c_a}$;
\item \label{it: bound sharp flat} $k_a^\sharp>k_a^\flat$,
\end{enumerate}
for each $a\in(\Z/t)_\Sigma$.
\end{lemma}
\begin{proof}
We first check item~\ref{it: bound sharp}. We write $\varepsilon\in\{1,-1\}$ for the unique direction such that $k_{a,1}$ is the $\varepsilon$-successor of $k_{a,0}$. If $k_a^\sharp=k_{a,0}$, then we have nothing to prove. Otherwise, there exists a unique choice of $a'\in(\Z/t)_\Sigma$ such that the fixed $\varepsilon$-tour of $(v_\Sigma^{\Omega^\pm})^{-1}|_{\mathbf{n}_\Sigma}$ contains a $\varepsilon$-jump at $k_{a'}^{[\varepsilon]}$, and this $\varepsilon$-jump covers $k_{a,0}$. If $a'=a-\varepsilon$, then we have $k_{a'}^{[\varepsilon]}<k_{a',0}=k_{a,0}$. If $a'\neq a-\varepsilon$, then we have $k_{a'}^{[\varepsilon]}<k_{a-\varepsilon,c_{a-\varepsilon}-1}\leq k_{a-\varepsilon,0}=k_{a,0}$. This implies $k_a^\sharp\leq \max\{k_{a'}^{[\varepsilon]},k_{a,1}\}<k_{a,0}$. On the other hand, we always have $k_a^\sharp\geq k_{a'}^{[\varepsilon]}>k_a^\prime$ as the $\varepsilon$-jump at $k_{a'}^{[\varepsilon]}$ covers $k_{a,0}$. The proof of item~\ref{it: bound sharp} is thus finished.

Now we check item~\ref{it: bound flat}. We write $\varepsilon\in\{1,-1\}$ for the unique direction such that $\min\mathbf{n}^{a,-}$ is the $\varepsilon$-successor of $k_{a,c_a}$. If $k_a^\flat=k_{a,c_a}$, then we have nothing to prove. Otherwise, there exists a unique choice of $a'\in(\Z/t)_\Sigma$ such that the fixed $\varepsilon$-tour of $(v_\Sigma^{\Omega^\pm})^{-1}|_{\mathbf{n}_\Sigma}$ contains a $\varepsilon$-jump at $k_{a'}^{[\varepsilon]}$, and this $\varepsilon$-jump covers $k_{a,c_a}$. If $a'=a-\varepsilon$, then we have $k_{a'}^{[\varepsilon]}>k_{a',c_{a'}}=k_{a,c_a}$. If $a'\neq a-\varepsilon$, then we have $k_{a'}^{[\varepsilon]}>k_{a-\varepsilon}^\prime\geq k_{a-\varepsilon,c_{a-\varepsilon}}=k_{a,c_a}$. Hence, we always have $k_a^\flat=k_{a'}^{[\varepsilon]}>k_{a,c_a}$. On the other hand, we always have $k_a^\flat= k_{a'}^{[\varepsilon]}<k_{a,c_a-1}$ as the $\varepsilon$-jump at $k_{a'}^{[\varepsilon]}$ covers $k_{a,c_a}$. The proof of item~\ref{it: bound flat} is thus finished.

It remains to check item~\ref{it: bound sharp flat}. If either $k_a^\sharp=k_{a,0}$ or $k_a^\flat=k_{a,c_a}$, then we have nothing to prove. Hence, we assume from now that $k_a^\sharp<k_{a,0}$ and $k_a^\flat>k_{a,c_a}$. We write $\varepsilon\in\{1,-1\}$ for the unique direction such that $\min\mathbf{n}^{a,-}$ is the $\varepsilon$-successor of $k_{a,c_a}$, and thus $k_{a,1}$ is the $-\varepsilon$-successor of $k_{a,0}$. Our assumption ensures the existence of $a',a''\in(\Z/t)_\Sigma$ such that
\begin{itemize}
\item the fixed $\varepsilon$-tour of $(v_\Sigma^{\Omega^\pm})^{-1}|_{\mathbf{n}_\Sigma}$ contains a $\varepsilon$-jump at $k_{a'}^{[\varepsilon]}$, and this $\varepsilon$-jump covers $k_{a,c_a}$;
\item the fixed $-\varepsilon$-tour of $(v_\Sigma^{\Omega^\pm})^{-1}|_{\mathbf{n}_\Sigma}$ contains a $-\varepsilon$-jump at $k_{a''}^{[-\varepsilon]}$, and this $-\varepsilon$-jump covers $k_{a,0}$.
\end{itemize}
These items imply that $k_a^\flat=k_{a'}^{[\varepsilon]}<k_{a,c_a-1}$ and $k_a^\sharp=k_{a''}^{[-\varepsilon]}>k'_a$.
Hence, if $k^\sharp_a\leq k^\flat_a$ we must have $(v_\Sigma^{\Omega^\pm})^{-1}(k_{a'}^{[\varepsilon]})=k_{a,0}$ and $(v_\Sigma^{\Omega^\pm})^{-1}(k_{a''}^{[-\varepsilon]})=k_{a,1}$ and this contradicts item~\ref{it: interaction of jumps} of Definition~\ref{def: oriented permutation}.
\end{proof}

We now define $\Omega_{a,k,j}$. We set $\Omega_{a,k,j}\defeq\emptyset$ if $j\neq j_a$, and set

$$
\Omega_{a,k,j_a}\defeq
\left\{\begin{array}{cl}
\Omega_{\psi_a,k}\setminus\{((i_{a,0},i_{a,1}),j_a)\}&\hbox{if $c_a\geq 2$ and $k_a^\sharp\geq k>\min\{v_{j_a}^{\Omega^\pm}(k_{a,1}),k_a^\sharp\}$};\\
\Omega_{\psi_a,k}&\hbox{if $c_a\geq 2$ and $\min\{v_{j_a}^{\Omega^\pm}(k_{a,1}),k_a^\sharp\}\geq k>k_a^\flat$};\\
\Omega_{\psi_a,k}&\hbox{if $c_a=1$ and $k_a^\sharp\geq k>k_a^\flat$};\\
\emptyset&\hbox{if $k>k_a^\sharp$ or $k\leq k_a^\flat$}.
\end{array}\right.
$$
Observe that when $c_a\geq 2$, the condition $k_a^\sharp>v_\Sigma^{\Omega^\pm}(k_{a,1})$ is equivalent to $k_a^\sharp=k_{a,1}>v_\Sigma^{\Omega^\pm}(k_{a,1})$.
We deduce that $\Omega_{a,k,j}\neq \emptyset$ if and only if $j=j_a$ and $k_a^\sharp\geq k>k_a^\flat$.

For each $a\in\Z/t$ and $(k,j)\in\mathbf{n}_\cJ$, we write $\al^{\Omega^\pm}_{a,k,j}\defeq \sum_{(\beta,j_a)\in\Omega_{a,k,j}}\beta\in\Phi^+_\xi\sqcup\{0\}$, and $\al^{\Omega^\pm}_{a,k,j}=(i_{a,k,j},i_{a,k,j}^\prime)$ whenever $\Omega_{a,k,j}\neq \emptyset$. For technical convenience, we put extra definition $\Omega_{a,n+1,j}\defeq\emptyset$, $\mathbf{D}^{\Omega^\pm}_{n+1,j}\defeq\{\emptyset\}$ and $\al^{\Omega^\pm}_{n+1,j}\defeq 0$ for each $a\in\Z/t$ and $j\in\cJ$. Recall $\mathbf{n}_{\Sigma,1}$ from the paragraph right before Lemma~\ref{lem: backwards jump}. We also write $\mathbf{n}_{\Sigma,-1}$ for the orbit of the fixed $-1$-tour for the oriented permutation $(v_\Sigma^{\Omega^\pm})^{-1}|_{\mathbf{n}_\Sigma}$, for each $\Sigma\in\pi_0(\Omega^\pm)$.

\begin{lemma}\label{lem: std type III}
Let $\Omega^\pm$ be a constructible $\Lambda$-lift of type $\rm{III}$. Then we have
\begin{enumerate}[label=(\roman*)]
\item \label{it: empty intersection} $\{i_{a,k,j},i_{a,k,j}^\prime\}\cap\{i_{a',k,j},i_{a',k,j}^\prime\}=\emptyset$ \emph{(}and thus $\Omega_{a,k,j}\cap\Omega_{a',k,j}=\emptyset$\emph{)} for each $(k,j)\in\mathbf{n}_\cJ$ and each $a\neq a'\in\Z/t$ with $\Omega_{a,k,j}\neq \emptyset\neq \Omega_{a',k,j}$;
\item \label{it: unique decomposition} if $\Omega_{a,k,j_a}\neq\emptyset$, then $\Omega_{a,k,j_a}=\Omega_{(\al^{\Omega^\pm}_{a,k,j_a},j_a),\Lambda}^{\rm{max}}$ is $\Lambda$-exceptional and
    \begin{equation}\label{equ: key bound on k}
    \{\Omega'\in \mathbf{D}_{(\al^{\Omega^\pm}_{a,k,j_a},j_a),\Lambda}\mid u_{j_a}(i_{\Omega',1})\geq k\}=\{\Omega_{a,k,j_a}\};
    \end{equation}
\item \label{it: indexed decomposition}
    for each $(k,j)\in\mathbf{n}_\cJ$, $\bigsqcup_{a\in\Z/t}\Omega_{a,k,j}\in \mathbf{D}^{\Omega^\pm}_{k,j}$;
\item \label{it: gap}
    for each $\Sigma\in\pi_0(\Omega^\pm)$,
    $(k,j)\in(\mathbf{n}_{\Sigma,1}\sqcup\mathbf{n}_{\Sigma,-1})\times\{j_\Sigma\}$ if and only if there exists $a\in(\Z/t)_{\Sigma}$ such that $\Omega_{a,k,j}\neq \Omega_{a,k+1,j}$.
\end{enumerate}
\end{lemma}
\begin{proof}
We first prove item~\ref{it: empty intersection}. It is clear that $i_{a',k,j}\neq i_{a,k,j}^\prime$ for each $a\neq a'\in\Z/t$ and each $(k,j)\in\mathbf{n}_\cJ$, thanks to Conditions~\rm{III}-\ref{it: III 4} and \rm{III}-\ref{it: III 5} and the fact that $\Omega^\pm$ is a $\Lambda$-lift.

If there exist $a\neq a'\in\Z/t$ such that $\Omega_{a,k,j}\neq \emptyset\neq \Omega_{a',k,j}$ and $i_{a,k,j}^\prime=i_{a',k,j}^\prime$, then there exist $\Sigma\in\pi_0(\Omega^\pm)$ and $\varepsilon\in\{1,-1\}$ such that $a,a-\varepsilon\in(\Z/t)_\Sigma$ with $a'=a-\varepsilon$ and $i_{a,k,j}^\prime=i_{a',k,j}^\prime=i_{a,c_a}=i_{a-\varepsilon,c_{a-\varepsilon}}$. Upon exchanging $a,a'$, we may assume that $(v_\Sigma^{\Omega^\pm})^{-1}|_{\mathbf{n}_\Sigma}$ has direction $\varepsilon$ at $k_{a,c_a}$ (cf.~Definition~\ref{def: oriented permutation}) and thus $k_{a-\varepsilon}^\flat>k_a^\flat=k_{a,c_a}$. Hence, $(v_\Sigma^{\Omega^\pm})^{-1}|_{\mathbf{n}_\Sigma}$ has a $-\varepsilon$-jump at $k_{a-\varepsilon}^\flat$ which covers $k_{a,c_a}$, which implies $k_{a-\varepsilon}^\flat\in\{k_a^\sharp,k_{a,c_a-1}\}$.
The condition $i_{a,k,j}^\prime=i_{a,c_a}$ forces $k\leq \min\{k_a^\sharp,k_{a,c_a-1}\}\leq k_{a-\varepsilon}^\flat$ hence $k\leq k_{a-\varepsilon}^{\flat}$.
But $\Omega_{a-\varepsilon,k,j}\neq \emptyset$ forces $k>k_{a-\varepsilon}^\flat$, which is a contradiction.

If there exist $a\neq a'\in\Z/t$ such that $\Omega_{a,k,j}\neq \emptyset\neq \Omega_{a',k,j}$ and $i_{a,k,j}=i_{a',k,j}$, then there exist $\Sigma\in\pi_0(\Omega^\pm)$ and $\varepsilon\in\{1,-1\}$ such that $a,a+\varepsilon\in(\Z/t)_\Sigma$ with $a'=a+\varepsilon$ and $i_{a,k,j}=i_{a',k,j}=i_{a,0}=i_{a+\varepsilon,0}$. Upon exchanging $a,a'$, we may assume that $(v_\Sigma^{\Omega^\pm})^{-1}|_{\mathbf{n}_\Sigma}$ has direction $\varepsilon$ at $k_{a,0}$ (cf.~Definition~\ref{def: oriented permutation}) and thus $k_{a,0}=k_{a+\varepsilon}^\sharp>k_a^\sharp$. Hence, $(v_\Sigma^{\Omega^\pm})^{-1}|_{\mathbf{n}_\Sigma}$ has a $-\varepsilon$-jump at $k'\in\{k_{a+\varepsilon}^\flat,k_{a+\varepsilon}^\prime\}$ which covers $k_{a,0}$.
The condition $i_{a+\varepsilon,k,j}=i_{a,0}$ forces $k>\max\{k_{a+\varepsilon}^\flat\ k_{a+\varepsilon}^\prime\}$ hence $k>k'$.
The condition $i_{a,k,j}=i_{a,0}$ forces either $c_a=1$ or $c_a\geq 2$ and $\min\{v_{j_a}^{\Omega^\pm}(k_{a,1}),k_a^\sharp\}\geq k>k'$.
In the first case $k'=k_{a}^\sharp<k$ and hence $\Omega_{a,k,j}=\emptyset$, a contradiction.
In the second case, since $k_a^\sharp\in\{k_{a,1},\ k'\}$ we have $k_a^\sharp=k_{a,1}\geq k>k'$.
However this forces $v_{j_a}^{\Omega^\pm}(k_{a,1})=k'$, which is a contradiction.
Thus we finished the proof of item~\ref{it: empty intersection}.

Now we consider item~\ref{it: unique decomposition}. If $\Omega_{a,k,j}=\Omega_{\psi_a,k}$, then item~\ref{it: unique decomposition} is a direct consequence of Lemma~\ref{lem: unique k decomposition}. Therefore it suffices to prove item~\ref{it: unique decomposition} when $\Omega_{a,k,j}=\Omega_{\psi_a,k}\setminus\{((i_{a,0},i_{a,1}),j_a)\}\neq \emptyset$. But since    $\Omega_{a,k,j}\subseteq\Omega_{\psi_a,k}$ has $i_{a,k,j_a}^\prime=i_{\al_{\psi_a,k}}^\prime$, this follows again from the corresponding property of $\Omega_{\psi_a,k}$ in Lemma~\ref{lem: unique k decomposition}.

We now prove item~\ref{it: indexed decomposition} for each $(k,j)\in\mathbf{n}_\cJ$ by a decreasing induction on $k$.
The initial step is given by $\Omega_{a,n+1,j}=\emptyset$, $\mathbf{D}^{\Omega^\pm}_{n+1,j}=\{\emptyset\}$ and $\al^{\Omega^\pm}_{n+1,j}=0$ for each $a\in\Z/t$ and $j\in\cJ$.
Suppose we already have $\bigsqcup_{a\in\Z/t}\Omega_{a,k+1,j}\in \mathbf{D}^{\Omega^\pm}_{k+1,j}$ for $(k,j)\in \mathbf{n}_\cJ$. We want to show
\begin{equation}\label{equ: target inclusion}
\underset{a\in\Z/t}{\bigsqcup}\Omega_{a,k,j}\in \mathbf{D}^{\Omega^\pm}_{k,j}.
\end{equation}
We divide into two cases:

\textbf{Case~A:} $(v_j^{\Omega^\pm})^{-1}(k)=w_j^{-1}(k)$, namely $(v_\Sigma^{\Omega^\pm})^{-1}(k)=k$ for each $\Sigma\in\pi_0(\Omega^\pm)$ with $j_\Sigma=j$. Then it is clear that $\al^{\Omega^\pm}_{k,j}=\al^{\Omega^\pm}_{k+1,j}$ and $\mathbf{D}^{\Omega^\pm}_{k,j}\supseteq\mathbf{D}^{\Omega^\pm}_{k+1,j}$. If $(k,j)\notin\mathbf{n}_\Sigma\times\{j_\Sigma\}$ for any $\Sigma\in\pi_0(\Omega^\pm)$, then we clearly have $\Omega_{a,k,j}=\Omega_{a,k+1,j}$ for each $a\in\Z/t$. If there exists $\Sigma\in\pi_0(\Omega^\pm)$ such that $k\in\mathbf{n}_\Sigma$, $j=j_\Sigma$ and $k$ does not lie in the orbit of either the fixed $1$-tour or $-1$-tour of $(v_\Sigma^{\Omega^\pm})^{-1}|_{\mathbf{n}_\Sigma}$, then there exists a unique $a\in(\Z/t)_\Sigma$ such that exactly one of the following holds
\begin{itemize}
\item $k\in\mathbf{n}^{a,+}$, $k>k_a^\sharp$;
\item $k\in\mathbf{n}^{a,-}$, $k<k_a^\flat$,
\end{itemize}
which forces $\Omega_{a,k,j}=\Omega_{a,k+1,j}=\emptyset$ and $\Omega_{a',k,j}=\Omega_{a',k+1,j}$ for each $a'\in\Z/t \setminus\{a\}$. Thus (\ref{equ: target inclusion}) holds.

\textbf{Case~B:} $(v_j^{\Omega^\pm})^{-1}(k)\neq w_j^{-1}(k)$. Thus there exists $\Sigma\in\pi_0(\Omega^\pm)$ and $\varepsilon\in\{1,-1\}$ such that $j_\Sigma=j$ and $k$ lies in the orbit of the fixed $\varepsilon$-tour of $(v_\Sigma^{\Omega^\pm})^{-1}|_{\mathbf{n}_\Sigma}$. Note that $\Sigma$ is uniquely determined by $(k,j)$ (as $\mathbf{n}_\Sigma\cap\mathbf{n}_{\Sigma'}=\emptyset$ for different $\Sigma,\Sigma'\in\pi_0(\Omega^\pm)$ satisfying $j_\Sigma=j_{\Sigma'}=j$). In the following, we prove $\bigsqcup_{a\in\Z/t}\Omega_{a,k,j}\in \mathbf{D}^{\Omega^\pm}_{k,j}$ by a direct comparison between $\Omega_{a,k,j}$ and $\Omega_{a,k+1,j}$ for each $a\in\Z/t$. We use the notation $(i,i')$ for an element of $\Phi^+\sqcup\{0\}\sqcup\Phi^-$, for arbitrary two integers $1\leq i,i'\leq n$.

If there exists $a\in(\Z/t)_\Sigma$ such that $k\in(\mathbf{n}^{a,+}\setminus\{k_{a,c_a}\})\sqcup\{k_{a,0}\}$, $k_{a,1}$ is the $\varepsilon$-successor of $k_{a,0}$ and the fixed $\varepsilon$-tour of $(v_\Sigma^{\Omega^\pm})^{-1}|_{\mathbf{n}_\Sigma}$ contains a $\varepsilon$-crawl at $k$, then by inspection
$\Omega_{a,k,j}=\Omega_{a,k+1,j}\sqcup\{((u_j^{-1}(k),u_j^{-1}(v_\Sigma^{\Omega^\pm})^{-1}(k)),j)\}$, $\al^{\Omega^\pm}_{k,j}=\al^{\Omega^\pm}_{k+1,j}+(u_j^{-1}(k),u_j^{-1}(v_\Sigma^{\Omega^\pm})^{-1}(k))$
and $\Omega_{a',k,j}=\Omega_{a',k+1,j}$ for each $a'\in\Z/t\setminus\{a\}$. Thus (\ref{equ: target inclusion}) holds.

If there exists $a\in(\Z/t)_\Sigma$ such that $k\in(\mathbf{n}^{a,-}\setminus\{k_{a,0}\})\sqcup\{k_{a,c_a}\}$, $\min\mathbf{n}^{a,-}$ is the $\varepsilon$-successor of $k_{a,c_a}$ and the fixed $\varepsilon$-tour of $(v_\Sigma^{\Omega^\pm})^{-1}|_{\mathbf{n}_\Sigma}$ contains a $\varepsilon$-crawl at $k$, then by inspection
we necessarily have $u_j(i_{a,k,j})=k$, $u_j(i_{a,k+1,j})=(v_\Sigma^{\Omega^\pm})^{-1}(k)$, which implies that
$$\al^{\Omega^\pm}_{k,j}-\al^{\Omega^\pm}_{k+1,j}=(u_j^{-1}(k),u_j^{-1}(v_\Sigma^{\Omega^\pm})^{-1}(k))=(i_{a,k,j},i_{a,k+1,j})=\al^{\Omega^\pm}_{a,k,j}-\al^{\Omega^\pm}_{a,k+1,j}.$$
As we clearly have $\Omega_{a',k,j}=\Omega_{a',k+1,j}$ for each $a'\in\Z/t\setminus\{a\}$ in this case, (\ref{equ: target inclusion}) holds.

If the fixed $\varepsilon$-tour of $(v_\Sigma^{\Omega^\pm})^{-1}|_{\mathbf{n}_\Sigma}$ contains a $\varepsilon$-jump at $k=k_a^{[\varepsilon]}$ for some $a,\varepsilon,b$ as in Definition~\ref{def: jump of map}, then we have
\begin{itemize}
\item $\Omega_{a+b'\varepsilon,k,j}=\Omega_{a+b'\varepsilon}$ and $\Omega_{a+b'\varepsilon,k+1,j}=\emptyset$ for each $1\leq b'\leq b-1$ satisfying $k_{a+b'\varepsilon,c_{a+b'\varepsilon}}=k_{a+(b'+1)\varepsilon,c_{a+(b'+1)\varepsilon}}$;
\item $\Omega_{a+b'\varepsilon,k,j}=\emptyset$ and $\Omega_{a+b'\varepsilon,k+1,j}=\Omega_{a+b'\varepsilon}$ for each $1\leq b'\leq b-1$ satisfying $k_{a+b'\varepsilon,0}=k_{a+(b'+1)\varepsilon,0}$;
\item if $k=k_{a,c_a-1}$ and $b\leq b_\Sigma-1$, then we have $\Omega_{a,k,j}=\Omega_{a,k+1,j}\sqcup\{((u_j^{-1}(k),i_{a,c_a}),j)\}$;
\item if $k=k_a^\prime$ and $b\leq b_\Sigma-1$, then we have $\al^{\Omega^\pm}_{a,k,j}=\al^{\Omega^\pm}_{a,k+1,j}-(i_{a,0},u_j^{-1}(k))$;
\item if $k_{a+b\varepsilon,0}=k_{a+(b-1)\varepsilon,0}$ and $b\leq b_\Sigma-1$, then we have $$\Omega_{a+b\varepsilon,k,j}=\Omega_{a+b\varepsilon,k+1,j}\sqcup\{((i_{a+b\varepsilon,0},u_j^{-1}(v_\Sigma^{\Omega^\pm})^{-1}(k)),j)\}$$ with $u_j^{-1}(v_\Sigma^{\Omega^\pm})^{-1}(k)=i_{a+b\varepsilon,1}$ (and note that $k>k_{a+b\varepsilon}^\flat$ follows from item~\ref{it: interaction of jumps} of Definition~\ref{def: oriented permutation} when the fixed $-\varepsilon$-tour of $(v_\Sigma^{\Omega^\pm})^{-1}|_{\mathbf{n}_\Sigma}$ contains a $-\varepsilon$-jump at $v_\Sigma^{\Omega^\pm}(k_{a+b\varepsilon,0})$);
\item if $k_{a+b\varepsilon,c_{a+b\varepsilon}}=k_{a+(b-1)\varepsilon,c_{a+(b-1)\varepsilon}}$ and $b\leq b_\Sigma-1$, then we have $\Omega_{a+b\varepsilon,k,j}=\emptyset$ and $$\al_{a+b\varepsilon,k+1,j}=((u_j^{-1}(v_\Sigma^{\Omega^\pm})^{-1}(k),i_{a+b\varepsilon,c_{a+b\varepsilon}}),j)$$
    (and note that $\Omega_{a+b\varepsilon,k+1,j}=\Omega_{a+b\varepsilon}$ follows from item~\ref{it: interaction of jumps} of Definition~\ref{def: oriented permutation} when $(v_\Sigma^{\Omega^\pm})^{-1}(k)=k_{a+b\varepsilon,0}$,);
\item if $b=b_\Sigma$, then $c_a\geq 2$, $k=k_{a,c_a-1}$, $(v_\Sigma^{\Omega^\pm})^{-1}(k)=k_{a,1}$ (cf.~item~\ref{it: degeneration of tour} of Definition~\ref{def: oriented permutation}) and $$\Omega_{a,k,j}=\Omega_{a,k+1,j}\sqcup\{((u_j^{-1}(k),i_{a,c_a}),j),((i_{a,0},u_j^{-1}(v_\Sigma^{\Omega^\pm})^{-1}(k)),j)\};$$
\item $\Omega_{a',k,j}=\Omega_{a',k+1,j}$ for each $a'\notin\{a+b'\varepsilon\mid 0\leq b'\leq b\}$.
\end{itemize}
In all cases above, we check that
$$\al^{\Omega^\pm}_{k,j}-\al^{\Omega^\pm}_{k+1,j}=(u_j^{-1}(k),u_j^{-1}(v_\Sigma^{\Omega^\pm})^{-1}(k))=\underset{0\leq b'\leq b }{\sum}\al^{\Omega^\pm}_{a+b'\varepsilon,k,j}-\al^{\Omega^\pm}_{a+b'\varepsilon,k+1,j},$$
thus (\ref{equ: target inclusion}) holds.
Note that an example of the above comparison between $\Omega_{a,k,j}$ and $\Omega_{a,k+1,j}$ is visualized in Figure~\ref{fig:fromKtoK1}.

We have now checked (\ref{equ: target inclusion}) in all possible cases, and the argument above actually proves item~\ref{it: gap} at the same time.
\end{proof}

\begin{lemma}\label{lem: full subset case}
Let $\Omega^\pm$ be a constructible $\Lambda$-lift of type $\rm{III}$, $\Sigma\in\pi_0(\Omega^\pm)$ be a connected component, $a\in(\Z/t)_\Sigma^-$ be an element such that there exists $k\in\mathbf{n}$ that satisfies $\Omega_{a,k,j_a}=\Omega_a$. Then there exists unique $k_{a,\star}$ and $k_{a,\star}^\prime$ such that the following conditions are equivalent:
\begin{itemize}
\item $\Omega_{a,k,j_a}=\Omega_a$;
\item $k_{a,\star}\geq k>k_{a,\star}^\prime$.
\end{itemize}
Furthermore, $k_{a,\star}\in\{k_{a,c_a-1}, v_\Sigma^{\Omega^\pm}(k_{a,1}),k_a^\sharp\}\cap\mathbf{n}_{\Sigma,-1}$,
$k_{a,\star}^\prime=\max\{k_a^\prime,k_a^\flat\}\in\mathbf{n}_{\Sigma,1}$, and $k_{a,\star}\leq k_{a,c_a-1}$.
\end{lemma}
\begin{proof} The existence is clear from the definition of $\Omega_{a,k,j_a}$, as is the fact that $k_{a,\star}^\prime=\max\{k_a^\prime,k_a^\flat\}$ and $k_{a,\star}\in \{k_{a,c_a-1}, v_\Sigma^{\Omega^\pm}(k_{a,1}),k_a^\sharp\}$. The fact that $k_{a,\star}\in \mathbf{n}_{\Sigma,-1}$ and $k_{a,\star}^\prime \in  \mathbf{n}_{\Sigma,1}$ follows from a case by case checking using the definition of $k_a^\sharp$ and $k_a^\flat$. %
Finally, we note that $\Omega_{a,k_{a,\star},j_a}=\Omega_a$ forces $\Omega_{\psi_a,k_{a,\star}}=\Omega_a$ which necessarily implies $k_{a,\star}\leq k_{a,c_a-1}$ by the definition of $\Omega_{\psi_a,k_{a,\star}}$. %
\end{proof}

\begin{lemma}\label{lem: criterion for emptyset}
Let $\Omega^\pm$ be a constructible $\Lambda$-lift of type $\rm{III}$, and let $a\in\Z/t$, $\varepsilon\in\{1,-1\}$ and $(k,j)\in\mathbf{n}_\cJ$ be elements such that the following conditions hold:
\begin{itemize}
\item $\Omega_{a,k,j}\neq \emptyset$ and $i_{a,k,j}^\prime=i_{a,c_a}=i_{a+\varepsilon,c_{a+\varepsilon}}$;
\item either $c_{a+\varepsilon}=1$ or $k\leq k_{a+\varepsilon,c_{a+\varepsilon}-1}$.
\end{itemize}
Then we have $\Omega_{a+\varepsilon,k,j}=\emptyset$.
\end{lemma}
\begin{proof}
By checking the definition of $\Omega_{a+\varepsilon,k,j}$, we deduce from the second item that either $\Omega_{a+\varepsilon,k,j}=\emptyset$ or $\Omega_{a+\varepsilon,k,j}\neq \emptyset$ and $i_{a+\varepsilon,k,j}^\prime=i_{a+\varepsilon,c_{a+\varepsilon}}$. However, as we have $i_{a+\varepsilon,k,j}^\prime\neq i_{a,k,j}^\prime$ whenever $\Omega_{a,k,j}\neq \emptyset\neq \Omega_{a+\varepsilon,k,j}$ thanks to item~\ref{it: empty intersection} of Lemma~\ref{lem: std type III}, this together with $i_{a,k,j}^\prime=i_{a,c_a}=i_{a+\varepsilon,c_{a+\varepsilon}}$ forces $\Omega_{a+\varepsilon,k,j}=\emptyset$.
\end{proof}

\begin{lemma}\label{lem: product formula}
Let $\Omega^\pm$ be a constructible $\Lambda$-lift of type $\rm{III}$, and let $\Sigma\in\pi_0(\Omega^\pm)$ be a connected component. Then for each $a'\in(\Z/t)_\Sigma$ we have
$$\underset{k\in\mathbf{n}_{\Sigma,1}}{\prod}F_\xi^{\Omega_{a',k,j_\Sigma}}(F_\xi^{\Omega_{a',k+1,j_\Sigma}})^{-1}=
\left\{\begin{array}{cl}
F_\xi^{\Omega_{a'}}&\hbox{if $a'\in (\Z/t)_\Sigma^+$};\\
(F_\xi^{\Omega_{a'}})^{-1}&\hbox{if $a'\in (\Z/t)_\Sigma^-$}.
\end{array}\right.$$
\end{lemma}
\begin{proof}
For each $a'\in(\Z/t)_\Sigma$, we define $\mathbf{n}_{\Sigma,1}^{a'}\defeq \{k\in\mathbf{n}_{\Sigma,1}\mid \Omega_{a',k,j_\Sigma}\neq \Omega_{a',k+1,j_\Sigma}\}$, and it is clear that
$$\underset{k\in\mathbf{n}_{\Sigma,1}}{\prod}F_\xi^{\Omega_{a',k,j_\Sigma}}(F_\xi^{\Omega_{a',k+1,j_\Sigma}})^{-1}=\underset{k\in\mathbf{n}_{\Sigma,1}^{a'}}{\prod}F_\xi^{\Omega_{a',k,j_\Sigma}}(F_\xi^{\Omega_{a',k+1,j_\Sigma}})^{-1}.$$

Assume that $a'\in(\Z/t)_\Sigma^+$. From the proof of Lemma~\ref{lem: std type III} exactly one of the following holds:
\begin{itemize}
\item if $(\mathbf{n}^{a',+}\setminus\{k_{a',c_{a'}}\})\cap\mathbf{n}_{\Sigma,1}\neq \emptyset$ and $v_\Sigma^{\Omega^\pm}(k_{a',1})\notin \mathbf{n}^{a',+}\setminus\{k_{a',c_{a'}}\}$, then $\mathbf{n}_{\Sigma,1}^{a'}=\{v_\Sigma^{\Omega^\pm}(k_{a',1})\}\sqcup(\mathbf{n}^{a',+}\setminus\{k_{a',c_{a'}}\})$ and

\begin{itemize}
\item[$\circ$]
$\Omega_{a',k_{a',c},j_\Sigma}=\Omega_{a',k_{a',c}+1,j_\Sigma}\sqcup\{((i_{a',c},i_{a',c+1}),j_\Sigma)\}$
for each $1\leq c\leq c_{a'}-1$;
\item[$\circ$] $\Omega_{a',v_\Sigma^{\Omega^\pm}(k_{a',1}),j_\Sigma}=\Omega_{a',v_\Sigma^{\Omega^\pm}(k_{a',1})+1,j_\Sigma}\sqcup\{((i_{a',0},i_{a',1}),j_\Sigma)\}$.
\end{itemize}

\item if $(\mathbf{n}^{a',+}\setminus\{k_{a',c_{a'}}\})\cap\mathbf{n}_{\Sigma,1}\neq \emptyset$ but $v_\Sigma^{\Omega^\pm}(k_{a',1})\in \mathbf{n}^{a',+}\setminus\{k_{a',c_{a'}}\}$, then $v_\Sigma^{\Omega^\pm}(k_{a',1})=k_{a',c_{a'}-1}$, $\mathbf{n}_{\Sigma,1}=\mathbf{n}_{\Sigma,1}^{a'}=\mathbf{n}^{a',+}\setminus\{k_{a',c_{a'}}\}$ and
\begin{itemize}
\item[$\circ$] $\Omega_{a',k_{a',c},j_\Sigma}=\Omega_{a',k_{a',c}+1,j_\Sigma}\sqcup\{((i_{a',c},i_{a',c+1}),j_\Sigma)\}$
for each $1\leq c\leq c_{a'}-2$;
\item[$\circ$] $\Omega_{a',k_{a',c_{a'}-1},j_\Sigma}=\Omega_{a',k_{a',c_{a'}-1}+1,j_\Sigma}\sqcup\{((i_{a',0},i_{a',1}),j_\Sigma),((i_{a',c_{a'}-1},i_{a',c_{a'}}),j_\Sigma)\}$.
\end{itemize}

\item if $(\mathbf{n}^{a',+}\setminus\{k_{a',c_{a'}}\})\cap\mathbf{n}_{\Sigma,1}=\emptyset$, then $\mathbf{n}_{\Sigma,1}^{a'}=\{k_{a'}^\sharp\}$ and  $\Omega_{a',k_{a'}^\sharp,j_\Sigma}=\Omega_{a'}$ and $\Omega_{a',k_{a'}^\sharp+1,j_\Sigma}=\emptyset$.
\end{itemize}
From the above descriptions, in all cases we get
$$\underset{k\in\mathbf{n}_{\Sigma,1}^{a'}}{\prod}{F_\xi^{\Omega_{a',k,j_\Sigma}}}{(F_\xi^{\Omega_{a',k+1,j_\Sigma}})^{-1}}=F_\xi^{\Omega_{a'}}.$$

Assume that $a'\in(\Z/t)_\Sigma^-$. From the proof of Lemma~\ref{lem: std type III} exactly one of the following holds:
\begin{itemize}
\item if $(\mathbf{n}^{a',-}\setminus\{k_{a',0}\})\cap\mathbf{n}_{\Sigma,1}\neq \emptyset$, then $\mathbf{n}_{\Sigma,1}^{a'}=\{k_{a'}^\flat\}\sqcup\{k'\in\mathbf{n}^{a',-}\mid k'\neq k_{a',0},\,k'> k_{a'}^\flat\}$, and
\begin{itemize}
\item[$\circ$] for each $1\leq s\leq d_{a'}$ and $2\leq e\leq e_{a',s}$ satisfying $k_{a'}^{s,e} > k_{a'}^\flat$ we have
$$
\Omega_{a',k_{a'}^{s,e},j_\Sigma}\sqcup\{((i_{a'}^{s,e-1},i_{a',c_{a'}^s}),j_\Sigma)\}=\Omega_{a',k_{a'}^{s,e}+1,j_\Sigma}\sqcup\{((i_{a'}^{s,e},i_{a',c_{a'}^s}),j_\Sigma)\};
$$
\item[$\circ$] if $d_{a'}\geq 1$, $e_{a',1}\geq 1$ and $k_{a'}^{1,1} > k_{a'}^\flat$, then we have
$$\Omega_{a',k_{a'}^{1,1},j_\Sigma}\sqcup\{((i_{a',c},i_{a',c+1}),j_\Sigma)\mid 0\leq c\leq c_{a'}^1-1\}=\Omega_{a',k_{a'}^{1,1}+1,j_\Sigma}\sqcup\{((i_{a'}^{1,1},i_{a',c_{a'}^1}),j_\Sigma)\};$$
\item[$\circ$] for each $2\leq s\leq d_{a'}$ satisfying $k_{a'}^{s,1}>k_{a'}^\flat$ we have
\begin{multline*}
\Omega_{a',k_{a'}^{s,1},j_\Sigma}\sqcup\{((i_{a',c},i_{a',c+1}),j_\Sigma)\mid c_{a'}^{s-1}\leq c\leq c_{a'}^s-1\}\sqcup\{((i_{a'}^{s-1,e_{a',s-1}},i_{a',c_{a'}^{s-1}}),j_\Sigma)\}\\
=\Omega_{a',k_{a'}^{s,1}+1,j_\Sigma}\sqcup\{((i_{a'}^{s,1},i_{a',c_{a'}^s}),j_\Sigma)\};
\end{multline*}
\item[$\circ$] we have
$\Omega_{a',k_{a'}^\flat,j_\Sigma}=\emptyset\,\,\mbox{and}\,\,\widehat{\Omega}_{a',k_{a'}^\flat+1,j_\Sigma}=((u_{j_\Sigma}^{-1}(v_\Sigma^{\Omega^\pm})^{-1}(k_{a'}^\flat),i_{a,c_a}),j_\Sigma).$
\end{itemize}
\item if $(\mathbf{n}^{a',-}\setminus\{k_{a',0}\})\cap\mathbf{n}_{\Sigma,1}=\emptyset$, then $\mathbf{n}_{\Sigma,1}^{a'}=\{k_{a'}^\flat\}$ and $\Omega_{a',k_{a'}^\flat,j_\Sigma}=\emptyset$ and $\Omega_{a',k_{a'}^\flat+1,j_\Sigma}=\Omega_{a'}.$
\end{itemize}
From the above descriptions, in all cases we get
$$\underset{k\in\mathbf{n}_{\Sigma,1}^{a'}}{\prod}{F_\xi^{\Omega_{a',k,j_\Sigma}}}{(F_\xi^{\Omega_{a',k+1,j_\Sigma}})^{-1}}=(F_\xi^{\Omega_{a'}})^{-1}.$$
The proof is thus finished.
\end{proof}

We need the following condition.
\begin{cond}\label{cond: special interval}
The constructible $\Lambda$-lift $\Omega^\pm$ of type $\rm{III}$ satisfies the following conditions:
\begin{itemize}
\item $\Omega^+$ and $\Omega^-$ are not pseudo $\Lambda$-decompositions of the same element in $\widehat{\Lambda}$;
\item for each $\Lambda^\square$-interval $\Omega$ of $\Omega^\pm$ which is a pseudo $\Lambda$-decomposition of some $(\al_\Omega,j)\in\widehat{\Lambda}$, we have
    \begin{itemize}
    \item[$\circ$] $\Omega\setminus\{((i_{\Omega,1},i_{\al_\Omega}^\prime),j)\}\subseteq\big(\Omega_{(\al_\Omega,j),\Lambda}^{\rm{max}}\big)_\dagger$ (cf.~Definition~\ref{def: decomposition of element} for $i_{\Omega,1}$);
    \item[$\circ$] if $i_{\Omega,1}\neq i_{\Omega_{(\al_\Omega,j),\Lambda}^{\rm{max}},1}$, then $((i_{\Omega,1},i_{\al_\Omega}^\prime),j)\in\widehat{\Omega}$ and $u_j(i_{\Omega,1})<u_j(i_{\Omega_{(\al_\Omega,j),\Lambda}^{\rm{max}},1})$;
    \end{itemize}
\end{itemize}
\end{cond}

Recall $\cO_{\xi,\Lambda}^{\rm{ps}}$ from the paragraph before Proposition~\ref{prop: type III}.
\begin{lemma}\label{lem: reduce to type III with cond}
For each constructible $\Lambda$-lift $\Omega^\pm$ of type $\rm{III}$, if $F_\xi^{\Omega^\pm}\not\in\cO_{\xi,\Lambda}^{\rm{ps}}\cdot \cO_{\xi,\Lambda}^{<|\Omega^\pm|}$ then there exists a constructible $\Lambda$-lift $\Omega_0^\pm$ of type $\rm{III}$ satisfying Condition~\ref{cond: special interval} and $F_\xi^{\Omega_0^\pm}(F_\xi^{\Omega^\pm})^{-1}\in\cO_{\xi,\Lambda}^{\rm{ps}}$.
\end{lemma}
\begin{proof}
We define a new balanced pair $\Omega_0^\pm$ by the following step by step replacement: for each $\Lambda^\square$-interval $\Omega$ of $\Omega^\pm$ which is a pseudo $\Lambda$-decomposition of some $(\al_\Omega,j)\in\widehat{\Lambda}$, we replace $\Omega$ inside $\Omega^+$ or $\Omega^-$ with $(\Omega_{(\al_\Omega,j),\Lambda}^{\rm{max}})_\dagger$. It is clear that $F_\xi^{\Omega_0^\pm}(F_\xi^{\Omega^\pm})^{-1}\in\cO_{\xi,\Lambda}^{\rm{ps}}$. Then we check the definition of constructible $\Lambda$-lift of type $\rm{III}$ for the balanced pair $\Omega_0^\pm$ following the arguments in the proof of Theorem~\ref{thm: reduce to constructible}. The proof is finished by the observation that either $\Omega_0^\pm$ is a constructible $\Lambda$-lift $\Omega_0^\pm$ of type $\rm{III}$ satisfying Condition~\ref{cond: special interval}, or it satisfies $F_\xi^{\Omega_0^\pm}\in\cO_{\xi,\Lambda}^{\rm{ps}}\cdot \cO_{\xi,\Lambda}^{<|\Omega^\pm|}$.
\end{proof}

Let $\Omega^\pm$ be a constructible $\Lambda$-lift of type \rm{III} with a fixed choice of $v_\cJ^{\Omega^\pm}$ and $I_\cJ^{\Omega^\pm}$. We write $\pi_0^\square(\Omega^+)$ for the set of $\Lambda^\square$-intervals of $\Omega^\pm$ that are contained in $\Omega^+$. We are mainly interested in the following conditions on $\Omega^\pm$:

\begin{cond}\label{cond: simple type III}
The constructible $\Lambda$-lift $\Omega^\pm$ of type \rm{III} satisfies $\mathbf{D}^{\Omega^\pm}_{k,j}=\{\bigsqcup_{a\in\Z/t}\Omega_{a,k,j}\}$ for each $(k,j)\in\mathbf{n}_\cJ$.
\end{cond}

\begin{cond}\label{cond: non simple type III}
The constructible $\Lambda$-lift $\Omega^\pm$ of type \rm{III} satisfies Condition~\ref{cond: special interval} and the following
\begin{itemize}
\item $\Omega^+\sqcup\Omega^-\subseteq\mathrm{Supp}_{\xi,j}$;
\item for each $a\in(\Z/t)^-$, $\Omega_a\subseteq\Omega^-$ is a $\Lambda^\square$-interval of $\Omega^\pm$ and $k_{a,\star},k_{a,\star}'$ exist (cf.~Lemma~\ref{lem: full subset case});
\item each $\Omega\in\pi_0^\square(\Omega^+)$ is a pseudo $\Lambda$-decomposition of $(\al_\Omega,j)$ for some $(\al_\Omega,j)\in\widehat{\Lambda}$;
\item if we set
\begin{equation}\label{equ: star}
k_\star\defeq\min\left(\{k_{a,\star}\mid a\in(\Z/t)^-\}\cup\{u_j(i_{\Omega_{(\al_\Omega,j),\Lambda}^{\rm{max}},1})\mid \Omega\in\pi_0^\square(\Omega^+)\}\right)
\end{equation} and
\begin{equation}\label{equ: star prime}
k_\star'\defeq\max\{k_{a,\star}^\prime\mid a\in(\Z/t)^-\},
\end{equation}
then we have $k_\star>k_\star'$ and the following are equivalent:
      \begin{itemize}
      \item[$\circ$] $\#\mathbf{D}^{\Omega^\pm}_{k,j}\geq 2$;
      \item[$\circ$] $k_\star\geq k>k_\star'$;
      \item[$\circ$] $\Omega^-=\bigsqcup_{a\in\Z/t}\Omega_{a,k,j}\in\mathbf{D}^{\Omega^\pm}_{k,j}$ and $k\leq u_j(i_{\Omega_{(\al_\Omega,j),\Lambda}^{\rm{max}},1})$ for each $\Omega\in\pi_0^\square(\Omega^+)$;
      \end{itemize}
\item for each $k_\star\geq k>k_\star'$, $\Omega_{k,j}^\natural\in\mathbf{D}^{\Omega^\pm}_{k,j}\setminus\{\Omega^-\}$ if and only if $$\Omega_{k,j}^\natural=\underset{\Omega\in\pi_0^\square(\Omega^+)}{\bigsqcup}\Omega_{\Omega,k,j}^\natural$$
    where $\Omega_{\Omega,k,j}^\natural\in\mathbf{D}_{(\al_\Omega,j),\Lambda}$ with $u_j(i_{\Omega_{\Omega,k,j}^\natural,1})\geq k$ for each $\Omega\in\pi_0^\square(\Omega^+)$.
\end{itemize}
\end{cond}

For each constructible $\Lambda$-lift $\Omega^\pm$ of type $\rm{III}$ satisfying Condition~\ref{cond: non simple type III}, we define $k_\star''$ as the maximal integer (if exists) satisfying the following conditions:
\begin{itemize}
\item there exist $\Omega\in\pi_0^\square(\Omega^+)$ such that $k_\star''=u_j(i_{\Omega',1})$ for some $\Omega'\in\mathbf{D}_{(\al_\Omega,j),\Lambda}$;
\item there exist $\Sigma\in\pi_0(\Omega^\pm)$ such that $(k_\star'',j)\in\,](k^\star,j),(k^\star,j)]_{w_\cJ}$ for some $k^\star\in\mathbf{n}_{\Sigma,1}$;
\item $k_\star> k_\star''>k_\star'$.
\end{itemize}

\begin{cond}\label{cond: non simple type III: terminal}
The constructible $\Lambda$-lift $\Omega^\pm$ of type \rm{III} satisfies Condition~\ref{cond: non simple type III} and $k_\star''$ does not exist.
\end{cond}

\begin{lemma}\label{lem: element in the orbit}
Let $\Omega^\pm$ be a constructible $\Lambda$-lift that satisfies Condition~\ref{cond: non simple type III}. Then we have $k_\star\notin \bigsqcup_{\Sigma\in\pi_0(\Omega^\pm)}\mathbf{n}_{\Sigma,1}$ and $k_\star'\in\bigsqcup_{\Sigma\in\pi_0(\Omega^\pm)}\mathbf{n}_{\Sigma,1}$.
\end{lemma}
\begin{proof}
The fact $k_\star'\in\bigsqcup_{\Sigma\in\pi_0(\Omega^\pm)}\mathbf{n}_{\Sigma,1}$ follows directly from (\ref{equ: star prime}) the fact that $k_{a,\star}'\in\bigsqcup_{\Sigma\in\pi_0(\Omega^\pm)}\mathbf{n}_{\Sigma,1}$ for each $a\in(\Z/t)^-$ by Lemma~\ref{lem: full subset case}.
If $k_\star=k_{a,\star}$ for some $a\in(\Z/t)^-$, we have nothing to prove as $k_{a,\star}\in \bigsqcup_{\Sigma\in\pi_0(\Omega^\pm)}\mathbf{n}_{\Sigma,-1}$ by Lemma~\ref{lem: full subset case}. Hence, we may assume
from now that $k_\star=u_j(i_{\Omega_{(\al_\Omega,j),\Lambda}^{\rm{max}},1})$ for some $\Omega\in\pi_0^\square(\Omega^+)$ (cf.~(\ref{equ: star})). Assume on the contrary that $k_\star=u_j(i_{\Omega_{(\al_\Omega,j),\Lambda}^{\rm{max}},1})\in\mathbf{n}_{\Sigma,1}$ for some $\Sigma\in\pi_0(\Omega^\pm)$. It follows from Condition~\rm{III}-\ref{it: III 9} and Condition~\ref{cond: non simple type III} that there exists a unique $a'\in(\Z/t)_\Sigma^+$ such that $\Omega_{a'}\subseteq\Omega$ and $k_\star=u_j(i_{\Omega_{(\al_\Omega,j),\Lambda}^{\rm{max}},1})\in\mathbf{n}^{a',+}\sqcup\{k_{a',0}\}$. Then we deduce from Condition~\ref{cond: special interval} that $a=a'+1\in(\Z/t)_\Sigma^-$, $i_{a,c_a}=i_{a-1,c_{a-1}}=i_{\al_\Omega}^\prime$ and $u_j(i_{\Omega_{(\al_\Omega,j),\Lambda}^{\rm{max}},1})=k_{a-1,c_{a-1}-1}$.
If the fixed $1$-tour of $(v_\Sigma^{\Omega^\pm})^{-1}|_{\mathbf{n}_\Sigma}$ contains a $1$-jump at $k_{a-1,c_{a-1}-1}$ (which necessarily covers $k_{a,c_a}$), then we have $k_\star'\geq k_{a,\star}'\geq k_a^\flat=k_{a-1,c_{a-1}-1}=k_\star$ which is a contradiction. Otherwise, the fixed $-1$-tour of $(v_\Sigma^{\Omega^\pm})^{-1}|_{\mathbf{n}_\Sigma}$ contains a $-1$-jump at some $k$ which covers $k_{a,c_a}$, which implies that $k_\star=k_{a-1,c_{a-1}-1}>k\geq k_{a,\star}\geq k_\star$, which is also contradiction. The proof is thus finished.
\end{proof}

We have the following classification of constructible $\Lambda$-lifts of type \rm{III}.
\begin{lemma}\label{lem: class of type III}
Let $\Omega^\pm$ be a constructible $\Lambda$-lift of type \rm{III}. If both $\Omega^+$ and $\Omega^-$ are pseudo $\Lambda$-decompositions of some $(\al,j)\in\widehat{\Lambda}$, then Condition~\ref{cond: simple type III} holds. Otherwise, there exists a constructible $\Lambda$-lift $\Omega_0^\pm$ of type \rm{III} satisfying $F_\xi^{\Omega_0^\pm}(F_\xi^{\Omega^\pm})^{-1}\in \cO_{\xi,\Lambda}^{\rm{ps}}\cdot\cO_{\xi,\Lambda}^{<|\Omega^\pm|}$, such that one of the following holds:
\begin{itemize}
\item $F_\xi^{\Omega_0^\pm}\in \cO_{\xi,\Lambda}^{\rm{ps}}\cdot\cO_{\xi,\Lambda}^{<|\Omega^\pm|}$;
\item $\Omega_0^\pm$ satisfies Condition~\ref{cond: simple type III};
\item $\Omega_0^\pm$ satisfies Condition~\ref{cond: non simple type III: terminal}.
\end{itemize}
\end{lemma}

We first prove Proposition~\ref{prop: type III} assuming Lemma~\ref{lem: class of type III}, and then devote the rest of the section to the proof of Lemma~\ref{lem: class of type III}.  We recall $\langle Y\rangle_+$ for a subset $Y\subseteq\cO(\cC)$ from the paragraph right before Proposition~\ref{prop: type III}.

\begin{proof}[Proof of Proposition~\ref{prop: type III} assuming Lemma~\ref{lem: class of type III}]
Note that we fix a $\cC\in\cP_\cJ$ satisfying $\cC\subseteq\cN_{\xi,\Lambda}$.
We recall that $I_\cJ^{\Omega^\pm,\star}\subseteq I_\cJ^{\Omega^\pm}$ is the subset consisting of those $(k,j)$ satisfying $\mathbf{D}^{\Omega^\pm}_{k,j}\neq \mathbf{D}^{\Omega^\pm}_{k+1,j}$. As $\Omega^\pm$ is a constructible $\Lambda$-lift of type \rm{III}, it is clear that
\begin{equation}\label{equ: disjoint from -1}
(\mathbf{n}_{\Sigma,-1}\times\{j_\Sigma\})\cap I_\cJ^{\Omega^\pm}=\emptyset
\end{equation}
for each $\Sigma\in\pi_0(\Omega^\pm)$. It follows from Lemma~\ref{lem: class of type III} that we only need to treat a constructible $\Lambda$-lift $\Omega^\pm$ that satisfies either Condition~\ref{cond: simple type III} or Condition~\ref{cond: non simple type III: terminal}.

We first treat the case when $\Omega^\pm$ satisfies Condition~\ref{cond: simple type III}, namely $\mathbf{D}^{\Omega^\pm}_{k,j}=\{\bigsqcup_{a\in\Z/t}\Omega_{a,k,j}\}$ for each $(k,j)\in\mathbf{n}_\cJ$. It follows from Lemma~\ref{lem: from sets to formula} and $\mathbf{D}^{\Omega^\pm}_{k,j}=\{\bigsqcup_{a\in\Z/t}\Omega_{a,k,j}\}$ that $f_\xi^{\Omega^\pm}\in\Inv(\cC)$ and
$$
f_\xi^{\Omega^\pm}|_\cC\sim \underset{a\in\Z/t}{\prod}F_\xi^{\Omega^\pm,a}|_\cC
$$
with
$$
F_\xi^{\Omega^\pm,a}\defeq \underset{k\in\mathbf{n}_{\Sigma,1}}{\prod}F_\xi^{\Omega_{a,k,j_\Sigma}}(F_\xi^{\Omega_{a,k+1,j_\Sigma}})^{-1}
$$
for each $\Sigma\in\pi_0(\Omega^\pm)$ and each $a\in(\Z/t)_\Sigma$. It follows from Lemma~\ref{lem: product formula} that
$$F_\xi^{\Omega^\pm}|_\cC\sim F_\xi^{\Omega^+}|_\cC(F_\xi^{\Omega^-}|_\cC)^{-1} = \underset{a\in(\Z/t)^+}{\prod}F_\xi^{\Omega_a}|_\cC \cdot \underset{a\in(\Z/t)^-}{\prod}(F_\xi^{\Omega_a}|_\cC)^{-1}\sim f_\xi^{\Omega^\pm}|_\cC\in\cO_\cC.$$

Now we treat the case when $\Omega^\pm$ satisfies Condition~\ref{cond: non simple type III: terminal}, and in particular there does not exist $(\al,j)\in\widehat{\Lambda}$ such that both $\Omega^+$ and $\Omega^-$ are pseudo $\Lambda$-decompositions of $(\al,j)$. Recall from Lemma~\ref{lem: element in the orbit} that $k_\star\notin\bigsqcup_{\Sigma\in\pi_0(\Omega^\pm)}\mathbf{n}_{\Sigma,1}$ and $k_\star'\in\bigsqcup_{\Sigma\in\pi_0(\Omega^\pm)}\mathbf{n}_{\Sigma,1}$. As $\Omega_{a,k,j}$ for each $a\in\Z/t$ remains the same for each $k_\star\geq k>k_\star'$, we deduce that $k\notin\bigsqcup_{\Sigma\in\pi_0(\Omega^\pm)}\mathbf{n}_{\Sigma,1}$ for each $k_\star\geq k>k_\star'$. If there exists any $(k,j)\in I_\cJ^{\Omega^\pm,\star}$ with $k\notin \bigsqcup_{\Sigma\in\pi_0(\Omega^\pm)}\mathbf{n}_{\Sigma,1}$, then we deduce from (\ref{equ: disjoint from -1}), item~\ref{it: gap} of Lemma~\ref{lem: std type III}, and the last two items of Condition~\ref{cond: non simple type III} that $k_\star\geq k>k_\star'$ and there exist $\Omega\in\pi_0^\square(\Omega^+)$, $\Omega'\in\mathbf{D}_{(\al_\Omega,j),\Lambda}$ and $k'\in\mathbf{n}_{\Sigma,1}$ such that $k=u_j(i_{\Omega',1})$ and $(k,j)\in\,](k',j),(k',j)]_{w_\cJ}$. If $k=k_\star$, then it follows from $k_\star\notin\bigsqcup_{\Sigma\in\pi_0(\Omega^\pm)}\mathbf{n}_{\Sigma,1}$, Condition~\rm{III}-\ref{it: III 9} and Condition~\ref{cond: special interval} that $k=k_\star=u_j(i_{\Omega_{(\al_\Omega,j),\Lambda}^{\rm{max}},1})$ is the $1$-end of a connected component of $\Omega^+\sqcup\Omega^-$, which contradicts $(k,j)\in I_\cJ^{\Omega^\pm,\star}\subseteq I_\cJ^{\Omega^\pm}$. Hence, we deduce that $k<k_\star$ and $k_\star''$ must exist (as described after Condition~\ref{cond: non simple type III}) which is a contradiction. Consequently, we have
\begin{equation}\label{eq: I_J^Omega^pm type III}
I_\cJ^{\Omega^\pm,\star}=\left(\underset{\Sigma\in\pi_0(\Omega^\pm)}{\bigsqcup}\mathbf{n}_{\Sigma,1}\right)\times\{j\}.
\end{equation}
and
\begin{itemize}
\item $\mathbf{D}^{\Omega^\pm}_{k,j}=\{\bigsqcup_{a\in\Z/t}\Omega_{a,k,j}\}$ for each $(k,j)\in I_\cJ^{\Omega^\pm,\star}$;
\item $\mathbf{D}^{\Omega^\pm}_{k+1,j}=\{\bigsqcup_{a\in\Z/t}\Omega_{a,k+1,j}\}$ for each $(k,j)\in I_\cJ^{\Omega^\pm,\star}\setminus\{(k_\star',j)\}$.
\end{itemize}
It follows from Lemma~\ref{lem: general formula for det} and Condition~\ref{cond: non simple type III: terminal} that, if there exists $(k,j')\in\mathbf{n}_\cJ$ such that $f_{S^{j',\Omega^\pm}_{k},j'}|_\cC=0$, then we must have $k_\star\geq k>k_\star'$, $j'=j$ and
$$f_{S^{j,\Omega^\pm}_{k},j}|_{\cN_{\xi,\Lambda}}\sim F_\xi^{\Omega^-}+F\neq F_\xi^{\Omega^-}$$
for some polynomial $F$ satisfying $F(F_\xi^{\Omega^+})^{-1}\in\langle\cO_{\xi,\Lambda}^{\rm{ps}}\rangle_+$, which implies that
$$(F_\xi^{\Omega^\pm}|_\cC)^{-1}\sim F_\xi^{\Omega^-}|_\cC(F_\xi^{\Omega^+}|_\cC)^{-1} =-F|_\cC(F_\xi^{\Omega^+}|_\cC)^{-1} \in\langle\cO_\cC^{\rm{ps}}\rangle_+.$$
Hence, we may assume from now on that $f_{S^{j',\Omega^\pm}_{k},j'}|_\cC\neq 0$ for each $(k,j')\in\mathbf{n}_\cJ$.
It follows from (\ref{eq: I_J^Omega^pm type III}), Condition~\ref{cond: non simple type III}, and a simple variant of Lemma~\ref{lem: from sets to formula} that $f_\xi^{\Omega^\pm}\in\Inv(\cC)$ and
$$
(f_\xi^{\Omega^\pm}|_\cC)^{-1}\sim (F_\xi^{\Omega^-}|_\cC+F_\xi^{\Omega^+,\star}|_\cC) \underset{a\in\Z/t}{\prod}(F_\xi^{\Omega^\pm,a}|_\cC)^{-1}
$$
where
$$F_\xi^{\Omega^+,\star}\defeq\underset{\Omega\in\pi_0^\square(\Omega^+)}{\prod}\left(\underset{\Omega'\in\mathbf{D}_{(\al_\Omega,j)}}{\sum}\varepsilon(\Omega,\Omega')F_\xi^{\Omega'}\right)$$
and
$$
F_\xi^{\Omega^\pm,a}\defeq F_\xi^{\Omega_{a,k_\star',j}}\underset{\Sigma\in\pi_0(\Omega^\pm)}{\prod}\left(\underset{k\in\mathbf{n}_{\Sigma,1}\setminus\{k_\star'\}}{\prod}F_\xi^{\Omega_{a,k,j}}(F_\xi^{\Omega_{a,k+1,j}})^{-1}\right)
$$
for each $a\in\Z/t$. Here $\varepsilon(\Omega,\Omega')\in\{1,0,-1\}$ is a constant and
$\varepsilon(\Omega,\Omega')\neq 0$ if and only if $u_j(i_{\Omega',1})>k_\star'$.
It follows from Lemma~\ref{lem: product formula}, Lemma~\ref{lem: std type III}~\ref{it: gap}, and the fourth item in Condition~\ref{cond: non simple type III}  that
$$
F_\xi^{\Omega^\pm,a}=\left\{\begin{array}{cl}
1 & \hbox{if $a\in(\Z/t)^-$};\\
F_\xi^{\Omega_a} & \hbox{if $a\in(\Z/t)^+$}.
\end{array}\right.
$$
Since $F_\xi^{\Omega^+}=\prod_{a\in(\Z/t)^+}F_\xi^{\Omega_a}=\prod_{\Omega\in\pi_0^\square(\Omega^+)}F_\xi^\Omega$ and $\prod_{a\in(\Z/t)^+}F_\xi^{\Omega_a}=\prod_{a\in\Z/t}F_\xi^{\Omega^\pm,a}$, we have
$$F_\xi^{\Omega^-}|_\cC(F_\xi^{\Omega^+}|_\cC)^{-1}+F_\xi^{\Omega^+,\star}|_\cC(F_\xi^{\Omega^+}|_\cC)^{-1}= (F_\xi^{\Omega^-}|_\cC+F_\xi^{\Omega^+,\star}|_\cC)(F_\xi^{\Omega^+}|_\cC)^{-1} \sim(f_\xi^{\Omega^\pm}|_\cC)^{-1}$$
with
$$ F_\xi^{\Omega^+,\star}|_\cC(F_\xi^{\Omega^+}|_\cC)^{-1} =\underset{\Omega\in\pi_0^\square(\Omega^+)}{\prod}\left(\underset{\Omega'\in\mathbf{D}_{(\al_\Omega,j)}}{\sum} \varepsilon(\Omega,\Omega')F_\xi^{\Omega'}|_\cC(F_\xi^\Omega|_\cC)^{-1}\right) \in\langle\cO_\cC^{\rm{ps}}\rangle_+.$$
As $(f_\xi^{\Omega^\pm}|_\cC)^{-1}\in\cO_\cC$,
we conclude that
$$(F_\xi^{\Omega^\pm}|_\cC)^{-1}\sim F_\xi^{\Omega^-}|_\cC(F_\xi^{\Omega^+}|_\cC)^{-1}\in\langle\cO_\cC^{\rm{ps}}\cdot\cO_\cC\rangle_+,$$ which completes the proof.
\end{proof}

The rest of this section is devoted to the proof of Lemma~\ref{lem: class of type III}.
\begin{lemma}\label{lem: pair of connected sets}
Let $\Omega^\pm$ be a constructible $\Lambda$-lift of type $\rm{III}$, and let $(k,j)\in\mathbf{n}_\cJ$ be an element and $a,\, a'\in\Z/t$ be two distinct elements such that $\Omega_{a,k,j}\neq \emptyset\neq \Omega_{a',k,j}$. Assume that there does not exist $(\al,j)\in\widehat{\Lambda}$ such that both $\Omega^+$ and $\Omega^-$ are pseudo $\Lambda$-decompositions of $(\al,j)$. If $((i,i^\prime),j)\in\widehat{\Lambda}$ for some $i\in\mathbf{I}_{\Omega_{a',k,j}}$ and some $i'\in\mathbf{I}_{\Omega_{a,k,j}}^\prime$, then there exists a pseudo $\Lambda$-decomposition $\Omega$ of some $(\al',j)\in\widehat{\Lambda}$ such that $(i_{a',0},j)\in\mathbf{I}_{\widehat{\Omega}}$, $(i_{a,c_a},j)\in\mathbf{I}_{\widehat{\Omega}}^\prime$ and $\Omega\subseteq\Omega^+\sqcup\Omega^-$. More precisely, there exists $\varepsilon\in\{1,-1\}$ such that one of the following holds
\begin{enumerate}[label=(\arabic*)]
\item \label{it: conn 1} $\Omega_a\cap\Omega=\emptyset$, $\Omega_{a'}\subseteq\Omega$ and $i'=i_{a,c_a}=i_{a-\varepsilon,c_{a-\varepsilon}}$;
\item \label{it: conn 2} $\Omega_a\subseteq\Omega$, $\Omega_{a'}\cap\Omega=\emptyset$ and $i=i_{a',0}=i_{a'+\varepsilon,0}$;
\item \label{it: conn 3} $\Omega_a\cap\Omega=\emptyset=\Omega_{a'}\cap\Omega$, $i'=i_{a,c_a}=i_{a-\varepsilon,c_{a-\varepsilon}}$ and $i=i_{a',0}=i_{a'+\varepsilon,0}$;
\item \label{it: conn 4} $\Omega_a,\Omega_{a'}\subseteq\Omega$.
\end{enumerate}
\end{lemma}
\begin{proof}
It follows from $((i,i'),j)\in\widehat{\Lambda}$, $i\in\mathbf{I}_{\Omega_{a',k,j}}$ and $i'\in\mathbf{I}_{\Omega_{a,k,j}}^\prime$ that we have $i_{a',0}\leq i<i'\leq i_{a,c_a}$ and $((i_{a',0},i_{a,c_a}),j)\in\widehat{\Lambda}$. It follows from Condition~$\rm{III}$-\ref{it: III 6} that there exists a pseudo $\Lambda$-decomposition $\Omega$ of some $(\al',j)\in\widehat{\Lambda}$ such that $(i_{a',0},j)\in\mathbf{I}_{\widehat{\Omega}}$, $(i_{a,c_a},j)\in\mathbf{I}_{\widehat{\Omega}}^\prime$ and $\Omega\subseteq\Omega^+\sqcup\Omega^-$. Hence there exist $a_1\in\Z/t$, $\varepsilon\in\{1,-1\}$ and $s_1\leq \#\widehat{\Omega}-1$ such that $i_{a_1,c_{a_1}}=i_{a,c_a}$, $i_{a',0}=i_{a_1-s_1\varepsilon,0}$ and
$$\underset{0\leq s'\leq s_1}{\bigsqcup}\Omega_{a_1-s'\varepsilon}\subseteq \Omega.$$

It follows from $i_{a_1,c_{a_1}}=i_{a,c_a}$ and $i_{a',0}=i_{a_1-s_1\varepsilon,0}$ that exactly one of the following holds
\begin{itemize}
\item $a=a_1+\varepsilon$ and $a'=a_1-s_1\varepsilon$;
\item $a_1=a$ and $a'=a_1-(s_1+1)\varepsilon$;
\item $a=a_1+\varepsilon$ and $a'=a_1-(s_1+1)\varepsilon$;
\item $a_1=a$ and $a'=a_1-s_1\varepsilon$.
\end{itemize}
It is clear that these four cases correspond to the four cases in the statement of the lemma. The equalities involving $i$ and $i'$ follow from the fact that $(i,j)$ and $(i',j)$ should lie in the same $\Lambda^\square$-interval of $\Omega^\pm$, thanks to Condition~\rm{III}-\ref{it: III 7}.
\end{proof}
The four cases listed in Lemma~\ref{lem: pair of connected sets} are visualized in Figure~\ref{fig:4cases}.

\begin{lemma}\label{lem: unique non simple}
Let $\Omega^\pm$ be constructible $\Lambda$-lift of type $\rm{III}$ with $\Omega^+\sqcup\Omega^-$ being circular (cf.~the paragraph before Definition~\ref{def: next element}), and let $j\in\cJ$ be the unique embedding such that $\Omega^+\sqcup\Omega^-\subseteq\mathrm{Supp}_{\xi,j}$. Then there do not exist two different elements $k,k'\in\mathbf{n}$ such that
\begin{enumerate}
\item $\Omega^+=\bigsqcup_{a\in\Z/t}\Omega_{a,k,j}$ and $\Omega^-=\bigsqcup_{a\in\Z/t}\Omega_{a,k',j}$;
\item $k\leq \min\{k_{a,c_a-1}\mid a\in(\Z/t)^-\}$ and $k'\leq \min\{k_{a,c_a-1}\mid a\in(\Z/t)^+\}$.
\end{enumerate}
\end{lemma}
\begin{proof}
We assume without loss of generality that $k>k'$. It follows from Lemma~\ref{lem: full subset case} that $k_{a,c_a-1}\geq k_{a,\star}\geq k>k_{a,\star}'$ (resp.~$k_{a,c_a-1}\geq k_{a,\star}\geq k'>k_{a,\star}'$) for each $a\in(\Z/t)^+$ (resp.~for each $a\in(\Z/t)^-$). In particular, we have $k'<k\leq \min\{k_{a,c_a-1}\mid a\in \Z/t\}$.

We choose an arbitrary $a\in(\Z/t)^-$ and note that $a-1\in(\Z/t)^+$ with $k_{a,c_a}=k_{a-1,c_{a-1}}$. Our assumption together with Condition~\rm{III}-\ref{it: III 9} implies that $\Omega_{a,k,j}=\emptyset$, $\Omega_{a-1,k,j}=\Omega_{a-1}$, $\Omega_{a,k',j}=\Omega_{a}$ and $\Omega_{a-1,k',j}=\emptyset$. Hence, we deduce that $k_{a-1}^\sharp\geq k>\max\{k_{a-1}^\flat,k_a^\sharp\}$ and $\min\{k_{a-1}^\flat,k_a^\sharp\}\geq k'>k_a^\flat$. If $k_{a-1}^\flat<k_a^\sharp$, then for each $k''\in\mathbf{n}$ satisfying $k_a^\sharp\geq k''>k_{a-1}^\flat$, we have $$k'\leq k_{a-1}^\flat<k''\leq k_a^\sharp<k\leq \min\{k_{a,c_a-1}\mid a\in \Z/t\}$$ and $\Omega_{a-1,k'',j}\neq \emptyset\neq \Omega_{a,k'',j}$, which contradicts Lemma~\ref{lem: criterion for emptyset}. Therefore we must have $k_{a-1}^\flat\geq k_a^\sharp>k_{a,c_a}=k_{a-1,c_{a-1}}$, and thus the fixed $-1$-tour of $v_{\Omega^+\sqcup\Omega^-}^{\Omega^\pm}$ contains a $-1$-jump that covers $k_{a,c_a}$.

If there exists $a\in(\Z/t)^-$ such that the fixed $-1$-tour of $v_{\Omega^+\sqcup\Omega^-}^{\Omega^\pm}$ contains a $-1$-jump at $k_{a,c_a-1}$ that covers $k_{a,c_a}$, then we have $k_{a-1}^\flat=k_{a,c_a-1}\geq k>k_{a-1}^\flat$ which is a contradiction. Consequently, for each $a\in(\Z/t)^-$, there exists a unique choice of $a_1,a_2\in(\Z/t)^-$ such that the fixed $-1$-tour of $v_{\Omega^+\sqcup\Omega^-}^{\Omega^\pm}$ contains a $-1$-jump at $k_{a_1}^\prime$ that covers $k_{a,c_a}$, and moreover satisfies $(v_{\Omega^+\sqcup\Omega^-}^{\Omega^\pm})^{-1}(k_{a_1}^\prime)\in \mathbf{n}^{a_2,-}\setminus\{k_{a_2,0}\}$. In particular, we have $k_{a_1}^\prime<(v_{\Omega^+\sqcup\Omega^-}^{\Omega^\pm})^{-1}(k_{a_1}^\prime)\leq k_{a_2}^\prime$. However, if we consider all $-1$-jumps contained in the fixed $-1$-tour of $v_{\Omega^+\sqcup\Omega^-}^{\Omega^\pm}$, we obtain a sequence of elements $a_1,a_2,\dots,a_s\in(\Z/t)^-$ satisfying $k_{a_1}^\prime<k_{a_2}^\prime<\cdots<k_{a_s}^\prime<k_{a_1}^\prime$, which is a contradiction. In all, we have shown that such $k$ and $k'$ do not exist.
\end{proof}

\begin{lemma}\label{lem: unique indexed decomposition}
Let $\Omega^\pm$ be a constructible $\Lambda$-lift of type $\rm{III}$ which satisfies Condition~\ref{cond: special interval}. Then exactly one of the following two possibilities holds:
\begin{itemize}
\item $\Omega^\pm$ satisfies Condition~\ref{cond: simple type III};
\item upon replacing $\Omega^\pm$ with its inverse (cf.~Definition~\ref{def: separated condition}) without changing $v_\cJ^{\Omega^\pm}$, $\Omega^\pm$ satisfies Condition~\ref{cond: non simple type III}.
\end{itemize}
\end{lemma}

\begin{proof}
We fix a pair $(k,j)\in\mathbf{n}_\cJ$ and an arbitrary element $\Omega_{k,j}^\natural\in \mathbf{D}^{\Omega^\pm}_{k,j}$ throughout the proof.

We choose two non-empty subsets $$\Omega_\sharp\subseteq\underset{a\in\Z/t}{\bigsqcup}\Omega_{a,k,j}\,\,\mbox{and}\,\,\Omega_\flat\subseteq\Omega_{k,j}^\natural$$ such that
\begin{equation}\label{equ: balanced condition for a pair}
\sum_{\beta\in\Omega_\sharp}\beta=\sum_{\beta\in\Omega_\flat}\beta
\end{equation}
and there does not exist proper non-empty subsets of $\Omega_\sharp$ and $\Omega_\flat$ satisfying the similar equality. According to (\ref{equ: balanced condition for a pair}) and the minimality condition on $\Omega_\sharp$ and $\Omega_\flat$, there exist $s\geq 1$ and an ordering $\widehat{\Omega}_\sharp=\{(\al_{\sharp,s'},j)\mid s'\in\Z/s\}$ and an ordering $\widehat{\Omega}_\flat=\{(\al_{\flat,s'},j)\mid s'\in\Z/s\}$ such that $i_{\al_{\sharp,s'}}^\prime=i_{\al_{\flat,s'}}^\prime$ and $i_{\al_{\flat,s'}}=i_{\al_{\sharp,s'+1}}$ for each $s'\in\Z/s$. In particular, we observe that
\begin{equation}\label{equ: connecting element}
((i_{\al_{\sharp,s'+1}},i_{\al_{\sharp,s'}}^\prime),j)\in\widehat{\Lambda}
\end{equation}
for each $s'\in\Z/s$. We have the decompositions
$$\Omega_\sharp=\underset{s'\in\Z/s}{\bigsqcup}\Omega_{\sharp,s'}\,\,\mbox{and}\,\,\Omega_\flat=\underset{s'\in\Z/s}{\bigsqcup}\Omega_{\flat,s'}$$
that satisfy $$\sum_{(\beta,j)\in\Omega_{\sharp,s'}}\beta=\al_{\sharp,s'}\,\,\mbox{and}\,\,\sum_{(\beta,j)\in\Omega_{\flat,s'}}\beta=\al_{\flat,s'}$$ for each $s'\in\Z/s$. Moreover, we observe that the sets $\{i_{\al_{\sharp,s'}},i_{\al_{\sharp,s'}}^\prime\}$ are disjoint for different choices of $s'\in\Z/s$ and we have (cf.~Definition~\ref{def: separated condition})
$$\Delta_{\Omega_\sharp}=\Delta_{\Omega_\flat}=\underset{s'\in\Z/s}{\bigsqcup}\{(i_{\al_{\sharp,s'}},j),(i_{\al_{\sharp,s'}}^\prime,j)\}.$$
It follows from $\Omega_\sharp\subseteq \bigsqcup_{a\in\Z/t}\Omega_{a,k,j}$ that, for each $s'\in\Z/s$, there exists a unique $a\in\Z/t$ such that $\Omega_{\sharp,s'}\subseteq\Omega_{a,k,j}\neq \emptyset$, and thus we have a well-defined map $\phi:~\Z/s\rightarrow\Z/t$. If $s=1$, we say that the pair $\Omega_\sharp,\Omega_\flat$ is \emph{simple}.

If the pair $\Omega_\sharp,\Omega_\flat$ is simple, then there exist $a\in\Z/t$ and $(\al,j)\in\widehat{\Lambda}$ (with $j=j_a$) such that $\Omega_\sharp\subseteq\Omega_{a,k,j}$ and $\Omega_\sharp,\Omega_\flat\in\mathbf{D}_{(\al,j),\Lambda}$. Using the minimality condition (under inclusion of subsets) on the choice of $\Omega_\sharp,\Omega_\flat$, we observe that either $\Omega_\sharp=\Omega_\flat=\{(\al,j)\}$ or $\Omega_\sharp\cap\Omega_\flat=\emptyset$. If $\Omega_\sharp\cap\Omega_\flat=\emptyset$ and $i_\al^\prime\neq i_{a,k,j}^\prime$, this contradicts the fact that $\Omega_{a,k,j}$ is $\Lambda$-exceptional (cf. item~\ref{it: unique decomposition} of Lemma~\ref{lem: std type III}). If $\Omega_\sharp\cap\Omega_\flat=\emptyset$ and $i_\al^\prime=i_{a,k,j}^\prime$, we clearly have $i_{\Omega_{a,k,j},1}=i_{\Omega_\sharp,1}\neq i_{\Omega_\flat,1}$ (which implies $u_j(i_{\Omega_\flat,1})<u_j(i_{\Omega_\sharp,1})=u_j(i_{\Omega_{a,k,j},1})$
using $\Omega_{a,k,j}=\Omega_{(\al^{\Omega^\pm}_{a,k,j},j),\Lambda}^{\rm{max}}$), and thus contradicts (\ref{equ: key bound on k}). Consequently, we have shown that, if the pair $\Omega_\sharp,\Omega_\flat$ is simple, then we must have $\Omega_\sharp=\Omega_\flat$.

Now we treat a pair $\Omega_\sharp,\Omega_\flat$ which is not simple, namely $s\geq 2$ and thus $s'-1\neq s'\neq s'+1$ for each $s'\in\Z/s$. It follows from Lemma~\ref{lem: pair of connected sets} that the element (\ref{equ: connecting element}) should fall into one out of four cases there, for each $s'\in\Z/s$. In the following, when we say case~\ref{it: conn 1}, case~\ref{it: conn 2}, case~\ref{it: conn 3}, or case~\ref{it: conn 4}, we are always referring to Lemma~\ref{lem: pair of connected sets}. We will show that for each $s'\in\Z/s$, the element (\ref{equ: connecting element}) falls into case~\ref{it: conn 3}.

If there exists $s'\in\Z/s$ such that the element (\ref{equ: connecting element}) falls into case~\ref{it: conn 1}, then $i_{\al_{\sharp,s'}}^\prime=i_{\phi(s'),c_{\phi(s')}}$ and there exists a pseudo $\Lambda$-decomposition $\Omega_{\star,s',s'+1}$ of $((i,i_{\phi(s'),c_{\phi(s')}}),j)$ for some $i$ such that $\Omega_{\phi(s'+1)}\subseteq\Omega_{\star,s',s'+1}\subseteq\Omega^+\sqcup\Omega^-$ and $\Omega_{\phi(s')}\cap\Omega_{\star,s',s'+1}=\emptyset$. Note that we have $i_{\al_{\sharp,s'+1}}^\prime\neq i_{\al_{\sharp,s'}}^\prime=i_{\phi(s'),c_{\phi(s')}}$. The inclusions $\Omega_{\sharp,s'+1}\subseteq\Omega_{\phi(s'+1),k,j}$ and $\Omega_{\phi(s'+1)}\subseteq\Omega_{\star,s',s'+1}$ imply that
\begin{equation}\label{equ: first case}
((i_{\al_{\sharp,s'+1}}^\prime,i_{\phi(s'),c_{\phi(s')}}),j)\in\widehat{\Lambda}.
\end{equation}
If $((i_{\al_{\sharp,s'+2}},i_{\al_{\sharp,s'+1}}^\prime),j)\in\widehat{\Lambda}$ falls into either case~\ref{it: conn 1} or case~\ref{it: conn 3}, then $(i_{\al_{\sharp,s'+2}}^\prime,j)$ and $(i_{\phi(s'),c_{\phi(s')}},j)$ do not lie in the same $\Lambda^\square$-interval, which together with Condition~\rm{III}-\ref{it: III 7} and (\ref{equ: first case}) leads to a contradiction.
Hence, the element $((i_{\al_{\sharp,s'+2}},i_{\al_{\sharp,s'+1}}^\prime),j)\in\widehat{\Lambda}$ falls into either case~\ref{it: conn 2} or case~\ref{it: conn 4}.

If there exists $s'\in\Z/s$ such that the element (\ref{equ: connecting element}) falls into case~\ref{it: conn 2}, then $i_{\al_{\sharp,s'+1}}=i_{\phi(s'+1),0}$ and there exists a pseudo $\Lambda$-decomposition $\Omega_{\star,s',s'+1}$ of $((i_{\phi(s'+1),0},i'),j)$ for some $i'$ such that $\Omega_{\phi(s')}\subseteq\Omega_{\star,s',s'+1}\subseteq\Omega^+\sqcup\Omega^-$ and $\Omega_{\phi(s'+1)}\cap\Omega_{\star,s',s'+1}=\emptyset$. We can argue similarly using Condition~\rm{III}-\ref{it: III 7} and deduce that the element $((i_{\al_{\sharp,s'}},i_{\al_{\sharp,s'-1}}^\prime),j)\in\widehat{\Lambda}$ falls into either case~\ref{it: conn 1} or case~\ref{it: conn 4}.

If there exists $s'\in\Z/s$ such that the element (\ref{equ: connecting element}) falls into case~\ref{it: conn 4}, then there exists a pseudo $\Lambda$-decomposition $\Omega_{\star,s',s'+1}$ of some $((i,i'),j)$ such that $\Omega_{\phi(s')},\,\Omega_{\phi(s'+1)}\subseteq\Omega_{\star,s',s'+1}\subseteq\Omega^+\sqcup\Omega^-$. In particular, we have
$$((i_{\al_{\sharp,s'}},i'),j),\,((i,i_{\al_{\sharp,s'+1}}^\prime),j)\in\widehat{\Lambda}.$$
Similar argument as above using Condition~\rm{III}-\ref{it: III 7} implies that the element $((i_{\al_{\sharp,s'}},i_{\al_{\sharp,s'-1}}^\prime),j)\in\widehat{\Lambda}$ falls into case~\ref{it: conn 1} or case~\ref{it: conn 4}, and the element $((i_{\al_{\sharp,s'+2}},i_{\al_{\sharp,s'+1}}^\prime),j)\in\widehat{\Lambda}$ falls into either case~\ref{it: conn 2} or case~\ref{it: conn 4}. Moreover, using previous two paragraphs, we obtain a unique choice of $t_+,t_-\geq 1$ such that $((i_{\al_{\sharp,s'+t'+1}},i_{\al_{\sharp,s'+t'}}^\prime),j)\in\widehat{\Lambda}$ falls into case~\ref{it: conn 4} for each $-t_-+1\leq t'\leq t_+-1$, $((i_{\al_{\sharp,s'-t_-+1}},i_{\al_{\sharp,s'-t_-}}^\prime),j)\in\widehat{\Lambda}$ falls into case~\ref{it: conn 1} and $((i_{\al_{\sharp,s'+t_++1}},i_{\al_{\sharp,s'+t_+}}^\prime),j)\in\widehat{\Lambda}$ falls into case~\ref{it: conn 2}.

Combining the three paragraphs above, we conclude that, if there exists $s'\in\Z/s$ such that the element (\ref{equ: connecting element}) does not fall into case~\ref{it: conn 3},
then there exist $s_1',s_2'\in\Z/s$ and $t'\geq 1$ with $s_2'=s_1'+t'$ such that $((i_{\al_{\sharp,s_1'+1}},i_{\al_{\sharp,s_1'}}^\prime),j)\in\widehat{\Lambda}$ falls into case~\ref{it: conn 1}, $((i_{\al_{\sharp,s_2'+1}},i_{\al_{\sharp,s_2'}}^\prime),j)\in\widehat{\Lambda}$ falls into case~\ref{it: conn 2}, and $((i_{\al_{\sharp,s_1'+t''+1}},i_{\al_{\sharp,s_1'+t''}}^\prime),j)\in\widehat{\Lambda}$ falls into case~\ref{it: conn 4} for each $1\leq t''\leq t'-1$. More precisely, we have
\begin{equation}\label{equ: case 1 and 2}
i_{\al_{\sharp,s_1'}}^\prime=i_{\phi(s_1'),c_{\phi(s_1')}},\,\,i_{\al_{\sharp,s_2'+1}}=i_{\phi(s_2'+1),0}
\end{equation}
and $\Omega_{\star,s_1'+t'',s_1'+t''+1}$ is a pseudo $\Lambda$-decomposition of some root that satisfies
\begin{itemize}
\item $\Omega_{\star,s_1',s_1'+1}\subseteq\Omega^+\sqcup\Omega^-$ is a pseudo $\Lambda$-decomposition of $((i,i_{\phi(s_1'),c_{\phi(s_1')}}),j)$ for some $i$ and $\Omega_{\phi(s_1'+1)}\subseteq \Omega_{\star,s_1',s_1'+1}$;
\item $\Omega_{\star,s_1'+t',s_1'+t'+1}\subseteq\Omega^+\sqcup\Omega^-$ is a pseudo $\Lambda$-decomposition of $((i_{\phi(s_1'+t'+1),0},i'),j)$ for some $i'$ and $\Omega_{\phi(s_1'+t')}\subseteq \Omega_{\star,s_1'+t',s_1'+t'+1}$;
\item for each $1\leq t''\leq t'-1$, $\Omega_{\star,s_1'+t'',s_1'+t''+1}\subseteq\Omega^+\sqcup\Omega^-$ is a pseudo $\Lambda$-decomposition (of some root) that contains $\Omega_{\phi(s_1'+t'')}$ and $ \Omega_{\phi(s_1'+t''+1)}$.
\end{itemize}
This forces
$$\Omega_{\star,s_1',s_2'+1}\defeq \underset{0\leq t''\leq t'}{\bigcup}\Omega_{\star,s_1'+t'',s_1'+t''+1}$$
to be a pseudo $\Lambda$-decomposition of $((i_{\phi(s_2'+1),0},i_{\phi(s_1'),c_{\phi(s_1')}}),j)\in\widehat{\Lambda}$ which is also a $\Lambda^\square$-interval of $\Omega^\pm$.
As $\Omega_{\flat,s_1'}$ is a $\Lambda$-decomposition of $((i_{\al_{\sharp,s_1'+1}},i_{\al_{\sharp,s_1'}}^\prime),j)\in\widehat{\Lambda}$ with $u_j(i_{\Omega_{\flat,s_1'},1})\geq k$, we deduce that
\begin{equation*}
k\leq u_j(i_{\Omega_{((i_{\al_{\sharp,s_1'+1}},i_{\al_{\sharp,s_1'}}^\prime),j),\Lambda}^{\rm{max}},1})\leq u_j(i_{\Omega_{((i_{\phi(s_2'+1),0},i_{\phi(s_1'),c_{\phi(s_1')}}),j),\Lambda}^{\rm{max}},1}),
\end{equation*}
which together with $k_{\phi(s_1'+1),0}\geq k>k_{\phi(s_1'+1),c_{\phi(s_1'+1)}}$ (as $\emptyset\neq \Omega_{\sharp,s_1'+1}\subseteq\Omega_{\phi(s_1'+1),k,j}$) and Condition~\ref{cond: special interval} forces that
\begin{equation}\label{equ: common end point}
i_{\phi(s_1'+1),c_{\phi(s_1'+1)}}=i_{\phi(s_1'),c_{\phi(s_1')}}\,\,\mbox{and}\,\,k\leq u_j(i_{\Omega_{\star,s_1',s_2'+1},1}).
\end{equation}
However, as we have (using (\ref{equ: common end point}) and (\ref{equ: case 1 and 2})) $$i_{\al_{\sharp,s_1'}}^\prime=i_{\phi(s_1'),k,j}^\prime=i_{\phi(s_1'),c_{\phi(s_1')}}=i_{\phi(s_1'+1),c_{\phi(s_1'+1)}}$$ and $$k\leq u_j(i_{\Omega_{\star,s_1',s_2'+1},1})=k_{\phi(s_1'+1),c_{\phi(s_1'+1)}-1},$$ we necessarily have $\Omega_{\phi(s_1'+1),k,j}=\emptyset$ by Lemma~\ref{lem: criterion for emptyset}, which contradicts the fact that $\emptyset\neq \Omega_{\sharp,s_1'+1}\subseteq\Omega_{\phi(s_1'+1),k,j}$.

Up to this stage, we have just shown that for each $s'\in\Z/s$, the element (\ref{equ: connecting element}) falls into case~\ref{it: conn 3} in Lemma~\ref{lem: pair of connected sets}, and thus $i_{\al_{\sharp,s'}}^\prime=i_{\phi(s'),c_{\phi(s')}}$, $i_{\al_{\sharp,s'+1}}=i_{\phi(s'+1),0}$ and $\Omega_{\flat,s'}$ is a $\Lambda$-decomposition of $$\al_{\flat,s'}=((i_{\phi(s'+1),0},i_{\phi(s'),c_{\phi(s')}}),j)\in\widehat{\Lambda}$$
satisfying $u_j(i_{\Omega_{\flat,s'},1})\geq k$. Consequently, for each $s'\in\Z/s$, we have
\begin{itemize}
\item $\Omega_{\sharp,s'}$ is a $\Lambda^\square$-interval of $\Omega^\pm$;
\item there exists a $\Lambda^\square$-interval of $\Omega^\pm$ which is a pseudo $\Lambda$-decomposition of $(\al_{\flat,s'},j)$;
\item $\Omega_{\flat,s'}\in\mathbf{D}_{(\al_{\flat,s'},j),\Lambda}$ and
\begin{equation}\label{equ: bound on k prime}
k\leq u_j(i_{\Omega_{\flat,s'},1})\leq u_j(i_{\Omega_{(\al_{\flat,s'},j),\Lambda}^{\rm{max}},1})
\end{equation}
\end{itemize}
Upon replacing $\Omega^\pm$ with its inverse (cf.~Definition~\ref{def: separated condition}), we may assume that $\Omega_\sharp\subseteq\Omega^-$ by the discussion above, and note that
\begin{itemize}
\item $\phi$ is injective, $\phi(\Z/s)=(\Z/t)^-$ and $\Omega_\sharp=\Omega^-$;
\item $j_{\phi(s')}=j$ and $\Omega_{\sharp,s'}=\Omega_{\phi(s'),k,j}=\Omega_{\phi(s')}$ for each $s'\in\Z/s$;
\item for each $\Lambda^\square$-interval $\Omega\in\pi_0^\square(\Omega^+)$, there exists a unique $s'\in\Z/s$ such that $\Omega$ is a pseudo $\Lambda$-decomposition of $(\al_{\flat,s'},j)$.
\end{itemize}
For each $a\in(\Z/t)^-=\phi(\Z/s)$, as $\Omega_{a,k,j}=\Omega_a$, we deduce from Lemma~\ref{lem: full subset case} that $$k_{a,\star}\geq k>k_{a,\star}'.$$
On the other hand, as we have (\ref{equ: bound on k prime}) for each $s'\in\Z/s$, we deduce from Condition~\ref{cond: special interval} that $k\leq k_{a,c_a}$ and so $\Omega_{a,k,j}=\emptyset$ for each $a\in(\Z/t)^+$ with $k_{a,c_a}\neq k_{a',c_{a'}}$ for all $a'\in(\Z/t)^-$. Moreover, if there exist $s'\in\Z/s$ and $a\in(\Z/t)^+$ such that $i_{a,c_a}=i_{\phi(s'),c_{\phi(s')}}=i_{a+1,c_{a+1}}$ and $\Omega_{a,k,j}\neq \emptyset$, then we clearly have $k_{a,0}\geq k>k_{a,c_a}$. Note that $\Omega_{a+1,k,j}\neq\emptyset$ and $i_{\phi(s'),k,j}^\prime=i_{\phi(s'),c_{\phi(s')}}$, as $a+1=\phi(s')\in(\Z/t)^-$. It follows from Condition~\ref{cond: special interval} that either $c_a=1$ or $u_j(i_{\Omega_{(\al_{\flat,s'},j),\Lambda}^{\rm{max}},1})=k_{a,c_a-1}$, which together with (\ref{equ: bound on k prime}) and $k_{a,0}\geq k>k_{a,c_a}$ implies that $k\leq k_{a,c_a-1}$. Now, we apply Lemma~\ref{lem: criterion for emptyset} and conlude $\Omega_{a,k,j}=\emptyset$, which is a contradiction. Hence, we have $\Omega_{a,k,j}=\emptyset$ for each $a\in(\Z/t)^+$ in the current situation.

Now we return to our fixed $(k,j)$ and the associated set $\mathbf{D}^{\Omega^\pm}_{k,j}$. If all possible choices of pairs $\Omega_\sharp,\Omega_\flat$ are simple for each choice of $\Omega_{k,j}^\natural\in \mathbf{D}^{\Omega^\pm}_{k,j}$, then we always have $\Omega_\sharp=\Omega_\flat$ for each choice of pair, and we can write $\bigsqcup_{a\in\Z/t}\Omega_{a,k,j}$ (resp.~$\Omega_{k,j}^\natural$) as disjoint unions of $\Omega_\sharp$ (resp.~$\Omega_\flat$) for certain choices of pairs $\Omega_\sharp,\Omega_\flat$ and deduce that $\Omega_{k,j}^\natural=\bigsqcup_{a\in\Z/t}\Omega_{a,k,j}$.

It remains to consider the case when there exists one choice of $\Omega_{k,j}^\natural\in \mathbf{D}^{\Omega^\pm}_{k,j}$ with $\Omega_{k,j}^\natural\neq \bigsqcup_{a\in\Z/t}\Omega_{a,k,j}$, namely there exists a non-simple pair $\Omega_\sharp,\Omega_\flat$ satisfying $\Omega_\sharp\subseteq \bigsqcup_{a\in\Z/t}\Omega_{a,k,j}$ and $\Omega_\flat\subseteq \Omega_{k,j}^\natural$. According to our previous discussion, upon replacing $\Omega^\pm$ with its inverse, we have
\begin{itemize}
\item $\Omega^+\sqcup\Omega^-\subseteq\mathrm{Supp}_{\xi,j}$;
\item $\Omega_\sharp=\Omega^-=\bigsqcup_{a\in\Z/t}\Omega_{a,k,j}$;
\item $k_{a,\star}\geq k>k_{a,\star}'$ for each $a\in(\Z/t)^-$;
\item for each $\Lambda^\square$-interval $\Omega\in\pi_0^\square(\Omega^+)$, there exists $(\al_\Omega,j)\in\widehat{\Lambda}$ such that $\Omega$ is a pseudo $\Lambda$-decomposition of $(\al_\Omega,j)$;
\item there exists $\Omega_{\Omega,k,j}^\natural\in\mathbf{D}_{(\al_\Omega,j),\Lambda}$ with $u_j(i_{\Omega_{\Omega,k,j}^\natural,1})\geq k$ for each $\Omega\in\pi_0^\square(\Omega^+)$, such that $$\Omega_{k,j}^\natural=\Omega_\flat=\underset{\Omega\in\pi_0^\square(\Omega^+)}{\bigsqcup}\Omega_{\Omega,k,j}^\natural.$$
\end{itemize}

We recall the definition of $k_\star$ and $k_\star'$ from (\ref{equ: star}) and (\ref{equ: star prime}) respectively. Then it follows from previous discussion and Lemma~\ref{lem: unique non simple} that the following conditions on $k'\in\mathbf{n}$ are equivalent (for the above choice of $\Omega^\pm$)%
\begin{itemize}
\item $k_\star\geq k'>k_\star'$;
\item $\#\mathbf{D}^{\Omega^\pm}_{k',j}\geq 2$;
\item $\bigsqcup_{a\in\Z/t}\Omega_{a,k',j}=\Omega^-$ and $k'\leq u_j(i_{\Omega_{(\al_\Omega,j),\Lambda}^{\rm{max}},1})$ for each $\Omega\in\pi_0^\square(\Omega^+)$.
\end{itemize}

Finally, we conclude the following
\begin{itemize}
\item If all possible choices of the pair $\Omega_\sharp,\Omega_\flat$ for all $(k,j)\in\mathbf{n}_\cJ$ are simple, then $\Omega^\pm$ satisfies Condition~\ref{cond: simple type III}.
\item If there exists some $(k,j)\in\mathbf{n}_\cJ$ and a choice of $\Omega_\sharp,\Omega_\flat$ which are not simple, we can decide to choose $\Omega^\pm$ or its inverse from this pair $\Omega_\sharp,\Omega_\flat$ and then define $k_\star,k_\star'\in\mathbf{n}$ as above. The very existence of such $(k,j)$ ensures that $k_\star>k_\star'$. The rest of properties satisfied by the pair $\Omega_\sharp,\Omega_\flat$ implies that either $\Omega^\pm$ or its inverse satisfies Condition~\ref{cond: non simple type III}.
\end{itemize}
The proof is thus finished.
\end{proof}

\begin{lemma}\label{lem: type III pair of pseudo}
Let $\Omega^\pm$ be a constructible $\Lambda$-lift of type $\rm{III}$ with both $\Omega^+$ and $\Omega^-$ being pseudo $\Lambda$-decomposition of some $(\al,j)\in\widehat{\Lambda}$. Then $\Omega^\pm$ satisfies Condition~\ref{cond: simple type III}.
\end{lemma}
\begin{proof}
We borrow all notation around $\Omega_\sharp,\Omega_\flat$ from the proof of Lemma~\ref{lem: unique indexed decomposition}, and prove that all possible choices of pairs $\Omega_\sharp,\Omega_\flat$ are simple, which is enough to conclude the result by the discussion in the proof of Lemma~\ref{lem: unique indexed decomposition}.

If $\phi(\Z/s)\subseteq(\Z/t)^+$ or $\phi(\Z/s)\subseteq(\Z/t)^-$, we choose $s'\in\Z/s$ such that the integer $1\leq m\leq r_\xi$ satisfying $i_{\al_{\phi(s')}}\in[m]_\xi$ is minimal possible and $u_j(i_{\al_{\sharp,s'}}^\prime)$ is maximal possible for the fixed choice of $\phi(s')$. Then we deduce from $((i_{\al_{\sharp,s'+1}},i_{\al_{\sharp,s'}}^\prime),j)\in\widehat{\Lambda}$ that $\phi(s'+1)=\phi(s')$ (by minimality of $m$) and thus $u_j(i_{\al_{\sharp,s'+1}})>u_j(i_{\al_{\sharp,s'+1}}^\prime)\geq u_j(i_{\al_{\sharp,s'}}^\prime)$ (as we have either $s=1$ or $\Omega_{\sharp,s'}\cap\Omega_{\sharp,s'+1}=\emptyset$), which together with maximality of $u_j(i_{\al_{\sharp,s'}}^\prime)$ implies $s=1$.

Assume for the moment that $\phi(\Z/s)\cap(\Z/t)^+\neq \emptyset\neq \phi(\Z/s)\cap(\Z/t)^-$ and $s\geq 2$. We choose $s_1',s_2'\in\Z/s$ such that $\phi(s_1'),\phi(s_2'+1)\in(\Z/t)^+$ and $\phi(s_1'+1),\phi(s_2')\in(\Z/t)^-$. For each $i=1,2$, we deduce from $((i_{\al_{\sharp,s_i'+1}},i_{\al_{\sharp,s_i'}}^\prime),j)\in\widehat{\Lambda}$ and Condition~\rm{III}-\ref{it: III 7} that either $i_{\al_{\sharp,s_i'+1}}=i_\al$ or $i_{\al_{\sharp,s_i'}}^\prime=i_\al^\prime$. As we clearly have $s_1'\neq s_2'$, we deduce that $i_{\al_{\sharp,s_1'+1}}\neq i_{\al_{\sharp,s_2'+1}}$ and $i_{\al_{\sharp,s_1'}}^\prime\neq i_{\al_{\sharp,s_1'}}^\prime$. Hence, we may assume without loss of generality that $i_{\al_{\sharp,s_1'+1}}=i_\al$ and $i_{\al_{\sharp,s_2'}}^\prime=i_\al^\prime$, which implies that $i_{\al_{\sharp,s'}}\neq i_\al$ for each $s'\neq s_1'+1$ and $i_{\al_{\sharp,s'}}^\prime\neq i_\al^\prime$ for each $s'\neq s_2'$. However $i_{\al_{\sharp,s_2'}}^\prime=i_\al^\prime$ also forces $\phi(s_2'-1)\in(\Z/t)^+$, which together with $((i_{\al_{\sharp,s_2'}},i_{\al_{\sharp,s_2'-1}}^\prime),j)\in\widehat{\Lambda}$, Condition~\rm{III}-\ref{it: III 7} and $i_{\al_{\sharp,s_2'-1}}^\prime\neq i_\al^\prime$ forces $i_{\al_{\sharp,s_2'}}=i_\al$ and thus $\al_{\sharp,s_2'}=\al$. This is impossible as $\widehat{\Omega}^-\neq \{(\al,j)\}$.

Up to this stage, we have shown that, if both $\Omega^+$ and $\Omega^-$ are pseudo $\Lambda$-decompositions of some $(\al,j)\in\widehat{\Lambda}$, then all possible choices of pairs $\Omega_\sharp,\Omega_\flat$ are simple. The proof is thus finished
\end{proof}

Let $\Omega^\pm$ be a constructible $\Lambda$-lift of type $\rm{III}$ satisfying Condition~\ref{cond: non simple type III}. Assume for the moment that $\Omega^\pm$ fails Condition~\ref{cond: non simple type III: terminal}, namely there exist $k_\star''$, $\Omega$, $\Omega'$, $\Sigma$, and $k^\star$ as described right after Condition~\ref{cond: non simple type III}. Then we observe that
\begin{itemize}
\item $u_j(i_{\al_\Omega}^\prime)<u_j(i_{\Omega',1})=k_\star''< k_\star\leq u_j(i_{\Omega_{(\al_\Omega,j),\Lambda}^{\rm{max}},1})\leq u_j(i_{\al_\Omega})$ and $k^\star\notin\{u_j(i_{\al_\Omega}),u_j(i_{\al_\Omega}^\prime)\}$;
\item it follows from Condition~\rm{III}-\ref{it: III 9} and $(k_\star'',j)\in\,](k^\star,j),(k^\star,j)]_{w_\cJ}$ that there exists $a^\star\in(\Z/t)_\Sigma^+$ such that $\Omega_{a^\ast}\subseteq\Omega$ and $k^\star\in(\mathbf{n}^{a^\star,+}\setminus\{k_{a^\star,c_{a^\star}}\})\sqcup\{k_{a^\star,0}\}$.
\end{itemize}
Hence, we can define a new balanced pair $\Omega_\diamond^\pm$ by
\begin{equation}\label{eq: Omega_diamond}
\Omega_\diamond^-\defeq \Omega^-\,\,\mbox{ and }\,\,\Omega_\diamond^+\defeq (\Omega^+\setminus\Omega)\sqcup\{((u_j^{-1}(k_\star''),i_{\al_\Omega}^\prime),j)\}\sqcup\Omega''
\end{equation}
where $\Omega''\subsetneq\Omega$ is the unique subset which makes $\{((u_j^{-1}(k_\star''),i_{\al_\Omega}^\prime),j)\}\sqcup\Omega''$ a pseudo $\Lambda$-decomposition of $(\al_\Omega,j)$. Note that we have
\begin{equation}\label{equ: sub of max dagger}
\Omega''\subseteq\Omega\setminus\{((i_{\Omega,1},i_{\al_\Omega}^\prime),j)\}\subseteq \big(\Omega_{(\al_\Omega,j),\Lambda}^{\rm{max}}\big)_\dagger
\end{equation}
as $\Omega^\pm$ satisfies Condition~\ref{cond: special interval}.
It is clear that we always have $F_\xi^{\Omega_\diamond^\pm}(F_\xi^{\Omega^\pm})^{-1}\in\cO_{\xi,\Lambda}^{\rm{ps}}$. If $\Omega_\diamond^\pm$ is a constructible $\Lambda$-lift of type $\rm{III}$, we set $t_\diamond\defeq \#(\widehat{\Omega}_\diamond^+\sqcup\widehat{\Omega}_\diamond^-)$, indexed $\widehat{\Omega}_\diamond^+\sqcup\widehat{\Omega}_\diamond^-$ by $\Z/t_\diamond$, and then define $\Sigma_\diamond\in\pi_0(\Omega_\diamond^\pm)$ and $a_\diamond\in(\Z/t_\diamond)_{\Sigma_\diamond}^-$ by $i_{a_\diamond,c_{a_\diamond}}=i_{a_\diamond-1,c_{a_\diamond-1}}=i_{\al_\Omega}^\prime$. As $k_\star''$ is the $-1$-end of $\Sigma_\diamond$ and $k_\star''<k_\star\leq k_{a,\star}\leq k_{a,c_a-1}=k_{a_\diamond,c_{a_\diamond}-1}$ for the unique $a\in(\Z/t)^-$ satisfying $i_{a,c_a}=i_{a-1,c_{a-1}}=i_{\al_\Omega}^\prime$, we can and do choose $v_{\Sigma_\diamond}^{\Omega_\diamond^\pm}$ such that the fixed $1$-tour of $(v_{\Sigma_\diamond}^{\Omega_\diamond^\pm})^{-1}|_{\mathbf{n}_{\Sigma_\diamond}}$ contains a $1$-jump at $k_\star''=k_{a_\diamond-1,c_{a_\diamond-1}-1}$ (using the construction in Proposition~\ref{prop:oriented:non circular}). If moreover $\Omega_\diamond^\pm$ satisfies Condition~\ref{cond: non simple type III}, we define $k_{\diamond,\star}$ and $k_{\diamond,\star}'$ so that $\#\mathbf{D}^{\Omega_\diamond^\pm}_{k,j}\geq 2$ if and only if $k_{\diamond,\star}\geq k>k_{\diamond,\star}'$.

\begin{lemma}\label{lem: diamond operation}
Let $\Omega^\pm$ be a constructible $\Lambda$-lift of type $\rm{III}$ which satisfies Condition~\ref{cond: non simple type III}. Assume moreover that $k_\star''$ exists as above. If $F_\xi^{\Omega_\diamond^\pm}\not\in\cO_{\xi,\Lambda}^{\rm{ps}}\cdot\cO_{\xi,\Lambda}^{<|\Omega^\pm|}$, then exactly one of the following holds:
\begin{itemize}
\item $\Omega_\diamond^\pm$ is a constructible $\Lambda$-lift of type $\rm{III}$ that satisfies Condition~\ref{cond: simple type III};
\item $\Omega_\diamond^\pm$ is a constructible $\Lambda$-lift of type $\rm{III}$ that satisfies Condition~\ref{cond: non simple type III} and $k_{\diamond,\star}'>k_\star'$.
\end{itemize}
\end{lemma}
\begin{proof}
We keep the notation $\Omega$, $k^\star$, $\Sigma_\diamond$ and $a_\diamond$ attached to $k_\star''$ as above. It is clear that $|\Omega_\diamond^\pm|=|\Omega^\pm|$. If $\Omega_\diamond^\pm$ is not a $\Lambda$-lift, we clearly have $F_\xi^{\Omega_\diamond^\pm}\in \cO_{\xi,\Lambda}^{<|\Omega^\pm|}$ by Lemma~\ref{lem: union of lifts}. If $\Omega_\diamond^\pm$ is a $\Lambda$-lift that violates Condition~\rm{III}-\ref{it: III 4}, \rm{III}-\ref{it: III 5}, \rm{III}-\ref{it: III 7}, \rm{III}-\ref{it: III 8} or \rm{III}-\ref{it: III 9}, then we clearly have $F_\xi^{\Omega_\diamond^\pm}\in\cO_{\xi,\Lambda}^{\rm{ps}}\cdot\cO_{\xi,\Lambda}^{<|\Omega^\pm|}$ by the same argument as in the proof of Theorem~\ref{thm: reduce to constructible}. Otherwise, $\Omega_\diamond^\pm$ automatically satisfies Condition~\rm{III}-\ref{it: III 1}, \rm{III}-\ref{it: III 2}, \rm{III}-\ref{it: III 3} and \rm{III}-\ref{it: III 6}, and thus is a constructible $\Lambda$-lift of type \rm{III}. It is also clear that $\Omega_\diamond^\pm$ automatically satisfies Condition~\ref{cond: special interval} by its construction (using $k_\star''< u_j(i_{\Omega_{(\al_\Omega,j),\Lambda}^{\rm{max}},1})$ and (\ref{equ: sub of max dagger})). If $\Omega_\diamond^\pm$ satisfies Condition~\ref{cond: simple type III}, we have nothing to prove. Otherwise, $\Omega_\diamond^\pm$ satisfies Condition~\ref{cond: non simple type III} (using the fact that the inverse of $\Omega_\diamond^\pm$ can never satisfy Condition~\ref{cond: non simple type III} as $\{((u_j^{-1}(k_\star''),i_{\al_\Omega}^\prime),j)\}\sqcup\Omega''$ is not a $\Lambda$-decomposition of $(\al_\Omega,j)$). As we have chosen $v_\cJ^{\Omega_\diamond^\pm}$ such that the fixed $1$-tour of $(v_{\Sigma_\diamond}^{\Omega_\diamond^\pm})^{-1}|_{\mathbf{n}_{\Sigma_\diamond}}$ contains a $1$-jump at $k_\star''=k_{a_\diamond-1,0}=k_{a_\diamond-1,c_{a_\diamond-1}-1}$ (which necessarily covers $k_{a_\diamond,c_{a_\diamond}}$), we deduce that $k_{\diamond,\star}'\geq k_{a_\diamond,\star}'\geq k_{a_\diamond}^\flat=k_\star''>k_\star'$. The proof is thus finished.
\end{proof}

\begin{lemma}\label{lem: limit of diamond operation}
Let $\Omega^\pm$ be a constructible $\Lambda$-lift of type $\rm{III}$ which satisfies Condition~\ref{cond: non simple type III}. Then there exists a constructible $\Lambda$-lift $\Omega_\heartsuit^\pm$ of type $\rm{III}$ such that $F_\xi^{\Omega_\heartsuit^\pm}(F_\xi^{\Omega^\pm})^{-1}\in\cO_{\xi,\Lambda}^{\rm{ps}}$ and at least one of the following holds
\begin{itemize}
\item $F_\xi^{\Omega_\heartsuit^\pm}\in\cO_{\xi,\Lambda}^{\rm{ps}}\cdot\cO_{\xi,\Lambda}^{<|\Omega^\pm|}$;
\item $\Omega_\heartsuit^\pm$ satisfies Condition~\ref{cond: simple type III};
\item $\Omega_\heartsuit^\pm$ satisfies Condition~\ref{cond: non simple type III: terminal}.
\end{itemize}
\end{lemma}
\begin{proof}
It suffices to treat the case when $F_\xi^{\Omega^\pm}\notin\cO_{\xi,\Lambda}^{\rm{ps}}\cdot\cO_{\xi,\Lambda}^{<|\Omega^\pm|}$, $\Omega^\pm$ satisfies Condition~\ref{cond: non simple type III} and moreover $k_\star''$ exists, as otherwise we can always set $\Omega_\heartsuit^\pm\defeq \Omega^\pm$. Hence, we can define a new balanced pair $\Omega_\diamond^\pm$ as in (\ref{eq: Omega_diamond}) which satisfies $F_\xi^{\Omega_\diamond^\pm}(F_\xi^{\Omega^\pm})^{-1}\in\cO_{\xi,\Lambda}^{\rm{ps}}$. Note that $\Omega_\diamond^\pm$ is necessarily a constructible $\Lambda$-lift of type \rm{III} satisfying Condition~\ref{cond: special interval} as $F_\xi^{\Omega^\pm}\notin\cO_{\xi,\Lambda}^{\rm{ps}}\cdot\cO_{\xi,\Lambda}^{<|\Omega^\pm|}$. If $\Omega_\diamond^\pm$ satisfies Condition~\ref{cond: simple type III}, we set $\Omega_\heartsuit^\pm\defeq \Omega_\diamond^\pm$. Otherwise, it follows from Lemma~\ref{lem: diamond operation} that $\Omega_\diamond^\pm$ satisfies Condition~\ref{cond: non simple type III} with $k_{\diamond,\star}'>k_\star'$. We can repeat the construction $\Omega^\pm\mapsto \Omega_\diamond^\pm$ and carry on an induction on $k_\star'$. The induction must end and as $F_\xi^{\Omega^\pm}\notin\cO_{\xi,\Lambda}^{\rm{ps}}\cdot\cO_{\xi,\Lambda}^{<|\Omega^\pm|}$, we must arrive a constructible $\Lambda$-lift $\Omega_\heartsuit^\pm$ of type $\rm{III}$ which satisfies $F_\xi^{\Omega_\heartsuit^\pm}(F_\xi^{\Omega^\pm})^{-1}\in\cO_{\xi,\Lambda}^{\rm{ps}}$ and falls into the second or third possibility in the statement of this lemma. The proof is thus finished.
\end{proof}

\begin{proof}[Proof of Lemma~\ref{lem: class of type III}]
The first part directly follows from Lemma~\ref{lem: type III pair of pseudo}. The second part also follows from Lemma~\ref{lem: reduce to type III with cond}, Lemma~\ref{lem: unique indexed decomposition}, and Lemma~\ref{lem: limit of diamond operation}.
\end{proof}

\subsection{Main results on invariant functions: proof}\label{sub:main:criterions}
In this section, we combine the results from \S\,\ref{sub:exp:for} and \S\,\ref{sub: cons lifts} to prove our main results on invariant functions, namely Theorem~\ref{thm: constructible and inv fun} and Corollary~\ref{cor: separate points}. In particular, this completes the proof of Statement~\ref{state: goal}.

We fix a $\cC\in\cP_\cJ$ satisfying $\cC\subseteq\cN_{\xi,\Lambda}$ as usual. In the following, we introduce a list of subgroups of $\cO(\cC)^\times$ to be used afterwards. Given a set $S$ of balanced pairs, we can define the \emph{subgroup of $\cO(\cC)^\times$ associated with $S$} by the subgroup generated by $\cO_\cC^{\rm{ss}}$ and $F_\xi^{\Omega^\pm}|_\cC$ for all balanced pairs $\Omega^\pm$ in~$S$.

We recall from Proposition~\ref{prop: type III} that $\cO_\cC^{\rm{ps}}$ is the subgroup of $\cO(\cC)^\times$ associated with the set of balanced pairs $\Omega^\pm$ with both $\Omega^+$ and $\Omega^-$ being pseudo $\Lambda$-decompositions of some $(\al,j)\in\widehat{\Lambda}$. For each $\gamma\in\widehat{\Lambda}^\square$, we write $\cO_\cC^{\rm{ps},\rm{III},\gamma}\subseteq\cO_\cC^{\rm{ps}}$ for the subgroup associated with the set of all balanced pair $\Omega^\pm$ such that both $\Omega^+$ and $\Omega^-$ are pseudo $\Lambda$-decompositions of some $(\al,j)\in\widehat{\Lambda}\cap\mathrm{Supp}_{\xi,\cJ}^\gamma$ with $\widehat{\Omega}^+\neq\{(\al,j)\}\neq\widehat{\Omega}^-$ (cf. Proposition~\ref{prop: special type III}). We write $\cO_\cC^{\rm{I},\rm{ext}}\subseteq\cO_\cC^{\rm{con}}$ (resp.~$\cO_\cC^{\rm{II},\rm{ext}}\subseteq\cO_\cC^{\rm{con}}$) for the subgroup associated with the set of all constructible $\Lambda$-lifts $\Omega^\pm$ of type $\rm{I}$ (resp.~of type $\rm{II}$) with $\Omega^+$ being $\Lambda$-extremal.

We say that a balanced pair $\Omega^\pm$ is of \emph{type \rm{I}-\rm{max}} (cf. Lemma~\ref{lem: special case 3}) if there exists $(\al,j)\in\widehat{\Lambda}$ such that
\begin{itemize}
\item $\Omega^-=\Omega_{(\al,j),\Lambda}^{\rm{max}}$ is $\Lambda$-exceptional and $\Lambda$-ordinary;
\item $\Omega^+\in\mathbf{D}_{(\al,j),\Lambda}$ satisfies $i_{\Omega^+,1}=i_\psi^{1,e}$ for some $1\leq e\leq e_{\psi,1}$ (with $\psi=(\Omega^-,\Lambda)$).
\end{itemize}
For each $\gamma\in\widehat{\Lambda}^\square$, we write $\cO_\cC^{\rm{I},\rm{max},\gamma}\subseteq\cO(\cC)^\times$ for the subgroup associated with the set of all balanced pairs $\Omega^\pm$ of type \rm{I}-\rm{max} with $|\Omega^\pm|=\gamma$.

We say that a balanced pair $\Omega^\pm$ is of \emph{type \rm{II}-\rm{max}} (cf. Proposition~\ref{prop: special type II}) if there exists $(\al,j)\in\widehat{\Lambda}$ such that
\begin{itemize}
\item $\Omega^-$ is a pseudo $\Lambda$-decomposition of $(\al,j)$ satisfying $\widehat{\Omega}^-\neq \{(\al,j)\}$;
\item $\Omega^+\in\mathbf{D}_{(\al,j),\Lambda}$ is $\Lambda$-ordinary;
\item either $\Omega^+=\Omega_{(\al,j),\Lambda}^{\rm{max}}$ or $\Omega^+$ is $\Lambda$-extremal.
\end{itemize}
For each $\gamma\in\widehat{\Lambda}^\square$, we write $\cO_\cC^{\rm{II},\rm{max},\gamma}\subseteq\cO(\cC)^\times$ for the subgroup associated with the set of all balanced pairs $\Omega^\pm$ of type \rm{II}-\rm{max} with $|\Omega^\pm|=\gamma$.

We say that a balanced pair $\Omega^\pm$ is of \emph{type \rm{I}-\rm{exp}} (cf. Proposition~\ref{prop: maximal vers exceptional: ord}) if there exists $(\al,j)\in\widehat{\Lambda}$ such that
\begin{itemize}
\item $\Omega^-=\Omega_{(\al,j),\Lambda}^{\rm{max}}$ is $\Lambda$-ordinary;
\item $\Omega^+\in\mathbf{D}_{(\al,j),\Lambda}$ is $\Lambda$-exceptional and $\Lambda$-ordinary.
\end{itemize}
For each balanced pair $\Omega^\pm$ of type \rm{I}-\rm{exp}, we write $\cO_\cC^{\Omega^\pm,<}\subseteq\cO(\cC)^\times$ for the subgroup associated with the set of all balanced pairs $\Omega_0^\pm$ satisfying
\begin{itemize}
\item $\Omega_0^\pm$ is of type \rm{I}-\rm{exp};
\item $\Omega_0^-=\Omega^-$ and $\Omega^+<\Omega_0^+$.
\end{itemize}
For each $\gamma\in\widehat{\Lambda}^\square$, we write $\cO_\cC^{\rm{I},\rm{exp},\gamma}\subseteq\cO(\cC)^\times$ for the subgroup associated with the set of all balanced pairs $\Omega^\pm$ of type \rm{I}-\rm{exp} with $|\Omega^\pm|=\gamma$.

We say that a balanced pair $\Omega^\pm$ is of \emph{type \rm{II}-\rm{exp}} (cf. Proposition~\ref{prop: maximal vers exceptional: non ord}) if there exists $(\al,j)\in\widehat{\Lambda}$ such that
\begin{itemize}
\item $\Omega_{(\al,j),\Lambda}^{\rm{max}}$ is not $\Lambda$-ordinary and $\Omega^-=(\Omega_{(\al,j),\Lambda}^{\rm{max}})_\dagger$;
\item $\Omega^+\in\mathbf{D}_{(\al,j),\Lambda}$ is $\Lambda$-exceptional and $\Lambda$-ordinary.
\end{itemize}
For each balanced pair $\Omega^\pm$ of type \rm{II}-\rm{exp}, we write $\cO_\cC^{\Omega^\pm,<}\subseteq\cO(\cC)^\times$ for the subgroup associated with the set of all balanced pairs $\Omega_0^\pm$ satisfying
\begin{itemize}
\item $\Omega_0^\pm$ is of \emph{type \rm{II}-\rm{exp}};
\item $\Omega_0^-=\Omega^-$ and $\Omega^+<\Omega_0^+$.
\end{itemize}
For each $\gamma\in\widehat{\Lambda}^\square$, we write $\cO_\cC^{\rm{II},\rm{exp},\gamma}\subseteq\cO(\cC)^\times$ for the subgroup associated with the set of all balanced pairs $\Omega^\pm$ of type \rm{II}-\rm{exp} with $|\Omega^\pm|=\gamma$.

We recall the subring $\cO_\cC\subseteq\cO(\cC)$ from Definition~\ref{def: similar and below a block}.  We also recall $\langle Y\rangle_+$ for a subset $Y\subseteq\cO(\cC)$ from the paragraph right before Proposition~\ref{prop: type III}.

\begin{lemma}\label{lem: ps III}
Let $(\al,j)\in\widehat{\Lambda}\cap\mathrm{Supp}_{\xi,\cJ}^\gamma$ be an element. If $\Omega^\pm$ is a balanced pair such that
\begin{itemize}
\item $\Omega^-$ is a pseudo $\Lambda$-decomposition of $(\al,j)$ with $\widehat{\Omega}^-\neq\{(\al,j)\}$;
\item $\Omega^+$ is a pseudo $\Lambda$-decomposition of $(\al,j)$ with either $\widehat{\Omega}^+\neq\{(\al,j)\}$ or $\Omega^+\in\mathbf{D}_{(\al,j),\Lambda}$ being not $\Lambda$-ordinary;
\end{itemize}
Then we have $$F_\xi^{\Omega^\pm}|_\cC\in\langle\cO_\cC^{<\gamma}\cdot\cO_\cC\rangle_+.$$ In particular, we have $\cO_\cC^{\rm{ps},\rm{III},\gamma}\subseteq \langle\cO_\cC^{<\gamma}\cdot\cO_\cC\rangle_+$.
\end{lemma}

\begin{proof}
If $\Omega^+\in\mathbf{D}_{(\al,j),\Lambda}$ is not $\Lambda$-ordinary, we can replace the balanced pair $\Omega^\pm$ with $\Omega^+_\dagger,\Omega^-$. Consequently, we may assume without loss of generality that both $\Omega^+$ and $\Omega^-$ are pseudo $\Lambda$-decompositions of $(\al,j)$ with $\widehat{\Omega}^+\neq\{(\al,j)\}\neq \widehat{\Omega}^-$. Then it follows from Proposition~\ref{prop: special type III} that there exists a constructible $\Lambda$-lift $\Omega_0^\pm$ of type $\rm{III}$ such that
\begin{itemize}
\item both $\Omega_0^+$ and $\Omega_0^-$ are pseudo $\Lambda$-decompositions of $(\al,j)$;
\item $F_\xi^{\Omega_0^\pm}|_\cC(F_\xi^{\Omega^\pm}|_\cC)^{-1} \in \cO_\cC^{<\gamma}$.
\end{itemize}
By Proposition~\ref{prop: type III} we have $F_\xi^{\Omega_0^\pm}\in\cO_\cC$, and so the proof is finished by $F_\xi^{\Omega_0^\pm}|_\cC(F_\xi^{\Omega^\pm}|_\cC)^{-1}\in \cO_\cC^{<\gamma}$.
\end{proof}

\begin{lemma}\label{lem: I II ext}
We have $$\cO_\cC^{\rm{I},\rm{ext}},\,\cO_\cC^{\rm{II},\rm{ext}}\subseteq\cO_\cC.$$
\end{lemma}
\begin{proof}
This follows directly from Proposition~\ref{prop: type I extremal} and Proposition~\ref{prop: type II extremal}.
\end{proof}

\begin{lemma}\label{lem: I max}
If $\Omega^\pm$ is a balanced pair of type \rm{I}-\rm{max} with $|\Omega^\pm|=\gamma$ for some $\gamma\in\widehat{\Lambda}^\square$, then we have
$$F_\xi^{\Omega^\pm}|_\cC\in \langle\cO_\cC^{<\gamma}\cdot\cO_\cC\rangle_+.$$
Moreover, we have $\cO_\cC^{\rm{I},\rm{max},\gamma}\subseteq \langle\cO_\cC^{<\gamma}\cdot\cO_\cC\rangle_+$.
\end{lemma}

\begin{proof}
We write $(\al,j)\in\widehat{\Lambda}\cap\mathrm{Supp}_{\xi,\cJ}^\gamma$ for the element such that $\Omega^-=\Omega_{(\al,j),\Lambda}^{\rm{max}}$. If $F_\xi^{\Omega^\pm}|_\cC\in \cO_\cC^{<\gamma}$, then we have nothing to prove. Otherwise, it follows from Lemma~\ref{lem: special case 3} that there exists a pseudo $\Lambda$-decomposition $\Omega'$ of $(\al,j)$, which is $\Lambda$-equivalent to $\Omega^+$ with level $<\gamma$, such that $i_{\Omega',1}=i_{\psi_2}^{1,1}$ (with $\psi_2=(\Omega^-,\Lambda)$) and exactly one of the following holds:
\begin{itemize}
\item $\Omega'\in\mathbf{D}_{(\al,j),\Lambda}$ and the balanced pair $\Omega',\Omega^-$ is a constructible $\Lambda$-lift of type $\rm{I}$;
\item $\widehat{\Omega}'\neq \{(\al,j)\}$ and the balanced pair $\Omega^-,\Omega'$ is a constructible $\Lambda$-lift of type $\rm{II}$.
\end{itemize}
We set
$$\Omega_0^+\defeq\left\{
  \begin{array}{ll}
    \Omega' & \hbox{if $\Omega'\in\mathbf{D}_{(\al,j),\Lambda}$;} \\
    \Omega^- & \hbox{if $\widehat{\Omega}'\neq \{(\al,j)\}$}
  \end{array}
\right.
\,\,\mbox{ and }\,\,\Omega_0^-\defeq\left\{
  \begin{array}{ll}
    \Omega^- & \hbox{if $\Omega'\in\mathbf{D}_{(\al,j),\Lambda}$;} \\
    \Omega' & \hbox{if $\widehat{\Omega}'\neq \{(\al,j)\}$.}
  \end{array}
\right.
$$
If $\Omega'\in\mathbf{D}_{(\al,j),\Lambda}$ is $\Lambda$-extremal, then it follows from Proposition~\ref{prop: type I extremal} that $F_\xi^{\Omega_0^\pm}|_\cC\in\cO_\cC$ and thus $F_\xi^{\Omega^\pm}|_\cC\in \langle\cO_\cC^{<\gamma}\cdot\cO_\cC\rangle_+$. If $\Omega'\in\mathbf{D}_{(\al,j),\Lambda}$ is $\Lambda$-exceptional, then it follows from Proposition~\ref{prop: type I exceptional} that $F_\xi^{\Omega_0^\pm}|_\cC+\varepsilon\in\cO_\cC$ for some $\varepsilon\in\{1,-1\}$ and thus $F_\xi^{\Omega^\pm}\in \langle\cO_\cC^{<\gamma}\cdot\cO_\cC\rangle_+$. If $\widehat{\Omega}'\neq \{(\al,j)\}$, then it follows from Proposition~\ref{prop: type II exceptional}, Proposition~\ref{prop: type II extremal} and $\Omega_0^+=\Omega_{(\al,j),\Lambda}^{\rm{max}}$ that $F_\xi^{\Omega_0^\pm}|_\cC\in\cO_\cC$ and thus $F_\xi^{\Omega^\pm}|_\cC\in \langle\cO_\cC^{<\gamma}\cdot\cO_\cC\rangle_+$. The proof is thus finished.
\end{proof}

\begin{lemma}\label{lem: II max}
If $\Omega^\pm$ is a balanced pair of type \rm{II}-\rm{max} with $|\Omega^\pm|=\gamma$ for some $\gamma\in\widehat{\Lambda}^\square$, then we have
$$F_\xi^{\Omega^\pm}|_\cC\in \langle\cO_\cC^{\rm{ps},\rm{III},\gamma}\cdot\cO_\cC^{\rm{I},\rm{max},\gamma} \cdot\cO_\cC^{<\gamma}\cdot\cO_\cC\rangle_+.$$
Moreover, we have $\cO_\cC^{\rm{II},\rm{max},\gamma}\subseteq \langle\cO_\cC^{<\gamma}\cdot\cO_\cC\rangle_+$.
\end{lemma}

\begin{proof}
We write $(\al,j)\in\widehat{\Lambda}\cap\mathrm{Supp}_{\xi,\cJ}^\gamma$ for the element such that $\Omega^+\in \mathbf{D}_{(\al,j),\Lambda}$. If $F_\xi^{\Omega^\pm}|_\cC\in \cO_\cC^{<\gamma}$, then we have nothing to prove. Otherwise, it follows from Proposition~\ref{prop: special type II} that there exists a pseudo $\Lambda$-decomposition $\Omega'$ of $(\al,j)$ with $\widehat{\Omega}'\neq \{(\al,j)\}$ such that the balanced pair $\Omega_0^\pm$ defined by $\Omega_0^+\defeq \Omega^+$ and $\Omega_0^-\defeq\Omega'$ satisfies one of the following:
\begin{itemize}
\item $F_\xi^{\Omega_0^\pm}|_\cC\in\cO_\cC^{<\gamma}$;
\item $\Omega_0^\pm$ is a constructible $\Lambda$-lift of type $\rm{II}$;
\item there exists $\Omega''\in\mathbf{D}_{(\al,j),\Lambda}$ such that the balanced pair $\Omega'',\Omega^+$ satisfies the conditions in Lemma~\ref{lem: special case 3} and $F_\xi^{\Omega_1^\pm}|_\cC\in\cO_\cC^{<\gamma}$ for the balanced pair defined by $\Omega_1^+\defeq \Omega''$ and $\Omega_1^-\defeq \Omega_0^-$.
\end{itemize}
It is clear that $F_\xi^{\Omega_2^\pm}|_\cC\in\cO_\cC^{\rm{ps},\rm{III},\gamma}$ for the balanced pair defined by $\Omega_2^+\defeq \Omega'$ and $\Omega_2^-\defeq \Omega^-$, and there is nothing to prove if $F_\xi^{\Omega_0^\pm}|_\cC\in\cO_\cC^{<\gamma}$. If $\Omega_0^\pm$ is a constructible $\Lambda$-lift of type $\rm{II}$ and $\Omega_0^+=\Omega^+\in\mathbf{D}_{(\al,j),\Lambda}$ is $\Lambda$-extremal, then it follows from Proposition~\ref{prop: type II extremal} that $F_\xi^{\Omega_0^\pm}|_\cC\in\cO_\cC$ and thus $F_\xi^{\Omega^\pm}|_\cC\in \langle\cO_\cC^{\rm{ps},\rm{III},\gamma}\cdot\cO_\cC^{<\gamma}\cdot\cO_\cC\rangle_+$. If $\Omega_0^\pm$ is a constructible $\Lambda$-lift of type $\rm{II}$ and $\Omega_0^+=\Omega^+=\Omega_{(\al,j),\Lambda}^{\rm{max}}$ is $\Lambda$-exceptional, then it follows from Proposition~\ref{prop: type II exceptional} that $F_\xi^{\Omega_0^\pm}|_\cC\in\cO_\cC$ and thus $F_\xi^{\Omega^\pm}|_\cC\in \langle\cO_\cC^{\rm{ps},\rm{III},\gamma}\cdot\cO_\cC^{<\gamma}\cdot\cO_\cC\rangle_+$.
If such $\Omega''$ exists, then we also have $F_\xi^{\Omega_3^\pm}|_\cC\in \cO_\cC^{\rm{I},\rm{max},\gamma}$ for the balanced pair $\Omega_3^\pm$ defined by $\Omega_3^+\defeq \Omega''$ and $\Omega_3^-\defeq \Omega_0^+=\Omega^+=\Omega_{(\al,j),\Lambda}^{\rm{max}}$. Therefore we deduce that $F_\xi^{\Omega_0^\pm}|_\cC\in \cO_\cC^{\rm{I},\rm{max},\gamma}\cdot\cO_\cC^{<\gamma}$,
which clearly implies the desired result on $F_\xi^{\Omega^\pm}|_\cC$. The last part is immediate from Lemma~\ref{lem: ps III} and Lemma~\ref{lem: I max}. Hence, the proof is finished.
\end{proof}

\begin{lemma}\label{lem: I exp}
If $\Omega^\pm$ is a balanced pair of type \rm{I}-\rm{exp} with $|\Omega^\pm|=\gamma$ for some $\gamma\in\widehat{\Lambda}^\square$, then we have
$$F_\xi^{\Omega^\pm}|_\cC\in \langle\cO_\cC^{\Omega^\pm,<}\cdot\cO_\cC^{\rm{ps},\rm{III},\gamma}\cdot\cO_\cC^{\rm{I},\rm{max},\gamma} \cdot\cO_\cC^{\rm{II},\rm{max},\gamma}\cdot\cO_\cC^{<\gamma}\cdot\cO_\cC\rangle_+.$$
\end{lemma}

\begin{proof}
We write $(\al,j)\in\widehat{\Lambda}\cap\mathrm{Supp}_{\xi,\cJ}^\gamma$ for the element such that $\Omega^+,\Omega^-\in\mathbf{D}_{(\al,j),\Lambda}$. According to Proposition~\ref{prop: maximal vers exceptional: ord}, there are four possibilities as follow:
\begin{itemize}
\item $F_\xi^{\Omega^\pm}|_\cC\in\cO_\cC^{<\gamma}$;
\item $\Omega^\pm$ is a constructible $\Lambda$-lift of type $\rm{I}$;
\item there exists a pseudo $\Lambda$-decomposition $\Omega'$ of $(\al,j)$ such that the balanced pair $\Omega^+,\Omega'$ (resp. the balanced pair $\Omega^-,\Omega'$) satisfies the conditions of Lemma~\ref{lem: special case 1} (resp. of Proposition~\ref{prop: special type II});
\item there exists $\Omega'\in\mathbf{D}_{(\al,j),\Lambda}$ such that the balanced pair $\Omega',\Omega^-$ satisfies the conditions of Lemma~\ref{lem: special case 3} and $F_\xi^{\Omega_0^\pm}|_\cC\in\cO_\cC^{<\gamma}$ for the balanced pair $\Omega_0^\pm$ defined by $\Omega_0^+\defeq \Omega'$ and $\Omega_0^-\defeq \Omega^+$.
\end{itemize}
If $F_\xi^{\Omega^\pm}|_\cC\in \cO_\cC^{<\gamma}$, then we have nothing to prove. If $\Omega^\pm$ falls into the fourth case of Proposition~\ref{prop: maximal vers exceptional: ord}, then we deduce that $F_\xi^{\Omega^\pm}|_\cC\in \cO_\cC^{\rm{I},\rm{max},\gamma}\cdot\cO_\cC^{<\gamma}$.

Now we treat the case when $\Omega^\pm$ falls into the second case of Proposition~\ref{prop: maximal vers exceptional: ord}. It follows from Proposition~\ref{prop: type I exceptional} that $F_\xi^{\Omega^\pm}|_\cC$ belongs to the subring of $\cO(\cC)$ generated by $\cO_\cC$ and $F_\xi^{\Omega_1^\pm}|_\cC$ for all balanced pairs $\Omega_1^\pm$ satisfying $\Omega_1^-=\Omega^-$, $\Omega_1^+\in\mathbf{D}_{(\al,j),\Lambda}$ and $\Omega^+<\Omega_1^+$. We choose an arbitrary such $\Omega_1^\pm$. According to Lemma~\ref{lem: reduce to extremal}, it is harmless to consider only those $\Omega_1^\pm$ with $\Omega_1^+$ being either $\Lambda$-exceptional or $\Lambda$-extremal. If $\Omega_1^+$ is not $\Lambda$-ordinary, then the balanced pair $\Omega_1^-,(\Omega_1^+)_\dagger$ is of type \rm{II}-\rm{max}, which implies that $F_\xi^{\Omega_1^\pm}|_\cC\in \cO_\cC^{\rm{II},\rm{max},\gamma}\cdot\cO_\cC^{<\gamma}$. If $\Omega_1^+$ is $\Lambda$-extremal and $\Lambda$-ordinary, then it follows from the proof of Theorem~\ref{thm: pair of decompositions} that either $\Omega_1^\pm$ is a constructible $\Lambda$-lift of type $\rm{I}$ or $F_\xi^{\Omega_1^\pm}|_\cC\in\cO_\cC^{<\gamma}$, which together with Proposition~\ref{prop: type I extremal} implies that $F_\xi^{\Omega_1^\pm}|_\cC\in\langle\cO_\cC^{<\gamma}\cdot\cO_\cC\rangle_+$. If $\Omega_1^+$ is $\Lambda$-exceptional and $\Lambda$-ordinary, then we clearly have $F_\xi^{\Omega_1^\pm}|_\cC\in\cO_\cC^{\Omega^\pm,<}$. In all, we always have
$$F_\xi^{\Omega_1^\pm}|_\cC\in\langle\cO_\cC^{\Omega^\pm,<}\cdot\cO_\cC^{\rm{II},\rm{max},\gamma}\cdot\cO_\cC^{<\gamma}\cdot\cO_\cC\rangle_+$$
when $\Omega^\pm$ falls into the second case of Proposition~\ref{prop: maximal vers exceptional: ord}.

Finally, we treat the case when $\Omega^\pm$ falls into the third case of Proposition~\ref{prop: maximal vers exceptional: ord}. It is clear that we have $F_\xi^{\Omega_2^\pm}|_\cC\in \cO_\cC^{\rm{II},\rm{max},\gamma}$ for the balanced pair $\Omega_2^\pm$ defined by $\Omega_2^+\defeq \Omega^-$ and $\Omega_2^-\defeq \Omega'$. Applying Lemma~\ref{lem: special case 1} to the balanced pair $\Omega^+,\Omega'$, there exists a pseudo $\Lambda$-decomposition $\Omega''$ of $(\al,j)$ such that
\begin{itemize}
\item either $\Omega''=\Omega^+$ or the balanced pair $\Omega^+,\Omega''$ is a constructible $\Lambda$-lift of type $\rm{II}$;
\item $F_\xi^{\Omega_3^\pm}|_\cC\in \cO_\cC^{<\gamma}$ for the balanced pair defined by $\Omega_3^+\defeq \Omega''$ and $\Omega_3^-\defeq \Omega'$.
\end{itemize}
If $\Omega''=\Omega^+$, then we clearly have $F_\xi^{\Omega^\pm}|_\cC\in \cO_\cC^{\rm{II},\rm{max},\gamma}\cdot\cO_\cC^{<\gamma}$. Therefore, we assume from now on that the balanced pair $\Omega_4^\pm$ defined by $\Omega_4^+\defeq\Omega^+$ and $\Omega_4^-\defeq\Omega''$ is a constructible $\Lambda$-lift of type $\rm{II}$. It follows from Proposition~\ref{prop: type II exceptional} that $F_\xi^{\Omega_4^\pm}|_\cC$ belongs to the subring of $\cO(\cC)$ generated by $\cO_\cC$ and $F_\xi^{\Omega_5^\pm}|_\cC$ for all balanced pairs $\Omega_5^\pm$ satisfying $\Omega_5^-=\Omega_4^-$, $\Omega_5^+\in\mathbf{D}_{(\al,j),\Lambda}$ and $\Omega_4^+<\Omega_5^+$. We choose an arbitrary such $\Omega_5^\pm$. According to Lemma~\ref{lem: reduce to extremal}, it is harmless to consider only those $\Omega_5^\pm$ with $\Omega_5^+$ being either $\Lambda$-exceptional or $\Lambda$-extremal. If $\Omega_5^+$ is not $\Lambda$-ordinary, then we deduce from Lemma~\ref{lem: reduce to ordinary} that $F_\xi^{\Omega_5^\pm}|_\cC\in \langle\cO_\cC^{\rm{ps},\rm{III},\gamma}\cdot\cO_\cC^{<\gamma}\cdot \cO_\cC\rangle_+$. If $\Omega_5^+$ is $\Lambda$-extremal and $\Lambda$-ordinary, then we deduce from Proposition~\ref{prop: special type II} together with Proposition~\ref{prop: type II extremal} that $F_\xi^{\Omega_5^\pm}|_\cC\in \langle\cO_\cC^{\rm{ps},\rm{III},\gamma}\cdot\cO_\cC^{<\gamma}\cdot \cO_\cC\rangle_+$. Finally, if $\Omega_5^+$ is $\Lambda$-exceptional and $\Lambda$-ordinary then we clearly have
$$(F_\xi^{\Omega_2^\pm}|_\cC)^{-1}\cdot F_\xi^{\Omega_3^\pm}|_\cC\cdot F_\xi^{\Omega_5^\pm}|_\cC \in\cO_\cC^{\Omega^\pm,<},$$
which implies that
$$F_\xi^{\Omega_5^\pm}|_\cC\in\cO_\cC^{\Omega^\pm,<}\cdot \cO_\cC^{\rm{II},\rm{max},\gamma}\cdot \cO_\cC^{<\gamma}.$$
In all, we always have
$$F_\xi^{\Omega_1^\pm}|_\cC\in\langle\cO_\cC^{\Omega^\pm,<}\cdot\cO_\cC^{\rm{II},\rm{max},\gamma} \cdot\cO_\cC^{\rm{ps},\rm{III},\gamma}\cdot\cO_\cC^{<\gamma}\cdot\cO_\cC\rangle_+$$
when $\Omega^\pm$ falls into the third case of Proposition~\ref{prop: maximal vers exceptional: ord}. Hence, the proof is finished.
\end{proof}

\begin{lemma}\label{lem: II exp}
If $\Omega^\pm$ is a balanced pair of type \rm{II}-\rm{exp} with $|\Omega^\pm|=\gamma$ for some $\gamma\in\widehat{\Lambda}^\square$, then we have
$$F_\xi^{\Omega^\pm}|_\cC\in \langle\cO_\cC^{\Omega^\pm,<}\cdot\cO_\cC^{\rm{ps},\rm{III},\gamma}\cdot\cO_\cC^{<\gamma}\cdot\cO_\cC\rangle_+.$$
\end{lemma}

\begin{proof}
We write $(\al,j)\in\widehat{\Lambda}\cap\mathrm{Supp}_{\xi,\cJ}^\gamma$ for the element such that $\Omega^+\in\mathbf{D}_{(\al,j),\Lambda}$. If $F_\xi^{\Omega^\pm}|_\cC\in \cO_\cC^{<\gamma}$, then we have nothing to prove. Otherwise, from Proposition~\ref{prop: maximal vers exceptional: non ord} there exists a pseudo $\Lambda$-decomposition $\Omega'$ of $(\al,j)$ such that
      \begin{itemize}
      \item the balanced pair $\Omega^+,\Omega'$ is a constructible $\Lambda$-lift of type $\rm{II}$;
      \item $F_\xi^{\Omega_0^\pm}|_\cC\in\cO_\cC^{<\gamma}$ for the balanced pair $\Omega_0^\pm$ defined by $\Omega_0^+\defeq \Omega'$ and $\Omega_0^-\defeq \Omega^-$.
      \end{itemize}
If we let $\Omega_1^+\defeq \Omega^+$ and $\Omega_1^-\defeq \Omega'$, then it follows from Proposition~\ref{prop: type II exceptional} that $F_\xi^{\Omega_1^\pm}|_\cC$ belongs to the subring of $\cO(\cC)$ generated by $\cO_\cC$ and $F_\xi^{\Omega_2^\pm}|_\cC$ for all balanced pairs $\Omega_2^\pm$ satisfying $\Omega_2^-=\Omega_1^-$, $\Omega_2^+\in\mathbf{D}_{(\al,j),\Lambda}$ and $\Omega_1^+<\Omega_2^+$. We choose an arbitrary such $\Omega_2^\pm$. According to Lemma~\ref{lem: reduce to extremal}, it is harmless to consider only those $\Omega_2^\pm$ with $\Omega_2^+$ being either $\Lambda$-exceptional or $\Lambda$-extremal. If $\Omega_2^+$ is not $\Lambda$-ordinary then we deduce from Lemma~\ref{lem: reduce to ordinary} that $F_\xi^{\Omega_2^\pm}|_\cC\in \langle\cO_\cC^{\rm{ps},\rm{III},\gamma}\cdot\cO_\cC^{<\gamma}\cdot \cO_\cC\rangle_+$. If $\Omega_2^+$ is $\Lambda$-extremal and $\Lambda$-ordinary then we deduce from Proposition~\ref{prop: special type II} together with Proposition~\ref{prop: type II extremal} that $F_\xi^{\Omega_2^\pm}|_\cC\in \langle\cO_\cC^{\rm{ps},\rm{III},\gamma}\cdot\cO_\cC^{<\gamma}\cdot \cO_\cC\rangle_+$. Finally, if $\Omega_2^+$ is $\Lambda$-exceptional and $\Lambda$-ordinary then we clearly have
$$F_\xi^{\Omega_0^\pm}|_\cC\cdot F_\xi^{\Omega_2^\pm}|_\cC\in\cO_\cC^{\Omega^\pm,<},$$
which implies that
$$F_\xi^{\Omega_2^\pm}|_\cC\in\cO_\cC^{\Omega^\pm,<}\cdot \cO_\cC^{<\gamma}.$$
In all, we always have
$$F_\xi^{\Omega_1^\pm}|_\cC\in\langle\cO_\cC^{\Omega^\pm,<} \cdot\cO_\cC^{\rm{ps},\rm{III},\gamma}\cdot\cO_\cC^{<\gamma}\cdot\cO_\cC\rangle_+.$$
Hence, the proof is finished.
\end{proof}

\begin{prop}\label{prop: ps and inv}
We have $$\cO_\cC^{\rm{ps}}\subseteq\cO_\cC.$$
\end{prop}
\begin{proof}
Let $\Omega^\pm$ be a balanced pair with both $\Omega^+$ and $\Omega^-$ pseudo $\Lambda$-decompositions of $(\al,j)\in\widehat{\Lambda}\cap\mathrm{Supp}_{\xi,\cJ}^\gamma$. We argue by induction on $\gamma$ and thus can assume that $\cO_\cC^{<\gamma}\subseteq \cO_\cC$. It follows immediately from Lemma~\ref{lem: ps III}, Lemma~\ref{lem: I max} and Lemma~\ref{lem: II max} that we have
$$\cO_\cC^{\rm{ps},\rm{III},\gamma},\,\,\cO_\cC^{\rm{I},\rm{max},\gamma},\,\,\cO_\cC^{\rm{II},\rm{max},\gamma}\subseteq\cO_\cC.$$
Consequently, if $\Omega^\pm$ is of type \rm{I}-\rm{exp} (resp.~of type \rm{II}-\rm{exp}) for some $(\al,j)\in\widehat{\Lambda}\cap\mathrm{Supp}_{\xi,\cJ}^\gamma$, it follows from Lemma~\ref{lem: I exp} (resp.~Lemma~\ref{lem: II exp}) and an induction on the partial order $<$ on the set $\mathbf{D}_{(\al,j),\Lambda}$ that $F_\xi^{\Omega^\pm}|_\cC\in\cO_\cC$. In particular, we have
$$\cO_\cC^{\rm{I},\rm{exp},\gamma},\,\,\cO_\cC^{\rm{II},\rm{exp},\gamma}\subseteq\cO_\cC.$$

Now we return to a general $\Omega^\pm$ with both $\Omega^+$ and $\Omega^-$ pseudo $\Lambda$-decompositions of $(\al,j)\in\widehat{\Lambda}\cap\mathrm{Supp}_{\xi,\cJ}^\gamma$. A crucial observation from the proof of Theorem~\ref{thm: pair of decompositions} (upon restriction to $\cC$) is that
$$F_\xi^{\Omega^\pm}|_\cC\in\cO_\cC^{\rm{ps},\rm{III},\gamma}\cdot\cO_\cC^{\rm{I},\rm{max},\gamma}\cdot\cO_\cC^{\rm{II},\rm{max},\gamma}\cdot\cO_\cC^{\rm{I},\rm{exp},\gamma}\cdot\cO_\cC^{\rm{II},\rm{exp},\gamma}\cdot\cO_\cC^{\rm{I},\rm{ext}}\cdot\cO_\cC^{\rm{II},\rm{ext}}\cdot\cO_\cC^{<\gamma},$$
which together with Lemma~\ref{lem: I II ext} and previous discussion clearly implies that $F_\xi^{\Omega^\pm}|_\cC\in\cO_\cC$. The proof is thus finished.
\end{proof}

Now we are ready to prove Statement~\ref{state: goal}.
\begin{thm}\label{thm: constructible and inv fun}
For each $\Lambda$-lift $\Omega^\pm$, we have $$F_\xi^{\Omega^\pm}|_\cC\in\cO_\cC.$$
\end{thm}
\begin{proof}
If there exists $(\al,j)\in\widehat{\Lambda}$ such that both $\Omega^+$ and $\Omega^-$ are pseudo $\Lambda$-decompositions of $(\al,j)$, then we clearly have $F_\xi^{\Omega^\pm}|_\cC\in\cO_\cC$ thanks to Proposition~\ref{prop: ps and inv}. According to Theorem~\ref{thm: reduce to constructible}, it suffices to treat the case when $\Omega^\pm$ is a constructible $\Lambda$-lift of type $\rm{III}$. Then it follows from Proposition~\ref{prop: type III} that
$$F_\xi^{\Omega^\pm}|_\cC\in\langle\cO_\cC^{\rm{ps}}\cdot\cO_\cC^{<|\Omega^\pm|}\cdot\cO_\cC\rangle_+,$$
which together with Proposition~\ref{prop: ps and inv} and an induction on $|\Omega^\pm|$ finishes the proof.
\end{proof}

\begin{cor}\label{cor: separate points}
Statement~\ref{state: separate points prime} is true for each $\cC\in\cP_\cJ$.
\end{cor}
\begin{proof}
This follows directly from Lemma~\ref{lem: reduce to sec state} and Theorem~\ref{thm: constructible and inv fun}.
\end{proof}

\subsection{Summary and Examples}\label{sub: examples}
In this section, we illustrate some key ideas in the proofs of results from \S\,\ref{sec:comb:lifts}, \S\,\ref{sec:const:inv} and \S\,\ref{sec:inv:cons} through examples. We fix a choice of $w_\cJ\in\un{W}$, $\xi\in \Xi_{w_\cJ}$, $\Lambda\subseteq \mathrm{Supp}_{\xi,\cJ}$ and $\cC\in\cP$ satisfying $\cC\subseteq\cN_{\xi,\Lambda}$.

The set of constructible $\Lambda$-lifts (see Definition~\ref{def: constructible lifts}) satisfies two crucial properties: it generates all $\Lambda$-lifts or equivalently all balanced pairs (see \S\,\ref{sub: cons lifts} and \S\,\ref{sub:main:criterions}), and at the same time it is generated by the set of invariant functions (in the sense of \S\,\ref{sub:exp:for}). The conditions in Definition~\ref{def: constructible lifts} are carefully chosen to balance these two properties and are exactly the ones used in the proofs in \S\,\ref{sec:inv:cons}.
Although the logic of \S\,\ref{sec:comb:lifts}, \S\,\ref{sec:const:inv} and \S\,\ref{sec:inv:cons} is to define the set of constructible $\Lambda$-lifts, construct an invariant function for each constructible $\Lambda$-lift and then compute their restrictions to $\cC$, we suggest the readers read these sections simultaneously in the order type \rm{I}, type \rm{II} and then finally type \rm{III}.

For each constructible $\Lambda$-lift $\Omega^\pm$ and each $(k,j)\in\mathbf{n}_\cJ$, we attach the data $v_{\cJ}^{\Omega^\pm}$, $I_{\cJ}^{\Omega^\pm}$, $\al^{\Omega^\pm}_{k,j}$ and $\mathbf{D}^{\Omega^\pm}_{k,j}$ as in \S\,\ref{sec:const:inv} and \S\,\ref{sub: notation for each k j}. The element $v_{\cJ}^{\Omega^\pm}$ determines a growing sequence of minors $f_{S_{\bullet}^{j,\Omega^\pm},j}$ for each embedding $j\in\cJ$, which altogether generate the group of invertible sections of the open Bruhat cell $\cM^\circ_{v_{\cJ}^{\Omega^\pm}}=\un{U}\backslash\un{U}\un{T}w_0\un{U}w_0v_{\cJ}^{\Omega^\pm}$. For each $(k,j)\in\mathbf{n}_\cJ$, the restriction $f_{S_k^{j,\Omega^\pm},j}|_{\cN_{\xi,\Lambda}}$ can be read off from the set $\mathbf{D}^{\Omega^\pm}_{k,j}$ using Lemma~\ref{lem: from sets to formula}. Note that $f_{S_k^{j,\Omega^\pm},j}|_{\cN_{\xi,\Lambda}}$ is a monomial with variables $\{D_{\xi,k}^{(j)}\mid (k,j)\in\mathbf{n}_\cJ\}$ and $\{u_\xi^{(\al,j)}\mid (\al,j)\in\Lambda\}$ (or equivalently invertible on $\cN_{\xi,\Lambda}$) if and only if $\#\mathbf{D}^{\Omega^\pm}_{k,j}=1$. Assuming the first and third families of conditions in Definition~\ref{def: constructible lifts} (in the sense of Remark~\ref{rmk: motivation of def}), we actually prove in \S\,\ref{sub: exp type I}, \S\,\ref{sub: exp type II} and \S\,\ref{sub: exp type III} that
\begin{itemize}
\item $\mathbf{D}^{\Omega^\pm}_{k,j}\neq \emptyset$ for each $(k,j)\in\mathbf{n}_\cJ$;
\item if there exists $(k,j)\in\mathbf{n}_\cJ$ such that $\#\mathbf{D}^{\Omega^\pm}_{k,j}\geq 2$, then there exists $n\geq k_\star>k_\star'\geq 1$ and $j_0\in\cJ$ such that $\#\mathbf{D}^{\Omega^\pm}_{k,j}\geq 2$ if and only if $j=j_0$ and $k_\star\geq k>k_\star'$.
\end{itemize}
Note that $\#\mathbf{D}^{\Omega^\pm}_{k,j}=1$ for each $(k,j)\in\mathbf{n}_\cJ$ if and only if $\cN_{\xi,\Lambda}\subseteq \cM^\circ_{v_{\cJ}^{\Omega^\pm}}$.
This control of $\mathbf{D}^{\Omega^\pm}_{k,j}$ (or the relative position between $\cN_{\xi,\Lambda}$ and $\cM^\circ_{v_{\cJ}^{\Omega^\pm}}$) relies on the construction (see \S\,\ref{sub:comb:gen}) of the data $d_{\psi}$, $c_{\psi}^s$ and $i_\psi^{s,e}$ for each $1\leq s\leq d_\psi$, $1\leq e\leq e_{\psi,s}$ and each pair $\psi=(\Omega,\Lambda)$ with $\Omega$ a $\Lambda$-decomposition (of some $(\al,j)\in\widehat{\Lambda}$ determined by $\Omega$). This key construction is motivated by the following example.
\begin{exam}\label{exam:maximally nonsplit niveau one}
Assume that $n\geq 3$, $\cJ=\{j_0\}$, $w_{j_0}=1$ and $u_{j_0}=w_0$, which implies that $\un{W}\cong W$, $M_\xi=T$, $\cN_\xi=\cM_{1}^\circ$ and $\mathrm{Supp}_{\xi,\cJ}=\Phi^+\times\{j_0\}$. We choose a subset $\Lambda\subseteq \mathrm{Supp}_{\xi,\cJ}$ that contains $((i,i+1),j_0)$ for all $1\leq i\leq n-1$ as well as $((1,n),j_0)$. This implies that $\rhobar_{x,\lambda+\eta}$ is \emph{maximally nonsplit} (namely does not have reducible semisimple subquotient) for each $x\in\cN_{\xi,\Lambda}$. Note that for each $\Lambda$-decomposition $\Omega^+$ of $((1,n),j_0)$ with $\Omega^+\neq\Omega^-\defeq \Omega_{\{((1,n),j_0)\}}^{\rm{max}}=\{((1,n),j_0)\}$, $\Omega^\pm$ is a $\Lambda$-lift. We wish to find a suitable $\Omega^+$ such that $\Omega^\pm$ is $\Lambda$-constructible and then construct an invariant function $f_\xi^{\Omega^\pm}$. Recall that $\mathbf{D}_{((1,n),j_0),\Lambda}$ is the set of all $\Lambda$-decompositions of $((1,n),j_0)$, equipped with a partial order (see Definition~\ref{def: partial order}). We choose $\Omega^+$ to be the unique maximal element of the set $\mathbf{D}_{((1,n),j_0),\Lambda}\setminus \{((1,n),j_0)\}$. Let $\psi_1\defeq (\Omega^+,\Lambda)$ and $\psi_2\defeq (\Omega^-,\Lambda)$. We recall from \S\,\ref{sub:comb:gen} the data $d_{\psi}$, $c_{\psi}^s$ and $i_\psi^{s,e}$ for each $\psi\in\{\psi_1,\psi_2\}$, $1\leq s\leq d_\psi$ and $1\leq e\leq e_{\psi,s}$. It is easy to check that both $\Omega^+$ and $\Omega^-$ are $\Lambda$-ordinary (using $M_\xi=T$), and $\Omega^+$ is either $\Lambda$-exceptional or $\Lambda$-extremal. We deduce easily from $i_{\psi_2}^{1,1}=i_{\Omega^+,1}$ that $\Omega^\pm$ is $\Lambda$-constructible of type~\rm{I}. Hence, we can attach to $\Omega^\pm$ an invariant function $f_\xi^{\Omega^\pm}=f_{v_{\cJ}^{\Omega^\pm},I_{\cJ}^{\Omega^\pm}}$ as in \S\,\ref{sub: type I}. For each $a=1,2$, we recall the shortened notation $k_{a,c}$ for each $0\leq c\leq c_a$ and $k_a^{s,e}$ for each $1\leq s\leq d_a$ and $1\leq e\leq e_{a,s}$ from the beginning of \S\,\ref{sec:const:inv} (with $c_2=1$, $k_{1,0}=k_{2,0}=n$, $k_{1,c_1}=k_{2,c_2}=1$ and $k_2^{1,1}=k_{1,c_1-1}$). Then we have $v_{j_0}^{\Omega^\pm}=(k_{1,c_1-1},\cdots,k_{1,2},k_{1,1})(k_{1,0},k_1^{1,1},\cdots,k_1^{1,e_{1,1}},\cdots,k_1^{d_1,e_{1,d_1}},k_{1,c_1})$ and $I_{j_0}^{\Omega^\pm}=\{k_{1,c}\mid 1\leq c\leq c_1-1\}$ if $\Omega^+$ is $\Lambda$-exceptional, and $v_{j_0}^{\Omega^\pm}=(k_{1,c_1-1},\cdots,k_{1,2},k_{1,1},k_{1,0})(k_1^{1,1},\cdots,k_1^{1,e_{1,1}},\cdots,k_1^{d_1,e_{1,d_1}},k_{1,c_1})$ and $I_{j_0}^{\Omega^\pm}=\{k_{1,c}\mid 0\leq c\leq c_1-1\}$ if $\Omega^+$ is $\Lambda$-extremal.
\end{exam}

Let $\Omega^\pm$ be a constructible $\Lambda$-lift. Now that $v_{\cJ}^{\Omega^\pm}$ is chosen and $\mathbf{D}^{\Omega^\pm}_{k,j}$ is well understood, we need to find a subset $I_{\cJ}^{\Omega^\pm}\subseteq\mathbf{n}_\cJ$ such that
\begin{itemize}
\item $I_{\cJ}^{\Omega^\pm}$ is a union of $(v_{\cJ}^{\Omega^\pm},1)$-orbits;
\item the study of $f_\xi^{\Omega^\pm}=f_{v_{\cJ}^{\Omega^\pm},I_{\cJ}^{\Omega^\pm}}$ relates $F_\xi^{\Omega^\pm}|_\cC$ with $\cO_\cC$ (see \S\,\ref{sub:exp:for} for precise statements).
\end{itemize}
When $\cN_{\xi,\Lambda}\subseteq \cM^\circ_{v_{\cJ}^{\Omega^\pm}}$ and thus $f_\xi^{\Omega^\pm}$ is invertible along $\cN_{\xi,\Lambda}$, we can even find $I_{\cJ}^{\Omega^\pm}$ such that $f_\xi^{\Omega^\pm}|_{\cN_{\xi,\Lambda}}\sim F_\xi^{\Omega^\pm}$. In order to compute $f_\xi^{\Omega^\pm}|_\cC$ (if exists) using Lemma~\ref{lem: from sets to formula}, we only need to control the subset $I_{\cJ}^{\Omega^\pm,\star}\subseteq I_{\cJ}^{\Omega^\pm}$ of $(k,j)$ when $\mathbf{D}^{\Omega^\pm}_{k,j}\neq \mathbf{D}^{\Omega^\pm}_{k+1,j}$. Recall from the beginning of \S\,\ref{sec:const:inv} that we have a decomposition of index $\Z/t=(\Z/t)^+\sqcup(\Z/t)^-$, which induces a decomposition $\al^{\Omega^\pm}_{k,j}=\al^{\Omega^\pm}_{+,k,j}+\al^{\Omega^\pm}_{-,k,j}$ with $\al^{\Omega^\pm}_{\bullet,k,j}\defeq \sum_{a\in(\Z/t)^\bullet}\al^{\Omega^\pm}_{a,k,j}$ for each $\bullet\in\{+,-\}$ (see \S\,\ref{sub: exp type I}, \S\,\ref{sub: exp type II} and \S\,\ref{sub: exp type III} for the definition of $\al^{\Omega^\pm}_{a,k,j}$ for type \rm{I}, \rm{II} and \rm{III} respectively). Our choice of $v_{\cJ}^{\Omega^\pm}$ ensures that either $\al^{\Omega^\pm}_{+,k,j}=\al^{\Omega^\pm}_{+,k+1,j}$ or $\al^{\Omega^\pm}_{-,k,j}=\al^{\Omega^\pm}_{-,k+1,j}$ or both, for each $(k,j)\in\mathbf{n}_\cJ$. We can thus define a subset $\mathbf{n}_\cJ^{\Omega^\pm,+}$ (resp.~$\mathbf{n}_\cJ^{\Omega^\pm,-}$) of $\mathbf{n}_\cJ$ by considering all pairs $(k,j)$ satisfying $\al^{\Omega^\pm}_{+,k,j}>\al^{\Omega^\pm}_{+,k+1,j}$ or $\al^{\Omega^\pm}_{-,k,j}<\al^{\Omega^\pm}_{-,k+1,j}$ (resp.~$\al^{\Omega^\pm}_{+,k,j}<\al^{\Omega^\pm}_{+,k+1,j}$ or $\al^{\Omega^\pm}_{-,k,j}>\al^{\Omega^\pm}_{-,k+1,j}$). We wish to choose $I_{\cJ}^{\Omega^\pm}$ as the union of $(v_{\cJ}^{\Omega^\pm},1)$-orbits inside $\mathbf{n}_\cJ$ generated by $\mathbf{n}_\cJ^{\Omega^\pm,+}$ such that $I_{\cJ}^{\Omega^\pm,\star}=\mathbf{n}_\cJ^{\Omega^\pm,+}$, but this relies on a crucial fact that the $(v_{\cJ}^{\Omega^\pm},1)$-orbits generated by $\mathbf{n}_\cJ^{\Omega^\pm,+}$ are disjoint from those generated by $\mathbf{n}_\cJ^{\Omega^\pm,-}$. This is why we need the second family of conditions in Definition~\ref{def: constructible lifts} (in the sense of Remark~\ref{rmk: motivation of def}). For example, if $\Omega^\pm$ is constructible of type \rm{III} with $\Omega^+\sqcup\Omega^-$ being circular, then we have
\begin{itemize}
\item $\mathbf{n}_\cJ^{\Omega^\pm,+}=\mathbf{n}_{\Omega^+\sqcup\Omega^-,1}\times\{j_{\Omega^+\sqcup\Omega^-}\}$ and $\mathbf{n}_\cJ^{\Omega^\pm,-}=\mathbf{n}_{\Omega^+\sqcup\Omega^-,-1}\times\{j_{\Omega^+\sqcup\Omega^-}\}$;
\item $I_{\cJ}^{\Omega^\pm}=\bigcup_{k\in\mathbf{n}_{\Omega^+\sqcup\Omega^-,1}}](k,j_{\Omega^+\sqcup\Omega^-}),(k,j_{\Omega^+\sqcup\Omega^-})]_{w_\cJ}$ is the $(v_{\cJ}^{\Omega^\pm},1)$-orbit generated by $\mathbf{n}_\cJ^{\Omega^\pm,+}$ which is disjoint from
$\bigcup_{k\in\mathbf{n}_{\Omega^+\sqcup\Omega^-,-1}}](k,j_{\Omega^+\sqcup\Omega^-}),(k,j_{\Omega^+\sqcup\Omega^-})]_{w_\cJ}$.
\end{itemize}

As part of the second family of conditions in Definition~\ref{def: constructible lifts} (see Condition~\rm{I}-\ref{it: I 3}, \rm{II}-\ref{it: II 3} and \rm{III}-\ref{it: III 3}), it is crucial for us to consider $\Lambda$-ordinary $\Lambda$-decompositions (see Definition~\ref{def: ordinary decomposition}) and ordinarization (see Lemma~\ref{lem: reduce to ordinary}).
\begin{exam}\label{exam:non ord}
Assume that $n\geq 4$, $\cJ=\{j_0\}$, $u_{j_0}=w_0$ and $w_{j_0}$ restricts to a $(n-2)$-cycle of the set $\{2,3,\cdots,n-1\}$. This implies that $\un{W}\cong W$, $r_\xi=3$, $M_\xi=\GL_1\times\GL_{n-2}\times\GL_1$ and $\mathrm{Supp}_{\xi,\cJ}\subseteq\Phi^+\times\{j_0\}$ is the subset consisting of $(\al,j_0)$ satisfying either $i_\al=1$ or $i_\al^\prime=n$. Let $\Lambda\subseteq\mathrm{Supp}_{\xi,\cJ}$ be a subset containing $((1,n),j_0)$ and thus $\widehat{\Lambda}=\Lambda$. Then there exists a unique $\cC\in\cP_\cJ$ such that $\cC\subseteq \cN_{\xi,\Lambda}$ is open. For each $(\al,j_0)\in\Lambda$, we note that $\mathbf{D}_{(\al,j_0),\Lambda}\neq\{\{(\al,j_0)\}\}$ forces $\al=(1,n)$. It is obvious that $\{((1,n),j_0)\}$ is a $\Lambda$-ordinary $\Lambda$-decomposition of $((1,n),j_0)$. Let $\Omega\in\mathbf{D}_{((1,n),j_0),\Lambda}\setminus\{\{((1,n),j_0)\}\}$ be a $\Lambda$-decomposition, which implies that $2\leq i_{\Omega,1}\leq n-1$ and $\Omega=\{((1,i_{\Omega,1}),j_0),((i_{\Omega,1},n),j_0)\}$. We write $\psi=(\Omega,\Lambda)$ and observe that $d_\psi=1$, $c_\psi^1=1$ and thus $\Omega$ is $\Lambda$-exceptional. Note that $r_\xi=3$, $[1]_\xi=\{1\}$, $[2]_\xi=\{2,\cdots,n-1\}$ and $[3]_\xi=\{n\}$. Hence, $\Omega$ is $\Lambda$-ordinary if and only if $e_{\psi,1}=0$ (namely $u_{j_0}(i_{\Omega,1})=\min\{u_{j_0}(i_{\Omega',1})\mid \Omega'\in\mathbf{D}_{((1,n),j_0),\Lambda}\}$). Moreover, if $\Omega$ is not $\Lambda$-ordinary, then $e_{\psi,1}\geq 1$ and $\Omega_\dagger=\{((1,i_\psi^{1,1}),j_0),((i_{\Omega,1},n),j_0)\}$ (see (\ref{eq: dagger})). In fact, we have the following two possibilities:
\begin{itemize}
\item if $\Omega$ is $\Lambda$-ordinary, then the pair $\Omega, \{((1,n),j_0)\}$ forms a constructible $\Lambda$-lift of type \rm{I} with $$f_\xi^{\Omega, \{((1,n),j_0)\}}|_{\cN_{\xi,\Lambda}}\sim \frac{u_\xi^{((1,n),j_0)}-u_\xi^{((1,i_{\Omega,1}),j_0)}u_\xi^{((i_{\Omega,1},n),j_0)}}{u_\xi^{((1,n),j_0)}};$$
\item if $\Omega$ is not $\Lambda$-ordinary, then $\{((1,n),j_0)\}, \Omega_\dagger$ is a constructible $\Lambda$-lift of type \rm{II} and $$f_\xi^{\{((1,n),j_0)\}, \Omega_\dagger}|_{\cN_{\xi,\Lambda}}\sim \frac{u_\xi^{((1,n),j_0)}}{u_\xi^{((1,i_\psi^{1,1}),j_0)}u_\xi^{((i_{\Omega,1},n),j_0)}}.$$
\end{itemize}
For each pair of integers $2\leq i\neq i'\leq n-1$, we also observe that
\begin{itemize}
\item if $((1,i),j_0),((1,i'),j_0)\in\Lambda$, then the pair $\{((1,i),j_0)\},\{((1,i'),j_0)\}$ forms a constructible $\Lambda$-lift of type \rm{III} with $f_\xi^{\{((1,i),j_0)\},\{((1,i'),j_0)\}}|_{\cN_{\xi,\Lambda}}\sim \frac{u_\xi^{((1,i),j_0)}}{u_\xi^{((1,i'),j_0)}}$;
\item if $((i,n),j_0),((i',n),j_0)\in\Lambda$, then the pair $\{((i,n),j_0)\},\{((i',n),j_0)\}$ forms a constructible $\Lambda$-lift of type \rm{III} with $f_\xi^{\{((i,n),j_0)\},\{((i',n),j_0)\}}|_{\cN_{\xi,\Lambda}}\sim \frac{u_\xi^{((i,n),j_0)}}{u_\xi^{((i',n),j_0)}}$.
\end{itemize}
We can easily prove Theorem~\ref{thm: constructible and inv fun} for $\cC$ by combining the invariant functions listed above. Now we specialize to the case when $n=4$ and $\Lambda=\mathrm{Supp}_{\xi,\cJ}$, which implies that the $\Lambda$-decomposition $\{((1,2),j_0),((2,4),j_0)\}$ is not $\Lambda$-ordinary. One can actually check that there does not exist $g\in \Inv(\cC)$ such that $g|_{\cN_{\xi,\Lambda}}$ is similar to either $\frac{u_\xi^{((1,2),j_0)}u_\xi^{((2,4),j_0)}}{u_\xi^{((1,4),j_0)}}$ or $\frac{u_\xi^{((1,4),j_0)}-u_\xi^{((1,2),j_0)}u_\xi^{((2,4),j_0)}}{u_\xi^{((1,4),j_0)}}$.
\end{exam}

The following example gives another lower bound for the amount of combinatorics necessary for the proof of Theorem~\ref{thm: constructible and inv fun}.
\begin{exam}\label{exam:one loop}
Assume that $n\geq 3$, $\cJ=\{j_0\}$, $w_{j_0}=1$, $u_{j_0}=w_0$, which implies that $\un{W}\cong W$, $M_\xi=T$, $\cN_\xi=\cM_{1}^\circ$ and $\mathrm{Supp}_{\xi,\cJ}=\Phi^+\times\{j_0\}$.
We choose a $n$-cycle $w\in W$ and define $\Lambda_{w}\defeq \Omega_{w}^+\sqcup\Omega_{w}^-\subseteq\mathrm{Supp}_{\xi,\cJ}$ by
$$
\left\{\begin{array}{c}
\Omega_{w}^+\defeq \{(k,w(k))\mid 1\leq k\leq n,\,k<w(k)\}\times\{j_0\};\\
\Omega_{w}^-\defeq \{(w(k),k)\mid 1\leq k\leq n,\,k>w(k)\}\times\{j_0\}.
\end{array}\right.
$$
It turns out that
\begin{itemize}
\item
$\cN_{\xi,\Lambda_{w}}\slash{{\sim}_{\un{T}\text{-\textnormal{cnj}}}} \cong \Spec \F[(D_{\xi,k}^{(j_0)})^{\pm1}\mid 1\leq k\leq n][(F_\xi^{\Omega_{w}^\pm})^{\pm1}]$;
\item there exists a unique element $\cC_{w}\in\cP$ which is an open subscheme of $\cN_{\xi,\Lambda_{w}}$.
\end{itemize}
One can check that $f_{1,\{k\}}$ is invertible on $\cN_{\xi,\Lambda_{w}}$ and $f_{1,\{k\}}|_{\cN_{\xi,\Lambda_{w}}}=D_{\xi,k}^{(j_0)}$ for each $1\leq k\leq n$. To prove Theorem \ref{thm: constructible and inv fun} for $\cC_{w}$ we look for a permutation $v_{j_0}^{\Omega_{w}^\pm}\in W$ and a subset $I_{j_0}^{\Omega_{w}^\pm}\subseteq\{1,\dots,n\}$ that satisfies the following
\begin{itemize}
\item $v_{j_0}^{\Omega_{w}^\pm}(I_{j_0}^{\Omega_{w}^\pm})=I_{j_0}^{\Omega_{w}^\pm}$ and $f_\xi^{\Omega_{w}^\pm}=f_{v_{j_0}^{\Omega_{w}^\pm},I_{j_0}^{\Omega_{w}^\pm}}\in \Inv(\cC_{w})$;
\item $F_\xi^{\Omega_{w}^\pm}|_{\cC_{w}}$ can be generated from $f_\xi^{\Omega_{w}^\pm}|_{\cC_{w}}$ and $\{(D_{\xi,k}^{(j_0)}|_{\cC_{w}})^{\pm1}\mid 1\leq k\leq n\}$.
\end{itemize}
In particular, we see that to prove Theorem \ref{thm: constructible and inv fun} for $\cC_{w}$ we need to construct a pair $(v_{j_0}^{\Omega_{w}^\pm},I_{j_0}^{\Omega_{w}^\pm})$ for each $n$-cycle $w\in W$. This delicate combinatorial construction is done in \S\,\ref{sub: type III}. The properties of $f_\xi^{\Omega_{w}^\pm}=f_{v_{j_0}^{\Omega_{w}^\pm},I_{j_0}^{\Omega_{w}^\pm}}$ mentioned above are checked in \S\,\ref{sub: exp type III}.
\end{exam}
\clearpage{}%
\clearpage{}%
\section{$\tld{\cF\cL}_\cJ$, $\tld{\Fl}_\cJ$ and Serre weights}
\label{sec:FL:SW}
In this section we compare the Fontaine--Laffaille moduli space with moduli of Breuil--Kisin modules with tame descent data, inside the Emerton--Gee stack, and interpret the partition $\cP_\cJ$ on $\tld{\cF\cL}_\cJ$ using \emph{shapes} and \emph{extremal weights}. We extensively use the theory of local models introduced in \cite[\S\,4 and \S\,5]{MLM}.

\vspace{3mm}

We recall the notion of alcoves, admissible set and certain subsets in the extended affine Weyl group for $\un{G}$. \emph{Only in this section, we omit the subscript $\cJ$ from the notation of elements of $\un{W}$ or $\tld{\un{W}}$ for simplicity.} An alcove is a connected component of
the complement $X^*(\un{T})\otimes_{\Z}\R\ \setminus\ \big(\bigcup_{(\alpha,n)}H_{\alpha,n}\big)$ where we write $H_{\alpha,n}\defeq \{\mu:\ \langle\mu,\alpha^\vee\rangle=n\}$ for the root hyperplane associated to $(\alpha,n)\in \un{\Phi}^+\times \Z$. We say that an alcove $\un{A}$ is \emph{restricted} (resp.~\emph{dominant}) if $0<\langle\mu,\alpha^\vee\rangle<1$ (resp.~$\langle\mu,\alpha^\vee\rangle>0$) for all simple roots $\alpha\in \un{\Delta}$ and $\mu\in \un{A}$.
If $\un{A}_0 \subset X^*(\un{T})\otimes_{\Z}\R$ is the alcove defined by the condition $0<\langle\mu,\alpha^\vee\rangle<1$ for all positive roots $\alpha\in \un{\Phi}^+$, we let
\[\tld{\un{W}}^+\defeq\{\tld{w}\in \tld{\un{W}}\mid\tld{w}(\un{A}_0) \textrm{ is dominant}\}\]
and
\begin{equation*}
\tld{\un{W}}^+_1\defeq\{\tld{w}\in \tld{\un{W}}^+\mid\tld{w}(\un{A}_0) \textrm{ is restricted}\}.
\end{equation*}
Note that $\mu\in X^*(\un{T})\cap (p\un{A}_0)$ if and only if $\mu$ is $0$-generic Fontaine--Laffaille (cf.~Definition~\ref{defn:mGenFL}).
We fix an injection $\un{W}\hookrightarrow\tld{\un{W}}$ whose composition with the surjection $\tld{\un{W}}\twoheadrightarrow\un{W}$ is the identity map. We also write $\tld{w}_h  =( \tld{w}_{h,i}) \in \tld{\un{W}}^+_1$ for the element $w_0 t_{-\eta}$.

The alcove $\un{A}_0$ defines a Bruhat order on $\un{W}_a$ denoted by $\leq$.  By letting $\un{\Omega}$ denote the stabilizer of $\un{A}_0$, we have $\tld{\un{W}} = \un{W}_a \rtimes \un{\Omega}$ and so $\tld{\un{W}}$ inherits a Bruhat order as well: for $\tld{w}_1, \tld{w}_2\in \un{W}_a$ and $\tld{w}\in \un{\Omega}$, $\tld{w}_1\tld{w}\leq \tld{w}_2\tld{w}$ if and only if $\tld{w}_1\leq \tld{w}_2$, and elements in different right $\un{W}_a$-cosets are incomparable.
We extend the Coxeter length function $\ell$ on $\un{W}_a$ to $\tld{\un{W}}$ by setting $\ell(\tld{w}\delta)\defeq \ell(\tld{w})$ if $\tld{w}\in\un{W}_a$, $\delta\in\un{\Omega}$.
If $\lambda \in X^*(\un{T})$ we define
\[
\Adm(\lambda) \defeq \left\{ \tld{w} \in \tld{\un{W}} \mid \tld{w} \leq t_{w(\lambda)} \text{ for some } w \in \un{W}\right\}.
\]
We define an involution $\tld{w}\mapsto \tld{w}^*$ of $\tld{\un{W}}$ by $((wt_\nu)^*)_j\defeq t_{\nu_{j}} w_{j}^{-1}$. This involution does not preserve the Bruhat order on $\tld{\un{W}}$ fixed above. Note that \cite[Definition 2.1.2]{LLL} writes $\tld{\un{W}}^{\vee}$ for the group $\tld{\un{W}}$ equipped with the Bruhat order defined by the \emph{antidominant} base alcove, which makes $\tld{w}\mapsto \tld{w}^*$ an order preserving, involutive anti-isomorphism between $\tld{\un{W}}^{\vee}$ and $\tld{\un{W}}$.

\subsection{Serre weights and Galois representations}
\label{sec:SWandGR}
In this section, we recall some background on Serre weights together with their relation to the Emerton--Gee stack. Then we use the notion \emph{specialization} to define the set of \emph{extremal weights} for Fontaine--Laffaille Galois representations (see equation (\ref{equ: obv wt})).

An absolutely irreducible $\F$-representation of $\GL_n(k)$ will be called a \emph{Serre weight}.
The set $X_1(\un{T})$ of $p$-restricted dominant weights is defined as
\[
X_1(\un{T}):=\{
\mu\in X(\un{T})\mid 0\leq \langle \mu,\alpha^\vee\rangle \leq p-1\text{ for all $\alpha^\vee\in \un{\Delta}^\vee$}
\}
\]
and by \cite[Lemma 9.2.4]{GHS}) we have a bijection
\begin{align}
\label{ex:bij:SW}
\nonumber
X_1(\un{T})/(p-\pi)X^0(\un{T})&\stackrel{\sim}{\longrightarrow}\big\{\mathrm{Serre\ Weights}\big\}_{/\sim}\\
\mu+ (p-\pi)X^0(\un{T})&\longmapsto
F(\mu)
\end{align}
Given an integer $m\geq 0$, we say that a Serre weight $F$ is \emph{$m$-generic Fontaine--Laffaille} if $F\cong F(\mu)$ with $\mu+\eta\in X^*(\un{T})$ being $m$-generic Fontaine--Laffaille (cf.~Definition~\ref{defn:mGenFL}). Note that this condition implies $\mu\in X_1(\un{T})$ and does not depend on the class of $\mu$ modulo $(p-\pi)X^0(\un{T})$.

\begin{defn}[\cite{MLM}, \S\,2.2]
\label{defn:LAP:SW}
A \emph{lowest alcove presentation} for  a Serre weight $V$ is an equivalence class of pairs $(\tld{w}_1,\omega)$, where $\tld{w}_1\in \tld{\un{W}}^+_1$ and $\omega\in X^*(\un{T})$ is a $0$-generic
Fontaine--Laffaille weight
(with equivalence relation $(\tld{w}_1,\omega)\sim (t_\nu\tld{w}_1,\omega-\nu)$ for $\nu\in X^0(\un{T})$) such that
\[
V\cong F_{(\tld{w}_1,\omega)}\defeq F(\pi^{-1}(\tld{w}_1)\cdot (\omega-\eta)).
\]
We say that a lowest alcove presentation $(\tld{w}_1,\omega)$ of a Serre weight $F_{(\tld{w}_1,\omega)}$ is \emph{compatible with an algebraic central character $\zeta\in X^*(\un{Z})$} if $t_{\omega-\eta}\tld{w}_1\un{W}_a$ corresponds to $\zeta$ via the isomorphism $\tld{\un{W}}/\un{W}_a\stackrel{\sim}{\ra} X^*(\un{Z})$.
\end{defn}

We let $\cX_n$ denote the Noetherian formal algebraic stack over $\Spf\cO$ defined in \cite[Definition 3.2.1]{EGstack}.
Its restriction to a complete local Noetherian $\cO$-algebra $R$ with finite residue field is equivalent to the groupoid of continuous $G_K$-representations over rank $n$ projective $R$-modules. In \cite[Theorem 6.5.1]{EGstack} the authors establish a bijection between Serre weights and the irreducible components of $\cX_{n,\mathrm{red}}$, the latter denoting the reduced structure underlying the special fibre of $\cX_n$.
For a Serre weight $V$ we define $\cC_V$ as the irreducible component of $\cX_{n,\mathrm{red}}$ corresponding via \cite[Theorem 6.5.1]{EGstack} to the Serre weight $V^\vee\otimes{\det}^{n-1}$.
This is compatible with \cite[\S\,7.4]{MLM}.

\vspace{2mm}

Let $\rhobar:G_K\ra\GL_n(\F)$ be a continuous Galois representation which we consider as an $\F$-point in $|\cX_n(\F)|$.
We define
\[
W^g(\rhobar)\defeq\big\{V \,\, |\,\, \rhobar\in |\cC_V(\F)|\big\}.
\]

If $\taubar$ is a tame inertial $\F$-type such that $[\taubar]$ has a lowest alcove presentation $(s, \mu)$ where $\mu+\eta$ is $n$-generic Fontaine--Laffaille, we have a set of Serre weights $W^?(\taubar)$ associated with it as in \cite[Definition 9.2.5]{GHS} (cf.~\cite[Definition 2.2.11]{LLL}).

\vspace{2mm}

Let $\lambda\in X^*_+(\un{T})$ be a dominant weight with $\lambda+\eta$ being Fontaine--Laffaille (cf.~Definition~\ref{defn:mGenFL}). Recall from \S\,\ref{subsub:FLT} the scheme $\tld{\cF\cL}_{\cJ}=\un{U}\backslash \un{G}$.
We have a formally smooth morphism $\tld{\cF\cL}_{\cJ}\ra \un{B}\backslash \un{G}$ which makes $\tld{\cF\cL}_{\cJ}$ a $\un{T}$-torsor over $\un{B}\backslash \un{G}$.
The $\un{T}$-action on $\un{G}$ induced by right multiplication descends to a $\un{T}$-action on $\tld{\cF\cL}_{\cJ}$ and on $\un{B}\backslash \un{G}$.
Let $x\in\tld{\cF\cL}_{\cJ}(\F)$ be an element such that $\rhobar_{x,\lambda+\eta}\cong \rhobar$ and write $\ovl{x}$ for its image in $\un{B}\backslash \un{G}(\F)$.
A \emph{specialization $\rhobar^{\speci}$ of $\rhobar$} is a tame inertial $\F$-type which corresponds to a $\un{T}$-fixed point in the Zariski closure of $\ovl{x}\cdot \un{T}$.
We write $\rhobar\leadsto \rhobar^{\speci}$ to mean that $\rhobar^{\speci}$ is a specialization of $\rhobar$.

We can characterize the representations $\rhobar$ which have a given specialization.
\begin{lemma}\label{lem: specialization and open cell}
Let $x\in\tld{\cF\cL}_\cJ(\F)$ be a point and $w\in\un{W}$ be an element.
Then $\rhobar_{x,\lambda+\eta}\leadsto \taubar(w^{-1},\lambda+\eta)$ if and only if $x\in \cM^\circ_{w}(\F)$.
\end{lemma}
\begin{proof}
According to Lemma~\ref{lem: class of irr}, it suffices to show that, given $\ovl{x}\in\un{B}\backslash \un{G}(\F)$, the Zariski closure of $\ovl{x}\cdot \un{T}(\F)$ contains $\un{B}\backslash \un{B}w$ if and only if $\ovl{x}\in\un{B}\backslash \un{B}w_0\un{B}w_0w$. As the complement of $\un{B}\backslash \un{B}w_0\un{B}w_0w$ is Zariski closed in $\un{B}\backslash \un{G}$ and does not contain $\un{B}\backslash \un{B}w$, we just need to show that the Zariski closure of $\un{B}\backslash \un{B}w_0 A w_0w \un{T}=\un{B}\backslash \un{B}w_0 A \un{T} w_0w$ contains $\un{B}\backslash \un{B}w$, for each $A=(A^{(j)})_{j\in\cJ}\in\un{U}(\F)$. We consider the morphism
\begin{equation}\label{equ: tends to zero}
\bG_m\rightarrow\un{B}\subseteq\un{G}:~x\mapsto \Diag(x^{n-1},\dots,x,1)\cdot A\cdot \Diag(x^{-n+1},\dots,x^{-1},1)
\end{equation}
which clearly extends to a morphism $\bA^1\rightarrow\un{G}$ that contains $1$ in the image. This implies that $\un{B}\backslash \un{B}w$ is in the Zariski closure of $\un{B}\backslash \un{B}w_0 A \un{T} w_0w$ and the proof is finished.
\end{proof}

\begin{rmk}\label{rmk: specialization and strata}
As each $\cC\in\cP_\cJ$ is stable under both left and right $\un{T}$-multiplication, the proof of Lemma~\ref{lem: specialization and open cell} also shows that $\cC\subseteq\cM_{w}^\circ$ if and only if the Zariski closure $\overline{\cC}$ of $\cC$ contains $\overline{\cM}_{w}$, which is the fiber of $\tld{\cF\cL}_\cJ\twoheadrightarrow\un{B}\backslash\un{G}$ over $w$. In particular, for each $x\in\cC(\F)$, we have
$$\{\rhobar^{\speci}\mid\rhobar_{x,\lambda+\eta}\leadsto \rhobar^{\speci}\}=\{\taubar(w^{-1},\lambda+\eta)\mid \overline{\cM}_{w}\subseteq\overline{\cC}\}.$$
\end{rmk}

If $\taubar$ is a tame inertial $\F$-type such that $[\taubar]$ has a lowest alcove presentation $(s, \mu)$ where $\mu+\eta$ is $n$-generic Fontaine--Laffaille, we have a subset $W_{\textrm{extr}}(\taubar)\subseteq W^?(\taubar)$ defined in \cite[Definition 7.1.3]{GHS} (see also \cite[Definition 2.6.3]{MLM}).
The choice of the lowest alcove presentation $(s, \mu)$ of $\taubar$ gives a bijection $\un{W}\ra W_\obv(\taubar)$ defined by $w\mapsto F_{(\tld{w}, \tld{w}(\taubar) \tld{w}^{-1}(0))}$ where $\tld{w}$ in the image of $w$ under our fixed injection $\un{W}\hookrightarrow\tld{\un{W}}^+_1$.
When the lowest alcove presentation of $\taubar$ is understood, the image of $w\in\un{W}$ via this bijection will be denoted by $V_{\taubar,w}$, and called the extremal weight of $\taubar$ corresponding to $w$.

Let $\lambda+\eta$ be $n$-generic Fontaine--Laffaille. Let $\rhobar:G_K\ra\GL_n(\F)$ be a continuous Galois representation satisfying $\rhobar\cong\rhobar_{x,\lambda+\eta}$ for some $x\in\tld{\cF\cL}_\cJ(\F)$ (and thus $\rhobar$ is $n$-generic, cf.~Definition~\ref{def:gen:Gal}). Then for each specialization $\rhobar^{\speci}$ of $\rhobar$, there exists $w\in\un{W}$ such that $\rhobar^{\speci}\cong\tau(w^{-1},\lambda+\eta)$, and thus $W_{\obv}(\rhobar^{\speci})$ is defined.
We define the set $W_{\obv}(\rhobar)$ of extremal weights of~$\rhobar$ as follows:
\begin{equation}\label{equ: obv wt}
W_{\obv}(\rhobar)\defeq W^g(\rhobar)\cap \bigcup_{\substack{\rhobar\leadsto \rhobar^{\speci}}}W_{\obv}(\rhobar^{\speci}).
\end{equation}

\subsection{Local models for the Emerton--Gee stacks and their components}
\label{subsec:LMforEG}
In this section, we recall some results on the local models of \cite[\S\,4 and \S\,5]{MLM}, and describe how the moduli of Fontaine--Laffaille modules fit into the theory (Propositions~\ref{prop:rel:cat} and~\ref{prop:rel:cat:1})

We fix a $(3n-1)$-generic Fontaine--Laffaille weight $\lambda+\eta\in X^*(\un{T})$.
Since all schemes are defined over $\Spec \F$ we omit the subscript $\bullet_{\F}$ from the notation when considering the base change to $\F$ of an object $\bullet$ defined over $\cO$ (e.g.~$\GL_{n,\F}$ will be denoted by $\GL_n$ and so on).
This shall cause no confusion.

We write $L\GL_n$ for the loop group on Noetherian $\F$-algebras $R\mapsto \GL_n(R(\!(v)\!))$
and $\Iw$ (resp.~$\Iw_1$) for its Iwahori (resp.~pro-$v$ Iwahori) subgroup
\begin{align*}
R\mapsto&\left\{A\in \GL_n(R[\![v]\!])\mid \text{$A$ is upper triangular modulo $v$}\right\}
\\
(\text{resp.~}
R\mapsto&\left\{A\in \GL_n(R[\![v]\!])\mid\text{$A$ is unipotent upper triangular modulo $v$}\right\}).
\end{align*}
We define an affine flag variety $\Fl$ (resp.~$\tld{\Fl}$) as the (fpqc) sheafification of the presheaf $$R\mapsto \Iw(R)\backslash L\GL_n(R)\quad (\mbox{resp.}~R\mapsto \Iw_1(R)\backslash L\GL_n(R)).$$
Then $\Fl$ is an ind-proper ind-scheme, and the natural map $\tld{\Fl}\ra\Fl$ is a $T$-torsor.
We write the products (over $\F$)
\[
\Fl_{\cJ}\defeq \prod_{j\in\cJ} \Fl\quad\mbox{ and }\quad \tld{\Fl}_{\cJ}\defeq \prod_{j\in\cJ} \tld{\Fl},
\]
and have a $\un{T}$-torsor $\tld{\Fl}_\cJ\ra \Fl_\cJ$, which is the product over $\cJ$ of the $T$-torsor $ \tld{\Fl}\ra \Fl$ above.

Let $a\leq b$ be integers.
If in the definition of $L\GL_n$ we impose the further conditions $v^{-a}A,\ v^b A^{-1}\in \mathrm{M}_n(R[\![v]\!])$, we have the subfunctor $L^{[a,b]}\GL_n$ of $L\GL_n$ which induces finite type subschemes ${\Fl}^{[a,b]}$, $\tld{\Fl}^{[a,b]}$ in ${\Fl}$ and $\tld{\Fl}$ respectively.  We define ${\Fl}^{[a,b]}_\cJ$ and $\tld{\Fl}^{[a,b]}_\cJ$ analogously.

For $\tld{w}\in \tld{W}$, we define $$S^\circ_\F(\tld{w})\defeq\Iw\backslash \Iw\tld{w}\Iw \subset \Fl \quad(\mbox{resp}.~\tld{S}^\circ_\F(\tld{w})\defeq \Iw_1\backslash \Iw\tld{w}\Iw\subset \tld{\Fl})$$ that are called the open affine Schubert cell associated to $\tld{w}$.
For $\tld{z}_\cJ=(\tld{z}_{j})_{j\in\cJ}\in\tld{\un{W}}$ write
\[
S^{\circ}_\F(\tld{z}_\cJ)=\prod_{j\in\cJ}S^\circ_\F(\tld{z}_{j})\quad\mbox{ and }\quad \tld{S}^{\circ}_\F(\tld{z}_\cJ)=\prod_{j\in\cJ}\tld{S}^\circ_\F(\tld{z}_{j}).
\]

We consider the closed sub ind-scheme ${\Fl}^{\nabla_0}$ of $\Fl$  which is the (fpqc) sheafification of the functor
\begin{equation}
\label{eq:Loop:nabla}
R\mapsto\left\{\Iw(R) A\in \Iw(R)\backslash \GL_n(R(\!(v)\!))\ |\
\Big(v\frac{d}{dv} A\Big)A^{-1}\in\frac{1}{v}\Lie \Iw_1(R)
\right\}
\end{equation}
By taking products over $\cJ$ we obtain the closed sub ind-scheme ${\Fl}^{\nabla_0}_\cJ$ of ${\Fl}_\cJ$.
We define $\tld{\Fl}^{\nabla_0}_\cJ$ as the pull back  of $\Fl_\cJ^{\nabla_0}$ along $ \tld{\Fl}_\cJ\ra \Fl_\cJ$.

Denote by $Y^{[0,n-1],\tau}_{\F}$ the base change to $\F$ of the groupoid of Kisin modules with height in $[0,n-1]$ and type $\tau$ (cf.~Defninition~\ref{defn:FCris:dd}).
Given a lowest alcove presentation $(s,\mu)$ of $\tau$ where $\mu+\eta$ is $n$-generic Fontaine--Laffaille, %
\cite[Corollary 5.2.3]{MLM} gives a natural map
\[\pi_{(s,\mu)}:Y^{[0,n-1],\tau}_{\F}\cong [(\tld{\Gr}^{[0,n-1]}_{\cG,\F})^{\cJ}/_{(s,\mu)}\un{T}]\into [\tld{\Fl}^{[0,n-1]}_{\cJ}/_{(s,\mu)}\un{T}] \]
given by sending $\fM$ to the class of $A_{\fM,\beta}$, for any choice of eigenbasis $\beta$ (cf.~\emph{loc.~cit.}~for the definition of $(\tld{\Gr}^{[0,n-1]}_{\cG,\F})^{\cJ}$ and the $\un{T}$-action involved).
We define $\tld{Y}^{[0,n-1],\tau}_\F$ to be the pull-back of $Y^{[0,n-1],\tau}_\F$ under the $\un{T}$-torsor $\tld{\Fl}^{[0,n-1]}_{\cJ}\to [\tld{\Fl}^{[0,n-1]}_{\cJ}/_{(s,\mu)}\un{T}]$.
In particular we have a natural map
\begin{equation}\label{eq: pi tilde s mu}
\tld{\pi}_{(s,\mu)}: \tld{Y}^{[0,n-1],\tau}_\F \to \tld{\Fl}^{[0,n-1]}_{\cJ}.
\end{equation}

Let $\lambda+\eta$ be the weight fixed at the beginning of this section.
Let $\zeta_\lambda\in X^*(\un{Z})$ correspond to the class $t_\lambda \un{W}_a\in \tld{\un{W}}/\un{W}_a\cong X^*(\un{Z})$.
Let $V$ be a Serre weight with lowest alcove presentation $(\tld{w}_1,\omega)$ compatible with $\zeta_\lambda$, such that $\omega$ is $(n-1)$-generic Fontaine--Laffaille.
We define ${C}^{\zeta_\lambda}_{V}$ as the Zariski closure of ${{S}^\circ_\F(\tld{w}_1^* w_0)\tld{s}^*\cap {\Fl}_\cJ^{\nabla_0}}$ for an arbitrary $\tld{s}\in \tld{\un{W}}$ satisfying $\tld{s}(0)=\omega$ \cite[equation (4.9)]{MLM}.
It is a closed irreducible subvariety of $\Fl^{\nabla_0}_\cJ$ of dimension $\binom{n}{2}[K:\Qp]$, which does not depend on the equivalence class of the lowest alcove presentation of $V$ compatible with $\zeta_\lambda$.
We define $\tld{C}^{\zeta_\lambda}_{V}$ as the pullback of ${C}^{\zeta_\lambda}_{V}$ along $ \tld{\Fl}_\cJ\ra \Fl_\cJ$.
If $V\cong F(\mu)$ for some $\mu\in X_1(\un{T})$ which is $(3n-1)$-deep (in the sense of \cite[Definition 2.1.10]{MLM}) we obtain from \cite[Theorem 7.4.2]{MLM} (see also \emph{loc.~cit.}~Remark 7.4.3(2)) a $\un{T}$-torsor
\begin{equation}
\label{eq:LM:component}
\tld{C}^{\zeta_\lambda}_{V}\stackrel{f.s.}{\longrightarrow}\cC_{V}\subseteq \cX_{n,\mathrm{red}}.
\end{equation}

Fix once and for all the lowest alcove presentation $(1,\lambda+\eta)$ for the Serre weight $F(\lambda)$, so that $F_{(1,\lambda+\eta)}=F(\lambda)$.
We have an action of $\un{T}$ on $\Fl_\cJ$ by right multiplication.
This action induces actions on $\Fl_\cJ^{\nabla_0}$, $C_V^{\zeta_\lambda}$ and $\tld{C}_V^{\zeta_\lambda}$.

We now want to relate the groupoids from \S\,\ref{subsubBKD}, \S\,\ref{subsub:Phi-mod}, and \S\,\ref{sub:MFSandGR}, the scheme $\tld{\cF\cL}_\cJ$, and the objects introduced above. Recall from Definition~\ref{def: Sch var} and \S\,\ref{subsubsec:PS} that given $w=(w_j)_{j\in\cJ},u=(u_j)_{j\in\cJ}\in\un{W}$ we have the Schubert cell $\tld{\cS}^\circ(w,u)$ and the Schubert variety $\tld{\cS}(w,u)$ in $\tld{\cF\cL}_\cJ$ associated with $(w, u)\in \un{W}\times\un{W}$. We also write $\cS^\circ(w,u)$ (resp.~$\cS(w,u)$) for the corresponding Schubert cell (resp.~Schubert variety) in $\un{B}\backslash \un{G}$.

We can now state the proposition resuming the relations among the objects introduced so far.
\begin{prop}
\label{prop:rel:cat}
Let $\lambda\in X^*_+(\un{T})$ be a dominant weight with $\lambda+\eta$ being Fontaine--Laffaille. Let $\tau$ be a tame inertial type with a lowest alcove presentation $(s,\mu)$ where $\mu+\eta$ is $2n$-generic Fontaine--Laffaille, such that $(s,\mu)$ is compatible with $\zeta_\lambda$ and satisfies $\tld{\Fl}^{[0,n-1]}_{\cJ}\tld{w}^*(\tau)\subseteq\tld{\Fl}^{[1-n,n-1]}_{\cJ}t_{\lambda+\eta}$.
Then we have a commutative diagram of groupoid-valued functors over Noetherian $\F$-algebras:
\begin{equation}
\label{diagram:rel:gpds}
\xymatrix{
&\tld{Y}^{[0,n-1], \tau}_\F\ar_{r_{\tld{w}^*(\tau)}}[d]\ar@/^1.0pc/^{\eps_\tau}[rrd]&&&\\
\tld{\cF\cL}_{\cJ}\ar^{r_{\lambda+\eta}\qquad}[r]&\tld{\Fl}^{[1-n,n-1]}_\cJ t_{\lambda+\eta}\ar[r]&\big[\tld{\Fl}^{[1-n,n-1]}_\cJ t_{\lambda+\eta}\slash{{\sim}_{\un{T}\text{-\textnormal{sh.cnj}}}}\big]\ar_-{\iota}[r]& \Phi\text{-}\Mod^{\text{\emph{\'et}},n}
}
\end{equation}
where the maps are described as follows:
\begin{enumerate}
\item the map $r_{\tld{w}^*(\tau)}$  is given by composing the map $\tld{\pi}_{(s,\mu)}:\tld{Y}^{[0,n-1],\tau}_\F\to \tld{\Fl}^{[0,n-1]}_{\cJ} $ with right multiplication by $\tld{w}^*(\tau)$ (which lands in $\tld{\Fl}^{[1-n,n-1]}_{\cJ}t_{\lambda+\eta}$ by assumption);
\item the top diagonal map is the composition of the natural map $\tld{Y}^{[0,n-1],\tau}_\F\to Y^{[0,n-1],\tau}_\F$ with the map $\eps_\tau$ defined in (\ref{eq:KMtoPHI});
\item the map $\iota$ is induced by the map sending $(A^{(j)}v^{\lambda_{j}+\eta_j})_{j\in\cJ}\in \prod_{j\in\cJ} L^{[1-n,n-1]}\GL_nv^{\lambda_{j}+\eta_j}$ to the free \'etale $\varphi$-module $\cM$ of rank $n$, such that the matrix of $\phi^{(j)}_\cM$ in the standard basis is given by $A^{(j)}v^{\lambda_{j}+\eta_j}$;
\item
the map $r_{\lambda+\eta}$ is induced by right multiplication by $t_{\lambda+\eta}$.
\end{enumerate}
Moreover, the map $\iota$ is a monomorphism of stacks and the image of $r_{\lambda+\eta}$ is contained in $\tld{\Fl}^{\nabla_0}_\cJ$ and equals $\tld{C}^{\zeta_\lambda}_{F(\lambda)}$.
In particular, the $\un{T}$-fixed points of $C^{\zeta_\lambda}_{F(\lambda)}$ are the elements $\un{W}t_{\lambda+\eta}$.
\end{prop}

\begin{proof}
The commutativity of the top triangle is \cite[Proposition 5.4.7]{MLM} (where we take $\nu$ to be $\lambda$ in the notation of \emph{loc.~cit}. and $a=1-n$, $b=n-1$).

Note that $\tld{\cF\cL}_{\cJ}$ naturally embeds in $\tld{\Fl}_{\cJ}$, by thinking of $\un{G}$ as constant matrices in the loop group. Thus its translate $r_{\lambda+\eta}(\tld{\cF\cL}_\cJ)$ is an irreducible closed subscheme of $\tld{\Fl}^{\nabla_0}_\cJ$, and contains $\Iw_{1,\cJ}\backslash \Iw_{1,\cJ}w_0\un{B}t_{\lambda+\eta}=\tld{S}_{\F}^{\circ}(w_0)t_{\lambda+\eta}=r_{\lambda+\eta}(\tld{\cS}^\circ(w_0,1))$ as an open dense subscheme. Since $\tld{C}^{\zeta_\lambda}_{F(\lambda)}$ is by definition the closure of $\tld{S}_\F^{\circ}(w_0)t_{\lambda+\eta}$, we conclude by dimension reasons.

The last assertion follows by passing to the quotient $\tld{C}^{\zeta_\lambda}_{F(\lambda)}\onto C^{\zeta_\lambda}_{F(\lambda)}$ and noting that the set of $\un{T}$-fixed points of $\un{B}\backslash \un{G}$ is precisely the image of $\un{W}\into \un{B}\backslash \un{G}$.
\end{proof}

Recall that we have a formally smooth morphism $\tld{\cF\cL}_{\cJ}\ra \mathrm{FL}_n^{\lambda+\eta}$.
We emphasize that there is a shift of the indices in $\cJ$ when we pass from $\mathrm{FL}_n^{\lambda+\eta}$ to $\Phi\text{-}\Mod^{\text{\'et},n}$ via the maps in Proposition~\ref{prop:rel:cat}. More precisely, if $A^{(j)}\in\GL_n(R)$ is the matrix of the Frobenius map $\phi^{(j)}: \gr^{\bullet}(M^{(j)}) \rightarrow M^{(j+1)}$ of a Fontaine--Laffaille module $M$ (with respect to a given basis of $M$), then $A^{(j)}v^{\lambda_{j}+\eta_j}$ is the matrix of $\phi^{(j)}_\cM:\cM^{(j-1)}\rightarrow\cM^{(j)}$ (with respect to an appropriate basis on the \'etale $\varphi$-module $\cM$ attached to $M$).

\begin{defn}
Let $\tld{w}\in\Adm(\eta)$.
Let $\tau$ be a tame inertial type with a lowest alcove presentation $(s,\mu)$ where $\mu+\eta$ is $2n$-generic Fontaine--Laffaille.
We define the locally closed substack $\tld{Y}^{[0,n-1],\tau}_{\tld{w}^*}$ of $\tld{Y}^{[0,n-1],\tau}_\F$ as the inverse image, via $\tld{\pi}_{(s,\mu)}$, of the open Schubert cell $\tld{S}_\F^\circ(\tld{w}^*)$ in $\tld{\Fl}_\cJ^{[0,n-1]}$.
\end{defn}

\begin{rmk}
\label{rmk:shape:strata}
\begin{enumerate}
\item
We note that for any finte extension $\F'$ of $\F$ and $\tld{w}\in\Adm(\eta)$, the objects of $\tld{Y}^{[0,n-1],\tau}_{\tld{w}^*}(\F')$ are exactly the Breuil--Kisin modules of shape $\tld{w}^*$, cf.~\cite[Definition 5.1.9]{MLM}.
\item
\label{it:open:substacks:2}
For each $\tld{w}\in\Adm(\eta)$, we recall from \cite[Definition 5.2.4(i)]{MLM} the open substack $Y^{[0,n-1],\tau}_{\F}(\tld{w}^*)\subset Y^{[0,n-1],\tau}_\F$, and
let $\tld{Y}^{[0,n-1]}_\F(\tld{w}^*)$ be its pre-image in $\tld{Y}^{[0,n-1]}_\F$.
Then there is an inclusion $\tld{Y}^{[0,n-1],\tau}_{\tld{w}^*}\subseteq \tld{Y}^{[0,n-1],\tau}_\F(\tld{w}^*)$, but this is not an equality in general since $\tld{Y}^{[0,n-1],\tau}_\F(\tld{w}^*)\subseteq \tld{Y}^{[0,n-1],\tau}_\F$ is always open, while $\tld{Y}^{[0,n-1],\tau}_{\tld{w}^*}$ is only locally closed.
\end{enumerate}
\end{rmk}

\begin{prop}
\label{prop:rel:cat:1}
Under the notation and hypotheses of Proposition~\ref{prop:rel:cat}, the diagram (\ref{diagram:rel:gpds}) can be completed as follows over the category of local Artinian $\F$-algebras with residue field $\F$:
\begin{equation*}
\xymatrix{
\tld{Y}^{[0,n-1], \tau}_\F\ar@/^2.0pc/^{T_{dd}^*}[rrrd]\ar_{r_{\tld{w}^*(\tau)}}[d]\ar@/^1.0pc/^{\eps_\tau}[rrd]&&&\\
\tld{\Fl}^{[1-n,n-1]}_\cJ t_{\lambda+\eta}\ar[r]&\big[\tld{\Fl}^{[1-n,n-1]}_\cJ t_{\lambda+\eta}\slash{{\sim}_{\un{T}\text{-\textnormal{sh.cnj}}}}\big]\ar_-{\iota}[r]&\Phi\text{-}\Mod^{\text{\emph{\'et}},n}\ar^{\bV^*_K}_{\sim}[r]& \Rep^n(G_{K_{\infty}})\\
\tld{\cF\cL}_{\cJ}\ar^{r_{\lambda+\eta}}[u]\ar^{\rhobar_{\bullet,\lambda+\eta}}[rrr]&&&
\Rep^n(G_{K}).\ar_{\Res}[u]
}
\end{equation*}
\end{prop}
\begin{proof}
Given Proposition~\ref{prop:rel:cat}, we only need to check the commutativity of the bottom square, and this follows easily from \cite[Lemma 2.2.8]{HLM}.
\end{proof}

\subsection{Relevant types and Serre weights}
\label{subsec:RTandSW}
We fix a $(3n-1)$-generic Fontaine--Laffaille weight $\lambda+\eta\in X^*(\un{T})$, and the lowest alcove presentation $(1,\lambda+\eta)$ for the Serre weight $F(\lambda)$. In this section, we introduce the set of \emph{$F(\lambda)$-relevant inertial types} and compute the pullback (to $\tld{\cF\cL}_\cJ$) of the \emph{shape stratification} on $\tld{Y}^{[0,n-1],\tau}_{\F}$ with $\tau$ being a $F(\lambda)$-relevant inertial type (see Proposition~\ref{prop:rel:cat}). Then for each $x\in\tld{\cF\cL}_\cJ$, we use the set of extremal weights to control the shape of $\rhobar_{x,\lambda+\eta}$ with respect to $F(\lambda)$-relevant types (see Lemma~\ref{lemma:shape:detect}).

\begin{defn}
\label{defn:relevant:types}
Assume that $\lambda+\eta$ is $(3n-1)$-generic Fontaine--Laffaille.
A tame inertial type is called \emph{$F(\lambda)$-relevant} if it admits a lowest alcove presentation of the form $(s,\lambda-s(\eta))$, for some $s\in\un{W}$.
\end{defn}
Given a $F(\lambda)$-relevant $\tau$ attached to some $s\in\un{W}$, we clearly have $\lambda+\eta-s(\eta)$ is $2n$-generic Fontaine--Laffaille and
\begin{equation}\label{equ:relevant:inclusion}
\tld{\Fl}^{[0,n-1]}_{\cJ}\tld{w}^*(\tau)=\tld{\Fl}^{[0,n-1]}_{\cJ}s^{-1}t_{\lambda+\eta-s(\eta)}\subseteq\tld{\Fl}^{[1-n,n-1]}_{\cJ}t_{\lambda+\eta}.
\end{equation}

\begin{rmk}
\label{rmk:relevant:outer}
Recall from \cite[\S\,2.3.1]{MLM} that given a suitably generic tame inertial type $\tau$ we have a set $\JH_\out(\ovl{\sigma(\tau)})$ of outer weights in $\JH(\ovl{\sigma(\tau)})$.
One can check that the set of $F(\lambda)$-relevant types is precisely the set of tame inertial type $\tau$ such that $\JH_\out(\ovl{\sigma(\tau)})$ is defined and contains $F(\lambda)$. In particular, $\ovl{\sigma(\tau)}$ has $F(\lambda)$ as a Jordan--H\"older factor with multiplicity one (cf. \cite[Proposition~10.1.2]{GHS}, \cite[\S\,II 8.19, 9.14, 9.16]{RAGS}).
\end{rmk}

\begin{lemma}
\label{lem:shape:Schubert}
Assume that $\lambda+\eta$ is $(3n-1)$-generic Fontaine--Laffaille.
Let $s\in \un{W}$ and $\tau=\tau(s, \lambda+\eta-s(\eta))$ be a $F(\lambda)$-relevant type.
Then for each $u\in \un{W}$ we have
\begin{equation}
\label{eq:intersec:cmp}
r_{\tld{w}^*(\tau)}\Big(\tld{Y}^{[0,n-1],\tau}_{u t_\eta}\Big)\cap \tld{C}_{F(\lambda)}^{\zeta_\lambda}=
r_{\lambda+\eta}\Big(\tld{\cS}^\circ(u w_0,w_0s^{-1})\Big).
\end{equation}
In particular, the stratification $\big\{\tld{Y}^{[0,n-1],\tau}_{\tld{w}^*}\big\}_{\tld{w}\in\Adm(\eta)}$ induces the stratification $\big\{\tld{\cS}^\circ(u w_0,w_0s^{-1})\big\}_{u\in\un{W}}$ on $\tld{\cF\cL}_\cJ$ and for any closed point $x\in \tld{\cF\cL}_\cJ(\F)$ the shape of $\rhobar_{x,\lambda+\eta}$ with respect to $\tau$ \emph{(}as in \cite[Definition 3.2.19]{LLL}\emph{)} is defined.
\end{lemma}
\begin{proof}
The assumption in Proposition~\ref{prop:rel:cat} holds by (\ref{equ:relevant:inclusion}). By Proposition~\ref{prop:rel:cat}, we have
\begin{equation}
\label{eq:inc:2}
\tld{C}^{\zeta_\lambda}_{F(\lambda)}=r_{\lambda+\eta}(\tld{\cF\cL}_\cJ) =\underset{u\in \un{W}}{\bigsqcup} r_{\lambda+\eta}\Big(\tld{\cS}^\circ(u w_0,w_0s^{-1})\Big).
\end{equation}
Now we observe that
\[r_{\tld{w}^*(\tau)}\Big(\tld{Y}^{[0,n-1],\tau}_{u t_\eta}\Big)=\Iw_{1,\cJ} \backslash \Iw_{1,\cJ}u t_{\eta}\Iw_{\cJ} t_{-\eta}s^{-1}t_{\lambda+\eta},\]
which contains $\Iw_{1,\cJ}\backslash\Iw_{1,\cJ}u (w_0 \un{B} w_0)s^{-1}t_{\lambda+\eta}=r_{\lambda+\eta}\Big(\tld{\cS}^\circ(u w_0,w_0s^{-1})\Big)$.
It follows that the left-hand side in equation (\ref{eq:intersec:cmp}) contains the right-hand side. The fact that the $\tld{Y}^{[0,n-1],\tau}_{\tld{w}^*}$ are disjoint for distinct $\tld{w}\in\Adm(\eta)$, together with equation (\ref{eq:inc:2}) implies the above containment is an equality.
\end{proof}

We recall the $\uparrow$ order on $p$-alcoves defined in \cite[\S\,II.6.5]{RAGS}, which induces the $\uparrow$ order on $\tld{\un{W}}$ (cf.~the discussion at the beginning of \cite[\S\,4]{LLL}).
\begin{lemma}
\label{lemma:shape:detect}
Assume that $\lambda+\eta$ is $(3n-1)$-generic Fontaine--Laffaille
and let $s\in\un{W}$. Let $\rhobar\defeq \rhobar_{x,\lambda+\eta}$ for some $x\in \tld{\cF\cL}_\cJ(\F)$.
Suppose that $\tau=\tau(s, \lambda+\eta-s(\eta))$ is $F(\lambda)$-relevant and $\tld{w}^*(\rhobar,\tau)\neq t_\eta$. Then $\# \Big(W_\obv(\rhobar)\cap \JH\big(\ovl{\sigma(\tau)}\big)\Big)>1$.
\end{lemma}
\begin{proof}
By Lemma~\ref{lem:shape:Schubert}, we must have $\tld{w}^*(\rhobar,\tau)=u t_\eta$ with $u\neq 1$ (or equivalently, we have $\ell\big(\tld{w}^*(\rhobar,\tau)\big)<\ell(t_\eta)$). Furthermore, by the same Lemma, we have $\rhobar \leadsto \rhobar^{\speci}\defeq \ovl{\tau}(s u^{-1},\lambda+\eta)$.

Set $w_1=u^{-1}\neq 1$ and $\tld{w}_1=t_{\eta_{w_1}}w_1\in \tld{\un{W}}_1^+$ be the image of $w_1$ under our fixed injection $\un{W}\hookrightarrow \tld{\un{W}}_1^+$ (see \cite[(5.1)]{herzig-duke} for the definition of $\eta_{w_1}\in X^*(\un{T})$).
Set $\omega=t_{\lambda+\eta}s w_1\tld{w}_1^{-1}(0)$.
We claim that
\[F_{({\tld{w}_1,\omega})}\in W_\obv(\rhobar)\cap \JH\big(\ovl{\sigma(\tau)}\big)\]
which would finish the proof.

It suffices to check the following:
\begin{enumerate}[label=(\alph*)]
\item
\label{en:a:pf}
$F_{({\tld{w}_1,\omega})}\in W_\obv(\rhobar^{\speci})$ is the extremal weight of $\rhobar^{\speci}$ corresponding to $w_1$.
\item
\label{en:b:pf}
$F_{({\tld{w}_1,\omega})}\in W^g(\rhobar)$.
\item
\label{en:c:pf}
$F_{({\tld{w}_1,\omega})}\in \JH\big(\ovl{\sigma(\tau)}\big)$.
\end{enumerate}

Item~\ref{en:a:pf} follows from \cite[Proposition 2.6.2]{MLM}.

As $\tau=\tau(s, \lambda+\eta-s(\eta))$ and $\lambda+\eta-s(\eta)$ is $2n$-generic Fontaine--Laffaille, we deduce from \cite[Remark 7.4.3(2)]{MLM} that (\ref{eq:LM:component}) is a $\un{T}$-torsor for $V=F_{(\tld{w}_1,\omega)}$.
Hence, item~\ref{en:b:pf} follows from the fact that
\[\tld{C}^{\zeta_\lambda}_{F_{(\tld{w}_1,\omega)}}\supseteq \tld{S}_\F^\circ(\tld{w}_1^*w_0)w_0t_{-\eta_{w_1}}s^{-1}t_{\lambda+\eta}\cap \tld{\Fl}_\cJ^{\nabla_0}\supseteq r_{\lambda+\eta}\Big(\tld{\cS}^\circ(u w_0,w_0s^{-1})\Big)
\]
(which implies that $\rhobar\in|\cC_{F_{(\tld{w}_1,\omega)}}(\F)|$).

Finally, to check the item \ref{en:c:pf} by \cite[Proposition 2.3.7]{MLM}, we need to find $\tld{w}_2=\kappa t_{\nu}\in \tld{\un{W}}^+$ such that
\begin{itemize}
\item $t_{\lambda+\eta-s(\eta)}s(-\nu)=\omega$
\item $\tld{w}_2\uparrow \tld{w}_h\tld{w}_1$.
\end{itemize}
The first item is equivalent to $\nu=\eta_{w_1}-\eta$.
Since $\eta_{w_1}-\eta\equiv w_0(\eta_{w_0w_1})$ modulo $X^0(\un{T})$
the second item is equivalent to
\[\kappa t_{w_0(\eta_{w_0w_1})}\uparrow t_{\eta_{w_0w_1}}w_0w_1,\]
but this follows from Lemma~\ref{lem:Bruhat_order_vertex} below applied to $\sigma=w_0w_1$. This finishes the proof. %
\end{proof}

\begin{lemma}\label{lem:Bruhat_order_vertex} Let $\sigma\in \un{W}$ and $\tld{\sigma}=t_{\eta_{\sigma}}\sigma\in \tld{\un{W}}^+_1$ is the image of $\sigma$ under our fixed injection $\un{W}\hookrightarrow \tld{\un{W}}^+_1$. Let $\sigma'\in \un{W}$ be the unique element such that $\sigma't_{w_0(\eta_{\sigma})}\in \tld{\un{W}}^+$. Then
\[\sigma't_{w_0(\eta_{\sigma})} \uparrow \tld{\sigma}\]
\end{lemma}

\begin{proof}
Consider the set of elements $t_{w_0(\eta_\sigma)}\un{W}\defeq\{t_{w_0(\eta_\sigma)}w\mid w\in \un{W}\}$ whose corresponding alcoves $\{t_{w_0(\eta_\sigma)}w(\un{A}_0) \mid w\in \un{W}\}$ all contain $v_0\defeq w_0(\eta_{\sigma})$ as a vertex.
By \cite[Lemma 7.5]{Haines-Ngo}, the set $t_{w_0(\eta_\sigma)}\un{W}$ has a unique minimal and maximal element.
We claim that $t_{w_0(\eta_{\sigma})}w_0$ is the maximal element. In fact, for each alcove $t_{w_0(\eta_{\sigma})}w_1(\un{A}_0)$, there exists an anti-dominant alcove $t_{w_0(\eta_{\sigma})}w_2(\un{A}_0)$ for some $w_2\in\un{W}$ such that $t_{w_0(\eta_{\sigma})}w_1 \leq t_{w_0(\eta_{\sigma})}w_2$. We clearly have $t_{w_0(\eta_{\sigma})}w_0 \uparrow t_{w_0(\eta_{\sigma})}w_2$.  As the $\uparrow$ order is opposite to the Bruhat order in the anti-dominant chamber by Wang's Theorem \cite[Theorem 4.1.1]{LLL}, we deduce that $t_{w_0(\eta_\sigma)}w_2\leq t_{w_0(\eta_\sigma)}w_0$ and thus $t_{w_0(\eta_\sigma)}w_1\leq t_{w_0(\eta_\sigma)}w_0$. Hence, $t_{w_0(\eta_{\sigma})}w_0$ is the maximal element in $t_{w_0(\eta_\sigma)}\un{W}$.

It follows from \cite[Lemma 7.5]{Haines-Ngo} that $t_{w_0(\eta_\sigma)}$ is the minimal element in $t_{w_0(\eta_\sigma)}\un{W}$. In particular $t_{w_0(\eta_\sigma)}\leq t_{w_0(\eta_\sigma)}w_0\sigma=w_0t_{\eta_\sigma}\sigma=w_0 \tld{\sigma}$.
We conclude from the fact that $\sigma't_{w_0(\eta_{\sigma})}$ and $\tld{\sigma}$ are the minimal length representatives in $\un{W}\backslash \un{W}_a$ of $t_{w_0(\eta_\sigma)}$ and $w_0 \tld{\sigma}$.
\end{proof}

\subsection{The partition $\cP_\cJ$ and extremal weights}
\label{subsec:PandOW}

In this section, we interpret the partition $\cP_\cJ$ defined in \S\,\ref{sub:finest:elements} in terms of extremal weights and specializations (see Theorem~\ref{thm: obv wt and partition}).
We fix a $n$-generic Fontaine--Laffaille weight $\lambda+\eta\in X^*(\un{T})$, and the lowest alcove presentation $(1,\lambda+\eta)$ for $F(\lambda)$. Hence, $W_{\obv}(\taubar)$ is defined whenever $\taubar=\taubar(w^{-1},\lambda+\eta)$ for some $w\in\un{W}$.

\begin{defn}
We define the following set of Serre weights
\[
SW(\lambda)\defeq \left\{
F_{(\tld{w}_1, t_{\lambda+\eta}s w_1\tld{w}_1^{-1}(0))},\ (s,\tld{w}_1)\in \un{W}\times \tld{\un{W}}_1^+
\right\}
\]
where we write $\tld{w}_1=t_{\eta_{w_1}}w_1$.
\end{defn}

We write $r_{\lambda+\eta}:\un{B}\backslash\un{G}\hookrightarrow\Fl_\cJ$ for the map induced from $r_{\lambda+\eta}$ in (\ref{diagram:rel:gpds}) (by abuse of notation).
We clearly have
$$SW(\lambda)=\underset{w\in\un{W}}{\bigcup}W_{\obv}(\taubar(w^{-1},\lambda+\eta)),$$
and thus $SW(\lambda)$ is exactly the union of $W_{\obv}(\rhobar_{x,\lambda+\eta})$ for $\ovl{x}$ running in the set of $\un{T}$-fixed $\F$-points of~$\un{B}\backslash\un{G}=r_{\lambda+\eta}^{-1}(C_{F(\lambda)}^{\zeta_\lambda})$, and $x$ any choice of lift of $\ovl{x}$ in $\tld{\cF\cL}_\cJ(\F)$.

\begin{prop}\label{prop: wt and Sch var}
Assume that $\lambda+\eta$ is $n$-generic Fontaine--Laffaille. For each $(s,\tld{w}_1)\in \un{W}\times \tld{\un{W}}_1^+$, we have
\begin{equation}\label{equ: intersection of component}
C_{F_{(\tld{w}_1, t_{\lambda+\eta}s w_1\tld{w}_1^{-1}(0))}}^{\zeta_\lambda}\cap C_{F(\lambda)}^{\zeta_\lambda}=r_{\lambda+\eta}(\cS(w_1^{-1} w_0,w_0s^{-1})).
\end{equation}
Moreover, the map
$$V\mapsto r_{\lambda+\eta}^{-1}\left(C_{V}^{\zeta_\lambda}\cap C_{F(\lambda)}^{\zeta_\lambda}\right)$$
induces a bijection between $SW(\lambda)$ and the set of Schubert varieties in $\un{B}\backslash\un{G}$.
\end{prop}
\begin{proof}
We write $\omega=t_{\lambda+\eta}s w_1\tld{w}_1^{-1}(0)$ and $\tld{w}_1=t_{\eta_{w_1}}w_1$. The equation (\ref{equ: intersection of component}) follows directly from (cf.~the proof of item~\ref{en:b:pf} of Lemma~\ref{lemma:shape:detect})
\begin{multline*}
\left(S_\F^\circ(\tld{w}_1^*w_0)w_0t_{-\eta_{w_1}}s^{-1}t_{\lambda+\eta}\cap \Fl_\cJ^{\nabla_0}\right)\cap r_{\lambda+\eta}(\un{B}\backslash\un{G})\\=S_\F^\circ(\tld{w}_1^*w_0)w_0t_{-\eta_{w_1}}s^{-1}t_{\lambda+\eta}\cap r_{\lambda+\eta}(\un{B}\backslash\un{G})=r_{\lambda+\eta}\Big(\cS^\circ(w_1^{-1} w_0,w_0s^{-1})\Big)
\end{multline*}
and the fact that $C_{F_{(\tld{w}_1, \omega)}}^{\zeta_\lambda}$ is the closure of $S_\F^\circ(\tld{w}_1^*w_0)w_0t_{-\eta_{w_1}}s^{-1}t_{\lambda+\eta}\cap \Fl_\cJ^{\nabla_0}$. The desired bijection follows from the fact that, given two elements $s,s'\in\un{W}$, we have $s(\eta_{w_1})=s'(\eta_{w_1})$ if and only if $\cS(w_1^{-1} w_0,w_0s^{-1})=\cS(w_1^{-1} w_0,w_0(s')^{-1})$.
\end{proof}

\begin{rmk}
It follows from Proposition~\ref{prop: refinement of Bruhat partition} and Proposition~\ref{prop: wt and Sch var} that $\cP_\cJ$ is the coarsest partition on $\tld{\cF\cL}_\cJ$ such that $r_{\lambda+\eta}^{-1}\left(\tld{C}_{V}^{\zeta_\lambda}\cap \tld{C}_{F(\lambda)}^{\zeta_\lambda}\right)$ is a union of elements in the partition, for each $V\in SW(\lambda)$.
\end{rmk}

Let $\rhobar\defeq \rhobar_{x,\lambda+\eta}$ for some $x\in\tld{\cF\cL}_\cJ(\F)$, and recall from \S\,\ref{sec:SWandGR} the set $W_\obv(\rhobar)$.
We consider the following enhancement $SP(\rhobar)$ of it: the elements of $SP(\rhobar)$ are \emph{pairs} $(V,\rhobar^{\speci})$ where $V\in W^g(\rhobar)$ and $\rhobar^{\speci}$ is a specialization of $\rhobar$ such that $V\in W_{\obv}(\rhobar^{\speci})$. We have a natural surjective map $SP(\rhobar)\onto W_\obv(\rhobar)$.
For each $(V,\rhobar^{\speci})\in SP(\rhobar)$, there exists a unique pair $s,w_1\in\un{W}$ such that $(s w_1,\lambda)$ is a lowest alcove presentation of $\rhobar^{\speci}$ and $V=V_{\rhobar^{\speci},w_1}$ is the extremal weight of $\rhobar^{\speci}$ corresponding to $w_1$ (with respect to $(s w_1,\lambda)$).
Note that the pair $(F(\lambda),\rhobar^{\speci})$ is an element of $SP(\rhobar)$ for each specialization $\rhobar\leadsto \rhobar^{\speci}$. For each $(V_{\rhobar^{\speci},w_1},\rhobar^{\speci})\in SP(\rhobar)$ with $\rhobar^{\speci}=\taubar(s w_1,\lambda+\eta)$, we set $\theta_{\rhobar}((V_{\rhobar^{\speci},w_1},\rhobar^{\speci}))\defeq s$. This defines a map $\theta_{\rhobar}:SP(\rhobar)\rightarrow \un{W}$.

If $\lambda+\eta$ is $(3n-1)$-generic Fontaine--Laffaille, then $\lambda+\eta-s(\eta)$ is $2n$-generic Fontaine--Laffaille for each $s\in\un{W}$, and we deduce from \cite[Remark 7.4.3(2)]{MLM} (cf.~the proof of item~\ref{en:b:pf} of Lemma~\ref{lemma:shape:detect}) that (\ref{eq:LM:component}) is a $\un{T}$-torsor for each $V\in SW(\lambda)$.

\begin{lemma}\label{lem: SP and Sch cell}
Assume that $\lambda+\eta$ is $(3n-1)$-generic Fontaine--Laffaille. Let $x\in\tld{\cF\cL}_\cJ(\F)$ be a point, and $\rhobar\defeq\rhobar_{x,\lambda+\eta}$. Then $(V_{\rhobar^{\speci},w_1},\rhobar^{\speci})\in SP(\rhobar)$ for some $\rhobar^{\speci}=\taubar(s w_1,\lambda+\eta)$ if and only if $x\in\tld{\cS}^\circ(w_1^{-1}w_0,w_0 s^{-1})$.
\end{lemma}
\begin{proof}
It is easy to see that $\tld{\cS}^\circ(w_1^{-1}w_0,w_0 s^{-1})=\tld{\cS}(w_1^{-1}w_0,w_0 s^{-1})\cap \cM_{(s w_1)^{-1}}^\circ$. Now this follows immediately from (\ref{eq:LM:component}), Lemma~\ref{lem: specialization and open cell} and Proposition~\ref{prop: wt and Sch var}.
\end{proof}

\begin{thm}\label{thm: obv wt and partition}
Assume that $\lambda+\eta$ is $(3n-1)$-generic Fontaine--Laffaille. For each $x\in\tld{\cF\cL}_\cJ(\F)$, the map $\theta_{\rhobar_{x,\lambda+\eta}}$ is bijective. Moreover, the following conditions are equivalent for two points $x,x'\in\tld{\cF\cL}_\cJ(\F)$:
\begin{enumerate}
\item \label{it: obv partition 1}
there exists $\cC\in\cP_\cJ$ such that $x,x'\in\cC(\F)$;
\item \label{it: obv partition 2}
$SP(\rhobar_{x,\lambda+\eta})=SP(\rhobar_{x',\lambda+\eta})$;
\item \label{it: obv partition 3}
$\{\rhobar^{\speci}\mid \rhobar_{x,\lambda+\eta}\leadsto \rhobar^{\speci}\}=\{\rhobar^{\speci}\mid \rhobar_{x',\lambda+\eta}\leadsto \rhobar^{\speci}\}$.
\end{enumerate}
\end{thm}
\begin{proof}
We first check the bijectivity of $\theta_{\rhobar_{x,\lambda+\eta}}$. Let $\cC$ be an element of $\cP_\cJ$ with $x\in\cC(\F)$. For a given $s\in\un{W}$, there exists a unique $w_1\in\un{W}$ such that $\cC\subseteq \tld{\cS}^\circ(w_1^{-1}w_0,w_0 s^{-1})$ by Proposition~\ref{prop: refinement of Bruhat partition}. Now by Lemma~\ref{lem: SP and Sch cell} we have a map $\un{W}\rightarrow SP(\rhobar)$, which can be readily checked to be the inverse map of $\theta_{\rhobar_{x,\lambda+\eta}}$.

We now check the equivalence. The equivalence between item~(\ref{it: obv partition 1}) and item~(\ref{it: obv partition 2}) follows immediately from Proposition~\ref{prop: refinement of Bruhat partition} and Lemma~\ref{lem: SP and Sch cell}. The equivalence between item~(\ref{it: obv partition 1}) and item~(\ref{it: obv partition 3}) follows from item~\ref{it: open cover:iii} of Lemma~\ref{lem: open cover} and Lemma~\ref{lem: specialization and open cell} (and thus is true for any dominant weight $\lambda\in X^*_+(\un{T})$ with $\lambda+\eta$ being Fontaine--Laffaille).
\end{proof}

\newpage

\section{Canonical lifts of invariant functions}
\label{sec:CLIF}
In this section, we interpret invariant functions on $\tld{\cF\cL}_\cJ$ as (normalized) mod-$p$ reduction of functions on the moduli of Weil--Deligne representations with $F(\lambda)$-relevant inertial types, for each $\lambda+\eta$ being $(3n-1)$-generic Fontaine--Laffaille. Our main result is Theorem~\ref{thm:Feval-2}.
\subsection{Sub inertial type and index set}\label{sub:dd}
Let $\tau$ be a $1$-generic tame inertial type for $K$.
In particular we have a lowest alcove presentation $(s_\cJ, \mu)\in\un{W}\times X^*(\un{T})$ for it, where $\mu+\eta$ is $1$-generic Fontaine--Laffaille and the characters $\{\chi_i\mid 1\leq i\leq n\}$ appearing in the decomposition (\ref{eq:def:type}) are pairwise distinct.
Recall the set $\mathbf{n}_\cJ$ from (\ref{equ: general index}) equipped with a right action of $\un{W}\rtimes\Z/f$.
We also recall the definition of $s_{\tau}$ from Definition~\ref{defn:tau}.
In this section, we prove in Lemma~\ref{lem: bijection tau} that there exists a bijection between the set of subsets $I_\cJ\subseteq\mathbf{n}_\cJ$ satisfying $I_\cJ\cdot(s_\cJ^{-1},1)=I_\cJ$, with the set of sub inertial types $\tau_1\subseteq\tau$ over $K$.

\vspace{2mm}

Given the inertial type $\tau$, we write $\tld{X}_\tau$ for the set $\{\chi_i\mid 1\leq i\leq n\}$ where the characters $\chi_i$ are defined in (\ref{eq:def:type}).
We fix a bijection $\tld{X}_\tau\xrightarrow{\sim}\mathbf{n}$ by sending $\chi_i$ to $i$. It follows from Definition~\ref{defn:tau} that $\chi_i^{p^f}=\chi_{s_{\tau}^{-1}(i)}$ for each $i\in\mathbf{n}$, and thus the bijection
$$\tld{X}_\tau\rightarrow \tld{X}_\tau:~\chi\mapsto \chi^{p^f}$$
corresponds to the permutation $s_{\tau}^{-1}$ on $\mathbf{n}$, under the fixed bijection $\tld{X}_\tau\xrightarrow{\sim}\mathbf{n}$. Let $\tau_1\subseteq \tau$ be a sub $I_K$-representation, then there exists a subset $\tld{X}_{\tau_1}\subseteq\tld{X}_\tau$ such that $\tau_1\cong\bigoplus_{\chi\in \tld{X}_{\tau_1}}\chi$. We notice that $\tau_1\subseteq\tau$ is a sub inertial type over $K$ if and only if $\tau_1^{p^f}\defeq \bigoplus_{\chi\in \tld{X}_{\tau_1}}\chi^{p^f}\cong \tau_1$, if and only if $\tld{X}_{\tau_1}$ corresponds to a union of orbits of $s_{\tau}$ under the bijection $\tld{X}_\tau\xrightarrow{\sim}\mathbf{n}$. We write $ X_\tau\defeq \mathbf{n}/s_{\tau}$ for the set of orbits of $s_{\tau}$, and then write $ X_{\tau_1}$ for the image of $\tld{X}_{\tau_1}$ under $\tld{X}_{\tau_1}\hookrightarrow \tld{X}_\tau\twoheadrightarrow  X_\tau$. Hence we obtain a natural bijection between the set of sub inertial types over $K$ and the power set of $ X_\tau$, given by $\tau_1\mapsto  X_{\tau_1}$. Note that we also understand $ X_\tau$  as a subset of the power set of $\tld{X}_\tau$, by viewing each $\Lambda\in  X_\tau$ as a single orbit inside $\tld{X}_\tau\xrightarrow{\sim}\mathbf{n}$ under the action of $s_{\tau}$.

\vspace{2mm}

We consider (in analogy with $\mathbf{n}_\cJ$) the set $\mathbf{n}_{\cJ'}\defeq\mathbf{n}\times\cJ'$ equipped with a right action of $\un{W}^r\rtimes\Z/f'$. We use the notation $I_\cJ=(I_j)_{j\in\cJ}$ for a subset of $\mathbf{n}_\cJ$ with each $I_j\subseteq\mathbf{n}$, and similarly the notation $I_{\cJ'}=(I_{j'})_{j'\in\cJ'}$ for a subset of $\mathbf{n}_{\cJ'}$. We write $s_{\cJ'}\in\un{W}^r$ for the image of $s_\cJ$ under the diagonal embedding.

Given a sub $I_K$-representation $\tau_1\subseteq\tau$, we obtain a subset $\tld{X}_{\tau_1}\subseteq\tld{X}_\tau$ such that $\tau_1\cong\bigoplus_{\chi\in\tld{X}_{\tau_1}}\chi$. To $\tau_1$ one can attach the set $I_{f'-1}\subseteq \mathbf{n}$ by the condition $$\tld{X}_{\tau_1}=\{\chi_i\mid i\in I_{f'-1}\}$$ and then define $I_{j'}\defeq (s'_{\mathrm{or}, j'})^{-1}(I_{f'-1})$ for each $j'\in\cJ'$. We observe that
$$I_{j'-1}=(s'_{\mathrm{or}, j'-1})^{-1} s'_{\mathrm{or}, j'}(I_{j'})=s_{j'}(I_{j'})$$
for each $j'\in\cJ'$, so that we have $I_{\cJ'}\cdot(s_{\cJ'}^{-1},1)=I_{\cJ'}$, where $I_{\cJ'}=(I_{j'})_{j'\in\cJ'}$.

Hence, the associations $\tau_1\mapsto I_{f'-1}\mapsto I_{\cJ'}$ gives rise to bijections between the following three sets:
\begin{equation}\label{equ: bijections, types}
\{\mbox{sub $I_K$-representations of $\tau$}\}\stackrel{\sim}{\longleftrightarrow}\{\mbox{subsets of }\mathbf{n}\}\stackrel{\sim}{\longleftrightarrow}\{I_{\cJ'}\subseteq\mathbf{n}_{\cJ'}\mid I_{\cJ'}\cdot(s_{\cJ'}^{-1},1)=I_{\cJ'}\}.
\end{equation}
Moreover, these bijections are compatible with the action of $s_{\tau}$ on $\mathbf{n}$. More precisely, we have the following.

\begin{lemma}\label{lem: bijection tau}
The maps in (\ref{equ: bijections, types}) induce bijections between the following three sets:
\begin{enumerate}[label=(\roman*)]
\item
\label{it: bijection tau:i}
the set of sub inertial types $\tau_1\subseteq\tau$ over $K$;
\item
\label{it: bijection tau:ii}
the set of subsets $I\subseteq\mathbf{n}$ satisfying $s_\tau(I)=I$;
\item
\label{it: bijection tau:iii}
the set of subsets $I_\cJ\subseteq\mathbf{n}_\cJ$ satisfying $I_\cJ\cdot(s_\cJ^{-1},1)=I_\cJ$.
\end{enumerate}
\end{lemma}

\begin{proof}
Since the maps between~\ref{it: bijection tau:i},~\ref{it: bijection tau:ii} and~\ref{it: bijection tau:iii} are induced from the ones in equation (\ref{equ: bijections, types}), we only need to show that the extra conditions are compatible. The bijection between~\ref{it: bijection tau:i} to~\ref{it: bijection tau:ii} follows directly from the discussion at the beginning of this section. %
The key observation for the bijection between~\ref{it: bijection tau:ii} and~\ref{it: bijection tau:iii} is that $s_\tau (I)=I$ if and only if $I_{\cJ'}$ is $f$-periodic (using (\ref{primeorient})). Then we finish the proof by noting that there is a natural bijection between the set of subsets $I_\cJ\subseteq\mathbf{n}_\cJ$ satisfying $I_\cJ\cdot(s_\cJ^{-1},1)=I_\cJ$, and the set of $f$-periodic subsets $I_{\cJ'}\subseteq\mathbf{n}_{\cJ'}$ satisfying $I_{\cJ'}\cdot(s_{\cJ'}^{-1},1)=I_{\cJ'}$.
\end{proof}

Given a sub inertial type $\tau_1\subseteq \tau$ over $K$, we write $I_\cJ^{\tau_1}$ for the subset of $\mathbf{n}_\cJ$ corresponding to $\tau_1$ via Lemma~\ref{lem: bijection tau}.

\subsection{Extremal shape}\label{sub:ext:shape}
In this section, we use the results from \S\,\ref{subsec:RTandSW} and \S\,\ref{sub:dd} to prove a comparison result, Proposition~\ref{prop: inv and Frob of Weil}.
Let $R$ be a Noetherian $\cO$-algebra.

\vspace{2mm}
Recall from \cite[\S\,5.2 and \S\,5.3]{MLM} the $\cO$-schemes of finite type $\tld{U}(t_\eta,\leq\!\eta) \subseteq \tld{U}^{[0,n-1]}(t_\eta)$. By \cite[Proposition 3.2.8, \S\,5.2 before Definition 5.2.4]{MLM} an element of $\tld{U}(t_\eta, \leq\!\eta)(R)$ is a collection $A=(A^{(j)})_{j\in\cJ}$ of $n\times n$ matrices with entries in $R[v+p]$ such that for each $j\in\cJ$
\begin{equation}
\label{eq:expl:matrix}
A_{ik}^{(j)}=v^{\delta_{i>k}}\Bigg(\delta_{i\geq k}\sum_{\ell=n-i}^{n-k-\delta_{i\neq k}}c_{ik,\ell}^{(j)}(v+p)^{\ell}\Bigg)
\end{equation}
for all $1\leq i,k\leq n$ and for all $n-i\leq \ell\leq n-k-\delta_{i\neq k}$, with moreover $c_{kk,n-k}^{(j)}\in R^\times$ for all $1\leq k\leq n$.
For each $(k,j)\in\mathbf{n}_\cJ$, we define
$$\varphi_{k,j}(A)\defeq\frac{1}{(n-k)!}\frac{d^{(n-k)}(A^{(j)}_{kk})}{dv^{(n-k)}}=c_{kk,n-k}^{(j)}
$$
for each $A\in \tld{U}(t_\eta, \leq\!\eta)(R)$ and each Noetherian $\cO$-algebra $R$, which gives a morphism $$\varphi_{k,j}:~\tld{U}(t_\eta, \leq\!\eta)\twoheadrightarrow \bG_{m,\cO}.$$ It is obvious that if $A\in \tld{U}(t_\eta, \leq\!\eta)(R)$, then
\begin{equation}\label{eq: extremal shape}
A^{(j)}_{kk}|_{v=0}=p^{n-k}\varphi_{k,j}(A)\in p^{n-k} R^\times
\end{equation}
for each $(k,j)\in\mathbf{n}_\cJ$.

Using equation~(\ref{eq:expl:matrix}), it is not difficult to see that there is a natural isomorphism $\tld{S}^\circ_\F(t_\eta) \simeq \tld{U}(t_\eta, \leq\!\eta)_{\F}$ where $\tld{U}(t_\eta, \leq\!\eta)_\F$ denotes the special fibre of $\tld{U}(t_\eta, \leq\!\eta)$, so that we have a closed immersion
\begin{equation}\label{eq: i t eta}
{i_{t_\eta}}:\ \tld{S}^\circ_\F(t_\eta)\simeq \tld{U}(t_\eta, \leq\!\eta)_{\F}\hookrightarrow \tld{U}(t_\eta, \leq\!\eta)
\end{equation}
where the latter is the canonical closed immersion of the special fibre.
We consider the natural projection to the $j$-th factor $\mathrm{Proj}_j:~\tld{S}^\circ_\F(t_\eta)\twoheadrightarrow \tld{S}^\circ_\F(t_{\eta_j})=\Iw_1\backslash\Iw_1 T t_{\eta_j}\Iw_1$ and define $\overline{\varphi}_{k,j}$ by the following composition
$$\tld{S}^\circ_\F(t_\eta)\xrightarrow{\mathrm{Proj}_j}\Iw_1\backslash\Iw_1 T t_{\eta_j}\Iw_1\cong T\times\Iw_1\backslash\Iw_1t_\eta\Iw_1\twoheadrightarrow T\xrightarrow{\varepsilon_k} \bG_{m,\F}$$
for each $(k,j)\in\mathbf{n}_\cJ$. Then it is clear that for each $(k,j)\in\mathbf{n}_\cJ$ we have
\begin{equation}\label{equ: reduction of diagonal}
\varphi_{k,j}\circ i_{t_\eta}=\iota\circ\overline{\varphi}_{k,j}
\end{equation}
where $\iota:\bG_{m,\F}\hookrightarrow\bG_{m,\cO}$ is the canonical closed immersion.

In what follows, given $\tld{u}_\cJ=u_\cJ t_\nu\in\tld{\un{W}}^\vee$, we write $r_{\tld{u}_\cJ}$ for the map
\begin{equation}\label{eq: r u tilde}
r_{\tld{u}_\cJ}\,:\,\tld{\cF\cL}_\cJ\rightarrow \tld{\Fl}_\cJ
\end{equation}
defined by right multiplication by $\tld{u}_\cJ$.

\begin{lemma}\label{lem: pull pack Frob eigenvalue prime}
Let $\tld{u}_\cJ=u_\cJ t_\eta\in\un{W}t_\eta$. Then we have $r_{\tld{u}_\cJ}(\cM_{u_\cJ^{-1}}^\circ)\subseteq \tld{S}^\circ_\F(t_\eta)$ and
\begin{equation*}
\overline{\varphi}_{k,j}\circ r_{\tld{u}_\cJ}=f_{S_{u_j^{-1},k},j}f_{S_{u_j^{-1},k+1},j}^{-1}:~\cM_{u_\cJ^{-1}}^\circ\twoheadrightarrow \bG_{m,\F}
\end{equation*}
for each  $(k,j)\in\mathbf{n}_\cJ$ \emph{(}with the convention $f_{S_{u_j^{-1},n+1},j}\defeq 1$\emph{)}.
\end{lemma}

\begin{proof}
It is clear that $r_{\tld{u}_\cJ}(\cM_{u_\cJ^{-1}}^\circ)\subseteq \tld{S}^\circ_\F(t_\eta)$ by definition. From Lemmas~\ref{lem: hypersurfaces are Schubert},~\ref{lem: explicit projection} and the definition of $r_{\tld{u}_\cJ}$ we have
\[
\xymatrix{
\mathcal{M}^\circ_{u^{-1}_\cJ}\ar@{->>}_-{\mathrm{Proj}_j}[r]\ar@{^{(}->}_{r_{\tld{u}_\cJ}}[d] \ar@/^1.5pc/^{f_{S_{u_j^{-1},k},j}f_{S_{u_j^{-1},k+1},j}^{-1}}[rrrr]&U\backslash UTw_0Uw_0 u^{-1}_j\ar^-{\sim}[r]\ar@{^{(}->}[d]&T\times U\backslash U w_0Uw_0u^{-1}_j\ar@{->>}[r]&T\ar@{=}[d]\ar_{\varepsilon_k}[r]&\bG_{m,\F}\ar@{=}[d]\\
\tld{S}^\circ_\F(t_\eta)\ar@{->>}^-{\mathrm{Proj}_j}[r]\ar@/_1pc/_{\overline{\varphi}_{k,j}}[rrrr]&\Iw_1\backslash\Iw_1 T t_{\eta_j}\Iw_1\ar^-{\sim}[r]&T\times\Iw_1\backslash\Iw_1t_{\eta_j}\Iw_1.2\ar@{->>}[r]&T\ar^{\varepsilon_k}[r]&\bG_{m,\F}
}
\]
which is easily checked to be commutative.
\end{proof}

We fix a tame inertial type together with a lowest alcove presentation: $\tau=\tau(s_\cJ,\lambda+\eta-s_\cJ(\eta))$ where $\lambda+\eta$ is $n$-generic Fontaine--Laffaille. Then for each sub inertial type $\tau_1\subseteq \tau$ we define
\begin{equation}\label{eq: varphi tau tau1}
\varphi_{\tau,\tau_1}\defeq \prod_{(k,j)\in I_\cJ^{\tau_1}}\varphi_{k,j}:~ \tld{U}(t_\eta, \leq\!\eta)\twoheadrightarrow\bG_{m,\cO}.
\end{equation}

\begin{lemma}\label{lem: pull pack Frob eigenvalue primeprime}
Let $\lambda+\eta$ be $n$-generic Fontaine--Laffaille
and let $\tau= \tau(s_\cJ,\lambda+\eta-s_{\cJ}(\eta))$ be a tame inertial type.
Then we have $t_{\lambda+\eta}\tld{w}^*(\tau)^{-1}=s_{\cJ}t_\eta$, and for each sub inertial type $\tau_1\subseteq\tau$
$$\varphi_{\tau,\tau_1}\circ i_{t_\eta}\circ r_{s_{\cJ}t_{\eta}}=\iota\circ f_{s_\cJ^{-1},I_\cJ^{\tau_1}} :~\cM_{s_\cJ^{-1}}^\circ\rightarrow\bG_{m,\cO}.$$
\end{lemma}

\begin{proof}
It is clear that we have $\tld{w}^*(\tau)t_{-\lambda-\eta}=t_{-\eta} s_{\cJ}^{-1}$. Now, the other identity follows directly from the definition of $f_{s_\cJ^{-1},I_\cJ^{\tau_1}}$ in \S\,\ref{sub:def:inv} and Lemma~\ref{lem: pull pack Frob eigenvalue prime}.
\end{proof}

From now on we assume that $\lambda+\eta$ is $(3n-1)$-generic Fontaine--Laffaille, in order to use the results of \S\,\ref{subsec:LMforEG}.
Consider the $p$-adic formal scheme $ \tld{Y}^{\leq \eta,\tau}(t_\eta)$ defined in \cite[Definition 5.2.4(2)]{MLM}.
We have an isomorphism
\begin{equation}\label{eq: pi tile tau}
\tld{Y}^{\leq \eta,\tau}(t_\eta)\xrightarrow{\sim} \tld{U}(t_\eta, \leq \eta)^{\wedge_p},
\end{equation}
which we denote $\tld{\pi}_\tau$ in what follows (and which lifts the restriction to $\tld{Y}^{\leq \eta,\tau}_\F(t_\eta)$ of the morphism $\tld{\pi}_{(s_\cJ,\mu)}: \tld{Y}^{[0,n-1],\tau}_\F \to \tld{\Fl}^{[0,n-1]}_{\cJ}$ defined in \S\,\ref{subsec:LMforEG}).
We let $\tld{Y}^{\leq \eta,\tau}$ be the closed $p$-adic formal algebraic substack of $\tld{Y}^{[0,n-1],\tau}$ characterized by the two itemized properties in \cite[p.~79]{MLM}.
By letting $\tld{Y}^{\leq \eta,\tau}_{t_\eta}\defeq \tld{Y}^{[0,n-1],\tau}_{t_\eta}\cap \tld{Y}^{\leq \eta,\tau}_{\F}$, we note that we have $\tld{Y}^{\leq \eta,\tau}_{\F}(t_\eta)=\tld{Y}^{\leq \eta,\tau}_{t_\eta}$ as $\tld{S}^\circ_\F(t_\eta)\simeq\tld{U}(t_\eta, \leq\!\eta)_{\F}$ and (\ref{eq: pi tile tau}) (cf. Remark~\ref{rmk:shape:strata} (\ref{it:open:substacks:2})).

\vspace{2mm}

Let $R$ be a $p$-adically complete Noetherian $\cO$-algebra and let $(\fM,\phi_\fM)\in \tld{Y}^{\leq \eta,\tau}(t_\eta)(R)$ be a $R$-point.
Assume $(\fM,\phi_\fM)$ admits a $t_\eta$-gauge basis $\beta$ (in the sense of \cite[Definition 5.2.6]{MLM}) and let  $C_{\fM,\beta}^{(j')}$, $A_{\fM,\beta}^{(j')}$ be the matrices associated to $(\fM,\phi_\fM)$ and $\beta$ as in  Definition~\ref{defn:eigenbasis} and equation (\ref{equ: explicit formula for the morphism}). (See \cite[Proposition 5.2.7]{MLM} for the existence.)
Then, by definition, the isomorphism $\tld{\pi}_\tau:~\tld{Y}^{\leq \eta,\tau}(t_\eta)\xrightarrow{\sim} \tld{U}(t_\eta, \leq \eta)^{\wedge_p}$ sends
$(\fM,\phi_\fM,\beta)$ to $(A_{\fM,\beta}^{(j)})_{j\in\cJ}$.
For each $\chi\in\tld{X}_\tau$ and each $j'\in\cJ'$, we write $\beta_\chi^{(j')}\in\beta^{(j')}$ for the element where $\Delta'$ acts by $\chi$, and if we write $\phi_\fM^{(j')}\Big(\phz^*\big(\beta^{(j'-1)}_\chi\big)\Big)$ as a linear combination with respect to $\beta^{(j')}$, then we set $C^{(j')}_{\fM,\beta, \chi}\in R[\![u']\!]$ as the coefficient of $\beta^{(j')}_\chi$. Hence, $\{C^{(j')}_{\fM,\beta,\chi}\}_{\chi\in\tld{X}_\tau}$ exhausts the diagonal entries of $C^{(j')}_{\fM,\beta}$. We define $A^{(j')}_{\fM,\beta,\chi}$ as the diagonal entry of $A^{(j')}_{\fM,\beta}$ which equals $C^{(j')}_{\fM,\beta,\chi}$, via the relation (\ref{equ: explicit formula for the morphism}). It follows from the definitions that
\begin{equation}\label{equ: relation between entries}
C^{(j')}_{\fM,\beta,\chi}=C^{(j'+f)}_{\fM,\beta,\chi^{p^f}}
\end{equation}
for each $\chi\in\tld{X}_\tau$ and each $j'\in\cJ'$.

Let $\tau_1\subseteq \tau$ be a sub inertial type over $K$, and set $r_\Lambda\defeq \#\Lambda$ for each $\Lambda\in X_\tau$. We define
\begin{equation}\label{equ: frob crys formula}
\phi_{\tau,\tau_1}((\fM,\phi_\fM,\beta))\defeq \prod_{\chi\in\tld{X}_{\tau_1}}\prod_{j'=0}^{f-1}C^{(j')}_{\fM,\beta,\chi}|_{u'=0}
\end{equation}
which gives a morphism of $p$-adic formal schemes
\begin{equation*}
\phi_{\tau,\tau_1}:~\tld{Y}^{\leq \eta,\tau}(t_\eta)\rightarrow\bA_\cO^{1,\wedge_p}.
\end{equation*}

\begin{lemma}\label{lem: Frob crys}
The morphism $\phi_{\tau,\tau_1}$ does not depend on the choice of the basis $\beta$.
\end{lemma}
\begin{proof}
We fix an arbitrary section $\theta:~ X_\tau\hookrightarrow\tld{X}_\tau$ of the quotient map $\tld{X}_\tau\twoheadrightarrow X_\tau$ (namely the choice of a character $\theta(\Lambda)\in \Lambda$ for each orbit $\Lambda\in X_\tau$). For each $\Lambda\in  X_\tau$, we deduce from (\ref{equ: relation between entries}) that
$$\prod_{\chi\in\Lambda}C^{(j')}_{\fM,\beta,\chi}=\prod_{k=0}^{r_\Lambda-1}C^{(j'+k f)}_{\fM,\beta,\theta(\Lambda)}$$
which implies that
\begin{equation*}
\prod_{\chi\in\Lambda}\prod_{j'=0}^{f-1}C^{(j')}_{\fM,\beta,\chi}=\prod_{j'=0}^{fr_\Lambda-1}C^{(j')}_{\fM,\beta,\theta(\Lambda)}.
\end{equation*}
Suppose $\beta^\prime$ is another choice of eigenbasis, then there exists $t_{\fM,\chi}^{(j')}\in R^\times$ such that
$$\beta_{\chi}^{\prime,(j')}=t_{\fM,\chi}^{(j')}\beta_{\chi}^{(j')}$$
for each $j'\in\cJ$ and $\chi\in\tld{X}_\tau$.
Note that $t_{\fM,\chi}^{(j')}=t_{\fM,\chi^{p^f}}^{(j'-f)}$ for each $\chi\in\tld{X}_\tau$ and each $j'\in\cJ'$. In particular, $(t_{\fM,\chi}^{(j')})_{j'\in\cJ'}$ is $fr_\Lambda$-periodic for each $\chi\in\Lambda\in  X_\tau$. Hence we deduce that
\begin{align*}
\prod_{\chi\in\Lambda}\prod_{j'=0}^{f-1}C^{(j')}_{\fM,\beta^\prime,\chi}& =\prod_{j'=0}^{fr_\Lambda-1}C^{(j')}_{\fM,\beta^\prime,\theta(\Lambda)}\\ &=\prod_{j'=0}^{fr_\Lambda-1}(t_{\fM,\theta(\Lambda)}^{(j')})^{-1}C^{(j')}_{\fM,\beta,\theta(\Lambda)}t_{\fM,\theta(\Lambda)}^{(j'-1)}\\ &=\prod_{j'=0}^{fr_\Lambda-1}C^{(j')}_{\fM,\beta,\theta(\Lambda)}=\prod_{\chi\in\Lambda}\prod_{j'=0}^{f-1}C^{(j')}_{\fM,\beta,\chi},
\end{align*}
which finishes the proof by taking product over $\Lambda\in X_{\tau_1}$.
\end{proof}

\vspace{2mm}

For each sub inertial type $\tau_1\subseteq\tau$ over $K$, we set
$$d_{\tau,\tau_1}\defeq \sum_{(k,j)\in I_\cJ^{\tau_1}}(n-k)\in\N$$
where $I_\cJ^{\tau_1}\subseteq\mathbf{n}_\cJ$ is the subset associate with $\tau_1$ via Lemma~\ref{lem: bijection tau}. We abuse the notation $\varphi_{\tau,\tau_1}$ for the associated morphism of $p$-adic formal schemes $\tld{U}^{\leq\eta}(t_\eta)^{\wedge_p}\twoheadrightarrow \bG_{m,\cO}^{\wedge_p}$ given by $p$-adic completion.

\begin{prop}\label{prop: inv and Frob of Weil}
We have
$$\phi_{\tau,\tau_1}=p^{d_{\tau,\tau_1}}\varphi_{\tau,\tau_1}\circ \tld{\pi}_\tau:~\tld{Y}^{\leq \eta,\tau}(t_\eta)\rightarrow \bA_\cO^{1,\wedge_p},$$
for each sub inertial type $\tau_1\subseteq\tau$ over $K$.
\end{prop}
\begin{proof}
Let $R$ be a $p$-adically complete Noetherian $\cO$-algebra and we check the equality for an arbitrary object $(\fM,\phi_\fM,\beta)\in \tld{Y}^{\leq \eta,\tau}(t_\eta)(R)$. Then it follows from (\ref{eq: extremal shape}) that
\begin{align*}
\phi_{\tau,\tau_1}((\fM,\phi_\fM,\beta)) &=\prod_{\chi\in\tld{X}_{\tau_1}}\prod_{j'=0}^{f-1}C^{(j')}_{\fM,\beta,\chi}|_{u'=0} \\ %
&=\prod_{(k,j)\in I_\cJ^{\tau_1}}p^{n-k}\varphi_{k,j}(A_{\fM,\beta}^{(j)}).
\end{align*}
Here, we emphasize that $A_{\fM,\beta}^{(j')}$ is $f$-periodic in $j'$ and thus we can write $A_{\fM,\beta}^{(j)}$ instead. Moreover, the second equality comes from Lemma~\ref{lem: bijection tau} together with the identity~(\ref{equ: explicit formula for the morphism}). Hence, we conclude that $$\phi_{\tau,\tau_1}((\fM,\phi_\fM,\beta))=p^{d_{\tau,\tau_1}}\varphi_{\tau,\tau_1}(A_{\fM,\beta}),$$ which finishes the proof.
\end{proof}

We end this subsection by summarizing the results discussed in \S\,\ref{sec:FL:SW} and \S\,\ref{sub:ext:shape}: for each sub inertial type $\tau_1\subseteq\tau=\tau(s_{\cJ},\mu)$ with $\mu=\lambda+\eta-s_{\cJ}(\eta)$ where $\lambda+\eta$ is $(3n-1)$-generic Fontaine--Laffaille, we have the following commutative diagram
\begin{equation*}
\xymatrix{
\tld{Y}^{[0,n-1],\tau}_{\F}\ar^{\tld{\pi}_{(s_{\cJ},\mu)}}[dd]\ar@/_3.0pc/_{r_{\tld{w}^*(\tau)}}[ddd]& \tld{Y}^{\leq \eta,\tau}_{\F}(t_\eta)\ar@{_{(}->}[l] \ar@{^{(}->}[r] &\tld{Y}^{\leq \eta,\tau}(t_\eta)\ar[dr]_{\tld{\pi}_\tau}^{\simeq} \ar@/^1.0pc/^{p^{-d_{\tau,\tau_1}}\phi_{\tau,\tau_1}}[drrr]  &&&\\
& \tld{Y}^{\leq \eta,\tau}_{t_\eta}\ar^{\tld{\pi}_{(s_{\cJ},\mu)}}[d]\ar[u]_{=} && \tld{U}(t_\eta, \leq\!\eta)^{\wedge_p}\ar@{_{(}->}[d]\ar^{\phz_{\tau,\tau_1}}[rr] &&\bG_{m,\cO}^{\wedge_p}\ar@{_{(}->}[d]\\
\tld{\Fl}_\cJ^{[0,n-1]}\ar^{\cdot \tld{w}^\ast(\tau)}[d]&\tld{S}^\circ_\F(t_\eta)\ar@{_{(}->}[l]\ar@{^{(}->}[rr]^{i_{t_{\eta}}} &&\tld{U}(t_\eta, \leq\!\eta)\ar^{\phz_{\tau,\tau_1}}[rr]&&\bG_{m,\cO}\\
\tld{\Fl}_\cJ^{[1-n,n-1]} v^{\lambda+\eta}& && &&\bG_{m,\F}\ar@{^{(}->}[u]_{\iota}\\
\tld{\cF\cL}_{\cJ}\ar_{r_{\lambda+\eta}}[u]\ar@/^4.0pc/^{r_{s_{\cJ}t_{\eta}}}[uu] &\cM_{s_\cJ^{-1}}^\circ\ar@{_{(}->}[l]\ar_{r_{s_{\cJ}t_\eta}}[uu] \ar@/_1.0pc/_{f_{s_{\cJ}^{-1},I_{\cJ}^{\tau_1}}}[urrrr] && &&
}
\end{equation*}
where
\begin{itemize}
\item $r_{\lambda+\eta}$ and $r_{\tld{w}^*(\tau)}$ are described in Proposition~\ref{prop:rel:cat};
\item $r_{s_{\cJ}t_\eta}$ is defined in (\ref{eq: r u tilde}), and $\cdot\tld{w}^*(\tau)$ is the right multiplication by $\tld{w}^*(\tau)$;
\item $\tld{\pi}_\tau$ is described in (\ref{eq: pi tile tau}), and $\tld{\pi}_{(s_{\cJ},\mu)}$ is defined in (\ref{eq: pi tilde s mu});
\item $i_{t_\eta}$ is defined in (\ref{eq: i t eta});
\item $\phi_{\tau,\tau_1}$ is defined in (\ref{equ: frob crys formula}) and $\varphi_{\tau,\tau_1}$ is in (\ref{eq: varphi tau tau1});
\item $\tld{Y}^{\leq \eta,\tau}_{\F}(t_\eta)=\tld{Y}^{\leq \eta,\tau}_{t_\eta}$ is illustrated in the paragraph of (\ref{eq: pi tile tau});
\item commutativity of the diagram follows from Lemma~\ref{lem: pull pack Frob eigenvalue prime}, Proposition~\ref{prop: inv and Frob of Weil}, and the paragraph of (\ref{eq: pi tile tau}).
\end{itemize}

\subsection{Invariant functions and Weil--Deligne representations}\label{sub:inv WD}
In this section, we apply the results of \S\,\ref{sub:ext:shape} to prove Theorem~\ref{thm:Feval-2}, which relates invariant function with some normalized mod-$p$ reduction of (product of) Frobenius eigenvalues of Weil--Deligne representations.

We denote by $\Rep^n_{E}(W_K)$ the groupoid of Weil--Deligne representations of $W_K$ over $n$-dimensional vector spaces over $E$. (In particular, if $\varsigma\in \Rep^n_{E}(W_K)$ then the restriction $\varsigma|_{I_K}$ is by definition an inertial type.) Let $g_p\in W_K$ be an element whose image in $W_{K}^{\mathrm{ab}}$ corresponds to $p$ via $\mathrm{Art}_{K}: K^\times \stackrel{\sim}{\ra}W_{K}^{\mathrm{ab}}$.
Then for a given $\varsigma\in \Rep^n_{E}(W_K)$, $g_p$ acts on $\wedge^n(\varsigma)$ by a scalar $\alpha_{\varsigma}\in E^\times$, which is often called the \emph{Frobenius eigenvalue} of $\varsigma$.

\vspace{2mm}

As at the beginning of \S\,\ref{sub:dd}, let $\tau$ be a tame inertial type with lowest alcove presentation $(s_\cJ,\mu)$ such that $\mu+\eta$ is $1$-generic Fontaine--Laffaille.
We have a covariant functor
\[
D^\tau:\ Y^{[0,n-1],\tau}(\cO)\ra \Rep^n_{E}(W_K)
\]
whose image lands in representations $\varsigma^\vee$ such that $\varsigma|_{I_K}\cong \tau$, and which is defined as follows.
If $(\fM,\phi_\fM)\in Y^{[0,n-1],\tau}(\cO)$, then $D(\fM)\defeq \big(\fM/u'\fM\big)\otimes_{\cO}E$ is a free $K'\otimes_{\Qp}E$-module of rank $n$, endowed with a $\phz$-semilinear and $E$-linear automorphism $\phi_{D(\fM)}\defeq \big(\phi_{\fM}\circ\phz^*\pmod{u'}\big)\otimes_{\cO}\Id_E$, and with a $K'$-semilinear and $E$-linear action of $\Delta=\Gal(L'/K)$, compatible with $\phi_{D(\fM)}$.
We thus define an action of $W_K$ on $D(\fM)$ by the following rule: $g\in W_K$ acts by the automorphism $\ovl{g}\phi_{D(\fM)}^{-\val(g)}$ where $\ovl{g}$ is the image of $g\in W_K$ in $\Delta$ (via the natural surjection $W_K\onto W_K/W_{L'}= \Delta$), and $\val(g)\in\Z$ is defined via $g\equiv \phz^{\val(g)}$ modulo $I_K$. Note that $W_K/I_K$ is generated by $\phz^f$.
This action of $W_K$ preserves the $E$-linear decomposition $D(\fM)=\oplus_{j'\in \cJ'}D(\fM)^{(j')}$ and it is easily seen (cf.~\cite[Lemme\ 2.2.1.2]{BM}) that the  $W_K$-representations $D(\fM)^{(j')}$ are all isomorphic.
We define $D^\tau(\fM)$ as the isomorphism class of $D(\fM)^{(j')}$ and, by construction, we have $D^\tau(\fM)|_{I_K}\cong (\tau^{\vee})\otimes_{\cO}E$.

\begin{lemma}\label{lem: Frob formula}
Let $(\fM,\phi_\fM,\beta)\in \tld{Y}^{\leq \eta,\tau}(t_\eta)(\cO)$ and $\varsigma_1$ be a $W_K$-representation satisfying $\varsigma_1^\vee\hookrightarrow D^\tau(\fM)$ and $\varsigma_1|_{I_K}\cong \tau_1$. Then we have
$$\al_{\varsigma_1}^{-1}=\alpha_{\varsigma_1^\vee}=\phi_{\tau,\tau_1}((\fM,\phi_\fM,\beta)).$$
\end{lemma}

\begin{proof}
It is clear that $\al_{\varsigma_1}^{-1}=\alpha_{\varsigma_1^\vee}$. To compute the Weil--Deligne representation from the given Breuil--Kisin module $\fM$, we choose $j'=f'-1$. As $\val(g_p)=-f$, $g_p$ acts on $D(\fM)^{(f'-1)}$ by $\ovl{g}_p\phi_{D(\fM)}^{f}$. More precisely, if we write $\phi_{D(\fM)}^{(j')}:D(\fM)^{(j'-1)}\rightarrow D(\fM)^{(j')}$ for the induced map from $\phi_{D(\fM)}$ via $D(\fM)=\oplus_{j'\in \cJ'}D(\fM)^{(j')}$ then the following diagram describes the action of $g_p$ on $D(\fM)^{(f'-1)}$:
\begin{equation*}
\xymatrix{
D(\fM)^{(f'-1)}\ar[rr]^{\phi_{D(\fM)}^{(0)}}&&D(\fM)^{(0)}\ar[rr]^{\phi_{D(\fM)}^{(1)}} &&D(\fM)^{(1)}\ar[rr]^{\phi_{D(\fM)}^{(2)}}&&\cdots\cdots\ar[rr]^{\phi_{D(\fM)}^{(f-1)}}&& D(\fM)^{(f-1)}.\ar@/^1.5pc/^{\ovl{g}_p}[llllllll]
}
\end{equation*}

By abuse of notation, we write $\beta^{(j')}$ for the basis of $D(\fM)^{(j')}$ induced from the $\beta^{(j')}$ of $\fM^{(j')}$. We also write $B_{\fM,\beta}^{(j')}$ for the matrix of $\phi_{D(\fM)}^{(j')}:D(\fM)^{(j'-1)}\rightarrow D(\fM)^{(j')}$ with respect to $\beta^{(j'-1)}$ and $\beta^{(j')}$. It is easy to see that $B_{\fM,\beta}^{(j')}=C_{\fM,\beta}^{(j')}|_{u'=0}$, so that $B_{\fM,\beta}^{(j')}$ is a diagonal matrix from (\ref{equ: explicit formula for the morphism}), which implies that $\{\beta^{(f'-1)}_{\chi}\mid\chi\in \tld{X}_{\tau_1}\}$ forms a basis for $\varsigma_1$. We further let $B_{\fM,\beta,\chi}^{(j')}\defeq C_{\fM,\beta,\chi}^{(j')}|_{u'=0}$ which is a diagonal entry of $B_{\fM,\beta}^{(j')}$.

We set $r_1\defeq \Dim_E\varsigma_1$ and note that $g_p$ acts trivially on $L=K(\pi)$, due to our choice of $g_p$, so that we have $\omega_f(g_p)=1$. Hence, from the description of $\ovl{g}_p\phi_{D(\fM)}^{f}$, we have
$$\alpha_{\varsigma_1^\vee}=\wedge^{r_1}\varsigma_1^\vee(g_p)= \prod_{\chi\in\tld{X}_{\tau_1}}\prod_{j'=0}^{f-1}B^{(j')}_{\fM,\beta,\chi} =\prod_{\chi\in\tld{X}_{\tau_1}}\prod_{j'=0}^{f-1}C^{(j')}_{\fM,\beta,\chi}|_{u'=0} =\phi_{\tau,\tau_1}((\fM,\phi_\fM,\beta)),$$
which finishes the proof.
\end{proof}

Let $\Rep_{\cO}^{[0,n-1],\tau}(G_K)$ denote the groupoid of $\cO$-lattices in potentially crystalline representations of $G_K$ over $n$-dimensional $E$-vector spaces, becoming crystalline over $L'$ with Hodge--Tate weights in $[0,n-1]$ and inertial type $\tau$.
We have a \emph{contravariant} functor of groupoids
\[
\Rep_{\cO}^{[0,n-1],\tau}(G_K)\ra Y^{[0,n-1],\tau}(\cO)
\]
constructed in \cite[Proposition 7.2.3]{MLM} (compatible with the results in \cite[\S\,1.3]{KisinFcrys}), noting that $\Rep_{\cO}^{[0,n-1],\tau}(G_K)=\cX_{n}^{[0,n-1],\tau}(\cO)$. We also note that if $\fM_\rho$ is the Kisin module associated to an $\cO$-lattice in a potentially crystalline representation $\rho$ then we have $\big(\fM_\rho/u'\fM_\rho\big)\otimes_\cO E=D^*_{\mathrm{cris}}(\rho|_{G_{L'}})$ (see \cite[Proposition 2.1.5]{KisinFcrys} for a version when $E=\Qp$). Here, we write $D_{\mathrm{cris}}$ for the covariant functor $D_{\mathrm{st}}^{L'}$ in \cite[Proposition~2.9]{savitt-CDT}.

Finally, recall from \S\,\ref{sec:notation:GT} the covariant functor $\mathrm{WD}: \Rep_{\cO}^{[0,n-1],\tau}(G_K)\ra \Rep^n_{E}(W_K)$ obtained from \cite[Appendix B.1]{CDT} (after extending to $E$ the coefficients of  the objects in $\Rep_{\cO}^{[0,n-1],\tau}(G_K)$).
We write $\mathrm{WD}^*$ for the composite of $\mathrm{WD}$ followed by taking the dual Weil--Deligne representation.

\begin{lemma}
\label{lem:frob:EV}
Let $\tau$ be a tame inertial type with lowest alcove presentation $(s_\cJ,\mu)$ such that $\mu+\eta$ is $1$-generic Fontaine--Laffaille. We have a commutative diagram of groupoids
\[
\xymatrix{
&Y^{[0,n-1],\tau}(\cO)\ar^{D^\tau}[r]
&\Rep^n_{E}(W_K)\\
&\Rep_{\cO}^{[0,n-1],\tau}(G_K)\ar[u]\ar_{\WD^*}[ur]&
}.
\]
\end{lemma}
\begin{proof}
The vertical functor sends a $\cO$-lattice in a potentially crystalline representation $\rho$ to the associated Breuil-Kisin module $\fM_\rho$. The result follows from $\big(\fM_\rho/u'\fM_\rho\big)\otimes_\cO E=D^*_{\mathrm{cris}}(\rho|_{G_{L'}})$ and keeping track of the descent data from $L'$ to $K$ (cf.~\cite[\S\,4.6]{EGstack}).
The genericity assumption on $\tau$ guarantees that the Galois representations under consideration are potentially \emph{crystalline} since the characters $\chi_i$ appearing in (\ref{eq:def:type}) are pairwise distinct.
\end{proof}

\vspace{2mm}

We consider a lift $\rho_{x,\lambda+\eta}^\circ\in \Rep_{\cO}^{[0,n-1],\tau}(G_K)$ of $\rhobar_{x,\lambda+\eta}$
to which we can associate a Weil--Deligne representation $\varsigma\defeq \WD(\rho^\circ_{x,\lambda+\eta}): W_K\rightarrow\mathrm{GL}_n(E)$ satisfying $\varsigma|_{I_K}\cong \tau$. %
Each sub inertial type $\tau_1\subseteq \tau$ for $K$ determines to a subrepresentation
$\varsigma_1\subseteq \varsigma$
satisfying $\varsigma_1|_{I_K}\cong \tau_1$. We consider $\al_{\varsigma_1}\in E^\times$ (the Frobenius eigenvalue of $\varsigma_1$ defined at the beginning of \S\,\ref{sub:inv WD}) which clearly depends on the choice of $\rho_{x,\lambda+\eta}^\circ$ and $\tau_1$.

\begin{thm}
\label{thm:Feval-2}
Let $\lambda+\eta$ be $(3n-1)$-generic Fontaine--Laffaille. Let $x\in\tld{\cF\cL}_\cJ(\F)$ and $s_\cJ\in\un{W}$ with $x\in\cM_{s_\cJ^{-1}}(\F)$. Then for each sub inertial type $\tau_1\subseteq\tau=\tau(s_\cJ,\lambda+\eta-s_\cJ(\eta))$ and each lift $\rho_{x,\lambda+\eta}^\circ$ of $\rhobar_{x,\lambda+\eta}$ as above, we have
$\val_p(\al_{\varsigma_1}^{-1})=d_{\tau,\tau_1}$ and
$$\frac{\al_{\varsigma_1}^{-1}}{p^{d_{\tau,\tau_1}}}\equiv f_{s_\cJ^{-1},I_\cJ^{\tau_1}}(x)\in \F^\times.$$
\end{thm}

\begin{proof}
We pick up an object $(\fM,\phi_\fM,\beta)\in \tld{Y}^{\leq \eta,\tau}(t_\eta)(\cO)$ whose image under $T^*_{dd}$ is isomorphic to $\rho_{x,\lambda+\eta}^\circ|_{G_{K_{\infty}}}$. Note that we can recover $\varsigma^\vee=\WD^*(\rho_{x,\lambda+\eta})$ from $(\fM,\phi_\fM,\beta)$ via Lemma~\ref{lem:frob:EV}. The result follows immediately from (\ref{equ: reduction of diagonal}), Lemma~\ref{lem: pull pack Frob eigenvalue primeprime}, Proposition~\ref{prop: inv and Frob of Weil} and Lemma~\ref{lem: Frob formula}.
\end{proof}

\clearpage{}%
\clearpage{}%
\section{Local--global compatibility}
\label{sec:LGC}

In this section, we use a set of \emph{normalized Hecke operators} and results from previous sections to prove our main result on local-global compatibility (see Theorem~\ref{thm:main:LGC:prime}).
\subsection{Hecke actions}
\label{subsub:Prel:HO}
In this section, we introduce the set of \emph{normalized $U_p$-operators} which will be eventually related to the set of invariant functions using suitably normalized local Langlands correspondence.
For this section, let $R$ be a commutative ring.
For a topological group $G$, we denote by $\mathrm{Rep}^{\rm{sm}}_R(G)$ the abelian category of smooth (i.e.~locally constant) representations of $G$ over $R$.

\subsubsection{Compact induction and Hecke algebras}\label{subsub:ind:Hecke}

Given a closed subgroup $H$ of a topological group~$G$ and an object $(\sigma,V)\in \mathrm{Rep}^{\rm{sm}}_R(H)$, we define the smooth induction
$$\Ind_H^{G}(\sigma)\defeq \{f: G\rightarrow V\mid f\textnormal{ is locally constant and } f(hg)=\sigma(h)f(g)\textnormal{ for all } h\in H,g\in G\}$$
and smooth compact induction $\ind_H^{G}(\sigma)$ which is the subspace of $\Ind_H^{G}(\sigma)$ consisting of functions with compact support modulo $H$.
(We often omit $V$ from the notation for simplicity.)
This induces two exact functors
$$\Ind_H^{G},\ind_H^{G}:\mathrm{Rep}^{\rm{sm}}_R(H)\rightarrow \mathrm{Rep}^{\rm{sm}}_R(G)$$
which are called the induction and the compact induction, respectively.
These are the right and left adjoints, respectively, of the restriction functor $\cdot\,\,|_{H}:\mathrm{Rep}^{\rm{sm}}_R(G)\rightarrow \mathrm{Rep}^{\rm{sm}}_R(H)$ by \emph{Frobenius reciprocity}.
Note that $\ind_H^{G}=\Ind_H^{G}$ %
when $H\backslash G$ is compact.
Let $v \in V$ and $g \in G$.
There is a unique map $[H,g\mapsto v]: G \ra V$ supported on $Hg$ such that $[H,g\mapsto v](hg) = \sigma(h)v$ for all $h\in H$, $g\in G$.
If $H$ is open and compact, then $[H,g\mapsto v] \in \ind_H^{G}(\sigma)$ and elements of this form span $\ind_H^{G}(\sigma)$.

Given two closed subgroups $H_1,H_2\subseteq G$ and $\sigma_{i}\in \mathrm{Rep}^{\rm{sm}}_R(H_i)$ for $i=1,2$, we consider the $R$-module
$$\cH_{H_2,H_1}^G(\sigma_{2},\sigma_{1})\defeq \Hom_G\left(\ind_{H_2}^G\sigma_{2},\ind_{H_1}^G\sigma_{1}\right)\cong \Hom_{H_2}(\sigma_{2},(\ind_{H_1}^G\sigma_{1})|_{H_2}).$$
The map
\begin{align}
\nonumber \Hom_{H_2}(\sigma_{2},(\ind_{H_1}^G\sigma_{1})|_{H_2}) &\ra \{\textnormal{loc.~const. }\phi: G \ra \Hom_R(\sigma_{2},\sigma_{1}) \mid \\
\label{eqn:heckemodule} & \hspace{1cm}\phi(h_1gh_2)=\sigma_{1}(h_1)\phi(g)\sigma_{2}(h_2) \textnormal{ for all } h_1\in H_1,g\in G,h_2\in H_2\}
\\
\nonumber \psi &\mapsto (g\mapsto \psi(-)(g))
\end{align}
is an isomorphism, giving another description of $\cH_{H_2,H_1}^G(\sigma_{2},\sigma_{1})$.
For a closed subgroup $H \subset G$, the space $\cH_{H,H}^G(\sigma,\sigma) = \End_G(\ind_H^G \sigma)$, which we simply denote by $\cH_H^G(\sigma)$, is naturally an $R$-algebra via composition.

Given $\phz\in \Hom_R(\sigma_2,\sigma_1)$, there is at most one function $f: G \ra \Hom_R(\sigma_2,\sigma_1)$ supported on the double coset $H_1 g H_2$ such that $f(h_1 g h_2) = \sigma_1(h_1) \phz \sigma_2(h_2)$ for all $h_1\in H_1$ and $h_2\in H_2$.
If this function exists, we denote it by $[H_1,g\mapsto \phz,H_2]$.
If $H_1$ and $H_2$ are furthermore compact and open, then $[H_1,g \mapsto \phz,H_2] \in \cH_{H_2,H_1}^G(\sigma_{2},\sigma_{1})$ (if it exists) using the identification~\eqref{eqn:heckemodule}.
Under this identification, we have
\begin{equation}\label{eq: formula of Hecke basis}
[H_1,g\mapsto \phz,H_2]([H_2,g'\mapsto v]) = \sum_{h_2 \in I} [H_1,gh_2g'\mapsto \phz(\sigma_2(h_2)(v))]
\end{equation}
where $H_1gH_2 = \bigsqcup_{h_2\in I} H_1gh_2$.

\subsubsection{$U_p$-operators}
\label{subsub:Up:act}

Recall that $K$ denotes a finite unramified extension of $\Qp$ of degree $f$, with ring of integers $\cO_K$ and residue field $k$. We write $q=p^f=\# k$.

We define the compact open subgroups $\mathbf{K}\defeq \GL_n(\cO_K)$ and $\mathbf{K}_1\defeq \mathrm{Ker}(\GL_n(\cO_K)\twoheadrightarrow \GL_n(k))$ of  $\GL_n(K)$ (equipped with the $p$-adic topology). We recall from \S\,\ref{intro:LAG} the pairing $X^\ast(T)\times X_\ast(T)\rightarrow\Z$ and the basis $(\varepsilon_1,\dots,\varepsilon_n)$ of $X^\ast(T)$.
Then $\{\al_1,\dots,\al_{n-1},\varepsilon_n\}$, where $\al_i=\varepsilon_i-\varepsilon_{i+1}$ for $1\leq i\leq n-1$, forms a basis of $X^\ast(T)$, and we denote by $\{\omega^{(1)},\dots,\omega^{(n)}\}$ the corresponding dual basis of $X_\ast(T)$.
Then $\omega^{(i)}=\sum_{k=1}^i\varepsilon_k^\vee$ for each $1\leq i\leq n$, where $\{\varepsilon_1^\vee,\dots,\varepsilon_n^\vee\}$ is the dual basis of $\{\varepsilon_1,\dots,\varepsilon_n\}$.
We abuse $\omega^{(i)}$ for the induced map $K^\times\rightarrow T(K)\subseteq \GL_n(K)$. (For instance, $\omega^{(i)}(p)\in T(K)$).

Let $L \subseteq\GL_n$ be the standard Levi subgroup with diagonal blocks $\GL_i\times\GL_{n-i}$ and consider the standard maximal parabolic subgroup $P^+\defeq L B$ with its opposite parabolic subgroup $P^-\defeq L w_0 B w_0$.
We denote the unipotent radical of $P^+$ (resp.~of $P^-$) by $N^+$ (resp.~by $N^-$).
Let $\mathbf{P}^+\subset \mathbf{K}$ and $\mathbf{P}^- \subset \mathbf{K}$ be the preimage of $P^+(k)$ and $P^-(k)$, respectively, under the reduction map.
Let $\mathbf{P}^+_1\subset \mathbf{K}$ (resp.~$\mathbf{P}^-_1 \subset \mathbf{K}$) be the unique maximal pro-$p$ subgroup of $\mathbf{P}^+$ (resp. of~$\mathbf{P}^-$).
Then the quotients $\mathbf{P}^+/\mathbf{P}^+_1$ and $\mathbf{P}^-/\mathbf{P}^-_1$ are naturally identified with $L(k)$.

Let $\sigma$ be an $R[L(k)]$-module which is a smooth $\mathbf{P}^+$-representation (resp. $\mathbf{P}^{-}$-representation) over $R$ by inflation.
Observing that $\omega^{(i)}(p)$ centralizes $L(\cO_K)$ and $\omega^{(i)}(p)\mathbf{P}^+ \omega^{(i)}(p)^{-1}=\mathbf{P}^{-}$, we have three elements
\begin{align*}
\mathbf{U}^{(i)}&\defeq [\mathbf{P}^+, \omega^{(i)}(p)^{-1} \mapsto \mathrm{id}_\sigma, \mathbf{P}^+] \in \cH_{\mathbf{P}^+}^{\GL_n(K)}(\sigma);\\
t_i&\defeq[\mathbf{P}^+, \omega^{(i)}(p)^{-1} \mapsto \mathrm{id}_\sigma,\mathbf{P}^-] \in \cH_{\mathbf{P}^-,\mathbf{P}^+}^{\GL_n(K)}(\sigma,\sigma);\\
S_\sigma &\defeq  [\mathbf{P}^-, 1 \mapsto \mathrm{id}_\sigma, \mathbf{P}^+] \in \cH_{\mathbf{P}^+,\mathbf{P}^-}^{\GL_n(K)}(\sigma,\sigma).
\end{align*}
Note that $\mathbf{P}^+ \omega^{(i)}(p)^{-1}\mathbf{P}^+= \omega^{(i)}(p)^{-1}\mathbf{P}^-\mathbf{P}^+=\big(\mathbf{P}^+ \omega^{(i)}(p)^{-1}\mathbf{P}^-\big)\big(\mathbf{P}^-\mathbf{P}^+\big)$.

\begin{lemma}\label{lemma:Upfactor}
Let $\omega^{(i)}(p)$, $\mathbf{P}^+$, $\mathbf{P}^-$, and $\sigma$ be as above.
Then
\[
\mathbf{U}^{(i)} = t_i \circ S_\sigma.
\]
\end{lemma}
\begin{proof}
By Frobenius reciprocity, it suffices to check that
\[
\mathbf{U}^{(i)}([\mathbf{P}^+,1\mapsto v]) = (t_i \circ S_\sigma)([\mathbf{P}^+,1\mapsto v])
\]
for all $v\in \sigma$.
Let $I$ be a set of representatives for $\mathbf{K}_1 \backslash \mathbf{P}^+_1$. Then we have $\mathbf{P}^+ \omega^{(i)}(p)^{-1}\mathbf{P}^+=\bigsqcup_{h\in I}\mathbf{P}^+ \omega^{(i)}(p)^{-1}h$
which implies that by (\ref{eq: formula of Hecke basis})
\begin{equation}\label{equ: exp Up}
\mathbf{U}^{(i)}([\mathbf{P}^+,1\mapsto v]) = \sum_{h\in I} [\mathbf{P}^+, \omega^{(i)}(p)^{-1}h\mapsto v].
\end{equation}
On the other hand, we also have $\mathbf{P}^-\mathbf{P}^+=\bigsqcup_{h\in I}\mathbf{P}^-h$, which implies that by (\ref{eq: formula of Hecke basis})
$$(t_i \circ S_\sigma)([\mathbf{P}^+,1\mapsto v])= t_i\big(\sum_{h\in I}[\mathbf{P}^-,h\mapsto v]\big)= \sum_{h\in I}[\mathbf{P}^+, \omega^{(i)}(p)^{-1}h\mapsto v].$$
(The last equality follows from $\mathbf{P}^+ \omega^{(i)}(p)^{-1}\mathbf{P}^-=\mathbf{P}^+ \omega^{(i)}(p)^{-1}$.) This completes the proof.
\end{proof}

Let $\tau_1$ and $\tau_2$ be tame inertial types of dimension $i$ and $n-i$, respectively, such that $\tau = \tau_1 \oplus \tau_2$ is a regular tame inertial type.
In the context when $\sigma = \sigma(\tau_1) \otimes_E \sigma(\tau_2)$, we denote $\mathbf{U}^{(i)}$ by $\mathbf{U}^{\tau_1}_\tau$.
Then the $\mathbf{K}$-type $\sigma(\tau)$ is isomorphic to $\Ind_{\mathbf{P}^+}^{\mathbf{K}} \sigma(\tau_1) \otimes_E \sigma(\tau_2)$.
The claimed isomorphism follows from \cite[Proposition 2.5.5]{MLM}, noting that by \cite[Lemma~4.7]{herzig-duke} the induction $\Ind_{\mathbf{P}^+}^{\mathbf{K}} \sigma(\tau_1) \otimes \sigma(\tau_2)$ is isomorphic to the Deligne--Lusztig representation attached to $\tau$ by \cite[Propostion~9.2.1]{GHS} and latter is irreducible (cf.~\cite[Corollary~2.3.5]{LLL}) since $\tau$ is regular.
Let $\varsigma$ be a Frobenius-semisimple representation $W_K \ra \GL_n(E)$ with $\varsigma|_{I_K} \cong \tau$.
Then $\varsigma$ is a direct sum of representations $\varsigma_1 \oplus \varsigma_2$ with $\varsigma_i|_{I_K} \cong \tau_i$. Recall from the beginning of \S\,\ref{sub:inv WD} that $\al_{\varsigma_1}$ is the Frobenius eigenvalue of $\varsigma_1$. Recall the map $r_p$ from \cite[\S\,1.8]{CEGGPS}. The isomorphism $\cH_{\mathbf{P}^+}^{\GL_n(K)}(\sigma)\cong \cH_{\mathbf{K}}^{\GL_n(K)}(\sigma(\tau))$ identify $\mathbf{U}^{\tau_1}_\tau$ with an element in $\cH_{\mathbf{K}}^{\GL_n(K)}(\sigma(\tau))$.

\begin{prop} \label{prop:LLC}
Let $\tau$, $\varsigma$, and $\varsigma_1$ be as above.
The operator $\mathbf{U}^{\tau_1}_\tau$ acts on
\[
\Hom_{\mathbf{P}^+}(\sigma(\tau_1) \otimes_E \sigma(\tau_2),r_p^{-1}(\varsigma)|_{\mathbf{P}^+}) \cong \Hom_{\mathbf{K}}(\sigma(\tau),r_p^{-1}(\varsigma)|_{\mathbf{K}})
\]
by the scalar $q^{\frac{i(2n-i-1)}{2}}\al_{\varsigma_1}^{-1}$.
\end{prop}
\begin{proof}
We write $\det_i: \GL_i\rightarrow\bG_m$ for the determinant.
The regularity of $\tau$ and the decomposition $\varsigma = \varsigma_1\oplus \varsigma_2$ imply that
\[
r_p^{-1}(\varsigma) = \Ind_{P^+(K)}^{\GL_n(K)} r_p^{-1}(\varsigma_1)|\det_i|^{n-i} \otimes r_p^{-1}(\varsigma_2).
\]
According to the discussion before \cite[Theorem~3.7]{CEGGPS} based on \cite[Proposition~2.1 and Theorem~4.1]{Dat99}, we deduce an isomorphism of commutative algebras (using $\tau$ is regular)
\begin{equation}\label{equ: isom of Hecke}
\cH_{L(\cO_K)}^{L(K)}(\sigma)\cong \cH_{\mathbf{K}}^{\GL_n(K)}(\sigma(\tau))
\end{equation}
which sends $[L(\cO_K), \omega^{(i)}(p)^{-1} \mapsto \mathrm{id}_\sigma, L(\cO_K)]$ to $\mathbf{U}^{\tau_1}_\tau$. In fact, we have an isomorphism
\[
\Hom_{L(\cO_K)}(\sigma(\tau_1) \otimes_E \sigma(\tau_2),r_p^{-1}(\varsigma_1)|\det_i|^{n-i} \otimes r_p^{-1}(\varsigma_2)) \cong \Hom_{\mathbf{K}}(\sigma(\tau),r_p^{-1}(\varsigma)|_{\mathbf{K}})
\]
which is $\cH_{L(\cO_K)}^{L(K)}(\sigma)$-equivariant under (\ref{equ: isom of Hecke}).
As $\omega^{(i)}(p)^{-1}$ centralizes $L(\cO_K)$, $[L(\cO_K), \omega^{(i)}(p)^{-1} \mapsto \mathrm{id}_\sigma, L(\cO_K)]$ acts on $\Hom_{L(\cO_K)}(\sigma(\tau_1) \otimes_E \sigma(\tau_2),r_p^{-1}(\varsigma_1)|\det_i|^{n-i} \otimes r_p^{-1}(\varsigma_2))$ by the same scalar as $\omega^{(i)}(p)^{-1}$ acting on $r_p^{-1}(\varsigma_1)|\det_i|^{n-i} \otimes r_p^{-1}(\varsigma_2)$. This equals the scalar by which $p^{-1}\mathrm{Id}_i$ acts on $r_p^{-1}(\varsigma_1)|\det_i|^{n-i}$ which is $q^{i(n-i)}q^{\frac{i(i-1)}{2}}\al_{\varsigma_1}^{-1}$ as $r_p^{-1}(\varsigma_1)\cong \mathrm{rec}_p^{-1}(\varsigma_1)\otimes_E |\det_i|^{(i-1)/2}$ (see \cite[\S\,1.8]{CEGGPS}) and $|p| = q^{-1}$. It is then clear that $\mathbf{U}^{\tau_1}_\tau$ acts on $\Hom_{\mathbf{K}}(\sigma(\tau),r_p^{-1}(\varsigma)|_{\mathbf{K}})$ by the same scalar.
\end{proof}

\subsubsection{Normalized $U_p$-operators}
\label{subsub:Up:act:Ltc}

We keep the notation $\mathbf{K}$, $\mathbf{P}^\pm$, $L$, $\tau_1$, $\tau$, $\sigma$ and $\mathbf{U}^{\tau_1}_\tau$ from \S\,\ref{subsub:Up:act}.
In particular, $\sigma=\sigma(\tau_1)\otimes_E\sigma(\tau_2)$ is an irreducible $L(k)$-representation over $E$ with $\sigma(\tau)\cong\ind_{\mathbf{P}^+}^{\mathbf{K}} \sigma$ being irreducible as well. Fix a $L(k)$-stable $\cO$-lattice $\sigma^\circ \subset \sigma$. It is clear that $S_\sigma=[\mathbf{P}^-, 1 \mapsto \mathrm{id}_\sigma, \mathbf{P}^+]$ is obtained from an embedding
\begin{equation}\label{equ: K-level of S_sigma}
\ind_{\mathbf{P}^+}^{\mathbf{K}} \sigma^\circ \ra \ind_{\mathbf{P}^-}^{\mathbf{K}} \sigma^\circ
\end{equation}
by applying $\ind_{\mathbf{K}}^{\GL_n(K)}$, and this embedding is an isomorphism after inverting $p$ as $\ind_{\mathbf{P}^+}^{\mathbf{K}} \sigma$ is irreducible. Hence, we obtain two $\mathbf{K}$-stable $\cO$-lattices $\ind_{\mathbf{P}^+}^{\mathbf{K}} \sigma^\circ\subseteq\ind_{\mathbf{P}^-}^{\mathbf{K}} \sigma^\circ$ inside $\sigma(\tau)$.

Let $\sigma(\tau)^\circ \subset \ind_{\mathbf{P}^+}^{\mathbf{K}} \sigma^\circ$ be another $\mathbf{K}$-stable $\cO$-lattice in $\sigma(\tau)$, and $\kappa \in \Z$
be the maximal integer such that $p^{-\kappa} \sigma(\tau)^\circ \subset \ind_{\mathbf{P}^-}^{\mathbf{K}} \sigma^\circ$ via (\ref{equ: K-level of S_sigma}). We set
\begin{equation}\label{equ: embedding of lattices}
S_{\sigma(\tau)}^+:\sigma(\tau)^\circ \hookrightarrow \ind_{\mathbf{P}^+}^{\mathbf{K}} \sigma^\circ\,\,\mbox{ and }\,\, S_{\sigma(\tau)}^-:\sigma(\tau)^\circ \hookrightarrow \ind_{\mathbf{P}^-}^{\mathbf{K}} \sigma^\circ
\end{equation}
where $S_{\sigma(\tau)}^+$ is the inclusion $\sigma(\tau)^\circ \subset \ind_{\mathbf{P}^+}^{\mathbf{K}} \sigma^\circ$ and $S_{\sigma(\tau)}^-$ is the composition $\sigma(\tau)^\circ \overset{p^{-\kappa}}{\ra} p^{-\kappa} \sigma(\tau)^\circ \subset \ind_{\mathbf{P}^-}^{\mathbf{K}} \sigma^\circ$.

Upon abusing the same notation for the maps induced from applying $\ind_{\mathbf{K}}^{\GL_n(K)}$, the maps $S_{\sigma(\tau)}^+$ and $S_{\sigma(\tau)}^-$ can be inserted into the following commutative diagram involving $\mathbf{U}^{(i)}$:
\begin{equation}\label{eqn:Up}
\xymatrix{
\ind_{\mathbf{P}^+}^{\GL_n(K)}\sigma^\circ\ar^{S_\sigma}[rr] \ar@/^2pc/^{\mathbf{U}^{\tau_1}_\tau}[rrrr] &&\ind_{\mathbf{P}^-}^{\GL_n(K)}\sigma^\circ\ar^{ t_i}_{\sim}[rr] &&\ind_{\mathbf{P}^{+}}^{\GL_n(K)}\sigma^\circ\\
\ind_{\mathbf{K}}^{\GL_n(K)} \sigma(\tau)^\circ \ar^{S_{\sigma(\tau)}^+}[u] \ar^{p^\kappa}[rr] &&\ind_{\mathbf{K}}^{\GL_n(K)}\sigma(\tau)^\circ\ar^{S_{\sigma(\tau)}^-}[u] && \ind_{\mathbf{K}}^{\GL_n(K)} \sigma(\tau)^\circ. \ar^{S_{\sigma(\tau)}^+}[u]
}
\end{equation}
The commutativity of the diagram follows from Lemma~\ref{lemma:Upfactor} and the definitions of~$S_{\sigma(\tau)}^\pm$.

We consider the following condition:
\begin{cond}\label{cond:isom0}
For an $\cO[\GL_n(K)]$-module $\pi$ the map
\[
\Hom_{\cO[\GL_n(K)]}(\ind_{\mathbf{P}^{+}}^{\GL_n(K)}\sigma^\circ,\pi) \ra \Hom_{\cO[\GL_n(K)]}(\ind_{\mathbf{K}}^{\GL_n(K)} \sigma(\tau)^\circ,\pi)
\]
induced by $S_{\sigma(\tau)}^+$ is an isomorphism.
\end{cond}

If Condition~\ref{cond:isom0} holds on an $\cO[\GL_n(K)]$-module $\pi$, then we can apply $\Hom_{\cO[\GL_n(K)]}(-,\pi)$ to the diagram~\eqref{eqn:Up} and complete it to a commutative diagram
\begin{equation}\label{eqn:Uppi}
\xymatrix{
\Hom(\ind_{\mathbf{P}^+}^{\GL_n(K)}\sigma^\circ,\pi)\ar_{S_{\sigma(\tau)}^+}[d] &\Hom(\ind_{\mathbf{P}^-}^{\GL_n(K)}\sigma^\circ,\pi)\ar_{S_\sigma}[l]\ar_{S_{\sigma(\tau)}^-}[d] &\Hom(\ind_{\mathbf{P}^{+}}^{\GL_n(K)}\sigma^\circ,\pi) \ar_{ t_i}^{\sim}[l] \ar_{S_{\sigma(\tau)}^+}[d]\ar@/_2pc/_{\mathbf{U}^{\tau_1}_\tau}[ll]\\
\Hom(\ind_{\mathbf{K}}^{\GL_n(K)} \sigma(\tau)^\circ,\pi) &\Hom(\ind_{\mathbf{K}}^{\GL_n(K)}\sigma(\tau)^\circ,\pi) \ar_{p^\kappa}[l] & \Hom(\ind_{\mathbf{K}}^{\GL_n(K)} \sigma(\tau)^\circ,\pi) \ar@{-->}_{\tld{\mathbf{U}}^{\tau_1}_\tau}[l]
}
\end{equation}
by setting $\tld{\mathbf{U}}^{\tau_1}_\tau\defeq S_{\sigma(\tau)}^-\circ  t_i \circ (S_{\sigma(\tau)}^+)^{-1}$. Here $\Hom$ denotes $\Hom_{\cO[\GL_n(K)]}$ and we abuse the same notation for various maps induced by applying $\Hom_{\cO[\GL_n(K)]}(-,\pi)$ to~\eqref{eqn:Up}.

\subsection{Axiomatic approach to recover Galois representations}
\label{subsub:AxSetup}

In this section, we give a general axiomatic approach to recover the local Galois representation using patching formalism and modularity of extremal weights. The main result is Theorem~\ref{thm:main:LGC}.
\subsubsection{Modules with arithmetic actions}\label{sec:arithmod}

We introduce an axiomatic context, based on \cite[\S\,3]{CEGGPS2}, in which we deduce our main local-global compatibility result.
Let $\rhobar: G_K\ra\GL_n(\F)$ be a continuous homomorphism and let $R_{\rhobar}^\Box$ be the universal $\cO$-lifting ring of $\rhobar$.
Let $R^v$ be a complete equidimensional local Noetherian $\cO$-flat algebra with residue field $\F$ and set
\[
R_{\infty}\defeq R_{\rhobar}^\Box\widehat{\otimes}_{\cO}R^v.
\]
(We suppress the dependence of $R_\infty$ on $R^v$.) We write $\fm \subset R_\infty$ for the maximal ideal of $R_\infty$, and let $\mu\defeq -w_0(\eta)$.
If $\tau$ is a tame inertial type, we let $R_{\rhobar}^{\tau,\mu}$ denote the potentially crystalline $\cO$-lifting ring of $\rhobar$ of type $\tau$ and parallel Hodge--Tate weights $\mu$ as defined in \cite{kisinPSS}, and set
\[
R_{\infty}(\tau)\defeq R_{\rhobar}^{\tau,\mu}\widehat{\otimes}_{R_{\rhobar}^\Box}R_\infty.
\]
If $\lambda \in X^*(\un{T})$ is dominant, we let $R_{\rhobar}^{\lambda}$ denote the crystalline $\cO$-lifting ring of $\rhobar$ of Hodge--Tate weights $\lambda$, and set
\[
R_{\infty}(\lambda)\defeq R_{\rhobar}^{\lambda+\mu}\widehat{\otimes}_{R_{\rhobar}^\Box}R_\infty.
\]

If $\theta$ is an $\cO[\![\mathbf{K}]\!]$-module which is finite over $\cO$ and $M$ is a pseudocompact $\cO[\![\mathbf{K}]\!]$-module with a compatible action of $\GL_n(K)$, then the tensor product
\begin{equation}
\label{eq:patch:fct}
M(\theta)\defeq M \otimes_{\cO[\![\mathbf{K}]\!]}\theta
\end{equation}
is an $\cH_{\mathbf{K}}^{\GL_n(K)}(\theta)$-module via the natural isomorphism
\begin{equation}\label{eq: Frob rec isomorph}
(M\otimes_{\cO[\![\mathbf{K}]\!]}\theta)^\vee\cong \Hom_{\cO[\mathbf{K}]}(\theta,M^\vee)\cong
\Hom_{\cO[\GL_n(K)]}(\ind_{\mathbf{K}}^{\GL_n(K)}\theta,M^\vee)
\end{equation}
where $(\,\cdot\,)^\vee\defeq \Hom^{\mathrm{cts}}_\cO(\,\,\cdot\,\,,E/\cO)$.
(The first isomorphism follows e.g.~from \cite[Lemma B.3]{gee-newton} and the second from Frobenius reciprocity).

Recall from \S\,\ref{sec:notation:GT} that $\varepsilon:G_{\Q_p}\rightarrow\Z_p^\times$ is the cyclotomic character with mod-$p$ reduction $\omega$ and that $\tld{\omega}$ is the Teichm\"uller lift of $\omega$. We use the same notation for their restriction to $G_K$ and the corresponding characters of $K^\times$ via the normalized Artin's reciprocity map (cf. \S\,\ref{subsec:notation}).

\begin{defn}
\label{defn:patch}
An \emph{arithmetic $R_\infty[\GL_n(K)]$-module} (or an $\cO[\GL_n(K)]$-module with an arithmetic action of $R_\infty$) is a non-zero $\cO$-module $M_\infty$ with commuting actions of $R_\infty$ and $\GL_n(K)$ satisfying the following axioms:
\begin{enumerate}
\item
\label{it:patch:1}
the $R_\infty[\mathbf{K}]$-action on $M_\infty$ extends to $R_\infty[\![\mathbf{K}]\!]$ making $M_\infty$ a finitely generated $R_\infty[\![\mathbf{K}]\!]$-module;
\item
\label{it:patch:2}
$M_\infty$ is projective in the category of pseudocompact $\cO[\![\mathbf{K}]\!]$-modules;
\item
\label{it:patch:2:5}
if $\tau$ is a tame inertial type and $\sigma(\tau)^\circ \subset \sigma(\tau)$ is an $\cO$-lattice, the $R_\infty$-action on $M_\infty(\sigma(\tau)^\circ)$ factors through $R_\infty(\tau)$, and $M_\infty(\sigma(\tau)^\circ)$ is a maximal Cohen--Macaulay $R_\infty(\tau)$-module;
\item
\label{it:patch:3}
if $\lambda \in X^*(\un{T})$, the $R_\infty$-action on $M_\infty(F(\lambda))$ factors through $R_\infty(\lambda)$;
\item
\label{it:patch:4}
the action of $\cH_{\mathbf{K}}^{\GL_n(K)}(\sigma(\tau)) \cong \cH_{\mathbf{K}}^{\GL_n(K)}(\sigma(\tau)^\circ)[1/p]$ on $M_\infty(\sigma(\tau)^\circ)[1/p]$ factors through the composite
\[
\cH_{\mathbf{K}}^{\GL_n(K)}(\sigma(\tau))\stackrel{\eta_\infty}{\longrightarrow}R_{\rhobar}^{\tau,\mu}[1/p]\longrightarrow R_{\infty}(\tau)[1/p]
\]
where the map $\eta_\infty$ is the map denoted by $\eta$ in \cite[Theorem 4.1]{CEGGPS};
\item
\label{it:patch:5}
if $V$ is a Serre weight and $\tau$ is a tame inertial type over $E$ such that $V\in \JH(\ovl{\sigma(\tau)})$, then the $R_\infty$-action on $M_\infty(V)$ factors through $R_\infty(\tau)_\F$, and $M_\infty(V)$ is a maximal Cohen--Macaulay $R_\infty(\tau)_\F$-module.
\end{enumerate}
\end{defn}
An arithmetic $R_\infty[\GL_n(K)]$-module $M_\infty$ thus defines a functor $\theta\mapsto M_\infty(\theta)$ from the category of $\cO[\![\mathbf{K}]\!]$-modules which are finite over $\cO$ to the category of finite $R_\infty$-modules.
The functor $M_\infty$ obtained this way is a weak patching functor in the sense of \cite[Definition 6.2.1]{MLM} (though we only consider trivial algebraic factors here).
We say that an arithmetic $R_\infty[\GL_n(K)]$-module $M_\infty$ is \emph{minimal} if the weak patching functor it represents is minimal in the sense of \cite[Definition 6.2.1(I)]{MLM}.

\begin{lemma}\label{lemma:minimalM}
Given a Fontaine--Laffaille Galois representation $\rhobar: G_K \ra \GL_n(\F)$, there exists a minimal arithmetic $R_\infty[\GL_n(K)]$-module $M_\infty$.
\end{lemma}
\begin{proof}
This follows from the proof of \cite[Proposition 6.2.4]{MLM} and \cite[Lemma 4.17(2)]{CEGGPS} using that Fontaine--Laffaille deformation rings are formally smooth, and all Fontaine--Laffaille lifts are potentially diagonalizable by \cite[Lemma 1.4.3(2)]{BLGGT}.
\end{proof}

Given an arithmetic $R_\infty[\GL_n(K)]$-module $M_\infty$, we define
\begin{equation}
W_{M_\infty}(\rhobar)\defeq \{V \mid V\text{ is a Serre weight such that $M_\infty(V\otimes_{\F}\omega^{n-1}\circ\det)\neq 0$} \}.
\end{equation}
For any $\cO[\![\mathbf{K}]\!]$-module $\theta$ which is finite over $\cO$, the finitely generated $R_\infty$-module $M_\infty(\theta)$ is nonzero if and only if $M_\infty(\theta)/\fm$ is nonzero by Nakayama's lemma.
Let $\pi_\infty$ be the admissible $\GL_n(K)$-representation $(M_\infty/\fm)^\vee$ over $\F$.
Then for an $\cO[\![\mathbf{K}]\!]$-module $\theta$ which is finite over $\cO$, $(M_\infty(\theta)/\fm)^\vee$ is isomorphic to $\Hom_{\mathbf{K}}(\theta,\pi_\infty|_{\mathbf{K}})$.
In particular, for a Serre weight $V$, $V\in W_{M_\infty}(\rhobar)$ if and only if $\Hom_{\mathbf{K}}(V\otimes_{\F}\omega^{n-1}\circ\det,\pi_\infty|_{\mathbf{K}}) \neq 0$.

If $S$ is a set of Serre weights, we write $S\otimes_{\F}\omega^{n-1}\circ\det\defeq \{V\otimes_{\F}\omega^{n-1}\circ\det\mid V\in S\}$. For each tame inertial type $\tau$, we clearly have $$\sigma(\tau\otimes_{\cO}\tld{\omega}^{n-1})\cong \sigma(\tau)\otimes_E\tld{\omega}^{n-1}\circ\det$$
and so
\begin{equation}\label{equ: det twist}
\JH(\ovl{\sigma(\tau\otimes_{\cO}\tld{\omega}^{n-1})})=\JH(\ovl{\sigma(\tau)})\otimes_{\F} \omega^{n-1}\circ\det.
\end{equation}
Here we recall that $\JH(\ovl{\sigma(\tau)})$ is the set of Jordan--H\"older factors of the mod-$p$ reduction of an arbitrary $\mathbf{K}$-stable $\cO$-lattice in $\sigma(\tau)$.

Recall the set $W_{\obv}(\rhobar)$ defined in (\ref{equ: obv wt}). The following condition is important for us.
\begin{cond}\label{cond:modularity of obv wt}
There exists $x\in\tld{\cF\cL}_\cJ(\F)$ such that $\rhobar\cong\rhobar_{x,\lambda+\eta}$ for a $(3n-1)$-generic Fontaine--Laffaille weight $\lambda+\eta$. Moreover, we have an inclusion $W_{\obv}(\rhobar)\subseteq W_{M_\infty}(\rhobar)$.
\end{cond}
Note that Condition~\ref{cond:modularity of obv wt} implies that $F(\lambda)\in W_{M_\infty}(\rhobar)$.
In fact, the converse is true under a stronger genericity hypothesis.

\begin{lemma}\label{lemma:condition}
If $\lambda\in X_1(\un{T})$ where $\lambda+\eta$ is $5n$-generic Fontaine--Laffaille and $M_\infty(F(\lambda)\otimes_{\F}\omega^{n-1}\circ\det) \neq 0$, then Condition~\ref{cond:modularity of obv wt} holds.
\end{lemma}
\begin{proof}
That $M_\infty(F(\lambda)\otimes_{\F}\omega^{n-1}\circ\det) \neq 0$ implies that $R_{\rhobar}^{\lambda+\eta}$ is nonzero by Definition \ref{defn:patch}~\eqref{it:patch:3}.
This implies the first part of Condition \ref{cond:modularity of obv wt}.
That $M_\infty(F(\lambda)\otimes_{\F}\omega^{n-1}\circ\det) \neq 0$ also implies that $M_\infty$ is \emph{extremal} in the sense of \cite[Definition 5.4.2]{OBW} (as $F(\lambda)\in W_{\obv}(\rbar|_{G_{F_{\tld{v}}}})$). We thus deduce from \cite[Theorem 5.4.6]{OBW} the second part of Condition \ref{cond:modularity of obv wt}.
\end{proof}

Given $\lambda+\eta$ which is $(3n-1)$-generic Fontaine--Laffaille, we recall the notion of $F(\lambda)$-relevant types from Definition~\ref{defn:relevant:types}.
\begin{lemma}\label{lem: criterion for shape}
Assume that $\rhobar$ satisfies Condition~\ref{cond:modularity of obv wt}. Let $\tau_0$ be a $F(\lambda)$-relevant inertial type. Then $\tld{w}^*(\rhobar,\tau_0)=t_\eta$ if and only if
$$\JH(\ovl{\sigma(\tau_0)}) \cap W_{M_\infty}(\rhobar)=\{F(\lambda)\}.$$ In particular, the set $W_{M_\infty}(\rhobar)$ determines the set of $F(\lambda)$-relevant types $\tau_0$ such that $\tld{w}^*(\rhobar,\tau_0)=t_\eta$.
\end{lemma}
\begin{proof}
Let $\tau_0$ be a $F(\lambda)$-relevant inertial type. If $\tld{w}^*(\rhobar,\tau_0)=t_\eta$, then it follows from Lemma~\ref{lem: specialization and open cell}, \cite[Theorem 5.1.1]{OBW} and \cite[Proposition 4.3.2, Remark 4.3.3]{LLL} that there exists a specialization $\rhobar \leadsto \rhobar^{\speci}$ such that
$$\JH(\ovl{\sigma(\tau_0)}) \cap W_{M_\infty}(\rhobar)\subseteq \JH(\ovl{\sigma(\tau_0)}) \cap W^?(\rhobar^{\speci})=\{F(\lambda)\},$$
which together with Condition~\ref{cond:modularity of obv wt} implies that $\JH(\ovl{\sigma}(\tau_0)) \cap W_{M_\infty}(\rhobar)=\{F(\lambda)\}$.
If $\tld{w}^*(\rhobar,\tau_0)\neq t_\eta$, then we have $\ell(\tld{w}^*(\rhobar,\tau_0))<\ell(t_\eta)$ (cf. Lemma~\ref{lem:shape:Schubert}) and deduce from Lemma~\ref{lemma:shape:detect} that there exists $V\in \Big(W_\obv(\rhobar)\cap \JH(\ovl{\sigma(\tau_0)}\Big)\setminus\{F(\lambda)\}$, which together with Condition~\ref{cond:modularity of obv wt} implies that $\{F(\lambda),V\}\subseteq \JH(\ovl{\sigma(\tau_0)}) \cap W_{M_\infty}(\rhobar)$. This shows that $\tld{w}^*(\rhobar,\tau_0)=t_\eta$ if and only if $\JH(\ovl{\sigma(\tau_0)}) \cap W_{M_\infty}(\rhobar)=\{F(\lambda)\}$, and so $W_{M_\infty}(\rhobar)$ determines the set of relevant types $\tau_0$ such that $\tld{w}^*(\rhobar,\tau_0)=t_\eta$.
\end{proof}

\subsubsection{Normalized $U_p$-action}\label{sec:norm-Up-act}
We keep the notation $\mathbf{K}$, $\mathbf{P}^\pm$, $L$, $\tau_1$, $\tau$, $\sigma$ as well as $\mathbf{U}^{\tau_1}_\tau$ from \S\,\ref{subsub:Up:act} and $S_{\sigma(\tau)}^\pm$ from (\ref{equ: embedding of lattices}).

We consider the following condition:
\begin{cond}\label{cond:isom}
\begin{enumerate}
\item \label{item:S+} $\JH(\coker S_{\sigma(\tau)}^+ \otimes_{\cO} \F) \cap \big(W_{M_\infty}(\rhobar)\otimes_{\F}\omega^{n-1}\circ\det\big) = \emptyset$.
\item \label{item:S-} $\JH(\coker S_{\sigma(\tau)}^- \otimes_{\cO} \F) \cap \big(W_{M_\infty}(\rhobar)\otimes_{\F}\omega^{n-1}\circ\det\big) = \emptyset$.
\end{enumerate}
\end{cond}

Assuming Condition~\ref{cond:isom}~\eqref{item:S+}, we have that the map $M_\infty(\sigma(\tau)^\circ) \ra M_\infty(\Ind_{\mathbf{P}^+}^{\mathbf{K}} \sigma^\circ)$ induced from $S_{\sigma(\tau)}^+$ is an isomorphism, or equivalently that Condition~\ref{cond:isom0} holds for $M_\infty^\vee$ (cf. (\ref{eq: Frob rec isomorph})).
This implies that the induced map $M_\infty(\sigma(\tau)^\circ)/\fm \ra M_\infty(\Ind_{\mathbf{P}^+}^{\mathbf{K}} \sigma^\circ)/\fm$ is an isomorphism, or equivalently that Condition~\ref{cond:isom0} holds for $\pi_\infty \defeq (M_\infty/\fm)^\vee$ (cf. (\ref{eq: Frob rec isomorph})).
Then we get an endomorphism $$\tld{\mathbf{U}}^{\tau_1}_\tau \in \End(\Hom_{\cO[\GL_n(K)]}(\ind_{\mathbf{K}}^{\GL_n(K)}\sigma(\tau)^\circ,\pi_\infty))\cong \End(\Hom_{\cO[\mathbf{K}]}(\sigma(\tau)^\circ,\pi_\infty))$$ as in~\eqref{eqn:Uppi}. As in~\eqref{eqn:Uppi}, we abuse the notation $S_{\sigma(\tau)}^+,S_{\sigma(\tau)}^-$ for the maps induced from (\ref{equ: embedding of lattices}) by applying $\Hom_{\cO[\mathbf{K}]}(-,\pi_\infty)$.

\begin{prop}\label{prop:pi-inv}
Assume Condition~\ref{cond:isom}~\eqref{item:S+} and that $R_{\rhobar}^{\tau,\mu}$ is regular.
Then $\eta_{\infty}(\mathbf{U}^{\tau_1}_\tau) \in p^\kappa R_{\rhobar}^{\tau,\mu}$ and $\tld{\mathbf{U}}^{\tau_1}_\tau \in \End(\Hom_{\cO[\mathbf{K}]}(\sigma(\tau)^\circ,\pi_\infty))$ acts by the scalar $p^{-\kappa} \eta_{\infty}(\mathbf{U}^{\tau_1}_\tau) \pmod{\fm} \in \F$.
If furthermore Condition~\ref{cond:isom}~\eqref{item:S-} holds, then $p^{-\kappa} \eta_{\infty}(\mathbf{U}^{\tau_1}_\tau) \pmod{\fm} \in \F^\times$.
\end{prop}
\begin{proof}
From~\eqref{eqn:Uppi}, there is an endomorphism $\tld{\mathbf{U}}^{\tau_1}_\tau$ of $\Hom_{\cO[\mathbf{K}]}(\sigma(\tau)^\circ, M_\infty^\vee)$ with the property that $p^\kappa \cdot \tld{\mathbf{U}}^{\tau_1}_\tau=S_{\sigma(\tau)}^+\circ\mathbf{U}^{\tau_1}_\tau\circ (S_{\sigma(\tau)}^+)^{-1}$. Moreover, if Condition~\ref{cond:isom}~\eqref{item:S-} holds, then the map $M_\infty(\sigma(\tau)^\circ) \ra M_\infty(\Ind_{\mathbf{P}^-}^{\mathbf{K}} \sigma^\circ)$ induced from $S_{\sigma(\tau)}^-$ is an isomorphism, and thus $\tld{\mathbf{U}}^{\tau_1}_\tau=S_{\sigma(\tau)}^-\circ t_i\circ (S_{\sigma(\tau)}^+)^{-1}$ is also an isomorphism.
If we denote by $\tld{\mathbf{U}}^{\tau_1}_\tau$ and $\mathbf{U}^{\tau_1}_\tau$ the respective Pontrjagin dual endomorphisms, then $\tld{\mathbf{U}}^{\tau_1}_\tau$ acts on $M_\infty(\sigma(\tau)^\circ)$ (intertwining the $R_\infty$-action) and $p^\kappa \cdot \tld{\mathbf{U}}^{\tau_1}_\tau$ acts on $M_\infty(\sigma(\tau)^\circ)$ by $\eta_{\infty}(\mathbf{U}^{\tau_1}_\tau)$.

There exists a minimal arithmetic $R_\infty'[\GL_n(K)]$-module $M_\infty'$ by Lemma~\ref{lemma:minimalM}.
Serre's theorem on finiteness of projective dimensions of finite $R_\infty'(\tau)$-modules and the Auslander--Buchsbaum formula imply that the maximal Cohen--Macaulay $R_\infty'(\tau)$-modules such as $M_\infty'(\sigma(\tau)^\circ)$ are free.
Then the $R_\infty'(\tau)$-rank of $M_\infty'(\sigma(\tau)^\circ)$ is one.
The above considerations apply, and we conclude that $\tld{\mathbf{U}}^{\tau_1}_\tau$ acts on $M_\infty'(\sigma(\tau)^\circ)$ by an element of $R_\infty'(\tau) \cong \End_{R_\infty}(M_\infty'(\sigma(\tau)^\circ)$.
Since $R_\infty'(\tau)$ is $p$-torsion free, this element must be $p^{-\kappa} \eta_{\infty}(\mathbf{U}^{\tau_1}_\tau)$.
In particular, $p^{-\kappa}\eta_{\infty}(\mathbf{U}^{\tau_1}_\tau) \in R_{\rhobar}^{\tau,\mu}$ and if Condition~\ref{cond:isom}~\eqref{item:S-} holds then $p^{-\kappa}\eta_{\infty}(\mathbf{U}^{\tau_1}_\tau) \in (R_{\rhobar}^{\tau,\mu})^\times$.

Using that $M_\infty(\sigma(\tau)^\circ)$ is $p$-torsion free and that $p^\kappa \cdot \tld{\mathbf{U}}^{\tau_1}_\tau$ acts on $M_\infty(\sigma(\tau)^\circ)$ by $\eta_{\infty}(\mathbf{U}^{\tau_1}_\tau)$, $\tld{\mathbf{U}}^{\tau_1}_\tau$ acts on $M_\infty(\sigma(\tau)^\circ)$ by $p^{-\kappa}\eta_{\infty}(\mathbf{U}^{\tau_1}_\tau) \in R_{\rhobar}^{\tau,\mu}$.
Then $\tld{\mathbf{U}}^{\tau_1}_\tau$ acts on $M_\infty(\sigma(\tau)^\circ)/\fm$ by $p^{-\kappa}\eta_{\infty}(\mathbf{U}^{\tau_1}_\tau) \pmod{\fm}$, and the result follows by applying Pontryagin duals.
\end{proof}

We now fix $\lambda+\eta\in X^*(\un{T})$ which is $(3n-1)$-generic Fontaine--Laffaille and let $\tau\otimes_{\cO}\tld{\omega}^{1-n}$ be a $F(\lambda)$-relevant type. Note that for an arbitrary choice of a $\mathbf{K}$-stable $\cO$-lattice $\sigma(\tau)^\circ\subseteq\sigma(\tau)$, $\ovl{\sigma}(\tau)\defeq \sigma(\tau)^\circ\otimes_{\cO}\F$ contains $F(\lambda)\otimes_{\F}\omega^{n-1}\circ\det$ as a Jordan--H\"older factor with multiplicity one (see Remark~\ref{rmk:relevant:outer} and (\ref{equ: det twist})).
Since $\sigma(\tau_1)$ and $\sigma(\tau_2)$ are defined over $E_0 \defeq W(\F)[p^{-1}]$, we can and do choose $E_0$-rational structures $\sigma(\tau_1)_{E_0}$ and $\sigma(\tau_2)_{E_0}$. We choose $\sigma_{E_0}^\circ \subset \sigma(\tau_1)_{E_0}\otimes_{E_0}\sigma(\tau_2)_{E_0}$ to be an arbitrary $L(k)$-stable $W(\F)$-lattice.
Then we choose $\sigma(\tau)_{E_0}^\circ \subset \Ind_{\mathbf{P}^+}^{\mathbf{K}} \sigma_{E_0}^\circ$ to be the unique $W(\F)$-lattice whose cosocle is isomorphic to $F(\lambda)\otimes_{\F}\omega^{n-1}\circ\det$ and whose image in $\Ind_{\mathbf{P}^+}^{\mathbf{K}} \sigma_{E_0}^\circ\otimes_{W(\F)} \F$ is nonzero. We fix the choice $\sigma^\circ\defeq \sigma_{E_0}^\circ\otimes_{W(\F)}\cO$ and $\sigma(\tau)^\circ\defeq \sigma(\tau)_{E_0}^\circ\otimes_{W(\F)}\cO$. As $\kappa\in \Z$ is the maximal integer such that $\sigma(\tau)^\circ \subset p^\kappa \Ind_{\mathbf{P}^-}^{\mathbf{K}} \sigma^\circ$ and $E_0/\Q_p$ is unramified, we deduce that the image of $\sigma(\tau)^\circ$ in $(p^\kappa \Ind_{\mathbf{P}^-}^{\mathbf{K}} \sigma^\circ )\otimes_{\cO} \F$ is nonzero.
\begin{lemma}\label{lem: check cond}
Assume that $\rhobar$ satisfies Condition~\ref{cond:modularity of obv wt}. Let $\tau\otimes_{\cO}\tld{\omega}^{1-n}$ be a $F(\lambda)$-relevant type such that $\tld{w}^*(\rhobar,\tau\otimes_{\cO}\tld{\omega}^{1-n})=t_\eta$. Then the choice of $\sigma^\circ$ and $\sigma(\tau)^\circ$ above satisfies Condition~\ref{cond:isom}~\eqref{item:S+} and Condition~\ref{cond:isom}~\eqref{item:S-}.
\end{lemma}
\begin{proof}
It follows from Lemma~\ref{lem: criterion for shape} (applied to $\tau_0=\tau\otimes_{\cO}\tld{\omega}^{1-n}$) and Remark~\ref{rmk:relevant:outer} that $$\JH(\ovl{\sigma(\tau\otimes_{\cO}\tld{\omega}^{1-n})}) \cap W_{M_\infty}(\rhobar)=\{F(\lambda)\}$$ and $\ovl{\sigma(\tau\otimes_{\cO}\tld{\omega}^{1-n})}$ contains $F(\lambda)$ as a Jordan--H\"older factor with multiplicity one. The choice of $\sigma(\tau)^\circ$ above forces the image of $S_{\sigma(\tau)}^+\otimes_{\cO}\F$ (resp.~ of $S_{\sigma(\tau)}^-\otimes_{\cO}\F$) to contain $F(\lambda)\otimes_{\F}\omega^{n-1}\circ\det$ as a Jordan--H\"older factor, which implies that $\coker S_{\sigma(\tau)}^+ \otimes_{\cO} \F$ (resp.~$\coker S_{\sigma(\tau)}^+ \otimes_{\cO} \F$) does not contain $F(\lambda)\otimes_{\F}\omega^{n-1}\circ\det$ as a Jordan--H\"older factor. Hence, Condition~\ref{cond:isom}~\eqref{item:S+} and Condition~\ref{cond:isom}~\eqref{item:S-} follow.
\end{proof}

\begin{prop}\label{prop:Up inv}
Assume that $\rhobar$ satisfies Condition~\ref{cond:modularity of obv wt}. Let $\tau\otimes_{\cO}\tld{\omega}^{1-n} = \tau(s_{\cJ},\lambda+\eta-s_{\cJ}(\eta))$ be a $F(\lambda)$-relevant type and $\tau_1\subseteq\tau$ a sub inertial type. Suppose that $I_{\cJ}$ is the set corresponding to $\tau_1\otimes_{\cO}\tld{\omega}^{1-n}$ via Lemma~\ref{lem: bijection tau} and that $\rhobar\cong\rhobar_{x,\lambda+\eta}$ for some $x\in \cM_{s_\cJ^{-1}}(\F)$. Then there exists a unique $\kappa \in \Z$ depending only on $\tau$ and $\tau_1$ such that $p^{-\kappa} \eta_{\infty}(\mathbf{U}_\tau^{\tau_1}) \pmod{\fm} = f_{s_{\cJ}^{-1},I_{\cJ}}(x)$.
\end{prop}
\begin{proof}
The fact that $\rhobar\cong\rhobar_{x,\lambda+\eta}$ for some $x\in \cM_{s_\cJ^{-1}}(\F)$ implies that $\tld{w}^*(\rhobar,\tau\otimes_{\cO}\tld{\omega}^{1-n})=t_\eta$.
It follows from \cite[Theorem 4.1.1]{OBW} that $R_{\rhobar}^{\tau,\mu}$ is formally smooth over $\cO$, and in particular regular. Then it follows from Proposition~\ref{prop:pi-inv} and Lemma~\ref{lem: check cond} that there exists $\kappa\in\Z$ depending only on $\tau$ and $\tau_1$ such that $p^{-\kappa} \eta_{\infty}(\mathbf{U}_\tau^{\tau_1})\in (R_{\rhobar}^{\tau,\mu})^\times$. We consider a $\rho^\circ$ associated with an arbitrary homomorphism $R_{\rhobar}^{\tau,\mu}\twoheadrightarrow\cO$ (with kernel $\fp$) and observe that $\rho^\circ_0\defeq \rho^\circ\otimes_{\cO}(\varepsilon^{n-1}\tld{\omega}^{1-n})$ is a potentially crystalline lift of $\rhobar$ with inertial type $\tau\otimes_{\cO}\tld{\omega}^{1-n}$ and Hodge--Tate weights $n-1,\dots,1,0$. We set $\varsigma\defeq \WD(\rho^\circ)$, $\varsigma_0\defeq \WD(\rho^\circ_0)$ and note that $\varsigma_0\cong\varsigma\otimes_{E}(|\mathrm{Art}_K^{-1}|^{n-1}\tld{\omega}^{1-n})$. We write $\varsigma_{0,1}\subseteq\varsigma_0$ for the unique subrepresentation satisfying $\varsigma_{0,1}|_{I_K}\cong \tau_1\otimes_{\cO}\tld{\omega}^{1-n}$. Hence, we may apply Theorem~\ref{thm:Feval-2} to $\varsigma_0$ and deduce that $\val_p(\al_{\varsigma_{0,1}}^{-1})=d_0\defeq d_{\tau\otimes_{\cO}\tld{\omega}^{1-n},\tau_1\otimes_{\cO}\tld{\omega}^{1-n}}$ and
$$\frac{\al_{\varsigma_{0,1}}^{-1}}{p^{d_0}}\equiv f_{s_\cJ^{-1},I_\cJ}(x)\in \F^\times.$$
It follows from Proposition~\ref{prop:LLC} and \cite[Theorem 4.1]{CEGGPS} that
$$\eta_{\infty}(\mathbf{U}_\tau^{\tau_1}) \pmod{\fp}=p^{\frac{fi(2n-i-1)}{2}}\al_{\varsigma_1}^{-1},$$
which together with the identity $\al_{\varsigma_1}^{-1}=|p|^{i(n-1)}\al_{\varsigma_{0,1}}^{-1}=p^{-fi(n-1)}\al_{\varsigma_{0,1}}^{-1}$ finishes the proof
\end{proof}

Note that the above proof also shows that
$$\kappa=\frac{fi(2n-i-1)}{2}+\val_p(\al_{\varsigma_1}^{-1})=\frac{fi(2n-i-1)}{2}+d_0-fi(n-1)=d_0-\frac{fi(i-1)}{2}.$$

\begin{rmk}\label{rmk:extr-isom}
Our choice of $\rhobar$ and $\tau$ satisfies the crucial condition that $\JH(\ovl{\sigma(\tau)}) \cap \big(W_{M_\infty}(\rhobar)\otimes_{\F}\omega^{n-1}\circ\det\big)$ contains a unique Serre weight $V$ which has multiplicity one in $\ovl{\sigma(\tau)}$. It is clear that Lemma~\ref{lem: check cond} admits an immediate generalization for any pair $\rhobar$, $\tau$ satisfying this condition. In fact, the map $\tld{\mathbf{U}}_\tau^{\tau_1}$ in~\eqref{eqn:Uppi} exists for such $\rhobar$ and $\tau$ (with $\pi_\infty=(M_\infty/\fm)^\vee$ and the choice of $\sigma^\circ$ and $\sigma(\tau)^\circ$ similar to that of Lemma~\ref{lem: check cond}).
\end{rmk}

\subsubsection{Recovering the Galois representations from extremal weights and arithmetic actions}
\label{sec:divide}
Let $\rhobar: G_K \ra \GL_n(\F)$ be a Galois representation.
Fix an arithmetic $R_\infty[\GL_n(K)]$-module $M_\infty$ for $\rhobar$.

\begin{lemma}
\label{lem:glob:algorithm}
Assume that $\rhobar$ satisfies Condition~\ref{cond:modularity of obv wt} for some $x\in\tld{\cF\cL}_\cJ(\F)$. Then the set $W_{M_\infty}(\rhobar)$ determines the unique $\cC\in\cP_\cJ$ satisfying $x\in\cC(\F)$.
\end{lemma}
\begin{proof}
We write $\cC\in\cP_\cJ$ for the unique element such that $x\in\cC(\F)$. It follows from Lemma~\ref{lem: criterion for shape} that $W_{M_\infty}(\rhobar)$ determines the set of $F(\lambda)$-relevant inertial types $\tau_0$ such that $\tld{w}^*(\rhobar,\tau_0)=t_\eta$. Then we apply Lemma~\ref{lem:shape:Schubert} and observe that $W_{M_\infty}(\rhobar)$ determines the set $\{w_\cJ\in\un{W}\mid \cC\subseteq \cM_{w_\cJ}^\circ\}$, which determines $\cC$ by the item~\ref{it: open cover:iii} of Lemma~\ref{lem: open cover} (applied to $\cC_{K_j}$ for each $j\in\cJ$, if $\cC=(\cC_{K_j})_{j\in\cJ}$).
\end{proof}

\begin{thm}
\label{thm:main:LGC}
Assume that $\rhobar$ satisfies Condition~\ref{cond:modularity of obv wt}. Let $M_\infty$ be an arithmetic $R_\infty[\GL_n(K)]$-module, and $\pi_\infty$ be the $\GL_n(K)$-representation $(M_\infty/\fm)^\vee$.
Then the conjugacy class of $\rhobar$ can be recovered from the isomorphism class of $\pi_\infty$ as an $\F[\GL_n(K)]$-module.
\end{thm}
\begin{proof}
We can write $\rhobar = \rhobar_{x,\lambda+\eta}$ for some $x\in \tld{\cF\cL}_{\cJ}(\F)$.
As a Serre weight $V$ satisfies $V\in W_{M_\infty}(\rhobar)$ if and only if $\Hom_{\mathbf{K}}(V\otimes_{\F}\omega^{n-1}\circ\det,\pi_\infty|_{\mathbf{K}}) \neq 0$, we deduce from Lemma~\ref{lem:glob:algorithm} that the $\mathbf{K}$-action on $\pi_\infty$ determines the unique $\cC\in\cP_\cJ$ such that $x \in \cC(\F)$.
Moreover, for each $g \in \Inv(\cC)$, we deduce from Lemma~\ref{lem: covering regular locus} that there exists $s_\cJ\in\un{W}$ and $I_\cJ\subseteq\mathbf{n}_\cJ$ such that $I_\cJ\cdot(s_\cJ^{-1},1)=I_\cJ$, $x\in\cM_{s_\cJ^{-1}}(\F)$ and $g=f_{s_\cJ^{-1},I_\cJ}$. Assume without loss of generality that $I_\cJ\neq \emptyset$. %
We consider the $F(\lambda)$-relevant inertial type $\tau\otimes_{\cO}\tld{\omega}^{1-n}=\tau(s_{\cJ},\lambda+\eta-s_{\cJ}(\eta))$ together with the sub inertial type $\tau_1\subseteq\tau$ such that $\tau_1\otimes_{\cO}\tld{\omega}^{1-n}$ corresponds to $I_\cJ$ via Lemma~\ref{lem: bijection tau}. Then we deduce from Proposition~\ref{prop:Up inv} that there exists $\kappa\in\Z$ depending only on $\tau$ and $\tau_1$ such that $\tld{\mathbf{U}}^{\tau_1}_{\tau}$ acts on $\Hom_{\mathbf{K}}(\sigma(\tau)^\circ,\pi_\infty|_{\mathbf{K}})$ by $g(x)$.
Since this action of $\tld{\mathbf{U}}^{\tau_1}_{\tau}$ only depends on the $\F[\GL_n(K)]$-action on $\pi_\infty$ and $(g(x))_{g\in \Inv(\cC)}$ determines $x$ by Corollary~\ref{cor: separate points}, the result follows.
\end{proof}

\subsection{Local-global compatibility for Hecke eigenspaces}
\label{sub:GSetup}
In this section, we apply Theorem~\ref{thm:main:LGC} to a favorable global setup and deduce our main result on local-global compatibility (see Theorem~\ref{thm:main:LGC:prime}).
We now fix the global setup for the main arithmetic application.
We follow the exposition and setup of \cite[\S\,7.1, \S\,4.2]{EGH}, (see also \cite[\S\,4.1, \S\,4.2 and \S\,4.5]{HLM}).

Let $F/\Q$ be a CM field and $F^+$ its maximal totally real subfield.
Assume that $F^+\neq \Q$, and that all places of $F^+$ above $p$ split in $F$.

We let $G_{/F^+}$ be a reductive group, which is an outer form of $\GL_n$ which splits over~$F$, such that $G(F^{+}_{v})\cong U_n(\R)$ for all $v\mid \infty$.
Then (cf.~\cite[\S\,7.1]{EGH}) $G$ admits a reductive model $\cG$ over $\cO_{F^+}[1/N]$, for some $N\in\N$ prime to $p$, together with an isomorphism $\iota:\cG_{/\cO_{F}[1/N]}\ra {\GL_n}_{/\cO_F[1/N]}$.
If $v\nmid N$ is a place of $F^+$ which is split in $F$, with decomposition $v=w w^c$, we get an isomorphism
\begin{equation}\label{eq: isom of auto groups}
\iota_w:\,\cG(\cO_{F^+_v})\stackrel{\sim}{\longrightarrow}\GL_n(\cO_{F_w}).
\end{equation}

For a compact open subgroup $U\leq G(\bA_{F^+}^{\infty})$ and a finite $\cO$-module $W$, the space of algebraic automorphic forms on $G$ of level $U$ and coefficients $W$ is defined to be
\begin{equation*}
S(U,W)\defeq \left\{f:\,G(F^{+})\backslash G(\bA^{\infty}_{F^{+}})/U \rightarrow W\right\}.
\end{equation*}
For a finite place $v\nmid N$ of $F^+$ we say that $U$ is \emph{unramified} at $v$ if one has a decomposition $U=\cG(\cO_{F_v^+})U^{v}$ for some compact open subgroup $U^v\leq G(\bA^{\infty,v}_{F^+})$.

Let $\cP_U$ denote the set consisting of finite places $w$ of $F$ such that $v\defeq w\vert_{F^+}$ is split in~$F$, $w\nmid pN$, and $U$ is unramified at $v$. For a subset $\cP\subseteq \cP_U$ of finite complement and closed with respect to complex conjugation we write $\bT^{\cP}=\cO[T^{(i)}_w\mid w\in\cP,\,i\in\{0,1,\cdots,n\}]$ for the abstract Hecke algebra on $\cP$, where the Hecke operator $T_w^{(i)}$ acts on $S(U,W)$ via the usual double coset operator
$$
\iota_w^{-1}\left[ \GL_n(\cO_{F_w}) \left(\begin{matrix}
      \varpi_{w}\mathrm{Id}_i &  \cr  & \mathrm{Id}_{n-i} \end{matrix} \right)
\GL_n(\cO_{F_w}) \right],
$$
where $\varpi_w$ denotes a uniformizer of $F_w$.

\vspace{2mm}

If $\rbar:G_F\rightarrow \GL_n(\F)$ is a continuous absolutely irreducible Galois representation, we write $\mathfrak{m}_{\rbar}$ for the ideal of $\bT^{\cP}$ with residue field $\F$ defined by the formula
$$
\det\left(1-\overline{r}(\mathrm{Frob}_w)X\right)=\sum_{j=0}^n (-1)^j(\mathbf{N}_{F/\Q}(w))^{\binom{j}{2}}(T_w^{(j)}\bmod \mathfrak{m}_{\rbar})X^j\quad \forall w\in \cP
$$
(and $\mathbf{N}_{F/\Q}(w)$ denotes the norm from $F$ to $\Q$ of the place $w$).
We emphasize that the ideal $\mathfrak{m}_{\rbar}$ above is as defined in \cite[\S\,2.3]{CEGGPS}, and \emph{differs from the ideal associated to $\rbar$ in} \cite[\S\,4.2]{EGH, HLM} (our ideal would have been denoted as $\mathfrak{m}_{\rbar^\vee}$ in \emph{loc.~cit.}).
We say that $\rbar$ is \emph{automorphic} if $S(U,\F)[\mathfrak{m}_{\rbar}]\neq 0$ for some level $U$.

We now assume that there is a place $\tld{v}|p$ of $F$ which is unramified and let $v \defeq \tld{v}|_{F^+}$.
We specialize the terminology and notation of \S\,\ref{sec:prel} with the unramified $p$-adic field $K$ taken to be $F_{\tld{v}}$. (In particular, $\un{G}$ is now $\big(\Res_{\cO_{F_{\tld{v}}}/\Zp}\GL_n\big)\otimes_{\Zp}\cO$ and $\mathbf{K}=\GL_n(\cO_{F_{\tld{v}}})\cong \cG(\cO_{F_v^+}$).)
Fix a level $U^v$ away from $v$, and consider the smooth $\cG(F^+_{v})$-representation
\[
\pi(\rbar)\defeq  \underset{\substack{\longrightarrow\\U_v\leq \cG(\cO_{F_v^+})}}{\lim}\, S(U^v U_v,  \F))[\fm_{\rbar}].
\]

We say that a compact open subgroup $U \subset \cG(\bA^\infty_{F^+})$ is \emph{sufficiently small} if for all $t\in \cG(\bA^\infty_{F^+})$, the order of the finite group $t \cG(F^+) t^{-1} \cap U$ is prime to $p$

\begin{prop}\label{prop:patch}
Suppose that $\cG(\cO_{F_v^+}) U^v$ is sufficiently small, $\pi(\rbar)$ is nonzero, and the image of $\rbar|_{G_{F(\zeta_p)}}$ is adequate.
Then there exist a local ring $(R_\infty,\fm)$ as in \S\,\ref{sec:arithmod}, an arithmetic $R_\infty$-module $M_\infty$, and an isomorphism
\begin{equation}
\label{it:Minfty}
(M_\infty/\fm)^\vee\stackrel{\sim}{\longrightarrow}\pi(\rbar)
\end{equation}
of $\GL_n(F_{\tld{v}})$-representations over $\F$.
\end{prop}
\begin{proof}
The existence follows from a modification of the Taylor--Wiles patching construction in \cite[\S\,A]{MLM} which we now summarize.
We write $S_p^+$ for the set of finite places of $F^+$ above $p$ excluding $v$.
For each $v'\in  S_p^+$ we fix a place $\tld{v}'$ of $F$ above it. %
Let $S^+$ be the union of $\{v\}$ and the set of places in $F^+$ not dividing $p$ where $U^v$ is ramified.
Let $S \defeq S_p^+ \cup S^+$.
Suppose that
\[
\left(\prod_{v' \in S_p^+} U_{v'}(N_{v'})\right)U^{S_p^+,v} \subset U^v
\]
for some compact open subgroup $U^{S_p^+,v} \in \cG(\bA_{F^+}^{S_p^+,v,\infty})$ where $\iota_{\tld{v}'}(U_{v'}(N_{v'}))$ is the $N_{\tld{v}'}$-th congruence subgroup.
For each $v' \in S_p^+$, let $\tau(N_{\tld{v}'})$ be the (finite) set of inertial types which factor through the (finite) quotient $I_{F_{\tld{v}'}}/I_{F_{\tld{v}'}}^{N_{\tld{v}'}-1}$ where $I_{F_{\tld{v}'}}^{N_{\tld{v}'}-1}$ denotes the subgroup in the upper numbering filtration.
Let $\rbar_{\tld{v}'}$ be the restriction of $\rbar$ to $G_{F_{\tld{v'}}}$ and $\Spec R^{\mu,N_{\tld{v}'}}_{\rbar_{\tld{v}'}}$ be the (reduced) scheme theoretic union of $\Spec R^{\tau,\mu}_{\rbar_{\tld{v}'}}$ for $\tau \in \tau(N_{\tld{v}'})$ where $R^{\tau,\mu}_{\rbar_{\tld{v}'}}$ denotes the semistable deformation ring of type $(\mu,\tau)$.
We let $\cG_n$ be the disconnected split reductive group scheme from \cite{CHT} (see \cite[\S\,A.3]{MLM}) and $D_{v'}^{N_{\tld{v}'}}$ denote the $\cG_n$-valued deformation problem corresponding to $R^{\mu,N_{\tld{v}'}}_{\rbar_{\tld{v}'}}$ (see \cite[\S\,A.4]{MLM}).

Let $\xi \defeq \varepsilon^{1-n}\delta_{F/F^+}^n: G_{F^+} \ra \cO^\times$ where $\delta_{F/F^+}$ is the quadratic character of $G_{F^+}/G_F$ and consider the global $\cG_n$-deformation datum
\[
\cS = (F/F^+,S,\cO,\rbar,\xi,\{D_{v'}^{\xi_{v'}}\}_{v'\in S^+} \cup \{D_{v'}^{N_{\tld{v}'}}\}_{v' \in S_p^+})
\]
in the sense of \cite[A.3.2]{MLM}.

Let $R_\infty$ be $R^{\mathrm{loc}}_{\cS,S}[\![x_1,\ldots,x_g]\!]$ where $g = q - [F:\Q]n(n-1)/2$ as in \cite[\S\,A.5]{MLM}.
We define
\[
M_\infty \defeq \varprojlim_{U_v \subset \mathbf{K},\,r} \cO \otimes_{\mathbf{R}} \prod_{m\in \N} M_{m,U_v,r}^\square
\]
where $\mathbf{R}$ and the map $\mathbf{R} \onto \cO$ are defined as in \cite[Appendix A.5]{MLM}, $W$ is taken to be $\cO$, and $M_{m,U_v,r}^\square$ is defined analogously to $M_{m,K_p,r}^\square$ varying one place dividing $p$ rather than all of them.

The result then follows from the proofs of \cite[Lemma A.1.1]{MLM} and \cite[Lemma 4.17]{CEGGPS} with the obvious modifications.
That the action of $R^{\mathrm{loc}}_{\cS,S}$ from \cite[\S\,A.5]{MLM} factors through the modified version of $R^{\mathrm{loc}}_{\cS,S}$ defined above follows from local-global compatibility at $p$ \cite{BLGGT2}
and the fact that the local Langlands correspondence for $\GL_n$ is depth preserving (see \cite[Proposition  4.2 and Lemma 4.3]{ABPS}). %
\end{proof}

\begin{thm}
\label{thm:main:LGC:prime}
Let $\rbar:G_F\rightarrow \GL_n(\F)$ be a continuous absolutely irreducible Galois representation, satisfying
\begin{enumerate}
\item
\label{it:autom:0}
$\cG(\cO_{F_v^+}) U^v$ is sufficiently small;
\item
\label{it:autom:1}
the image of $\rbar|_{G_{F(\zeta_p)}}$ is adequate;
\item
\label{it:autom:2} $\Hom_{\F[\mathbf{K}]}\big(F(\lambda)\otimes_{\F}\omega^{n-1}\circ\det,\pi(\rbar)|_{\mathbf{K}}\big)\neq 0$ for some weight $\lambda\in X_1(\un{T})$ with $\lambda+\eta$ being $5n$-generic Fontaine--Laffaille (cf.~Definition~\ref{defn:mGenFL}).
\end{enumerate}
Then the conjugacy class of $\rbar|_{G_{F_{\tld{v}}}}$ can be recovered from the isomorphism class of the $\GL_n(F_{\tld{v}})$-representation $\pi(\rbar)$.
\end{thm}
\begin{proof}
The hypotheses imply the hypotheses of Proposition \ref{prop:patch}.
Let $M_\infty$ be an arithmetic $R_\infty$-module as in this proposition.
Then condition~(\ref{it:autom:2}) implies Condition~\ref{cond:modularity of obv wt} holds by Lemma \ref{lemma:condition}, which together with Theorem~\ref{thm:main:LGC} finishes the proof.
\end{proof}

\clearpage{}%
\clearpage{}%

\begin{appendix}

\section{Figures}
Let $\Omega^\pm$ be a $\Lambda$-lift and $a\in\Z/t$. In the following example, a sequence of integers is listed decreasingly from left to right (from $k_{a,0}$ to $k_{a,6}$). For each choice of $n\geq k\geq 1$, the dashed black line below gives the corresponding set $\Omega_{\psi_a,k}$.
\begin{figure}[ht]
\caption{Example of $\Omega_{\psi_a,k}$ for $n\leq k\leq 1$}
\label{fig:attach:set:to:seq}
\includegraphics[scale=0.22]{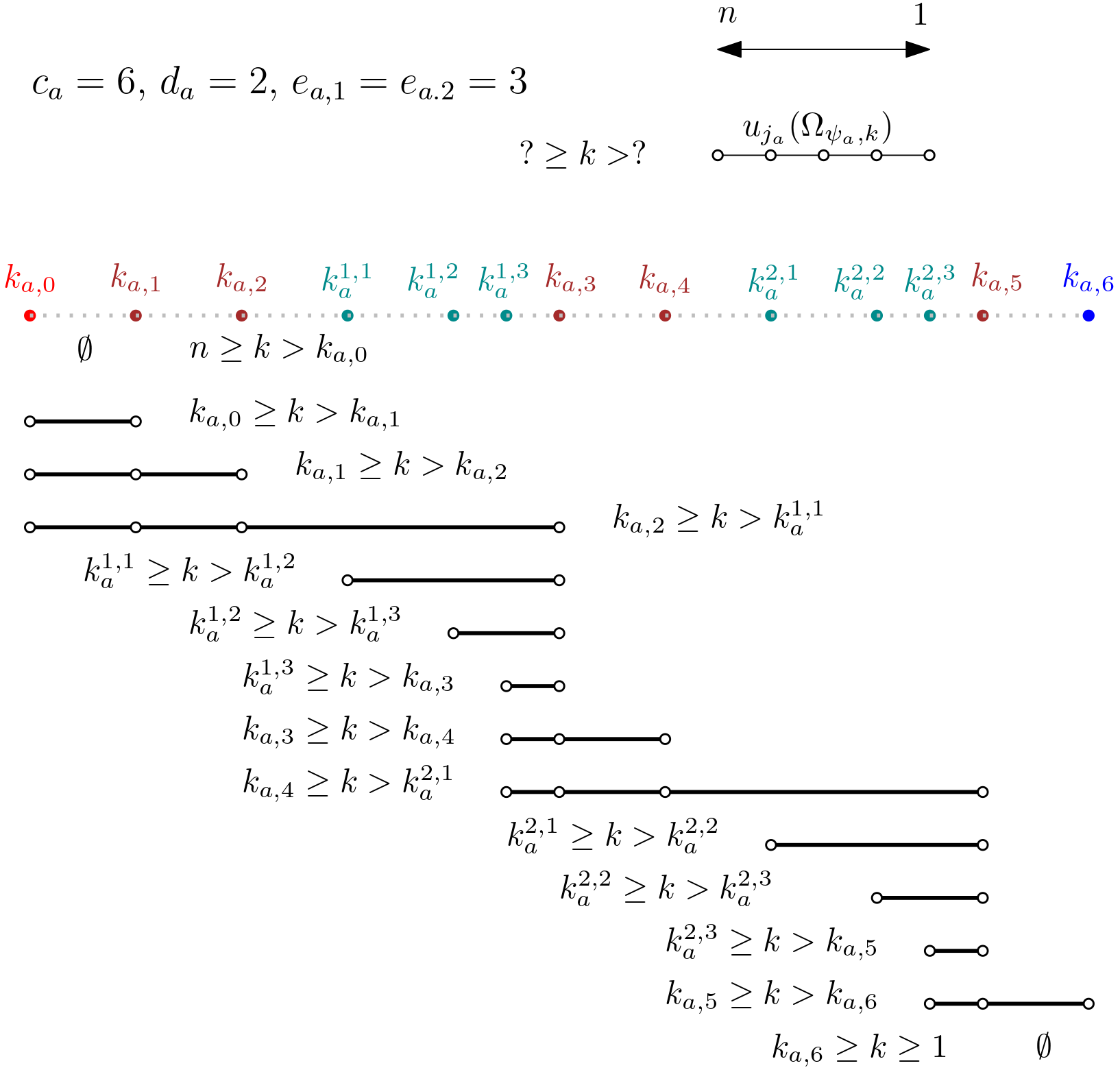}
\end{figure}%

\newpage

Let $\Omega^\pm$ be a constructible $\Lambda$-lift of type \rm{I}. The following figure illustrates $(v_{j_1}^{\Omega^\pm,\sharp})^{-1}$ and $(v_{j_1}^{\Omega^\pm,\flat})^{-1}$ explicitly. We distinguish these two permutations using two different colors. For each colored arrow, its target is the successor of its source under the corresponding permutation.
\begin{figure}[ht]
\caption{Examples of the permutations when $\Omega^\pm$ is a constructible $\Lambda$-lift of type~\rm{I}}
\label{fig:ex-TypeI}
\includegraphics[scale=0.23]{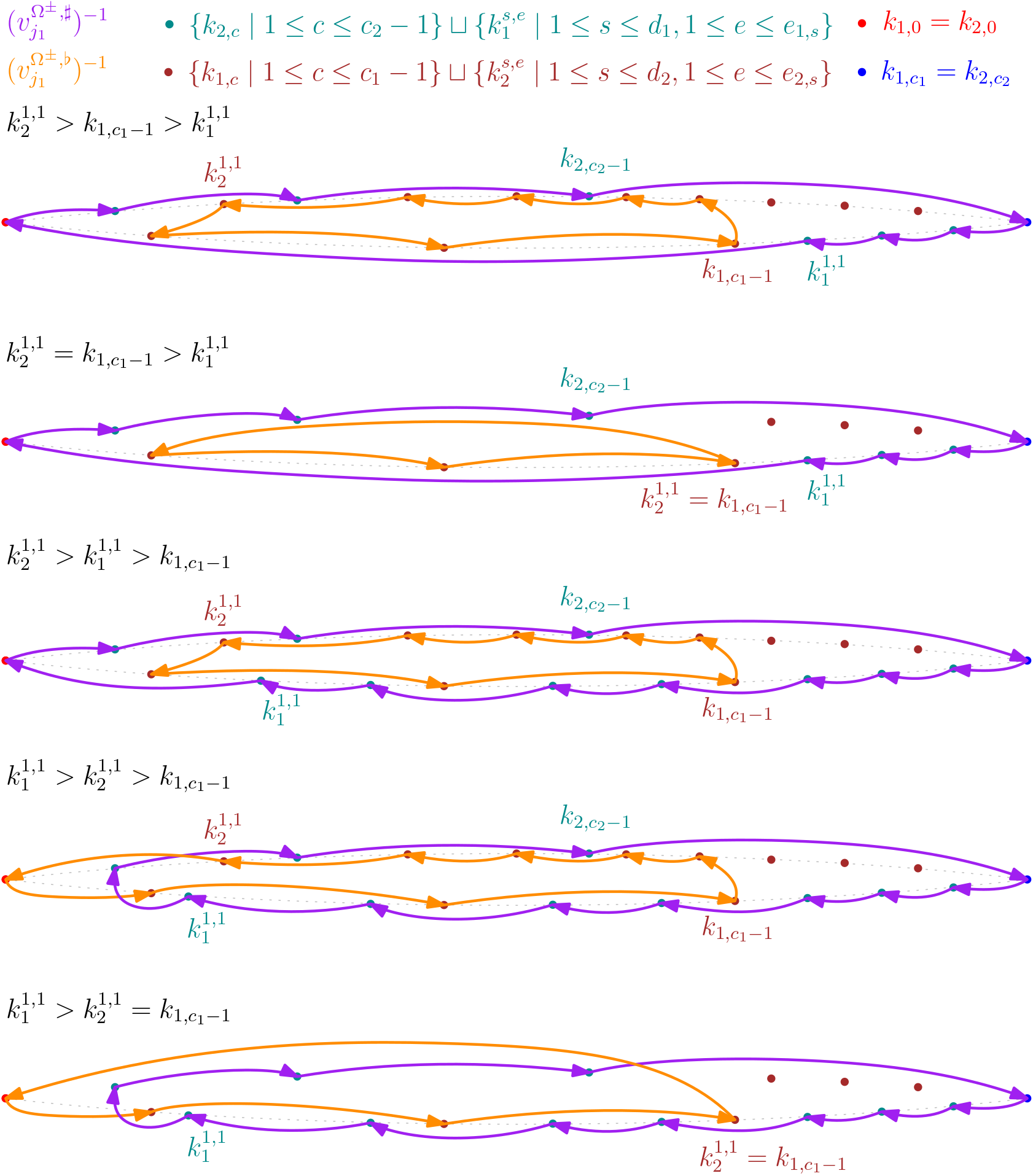}
\end{figure}%

\newpage

Let $\Omega^\pm$ be a constructible $\Lambda$-lift of type \rm{II}. The following figure illustrates $(v_{j_1}^{\Omega^\pm,\sharp})^{-1}$ explicitly. We split the orbit of this permutation into two parts and distinguish them using two different colors. For each colored arrow, its target is the successor of its source under $(v_{j_1}^{\Omega^\pm,\sharp})^{-1}$.
\begin{figure}[ht]
\caption{Examples of the permutation when $\Omega^\pm$ is a constructible $\Lambda$-lift of type~\rm{II}}
\label{fig:ex-TypeII}
\includegraphics[scale=0.23]{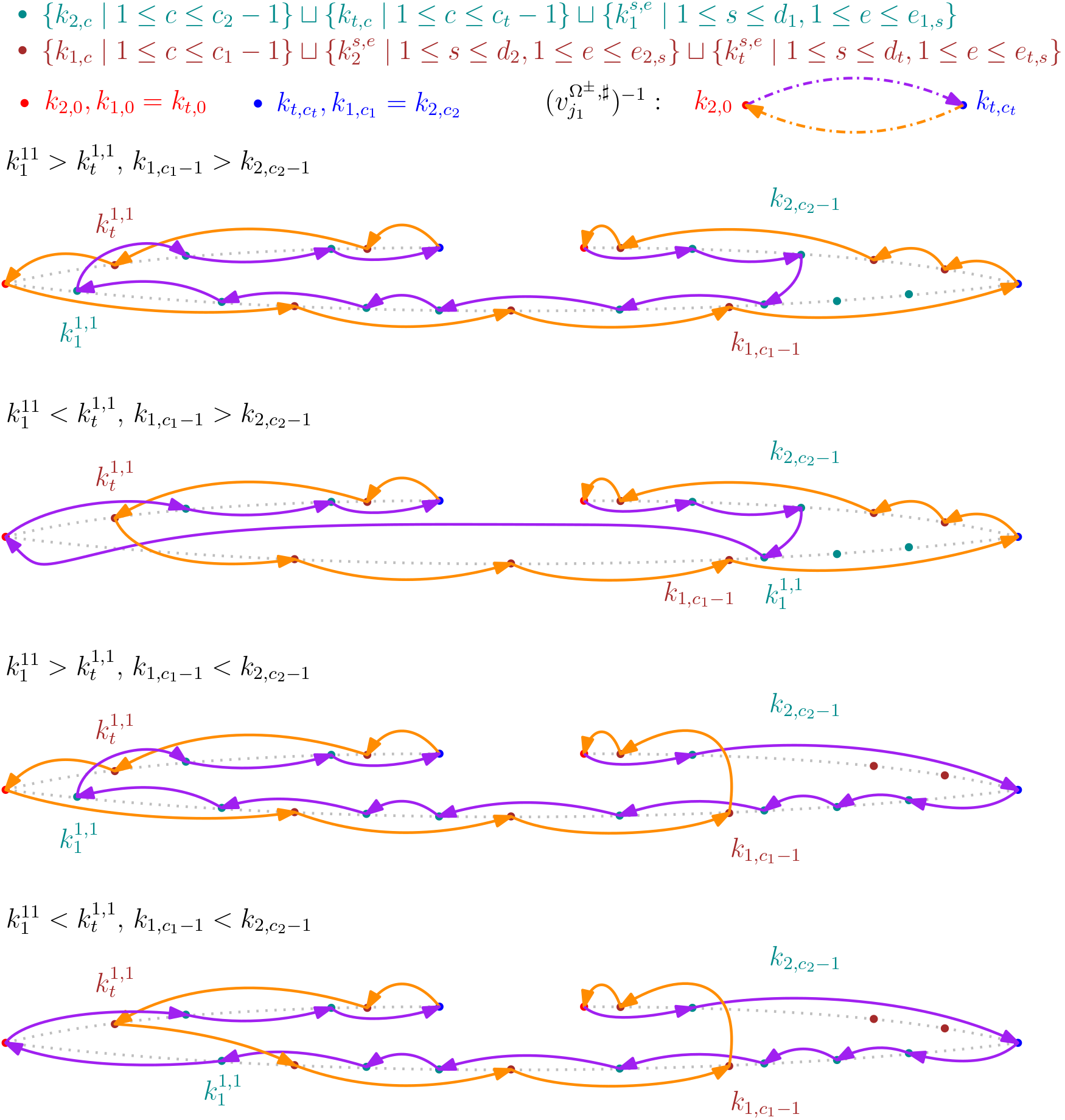}
\end{figure}%

\newpage

The following figure illustrates the partition of set $(\Z/t)_\Sigma=(\Z/t)_\Sigma^+\sqcup (\Z/t)_\Sigma^-$ for each connected component $\Sigma\in\pi_0(\Omega^\pm)$. In the figure, the index of each $\Lambda^\square$-interval is increasing as $a,a+1,a+2,\dots$ from left to right. The $+$ (resp.~$-$) sign means that $a,a+2,\dots$ are elements of $(\Z/t)_\Sigma^+$ (resp.~$a+1,a+3,\dots$ are elements of $(\Z/t)_\Sigma^-$).

\begin{figure}[ht]
\caption{Direction and sign}
\label{fig:two:dir}
\includegraphics[scale=0.18]{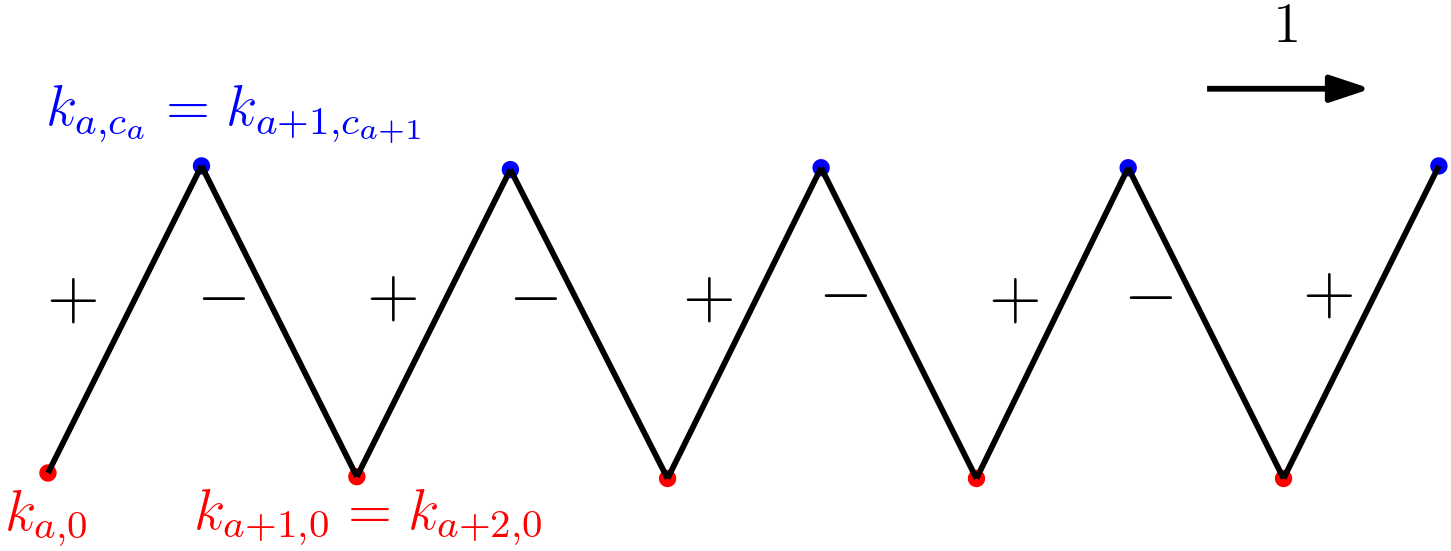}
\end{figure}%

In the following figure, we fix our choice of colors for different symbols from now on. In each figure that follows, we fix a connected component $\Sigma\in\pi_0(\Omega^\pm)$, a direction $\varepsilon\in\{1,-1\}$. Each time we treat an oriented permutation, we use orange (resp.~purple) for its fixed $\varepsilon$-tour (resp.~fixed$-\varepsilon$-tour). If $\varepsilon=1$, we use brown dots for elements of $\bigsqcup_{a'\in(\Z/t)_\Sigma^+}\mathbf{n}^{a',+}\setminus\{k_{a',c_{a'}}\}$ and $\bigsqcup_{a'\in(\Z/t)_\Sigma^-}\mathbf{n}^{a',-}\setminus\{k_{a',0}\}$,
and use green dots for elements of
$\bigsqcup_{a'\in(\Z/t)_\Sigma^-}\mathbf{n}^{a',+}\setminus\{k_{a',c_{a'}}\}$ and $\bigsqcup_{a'\in(\Z/t)_\Sigma^+}\mathbf{n}^{a',-}\setminus\{k_{a',0}\}$.
If $\varepsilon=-1$, we switch brown and green for the sets above.
\begin{figure}[ht]
\caption{Notation and color}
\label{fig:crawl}
\includegraphics[scale=0.27]{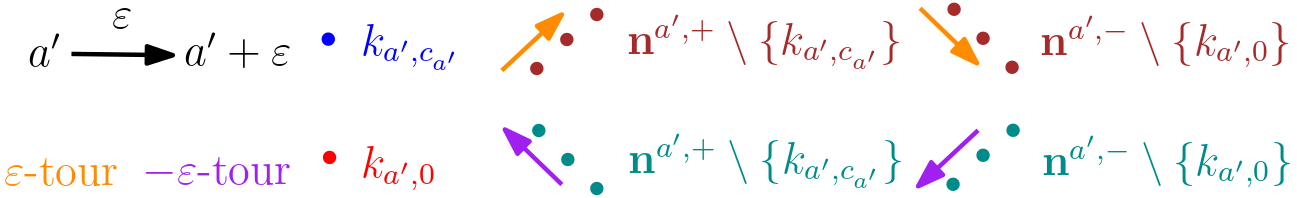}
\end{figure}%

\begin{figure}[ht]
\caption{A $\varepsilon$-crawl from $k$ to $k'$}
\label{fig:crawl}
\includegraphics[scale=0.22]{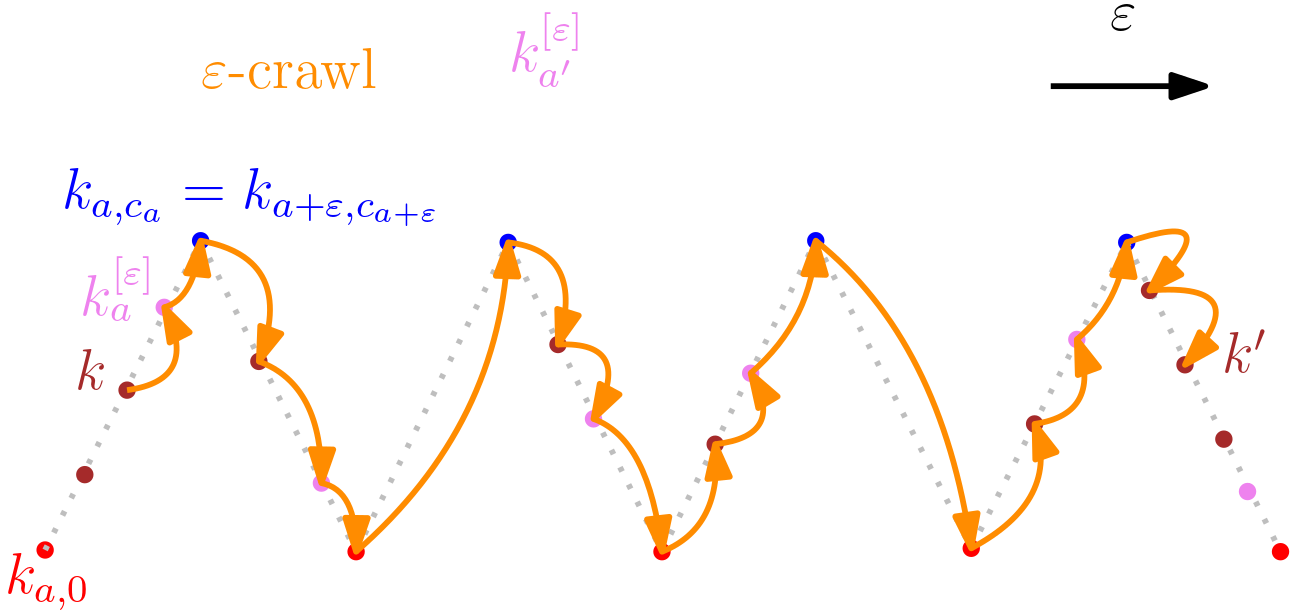}
\end{figure}%

\newpage

The following figure illustrates all four typical kinds of $\varepsilon$-jumps that appear.

\begin{figure}[ht]
\caption{A $\varepsilon$-jump at $k_a^{[\varepsilon]}$}
\label{fig:jumps}
\includegraphics[scale=0.26]{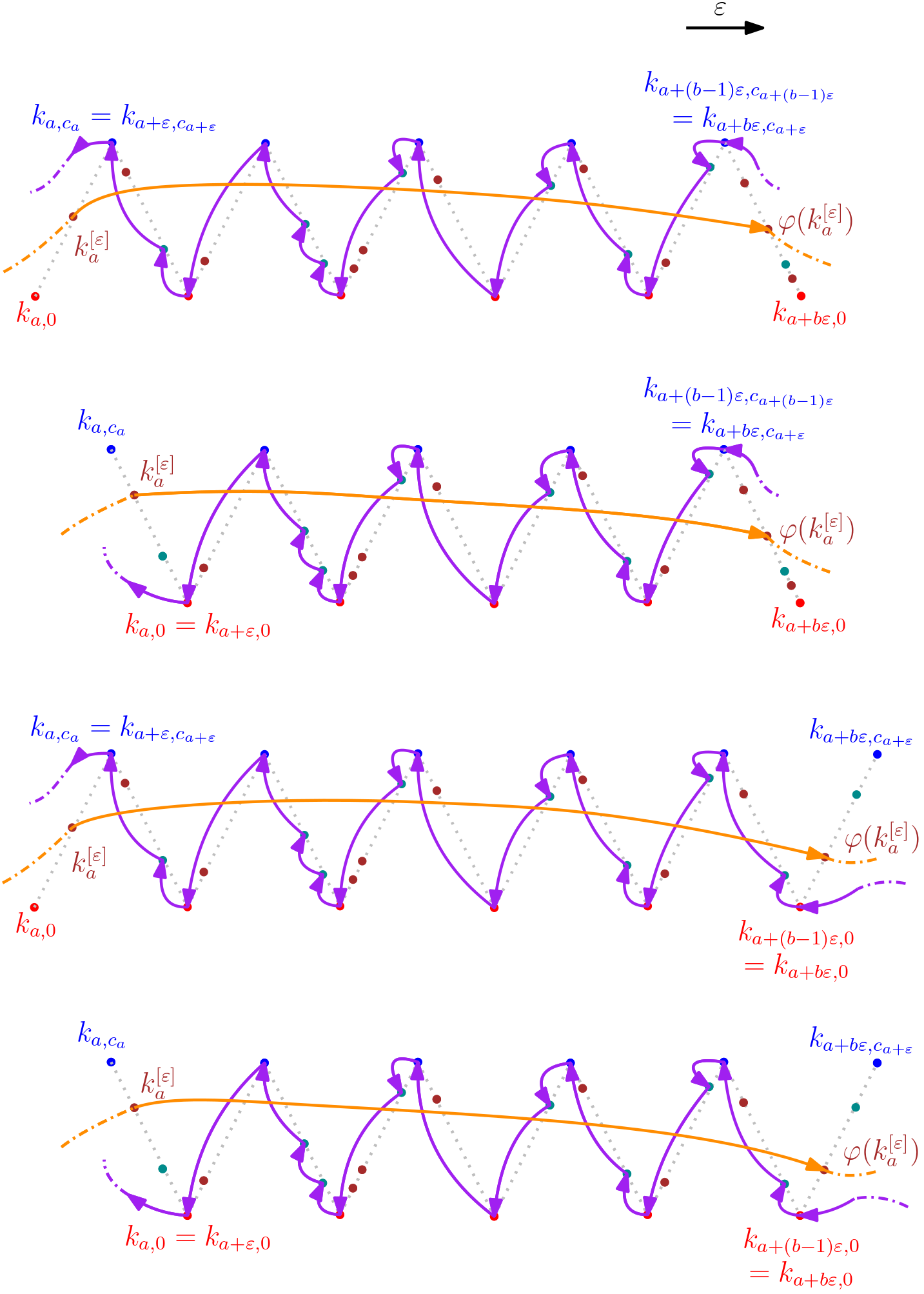}
\end{figure}%

\newpage

\begin{figure}[ht]
\caption{Examples of oriented permutations when $\Sigma$ is not circular}
\label{fig:OrPer}
\includegraphics[scale=0.25]{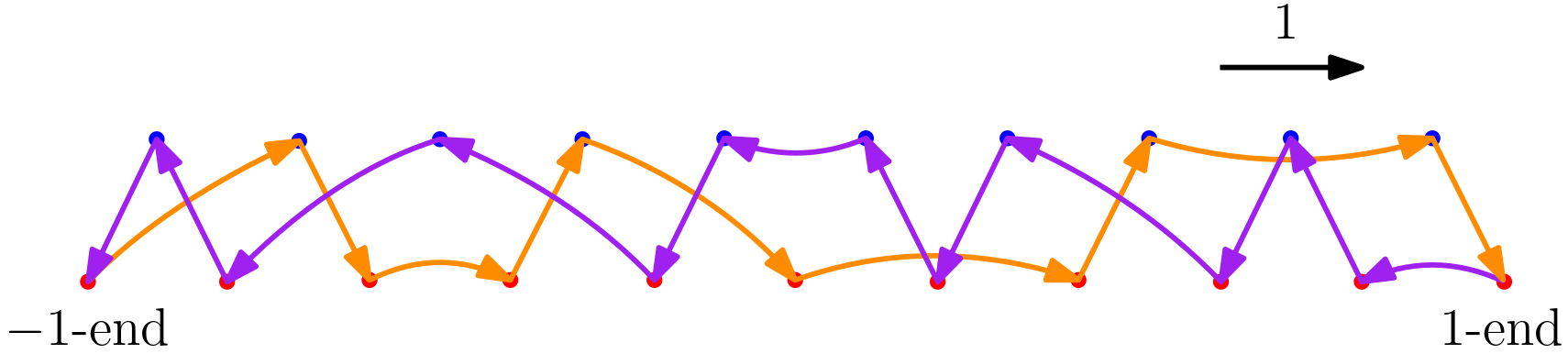}
\end{figure}%

\begin{figure}[ht]
\caption{Examples of oriented permutations when $\Sigma$ is circular}
\label{fig:OrPer}
\includegraphics[scale=0.28]{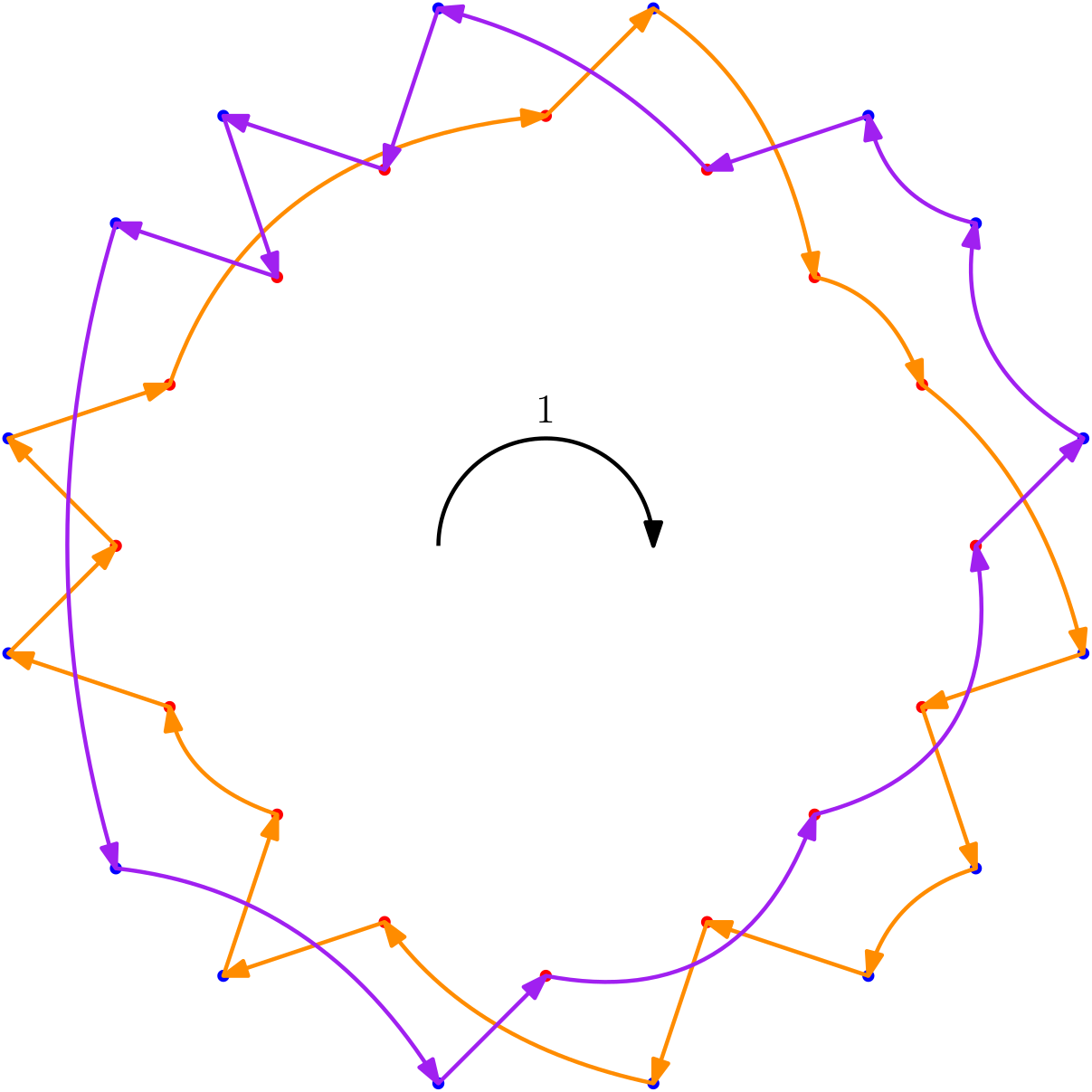}
\end{figure}%

\newpage

\begin{figure}[ht]
\caption{Item \ref{it: interaction of jumps} in Definition \ref{def: oriented permutation}}
\label{fig:item4Def}
\includegraphics[scale=0.22]{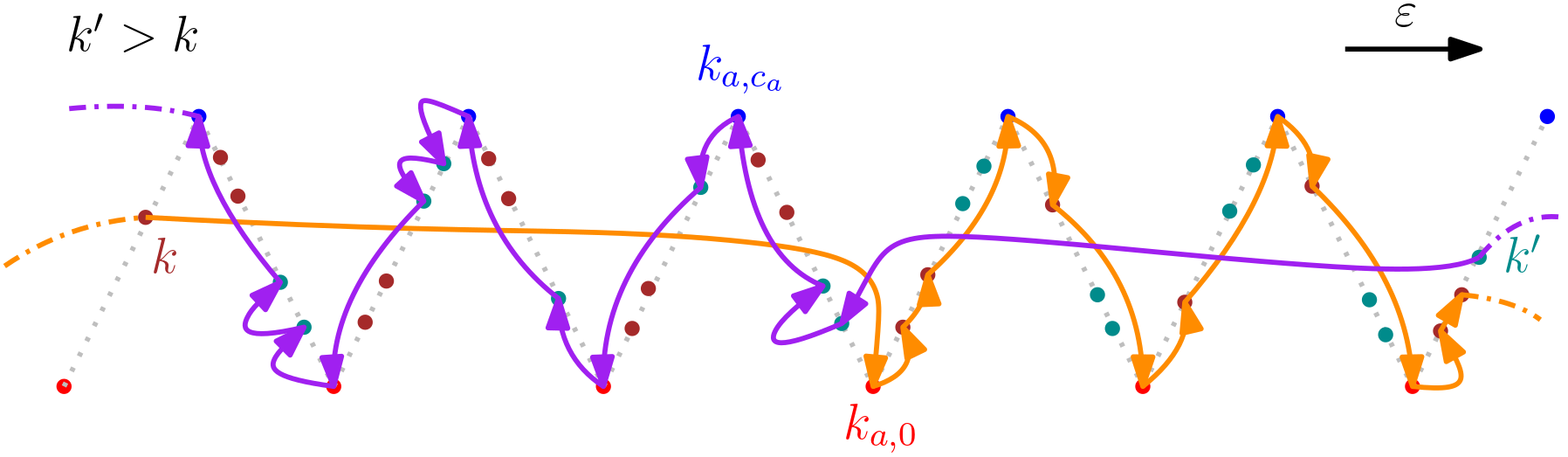}
\end{figure}%

\begin{figure}[ht]
\caption{An example of $\varphi_0$ when $\Sigma$ is not circular.}
\label{fig:step1}
\includegraphics[scale=0.20]{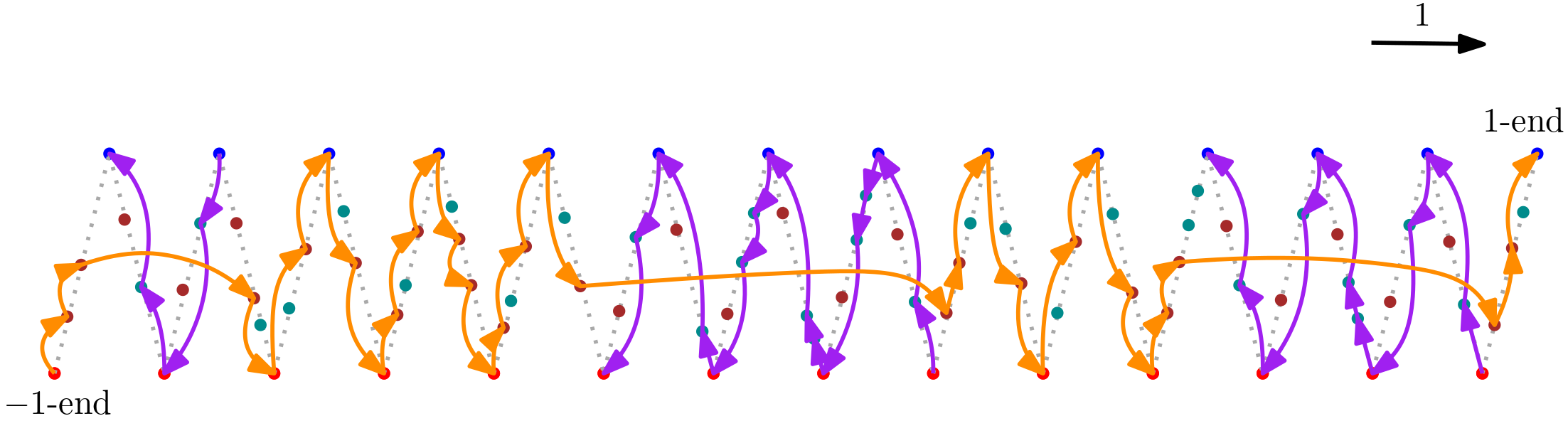}
\end{figure}%

\begin{figure}[ht]
\caption{Extend $\varphi_0$ to $\varphi_1$ by completing the $-1$-tour}
\label{fig:BJump}
\includegraphics[scale=0.26]{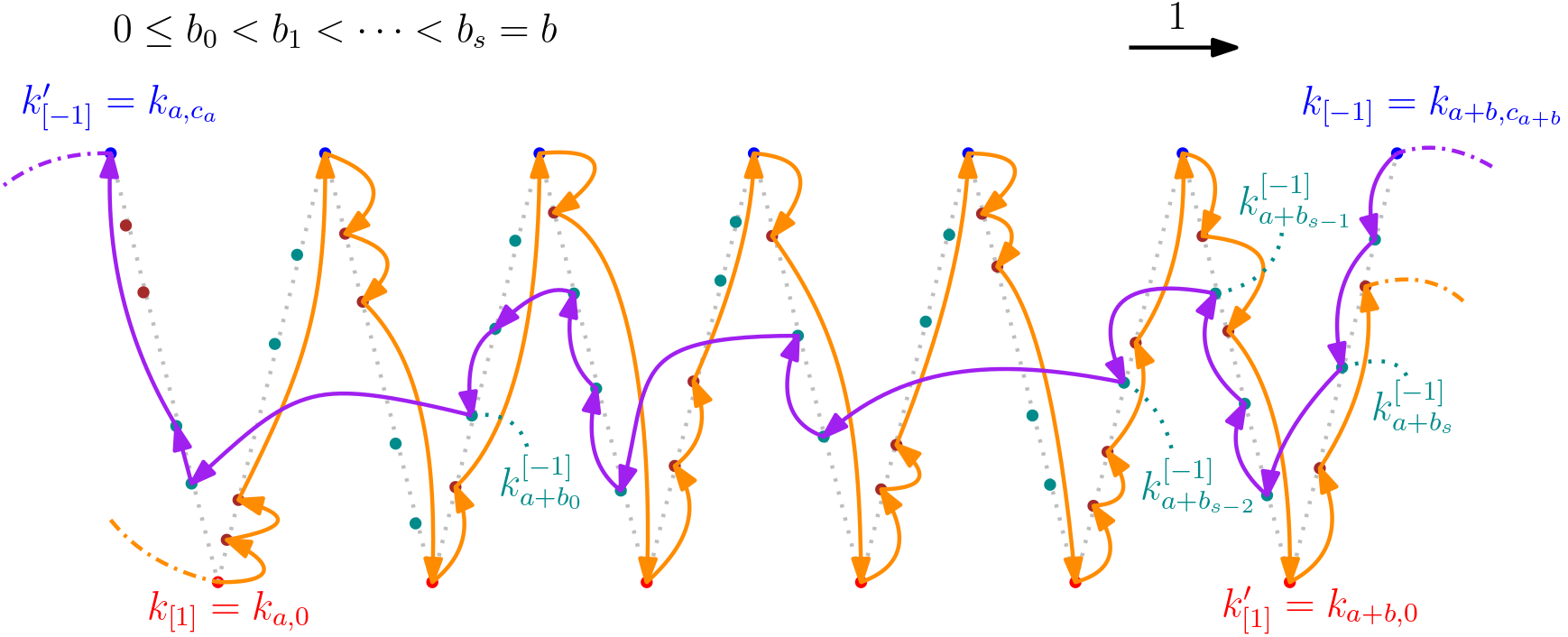}
\end{figure}%

\newpage

The definition of $k_a^\sharp$ and $k_a^\flat$ can be reduced to the cases illustrated in the following figure.

\begin{figure}[ht]
\caption{Definition of $k_a^\sharp$ and $k_a^\flat$}
\label{fig:sharp:flat}
\includegraphics[scale=0.25]{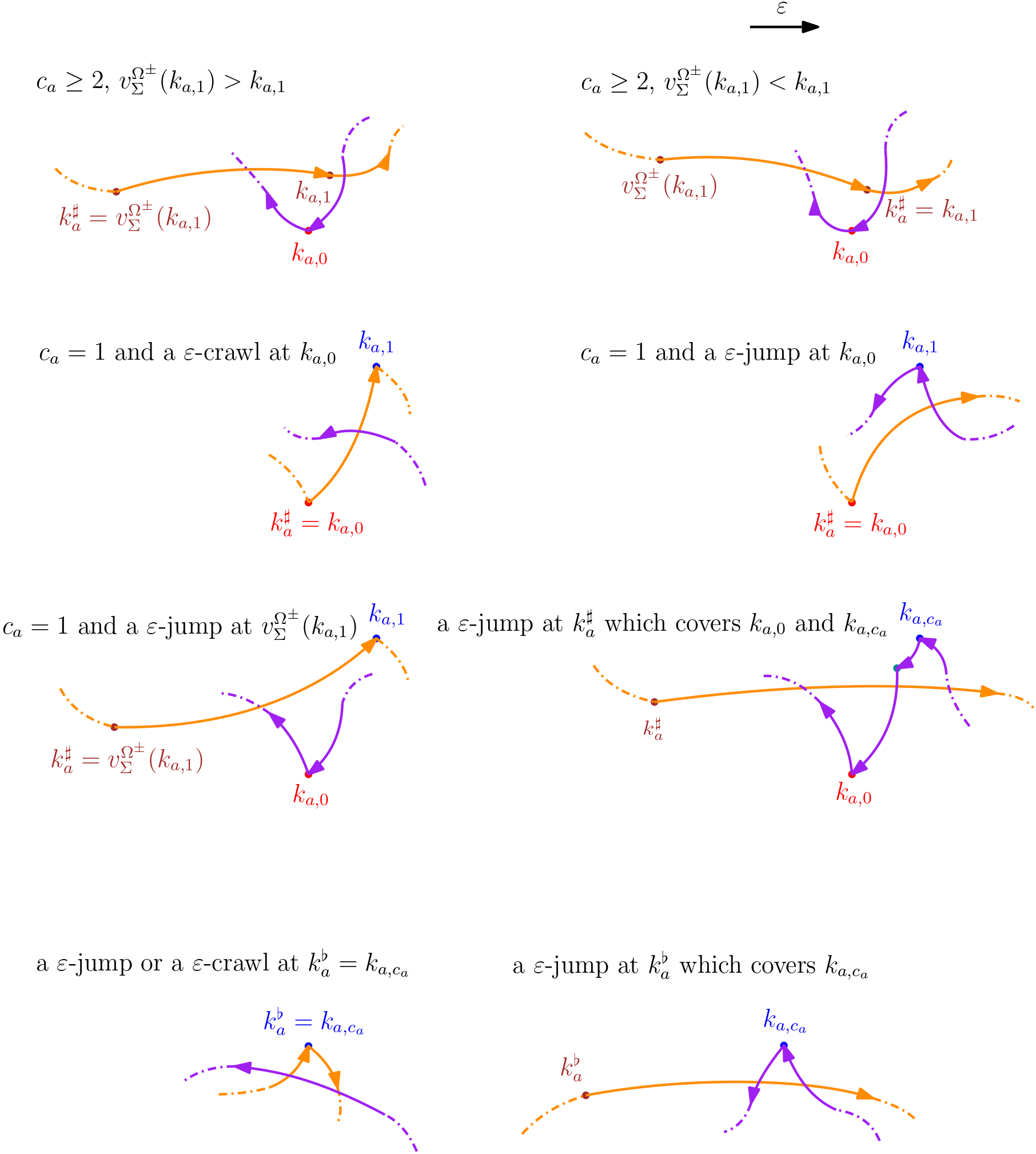}
\end{figure}%

\newpage

Let $\Omega^\pm$ be a constructible $\Lambda$-lift of type \rm{III}. When the fixed $\varepsilon$-tour of $(v_\Sigma^{\Omega^\pm})^{-1}$ contains a $\varepsilon$-jump at $k=k_a^{[\varepsilon]}$, the following figure sketches the sets $\Omega_{a,k,j}$ and $\Omega_{a,k+1,j}$ for each $a\in(\Z/t)_\Sigma$ satisfying $\Omega_{a,k,j}\neq\Omega_{a,k+1,j}$.

\begin{figure}[ht]
\caption{Comparison between $\Omega_{a,k,j}$ and $\Omega_{a,k+1,j}$}
\label{fig:fromKtoK1}
\includegraphics[scale=0.26]{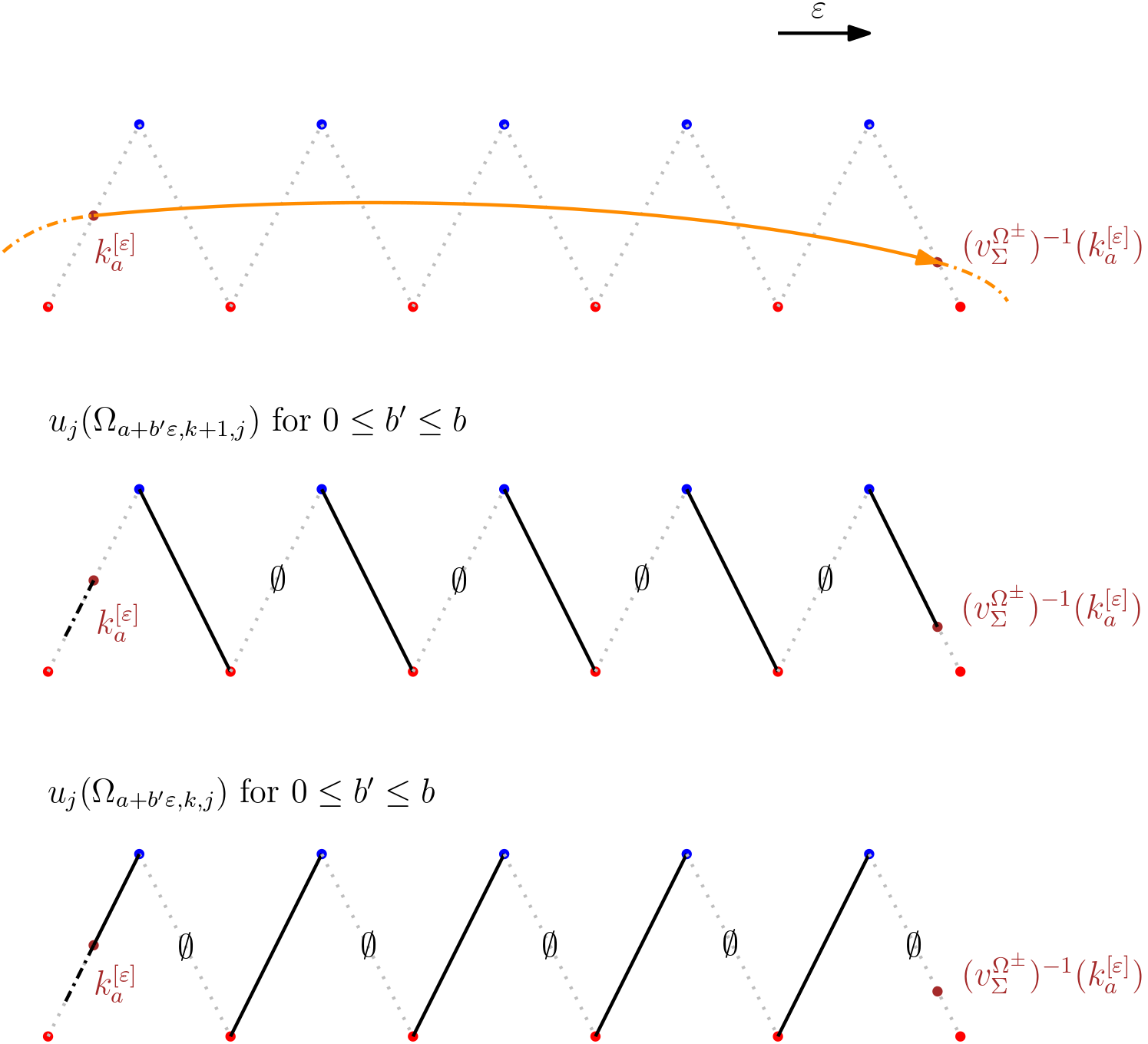}
\end{figure}%

\newpage

Let $\Omega^\pm$ be a constructible $\Lambda$-lift of type \rm{III}. The following figure illustrates the four cases in Lemma~\ref{lem: pair of connected sets}. Each grey doted line represents the image under $u_j$ of a $\Lambda^\square$-interval of $\Omega^\pm$.

\begin{figure}[ht]
\caption{Four cases in Lemma~\ref{lem: pair of connected sets}}
\label{fig:4cases}
\includegraphics[scale=0.26]{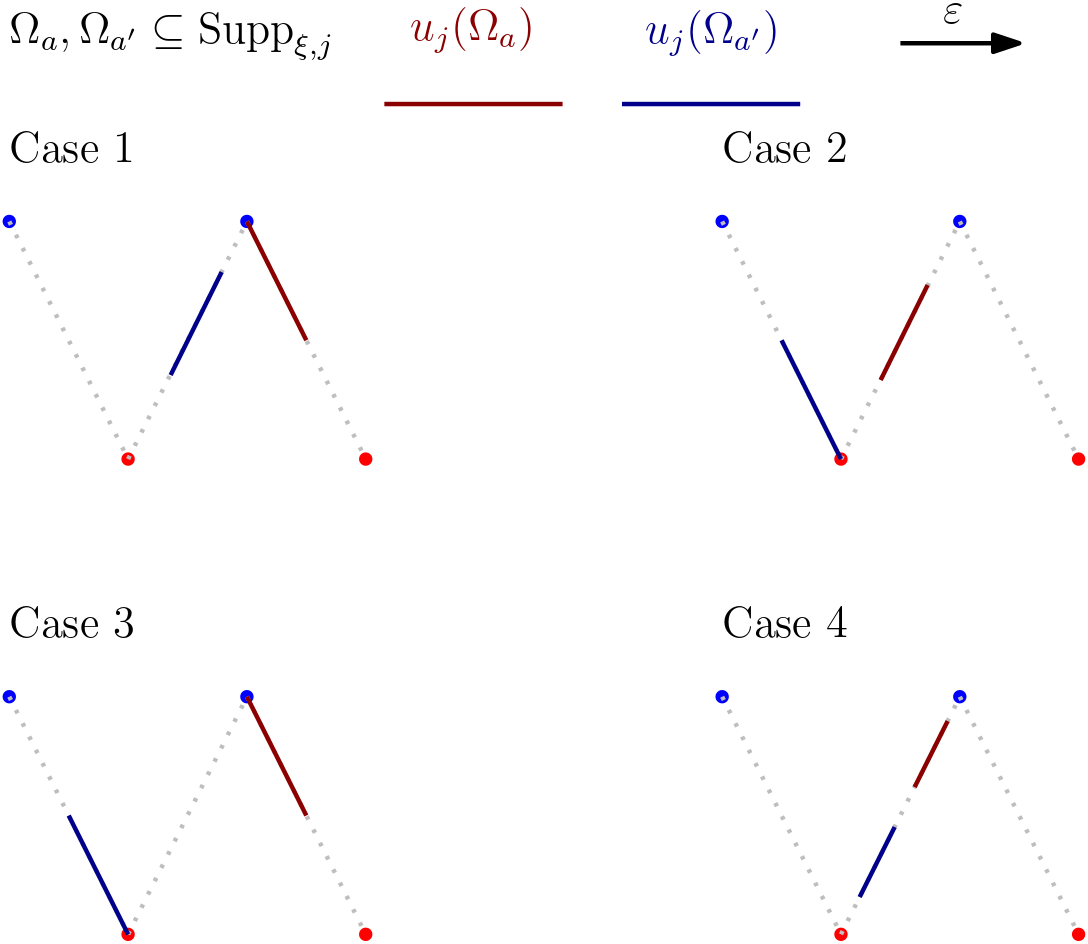}
\end{figure}%

\end{appendix} \clearpage{}%

\newpage
\bibliography{Biblio}
\bibliographystyle{amsalpha}

\end{document}